\documentclass[11pt,twoside,leqno]{book}

\usepackage{macros}

\usepackage{mathrsfs}

\usepackage{makeidx}

\usepackage[]{hyperref}
\hypersetup{
    colorlinks=true,       
    linkcolor=black,          
    citecolor=black,        
}

\usepackage{amsmath}
\usepackage{amsfonts,amssymb}
\usepackage{enumerate}
\usepackage{amsthm}
\usepackage{a4wide}
\usepackage{epsfig}
\usepackage{tikz}

\theoremstyle{plain}

\newtheorem*{prop*}{Proposition}
\newtheorem{theo}{Theorem}[section]
\newtheorem{prop}[theo]{Proposition}
\newtheorem{lemm}[theo]{Lemma}
\newtheorem{coro}[theo]{Corollary}
\newtheorem{assu}[theo]{Assumption}
\newtheorem{defi}[theo]{Definition}
\theoremstyle{definition}
\newtheorem{rema}[theo]{Remark}
\newtheorem*{rema*}{Remark}
\newtheorem*{remas}{Remarks}
\newtheorem{nota}[theo]{Notation}
\newtheorem*{nota*}{Notation}
\newtheorem*{notas*}{Notations}

\DeclareMathOperator{\cnx}{div}
\DeclareMathOperator{\cn}{div}
\DeclareMathOperator{\RE}{Re}
\DeclareMathOperator{\diff}{d}

\DeclareMathOperator{\Op}{Op}

\DeclareMathOperator{\sign}{sign}

\numberwithin{equation}{section}

\title{\huge{Sobolev estimates \\ 
for two dimensional gravity water waves}}
\author{Thomas Alazard \and Jean-Marc Delort}

\date{}

\makeindex

\begin{document}

\frontmatter

\maketitle

\tableofcontents

\parskip=10pt 
\lineskip=4pt

\def\indexname{Index of notations}
\printindex

\mainmatter

\begin{center}
\textbf{Abstract}
\end{center}

Our goal in this paper is to apply a normal forms method 
to estimate the Sobolev norms 
of the solutions of the water waves equation. 
We construct a paradifferential change of unknown, without
 derivatives losses, which eliminates the part of the quadratic terms that bring non zero contributions in a Sobolev energy inequality. Our approach is purely Eulerian: we work on 
the Craig-Sulem-Zakharov formulation of the water waves equation. 

In addition to these Sobolev estimates, we also prove $L^2$-estimates 
for the $\px^\alpha Z^\beta$-derivatives of the solutions of the water waves equation, where $Z$ is the Klainerman vector field $t\partial_t +2x\px$. 
These estimates are used in the paper~\cite{AlDelMain}. In that reference, we prove  a global existence result for the water waves equation with smooth, small, and
 decaying at infinity Cauchy data, and we obtain an asymptotic description in physical coordinates of the solution,
 which shows that modified scattering holds. 
 The proof of this global in time existence result relies on the simultaneous bootstrap of some H\"older and Sobolev a priori
estimates for the action of iterated Klainerman vector fields on the solutions of the water waves equation. The present paper
contains the proof of the Sobolev part of that bootstrap. 
\clearpage

\chapter*{Introduction}
\renewcommand{\theequation}{\thesection.\arabic{equation}}
\renewcommand{\thesection}{\arabic{section}}

\section{Description of the main results}

This paper 
addresses the well-posedness of the initial value problem for 
the motion of a two-dimensional incompressible fluid under the influence of gravity. 
At time~$t$, the fluid domain, denoted by 
$\Omega(t)$, has a free boundary  
described by the equation~$y=\eta(t,x)$, so that
$$
\Omega(t)=\left\{\, (x,y)\in \xR^2\,;\, y<\eta(t,x)\,\right\}.
$$
The velocity field $v\colon \Omega \rightarrow \xR^{2}$ is assumed to be irrotational and to satisfy the incompressible Euler equations. It follows that 
$v=\nabla_{x,y} \phi$ 
for some velocity potential~$\phi\colon \Omega\rightarrow \xR$ satisfying 
\begin{equation}\label{intro:1}
\Delta_{x,y}\phi=0,\quad 
\partial_{t} \phi +\mez \la \nabla_{x,y}\phi\ra^2 + P +g y = 0,
\end{equation}
where~$g>0$ 
is the acceleration of gravity, 
$P$ is the pressure term, $\nabla_{x,y}=(\partial_x,\partial_y)$ and $\Delta_{x,y}=\px^2+\partial_y^2$. 
Hereafter, the units of length and time are chosen so that~$g=1$. 

The water waves equations are then given by two boundary conditions on the free surface:
\begin{equation}\label{intro:2}
\left\{
\begin{aligned}
&\partial_{t} \eta = \sqrt{1+(\px \eta)^2}\, \partial_n \phi 
&&\text{on }\partial\Omega, \\
& P=0
&&\text{on }\partial\Omega,
\end{aligned}
\right.
\end{equation}
where~$\partial_n$ is the outward normal derivative of~$\Omega$, so that 
$\sqrt{1+(\px \eta)^2}\, \partial_n \phi =\partial_y\phi-(\px\eta)\px\phi$. 

It is well known that the linearized equation around the equilibrium $\eta=0$ and $\phi=0$ can be written under the form $\partial_t^2u+\Dx u=0$ where 
$\Dx$ is the Fourier multiplier with symbol~$\la \xi\ra$. 
Allowing oneself to oversimplify the problem, one can think of the 
linearized equation around a nontrivial solution as the equation 
$(\partial_t +V\px)^2u+a \Dx u=0$, where $V$ is the trace of the horizontal component of the velocity at the free surface and $a=-\partial_y P\arrowvert_{y=\eta}$ is the so-called Taylor coefficient. To insure that the Cauchy problem for the latter equation is well-posed, one has to require that $a$ is bounded from below by a positive constant. This is known as the Taylor sign condition; see \cite{Ebin} for an ill-posedness result without this requirement. 
That the well-posedness of the Cauchy problem depends on an assumption on the sub-principal term $a\Dx$ reflects the fact that 
the linearized equation has a double characteristic
, see Craig \cite[Section 4]{Craig1987} or Lannes \cite[Section 4.1]{LannesJAMS}. This leads to an apparent loss of $1/2$ derivative in the study of the 
Cauchy problem in Sobolev spaces. However, Nalimov~\cite{Nalimov} 
proved that, in Lagrangian coordinates, the Cauchy problem is well-posed 
locally in time, in the framework of Sobolev spaces, 
under an additional smallness assumption 
on the data; see also the results of Yosihara~\cite{Yosihara} and Craig~\cite{Craig1985}. 

Notice that if $\eta$ and $\phi$ are of size $\eps$ then 
$a=1+O(\eps)$ so that the Taylor sign condition is satisfied for $\eps$ 
small enough. As was first proved by Wu~\cite{WuInvent,WuJAMS}, 
 this property is always true, without smallness assumption 
(including the case that the interface is not a graph, 
as long as the interface is non self-intersecting). As a result, the well-posedness of the Cauchy problem     
was proved in~\cite{WuInvent,WuJAMS} without smallness assumption. 
Several extensions or different proofs are known and 
we refer the reader to C{\'o}rdoba, C{\'o}rdoba and Gancedo~\cite{CCG}, Coutand-Shkoller~\cite{CS}, 
 Lannes~\cite{LannesJAMS,LannesKelvin,LannesLivre}, Linblad~\cite{LindbladAnnals}, Masmoudi-Rousset~\cite{MR}, Shatah-Zeng \cite{SZ,SZ2}, 
 Zhang-Zhang~\cite{ZhangZhang} for recent results concerning 
 the gravity water waves equations. 

Two different approaches were used in the analysis 
of the water waves equations: 
the Lagrangean formulation with a more geometrical point of view and the Eulerian formulation in relation with 
microlocal analysis.
Our analysis  
is entirely based on the 
Eulerian formulation of the water waves equations: we shall work on the 
so-called Craig--Sulem--Zakharov system which we introduce below. 
Let us also mention that the idea of 
studying the water waves equations by 
means of microlocal 
analysis is influenced by 
the papers by Craig-Schanz-Sulem~\cite{CSS}, Lannes~\cite{LannesJAMS} 
and Iooss-Plotnikov~\cite{IP}. More precisely, we follow the paradifferential analysis 
introduced in \cite{AM} and further developed in \cite{ABZ1,ABZ3}. 
We explain later in this introduction how this allows to overcome the apparent loss 
of derivative in the Cauchy problem. 

Following Zakharov~\cite{Zakharov1968} 
and Craig and Sulem~\cite{CrSu}, we work 
with the trace of~$\phi$ at the free boundary
$$
\psi(t,x)=\phi(t,x,\eta(t,x)),
$$
and introduce the Dirichlet-Neumann operator~$G(\eta)$ 
that relates~$\psi$ to the normal derivative 
$\partial_n\phi$ of the potential by 
$$
(G(\eta) \psi)  (t,x)=\sqrt{1+(\partial_x\eta)^2}\,
\partial _n \phi\arrowvert_{y=\eta(t,x)}.
$$
Then~$(\eta,\psi)$ solves (see~\cite{CrSu}) 
the system
\begin{equation}\label{intro:3}
\left\{
\begin{aligned}
&\partial_t \eta=G(\eta)\psi,\\
&\partial_t \psi + \eta+ \frac{1}{2}(\partial_x \psi)^2
-\frac{1}{2(1+(\partial_x\eta)^2)}\bigl(G(\eta)\psi+(\partial_x  \eta )(\partial_x \psi)\bigr)^2= 0.
\end{aligned}
\right.
\end{equation}

Consider a classical solution $(\eta,\psi)$ of \e{intro:3}, such that $(\eta,\psi)$ belongs 
to $C^0([0,T];H^s(\xR))$ for some $T>0$ and $s>3/2$. Then 
it is proved in \cite{Bertinoro} that there exist a velocity potential $\phi$ and 
a pressure $P$ satisfying \e{intro:1} and \e{intro:2}. Thus it is sufficient to solve the 
Craig--Sulem--Zakharov formulation~\e{intro:3} 
of the water waves equations~\e{intro:1}-\e{intro:2}.

Our goal in this paper is to apply a normal forms method 
to estimate the Sobolev norms 
of the solutions to the water waves equations. In practice, one looks for a local diffeomorphism at 0 in~$H^s$, for $s$ large enough, so that the equation obtained by conjugation by this diffeomorphim be of the form 
of an equation with a cubic nonlinearity (while the water waves equations contain 
quadratic terms). 

The analysis of normal forms for the water waves system 
is motivated by physical considerations, such as the derivations of various equations 
in asymptotic regimes (see the recent paper by Totz and Wu \cite{TW}, 
the first rigorous results by Craig-Sulem-Sulem \cite{CrSuSu} 
and also the papers of Schneider and Wayne \cite{SW,SW2}). 
Another motivation is that, for solutions sufficiently small and sufficiently decaying at infinity of a dispersive equation, it is easier to prove global well-posedness for cubic nonlinearity. 
Let us mention that the results of this paper are used in \cite{AlDelMain} where we prove global existence of solutions for the two 
dimensional water waves equations with small, smooth, decaying at infinity Cauchy data, and get for these solutions a one term asymptotic expansion in physical variables when time goes to infinity. In particular, the form of these asymptotics shows that solutions do not scatter at infinity, i.e.\ do not behave like solutions of the linearized equation at zero. 

 Nonlinear changes of
unknowns, reducing the water waves equation to a cubic equation, have been known for quite a time (see Craig \cite{Craig96} or 
 Iooss and Plotnikov \cite[Lemma $1$]{IP-SW1}). 
 However, these transformations were losing derivatives, as a consequence of 
 the quasi-linear character of the problem (see \cite[Appendix C]{Wu10} for the study of the Poincar\'e-Shatah normal form associated to \e{intro:3}). 
 In her breakthrough paper, Wu~\cite{Wu09} 
proved that one can find good coordinates which overcome 
this loss of derivatives and ultimately proved an almost 
global existence result for two-dimensional gravity waves. 
Then Germain--Masmoudi--Shatah~\cite{GMS} and Wu~\cite{Wu10} have shown 
that the Cauchy problem for three-dimensional waves 
is globally in time well-posed for $\eps$ small enough 
(with linear scattering in Germain-Masmoudi-Shatah and no assumption about the 
decay to $0$ at spatial infinity of $\Dxmez \psi$ in Wu). Germain--Masmoudi--Shatah~\cite{GMS2} 
recently proved global existence 
for two-dimensional capillary waves.

We shall construct a paradifferential change of unknown, {\em without
 derivatives losses}, which eliminates the part of the quadratic terms that bring non zero contributions in a Sobolev energy inequality. Our main result is stated after we introduce some notations, but one can state one of its main corollary as follows: There exists $\gamma>0$ such that, for any $s\ge \gamma+1/2$, 
if $N_\gamma(t)=\lA \eta(t,\cdot)\rA_{\eC{\gamma}}
+\blA \Dxmez\psi(t,\cdot)\brA_{\eC{\gamma-\mez}}$ 
is small enough, then one can define an $H^s$-Sobolev energy, denoted by $M_s$, satisfying
\be\label{i3}
M_s(t)\sim \lA \eta(t,\cdot)\rA_{H^s(\xR)}^2+\blA \Dxmez\psi\brA_{H^{s-\mez}(\xR)}^2+
\blA (\nabla_{x,y}\phi)\arrowvert_{y=\eta}(t,\cdot)\brA_{H^{s-\mez}(\xR)}^2
\ee
and
\be\label{i4}
M_s(t)\le M_s(0)+\int_{0}^t C(N_\gamma(\tau))
N_\gamma(\tau)^2 M_s(\tau)\, d\tau.
\ee

Let us comment on these estimates. 
The key point is that the summand in the right hand side of \e{i4} is quadratic in $N_\gamma$ 
(while, for an equation containing quadratic terms in the 
nonlinearity, one obtains in general a linear bound). Then 
it follows from the Sobolev embedding that 
$\widetilde{M}_s(T)=\sup_{t\in [0,T]} M_s(t)$ satisfies 
$\widetilde{M}_s(T)\le M_s(0)+TC(\widetilde{M}_s(T))\widetilde{M}_s(T)^2$. 
This in turn implies that, if the initial data are of size $\eps$, namely if 
$M_s(0)=O(\eps^2)$ (notice that $M_s$ is linked to the square of the Sobolev norms)  
for some $s$ large enough, then the Cauchy problem is well-posed on a time 
interval of size $\eps^{-2}$ (see also the results in Totz and Wu~\cite{TW}).

Another important property is that 
the estimate \e{i4} is tame, which means that 
it is linear in the Sobolev norm ($\gamma$ is a fixed large enough number 
which might be much smaller than $s$). 
Eventually, let us notice that it would have been easier to obtain \e{i4} 
with $N_\gamma$ 
replaced by $N_\gamma(t)+\lA \mathcal{H}\eta(t,\cdot)\rA_{\eC{\gamma}}+\blA \mathcal{H}\Dxmez\psi(t,\cdot)\brA_{\eC{\gamma-\mez}}$ where $\mathcal{H}$ denotes the Hilbert transform.  
A fortiori, it would have been easier to obtain the previous bound with $N$ 
replaced by $\lA \eta(t,\cdot)\rA_{H^{\gamma}}+\blA \Dxmez\psi(t,\cdot)\brA_{H^{\gamma}}$, that is with H\"older norms 
replaced by Sobolev ones. 
However, the corresponding estimates would not be sufficient to prove global well-posedness in \cite{AlDelMain}. 

The smallness assumption on $N_\gamma$ enters essentially only for the following reason: we shall obtain 
$M_s$ as the square of the $H^s$-norm of some functions deduced from $\eta$ and $\psi$ by a nonlinear change of unknowns. If $N_\gamma$ is small enough, then this nonlinear change 
of unknowns is close to the identity. This is used to prove \e{i3}. 
 
The estimate~\e{i4} 
will be proved in Chapter~\ref{S:22} (in fact we shall prove an equivalent statement where 
the right-hand side of \e{i3} is replaced by $\lA \eta\rA_{H^s}+\blA \Dxmez\omega\brA_{H^s}$ where 
$\omega$ is defined in the next section of this introduction). 
To prove global well-posedness in \cite{AlDelMain}, our 
approach follows a variant of the vector fields method introduced by Klainerman in \cite{Kl,Kl2}. In particular, in this paper we shall 
not only study Sobolev estimates, that is $L^2$-estimates for derivatives 
$\px^\alpha$, but also $L^2$-estimates 
for $\px^\alpha Z^\beta$ where $Z=t\partial_t +2x\px$. This is 
the most difficult task of this work which 
will be achieved in Chapters~\ref{S:23} and \ref{S:24}. 
 
The vector field $Z$ appears for the following reason. If $(\eta,\psi)$ solves \e{intro:3}, then 
$$
\eta_\lambda(t,x)=\lambda^{-2}\eta\left(\lambda t, \lambda^2 x\right),\quad \psi_\lambda(t,x)=\lambda^{-3}\psi\left(\lambda t,\lambda^2 x\right)
\qquad (\lambda>0)
$$
are also solutions of the same equations. Now observe that 
for any function $C^1$ function $u$, there holds
$Z u (t,x)=\frac{d}{d\lambda}u(\lambda t,\lambda^2 x)\big\arrowvert_{\lambda=1}$. 
In particular, if $u$ solves the linearized 
water waves equation around the null solution, 
that is $\partial_t^2 u+\Dx u=0$, then so does $Zu$. 
This vector field already played an essential role in the above mentioned papers 
of Wu~\cite{Wu09} and Germain-Masmoudi-Shatah~\cite{GMS2}. 
We also refer the reader to Hur~\cite{Hur2} where a similar vector field is used to 
study the smoothing effect of surface tension.

Let us mention that the paper is self-contained.
We shall give simplified statements of our results in this introduction and refer the reader to the next chapters for precise statements. 
Let us also mention that Ionescu and Pusateri \cite{IonescuPusateri} have obtained independently a similar  global existence result to the one proved in \cite{AlDelMain}, under weaker decay assumption for the Cauchy data, and obtained an asymptotic description of the solutions in frequency variables.

\section{Properties of the Dirichlet-Neumann operator}

A notable part of the analysis consists in proving several estimates for the Dirichlet-Neumann operator. 
We present here some of the results on this topic which are proved in Chapter~\ref{chap:1} and in Chapter~\ref{S:21}. 

$\bullet$ \textbf{Definition of the Dirichlet-Neumann operator}

Let $\eta\colon \xR\rightarrow\xR$ be a smooth enough function and consider the 
open set
$$
\Omega\defn \{\,(x,y)\in\xR^2 \,;\, y<\eta (x)\,\}.
$$
It $\psi\colon\xR\rightarrow\xR$ is another function, and if we call 
$\phi\colon\Omega\rightarrow \xR$ the unique solution of 
$\Delta_{x,y}\phi=0$ in~$\Omega$ satisfying 
$\phi\arrowvert_{y=\eta(x)}=\psi$ and a convenient vanishing condition at $y\rightarrow -\infty$, 
one defines 
the Dirichlet-Neumann operator 
$G(\eta)$ by 
$G(\eta)\psi   =
\sqrt{1+(\partial_x\eta)^2}\,
\partial _n \phi\arrowvert_{y=\eta}$, 
where~$\partial_n$ is the outward normal derivative on~$\partial\Omega$. 
In Chapter~\ref{chap:1} we 
make precise the above definition and 
study the action of $G(\eta)$ on different spaces. 
In this outline we consider only the case where 
$\psi$ belongs to the homogeneous space 
$\dot{H}^{1/2}(\xR)$ or to the 
H\"older space 
$\eC{\gamma}(\xR)$ of order $\gamma\in [0,+\infty[$. 
(We refer to Chapter~\ref{chap:1} for the definition of these spaces and of the Sobolev 
or H\"older norms used below.)

\begin{prop*}Let $\gamma$ be a real number, $\gamma>2$, $\gamma\not\in\mez\xN$. Let $\eta$ be in $L^2\cap \eC{\gamma}(\xR)$ 
satisfying the condition 
\be\label{i20}
\etapetitgamma+\lA \eta'\rA_{\eC{-1}}^{1/2}\lA \eta'\rA_{H^{-1}}^{1/2}
<\delta.
\ee
Then 
$G(\eta)$ is well-defined and 
bounded from $\dot{H}^{1/2}(\xR)$ 
to $\dot{H}^{-1/2}(\xR)$ and satisfied an estimate
$$
\lA G(\eta)\psi\rA_{\dot{H}^{-1/2}}\le 
\Cetagamma\Dxmezpsi.
$$
Moreover, $G(\eta)$ satisfies when $\psi$ is in 
$\eC{\gamma}(\xR)$
\be\label{i22}
\lA G(\eta)\psi\rA_{\eC{\gamma-1}}\le 
\Cetagamma \blA \Dxmez \psi\brA_{\eC{\gamma-\mez}},
\ee
where $C(\cdot)$ is a non decreasing continuous function of its argument.
\end{prop*}

\begin{rema*}
Many results are known for the Dirichlet-Neumann operator (see for instance \cite{BGSW,CSS,LannesLivre} for results related to the analysis of water waves). 
The only novelty in the results proved in Chapter~\ref{chap:1} is that we shall consider 
more generally the case where $\psi$ belongs either to an homogeneous Sobolev space of order 
greater than $1/2$ or to an homogeneous 
H\"older spaces. 
As a corollary, notice that if we define $G_{1/2}(\eta)=\Dx^{-\mez}G(\eta)$, we 
obtain a bounded operator from $\dot{H}^{1/2}(\xR)$ to $L^2(\xR)$ satisfying 
$$
\lA G_{1/2}(\eta)\psi\rA_{L^{2}}\le 
\Cetagamma \Dxmezpsi.
$$
If we assume moreover that for some $0<\theta'<\theta<\mez$, 
$\blA \eta'\brA_{H^{-1}}^{1-2\theta'}\blA \eta'\brA_{\eC{-1}}^{2\theta'}$ is bounded, then we prove that, similarly, 
$\Dx^{-\mez+\theta}G(\eta)$ satisfies
$$
\blA \Dx^{-\mez+\theta}G(\eta)\psi\brA_{\eC{\gamma-\mez-\theta}}
\le C\bigl( \blA \eta'\brA_{\eC{\gamma-1}}\bigr) \blA \Dxmez\psi\brA_{\eC{\gamma-\mez}}.
$$
\end{rema*}

Hereafter, $\gamma$ always denote a real number such that $\gamma>2$ and 
$\gamma\not\in\mez\xN$. It is always assumed that the condition~\eqref{i20} holds 
for some small enough $\delta$. 

Let us introduce two functions that play a key role. 
Since $\h{-\mez}(\xR)\subset H^{-\mez}(\xR)$ and since  
$\eC{\gamma-1}(\xR)\cdot H^{-\mez}(\xR)\subset H^{-\mez}(\xR)$ for $\gamma>3/2$, 
the following functions are 
well-defined
\be\label{i24}
\B=\frac{G(\eta)\psi+(\partial_x\eta)( \partial_x\psi)}{1+(\partial_x\eta)^2},\quad
V=\partial_x\psi-\B\partial_x\eta.
\ee
These functions appear since one has 
$\B=(\partial_y \phi)\arrowvert_{\partial\Omega}$ 
and $V=(\px\phi)\arrowvert_{\partial\Omega}$, 
so that $\B$ (resp.\ $V$) 
is the trace of the vertical (resp.\ horizontal) component of the velocity at the free surface. 

$\bullet$ \textbf{Tame estimate for the Dirichlet-Neumann operator}

If $\eta\in C^\infty_b$, it is known since Calder\'on that $G(\eta)$ is a 
pseudo-differential operator of order~$1$ (see~\cite{SU,Taylor1,Treves}). 
This is true in any dimension. In dimension one, this result simplifies to
\be\label{i25}
G(\eta)\psi=\la D_x\ra \psi ~+~R(\eta)\psi,
\ee
where $R(\eta)f$ is a smoothing operator, 
bounded from $H^\mu$ to $H^{\mu+m}$ for any integer $m$. Namely, 
\be\label{i26}
\forall m\in \xN,~\exists K\ge 1,~\forall \mu\ge \mez, 
\quad \lA R_0(\eta)\psi\rA_{H^{\mu+m}}\le C\left( \lA \eta\rA_{H^{\mu+K}}\right)
 \lA \eta\rA_{H^{\mu+K}}
\lA \psi\rA_{H^\mu}.
\ee
Several results are known when $\eta$ is not smooth. 
Expressing $G(\eta)$ as a singular integral operator, it was proved
by Craig, Schanz and Sulem~\cite{CSS} that 
if $\eta$ is in $C^{k+1}$ and $\psi$ is in $H^{k+1}$ for some integer $k$, then 
$G(\eta)\psi$ belongs to $H^{k}$. 
Moreover, it was proved by Lannes~\cite{LannesJAMS} 
that when $\eta$ is a function with
limited smoothness, 
then $G(\eta)$ is a pseudo-differential operator
with symbol of limited regularity. This implies that if 
$\eta$ is in $H^{s}$ and $\psi$ is in $H^{s}$ for some $s$ large enough,
then $G(\eta)\psi$ belongs to $H^{s-1}$ (which was first established by 
Craig and Nicholls~\cite{CN} and Wu~\cite{WuInvent, WuJAMS} by different methods). 
We refer to \cite{ABZ3,ABZ1,SZ,SZ2} for results in rough domains.

We shall prove in Chapter~\ref{S:21} an estimate which complements the estimate \e{i26} in two directions. Firstly, notice that, 
for the analysis of the water waves equations, $\eta$ and 
$\psi$ are expected to have essentially the same regularity so that the constant $K$ corresponds to a loss of derivatives. 
We shall prove an estimate without loss 
of derivatives. In addition, we shall prove a tame estimate (which means an estimate linear with respect to the highest order norms). 

\begin{prop*}[Tame estimate for the Dirichlet-Neumann 
operator]
Let~$(s,\gamma)\in \xR^2$ be such that
$$ 
s-\mez>\gamma>3,\quad \gamma\not\in\mez\xN.
$$
Then, for all $(\eta,\psi)$ in $H^{s}(\xR)\times H^{s}(\xR)$ such that that the condition~\eqref{i20} holds, $G(\eta)\psi$ belongs to 
$H^{s-1}(\xR)$ and there exists a non decreasing function~$C\colon \xR\rightarrow \xR$ such that
\begin{multline}\label{i27}
\lA G(\eta)\psi-\Dx\psi\rA_{H^{s-1}}\\
\le C\left(\lA \eta\rA_{\eC{\gamma}}\right)\left\{
\dalpha \lA \eta\rA_{H^s}+\lA \eta\rA_{\eC{\gamma}}\blA \Dxmez \psi \brA_{H^{s-\mez}}\right\}.
\end{multline}
\end{prop*}
\begin{rema*}
It follows from \e{i27} and the triangle inequality that 
\be\label{i28}
\lA G(\eta)\psi\rA_{H^{s-1}}
\le C\left(\lA \eta\rA_{\eC{\gamma}}\right)\left\{
\dalpha \lA \eta\rA_{H^s}+\blA \Dxmez \psi \brA_{H^{s-\mez}}\right\}.
\ee
Other tame estimates, with H\"older norms replaced by Sobolev norms $H^{s_0}$ for some 
fixed real number $s_0$, have been proved in \cite{LannesJAMS} (see also~\cite{ABZ5}).
\end{rema*}

$\bullet$ \textbf{Paraproducts}

The proof of the previous proposition, as well as the proof of most of the following results, are based on 
paradifferential calculus. The results needed in this paper are recorded in Appendix~\ref{s2}. 
To make this introduction self-contained, we recall here the definition 
of paraproducts.

Consider a cut-off function $\theta$ in $C^\infty(\xR\times\xR)$ such that
$$
\theta(\xip,\xii)=1 \quad \text{if}\quad \la\xip\ra\le \eps_1\la \xii\ra,\qquad 
\theta(\xip,\xii)=0 \quad \text{if}\quad \la\xip\ra\geq \eps_2\la\xii\ra,
$$
with $0<\eps_1<\eps_2<1$. Given two functions $a=a(x)$ and $b=b(x)$ one writes
$$
ab=\frac{1}{(2\pi)^2}\iint e^{ix(\xip+\xii)}\widehat{a}(\xip)\widehat{b}(\xii)\, d\xip \, d\xii
=T_ab+T_ba +\RBony(a,b)
$$
where
\begin{align*}
T_a b&=\frac{1}{(2\pi)^2}\iint e^{ix(\xip+\xii)}\theta(\xip,\xii)\widehat{a}(\xip)\widehat{b}(\xii)\, d\xip \, d\xii,\\
T_b a&=\frac{1}{(2\pi)^2}\iint e^{ix(\xip+\xii)}\theta(\xii,\xip)\widehat{a}(\xip)\widehat{b}(\xii)\, d\xip \, d\xii,\\
\RBony(a,b)&=\frac{1}{(2\pi)^2}\iint e^{ix(\xip+\xii)}\bigl(1-\theta(\xip,\xii)-\theta(\xii,\xip)\bigr)\widehat{a}(\xip)\widehat{b}(\xii)\, d\xip \, d\xii.
\end{align*}
Then one says that $T_a b$ and $T_b a$ are paraproducts, while 
$\RBony(a,b)$ is a remainder. The key property is that a paraproduct 
by an $L^\infty$ function acts on any Sobolev spaces $H^s$ with $s$ in~$\xR$. 
The remainder term $\RBony(a,b)$ is smoother than the paraproducts $T_a b$ and $T_b a$ whenever one of the factors belongs to $\eC{\sigma}$ for some $\sigma>0$ (see \e{Bony3} in 
Appendix~\ref{s2}).

$\bullet$ \textbf{The quadratic terms}

We call \e{i27} a linearization 
formula since the right-hand side is quadratic in $(\eta,\psi)$. 
We shall prove much more precise results, with remainders quadratic in $(\eta,\psi)$ and 
estimated not only in 
$H^{s-1}$ but in $H^{s'}$ for some $s'\ge s$. 
To explain this improvement, we begin 
by considering only the linear and quadratic terms in 
$G(\eta)\psi$. Set
$$
G_{\quadratique}(\eta)\psi\defn \la D_x \ra\psi-\Dx(\eta\Dx\psi)-\partial_x(\eta\partial_x\psi).
$$
Then it is known that $G(\eta)\psi-G_{\quadratique}(\eta)\psi$ is cubic in $(\eta,\psi)$ (see~\cite{CSS} or \e{i28b} below).

Now write
$$
\Dx (\eta\Dx\psi)=\Dx\bigl(T_{\eta}\Dx\psi\bigr)+\Dx\bigl(T_{\Dx\psi}\eta\bigr)+\Dx\RBony(\eta,\Dx\psi) 
$$
and perform a similar decomposition of $\px(\eta\px\psi)$. Noticing the following 
cancellation (cf.\ Lemma~\ref{lemm:DxaDx} in Appendix~\ref{s2})  
\be\label{i28c}
\Dx \bigl(T_\eta\Dx \psi\bigr)+\px \bigl(T_\eta\px \psi\bigr)=0,
\ee
we conclude that
$$
G_{\quadratique}(\eta)\psi=\la D_x \ra\psi-\Dx\bigl(T_{\Dx\psi}\eta\bigr)
-\partial_x\bigl(T_{\partial_x\psi}\eta\bigr)
-\Dx \RBony(\eta,\Dx\psi)
-\partial_x\RBony(\eta,\partial_x\psi).
$$
The previous identity is better written under the form
\be\label{i28d}
G_{\quadratique}(\eta)\psi
=\Dx \bigl(\psi-T_{\Dx\psi}\eta\bigr)-\partial_x\bigl(T_{\partial_x\psi}\eta\bigr)+F_{\quadratique}(\eta)\psi,
\ee
where $F_{\quadratique}(\eta)\psi=-\Dx \RBony(\eta,\Dx\psi)
-\partial_x\RBony(\eta,\partial_x\psi)$. 
Assuming $s+\gamma> 1$, it follows from standard results (see \e{Bony3} in 
Appendix~\ref{s2}) that $F_{\quadratique}(\eta)$ is a smoothing operator:
\be\label{i28e}
\lA F_{\quadratique}(\eta)\psi\rA_{H^{s+\gamma-2}}\\
\le 
K\lA \eta\rA_{\eC{\gamma}}\blA \Dxmez \psi \brA_{H^{s-\mez}}.
\ee

$\bullet$ \textbf{The good unknown of Alinhac}

In the previous paragraph, we considered only the linear and quadratic terms 
$G_{\quadratique}(\eta)\psi$. 
To prove an identity similar to \e{i28d} for $G(\eta)\psi$, 
exploiting a cancellation analogous to \e{i28c}, 
as in \cite{AM,ABZ1}, we shall express the computations in terms of the ``good unknown'' of Alinhac $\omega$ 
defined by
$$
\omega=\psi-T_{\B}\eta
$$
where $\B$ is as given in \e{i24}. As explained in \cite{AM,ABZ1}, 
the idea of introducing $\omega$ is rooted 
in a cancellation first observed by Lannes~\cite{LannesJAMS} for the 
water waves equations linearized around a non trivial solution. Here, we want to 
explain that $\omega$ appears naturally when one introduces the operator of 
paracomposition of Alinhac~\cite{Alipara} associated to the change of variables 
that flattens the boundary $y=\eta(x)$ of the domain. 
This is a quite optimal way of keeping track of the limited 
smoothness of the change of coordinates. Though we shall not use this point of view, 
we explain here the ideas that 
underly the computations that will be made later.

To study the elliptic equation $\Delta_{x,y}\phi=0$ in $\Omega=\{(x,y)\in \xR^2\, 
;\, y<\eta(x)\}$, we shall reduce the problem to the negative half-space through the 
change of coordinates $\kappa\colon (x,z)\mapsto (x,z+\eta(x))$, which sends $\{ (x,z)\in \xR^2\,;\, z<0\}$ 
on $\Omega$. 
Then $\phi(x,y)$ solves $\Delta_{x,y}\phi=0$ if and only if $\varphi=\phi\circ \kappa=\phi(x,z+\eta(x))$ is a solution of 
$P\varphi=0$ in $z<0$, where
\be\label{i29}
P=(1+\eta'^2)\partial_z ^2+\px^2-2\eta'\px\partial_z-\eta''\partial_z 
\ee
(we denote by $\eta'$ the derivative $\px\eta$). The boundary condition 
$\phi\arrowvert_{y=\eta(x)}$ 
becomes $\varphi(x,0)=\psi(x)$ and $G(\eta)$ is given by
$$
G(\eta)\psi=\bigl[ (1+\eta'^2)\partial_z\varphi-\eta'\partial_x \varphi\bigr]\big\arrowvert_{z=0}.
$$

We first explain the main difficulty to handle a diffeomorphism 
with limited regularity. Let us use the notation $D=-i\partial$ and 
introduce the symbol
$$
p(x,\xi,\zeta)=(1+\eta'(x)^2)\zeta^2+\xi^2-2\eta'(x) \xi \zeta
+i\eta''(x)\zeta.
$$
Notice that 
$P=-p(x,D_x,D_z)$. We shall write  
$T_{p}\va$ for 
$T_{1+\eta'(x)^2}D_z^2+D_x^2-2T_{\eta'}D_xD_z+T_{\eta''}\partial_z$. 
Starting from $p(x,D_x,D_z)\varphi=0$, by using standard results for 
paralinearization of products, we find that 
$T_{p}\varphi=f_1$ for some source term $f_1$ which is 
continuous in $z$ with values in $H^{s-2}$ if $\eta$ is in $H^s$ and 
the first and second order derivatives in $x,z$ of $\varphi$ are 
bounded. 
The key point is that 
one can associate to $\kappa$ a paracomposition 
operator, denoted by $\kappa^*$, such that 
$T_{p}( \kappa^* \phi)=f_2$ 
for some smoother remainder term $f_2$. That is 
for some function $f_2$ continuous in $z$ with values in $H^{s+\gamma-4}$, 
if $\eta$ is in $H^s$ and if the derivatives in $x,z$ 
of order 
less than $\gamma$ of $\varphi$ are bounded (the key difference between $f_1$ and $f_2$ is that one cannot improve the 
regularity of $f_1$ by assuming that $\varphi$ is smoother). 

We shall not define $\kappa^*$, instead we recall
the two main properties of paracomposition operators 
(we refer to the original article~\cite{Alipara} for the general theory).
First, 
modulo a smooth remainder, one has
$$
\kappa^*\phi =\phi\circ \kappa-T_{\phi'\circ \kappa}\kappa
$$
where $\phi'$ denotes the differential of $\phi$. 
On the other hand, there is a symbolic calculus formula which allows 
to compute the commutator of $\kappa^*$ to a paradifferential operator. This formula implies that
$$
\kappa^* \Delta -T_{p} \kappa^* 
$$
is a smoothing operator (that is an operator bounded from $H^\mu$ to 
$H^{\mu+m}$ for any real number $\mu$, where 
$m$ is a positive number depending on the regularity of $\kappa$). 
Since $\Delta_{x,y}\phi=0$, this implies that 
$T_{p} \bigl(\phi\circ \kappa-T_{\phi'\circ \kappa}\kappa\bigr)$ is a smooth remainder term as asserted above. 

Now observe that 
$$
\omega=\bigl( \phi\circ \kappa-T_{\phi'\circ \kappa}\kappa\bigr)\big\arrowvert_{z=0}.
$$
This is the reason why the good unknown enters into the analysis. 
The previous argument is the key point to prove the following
\begin{prop*}[Paralinearisation of the 
Dirichlet-Neumann operator]
Define $F(\eta)\psi$ by
\begin{equation*}
G(\eta)\psi=\Dx \omega-\partial_x \bigl( T_{V}\eta\bigr)+F(\eta)\psi.
\end{equation*}
Let~$(s,\gamma)\in \xR^2$ be such that
$$ 
s-\mez>\gamma>3,\quad \gamma\not\in\mez\xN.
$$
For all $(\eta,\psi)$ in $H^{s}(\xR)\times H^{s}(\xR)$ such that that the condition~\eqref{i20} holds,
\be\label{i30}
\lA F(\eta)\psi\rA_{H^{s+\gamma-4}}\\
\le 
C\left(\lA \eta\rA_{\eC{\gamma}}\right)\left\{
\dalpha \lA \eta\rA_{H^s}+\lA \eta\rA_{\eC{\gamma}}\blA \Dxmez \psi \brA_{H^{s-\mez}}\right\}.
\ee
\end{prop*}

Our goal was to explain how to obtain an identity analogous to the identity \e{i28d} obtained by 
considering the linear and quadratic terms in $G(\eta)\psi$. 
To compare \e{i30} and \e{i28e}, notice that, 
from the definition of $\B$ and $V$ (see \e{i24}), $\B-\Dx\psi$ and 
$V-\px\psi$ are quadratic in $(\eta,\psi)$. Therefore, modulo cubic and higher order terms, 
$\Dx \omega-\partial_x \bigl( T_{V}\eta\bigr)$ is given by 
the expression $\Dx \bigl(\psi-T_{\Dx\psi}\eta\bigr)-\partial_x\bigl(T_{\partial_x\psi}\eta\bigr)$ which appears in the right hand side of \e{i28d}. 
We shall compare 
$F(\eta)\psi$ and $F_{\quadratique}(\eta)\psi$ in the next paragraph. 

The main interest of this proposition will be explained in the next section. 
At this point, we want to show that this estimate implies the tame estimate~\e{i27}. 
To do so, write the remainder $R(\eta)\psi$ in 
\e{i25} as $R(\eta)\psi=-\Dx \bigl( T_{\B}\eta\bigr)
-\partial_x \bigl( T_{V}\eta\bigr)+F(\eta)\psi$ since 
$\Dx \omega-\Dx\psi=-\Dx \bigl( T_{\B}\eta\bigr)$. 
The key point is that $(\eta,\psi)\rightarrow F(\eta)\psi$ is smoothing, with respect to both arguments, while the two other factors are operators of order $1$ acting on $\eta$. Indeed, 
as a paraproduct with an $L^\infty$ function acts on any 
Sobolev spaces, one has
\begin{align*}
&\lA \partial_x \bigl( T_{V}\eta\bigr)\rA_{H^{s-1}}\le K\lA V\rA_{L^\infty}\lA\eta\rA_{H^s},\\
&\lA \Dx \omega-\Dx\psi\rA_{H^{s-1}}=
\lA \Dx \bigl( T_{\B}\eta\bigr)\rA_{H^{s-1}}\le K\lA \B\rA_{L^\infty}\lA\eta\rA_{H^s}.
\end{align*}
On the other hand, directly from the definition~\e{i24} of $\B$, we deduce 
that 
$$
\lA \B\rA_{L^\infty}\le \lA G(\eta)\psi\rA_{L^{\infty}}+\lA \px\eta\rA_{L^\infty}
\lA \px\psi\rA_{L^\infty}. 
$$
Now the estimate \e{i22}Ê
implies that the right-hand side of the above inequality is bounded by $\Cetagamma \blA \Dxmez \psi\brA_{\eC{\gamma-\mez}}$. Writing 
$V=\px\psi-\B\px \eta$, we obtain the same estimate for the 
$L^\infty$-norm of $V$. This proves that \e{i30} implies \e{i27} (and hence \e{i28}).

$\bullet$ \textbf{Taylor expansions of the Dirichlet-Neumann operator}

Consider the Taylor expansion of the Dirichlet-Neumann operator~$G(\eta)$ 
as a function of~$\eta$, when $\eta$ goes to zero. 
Craig, Schanz and Sulem (see \cite{CSS} and~\cite[Chapter~$11$]{SuSu}) have shown that one can expand $G(\eta)$ as a sum of pseudo-differential operators and gave 
precise estimates for the remainders. 
Tame estimates are proved in \cite{CSS} and \cite{ASL,IP}.  
We shall complement these results by proving sharp 
tame estimates tailored to our purposes. 

\begin{prop*}
Assume that 
$$
s-1/2>\gamma\ge 14,\quad s\ge \mu\ge 5,\quad \gamma\not\in \mez \xN,
$$ 
and 
consider~$(\eta,\psi)\in H^{s+\mez}(\xR)\times(\eC{\gamma}(\xR)\cap H^{\mu+\mez}(\xR))$  
such that the condition~\eqref{i20} holds. Then 
there exists a non decreasing function~$C\colon \xR\rightarrow \xR$ such that, 
\begin{multline}\label{i37}
\lA F(\eta)\psi-F_{\quadratique}(\eta)\psi\rA_{H^{\mu+1}}\\
\le C(\lA \eta\rA_{\eC{\gamma}} )
\lA \eta\rA_{\eC{\gamma}} \Bigl\{\dalpha \lA \eta \rA_{H^s}+\lA \eta \rA_{\eC{\gamma}}
\blA \Dxmez\psi\brA_{H^\mu}\Bigr\},
\end{multline}
where recall that $F_{\quadratique}(\eta)\psi=-\Dx \RBony(\eta,\Dx\psi)
-\partial_x\RBony(\eta,\partial_x\psi)$.
\end{prop*}
Notice that the right-hand side is cubic in $(\eta,\psi)$ and that $F(\eta)-F_{\quadratique}(\eta)$ 
is a smoothing operator, bounded from $H^{\mu+\mez}$ to $H^{\mu+1}$ (in fact it is a smoothing operator 
of any order, assuming that $\gamma$ is large enough).

Let us prove that this estimate allows to recover an estimate for 
the difference of $G(\eta)\psi$ and its quadratic part $G_{\quadratique}(\eta)\psi$ introduced above. 
By definition of $F(\eta)\psi$ and $F_{\quadratique}(\eta)\psi$, one has
\begin{align*}
G(\eta)\psi&=\Dx\bigl( \psi-T_{B}\eta\bigr)-\partial_x \bigl( T_{V}\eta\bigr)+F(\eta)\psi,\\
G_{\quadratique}(\eta)\psi
&=\Dx \bigl(\psi-T_{\Dx\psi}\eta\bigr)-\partial_x\bigl(T_{\partial_x\psi}\eta\bigr)+F_{\quadratique}(\eta)\psi.
\end{align*}
Substracting these two expressions one obtains 
$$
G(\eta)\psi-G_{\quadratique}(\eta)\psi=-\Dx \bigl(T_{\B-\Dx\psi}\eta\bigr) -\px \bigl(T_{V-\px\psi}\eta \bigr)
+F(\eta)\psi-F_{\quadratique}(\eta)\psi.
$$
Noticing that the $L^\infty$-norms of 
$\B-\Dx\psi$ is bounded by $C\left(\lA \eta\rA_{\eC{\gamma}}\right)\lA \eta\rA_{\eC{\gamma}}\dalpha$, 
together with a similar estimate for the $L^\infty$-norm of $V-\px\psi$, and 
repeating arguments similar to those used in the previous paragraph, one finds that
\begin{multline}\label{i28b}
\lA G(\eta)\psi-G_{\quadratique}(\eta)\psi\rA_{H^{s-1}}\\
\le C\left(\lA \eta\rA_{\eC{\gamma}}\right)\lA \eta\rA_{\eC{\gamma}}\left\{
\dalpha \lA \eta\rA_{H^s}+\lA \eta\rA_{\eC{\gamma}}\blA \Dxmez \psi \brA_{H^{s-\mez}}\right\},
\end{multline}
for any $s\ge \gamma+1/2$, provided that $\gamma$ is large enough.  

On the other hand, 
we shall also need to study the case where $(\eta,\psi)\inÊ\eC{\gamma}\times H^\mu$ with 
$\gamma$ larger than $\mu$. Then we shall prove that 
$G(\eta)-\Dx$ and $G(\eta)-G_{\quadratique}(\eta)$ are smoothing operators, satisfying
\begin{align*}
&\lA G(\eta)\psi-\Dx\psi\rA_{H^{\gamma-3}}
\le C\left( \lA \eta\rA_{\eC{\gamma}}\right)\lA \eta\rA_{\eC{\gamma}}
\blA \Dxmez \psi \brA_{L^2},\\
&\lA G(\eta)\psi-G_{\quadratique}(\eta)\psi\rA_{H^{\gamma-4}}
\le C\left( \lA \eta\rA_{\eC{\gamma}}\right)\lA \eta\rA_{\eC{\gamma}}^2
\blA \Dxmez \psi \brA_{H^1}.
\end{align*}
\section{Paradifferential normal forms method}\label{S:I3}
The main goal of this paper is to prove that, 
given an {\em a priori} bound of some H\"older norm of 
$Z^{k'} (\eta+i\Dxmez\psi)$ for $k'\le s/2+k_0$,
we have an {\em a priori} estimate of some Sobolev norms of $Z^{k} (\eta+i\Dxmez \omega)$ for $k\le s$, 
where recall that 
$\omega=\psi-T_{\B(\eta)\psi}\eta$. 
The proof is by induction on $k\ge 0$. Each step is divided into two parts:
\begin{enumerate}
\item Quadratic approximations: in this step we paralinearize and symmetrize the equations. In addition, we 
identify the principal and subprincipal terms 
in the analysis of both the regularity and the homogeneity. 
\item Normal form: in this step we use a bilinear normal form transformation to compensate for the quadratic terms in the energy estimates.  
\end{enumerate}

For the sake of clarity, we begin by considering the case $k=0$. 
Our goal is to explain the proof of \e{i3}Ê
and \e{i4}. 

$\bullet$ \textbf{Quadratic and cubic terms in the equations}

The previous analysis of $G(\eta)\psi$ allows us to rewrite the first equation of 
\e{intro:3} as
$$
\partial_t \eta+\px \bigl(T_V \eta\bigr)- \Dx\omega =F(\eta)\psi.
$$
It turns out that it is much simpler to analyze the second equation of \e{intro:3}: 
expressing the computations in terms of the good unknown $\omega$, it is found that
$$
\partial_t \omega +T_V \partial_x \omega +T_{\ma}\eta=f,
$$
where $\ma$ is the Taylor coefficient and $f$ is a smoothing remainder 
\begin{align*}
f&=(T_{V}  T_{\partial_x\eta}-T_{V \partial_x\eta})\B
+(T_{V \partial_x\B}-T_{V}  T_{\partial_x\B})\eta\\
&\quad +\mez \RBony(\B,\B)-\mez\RBony(V,V)+T_V\RBony(\B,\partial_x\eta)
-\RBony(\B,V\px\eta)
\end{align*}
(the last four terms are remainders in the paralinearization of a products while 
the first two terms are estimated by symbolic calculus, see \e{esti:quant2-func}). 

It is convenient to symmetrize these equations by making act 
$T_{\sqrt{\ma}}$ (resp.\ $\Dxmez$) on the first (resp.\ second) equation. 
Set
$$
\vU
=\begin{pmatrix} T_{\sqrt{\ma}}\eta \\ \Dxmez \omega
\end{pmatrix}.
$$

We can now state the main consequence of the results given in the previous section. 

\begin{prop*}
The water waves system can be written under the form
\begin{equation}\label{i40}
\partial_t \vU +D\vU +Q(\vu)\vU+S(\vu)\vU+C(\vu)\vU=G,
\end{equation}
where $D=\begin{pmatrix} 0 & -\Dxmez \\ \Dxmez & 0\end{pmatrix}$, 
$\vu =\begin{pmatrix} \eta \\ \Dxmez\psi\end{pmatrix}$,
$Q(\vu)\vU$ and $S(\vu)\vU$ (resp.\ $C(\vu)\vU$ and $G$) are 
quadratic (resp.\ cubic terms). Moreover there exists $\rho>0$ such that, for $s$ large enough,
\begin{align*}
&\lA Q(\vu)\vU\rA_{H^{s-1}}
\le K \lA \vu\rA_{C^\rho}\lA \vU\rA_{H^s}, \\
&\lA S(\vu)\vU\rA_{H^{s+1}}
\le K\lA \vu\rA_{C^\rho}\lA \vU\rA_{H^s},\\
&\lA C(\vu)\vU\rA_{H^{s-1}}
\le C(\lA \vu\rA_{C^\rho})\lA \vu\rA_{C^\rho}^2\lA \vU\rA_{H^s},\\
&\lA G\rA_{H^s}\le C(\lA \vu\rA_{C^\rho})\lA \vu\rA_{C^\rho}^2\lA \vU\rA_{H^s}.
\end{align*}
\end{prop*}
\begin{rema*}
$i)$ The operators $Q(\vu)$, $S(\vu)$ and $C(\vu)$ are explicitly given in the proof. 
The previous estimates mean that 
$\vU\mapsto Q(\vu)\vU$ and $\vU\mapsto C(\vu)\vU$ (resp.\ 
$\vU\mapsto S(\vu)\vU$)
are linear operators of order $1$ (resp.\ $-1$) with tame 
dependence on $\vu$. 

$ii)$
For $\lA \eta\rA_{\eC{\gamma}}$ small enough, $\psi\rightarrow \psi-T_{\B(\eta)\psi}\eta$ is an isomorphism from $\eC{\gamma}$ to itself. Then one could write \e{i30} in terms of $\vU$ only. However, it is convenient to introduce $\vu$ because the H\"older bounds are most naturally proved for $\vu$ (see~\cite{AlDelMain} for these estimates). 
\end{rema*}

$\bullet$ 
\textbf{Quadratic normal form: strategy of the proof}

Recall that
$$
\vu=\begin{pmatrix} \eta \\ \Dxmez\psi\end{pmatrix},\quad 
\vU
=\begin{pmatrix} T_{\sqrt{\ma}}\eta \\ \Dxmez \omega
\end{pmatrix}.
$$
We want to implement the normal form approach by introducing a quadratic perturbation of 
$U$ of the form
$$
\Phi=\vU+E(\vu)\vU,
$$
where~$(\vu,\vU)\mapsto E(\vu)\vU$ is bilinear and chosen in such a way that the equation on~$\Phi$ is of the form
$$
\partial_t \Phi +D \Phi =N_{\tri}(\Phi),
$$
where~$N_{\tri}(\Phi)$ consists of cubic and higher order terms. 
To compute the equation satisfied by~$\Phi$, write
$$
\partial_t \Phi=\partial_t \vU +E(\partial_t \vu)\vU+E(\vu)\partial_t \vU.
$$
Hence, by replacing~$\partial_t \vU$ by~$-D\vU-(Q(\vu)+S(\vu))\vU$, we obtain that modulo cubic terms,
\begin{align*}
\partial_t \Phi&=-D\vU-(Q(\vu)+S(\vu))\vU-E(D\vu)\vU-E(\vu)D\vU
\\
&=-D\Phi+DE(\vu)\vU-(Q(\vu)+S(\vu))\vU
-E(D\vu)\vU-E(\vu)D\vU.
\end{align*}
It is thus tempting to seek~$E$ under the form $E=E_1+E_2$ such that 
\begin{align}
Q(\vu)U+E_1(D\vu)U+E_1(\vu)DU&=DE_1(\vu)U,\label{i41}\\
S(\vu)U+E_2(D\vu)U+E_2(\vu)DU&=DE_2(\vu)U.\label{i42}
\end{align}
However, one cannot solve these two equations directly for two different reasons. 
The equation \eqref{i41} 
leads to a loss of derivative: for a general~$\vu\in H^\infty$ and~$s\ge 0$, 
it is not possible to eliminate the quadratic terms~$Q(\vu)U$ by means of a bilinear Fourier multiplier 
$E_1$ such that~$U\mapsto E_1(\vu)\vU$ is bounded from~$H^s$ to~$H^s$. 
Instead we shall 
add other quadratic terms to the equation to 
compensate the worst terms. 
More precisely, our strategy consists in seeking a bounded bilinear Fourier multiplier~$\tilde{E_1}$ 
(such that~$\vU\mapsto \tilde{E_1}(\vu)\vU$ is bounded from~$H^s$ to~$H^s$) 
such that the operator $B_1(\vu)$ given by
\begin{equation}\label{i43}
B_1(\vu)\vU\defn D\tilde{E_1}(\vu)\vU-\tilde{E_1}(D\vu)\vU-\tilde{E_1}(\vu)D\vU,
\end{equation}
satisfies
\begin{equation*}
\RE\langle Q(\vu)\vU-B_1(\vu)\vU,\vU\rangle_{H^s\times H^s}=0.
\end{equation*}
The key point is that one can find~$B_1(\vu)$ such that~$\vU\mapsto B_1(\vu)\vU$ 
is bounded from~$H^s$ to~$H^s$. 
This follows from the fact that, 
while~$\vU\mapsto Q(\vu)\vU$ is an operator of order~$1$, the operator 
$Q(\vu)+Q(\vu)^*$ is an operator of order~$0$. 
Once~$B_1$ is so determined, we find a bounded bilinear 
transformation~$\tilde{E_1}$ such that \eqref{i43} is satisfied. 
We here use the fact that~$Q$ is a paradifferential operator 
so that one has some restrictions on the support of the symbols. 

The problem~\eqref{i42} leads to another technical issue. 
If one computes the bilinear Fourier multiplier $E_2(\vu)\vU$ 
which satisfies \eqref{i42} 
then one finds a bilinear Fourier multiplier 
$E_2$ such that~$U\mapsto E_2(\vu)\vU$ 
is bounded from~$H^s$ to~$H^s$, but whose operator 
norm satisfies only
$$
\lA E_2(\vu)\rA_{\Fl{H^s}{H^s}}\le 
K \lA \vu\rA_{\eC{\varrho}}+K\lA \mathcal{H}\vu \rA_{\eC{\varrho}},
$$
where $\mathcal{H}$ denotes the Hilbert transform. 
The problem is that, in general, $\lA \mathcal{H}\vu \rA_{\eC{\varrho}}$ 
is not controlled by 
$ \lA \vu\rA_{\eC{\varrho}}$. 
Again to circumvent this problem, instead of solving \eqref{i42}, 
we solve 
\begin{equation*}
B_2(\vu)\vU\defn D\tilde{E_2}(\vu)\vU-\tilde{E_2}(D\vu)\vU-\tilde{E_2}(\vu)D\vU,
\end{equation*}
where $B_2(\vu)$ satisfies
\begin{equation}\label{i44}
\RE\langle S(\vu)\vU-B_2(\vu)\vU,\vU\rangle_{H^s\times H^s}=0.
\end{equation}
The key point is that one can find~$B_2(\vu)$ such that the solution $\tilde{E}_2(\vu)$ to 
\eqref{i44} satisfies
$$
\blA \tilde{E}_2(\vu)\brA_{\Fl{H^s}{H^s}}\le K \lA \vu\rA_{\eC{\varrho}}.
$$

$\bullet$ \textbf{Paradifferential operators}

According to the previous discussion, we shall have to consider the equation
\begin{equation}\label{i45}
E(D\vu)\vU+E(\vu)D\vU-D \bigl[E(\vu)\vU\bigr]=\Pi (\vu)\vU,
\end{equation}
where~$(\vu,\vU)\mapsto E(\vu)\vU$ and~$(\vu,\vU)\mapsto \Pi(\vu)\vU$ 
are bilinear operators of the form
\begin{align}
E(\vu)\vU&=\sum_{1\le k\le 2}\frac{1}{(2\pi)^2}\int
e^{ix(\xip+\xii)} \widehat{\vu^k}(\xip)A^k(\xip,\xii)\widehat{\vU}(\xii) \,d\xip \, d\xii,\notag \\
\Pi(\vu)\vU&=\sum_{1\le k\le 2}\frac{1}{(2\pi)^2}\int
e^{ix(\xip+\xii)} \widehat{\vu^k}(\xip)M^k(\xip,\xii)\widehat{\vU}(\xii) \,d\xip \, d\xii,\label{i45b}
\end{align}
where $A^k$ and $M^k$ are $2\times 2$ matrices of symbols. 
We shall consider the problem~\eqref{i45} in two different cases according to the 
frequency interactions which are permitted in $E(\vu)\vU$ and $\Pi(\vu)\vU$. 
These cases are the following:
\begin{enumerate}[(i)]
\item The case where $\Pi(\vu)U$ is a low-high 
paraproduct, which means that 
there exists a constant $c\in \pol 0,1/2\por$ such that
$$
\supp M^k
\subset \Bigl\{ (\xip,\xii)\in\xR^2\, : \, \la\xii\ra\ge 1,~\la \xip\ra\le c \la\xii\ra\Bigr\}.
$$
The operator $Q(\vu)$ and its real part are of this type. 

\item The case where $\Pi(\vu)\vU$ is a high-high paraproduct 
which means that there exists a constant $C>0$ such that
$$
\supp M^k\subset \Bigl\{ (\xip,\xii)\in\xR^2\,:\, 
\la \xip+\xii\ra \le C ( 1+\min (\la \xip\ra,\la\xii\ra))\Bigr\}.
$$
This spectral assumption is satisfied by $S(\vu)$ and its real part.
\end{enumerate}
That one can reduce the analysis to considering such paradifferential operators is the key point to 
prove tame estimates. This allows us to prove the following result. 

\begin{prop*}
There exist $\gamma>0$ and a bilinear mapping $(\vu,\vU)\mapsto E(\vu)\vU$ satisfying, for any 
real number $\mu$ in $[-1,+\infty[$, 
\be\label{i47}
\lA E(\vu)f\rA_{H^\mu}\le K \lA \vu\rA_{\eC{3}}\lA f\rA_{H^\mu}
\ee
such that 
$\dot{\Phi}=(\id-\Delta)^{s/2}\big(\vU+E(\vu)\vU\big)$ (with $s$ large enough) satisfies 
$$
\partial_t\dot{\Phi}+D\dot{\Phi}+L(\vu)\dot{\Phi}
+C(\Ur)\dot{\Phi}=\Gamma
$$
where the operators $D$ and $C(\vu)$ are as in \e{i40}, the source term satisfies $$
\lA \Gamma\rA_{L^2}
\le C(\lA \vu\rA_{\eC{\gamma}})\lA \vu\rA_{\eC{\gamma}}^2\blA \dot{\Phi}\brA_{L^2}
$$ 
and
\be\label{i48}
\RE \langle L(\vu)\dot{\Phi},\dot{\Phi}\rangle =0
\ee
where $\langle \cdot,\cdot\rangle$ denotes the $L^2$-scalar product.
\end{prop*}

The proof of this proposition follows immediately from the analysis 
in Section~\ref{S:227}. We describe now how one proves 
the estimates \e{i3} and \e{i4}. Setting
$$
M_s(t)=\blA \dot{\Phi}(t,\cdot)\brA_{L^2}^2=\lA \vU+E(\vu)\vU\rA_{H^s}^2 ,
$$ 
the estimate \e{i4} follows from an $L^2$-estimate (the key point is that 
the quadratic terms $L(\vu)\dot{\Phi}$ do not contribute to the energy estimate in view 
of \e{i48}). Also \e{i3}Ê
follows from \e{i47} assuming that $\lA \vu\rA_{\eC{3}}$ is small enough (
to compare the right-hand side of \e{i3} with $M_s$, one has also to compare 
$\blA (\nabla_{x,y}\phi)\arrowvert_{y=\eta}\brA_{H^{s-\mez}}$ and $\blA \Dxmez\omega\brA_{H^s}$; 
this will be done in Chapter~\ref{S:21}).

\section{Iterated vector fields Z}


We describe now how one gets $L^2$-estimates similar to those of the preceding section when one 
makes act iterates of the Klainerman vector field $Z=t\partial_t+2x\px$ on 
$(\eta,\Dxmez\omega)$.

We fix real numbers $a$ and $\gamma$ with 
$\gamma\not\in\mez\xN$ and $a\gg \gamma\gg 1$. 
Given these two numbers, we fix three integers $s,s_0,s_1$ in $\xN$ such that
$$
s-a\ge s_1\ge s_0\ge \frac{s}{2}+\gamma.
$$
We also fix an integer $\rho$ larger than $s_0$. 
Our goal is to estimate the norm
$$
M_s^{(s_1)}(t)=\sum_{p=0}^{s_1} \Bigl( 
\blA Z^p \eta(t)\brA_{H^{s-p}}
+\blA \Dxmez Z^p \omega(t)\brA_{H^{s-p}}\Bigr),
$$
assuming some control of the H\"older norms
$$
N_\gamma(t)=\lA \eta(t)\rA_{\eC{\gamma}}+
\blA \Dxmez \psi(t)\brA_{\eC{\gamma}}
$$
and
$$
N_\rho^{(s_0)}(t)=\sum_{p=0}^{s_0} \Bigl( 
\blA Z^p \eta(t)\brA_{\eC{\rho-p}}
+\blA \Dxmez Z^p \psi(t)\brA_{\eC{\rho-p}}\Bigr).
$$

To estimate $M_s^{(s_1)}$ we shall estimate the 
$L^2$-norm of $\px^\alpha Z^n \vU$ for $(\alpha,n)$ in the set
$$
\mathcal{P}=\bigl\{ (\alpha,n)\in \xN\times \xN\,;\, 0\le n\le s_1,~0\le \alpha\le s-n\bigr\}.
$$
(In fact, we shall estimate
$$
\blA \px^\alpha Z^n \eta\brA_{H^\beta}+\blA  \Dxmez \px^\alpha Z^n \omega\brA_{H^\beta}
+\blA \Dxmez \px^\alpha Z^n  \psi\brA_{H^{\beta-\mez}},
$$
for some large enough exponent $\beta$, but small compared to $\gamma$; 
in this outline, we do not discuss this as well as other similar difficulties).

We shall proceed by induction. This requires to introduce a 
bijective map, denoted by $\Lambda$, from $\mathcal{P}$ to $\{0,1,\ldots,\# \mathcal{P}-1\}$. We find that it is convenient 
to chose $\Lambda$ such that 
$\Lambda(\alpha',n')< \Lambda(\alpha,n)$ 
holds if and  only if either 
$n'<n$ or [$n'=n$ and $\alpha'<\alpha$]. This corresponds to 
$$
\Lambda(\alpha,n)=\sum_{p=0}^{n-1}(s+1-p)+\alpha.
$$
Given an integer $K$ in $\{0,\ldots,\# \mathcal{P}\}$ we set
$$
\mathcal{M}_{K}= \sum_{\Lambda(\alpha',n')\le K-1}
\lA \px^{\alpha'} Z^{n'} \vU\rA_{L^2}.
$$

As alluded to above, the Hilbert transform 
appears at several place in the analysis. The problem is that it is not bounded on 
H\"older spaces and one has only an estimate of the form: 
for any $\rho\not\in\xN$, 
there exists $K>0$ and for any $\nu>0$, any $v\in \eC{\rho}\cap L^2$,
$$
\lA \mathcal{H}v\rA_{\eC{\rho}}\le K \Big[\lA v\rA_{\eC{\rho}}+\frac{1}{\nu}
\lA v\rA_{\eC{\rho}}^{1-\nu}\lA v\rA_{L^2}^\nu\Bigr].
$$
Here one cannot overcome this problem and we are lead to introduce the norms
$$
\avant=N_\rho^{(s_0)}+\frac{1}{\nu}\bigl(N_\rho^{(s_0)}\bigr)^{1-\nu}\bigl(\Avant\bigr)^\nu
$$
for some $\nu>0$ (the optimal choice is $\nu=\sqrt{\eps}$ for initial data of size $\eps$). 

We shall prove that there are for any $K=0,\ldots,\#\mathcal{P}-1$ a constant $A_K$ and a 
non-decreasing 
function $C_K(\cdot)$ such that for any $\nu$ in $]0,1]$, any positive numbers 
$T_0,T$ and any $t$ in $[T_0,T]$,
\be\label{i50}
\ba
\mathcal{M}_{K+1}(t)&\le A_K M_s^{(s_1)}(T_0)+C_K\big(N_\rho^{(s_0)}(t)\big)\big(1+\mathcal{N}_K(t)\big)\mathcal{M}_K(t)\\
&\quad +\int_{T_0}^t C_K\big(N_\rho^{(s_0)}(t')\big)\lA u(t',\cdot)\rA_{\eC{\gamma}}^2\Mr_{K+1}(t')\, dt'\\
&\quad +\int_{T_0}^t C_K\big( N_\rho^{(s_0)}(t')\big)\Nr_K(t')^2\Mr_{K}(t')\, dt'
\ea
\ee
(setting $\Nr_0\equiv 0$, $\Mr_0\equiv 0$ when $K= 0$).

This estimate will be used to prove that, if
for any $t\in [T_0,T[$ and any $\eps\in ]0,\eps_0]$
\be\label{i51}
\blA \Dxmez \psi(t,\cdot)\brA_{\eC{\gamma-\mez}}
+\lA \eta(t,\cdot)\rA_{\eC{\gamma}}=O\big( \eps t^{-\mez}\big)
\ee
and
\be\label{i52}
N_\rho^{(s_0)}(t)=O\big(\eps t^{-\mez+\nu}\big)
\ee
for some constant $0<\nu\ll 1$, then there is an increasing sequence $(\delta_K)_{0\le k\le \# \mathcal{P}}$, depending only on $\nu$ and $\eps$ such that for any $t$ in $[T_0,T[$ and any 
$\eps$,
\be\label{i53}
\mathcal{M}_{K}(t)=O\big(\eps t^{\delta_K}\big).
\ee
The proof is by induction on $K$. For $K=\# \mathcal{P}$ we obtain an estimate for $M_s^{(s_1)}$. 
The key point is that, when we use Gronwall lemma to deduce from \e{i50} a bound for 
$\mathcal{M}_{K+1}$, assuming that \e{i51}, \e{i52}, \e{i53} hold, the coefficient of 
$\mathcal{M}_{K+1}(t')$ in the first integral in \e{i50} is $O(\eps^2 t ^{-1})$ by \e{i51}. In that way, 
it induces only a $O(t^{\eps^2C})$ growth for $\mathcal{M}_{K+1}$. The fact that, on the other hand, 
$\mathcal{M}_{K}(t')$ in the second integral in \e{i50} is multiplied by a factor that may grow like 
$t^{-\mez+\delta}$ (with $0<\delta\ll 1$) is harmless, as $\mathcal{M}_{K}(t')$ is a source term, 
already estimated in the preceding step of the induction.


The proof of \e{i50} contains an analysis of independent interest. Namely, 
we shall prove various tame estimates for the action of iterated vector fields
$Z=t\partial_t+2x\px$ on the equations. 
Such estimates have already been obtained by Wu~\cite{Wu09} and 
Germain-Masmoudi-Shatah in \cite{GMS2}.  
We shall prove sharp 
tame estimates tailored to our purposes (one key point is to estimate the action of $Z^k$ on $F(\eta)\psi$). 
This part is quite technical 
and we refer the reader to Chapter~\ref{S:23} for precise statements.  
In this chapter, we shall prove that
$$
Z G(\eta)\psi=   G(\eta)\big((Z-2)\psi-\B Z\eta) -\partial_x( (Z\eta) V)+
2\left[G(\eta),\eta\right]\B +2V\partial_x \eta.
$$
Since $\B$ and $V$ are expressions of $\px\eta,\px\psi$ and $G(\eta)\psi$, one deduce from the above identity formulae for $Z\B$ and $ZV$. This allows by induction to 
express the action of iterated vector fields $Z$ on the Dirichlet-Neumann operator $G(\eta)$ 
in terms of convenient classes of multilinear operators.

\renewcommand{\thesection}{\arabic{chapter}.\arabic{section}}

\chapter{Statement of the main results}\label{chap:1} 

In this chapter, we state the main Sobolev estimate whose proof is the goal 
of this paper, and we describe the global
existence theorem for water waves equations  established in~\cite{AlDelMain} using these Sobolev bounds. Before stating the
result, we define in a precise way the Dirichlet-Neumann operator that appears in the Craig-Sulem-Zakharov version of the
water waves equation, and establish properties of this operator that are used in the sequel as well as in~\cite{AlDelMain}.

\section{Definitions and properties of the 
Dirichlet-Neumann operator}\label{S:11}

Let $\eta\colon \xR\rightarrow\xR$ be a smooth enough function and consider the 
open set
$$
\Omega\defn \{\,(x,y)\in\xR\times\xR \,;\, y<\eta (x)\,\}.
$$
If $\psi\colon\xR\rightarrow\xR$ is another function, and if we call 
$\phi\colon\Omega\rightarrow \xR$ the unique solution of 
$\Delta\phi=0$ in~$\Omega$ satisfying 
$\phi\arrowvert_{y=\eta(x)}=\psi$ and a convenient vanishing condition at $y\rightarrow -\infty$, one defines 
the Dirichlet-Neumann operator\index{Dirichlet-Neumann operator!$G(\eta)$} $G(\eta)$ by 
\begin{equation*}
G(\eta)\psi   =
\sqrt{1+(\partial_x\eta)^2}\,
\partial _n \phi\arrowvert_{y=\eta},
\end{equation*}
where~$\partial_n$ is the outward normal derivative on~$\partial\Omega$, so that
$$
G(\eta)\psi=(\partial_y \phi)(x,\eta(x))- (\partial_x \eta)(\partial_x \phi)(x,\eta(x)).
$$
The goal of this section is to make precise the above definition and 
to study the action of $G(\eta)$ on different spaces. 

We shall reduce the problem to the negative half-space through the change of coordinates
$(x,y)\mapsto (x,z=y-\eta(x))$, which sends $\Omega$ on 
$\{ (x,z)\in \xR^2\,;\, z<0\}$. 
Then $\phi(x,y)$ solves $\Delta\phi=0$ if and only if $\varphi(x,z)=\phi(x,z+\eta(x))$ is a solution of 
$P\varphi=0$ in $z<0$, where
\be\label{111}
P=(1+\eta'^2)\partial_z ^2+\px^2-2\eta'\px\partial_z-\eta''\partial_z 
\ee
(we denote by $\eta'$ the derivative $\px\eta$). The boundary condition 
becomes $\varphi(x,0)=\psi(x)$ and $G(\eta)$ is given by
$$
G(\eta)\psi=\bigl[ (1+\eta'^2)\partial_z\varphi-\eta'\partial_x \varphi\bigr]\big\arrowvert_{z=0}.
$$
It is convenient and natural to try to solve the boundary value problem 
$$
P\varphi=0, \quad\varphi\arrowvert_{z=0}=\psi
$$
when $\psi$ lies in homogeneous Sobolev spaces. 
Let us introduce them and fix some notation.

We denote by \index{Function spaces!$\Sr'_\infty(\xR)$, $\Sr'_1(\xR)$} 
$\Sr'_\infty(\xR)$ (resp.\ $\Sr'_1(\xR)$) the quotient space 
$\Sr'(\xR)/\xC[X]$ (resp.\ $\Sr'(\xR)/\xC$). If $\Sr_\infty(\xR)$ (resp.\ $\Sr_1(\xR)$) 
is the subspace of $\Sr(\xR)$ made of the functions orthogonal to any polynomial 
(resp.\ to the constants), $\Sr'_\infty(\xR)$ (resp.\ $\Sr'_1(\xR)$) is the dual of 
$\Sr_\infty(\xR)$ (resp.\ $\Sr_1(\xR)$). Since 
the Fourier transform realizes an isomorphism 
from $\Sr_\infty(\xR)$ (resp.\ $\Sr_1(\xR)$) to
$$
\widehat{\Sr}_\infty(\xR)=\{ u\in \Sr(\xR)\,;\, u^{(k)}(0)=0 \text{ for any }k\text{ in }\xN\}
$$ 
(resp. $\widehat{\Sr}_1(\xR)=\{ u\in \Sr(\xR)\,;\, u(0)=0\}$), we get by duality that the Fourier transform defines an isomorphism 
from $\Sr'_\infty(\xR)$ to $(\widehat{\Sr}_\infty)'(\xR)$, which is the quotient of $\Sr'(\xR)$ by the subspace of distributions supported in $\{0\}$ 
(resp.\ from $\Sr'_1(\xR)$ to $(\widehat{\Sr}_1)'(\xR)=\Sr'(\xR)/{\rm Vect}\, (\delta_0)$). 

Let $\phi\colon\xR\rightarrow \xR$ be a function defining a Littlewood-Paley decomposition (see Appendix~\ref{S:A.3}) 
and set for $j\in \xZ$, $\Delta_j=\phi(2^{-j}D)$. Then for any $u$ in $\Sr'_\infty(\xR)$, the series $\sum_{j\in\xZ} \Delta_j u$ 
converges to $u$ in $\Sr'_\infty(\xR)$ (for the weak-$*$ topology associated to the natural topology on $\Sr_\infty(\xR)$). 
Let us recall (an extension of) the usual definition of 
homogeneous Sobolev or H\"older spaces. 

\begin{defi}\label{ref:111}
Let $s',s$ be real numbers. One denotes by \index{Function spaces!$\h{s',s}$, homogeneous Sobolev spaces}
$\h{s',s}(\xR)$ (resp.\ \index{Function spaces!$\C{s',s}$, homogeneous Zygmund spaces}$\C{s',s}(\xR)$) 
the space of elements $u$ in $\Sr'_\infty(\xR)$ 
such that there is a sequence $(c_j)_{j\in\xZ}$ in $\ell^2(\xZ)$ (resp.\ a constant $C>0$) with for any $j$ 
in $\xZ$,
$$
\lA \Delta_j u\rA_{L^2}\le c_j 2^{-js' -j_+ s}
$$
(resp.\
$$
\lA \Delta_j u\rA_{L^\infty}\le C 2^{-js'-j_+ s})
$$
where $j_+=\max(j,0)$. We set $\h{s'}$ (resp.\ $\C{s'}$) when $s=0$.
\end{defi}
The series $\sum_{j=0}^{+\infty}\Delta_j u$ always converges in $\Sr'(\xR)$ under 
the preceding assumptions, but 
the same is not true for $\sum_{j=-\infty}^{-1}\Delta_j u$. If $u$ is in $\h{s',s}(\xR)$ 
with $s'<1/2$ (resp.\ in $\C{s',s}(\xR)$ with $s'<0$), 
then $\sum_{j=-\infty}^{-1}\Delta_j u$ converges normally in $L^\infty$, so in $\Sr'(\xR)$, and 
$u\rightarrow \sum_{-\infty}^{+\infty}\Delta_j u$ 
gives the unique dilation and translation invariant realization of $\h{s',s}$ (resp.\ $\C{s',s}(\xR)$) 
as a subspace of $\Sr'(\xR)$. 
One the other hand, if $s'\in [1/2,3/2\por$ (resp. $s'\in [0,1\por$), the space $\h{s'}(\xR)$ (resp. $\C{s'}(\xR)$) admits no translation commuting realization as a 
subspace of $\Sr'(\xR)$, but the map $u\rightarrow \sum_{-\infty}^{+\infty}\Delta_j u$ defines a dilation 
and translation commuting realization of these spaces as subspaces of $\Sr'_1(\xR)$. We refer to Bourdaud~\cite{Bourdaud} for these properties. 

For $k\in \xN$, we denote by $C_p^k(]-\infty,0],\Sr_\infty'(\xR))$ the space of functions $z\rightarrow u(z)$ defined 
on $]-\infty,0]$ with values in $\Sr'_\infty(\xR)$, 
such that for any $\theta$ in $\Sr_\infty(\xR)$, 
$z\rightarrow \langle u(z),\theta\rangle$ is $C^k$, and there is $M\in\xN$ and a continuous semi-norm $p$ on $\Sr_\infty(\xR)$, such that for any $k'=0,\ldots,k$, 
any $\theta$ in $\Sr_\infty(\xR)$, 
$$
\big\vert \pz^k \langle u(z),\theta\rangle\big\vert \le 
p(\theta) (1+|z|)^M.
$$
We denote by $\Dcal'(]-\infty,0[,\Sr_\infty'(\xR))$ the dual space 
of $C^\infty_0(]-\infty,0[)\otimes \Sr_\infty(\xR)$. We shall denote by $L^2(]-\infty,0],\Sr_\infty'(\xR))$ the subspace of 
$\mathcal{D}'(]-\infty,0[,\Sr_\infty'(\xR))$ made of those distributions 
$u$ such that for any $\theta\in \Sr_\infty(\xR)$, $z\rightarrow \langle u(z,\cdot),\theta\rangle$ is in $L^2(]-\infty,0])$ 
and there are continuous semi-norms $p$ on $\Sr_\infty(\xR)$ and an $L^2$-function $h$ on $]-\infty,0]$ so that for any $\theta$ in $\Sr_\infty(\xR)$, $\la\langle u(z,\cdot),\theta\rangle\ra \le p(\theta)h(z)$.

\begin{defi}\label{ref:112}
We denote by $E$ the space 
$$
E=\bigl\{ \va \in \Dcal'(]-\infty,0[,\Sr_\infty'(\xR))\,;\, 
\nabla_{x,z}\va\in L^2(]-\infty,0[\times \xR)\bigr\}.
$$
(We consider $L^2(]-\infty,0[\times \xR)$ 
as a subspace of $\Dcal'(]-\infty,0[,\Sr_\infty'(\xR))$ using that the natural map from $L^2(\xR)$ 
to $\Sr_\infty(\xR)$ is injective). We endow $E$ with the semi-norm $\lA \nabla_{x,z}\va\rA_{L^2L^2}$.
\end{defi}
\begin{remas}--- If $\va$ is in $E$, then $\va$ belongs to 
$C_p^0(]-\infty,0],\Sr_\infty'(\xR))$. In particular, $\va\az$ is well defined as an element of 
$\Sr_\infty'(\xR)$. Actually, if $\theta_1$ is a test function in $\Sr_\infty(\xR)$, it 
may be written $\theta_1=\px\tilde{\theta}_1$ for another function $\tilde{\theta}_1$ 
in $\Sr_\infty(\xR)$, so that, for any $\theta_0$ in $C^\infty_0(]-\infty,0[)$,
$$
\big\langle\va,\theta_0(z)\otimes \theta_1(x)\big\rangle =-\big\langle \px\va,\theta_0(z)\otimes \tilde{\theta}_1(x)\big\rangle
$$
which shows that $z\rightarrow \langle \va(z,\cdot),\theta_1\rangle$ is 
in $L^2(]-\infty,0[)$. Moreover, its $z$-derivative is also $L^2$, so that $z\rightarrow \langle \va(z,\cdot),\theta_1\rangle$ 
is a continuous bounded function.

--- The semi-norm $\lA \nabla_{x,z}\va\rA_{L^2L^2}$ is actually a norm on $E$, and $E$ 
endowed with that semi-norm is a Banach space. Actually, if $(\va_n)_n$ is a Cauchy sequence 
in~$E$, if $\theta_0,\theta_1,\tilde{\theta}_1$ are as above, we may write
$$
\la \langle \va_n-\va_m,\theta_0(z)\otimes \theta_1(x)\rangle \ra
\le \lA \px (\va_n-\va_m)\rA_{L^2L^2} \blA \theta_0\otimes \tilde{\theta}_1\brA_{L^2L^2}
$$
which shows that $(\va_n)_n$ converges to a limit $\va$ in 
$\Dcal'(]-\infty,0[,\Sr'_\infty(\xR))$. That limit $\va$ satisfies 
$\nabla_{x,z}\va \in L^2(]-\infty,0[\times \xR)$ i.e.\ belongs to $E$.
\end{remas}			

The space $E$ introduced in Definition~\ref{ref:112} is a natural one 
in view of the following lemma.

\begin{lemm}\label{ref:113}
Let $\psi$ be in $\Sr_\infty'(\xR)$. There is an equivalence between 
\begin{enumerate}[i)]
\item The function $x\rightarrow \psi(x)$ is in $\h{\mez}(\xR)$.
\item The function $(x,z)\rightarrow e^{z\Dx}\psi(x)$ is in $E$.
\end{enumerate}
Moreover
\be\label{114}
\blA \px \bigl(e^{z\Dx}\psi\bigr)\brA_{L^2L^2}^2+\blA \pz \bigl(e^{z\Dx}\psi\bigr)\brA_{L^2L^2}^2
=\Dxmezpsi^2.
\ee
\end{lemm}
\begin{proof}
If $\psi$ is in $\h{\mez}(\xR)$, it is clear that $(x,z)\rightarrow e^{z\Dx}\psi$ is a bounded 
function with values in $\Sr_\infty'(\xR)$. Moreover, 
$$
\blA \px \bigl(e^{z\Dx}\psi\bigr)\brA_{L^2L^2}^2
=\frac{1}{2\pi}
\int_{-\infty}^0\int e^{2z\la\xi\ra}\big\vert \widehat{\psi}(\xi)\big\vert^2\la \xi\ra^2\, d\xi dz
=\frac{1}{2}\Dxmezpsi^2
$$
and $\Dxmezpsi$ is equivalent to the $\Hmez$-norm. As a similar 
computation holds for the $\pz$-derivative, the conclusion follows.
\end{proof}

The preceding lemma gives a solution $e^{z\Dx}\psi$ to the boundary values problem 
$\Delta(e^{z\Dx}\psi)=0$ in $z<0$, $e^{z\Dx}\psi\az=\psi$. 
Let us study the corresponding non homogeneous problem. 

\begin{lemm}\label{ref:114}
Let $f$ be given in $L^2(]-\infty,0],\Sr_\infty'(\xR))$ and 
$\psi$ be in $\Srp$. There is a unique function $\va$ in 
$C^1_p(\infzerof,\Srp)$ solution of the equation $(\px^2+\pz^2)\va=f$ 
in $z<0$, $\va\az=\psi$. It is given by the equality between elements of $\Srp$ 
at fixed $z$:
\be\label{115}
\ba
\va(z,x)=e^{z\Dx}\psi&+\mez \int_{-\infty}^0e^{(z+z')\Dx} \Dx^{-1}f(z',\cdot)\, dz'\\
&-\mez \int_{-\infty}^0e^{-|z-z'|\Dx}\Dx^{-1}f(z',\cdot)\, dz'.
\ea
\ee
Moreover, if we assume that $\nabla_{x,z}\va$ is in 
$L^2(\infzerof\times \xR)$ (resp.\ that $\va$ is in 
$L^2(\infzerof\times \xR)$) the solution $\va$ is unique modulo 
constants (resp.\ is unique).
\end{lemm}
\begin{proof}
Let us show first that the integrals in the \rhs of \eqref{115} 
are converging ones when acting on a 
test function $\theta$ in $\Sr_\infty(\xR)$. By definition 
of $L^2(\infzerof,\Srp)$, there is a semi-norm $p$ on $\Sr_\infty(\xR)$, there is 
an $L^2(\infzerof)$ function 
$z\rightarrow h(z)$ such 
that $|\langle f(z',\cdot),\theta_1\rangle\le p(\theta_1)h(z')$ 
for any $\theta_1$ in $\Sr_\infty(\xR)$, any $z'<0$. Moreover, for any 
$N$, $\Dx^N$ is an isomorphism from $\Sr_\infty(\xR)$ to itself. We may write for fixed $z$ 
and for any $\theta$ in $\Sr_\infty(\xR)$
\begin{multline}\label{115a}
\int_{-\infty}^{z-1}\big\langle e^{-|z\pm z'|\Dx} \Dx^{-1}f(z',\cdot),\theta\big\rangle \, dz'\\
=
\int_{-\infty}^{z-1}\big\langle f(z',\cdot), e^{-|z\pm z'|\Dx} \bigl( |z-z'|\Dx \bigr)\bigl(\Dx^{-2}\theta\bigr)
\big\rangle \frac{dz'}{|z-z'|}\cdot
\end{multline}
Any semi-norm of the term in the \rhs of the bracket is controlled uniformly in $z<0$, 
$z'<0$. It follows that the integral converges. The same is trivially true for the integral from 
$z-1$ to $0$ of the integrand in the left hand side of \eqref{115a}. This shows also that the
\rhs of \eqref{115} is in $C_p^0(\infzerof,\Srp)$. Taking the $\pz$-derivative, we get 
in the same way that $\pz\va$ is in $C_p^0(\infzerof,\Srp)$. Moreover, 
a direct computation shows that we get a solution of $(\px^2+\pz^2)\va=f$ 
with the wanted boundary data.

To prove uniqueness, we have to check that when $\psi=0$, the unique function 
$\va$ in the space $C_p^1(\infzerof,\Srp)$ satisfying $(\px^2+\pz^2)\va=0$ in $z<0$, $\va\az=0$ 
is zero in that space. If we set 
$$
U=\begin{pmatrix} \va \\ \pz\va\end{pmatrix},\quad 
A(\OD)=\begin{pmatrix} 0 & 1 \\ \Dx^2 & 0\end{pmatrix},
$$
this is equivalent to checking that the only solution of 
$\partial_z U=A(\OD)U$ in $C_p^0(\infzerof,\Srp\times \Srp)$ with $U_1\az=0$ 
is zero. If we set
$$
P(\OD)=\begin{pmatrix} 1 & 1 \\ \Dx & -\Dx \end{pmatrix},
\quad V=\begin{pmatrix}V_1\\ V_2\end{pmatrix}=P(\OD)^{-1}U,
$$
we are reduced to verifying that the unique $V$ in $C_p^0(\infzerof,\Srp\times \Srp)$ 
such that 
$\pz V = \left(\begin{smallmatrix} \Dx & 0 \\ 0 &-\Dx\end{smallmatrix}\right)V$ and $V_1+V_2 \az=0$ 
is zero in that space. It is sufficient to check that
\be\label{116}
\ba
&V_2\in C_p^0(\infzerof,\Srp) \text{ and }\pz V_2+\Dx V_2=0 ~\Rightarrow ~V_2\equiv 0,\\
&V_1\in C_p^0(\infzerof,\Srp) \text{ and }\pz V_1-\Dx V_1=0 ~\Rightarrow ~V_1\equiv 0.
\ea
\ee
To prove the first implication, we take $\theta$ in 
$C_p^0(]-\infty,0[,\Sr_\infty(\xR))$ 
and set
$$
\widetilde{\theta}(z,x)
=\int_z^0 e^{-(z'-z)\Dx}\theta(z',\cdot)\, dz'.
$$
If $\theta_1$ 
is some $C^\infty_0(\infzerof)$ function equal to one close to zero, such that 
$\theta_1(z)\theta(z,x)=\theta(z,x)$, we may write for any $M>1$, using that 
$\widetilde{\theta}$ vanishes close to $z=0$,
\begin{align*}
0&=\int_{-\infty}^0 \big\langle (\pz+\Dx)V_2,\theta_1(z/M)\widetilde{\theta}(z,\cdot)\big\rangle\, dz\\
&=\int_{-\infty}^0 \big\langle V_2,\theta(z,\cdot)\big\rangle\, dz
-\int_{-\infty}^0 \Big\langle V_2,\frac{1}{M}\theta_1'\Bigl(\frac{z}{M}\Bigr)\widetilde{\theta}(z,\cdot)\Big\rangle\, dz
\end{align*}
and the conclusion will follow if one shows that the last integral goes to zero as $M$ goes to $+\infty$. 
Because of the fact that $V_2$ is assumed to be at most at polynomial growth, it is enough 
to show that any semi-norm of $\frac{1}{M} \theta_1'\bigl(\frac{z}{M}\bigr)\widetilde{\theta}(z,\cdot)$ 
in $\Sr_\infty(\xR)$ goes to zero more rapidly than $M^{-k}$ (or $|z|^{-k}$) for any 
$k$ when $M$ goes to $+\infty$. This follows from the fact that, as above, 
we may write $\widetilde{\theta}$ as 
$$
\int_{a}^{0}e^{-(z'-z)|\Dx} \bigl( (z'-z)\Dx \bigr)^N\bigl(\Dx^{-N}\theta(z',\cdot)\bigr)
\frac{dz'}{(z'-z)^N}
$$
if $a$ is such that $\supp \theta \subset [a,0]\times\xR$, if $z$ is in the support of 
$\theta_1'(z/M)$ and $N$ is an arbitrary integer. 

To prove the second implication \eqref{116}, we argue in the same way, taking
$$
\widetilde{\theta}(z,x)
=\int_{-\infty}^z e^{-(z-z')\Dx}\theta(z',\cdot)\, dz'
$$
and replacing 
$\theta_1$ by $1$. Since $V_1\az =0$ and $\widetilde{\theta}$ is supported for $z$ in a 
compact subset of $\infzerof$, we obtain
$$
0=\int_{-\infty}^0 \big\langle (\pz-\Dx)V_1,\widetilde{\theta}(z,\cdot)\big\rangle\, dz 
=-\int_{-\infty}^0 \langle V_1,\theta(z,\cdot)\rangle\, dz
$$
this implies the conclusion.

The above uniqueness statement holds in general only in 
$C^1_p(\infzerof,\Srp)$ i.e.\ modulo polynomials at fixed $z$. Let us check 
that if we assume moreover that 
$\nabla_{x,z}\va$ belongs to 
$L^2(\infzerof\times\xR)$, then $\va$ is constant. By what we have just seen, 
we already know that for any fixed $z$, 
$x\rightarrow \px \va(z,x)$ is zero in $\Srp$ i.e.\ is a polynomial. 
Consequently, for almost every $z$, $x\rightarrow \va(z,x)$ has to be 
a polynomial such that $\px \va (z,x)$ is in $L^2(dx)$. This implies that 
$\va$ has to be independent of $x$, which, together with the equation 
$(\px^2+\pz^2)\va=0$ implies that $\va$ is a constant. If we assume that 
$\va$ is in 
$L^2(\infzerof\times\xR)$, one proves in the same way that $\va$ is zero. This concludes the proof.
\end{proof}

We use now the preceding result to write the solution of the Dirichlet boundary values problem associated 
to the operator $P$ defined in \eqref{111} as the solution of a fixed point problem. 

\begin{lemm}\label{ref:115}
Let $\psi$ be in $\Sr_1'(\xR)$, $\eta$ in $\eC{\gamma}(\xR)$ with $\gamma>2$, 
$h_1,h_2$ two functions in $L^2(\infzeroo\times \xR)$, with 
$\pz h_1$ in $L^2(\infzeroo,H^{-1}(\xR))$. Let $\va$ be an element 
of the space $E$ of Definition~\ref{ref:112}, satisfying 
$P\va=\pz h_1+\px h_2$, $\va\az=\psi$. Then $\va$ is in $C^1_p(\infzerof,\Srp)$ 
and satisfies the equality between functions in $C^0_p(\infzerof,\Srp)$ 
\be\label{117}
\ba
\va(z,x)&=e^{z\Dx}\psi\\
&\quad +\mez \int_{-\infty}^0 e^{(z+z')\Dx}\Bigl[ \px \Dx^{-1}\bigl( \eta'\pz \va+h_2\bigr)
\Bigr] \, dz'\\
&\quad +\mez \int_{-\infty}^0 e^{(z+z')\Dx}\Bigl[ -\bigl( \eta'\px \va -\eta'^2\pz \va +h_1\bigr)\Bigr] \, dz'\\
&\quad +\mez \int_{-\infty}^0 e^{-|z-z'|\Dx}\Bigl[
- \px \Dx^{-1}\bigl( \eta'\pz \va+h_2\bigr)\Bigr] \, dz'\\
&\quad +\mez \int_{-\infty}^0 e^{-|z-z'|\Dx}\Bigl[
\sign(z-z')
\bigl( \eta'\px \va -\eta'^2\pz \va +h_1\bigr)\Bigr] \, dz'.
\ea
\ee
If we assume that $\psi$ is in $\h{\mez}(\xR)$, this equality holds 
modulo constants. Conversely, if $\va$ is in $E$ and satisfies \eqref{117}, then 
$P\va=\pz h_1+\px h_2$, $\va\az=\psi$.
\end{lemm}
\begin{proof}
The equation $P\va=\pz h_1+\px h_2$ implies
\be\label{118}
\pz^2\va=-\frac{1}{1+\eta'^2}\px^2\va+2\frac{\eta'}{1+\eta'^2}\px\pz\va
+\frac{\eta''}{1+\eta'^2}\pz\va
+\frac{\pz h_1+\px h_2}{1+\eta'^2}\cdot
\ee
The assumptions on $\eta$ imply that the coefficients of the first two and last terms 
(resp.\ of the third term) in the \rhs are in $\eC{\gamma-1}(\xR)$ (resp.\ $\eC{\gamma-2}(\xR)$). 
Since $\px^2\va$, $\px\pz\va$, $\pz h_1$, $\px h_2$ (resp.\ $\pz \va$) 
are in $L^2(\infzerof,H^{-1}(\xR))$ (resp.\ $L^2(\infzerof,L^2(\xR))$), 
property~\e{pr:sz} of the Appendix~\ref{s2} and assumption $\gamma>2$ imply that 
$\pz^2\va$ is in $L^2(\infzerof,H^{-1}(\xR))$. Consequently, if we set
\be\label{119}
f_1=\eta'\px\va-\eta'^2\pz\va+h_1,\quad 
f_2=\eta'\pz\va+h_2,\quad f=\pz f_1+\px f_2,
\ee
we obtain that $f$ is in $L^2(\infzerof,\Srp)$ and the equation 
$P\va=\pz h_1+\px h_2$, $\va\az=\psi$ may be rewritten
\be\label{1110}
(\px^2+\pz^2)\va=f,\quad \va(0,\cdot)=\psi.
\ee
Moreover, since $\pz^2\va$ is in $L^2(\infzerof,H^{-1}(\xR))$ and 
$\pz\va$ in $L^2(\infzeroo \times \xR)$, we conclude that $\pz\va$ is in 
$C^0_p(\infzerof,\Srp)$. Consequently, we may apply Lemma~\ref{ref:114} 
which shows that the unique $C^1_p(\infzerof,\Srp)$ solution to \eqref{1110} is given by \eqref{115}. 
If we replace $f$ by its value given in \eqref{119}, we deduce \eqref{117} from \eqref{115} if we can justify 
$\partial_{z'}$-integration by parts of the $\partial_{z'}f_1$ contribution to $f$. Let us do that 
for the second integral in the right hand side of \eqref{115} with $f$ replaced by 
$\partial_{z'}f_1$. Take $\theta$ a test function in $\Sr_\infty(\xR)$, 
$\theta_1$ in $C^\infty_0(\infzerof)$ equal to $1$ close to zero. 
Compute 
\be\label{1111}
\ba
&\int_{-\infty}^0 \big\langle e^{-|z-z'|\Dx}\partial_{z'}\Dx^{-1} f_1(z',\cdot),\theta\big\rangle \theta_1\Bigl(\frac{z'}{R}\Bigr)
\, dz'\\
&\qquad =\big\langle e^{|z|\Dx} f_1(0,\cdot),\theta\big\rangle\\
&\qquad \quad -\int_{-\infty}^0 \big\langle e^{-|z-z'|\Dx}\sign(z-z') f_1(z',\cdot),\theta\big\rangle
\theta_1\Bigl(\frac{z'}{R}\Bigr)
\, dz'\\
&\qquad\quad-\int_{-\infty}^0 \big\langle e^{-|z-z'|\Dx}\Dx^{-1} f_1(z',\cdot),\theta\big\rangle \frac{1}{R}
\theta_1'\Bigl(\frac{z'}{R}\Bigr)\, dz'.
\ea
\ee
Since $f_1$ is in $L^2L^2$, the first integral in the \rhs may be written  
$$
\frac{1}{2\pi}
\int_{-\infty}^0\int e^{-|z-z'|\la \xi\ra}\sign(z-z') \widehat{f}_1(z',\xi)\widehat{\theta}(-\xi)
\theta_1\Bigl(\frac{z'}{R}\Bigr)
\, dz' d\xi,
$$
where $\widehat{\theta}$ is in $\Sr(\xR)$ and vanishes at infinite order at $\xi=0$, 
converges when $R$ goes to $+\infty$ to the same quantity with $\theta_1$ replaced by $1$. 
On the other hand, the last integral
$$
\frac{1}{2\pi}
\int_{-\infty}^0\int \langle e^{-|z-z'|\la \xi\ra}\la \xi\ra^{-1}\widehat{f}_1(z',\xi)\widehat{\theta}(-\xi)
\frac{1}{R}\theta_1\Bigl(\frac{z'}{R}\Bigr)
\, dz' d\xi,
$$
goes to zero if $R$ goes to $+\infty$, using again the vanishing properties if $\widehat{\theta}$ at 
$\xi=0$. To finish the justification of the integration by parts, we just need to see that the left hand side of \eqref{111} 
converges when $R$ goes to infinity to the same quantity with $\theta_1$ dropped. 
This follows in the same way since $\pz f$ is in $L^2(\infzerof,H^{-1}(\xR))$. 

The equality \eqref{117} holds in the space $C_p^1(\infzerof,\Srp)$ i.e.\ modulo 
polynomials for each fixed $z$. To check that it actually holds 
modulo a constant when we assume that $\psi$ is in $\Hmez$, it is enough, according to Lemma~\ref{ref:114}, to 
verify that the $(x,z)$-gradient of both sides belongs to 
$L^2(\infzerof\times \xR)$. This is true for $\va$ by assumption. 
On the other hand, Lemma~\ref{ref:113} shows that $\nabla_{x,z}(e^{z\Dx}\psi)$ belongs 
to that space. It remains to show that if $g$ is in $L^2L^2$ then
$\int_{-\infty}^0e^{-|z\pm z'|\la \xi\ra}\la \xi\ra g(z',\xi)\, dz'$ is in $L^2(\infzerof\times \xR;dz d\xi)$, which is trivial.

Conversely, if $\va$ is in $C^0_p(\infzerof,\Srp)$  and 
$\nabla_{x,z}\va$ in $L^2(\infzerof\times \xR)$ and solves \eqref{117}, one checks that 
$P\va=\pz h_1+\px h_2$ by a direct computation.
\end{proof}

The main result of this section, that allows one to define rigorously the 
Dirichlet-Neumann operator, and prove some of its property, 
is the following.
\begin{prop}\label{ref:116}
Let $\gamma$ be a real number, $\gamma>2$, $\gamma\not\in\mez\xN$. 

$i)$ There is $\delta>0$ such that for any $\eta$ in 
$\eC{\gamma}(\xR)$ with $\etapetit<\delta$, 
for any $\psi$ in $\Hmez$, any $h=(h_1,h_2)$ in $L^2L^2$ with 
$\pz h_1$ in $L^2(]-\infty,0[,H^{-1}(\xR))$ 
the equation $P\va=\pz h_1+\px h_2$, $\va\az=\psi$ 
has a unique solution $\va$ in $E$. 
Moreover there is a continuous non decreasing function 
$C\colon \xR_+\rightarrow \xR_+$ such that for any $\eta,\va,\psi,h$ as 
above
\be\label{1112}
\lA \nabla_{x,z}\va\rA_{L^2L^2}\le \Ceta \Bigl( \blA \Dxmez\psi\brA_{L^2}
+\lA h\rA_{L^2L^2}\Bigr),
\ee
\be\label{1113}
\blA \nabla_{x,z}\bigl(\va-e^{z\Dx}\psi\bigr)\brA_{L^2L^2}
\le \Ceta \Bigl( \lA \eta'\rA_{L^\infty}\blA \Dxmez\psi\brA_{L^2}
+\lA h\rA_{L^2L^2}\Bigr)
\ee
Moreover, if $\etapetitgamma<\delta$ and $h=0$, then 
$\nabla_{x,z}\va$ is in $(L^\infty\cap C^0)(]-\infty,0],H^{-\mez}(\xR))$, 
$(1+\eta'^2)\pz\va-\eta'\px\va$ is in 
$(L^\infty\cap C^0)(]-\infty,0],\dot{H}^{-\mez}(\xR))$ and
\begin{align}
&\sup_{z\le 0}\blA \nabla_{x,z}e^{z\Dx}\psi\brA_{\dot{H}^{-\mez}}
\le C \Dxmezpsi,\label{1114}\\
&\sup_{z\le 0}\blA \nabla_{x,z}\big(\va-e^{z\Dx}\psi\big)\brA_{H^{-\mez}}
\le C \etapetit \Dxmezpsi,\label{1115}\\
&\sup_{z\le 0}\blA (1+\eta'^2)\pz\va-\eta'\px\varphi\brA_{\dot{H}^{-\mez}}
\le C \Dxmezpsi.\label{1116}
\end{align}

$i)$ bis. Let $\mu\in [0,+\infty[$, $\gamma>\mu+\tdm$ and assume that $\psi$ is in 
$\h{\mez,\mu+\mez}$. Then the unique function $\varphi$ found in $i)$ when $h=0$ is 
such that $\nabla_{x,z}\varphi$ is in $L^2(]-\infty,0],H^{\mu+\mez})$, 
$(1+\eta'^2)\pz\va-\eta'\px\va$ is in $C^0(]-\infty,0],\h{-\mez,\mu+\mez})\cap 
L^\infty(]-\infty,0],\h{-\mez,\mu+\mez})$ and 
\be\label{1116a}
\sup_{z\le 0} \blA \big((1+\eta'^2)\pz\va-\eta'\px\va\big)(z,\cdot)\brA_{H^\mu}
\le C\blA \Dxmez\psi\brA_{H^{\mu+\mez}}.
\ee

$ii)$ There is $\delta>0$ such that for any 
$\eta\in \eC{\gamma}(\xR)\cap L^2(\xR)$ satisfying 
\be\label{1117}
\etapetitgamma+\lA \eta'\rA_{\eC{-1}}^{1/2}\lA \eta'\rA_{H^{-1}}^{1/2}
<\delta,
\ee
any $\psi$ in $\Hmez\cap \C{\mez,\gamma-\mez}(\xR)$, the unique solution 
given in $i)$ when $h=0$ 
satisfies
\be\label{1118}
\lA \nabla_{x,z}\va\rA_{L^\infty(]-\infty,0],\eC{\gamma-1})}
\le \Cetagamma \blA \Dxmez\psi\brA_{\eC{\gamma-\mez}}.
\ee
Moreover, if $0<\theta'< \theta<\mez$ and if $\blA \eta'\brA_{H^{-1}}^{1-2\theta'}
\blA \eta'\brA_{\eC{-1}}^{2\theta'}$ is bounded, one has the estimate
\be\label{1118a}
\sup_{z\le 0}\blA \Dx^{-\mez+\theta}
\big( (1+\eta'^2)\pz\va-\eta'\px\va\big)(z,\cdot)\brA_{L^\infty}
\le C \blA \Dxmez \psi\brA_{\eC{\gamma-\mez}}.
\ee
\end{prop}
\begin{nota*}
We shall denote by $\Eg{\gamma}$ the set of 
couples
$$
(\eta,\psi)\in \eC{\gamma}(\xR)\times (\Hmez \cap \C{\mez,\gamma-\mez}(\xR))
$$
such that the condition \eqref{1117} holds. By the proposition, the boundary value problem 
$P\va=0$, $\va\az=\psi$ will have a unique solution $\va$ satisfying all 
the statements of $i)$ and $ii)$ of the proposition.
\end{nota*}
\begin{proof}
By Lemma~\ref{ref:115}, the equation $P\va=\pz h_1 +\px h_2$, $\va\az=\psi$ 
has a solution in $E$ if and only if the fixed point problem \eqref{117} has a solution 
in~$E$. Moreover, since \eqref{117} holds modulo constants, we get
\be\label{1119}
\ba
\nabla_{x,z}\va (z,x)&=e^{z\Dx} \begin{pmatrix} \px\psi \\ \Dx\psi\end{pmatrix}\\
&\quad +\int_{-\infty}^0 K(z,z')M(\eta')\cdot \nabla_{x,z}\va(z',\cdot)\, dz'\\
&\quad +\int_{-\infty}^0 K(z,z')M_0 h(z',\cdot)\, dz'+\begin{pmatrix} 0 \\ \eta'\px\va-\eta'^2\pz\va+h_1\end{pmatrix}
\ea
\ee
where $h=(h_1,h_2)$, $K(z,z')$, $M_0$, $M(\eta')$ are the matrices of operators
\begin{align}
K(z,z')&=\mez e^{(z+z')\Dx}\begin{pmatrix} \px & \px \\ \Dx & \Dx \end{pmatrix}\notag\\
&\quad +\mez e^{-|z-z'| \Dx}\begin{pmatrix} -\px & -(\sign(z-z'))\px \\
(\sign(z-z'))\Dx & \Dx \end{pmatrix},\notag\\
M(\eta')&=\begin{pmatrix} 0 & \px \Dx^{-1}\bigl(\eta' \cdot\bigr) \\
-\eta' & \eta'^2\end{pmatrix} ,\quad 
M_0=\begin{pmatrix} 0 & \px \Dx^{-1} \\ -1 & 0 \end{pmatrix}.\label{1120}
\end{align}
Let us notice first that if $U$ is in $L^2(\infzerof\times\xR)$, then 
$\bigl\vert\int_{-\infty}^0 \widehat{K(z,z')U(z',}\xi)\, dz'\bigr\vert$ 
may be bounded from above by expressions of the form
$$
\int_{-\infty}^0e^{-|z-z'| \la \xi\ra}\la\xi\ra \widehat{u}(z',\xi)\, dz'
$$
where $u$ stands for one component of $U$. It follows that 
$U\rightarrow \inin^0 K(z,z')U \, dz'$ is bounded from $L^2(\infzerof\times \xR)$ to itself. 
Moreover, 
\be\label{1121}
\ba
&\lA M_0 h\rA_{L^2}\le \lA h\rA_{L^2},\\
&\lA M(\eta')U\rA_{L^2}\le \Ceta\etapetit \lA U\rA_{L^2}.
\ea
\ee
Consequently
\be\label{1122}
\Big\lVert \nabla_{x,z}\va -e^{z\Dx} \begin{pmatrix}\px \psi\\ \Dx \psi\end{pmatrix}
\Big\rVert_{L^2L^2}
\le \Ceta \etapetit \lA \nabla_{x,z}\va\rA_{L^2L^2}
+C\lA h\rA_{L^2L^2}.
\ee
This, and the fact that by Lemma~\ref{ref:113}, $e^{z\Dx} \begin{pmatrix}\px \psi\\ \Dx \psi\end{pmatrix}$ 
is in $L^2L^2$ if $\psi$ belongs to $\Hmez$, implies that 
for $\etapetit$ small enough, the fixed point problem \eqref{117} has a unique (modulo constants) solution in 
$E$. Moreover, the norm of $\va$ in $E$, i.e.\ $\lA \nabla_{x,z}\va\rA_{L^2L^2}$ is bounded according to 
\eqref{1122} and \eqref{114} by 
$2\bigl( \Dxmezpsi+C\lA h\rA_{L^2L^2}\bigr)$ if $\etapetit$ is small enough. 
This gives \eqref{1112} and \e{1113}. 

We notice next that \e{1114} holds by definition of the $\Hmez$-norm. 
To prove \e{1115}, we shall show that the fixed point problem \eqref{117} has a unique (modulo constants) solution 
$\va$ in the subspace of $E$ formed by those functions~$\va$ for which $\sup_{z\le 0}\lA \nabla_{x,z}\va(z,\cdot)\rA_{H^{-1/2}}<+\infty$. 
Taking \e{1114} into account, we see from \e{1119} that it is enough to show that
\begin{multline}\label{1123}
\sup_{z\le 0}\lA \inin^0 K(z,z')M(\eta')\cdot \nabla_{x,z}\va(z',\cdot)\, dz'\rA_{\h{-\mez}}
\le \Ceta\etapetit \Dxmezpsi
\end{multline}
and
\be\label{1124}
\sup_{z\le 0}\blA \eta'\px\va-\eta'^2\pz\va\brA_{H^{-\mez}(\xR)}\\
\le \Cetagamma\etapetitgamma\sup_{z\le 0}\blA \nabla_{x,z}\va\brA_{H^{-\mez}(\xR)}.
\ee
Inequality \eqref{1124} follows from Property \eqref{pr:sz}Ê
in Appendix~\ref{s2}. 
Taking into account \eqref{1121} we see that \e{1123} will follows from \e{1112} if we prove that, for any $g$ in $L^2(\infzerof\times\xR)$, there is 
an $\ell^2(\xZ)$-sequence $(c_j)_j$ such that
$$
\sup_{z\le 0}\lA \inin^0 K(z,z')\bigl(\Delta_j g\bigr)(z',\cdot)\, dz'\rA_{L^2}\le c_j 2^{j/2}\lA g\rA_{L^2L^2}
$$
for any $j$ in $\xZ$. According to the definition of $K$, the left hand side of this inequality is bounded from above 
in terms of
$$
\lA \inin^0 e^{-|z-z'|\la \xi\ra}\la \xi\ra \indicator{C^{-1}<2^{-j}\la \xi\ra<C}\widehat{\Delta_j g}(z',\xi)\, dz'\rA
_{L^2(d\xi)}
$$
which has the wanted upper bound by Cauchy-Schwarz. 

To prove \eqref{1116}, we rewrite the second component of equality \eqref{1119} as 
\be\label{1125}
(1+\eta'^2)\pz\va-\eta'\px\va=e^{z\Dx}\Dx\psi
+\inin^0 \bigl[ K(z,z')M(\eta')\cdot \nabla_{x,z}\va(z',\cdot)\bigr]_2\, dz'
\ee
where $[\cdot]_2$ stands for the second component. 
By \eqref{1114}Ê
and \eqref{1123}, we conclude that \eqref{1116}Ê
holds. 

We notice also that the right hand side of \eqref{1125} is a continuous function of $z$ with values in $\h{-1/2}$. 
This is trivial for the $e^{z\Dx}\Dx\psi$ contribution. For the integral term, it suffices to show that 
if $g$ is in $L^2(\infzerof\times\xR)$, then 
$\lA \inin^0 \bigl[ K(z,z')-K(z_0,z')\bigr] g(z',\cdot)\, dz'\rA_{\h{-1/2}}$ goes to 
zero if $z$ goes to $z_0$. This reduces to showing that
$$
\lA \inin^0 \big\vert e^{-|z\pm z'|\la \xi\ra}-e^{-|z_0\pm z'| \la\xi\ra}\big\vert \la \xi\ra^{\mez} 
\big\vert \widehat{g}(z',\xi)\big\vert\, dz'\rA_{L^2(d\xi)}
$$
goes to zero if $z$ goes to $z_0$. This follows by dominated convergence, 
from Cauchy-Schwarz and the fact that
$$
C(z,z_0,\xi)=\inin^0 \big\vert e^{-|z\pm z'|\la \xi\ra}-e^{-|z_0\pm z'| \la\xi\ra}\big\vert^2 
\la \xi\ra \, dz'
$$
is uniformly bounded and goes to zero as $z$ goes to $z_0$ at fixed $\xi$. 

The same proof shows that $\px\va$ is also continuous on $]-\infty,0]$ with values 
in $\h{-\mez}(\xR)\subset H^{-\mez}(\xR)$. Using \e{1125} to express $\pz\va$ from 
$(1+\eta'^2)\pz\va-\eta'\px\va$ and $\px\va$, we conclude that $\pz\va$ is also continuous 
with values in $H^{-\mez}(\xR)$.

This concludes the proof of $i)$ of Proposition~\ref{ref:116}. 

$i)$ bis. By $i)$, we only need to study large frequencies. We notice that if 
$\psi$ is in $\h{\mez,\mu+\mez}$, $e^{z\Dx}\begin{pmatrix} \px\psi \\ \Dx\psi\end{pmatrix}$ 
is in $L^2(]-\infty,0],H^{\mu+\mez})$. Moreover, we have seen after 
\e{1120} that $U\rightarrow \int_{-\infty}^0 K(z,z')U(z',\cdot)\, dz'$ 
is bounded on $L^2L^2$. Consequently, for any $j>0$
$$
\lA \Delta_j \int_{-\infty}^0 K(z,z')M(\eta')\nabla_{x,z}\va(z',\cdot)\, dz'\rA_{L^2L^2} 
\le C \blA \Delta_j \bigl[ M(\eta')\nabla_{x,z}\va\bigr]\brA_{L^2L^2}.
$$
Since $\gamma>\mu+\tdm$, we have the product law 
$\eC{\gamma-1}\cdot H^{\mu+\mez}\subset H^{\mu+\mez}$ so the \rhs 
of the preceding equality is bounded from above by 
$$
\Cetagamma \etapetitgamma 2^{-j(\mu+\mez)}c_j(z')\blA \nabla_{x,z}\va\brA_{L^2 H^{\mu+\mez}}.
$$
where $\sum_j \lA c_j(z')\rA_{L^2(dz')}^2<+\infty$. We conclude that
$$
\lA \nabla_{x,z}\va(z,\cdot)-e^{z\Dx}\begin{pmatrix} \px\psi \\ \Dx\psi\end{pmatrix}
\rA_{L^2(]-\infty,0],H^{\mu+\mez})}
\le \Cetagamma \etapetitgamma \lA \nabla_{x,z}\va \rA_{L^2H^{\mu+\mez}},
$$
so that the fixed point giving $\va$ provides a solution in $E$ with 
$\nabla_{x,z}\va \in L^2(]-\infty,0],H^{\mu+\mez})$ and 
$\lA \nabla_{x,z}\va \rA_{L^2H^{\mu+\mez}}\le C \blA \Dxmez \psi\brA_{H^{\mu+\mez}}$. 

Let us check that $(1+\eta'^2)\pz\va-\eta'\px\va$ is in $L^\infty(]-\infty,0],\h{-\mez,\mu+\mez})$. 
The case of low frequencies follows again from $i)$. Thus, by \e{1125}, we just need to stuy for $j>0$ 
the $L^2$-norms in $x$ of 
\begin{align*}
&\Delta_j e^{z\Dx}\Dx \psi\\
&\Delta_j \int_{-\infty}^0 \Bigl[ K(z,z')M(\eta')\nabla_{x,z}\va(z',\cdot)\Bigr]_2 \, dz'.
\end{align*}
The $L^2(dx)$-norm of the first expression is bounded uniformly in $z\le 0$ 
by
$$
C2^{j/2}\blA \Delta_j \Dxmez \psi\brA_{L^2}\le C2^{-j\mu}\blA \Delta_j \Dxmez \psi\brA_{H^{\mu+\mez}}.
$$
On the other hand, the $L^2$-norm of the second quantity is smaller than
\be\label{1125a}
2^{-j(\mu+\mez)}\lA \int_{\infty}^0 e^{-|z-z'|\la\xi\ra} 
\indicator{C^{-1} 2^j<|\xi|<C 2^j} \la \xi\ra g_j(z',\xi)\, dz'\rA_{L^2(d\xi)}
\ee
where 
$$
g_j(z',\xi)=2^{j(\mu+\mez)} \widehat{\Delta_j \big[ M(\eta')\nabla_{x,z'}\va(z',\cdot)\big]}(\xi).
$$
By the product lax $\eC{\gamma-1}\cdot H^{\mu+\mez}\subset H^{\mu+\mez}$, we know 
that
$$
\sum_{j>0}\lA g_j\rA_{L^2L^2}^2\le \Cetagamma \lA \nabla_{x,z}\va\rA_{L^2 H^{\mu+\mez}}^2.
$$
Cauchy-Schwarz then shows that \e{1125a} is bounded from above by $C2^{-j\mu} \lA g_j\rA_{L^2L^2}$. 
This gives the wanted inequality \e{1116a}. The continuity is established as in $i)$. 

Before starting the proof of $ii)$, 
we state the following lemma.
\begin{lemm}\label{ref:117}
Let $\tilde{\phi}$ be in $C^\infty_0(\xR^*)$, $\tilde{\chi}$ in $C^\infty_0(\xR)$ with 
$\tilde{\chi}$ equal to one close to zero. Let $b$ be some function homogeneous 
of degree $r>0$, analytic outside $0$. For $j$ in $\xN^*$, $z$, $z'\le 0$, $x\in \xR$, 
define
\be\label{1126}
k_j^{\pm}(z,z',x)=\frac{1}{2\pi}\int e^{ix\xi-|z\pm z'|\la \xi\ra}
b(\xi)\tilde{\va}\bigl(2^{-j}\xi\bigr)\, d\xi.
\ee
Denote by $k_0^{\pm}(z,z',x)$ the similar integral with 
$\tilde{\va}\bigl(2^{-j}\xi\bigr)$ replaced by $\tilde{\chi}(\xi)$. 
There is $C>0$ such that for any $j$ in $\xN^*$,
\be\label{1127}
\sup_{z\le 0}\lA \int_{\xR} k_j^{\pm}(z,0,x-x')g(x')\, dx'\rA_{L^\infty(dx)}
\le C 2^{jr}\lA g\rA_{L^\infty}
\ee
and 
\be\label{1128}
\sup_{z\le 0}\lA \inin^{0}
\int_{\xR} k_j^{\pm}(z,z',x-x')g(z',x')\, dz' dx'\rA_{L^\infty(dx)}
\le C 2^{j(r-1)}\lA g\rA_{L^\infty L^\infty}.
\ee
Moreover, when $0<r<1$,
\be\label{1129}
\sup_{z\le 0}\lA \int_{\xR} k_0^{\pm}(z,0,x-x')g(x')\, dx'\rA_{L^\infty(dx)}
\le C \lA g\rA_{L^\infty}
\ee
and if $0<r\le1$ and $p\in ]1,1/(1-r)[$,
\be\label{1130}
\sup_{z\le 0}\lA \inin^{0}
\int_{\xR} k_0^{\pm}(z,z',x-x')g(z',x')\, dz' dx'\rA_{L^\infty(dx)}
\le C \lA g\rA_{L^\infty L^p}.
\ee
\end{lemm}
\begin{proof}
For $j$ in $\xN^*$, we perform the change of variables $\xi=2^j \xi'$ in 
\eqref{1126}. Making then $\partial_{\xi'}$-integration by parts, we get a bound
$$
\big\vert k_j^{\pm}(z,z',x)\big\vert\le C_N 2^{j(1+r)}\bigl( 1+2^j |x| 
+2^j |z\pm z'|\bigr)^{-N}
$$
for any $N$ in $\xN$. This implies immediately \eqref{1127} and \eqref{1128}. To treat the case $j=0$, we remark that, in the expression
$$
\int e^{ix\xi-|z\pm z'|\la \xi\ra} b(\xi)\tilde{\chi}(\xi)\, d\xi
$$
we may deform in the complex domain the integration contour close to $\xi=0$, replacing $\xi$ by $\xi+i \eps (\sign x)\xi$. We obtain
\be\label{1131}
\big\vert k_0^{\pm}(z,z',x)\big\vert
\le C\bigl( 1+|x|+|z-z'|\bigr)^{-1-r}.
\ee
Since $r>0$, \eqref{1129} follows at once. To get \eqref{1130}, we bound 
the left hand side by 
$$
\left(\sup_{z\le 0} \inin^{0}
\left[ \int_{\xR} \big\vert k_0^{\pm}(z,z',x')\big\vert^{p'} dx'\right]^{\frac{1}{p'}}\, dz' \right) \lA g\rA_{L^\infty L^p}
$$
where $p'>1$ is the conjugate exponent of $p$. 
Using the bound \eqref{1131} and $r>1/p'$, we get 
the finiteness of this quantity.
\end{proof}
{\em End of the proof of Proposition~\ref{ref:116}.} To prove $ii)$ 
of the proposition, it is enough 
to show that under the smallness condition \eqref{1117}, 
the fixed point problem~\e{117} has a unique (up to constants) 
solution in the subspace of those $\va$ in $E$ such that 
$\sup_{z\le 0}\blA \nabla_{x,z}\va\brA_{\eC{\gamma-1}}<+\infty$. 
According to \e{1119}, this will hold if we prove that
\be\label{1132}
\sup_{z\le 0} \blA e^{z\Dx} (\px\psi,\Dx\psi)\brA_{\eC{\gamma-1}}
\le C \blA \Dxmez\psi\brA_{\eC{\gamma-\mez}},
\ee
\be\label{1133}
\sup_{z\le 0} \blA \bigl(\eta'-\eta'^2\pz\va\bigr)(z,\cdot)\brA_{\eC{\gamma-1}}\\
\le \Cetagamma\etapetitgamma\blA \nabla_{x,z}\va\brA_{L^\infty\eC{\gamma-1}},
\ee
and
\begin{multline}\label{1134}
\sup_{z\le 0} \lA \inin^0 K(z,z')M(\eta')\cdot\nabla\va (z',\cdot)\, dz'\rA_{\eC{\gamma-1}}\\
\le \Cetagamma\Bigl(\etapetitgamma+\blA \eta'\brA_{\eC{-1}}^\mez 
\blA \eta'\brA_{H^{-1}}^\mez\Bigr)\blA \nabla_{x,z}\va\brA_{L^\infty\eC{\gamma-1}}.
\end{multline}
Moreover, these inequalities, \e{1119} and the smallness condition 
\eqref{1117} imply that estimate \e{1118} holds. 

We notice that \e{1133} is trivial. To prove \e{1132}, we write 
the function in the left hand side as $e^{z\Dx}b(\OD)\Dxmez \psi$ 
for some $b(\xi)$ homogeneous of degree $1/2$. Then using the notations of Lemma~\ref{ref:117}, for $j>0$,
\begin{align*}
&\Delta_j\bigl(e^{z\Dx}b(\OD)\Dxmez \psi\bigr) 
=\int k_j^{+}(z,0,x-x')\bigl[ \Dxmez \Delta_j\psi \bigr](x')\,dx',\\
&S_0 \bigl(e^{z\Dx}b(\OD)\Dxmez \psi\bigr) 
=\int k_0^{+}(z,0,x-x')\bigl[ \Dxmez S_0 \psi\bigr](x')\,dx'.
\end{align*}
Estimates \e{1127}, \e{1129} with $r=1/2$ show that the $L^\infty$-norm 
of these quantities is bounded by 
$2^{-j(\gamma-1/2)}\blA\Dxmez\psi\brA_{\eC{\gamma-1/2}}$ 
uniformly in $z\le 0$, whence \e{1132}. 

To prove \e{1133}, we notice that by \e{1120}, the operator associating 
to a $\xR^2$-valued function $g$, 
$\inin^0 K(z,z')\Delta_j g(z',\cdot)\, dz'$ 
(resp.\ $\inin^0 K(z,z')S_0 g(z',\cdot)\, dz'$) may be written from
$$
\inin^0 k_j^{\pm}(z,z',x-x')\Delta_j g_\ell(z',x')\, dx' 
$$
(resp.\ the same expression with $j=0$ and $\Delta_j$ replaced by $S_0$), 
where $g_\ell$ is a component of $g$, and $k_j$ is given by \e{1126} with $b$ 
homogeneous of degree $1$. It follows from \e{1128} with $r=1$ that 
\begin{multline*}
\sup_{z\le 0} \lA \Delta_j \inin^0 K(z,z')M(\eta')\cdot\nabla\va (z',\cdot)\, dz'\rA_{L^\infty} \\
\le C\blA \Delta_j M(\eta')\cdot\nabla\va\brA_{L^\infty L^\infty} 
\le C2^{-j(\gamma-1)} 
\sup_{z'\le 0} \lA \Delta_j M(\eta')\cdot\nabla\va(z',\cdot)\rA_{\eC{\gamma-1}}
\end{multline*} 
Since the Hilbert transform is bounded on the subspace of those $f$ in
$\eC{\gamma-1}$ whose Fourier transform vanishes on a neighborhood 
of the origin, the expression 
\e{1120} of $M(\eta')$ shows that this quantity is smaller than
\be\label{1134a}
\Cetagamma\etapetitgamma\sup_{z'\le 0} \lA \nabla\va(z',\cdot)\rA_{\eC{\gamma-1}}2^{-j(\gamma-1)}.
\ee
On the other hand, \e{1130}Ê
shows that
\be\label{1135}
\sup_{z\le 0} \lA S_0
\inin^0 K(z,z')M(\eta')\cdot\nabla\va (z',\cdot)\, dz'\rA_{L^\infty} 
\le C\blA S_0 M(\eta')\cdot\nabla\va\brA_{L^\infty L^p}. 
\ee
Since the Hilbert transform involved in the definition of $M(\eta')$ is 
bounded on $L^p$ for $1<p<\infty$, we see that we are reduced to estimating 
$\blA S_0(\eta'\nabla\va)\brA_{L^\infty L^p}$ and 
$\blA S_0(\eta'^2\nabla\va)\brA_{L^\infty L^p}$. Taking $p>2$, 
we conclude that \e{1135} is bounded from above by a multiple of 
$$
\blA S_0(\eta'\nabla\va)\brA_{L^\infty L^2}^{\frac{2}{p}}
\blA S_0(\eta'\nabla\va)\brA_{L^\infty L^\infty}^{1-\frac{2}{p}}+
\blA S_0(\eta'^2\nabla\va)\brA_{L^\infty L^2}^{\frac{2}{p}}
\blA S_0(\eta'^2\nabla\va)\brA_{L^\infty L^\infty}^{1-\frac{2}{p}}.
$$
We write for $k=1,2$, 
\begin{align*}
\blA S_0(\eta'^k\nabla\va)\brA_{L^\infty L^2}
\le \blA \eta'^2\nabla\va\brA_{L^\infty H^{-1}}
\le C \lA \eta'\rA_{H^{-1}}\lA \eta'\rA_{\eC{\gamma-1}}^{k-1}
\lA \nabla\va\rA_{L^\infty \eC{\gamma-1}},\\
\blA S_0(\eta'^k\nabla\va)\brA_{L^\infty L^\infty}
\le \blA \eta'^2\nabla\va\brA_{L^\infty \eC{-1}}
\le C \lA \eta'\rA_{\eC{-1}}\lA \eta'\rA_{\eC{\gamma-1}}^{k-1}
\lA \nabla\va\rA_{L^\infty \eC{\gamma-1}},
\end{align*}
using property \eqref{pr:sz} of the Appendix~\ref{s2} 
and the fact that the product is continuous 
from $\eC{\gamma-1}\times \eC{-1}$ to $\eC{-1}$. Taking 
for instance $p=4$, we get a bound for \e{1135} of the form 
$\Cetagamma \lA \eta'\rA_{H^{-1}}^{1/2}\lA \eta'\rA_{\eC{-1}}^{1/2}
\lA \nabla\va\rA_{L^\infty \eC{\gamma-1}}$. Combining with 
\e{1134a}, we obtain \e{1134}. 

Let us prove the last assertion in $ii)$. 
If we cut-off spectrally the quantity to be estimated outside a neighborhood of zero, 
the upper bound follows from \e{1118}. We have thus to study 
$$
\sup_{z\le 0}\blA \Dx^{-\mez+\theta}
\widetilde{\chi}(D_x)\big( (1+\eta'^2)\pz\va-\eta'\px\va\big)(z,\cdot)\brA_{L^\infty}
$$
where $\widetilde{\chi}\in C^\infty_0(\xR)$ is equal to one close to zero. By \e{1125}, the wanted inequality will follow from 
\be\label{1135a}
\ba
&\sup_{z\le 0}\blA \widetilde{\chi}(D_x)e^{z\Dx} \Dx^{\mez+\theta}\psi\brA_{L^\infty} 
\le C \blA \Dxmez \psi\brA_{\eC{\gamma-\mez}},\\
&\sup_{z\le 0}\lA \int_{-\infty}^0 \Bigl[ \Dx^{-\mez+\theta}\widetilde{\chi}(D_x)
K(z,z')M(\eta')\cdot\nabla_{x,z}\va(z',\cdot)\Bigr]_2\, dz'\rA\\
&\qquad\qquad \le \Cetagamma\blA \eta'\brA_{H^{-1}}^{1-2\theta}\blA \eta'\brA_{\eC{-1}}^{\theta'} 
\blA \nabla_{x,z}\va\brA_{L^\infty \eC{\gamma-1}}
\ea
\ee
from the boundedness assumption of $\blA \eta'\brA_{H^{-1}}^{1-2\theta}\blA \eta'\brA_{\eC{-1}}^{\theta'} $ 
and from \e{1118}. The first estimate follows from \e{1129} with 
$r=\mez+\theta$, as in the proof of \e{1132}. To prove the second inequality, we 
bound its left hand side from quantities
\be\label{1135b}
\sup_{z\le 0}\lA \int_{-\infty}^0 \int_\xR k_0^{\pm}(z,z',x-x')g(z',x')\, dz'\, dx\rA_{L^\infty}
\ee
where $k_0^{\pm}$ is given by an integral of the form \e{1126} with 
$\widetilde{\va}(2^{-j}\xi)$ replaced by $\widetilde{\chi}(\xi)$ and $b$ homogeneous 
of degree $r=\mez+\theta$, and 
where $g$ is any of the components of $S_0\bigl( M(\eta')\nabla_{x,z}\va(z',x')\bigr)$. 
By \e{1130}, we bound \e{1135b} by $C\lA g\rA_{L^\infty L^p}$ if $p<1/(\mez-\theta)$. 
We have seen above that this quantity is smaller than 
$$
\Cetagamma\blA \eta'\brA_{H^{-1}}^{2/p}\blA \eta'\brA_{\eC{-1}}^{1-2/p}\lA \nabla\va\rA_{L^\infty \eC{\gamma-1}}.
$$
Taking $\frac{1}{p}=\mez-\theta'$, we get the conclusion.  
\end{proof}

\begin{coro}\label{ref:118} Let $\eta$ be in $L^2\cap \eC{\gamma}(\xR)$ 
satisfying the condition \eqref{1117}. We define for $\psi$ in $\Hmez$ the Dirichlet-Neumann operator $G(\eta)$ as 
\be\label{1136}
G(\eta)\psi=\bigl[ (1+\eta'^2)\partial_z\varphi-\eta'\px\varphi\bigr]\big\arrowvert_{z=0}
\ee
where $\varphi$ is given by Proposition~\ref{ref:116}. Then 
$G(\eta)$ is bounded from $\dot{H}^{1/2}(\xR)$ 
to $\dot{H}^{-1/2}(\xR)$ and satisfied an estimate
\be\label{1137}
\lA G(\eta)\psi\rA_{\dot{H}^{-1/2}}\le 
\Cetagamma\Dxmezpsi.
\ee
In particular, if we define $G_{1/2}(\eta)=\Dx^{-\mez}G(\eta)$, we 
obtain a bounded operator from $\dot{H}^{1/2}(\xR)$ to $L^2(\xR)$ satisfying 
\be\label{1138}
\lA G_{1/2}(\eta)\psi\rA_{L^{2}}\le 
\Cetagamma \Dxmezpsi.
\ee
Moreover, $G(\eta)$ satisfies when $\psi$ is in $\C{\mez,\gamma-\mez}(\xR)$
\be\label{1139}
\lA G(\eta)\psi\rA_{\eC{\gamma-1}}\le 
\Cetagamma \blA \Dxmez \psi\brA_{\eC{\gamma-\mez}}.
\ee
where $C(\cdot)$ is a non decreasing continuous function of 
its argument.

If we assume moreover that for some $0<\theta'<\theta<\mez$, 
$\blA \eta'\brA_{H^{-1}}^{1-2\theta'}\blA \eta'\brA_{\eC{-1}}^{2\theta'}$ is bounded, then 
$\Dx^{-\mez+\theta}G(\eta)$ satifies
\be\label{1140}
\blA \Dx^{-\mez+\theta}G(\eta)\psi\brA_{\eC{\gamma-\mez-\theta}}
\le C\bigl( \blA \eta'\brA_{\eC{\gamma-1}}\bigr) \blA \Dxmez\psi\brA_{\eC{\gamma-\mez}}.
\ee
\end{coro}
\begin{proof}
Inequalities \e{1137} and \e{1138} follow from \e{1116}. The bound \e{1139} 
is a consequence of \e{1118}, the definition \e{1136} of $G(\eta)\psi$ and the fact that $\eC{\gamma-1}$ is an algebra.
\end{proof}

\section{Main Sobolev estimate}

Consider  a couple of real valued functions $(\eta,\psi)$ 
defined on $\xR\times \xR$ satisfying for $t\ge 1$ the system 
\begin{equation}\label{121}
\left\{
\begin{aligned}
&\partial_t \eta=G(\eta)\psi,\\
&\partial_t \psi + \eta+ \frac{1}{2}(\partial_x \psi)^2
-\frac{1}{2(1+(\partial_x\eta)^2)}\bigl(G(\eta)\psi+\partial_x  \eta \partial_x \psi\bigr)^2= 0,
\end{aligned}
\right.
\end{equation}
with Cauchy data small enough in 
a convenient space. 

The operator $G(\eta)$ in \eqref{121} and in the rest of this paper is the one defined by \eqref{1136} in Corollary \ref{ref:118}. We set, for $\eta,\psi$ smooth enough 
and small enough functions
\be\label{122}
\B(\eta)\psi=\frac{G(\eta)\psi+\partial_x  \eta \partial_x \psi}
{1+(\partial_x\eta)^2}\cdot
\ee

Let us recall a known local existence result 
(see~\cite{WuInvent,LannesLivre,ABZ3}).  

\begin{prop}\label{ref:121}
Let  $\gamma$ be in $]7/2,+\infty[\setminus \mez\xN$, $s\in \xN$ with $s>2\gamma-1/2$. 
There are $\delta_0>0$, $T>1$ such that for any couple $(\eta_0,\psi_0)$ 
in $H^{s}(\xR)\times \h{\mez,\gamma}(\xR)$ satisfying
\be\label{123}
\psi_0-T_{\B(\eta_0)\psi_0}\eta_0 \in \h{\mez,s}(\xR), \quad 
\lA \eta_0\rA_{\eC{\gamma}}+\blA \Dxmez \psi_0\brA_{\eC{\gamma-\mez}}<\delta_0,
\ee
equation \eqref{121} with Cauchy data $\eta\arrowvert_{t=1}=\eta_0$, $\psi\arrowvert_{t=1}
=\psi_0$ has a unique solution $(\eta,\psi)$ which is continuous on $[1,T]$ with values in 
\be\label{124}
\left\{\, (\eta,\psi)\in H^s(\xR)\times \h{\mez,\gamma}(\xR)\,;\, \psi-
T_{B(\eta)\psi}\eta\in \h{\mez,s}(\xR)\,\right\}.
\ee
Moreover, if the data are $O(\eps)$ on the indicated spaces, then $T\ge c/\eps$.
\end{prop}
\begin{remas}The assumption $\psi_0\in \h{\mez,\gamma}$ implies that $\psi_0$ is in 
$\C{\mez,\gamma-\mez}$ so that Corollary~\ref{ref:118} shows that 
$G(\eta_0)\psi_0$ whence $\B(\eta_0)\psi_0$ is in 
$\eC{\gamma-1}\subset L^\infty$. Consequently, by the 
first equality in \eqref{123}, $\Dxmez\psi$ is in $H^{s-\mez}\subset \eC{\gamma-\mez}$ as 
our assumption on $s$ implies that $s>\gamma+1/2$. This gives sense to the second assumption~
\eqref{123}. 

--- The well-known difficulty in the analysis of equation~\eqref{121} is that writing 
energy inequalities on the function $(\eta,\Dxmez\psi)$ makes appear an apparent loss 
of half a derivative. The way to circumvent that difficulty is now well-known: it is to bound the 
energy not of $(\eta,\Dxmez\psi)$, but of $(\eta,\Dxmez\omega)$, where $\omega$ is 
the ``good unknown'' of Alinhac, defined by $\omega=\psi-T_{\B(\eta)\psi}\eta$ (see Chapter~\ref{S:21}). This explains why the regularity assumption~\eqref{123} on the Cauchy data concerns 
$\psi_0-T_{\B(\eta_0)\psi_0}\eta_0$ and not $\psi_0$ itself. Notice that this function is in 
$\h{\mez,s}$ while $\psi_0$ itself, written 
from $\psi_0=\omega_0+T_{\B(\eta_0)\psi_0}\eta_0$ is only in $\h{\mez,s-\mez}$, because 
of the $H^s$-regularity of $\eta_0$.

--- By \eqref{1139} if $\psi$ is in $\C{\mez,\gamma-\mez}$ and $\eta$ is in $\eC{\gamma}$, 
$G(\eta)\psi$ is in $\eC{\gamma-1}$, so $\B(\eta)\psi$ is also in $\eC{\gamma-1}$ with 
$\lA \B(\eta)\psi\rA_{\eC{\gamma-1}}\le \Cetagamma \blA \Dxmez\psi\brA_{\eC{\gamma-\mez}}$. 
In particular, as a paraproduct with an $L^\infty$-function acts on any H\"older space,
$$
\blA \Dxmez T_{\B(\eta)\psi}\eta\brA_{\eC{\gamma-\mez}}
\le \Cetagamma \lA \eta\rA_{\eC{\gamma}}\blA \Dxmez\psi\brA_{\eC{\gamma-\mez}}.
$$
This shows that for $\lA \eta\rA_{\eC{\gamma}}$ small enough, $\psi\rightarrow \psi-T_{\B(\eta)\psi}\eta$ is an isomorphism from $\C{\mez,\gamma-\mez}$ to itself. In particular, if we are given $\omega$ in $\h{\mez,s}\subset \C{\mez,\gamma-\mez}$, we may find a unique $\psi$ in 
$\C{\mez,\gamma-\mez}$ such that $\omega=\psi-T_{\B(\eta)\psi}\eta$. In other words, 
when interested only in $\eC{\gamma-\mez}$-estimates for $\Dxmez\omega$, we may as well 
establish them on $\Dxmez\psi$ instead, as soon as 
$\lA \eta\rA_{\eC{\gamma}}$ stays small enough. 

--- We check in Appendix~\ref{S:A4} that our assumption \eqref{123} implies the 
one made by Lannes in \cite{LannesLivre} so that Proposition~\ref{ref:121} 
follows from Theorem~$4.35$ in \cite{LannesLivre}.

\end{remas}

Let us state now our main result. 

We fix real numbers 
$s,s_1,s_0$ satisfying, for some large enough numbers $a$ and $\gamma$ 
with $\gamma\not\in \mez\xN$ and $a\gg \gamma$, 
the following conditions
\be\label{125}
s,s_0,s_1\in\xN,\quad 
s-a\ge s_1\ge s_0\ge \frac{s}{2}+\gamma.
\ee

We shall prove $L^2$-estimates for the action of the vector field 
\be\label{131}
Z=t\partial_t+2x\px
\ee
on the unknown in equation \eqref{121}. We introduce the following notation: 


For $(\eta,\psi)$ a local smooth enough solution of \eqref{121}, we set $\omega=\psi-T_{\B(\eta)\psi}\eta$ and for any integer $k\le s_1$,\index{Norms!$M_s^{(k)}$}
\be\label{132}
M_s^{(k)}(t)=\sum_{p=0}^k \bigl( 
\blA Z^p \eta(t,\cdot)\brA_{H^{s-p}}
+\blA \Dxmez Z^p \omega(t,\cdot)\brA_{H^{s-p}}\bigr).
\ee
In the same way, for $\rho$ a positive number (that will be larger 
than $s_0$), we set for $k\le s_0$, \index{Norms!$N_\rho^{(s_0)}$}
\be\label{133}
N_\rho^{(k)}(t)=\sum_{p=0}^k \bigl( 
\blA Z^p \eta(t,\cdot)\brA_{\eC{\rho-p}}
+\blA \Dxmez Z^p \psi(t,\cdot)\brA_{\eC{\rho-p}}\bigr).
\ee

We consider the set of functions  $(\eta_0,\psi_0)$ satisfying for any integer $p\le s_1$
\be\label{126}
\begin{aligned}
&(x\px)^p\eta_0\in H^{s-p}(\xR),\quad (x\px)^p\psi_0\in \h{\mez,s-p-\mez}(\xR),\\
&(x\px)^p\bigl( \psi_0-T_{\B(\eta_0)\psi_0}\eta_0\bigr)\in \h{\mez,s-p}(\xR),
\end{aligned}
\ee
and such that the norm of the above functions in the indicated spaces is 
smaller than $1$. For $\epsilon\in ]0,1[$, we solve equation \eqref{121} with Cauchy data $\eta\arrowvert_{t=1}=\eps\eta_0$,
$\psi\arrowvert_{t=1}=\eps\psi_0$. According to that proposition, for any $T_0>1$, there is $\epsilon'_0>0$ such that 
if $\eps<\eps_0'$, equation \eqref{121} has a solution for 
$t\in [1,T_0]$. Moreover, by Proposition~\ref{ref:A4}, 
assumptions \eqref{126} remain valid at $t=T_0$. 

Our main result is the following:

\begin{theo}\label{ref:131}
There is a constant $B_2>0$ such that $M_s^{(s_1)}(T_0)<\uq B_2 \eps$, and for any constants $B_\infty>0$, 
$B_\infty'>0$ there is $\eps_0$ such that the following holds: 
Let $T>T_0$ be a number such that equation \eqref{121} 
with Cauchy data satisfying \eqref{126} has a solution satisfying the regularity properties of Proposition~\ref{ref:121} 
on $[T_0,T[\times \xR$ and such that 

$i)$ For any $t\in [T_0,T[$, and any $\eps\in ]0,\eps_0]$, 
\be\label{134}
\blA \Dxmez \psi(t,\cdot)\brA_{\eC{\gamma-\mez}}
+\lA \eta(t,\cdot)\rA_{\eC{\gamma}}\le B_\infty\eps t^{-\mez}.
\ee 

$ii)$ For any $t\in [T_0,T[$, any $\eps\in ]0,\eps_0]$
\be\label{135}
N_\rho^{(s_0)}(t)\le B_\infty \eps t^{-\mez+B_\infty' \eps^2}.
\ee
Then, there is an increasing sequence $(\delta_k)_{0\le k\le s_1}$, depending only on $B_\infty'$ and $\eps$ with $\delta_{s_1}<1/32$ such that for any $t$ in $[T_0,T[$, any 
$\eps$ in $]0,\eps_0]$, any $k\le s_1$,
\be\label{136}
M_s^{(k)}(t)\le \mez B_2\eps t^{\delta_k}.
\ee
\end{theo}

The rest of this paper will be devoted to the proof of the above theorem. In~\cite{AlDelMain}, it is shown that this result,
together with an $L^\infty$-estimate of the solutions of \eqref{121}, implies global existence and modified scattering for
solutions of \eqref{121} with Cauchy data $\varepsilon(\eta_0,\psi_0)$, where  $(\eta_0,\psi_0)$ satisfy\eqref{126}  and
$\varepsilon$ is small enough. For the reader's convenience, we reproduce below these two statements. The proofs are given in
\cite{AlDelMain}.

The $L^\infty$ conterpart of the Sobolev estimates of Theorem~\ref{ref:131} is the following:

\begin{theo}\label{ref:132}
Let $T>T_0$ be a number such that the 
equation \e{121} with Cauchy data satisfying \e{126} has a solution on $[T_0,T[\times \xR$ satisfying the regularity properties of Proposition~$\ref{ref:121}$. Assume 
that, for some constant $B_2>0$, for any $t\in [T_0,T[$, any $\eps$ in $]0,1]$, 
any $k\le s_1$, 
\be\label{137}
\ba
&M_{s}^{(k)}(t)\le B_2 \eps t^{\delta_k},\\
&N_{\rho}^{(s_0)}(t)\le \sqrt{\eps}<1
\ea
\ee
Then there are constants $B_\infty,B_\infty'>0$ depending only on $B_2$ and some 
$\eps_0'\in ]0,1]$, independent of 
$B_2$, 
such that, for any $t$ in $[T_0,T[$, any $\eps$ 
in $]0,\eps_0']$,
\be\label{138}
\ba
&N_\rho^{(s_0)}(t)\le \mez B_\infty \eps t^{-\mez+\eps^2 B_\infty'},\\
&\blA \Dxmez \psi(t,\cdot)\brA_{\eC{\gamma-\mez}}+
\lA \eta(t,\cdot)\rA_{\eC{\gamma}}\le \mez B_\infty \eps t^{-\mez}.
\ea
\ee
\end{theo}

The main result of global existence for the water waves equation with small Cauchy data deduced in~\cite{AlDelMain} from the
above estimates may be stated as:

\begin{theo}\label{ref:122}
There is $\eps_0>0$ such that for any $\eps\in ]0,\eps_0]$, any couple 
of functions $(\eta_0,\psi_0)$ satisfying condition \eqref{126},
and whose norm  in the indicated spaces is 
smaller than $1$, equation \eqref{121} with the Cauchy data 
$\eta\arrowvert_{t=1}=\eps\eta_0$, $\psi\arrowvert_{t=1}=\eps\psi_0$ 
has a unique solution 
$(\eta,\psi)$ which is defined and continuous on $[1,+\infty[$ with values 
in the set~\eqref{124}.

Moreover, $u=\Dxmez \psi+i\eta$ admits the following 
asymptotic expansion as $t$ goes to $+\infty$: 

There is a continuous function $\underline{\alpha}\colon \xR
\rightarrow \xC$, depending of $\eps$ but 
bounded uniformly in $\eps$, such that
\be\label{127}
u(t,x)=\frac{\eps}{\sqrt{t}}\underline{\alpha}\Bigl(\frac{x}{t}\Bigr)
\exp\Bigl(\frac{it}{4|x/t|}+\frac{i\eps^2}{64}\frac{\la \underline{\alpha}(x/t)\ra^2}{\la x/t\ra^5}\log (t)\Bigr)
+\eps t^{-\mez-\kappa}\rho(t,x)
\ee
where $\kappa$ is some positive number and $\rho$ 
is a function uniformly bounded for $t\ge 1$, $\eps\in ]0,\eps_0]$.
\end{theo}

\bigskip

In the rest of this paper, we prove Theorem~\ref{ref:131} i.e.\ we show estimates for the Sobolev norms 
$M_s^{(k)}(t)$ introduced in \eqref{132} assuming {\em a priori\/} H\"older estimates of the form \eqref{135}. To do so, 
we first need to establish a collection of estimates for the Dirichlet-Neumann 
operator. Chapter~\ref{S:21} will be devoted to such a task. Next we have to design a normal form method that 
will allow us to eliminate in the Sobolev energy the contributions coming from the quadratic part of the non-linearity. 
This is the object of Chapter~\ref{S:22}. Chapter~\ref{S:23} is devoted to the commutation of the $Z$-vector field 
to the water waves equation, and in particular to the Dirichlet-Neumann operator. In Chapter~\ref{S:24}, 
combining the results obtained so far, we prove the Sobolev estimates for the action of the $Z$-vector field on the 
solution we are looking for.

\chapter{Estimates for the Dirichlet-Neumann operator}\label{S:21}

The Dirichlet-Neumann operator $G(\eta)$ has been 
defined in the first section of Chapter~\ref{chap:1} 
(see Corollary~\ref{ref:118}) and $\dot{H}^{1/2}$-estimates 
have been obtained for it. The goal of this chapter is to prove 
Sobolev estimates for $G(\eta)$ and related operators. 
We shall make an extensive use of paradifferential operators. 
We refer to Appendix~\ref{s2} for the main definitions and results on this topic. 

We use in this chapter the notations introduced at the beginning of Chapter~\ref{chap:1}, in 
particular for the elliptic operator $P$ introduced in \eqref{111}. We shall consider a couple 
$(\eta,\psi)$ belonging to the set $\Eg{\gamma}$ introduced after the statement of Proposition~\ref{ref:116}. 
This implies in particular that estimates \eqref{1115} and \eqref{1118} hold. 

Given $(\eta,\psi)$ in $\Eg{\gamma}$ we introduce the notations
\index{Dirichlet-Neumann operator!$\B(\eta)\psi$}
\index{Dirichlet-Neumann operator!$V(\eta)\psi$}
\begin{equation}\label{211}
\B(\eta)\psi=\frac{G(\eta)\psi+(\partial_x\eta)( \partial_x\psi)}{1+(\partial_x\eta)^2},\quad
V(\eta)\psi=\partial_x\psi-(\B(\eta)\psi)\partial_x\eta.
\end{equation}
\begin{remas} $i)$ It follows from equality \eqref{1136} and the fact that $\va\az=\psi$ that
\begin{equation}\label{2110}
\left\{
\begin{aligned}
G(\eta)\psi&= (1+(\partial_x  \eta)^2)\partial_z \varphi
-\partial_x \eta \partial_x  \varphi\big\arrowvert_{z=0},\\[0.5ex]
\B(\eta)\psi&=\partial_z\varphi\arrowvert_{z=0},\\[0.5ex]
V(\eta)\psi&=\left(\partial_x\varphi-\partial_x\eta\partial_z\varphi\right)\arrowvert_{z=0}.
\end{aligned}
\right.
\end{equation}
If one goes back to the $(x,y)$-coordinates introduced at the beginning of Section~\ref{S:11}, 
for which the fluid domain $\Omega$ is given by $\{y<\eta(x)\}$ and the velocity 
potential is $\phi(x,y)=\varphi(x,y-\eta(x))$, one 
sees that $\B(\eta)\psi=(\partial_y \phi)\arrowvert_{\partial\Omega}$ 
and $V(\eta)\psi=(\px\phi)\arrowvert_{\partial\Omega}$. 

$ii)$ We rewrite, for further reference, the first equality of \eqref{2110} taking into account \eqref{211}, as 
\begin{equation}\label{212}
\B(\eta)\psi-(\px\eta)V(\eta)\psi=G(\eta)\psi.
\end{equation}
\end{remas}

Finally, we shall eventually denote $\eta'$ instead of $(\px\eta)$ to simplify some expressions.

It follows from \eqref{211}, the estimate \eqref{1118} 
and the classical product rule in H\"older spaces 
(see~Proposition~$8.6.8$ in \cite{Hormander}) that we have 
the following
\begin{lemm}\label{ref:211}
Let~$\gamma\in \pol3,+\infty\por\setminus \mez \xN$. 
There exists 
a non decreasing function~$C\colon \xR_+\rightarrow\xR_+$ such that, 
for all $(\eta,\psi)\in \Eg{\gamma}$,
\be\label{211-1}
\lA G(\eta)\psi\rA_{\eC{\gamma-1}}+
\lA \B(\eta)\psi\rA_{\eC{\gamma-1}}+\lA V(\eta)\psi\rA_{\eC{\gamma-1}}
\le 
\Cetagamma\dalpha.
\ee
\end{lemm}

\section{Main results}

We shall state in this section the main result that will be obtained in this chapter. 
We want to get estimates 
for the Dirichlet-Neumann operator $G(\eta)\psi$, as well as the related operators 
$\B(\eta)\psi$, $V(\eta)\psi$ introduced in \eqref{211}, in terms of Sobolev and H\"older norms 
of $\eta$ and $\psi$. The main result will be expressed 
in terms of the ``good unknown'' of Alinhac $\omega=\omega(\eta)\psi$ defined by the relation
\be\label{2112}
\omega(\eta)\psi=\psi-T_{\B(\eta)\psi}\eta.\index{Dirichlet-Neumann operator!$\omega=\omega(\eta)\psi$, good unknown of Alinhac}
\ee
We shall explain, in the comments following the statement of the next theorem, the interest of working 
with $(\eta,\omega)$ instead of $(\eta,\psi)$. Recall from the introduction that $\omega$ defined 
by \eqref{2112} appears naturally when one introduces the operator of 
paracomposition of Alinhac~\cite{Alipara} associated to the change of variables 
that flattens the boundary $y=\eta(x)$ of the fluid domain, namely 
$(x,y)\mapsto (x,z=y-\eta(x))$. This is a quite optimal way of keeping track of the limited 
smoothness of the change of coordinates. Though we shall not use this point of view 
here, it underlies the computations that will be made at the beginning of the next section. 

Let us now state our main result.

\begin{theo}\label{mainDN}
Let~$(s,\gamma)\in \xR^2$ be such that
$$ 
s-\mez>\gamma>3,\quad \gamma\not\in\mez\xN.
$$
There exists a non decreasing function~$C\colon \xR_+\rightarrow \xR_+$ 
such that for all $(\eta,\psi)$ in $H^{s}(\xR)\times \h{\mez,s-\mez}(\xR)$ such that 
$(\eta,\psi)$ belongs to the set $\Eg{\gamma}$ introduced after the statement 
of Proposition~\ref{ref:116}, the 
following properties hold:

$(i)$ (Tame estimate)
\begin{multline}\label{2113}
\lA G(\eta)\psi\rA_{H^{s-1}}+
\lA \B(\eta)\psi\rA_{H^{s-1}}+\lA V(\eta)\psi\rA_{H^{s-1}}\\
\le C\left(\lA \eta\rA_{\eC{\gamma}}\right)\left\{
\dalpha \lA \eta\rA_{H^s}+\blA \Dxmez \psi \brA_{H^{s-\mez}}\right\}.
\end{multline}

$(ii)$ (Paralinearization) Define~$F(\eta)\psi$ by\index{Dirichlet-Neumann operator!$F(\eta)\psi$}
\begin{equation}\label{n1}
G(\eta)\psi=\Dx \omega-\partial_x \bigl( T_{V(\eta)\psi}\eta\bigr)+F(\eta)\psi.
\end{equation}
Then
\be\label{n2}
\lA F(\eta)\psi\rA_{H^{s+\gamma-4}}\\
\le 
C\left(\lA \eta\rA_{\eC{\gamma}}\right)\left\{
\dalpha \lA \eta\rA_{H^s}+\lA \eta\rA_{\eC{\gamma}}\blA \Dxmez \psi \brA_{H^{s-\mez}}\right\}.
\ee

$(iii)$ (Linearization) 
\begin{multline}\label{2117}
\lA G(\eta)\psi-\Dx\psi\rA_{H^{s-1}}+\lA \B(\eta)\psi-\Dx\psi\rA_{H^{s-1}}
+\lA V(\eta)\psi-\partial_x\psi\rA_{H^{s-1}}\\
\le C\left(\lA \eta\rA_{\eC{\gamma}}\right)\left\{
\dalpha \lA \eta\rA_{H^s}+\lA \eta\rA_{\eC{\gamma}}\blA \Dxmez \psi \brA_{H^{s-\mez}}\right\}.
\end{multline}
\end{theo}

Let us comment on the above statement. 

--- All these estimates are tame: they depend linearly on the Sobolev norms. 
Moreover, we consider the case where $\eta$ and $\psi$ are at exactly 
the same level of regularity (i.e.\ $\eta$ in $H^s$ and 
$\psi$ in $\h{\mez,s-\mez}$). This is important to prove $H^s$-energy estimates for 
the water waves equation. Indeed, as already explained in the 
introduction, we shall write in Chapter~\ref{S:22} the water waves equation 
as a quasi-linear system in the unknowns $(\eta,\Dxmez\omega)$. To be able to 
obtain $H^s$-energy inequalities for this equation, it is important to check that the 
right-hand sides in the inequalities of Theorem~\ref{mainDN} are controlled by 
the $H^s$-norm of $(\eta,\Dxmez\omega)$. Let us show that this property holds. 
To do so notice that by Lemma~\ref{ref:211} 
if $(\eta,\psi)$ belongs to $\Eg{\gamma}$ then 
$\B(\eta)\psi$ belongs to $\eC{\gamma-1}$ so that $\B(\eta)\psi$ is in $L^\infty$. 
Then, as a paraproduct with an $L^\infty$-function acts on any Sobolev spaces, we have
\be\label{n50}
\begin{aligned}
\blA \Dxmez T_{\B(\eta)\psi}\eta\brA_{H^{s-\mez}}&\les \lA \B(\eta)\psi\rA_{L^\infty}
\lA \eta\rA_{H^s}\\
&\le C\left(\lA \eta\rA_{\eC{\gamma}}\right)\dalpha\lA \eta\rA_{H^s}.
\end{aligned}
\ee
Thus 
if we express $\psi$ as $\omega +T_{\B(\eta)\psi}\eta$ then one obtains
\be\label{n3}
\blA \Dxmez \psi \brA_{H^{s-\mez}}\le \blA \Dxmez\omega\brA_{H^{s-\mez}}
+C\left(\lA \eta\rA_{\eC{\gamma}}\right)\dalpha\lA \eta\rA_{H^s}.
\ee
Had we proved the statements of the theorem above with 
$\blA \Dxmez \psi \brA_{H^{s-\mez}}$ replaced with 
$\blA \Dxmez \psi \brA_{H^{s}}$, then this would give 
a bound in terms of $\lA \eta\rA_{H^{s+\mez}}$,  
preventing us to control this quantity from the $H^s$-energy (which is the $H^s$-norm of $(\eta,\Dxmez\omega)$).

--- It has been known since Calder\'on that, for $\eta$ a smooth function, 
$G(\eta)$ is a pseudo-differential operator that, in one dimension, differs from 
$\Dx$ by a smoothing remainder. The paralinearization result $(ii)$ above gives 
a more precise description of $G(\eta)\psi$ when $\eta$ has limited smoothness. Namely, 
this result states that $G(\eta)\psi-\Dx\psi$ is the sum of the ``explicit'' contribution 
$-\Dx T_{\B(\eta)\psi}-\px (T_{V(\eta)\psi}\eta)$ and of a smoothing remainder $F(\eta)\psi$.

--- Assertion $(iii)$ of the theorem computes the error one gets when approximating 
$G(\eta)\psi$, $\B(\eta)\psi$, $V(\eta)\psi$ by their linear part. In this direction, we mention that 
we shall prove two more technical statements that will be used below. In section~\ref{S:TaylorDN} 
we study the Taylor expansion at order $2$ and $3$ of $G(\eta)\psi$ and of related quantities as 
a function of $\eta$, when $\eta$ goes to zero. The explicit knowledge of this expansion 
will be used in the rest of the paper. In particular we 
shall prove that, for some explicit 
quadratic term 
$F_{\quadratique}(\eta)\psi$, $\lA F(\eta)\psi-F_{\quadratique}(\eta)\psi\rA_{H^{s}}$ is 
estimated by 
\begin{equation*}
C(\lA \eta\rA_{\eC{\gamma}} )
\lA \eta\rA_{\eC{\gamma}} \Bigl\{\dalpha \lA \eta \rA_{H^s}+\lA \eta \rA_{\eC{\gamma}}
\blA \Dxmez\psi\brA_{H^{s-\mez}}\Bigr\},
\end{equation*}
This estimate will allow us to have a quadratic approximation of the equations without loss of derivatives.

The proof of Theorem~\ref{mainDN} will be given in the next sections. Let us describe the strategy we shall use. 

To be able to obtain estimates for $G(\eta)\psi$ (and the other quantities $\B(\eta)\psi$, $V(\eta)\psi$), 
we need to return to the definition of this function from the boundary values of the solution 
$\va$ of the elliptic boundary values problem $P\va=0$, $\va\az=\psi$, where 
$P$ is given by \eqref{111}. The beginning of the next section is devoted to the study of a 
related elliptic paradifferential problem $T_{p_0}W=f$, 
$W\az=\omega$, where $W=\va-T_{\pz\va}\eta$ is a function whose boundary value 
is the new unknown $\omega$, and where $p_0$ is the symbol of $P$. The point is that 
the choice of $W$ is made so that the right hand side 
$f=T_{p_0}W$ is a continuous function of $z$ with values in $H^{s+\gamma-3}$ 
($\subset H^s$ if $\gamma>3$) while a mere paralinearization 
of $P\va=0$ would give that $T_{p_0}$ is a continuous function of $z$ with values in $H^{s-1}$. 
This gain of smoothness in the right hand side will be instrumental in the proof of 
the estimate in $(ii)$ of the theorem. 

Once the elliptic problem satisfied by $W$ is established, we deduce from it bounds for $W$ in $z<0$ 
in terms of $W\az$ and $\eta$ (see Proposition~\ref{T13}). They are proved microlocally 
decomposing the elliptic boundary value problem into two coupled forward 
and backward parabolic equations, and performing a bootstrap argument exploiting the gain of smoothness of 
$f=T_{p_0}W$ explained above.

These estimates of $W$ are next used in section~\ref{S:213-b}, which is devoted to 
the proof of the tame estimate \eqref{2113} from the bounds of $W$ in $z<0$.

Section~\ref{S:paraDN} studies the paralinearization of the Dirichlet-Neumann operator: 
one establishes that $F(\eta)$ defined by
$$
F(\eta)\psi=G(\eta)\psi-\bigl(\Dx \omega-\px(T_{V(\eta)\psi}\eta)
$$
is a smoothing operator satisfying \eqref{n2}, using again the bounds on $W$ obtained in 
section~\ref{S:213}.

The assertions of the statement $(iv)$ of the theorem are deduced from the preceding result in 
section~\ref{S:LinearizationDN}. 

We end up Chapter~\ref{S:21} with a section devoted to a variant of the estimates of 
Theorem~\ref{mainDN}. Actually, inequalities \eqref{2113} and \eqref{2117} hold 
when $\eta$ and $\psi$ are at the same level of regularity (i.e.\ $\eta$ in $H^s$ and 
$\psi$ in $\h{\mez,s-\mez}$). We shall need estimates of the same type when $\eta$ is smoother 
than $\psi$, namely $\eta$ in $\eC{\gamma}$ and $\psi$ in $\h{\mez,\mu-\mez}$ for some 
$\mu\le \gamma-2$. These bounds are established in Section~\ref{S:smooth}.

\section{Sharp estimates}\label{S:213}


Let us introduce the following notation. Set
$$
a=\frac{1}{1+\eta'^2},\quad b=-2a\eta',\quad c=a\eta'',
$$
where $\eta'$ stands for $\px\eta$. 
Then the solution $\va$ of $P\va=0$, $\va\az=\psi$ obtained in 
Proposition~\ref{ref:116} satisfies,
\begin{align}
&\partial_z^2 \varphi+a\partial_x^2\varphi +b\partial_x\partial_z \varphi-c\partial_z \varphi=0
\quad \text{in }\{z<0\},\label{2120}\\
&\va\az=\psi.\label{2121}
\end{align}

\begin{assu}\label{ref:217}
We fix~$(s,\mu,\gamma)\in \xR^3$ such that
$$
s-\mez>\gamma>3,\quad 0\le \mu\le s,\quad \gamma\not\in\mez\xN.
$$
Throughout this section, we assume that 
$(\eta,\psi)$ is in the set $\Eg{\gamma}$ defined after the statement of Proposition~\ref{ref:116} and that 
moreover $(\eta,\psi)\in H^{s}\times \h{\mez,\mu}$ is such that $\omega\in \h{\mez,\mu+\mez}$.
\end{assu}

We introduce the function defined 
on $\{(x,z)\,;\, z<0\}$ 
\be\label{2122}
W=\va-T_{\pz\va}\eta
\ee
where the paraproduct is taken relatively to the $x$-variable alone, 
$z<0$ playing the role of a parameter. In particular by \eqref{2110}, 
$W\az=\psi-T_{\B(\eta)\psi}\eta=\omega(\eta)\psi$. Our goal is to study the regularity 
of $\va,W$ in terms of the regularity of $\psi,\eta$ and $\omega$. 

Let us set a notation that will be used constantly below. If $u$ is defined on $\{z<0\}$, 
we shall denote by~$\lA u\rA_{H^r}$ the~$z$-dependent function defined by 
$\lA u\rA_{H^r}(z)=\lA u(z,\cdot)\rA_{H^r}$. 
The inequality~$\lA f\rA_{H^r}\le \lA g\rA_{H^{r'}}$ thus means that 
$\lA f(z)\rA_{H^r}\le \lA g(z)\rA_{H^{r'}}$ for any~$z$ such that~$f(z)$ and~$g(z)$ are well defined. 
We denote by~$C$ various non decreasing functions of their arguments.

\begin{lemm}\label{good}
There exists a non decreasing function~$C\colon \xR_+\rightarrow \xR_+$ such that
\be\label{2122a}
\partial_z^2 W +(\id+T_{a-1})\partial_x^2 W +T_b \partial_x\partial_z W-T_c \partial_z W\defequal f
\ee
satisfies the bound
\begin{equation}\label{2123}
\sup_{z\le 0} \lA f(z)\rA_{H^{s+\gamma-3}}\le 
\Cetagamma \lA\eta\rA_{H^s} 
\sup_{z\le 0} \lA \nabla_{x,z}\varphi (z)\rA_{\eC{\gamma-1}}.
\end{equation}
\end{lemm}
\begin{rema}\label{rema:Ttilde}
--- In equation \eqref{2122a} above, we make appear 
as a coefficient of $\px^2W$ the operator $(\id+T_{a-1})$ 
instead of $T_a$. By definition~\eqref{eq.para} of the paradifferential operators, 
$T_1-\id$ is a Fourier multiplier 
whose symbol is supported for~$|\xi|\le 2$. 
Therefore,~$(\id+T_{a-1})-T_a=\id-T_1$ is a smoothing operator. 
Nevertheless, we prefer to use $(\id+T_{a-1})$ 
instead of $T_a$ because $a-1=O(\eta'^2)$, $\eta'\rightarrow 0$, 
so that the remainder coming from symbolic calculus will vanish at $\eta'=0$. 
In that way, we shall get the quadratic bound \eqref{2123} instead of a mere 
sub-linear bound as $(\eta,\va)\rightarrow (0,0)$.

--- The idea of the proof of the proposition is as follows: we shall 
paralinearize equation~\e{2122}. This will give us 
$$
\partial_z^2 \varphi 
+(\id+T_{a-1})\partial_x^2 \varphi 
+T_b \partial_x\partial_z \varphi -T_c \partial_z \varphi =f_1'+f_2',
$$
where $f'_2$ is a a remainder that has similar bounds as $f$ in \eqref{2123} and $f_1'$ is made from 
expressions of type $T_{\px^2 \va}(a-1)$, $T_{\px\pz\va}b$, $T_{\pz\va}c$. These contributions 
will not be smoother than $\eta''$ (since $c$ involves $\eta''$) 
i.e.\ will not be in a better space than $H^{s-2}$ if $\eta$ is in $H^s$. The gain in introducing 
$W$ instead of $\va$ lies in the fact that 
$$
\Bigl(\partial_z^2 
+(\id+T_{a-1})\partial_x^2 +T_b \partial_x\partial_z  -T_c \partial_z\Bigr) T_{\pz\va}\eta
$$
will be equal (up to smooth remainders) to $f_1'$, which gives the asserted result. 
\end{rema}

To start the proof, we first obtain a paradifferential description of the coefficients $a,b,c$ 
in \e{2120}. 
\begin{lemm}\label{ref:219a}
One may write
\be\label{2123a}
a-1=T_{ab}\eta'+r_1,\quad b=T_{b^2-2a}\eta'+r_2,\quad 
c=T_a\eta''+T_{ab\eta'}\eta'+r_3
\ee
where $r_\ell$, $\ell=1,2,3$, belong to $H^{s+\gamma-3}$ and satisfy
\be\label{2123b}
\lA r_\ell\rA_{H^{s+\gamma-3}}\le \Cetagamma \etapetitgamma \lA \eta'\rA_{H^{s-1}}, \quad 
\ell=1,2,3.
\ee
\end{lemm}
\begin{proof}
We use the fact (see section~$5.2.3$ in \cite{MePise}) that if $F$ is a smooth function vanishing at $0$ and if $u$ is in $H^{s}(\xR)$ 
with $s>1/2$, then
$$
F(\eta')=T_{F'(\eta')}u+R(\eta')
$$
where $R(\eta')\in H^{2s-5/2}(\xR)$ and 
$\lA R(\eta')\rA_{H^{s+\gamma-2}}\le C(\lA \eta'\rA_{\eC{\gamma-1}})\lA \eta'\rA_{H^{s-1}}$. 

Since $a-1=F_1(\eta')$ and $b=F_2(\eta')$ with 
$$
F_1(u)=-\frac{u^2}{1+u^2},\quad F_2(u)=-\frac{2u}{1+u^2},
$$
so that $F_1'(\eta')=ab$ and $F_2'(\eta')=b^2-2a$, 
we obtain the first two formulas in \e{2123a}. To get the last one we write
$$
c=a\eta''=T_a\eta''+T_{\eta''}a+r_3^1,
$$
where the remainder $r_3^1$ is in $H^{s+\gamma-3}$ 
by the paraproduct formula \eqref{Bony3}Ê
in Appendix~\ref{s2} and satisfies the bound \e{2123b}. We use the first 
equality \e{2123a} to express $a$ in $T_{\eta''}a$. We get
$$
T_{\eta''}a=T_{\eta''}1+T_{\eta''}T_{ab}\eta'+r_3^2,
$$
for a new remainder of the same type $r_3^2$ (as a paraproduct with 
an $L^\infty$ function acts on any Sobolev spaces, see~\e{esti:quant0}). Finally, 
by symbolic calculus (see~\eqref{esti:quant2-func}), $T_{\eta''}T_{ab}\eta'=
T_{ab\eta''}\eta'$ modulo another remainder of the same type. 
Since $T_{\eta''}1=0$ by definition of a paradifferential operator, this concludes the proof 
of the lemma.
\end{proof}

\begin{proof}[Proof of Lemma~\ref{good}]
We use the notation $D=-i\partial$. 
If $p_0(x,\xi,\zeta)$ is a polynomial in $\zeta$, with coefficients that are 
paradifferential symbols in $(x,\xi)$ i.e.\ 
$$
p_0(x,\xi,\zeta)=\sum_\alpha p_0^\alpha(x,\xi)\zeta^\alpha,
$$
we shall write 
$T_{p_0}\va$ for 
$\sum_{\alpha}T_{p_0^\alpha} (D_z^\alpha \va)(z,\cdot)$. 

Let us write the contributions to the left hand side of \e{2120} as 
\begin{align*}
(a-1)D_x^2\va &=T_{a-1}(D_x^2\va)+T_{D_x^2\va}(a-1)+R_1,\\
bD_xD_z\va &=T_{b}(D_xD_z\va)+T_{D_xD_z\va}b+R_2,\\
cD_z\va &=T_{c}(D_z\va)+T_{D_z\va}c+R_3,
\end{align*}
where $R_\ell$, $\ell=1,2,3$, the remainders 
in the paralinearization formula, satisfy
estimate \e{2123}. In the second term in the \rhs 
of the above equalities, we express $a-1$, $b$, $c$ 
using \e{2123a}. The remainders $r_\ell$ in \e{2123a} 
will give rise, according to \e{2123b}, 
to new contributions satisfying \e{2123}. 

Now we introduce 
$$
p_0(x,\xi,\zeta)=\zeta^2+\xi^2+(a-1)\xi^2+b\xi\zeta+ic\zeta
$$
and
$$
\widetilde{T}_{p_0}=D_z^2 
+(\id+T_{a-1})D_x^2 +T_b D_xD_z  -T_c D_z.
$$
Notice that we do not have $T_{\xi^2}=D_x^2$ (because we assume 
in Definition~\ref{defi:theta} that 
the cut-off function $\theta$, which enters into the definition~\e{eq.para} of 
paradifferential operators, satisfies $\theta(\xip,\xii)=0$ for $\la\xii\ra\le 1$). 
However, $T_{p_0}-\widetilde{T}_{p_0}=T_{\xi^2}-D_x^2$ is a smoothing operator. 
Then we see that \e{2120} may be rewritten as
\be\label{2123c}
\ba
\widetilde{T}_{p_0}\va &=-T_{D_x^2\va}T_{ab}\eta'-T_{D_xD_z \va}T_{b^2-2a}\eta'\\
&\quad -i T_{D_z\va}T_{ab\eta''}\eta'-iT_{D_z\va}T_a\eta''+r
\ea
\ee
where $r$ satisfies \e{2123}. Since $D_x^2\va$, $D_xD_z\va$ (resp.\ $D_z\va$) 
is in $L^\infty(\infzerof,\eC{\gamma-2})$ (resp.\ $L^\infty(\infzerof,\eC{\gamma-1})$) 
and $ab$, $b^2-2a$, $a$ (resp.\ $ab\eta''$) belong to $\eC{\gamma-1}$ 
(resp.\ $\eC{\gamma-2}$), it follows 
from the symbolic calculus result \e{esti:quant2-func} that the differences 
$$
T_{D_x^2\va}T_{ab}-T_{ab D_x^2\va},\quad 
T_{D_xD_z \va}T_{b^2-2a}-T_{(b^2-2a)D_xD_z \va},
\quad T_{D_z\va}T_{ab\eta''}-T_{ab\eta'' D_z \va}
$$
are operators in 
$\mathcal{L}(H^{s-1},H^{s+\gamma-3})$ (resp.\ $T_{Dz \va}T_a-T_{a D_z \va}$ is an operator in $\mathcal{L}(H^{s-1},H^{s+\gamma-2})$) 
with operator norms bounded from above by 
$$
\Cetagamma\sup_{z\le 0}\lA \nabla_{x,z}\va\rA_{\eC{\gamma-1}}.
$$
We conclude that \e{2123c} may be written
\be\label{2123d}
\widetilde{T}_{p_0}\va=T_q\eta'+r
\ee
where $r$ is a remainder satisfying \e{2123}, and where $q$ is the symbol
\be\label{2123e}
q(x,\xi,\zeta)=ab\px^2\va 
+(b^2-2a)\px\pz\va-ab\eta''\pz\va
-ia (\pz\va)\xi.
\ee
By definition of $W$, the left hand side of \e{2122a} is up to 
sign $\widetilde{T}_{p_0}(\va-T_{\pz\va}\eta)$, so that 
taking \e{2123d} into account, and remembering that 
$T_{p_0}-\widetilde{T}_{p_0}=T_{\xi^2}-D_x^2$ is a smoothing operator, 
we see that the proposition follows from 
the following lemma.
\begin{lemm}\label{ref:219b}
Under Assumption~\ref{ref:217},
$$
\sup_{z\le 0}
\lA T_q \eta'-T_{p_0}T_{\pz\va}\eta\rA_{H^{s+\gamma-3}}\\
\le \Cetagamma \lA \eta\rA_{H^s}\sup_{z\le 0}\lA \nabla_{x,z}\va\rA_{\eC{\gamma-1}}.
$$
\end{lemm}

By the formula of composition of paradifferential operators 
\e{esti:quant2sharp}, which is exact at order $3$ since 
$p_0(x,\cdot)$ is a polynomial of order $2$ in $(\xi,\zeta)$, we may write
\be\label{2123f}
T_{p_0}T_{\pz\va}=T_{p_0\pz\va}+T_{g_1}+T_{g_2}+R
\ee
where $R$ is an operator satisfying
$$
\lA R\rA_{\Fl{H^s}{H^{s+\gamma-3}}}\le 
\Cetagamma \sup_{z\le 0}\lA \nabla_{x,z}\va\rA_{\eC{\gamma-1}}
$$
and where $g_1,g_2$ are given by 
\begin{align*}
g_1(x,\xi,\zeta)&=\frac{1}{i}\Bigl( \partial_\zeta p_0 \pz^2\va 
+(\partial_\xi p_0)(\px\pz\va)\Bigr),\\
g_2(x,\xi,\zeta)&=-\mez \Big((\partial_\zeta^2p_0)(\pz^3\va)
+2(\partial_\zeta \partial_\xi p_0)(\px\pz^2\va)
+(\partial_\xi^2 p_0)(\px^2\pz\va)\Bigr)
\end{align*}
Computing these expressions using that \e{2120} implies that 
$$
\bigl(\pz^2+a\px^2+b\px\pz\bigr)\pz\va =c\pz^2\va,
$$
we obtain 
\be\label{2123g}
\ba
g_1(x,\xi,\zeta)&=\frac{1}{i}\Bigl(
(2\zeta+b\xi+ic)\pz^2\va+(2a\xi+b\zeta)\px\pz\va\Bigr),\\
g_2(x,\xi,\zeta)&=-c\pz^2\va.
\ea
\ee

Finally, we get that the \rhs of \e{2123f}Ê
may be written $T_e+R$ 
where $e$ is a symbol of the form
$$
e(x,z,\xi,\zeta)=\zeta^2\Gamma_0(x,z)+\zeta \Gamma_1(x,z,\xi)
+\Gamma_2(x,z,\xi),
$$
where $\Gamma_0$ is a function of $(x,z)$, $\Gamma_1$, $\Gamma_2$ 
are symbols in $(x,\xi)$ depending on the parameter $z$, with
$$
\Gamma_2(x,z,\xi)=a(\pz\va)\xi^2-ib(\pz^2\va)\xi
-2ia (\px\pz\va)\xi.
$$
We are reduced to showing $T_q\eta'-T_e\eta=0$. Since 
$\eta$ does not depend on $z$, we have $T_e\eta=T_{\Gamma_2}\eta$, 
so that it is enough to check that $\Gamma_2(x,z,\xi)=q(x,z,\xi)(i\xi)$. 
This follows from the above definition of $\Gamma_2$ where we substitute to $\pz^2\va$ 
its expression $\pz^2\va=-a\px^2\va-b\px\pz\va+c\pz\va$ coming from \e{2120}, 
remembering that $c=a\eta''$. This concludes the proof.
\end{proof}

We thus have proved that the unknown~$W$ solves 
the paradifferential equation 
$\mathcal{P}W=f$, where
\begin{equation}\label{n99}
\mathcal{P}=\partial_z^2 
+(\id+T_{a-1})\px^2 +T_b \px\partial_z -T_c \partial_z.
\end{equation}
Our next task is to find  
two operators~$P_{-}$ and~$P_{+}$ such that 
$$
\mathcal{P}=(\partial_z-P_{-})(\partial_{z}-P_{+})
$$ 
modulo an admissible remainder.

\begin{lemm}\label{T10}
Set
$$
P_{-}=-\Dx+T_{p+|\xi|},\quad 
P_{+}=\Dx +T_{P-|\xi|}
$$
where~$p=p(x,\xi)$ and~$P=P(x,\xi)$ are two symbols given by
\begin{equation}\label{n100}
\begin{aligned}
p(x,\xi)&=a(x)\left( i \px\eta(x) \xi -\la \xi\ra \right)+c(x), \\[0.5ex]
P(x,\xi)&= a(x)\left(i \px\eta(x) \xi +\la \xi\ra\right).
\end{aligned}
\end{equation}
Then
\begin{equation}\label{n101}
\bigl( \partial_z - P_{-} \bigr)
\bigl( \partial_z - P_{+} \bigr)=\mathcal{P}+R_0,
\end{equation}
where~$\mathcal{P}$ is given by \eqref{n99} and 
$R_0$ is a smoothing operator, satisfying
\begin{equation}\label{n102}
\lA R_0 u\rA_{H^{\mu+\gamma-3}}\le 
C(\lA \eta\rA_{\eC{\gamma}}) \lA \eta\rA_{\eC{\gamma}}^2
\lA \px u\rA_{H^{\mu-1}},
\end{equation}
for any~$\mu\in \xR$ and any~$u\in H^\mu(\xR)$.
\end{lemm}
\begin{proof}
Below we freely use the facts that, for any symbol $a=a(x,\xi)$,
$$
T_{a(x,\xi)(i\xi)}=T_{a}\px,\quad T_{a(x,\xi)|\xi|}=T_a\Dx, \quad 
\px (T_a u)=T_a \px u +T_{\px a}u.
$$
Since~$b=-2 a \px \eta$, by definition of~$P_{-}$ and 
$P_{+}$, we have
$P_{-}+P_{+}=-T_b\px +T_c$. 
Consequently, we have \eqref{n101} with
\begin{align*}
R_0&= P_{-}P_{+}-(\id+T_{a-1})\partial_x^2\\
&=T_{p+|\xi|}\Dx -\Dx T_{P-\la\xi\ra}+T_{p+|\xi|}T_{P-\la\xi\ra}-T_{a-1}\partial_x^2.
\end{align*}

The proof of \eqref{n102} is in two steps. We first give an exact formula for $T_{p+|\xi|}\Dx -\Dx T_{P-\la\xi\ra}$. 
Namely we prove that 
$T_{p+|\xi|}\Dx -\Dx T_{P-\la\xi\ra}=T_{a-1}\partial_x^2-T_q$ for some explicit symbol $q$. 
Then we use symbolic calculus to estimate the difference between $T_{p+|\xi|}T_{P-|\xi|}$ and $T_q$.

To compute $T_{p+|\xi|}\Dx -\Dx T_{P-\la\xi\ra}$ we use 
the two following identities (see Lemma~\ref{lemm:DxaDx}): for any function $a=a(x)$ in $L^\infty(\xR)$ 
and any function $u$ in $L^2(\xR)$, 
\begin{align}
&\Dx T_a\Dx u +\px T_a\px u=0,\label{i1}\\
&\Dx T_a \px u-\px T_a \Dx u=0.\label{i2}
\end{align}
Now, by definition,
\begin{align*}
p+|\xi|=a \eta'(i\xi)+a\eta''+(1-a)|\xi|,\quad
P-|\xi|=a \eta'(i\xi) +(a-1)|\xi|,
\end{align*}
so
\begin{multline}\label{pP1}
T_{p+|\xi|}\Dx -\Dx T_{P-\la\xi\ra}\\
=T_{a\eta'}\px\Dx+T_{a\eta''}\Dx
+T_{1-a}\Dx^2-\Dx T_{a\eta'}\px-\Dx T_{a-1}\Dx.
\end{multline}
Since $T_{a\eta'}\px+T_{a\eta''}=\px \bigl(T_{a\eta'}\cdot)-T_{(\px a)\eta'}$, 
the identity \eqref{i2}Ê
implies that 
\be\label{pP2}
T_{a\eta'}\px\Dx+T_{a\eta''}\Dx=\Dx T_{a\eta'}\px -T_{(\px a)\eta'}\Dx.
\ee
On the other hand, \eqref{i1} implies that
\be\label{pP3}
\Dx T_{a-1}\Dx=-\px T_{a-1}\px=-T_{\px a}\px -T_{a-1}\px^2.
\ee
Setting \eqref{pP2} and \eqref{pP3} in \eqref{pP1}, we obtain that
\begin{align*}
T_{p+|\xi|}\Dx -\Dx T_{P-\la\xi\ra}
&=-T_{(\px a)\eta'}\Dx+T_{1-a}\Dx^2+T_{\px a}\px +T_{a-1}\px^2\\
&=T_{(a-1)}\px^2 -T_q
\end{align*}
with
\be\label{n10}
q=\eta' (\px a) |\xi|-\px a(i\xi)+(a-1)|\xi|^2.
\ee
We conclude that $R_0=T_{p+|\xi|}T_{P-|\xi|}-T_q$. 

It remains to estimate the difference between $T_{p+|\xi|}T_{P-|\xi|}$ and $T_q$. 
To compute $T_{p+|\xi|}T_{P-|\xi|}$, it is convenient to introduce the symbol
\begin{equation*}
\wp(x,\xi)=a(x)\left( i \px\eta(x) \xi -\la \xi\ra \right),
\end{equation*}
and to decompose $p$ as $\wp+c$. Since~$\partial_\xi^k \wp(x,\xi)=0$ 
and~$\partial_\xi^k |\xi|=0$ for~$k\ge 2$ and~$\xi\neq 0$ 
and since 
the symbols~$\wp,P$ belong to~$\Gamma^{1}_{\gamma-1}(\xR)$,   
using~\eqref{esti:quant2sharp} applied with~$(m,m',\rho)=(1,1,\gamma-1)$, we obtain 
that
$$
T_{\wp+|\xi|}T_{P-|\xi|}=T_{q_1}+Q_1
$$
where~$Q_1$ is of order~$3-\gamma$ and the symbol~$q_1$ is given by 
\begin{align*}
q_1=(\wp+|\xi|)(P-|\xi|)+\frac{1}{i}\partial_\xi (\wp+|\xi|)\px (P-|\xi|).
\end{align*}
This simplifies to
$$
q_1=-\xi^2+\wp P+2a\xi^2+\frac{1}{i}\partial_\xi \wp \px P+\frac{1}{i}\frac{\xi}{|\xi|}\px P.
$$
On the other hand, using the notation~\e{defi:norms}, we have
$$
M^1_{\gamma-1}(p+|\xi|)+M^1_{\gamma-1}
(P-|\xi|)\le C(\lA \eta\rA_{\eC{\gamma}})\lA \eta\rA_{\eC{\gamma}},
$$
and hence~$\lA Q_1\rA_{\Fl{H^t}{H^{t+\gamma-3}}}\le 
C(\lA \eta\rA_{\eC{\gamma}})\lA \eta\rA_{\eC{\gamma}}^2$. 

Similarly,
~\eqref{esti:quant2sharp} applied with~$(m,m',\rho)=(0,1,2-\gamma)$ implies that
$$
T_{c}T_{P-|\xi|}=T_{c(P-|\xi|)}+Q_2
$$
where~$\lA Q_2\rA_{\Fl{H^t}{H^{t+\gamma-3}}}\le 
C(\lA \eta\rA_{\eC{\gamma}})\lA \eta\rA_{\eC{\gamma}}^2$. 

The previous observations yield 
$T_{p+|\xi|}T_{P-|\xi|}=T_{\tau}+Q_1+Q_2$ with
$$
\tau=-\xi^2+\wp P+2a\xi^2+\frac{1}{i}\partial_\xi \wp \px P+\frac{1}{i}\frac{\xi}{|\xi|}\px P+
c(P-|\xi|).
$$
Now using the calculation results
$$
\wp P=-a\xi^2,\qquad 
\frac{1}{i}(\partial_\xi \wp)( \partial_x P)+c P=0,
$$
we obtain that
$$
\tau=(a-1)|\xi|^2+\frac{1}{i}\frac{\xi}{|\xi|}\px P-c|\xi|
$$
and it is easily verified that $\tau=q$ where $q$ is given by \eqref{n10} (recalling that $c=a\eta''$). 
We conclude that $R_0=Q_1+Q_2$ and the previous observations yield 
\begin{equation*}
\lA R_0 u\rA_{H^{\mu+\gamma-3}}\le C(\lA \eta\rA_{\eC{\gamma}})
\lA \eta\rA_{\eC{\gamma}}^2\lA u\rA_{H^{\mu}}.
\end{equation*}

Having proved this first estimate for the remainder, 
we prove it is estimated by the derivative of~$u$ only: 
\begin{equation*}
\lA R_0  u\rA_{H^{\mu}}\le C(\lA \eta\rA_{\eC{\gamma}}) \lA \eta\rA_{\eC{\gamma}}^2
\lA \px u\rA_{H^{\mu-1}}.
\end{equation*}
To do so, introduce~$\tilde{\kappa}=\tilde{\kappa}(\xi)$ such that 
$\tilde{\kappa}(\xi)=1$ for~$\la \xi\ra\ge 1/3$ and~$\tilde{\kappa}(\xi)=0$ for~$\la\xi\ra\le 1/4$. 
Split~$R_0$ as 
$$
R_0 \tilde{\kappa}(D_x)+R_0(\id -\tilde{\kappa}(D_x)).
$$
Notice that~$R_0(\id -\tilde{\kappa}(D_x))=0$ since~$A (\id -\tilde{\kappa}(D_x))=0$ 
for any paradifferential operator~$A$. On the other hand, 
\begin{align*}
\lA R_0 \tilde{\kappa}(D_x) u \rA_{H^{\mu+\gamma-3}}&\le 
C(\lA \eta\rA_{\eC{\gamma}}) \lA \eta\rA_{\eC{\gamma}}^2
\lA \tilde{\kappa}(D_x) u \rA_{H^{\mu}}\\
&\le C(\lA \eta\rA_{\eC{\gamma}}) \lA \eta\rA_{\eC{\gamma}}^2
\lA \partial_x u \rA_{H^{\mu-1}}.
\end{align*}
This completes the proof.
\end{proof}

By construction, it follows from the previous lemma that
$$
(\partial_z -P_{-})(\partial_z-P_{+})W=
\mathcal{P}W +R_0W.
$$
On the other hand,~$f\defn \mathcal{P}W$ is estimated by~\eqref{2123}. 
Introduce now
$$
\underline{w}=(\partial_z -P_{+})W.
$$
Then
\begin{equation}\label{eq:wW}
\left\{
\begin{aligned}
&(\partial_z -P_{-})\underline{w}=f+R_0W,\\[0.5ex]
&(\partial_z-P_{+})W=\underline{w}.
\end{aligned}
\right.
\end{equation}
Since~$\RE p(x,\xi) \le -c \la \xi\ra$ for $1\ll |\xi|$, the first equation in~\eqref{eq:wW}Ê
is parabolic. 
Since $\RE P(x,\xi)\ge c\la \xi\ra$, 
the backward Cauchy problem is well posed for the second equation. 
Hence, up to time reversal in the second equation, 
System~\eqref{eq:wW} is a system of two paradifferential 
parabolic equations. We begin by recalling a classical estimate for such equations. 
\begin{lemm}\label{T11}
Let~$\mu\in \xR$,~$T\in [0,+\infty)$. Let 
$u$ in $C^0([0,T];H^\mu(\xR))\cap C^1([0,T];H^{\mu-1}(\xR))$ and 
$F$ in $L^\infty([0,T];H^{\mu}(\xR))$ satisfying 
$$
\partial_t u +\Dx u +T_{q-|\xi|}u=F,
$$
for some symbol~$q\in \Gamma^1_{1}(\xR)$ (independent of time) 
such that~$\RE q \ge c \la \xi\ra$. 
Then, for any~$\eps>0$, $u$ belongs to $C^0([0,T];H^{\mu+1-\eps}(\xR))$ and 
there exists a positive 
constant~$K$ 
depending 
on~$M^1_1(q)$ (see~\eqref{defi:norms}) such that
\begin{equation}\label{n104}
\lA u\rA_{L^\infty([0,T];H^{\mu+1-\eps})}
\le K\lA u(0)\rA_{H^{\mu+1-\eps}}+K\lA F\rA_{L^\infty([0,T];H^{\mu})}
+K\lA u\rA_{L^\infty([0,T];H^\mu)}.
\end{equation}
\end{lemm}
\begin{proof}This follows from \cite{Tougeron} 
(see also~\cite[Prop. 4.10]{AM} and \cite[Prop. 3.19]{ABZ1}). 
We recall the proof for the sake of completeness. Write
$$
\partial_t u +T_q u=g\defn F+(T_{|\xi|}-\Dx)u.
$$
Since~$T_{|\xi|}-\Dx$ is a smoothing operator we have
$$
\lA g\rA_{L^\infty([0,T];H^\mu)}\les \lA F\rA_{L^\infty([0,T];H^\mu)}+\lA u\rA_{L^\infty([0,T];H^\mu)}.
$$
Given $\tau\le 0$, one denotes by $e(\tau,\cdot,\cdot)$ or simply $e(\tau)$ the symbol defined by 
$e(\tau,x,\xi)=\exp(\tau q(x,\xi))$ so that 
$e(0,x,\xi)=1$ and $\partial_\tau e(\tau,x,\xi)=e(\tau,x,\xi)q(x,\xi)$. 

Now, given~$y\in [0,T]$ and~$t\in [0,y]$, write
$$
\partial_{t}\left(T_{e(t-y)} u \right)=T_{e(t-y)} g 
+\left(T_{\partial_t e(t-y)}u-T_{e(t-y)}T_{q}\right)u
$$ 
and integrate on~$[0,y]$ to obtain
\begin{equation*}
T_1 u(y)=T_{e(-y)}u(0)+\int_0^y \Bigl\{ T_{e(t-y)}
g(t)+\left(T_{\partial_t e(t-y)}-T_{e(t-y)}T_q\right)u(t)\Bigr\} dt.
\end{equation*}
Which is better formulated as
\begin{equation*}
u(y)=T_{e(-y)}u(0)+\int_0^y \Bigl\{T_{e(t-y)} g(t)+S(t-y)u(t)\, \Bigr\}dt+(\id -T_1)u(y),
\end{equation*}
with $S(\tau)\defn (T_{\partial_\tau e(\tau)}u-T_{e(\tau)}T_q)$. 

According
to our assumption that~$\RE q \geq c\la\xi\ra$,~$q\in \Gamma_1^1(\xR)$, 
we see that 
$e(\tau)$ belongs uniformly to~$\Gamma_1^0(\xR)$ for $\tau\in [-T,0]$; 
which means that 
$\sup_{\tau\in [-T,0]} M^0_1(e(\tau,\cdot,\cdot))\le C( M^1_1(q))$ where 
the semi-norm~$M^1_0(q)$ is  
as defined in~\eqref{defi:norms}. 
Therefore~$\partial_\tau e=eq$ belongs uniformly 
to~$\Gamma_1^1(\xR)$. It follows 
from symbolic calculus (see~\eqref{esti:quant2}) 
that~$S(\tau)=T_{e(\tau)q}-T_{e(\tau)} T_q$ 
is uniformly of order~$0$. Therefore 
there exists a constant $K$ depending only 
on $M^1_1(q)$ such that, for any $y\in [0,T]$ 
and any 
$t\in [0,y]$, 
$$
\lA S(t-y) u(t)\rA_{H^{\mu}}\le K \lA u(t)\rA_{H^\mu},
$$
Similarly, \eqref{esti:quant0} implies that
$$
\lA T_{e(-y)}u(0)\rA_{H^{\mu+1-\eps}}\le K \lA u(0)\rA_{H^{\mu+1-\eps}}.
$$
On the other hand, 
$\la y-t\ra^{1-\eps} \langle \xi\rangle ^{1-\eps} e(t-y,x,\xi)$ 
is uniformly of order~$0$ so that
$$
\int_0^y \lA  T_{e(t-y)} g(t)\rA_{H^{\mu+1-\eps}}\, dt
\les \lA g\rA_{L^\infty([0,y];H^\mu)}.
$$
It follows that there exists a constant~$K$ 
depending only on~$M^1_1(q)$ such that, 
for all~$y\in [0,T]$, 
\begin{equation*}
\lA u(y)\rA_{H^{\mu+1-\eps}}\le 
K\lA u(0)\rA_{H^{\mu+1-\eps}}+K\lA F\rA_{L^\infty([0,y];H^{\mu})}
+K\lA u\rA_{L^\infty([0,y];H^\mu)}.
\end{equation*}
This proves that $u\in L^{\infty}([0,T];H^{\mu+1-\eps}(\xR))$. 
Since $u\in C^{0}([0,T];H^{\mu}(\xR))$ by assumption, 
this implies, by interpolation, that 
$u\in C^{0}([0,T];H^{\mu+1-2\eps}(\xR))$. 
This gives the desired result with $\eps$ replaced with $2\eps$.
\end{proof}

We are now in position to estimate $(\underline{w},W)$ 
by using the previous lemma and the 
fact that $(\underline{w},W)$ satisfy \eqref{eq:wW}. For 
later purposes, it is convenient 
to state this as a general result. 

\begin{lemm}\label{T12}
Consider $\tau<0$, $\mu\in \xR$ and $\eps>0$. 

$(i)$ Let $\underline{v}$ in $L^\infty([\tau,0];H^{\mu+1}(\xR))$,  
$V$ in $C^0([\tau,0];H^{\mu+1}(\xR))\cap C^1([\tau,0];H^{\mu}(\xR))$ satisfying 
\begin{equation*}
(\partial_z-P_{+})V=\underline{v}.
\end{equation*}
If $\px V(0)\in H^{\mu+1-\eps}(\xR)$ 
then $V\in C^{0}([\tau,0];H^{\mu+2-\eps}(\xR))$ 
and there exists a non decreasing function $C$ depending only 
on $\gamma,\tau,\mu,\eps$ such that 
\begin{multline}\label{n105}
\lA \nabla_{x,z}V\rA_{L^\infty([\tau,0];H^{\mu+1-\eps})}\\
\le C(\lA \eta\rA_{\eC{\gamma}}) \Bigl(\lA \px V(0)\rA_{H^{\mu+1-\epsilon}}
+\lA \underline{v}\rA_{L^\infty([\tau,0];H^{\mu+1)}}
+\lA \nabla_{x,z}V\rA_{L^\infty([\tau,0];H^{\mu)}}\Bigr).
\end{multline}

$(ii)$ Consider $V$ 
in $L^\infty([\tau,0];H^{\mu-(\gamma-3)}(\xR))$, $\underline{v}$ 
in $C^0([\tau,0];H^{\mu}(\xR))\cap C^1([\tau,0];H^{\mu-1}(\xR))$, and $f$ 
in $L^\infty([\tau,0];H^\mu(\xR))$ 
satisfying
\begin{equation}\label{n106}
(\partial_z -P_{-})\underline{v}=f+R_0V.
\end{equation}
Then, for any $\tau'$ in $\pol\tau,0\por$, $\underline{v}$ belongs to 
$C^{0}([\tau',0];H^{\mu+1-\eps}(\xR))$ and there exists 
a non decreasing function $C$ depending only 
on $\gamma,\tau,\tau',\mu,\eps$ such that 
\begin{equation}\label{n107}
\begin{aligned}
\lA \underline{v}\rA_{L^\infty([\tau',0];H^{\mu+1-\epsilon})}
&\le C(\lA \eta\rA_{\eC{\gamma}}) \Bigl( \lA f\rA_{L^\infty([\tau,0];H^\mu)}
+\lA \underline{v}\rA_{L^\infty([\tau,0];H^\mu)}\Bigr)\\
&\quad +C(\lA \eta\rA_{\eC{\gamma}})\lA \eta\rA_{\eC{\gamma}}
\lA \nabla_{x,z}V\rA_{L^\infty([\tau,0];H^{\mu-1-(\gamma-3)})}.
\end{aligned}
\end{equation}
\end{lemm}
\begin{proof}
To prove statement $(i)$ we apply Lemma~\ref{T11} to the auxiliary function 
$u(t,x)=(\partial_x V)(-t,x)$ which satisfies
$$
\partial_t u +\Dx u +T_{P-|\xi|}u=G
$$
where~$P$ is given by~\eqref{n100} and where~$G(t,x)=-(\px \underline{v}+T_{\px P}V)(-t,x)$. 
Thus the estimate \eqref{n104} applied with~$T=-\tau$ 
implies that there exists 
a positive constant 
$K=K(\lA \eta\rA_{\eC{\gamma}})$ such that
\begin{align*}
\lA u\rA_{L^\infty([0,-\tau];H^{\mu+1-\eps})}
&\le K\lA u(0)\rA_{H^{\mu+1-\eps}}+K\lA G\rA_{L^\infty([0,-\tau];H^{\mu})}\\
&\quad +K\lA u\rA_{L^\infty([0,-\tau],H^\mu)}.
\end{align*}
This yields
\begin{align*}
\lA \px V\rA_{L^\infty([\tau,0];H^{\mu+1-\eps})}
&\le K \lA \px V (0)\rA_{H^{\mu+1-\eps}}+
K
\lA \px \underline{v}\rA_{L^\infty([\tau,0];H^{\mu})}\\
&\quad+K\lA \px V\rA_{L^\infty([\tau,0];H^\mu)}.
\end{align*}
Since~$\partial_z V=P_{+}V+\underline{v}$ can be estimated by means of~$\px V$ 
and~$\underline{v}$, we obtain \eqref{n105}.

To prove statement $(ii)$ we apply Lemma~\ref{T11} to the auxiliary function  
$u(t,x)=t\, \underline{v}(t+\tau,x)$ which satisfies~$u(0,x)=0$ 
and  
$$
\partial_t u +\Dx u +T_{q-|\xi|}u =F
$$
with~$q=-p$ (where~$p$ is given by~\eqref{n100}) and 
$$
F(t,x)= t f(t+\tau,x)+t (R_0V)(t+\tau,x)+\underline{v}(t+\tau,x).
$$
It follows from \eqref{n102} and the assumption $\gamma>3$ 
that $\lA F\rA_{L^\infty([0,-\tau];H^{\mu})}$ is 
bounded by the right-hand side 
of \eqref{n107}. 

Since~$u\arrowvert_{t=0}=0$, the parabolic estimate~\eqref{n104} implies that
$$
\lA u\rA_{L^\infty([0,-\tau];H^{\mu+1-\eps})}
\le K\lA F\rA_{L^\infty([0,-\tau];H^{\mu})}
+K\lA u\rA_{L^\infty([0,-\tau],H^\mu)}.
$$
Clearly, 
$$
\lA u\rA_{L^\infty([0,-\tau];H^{\mu})}
=\sup_{z\in [\tau,0]} \snorm{(z-\tau)\underline{v}(z)}_{H^{\mu}}
\le |\tau|  \sup_{z\in [\tau,0]} \snorm{\underline{v}(z)}_{H^{\mu}}
$$
and
\begin{align*}
\sup_{z\in [\tau',0]} \snorm{\underline{v}}_{H^{\mu+1-\eps}}
&\le \frac{1}{|\tau-\tau'|}
\sup_{z\in [\tau',0]} \snorm{(z-\tau) \underline{v}(z)}_{H^{\mu+1-\eps}}\\
&\le \frac{1}{|\tau-\tau'|} \sup_{t\in [0,-\tau]} \snorm{u(t)}_{H^{\mu+1-\eps}}.
\end{align*}
Therefore, the previous estimates imply \eqref{n107}.
\end{proof}

We are now in position to prove the main result of this section. 
Given~$\tau<0$, we use the notations
\begin{equation}\label{n108}
\begin{aligned}
E(\tau)&\defn \sup_{z\in [\tau,0]}\left\{ \snorm{\partial_z \varphi}_{H^{-1/2}}
+\snorm{\partial_x \varphi-\partial_x\eta \partial_z \varphi}_{H^{-1/2}}\right\},\\
D(\tau)&\defn
\sup_{z\in [\tau,0]}
\blA \partial_{z}\varphi-\Dx \varphi\brA_{H^{-\mezl}}.
\end{aligned}
\end{equation}
\begin{prop}\label{T13}
Let~$(s,\mu,\gamma)\in \xR^3$ be such that
$$
s-\mez>\gamma>3,\quad 0\le \mu\le s,\quad \gamma\not\in\mez \xN,
$$
and assume that 
$(\eta,\psi)$ is in the set $\Eg{\gamma}$ defined after the 
statement of Proposition~\ref{ref:116} and that 
moreover $(\eta,\psi)\in H^{s}\times \h{\mez,\mu}$ is such that $\omega\in \h{\mez,\mu+\mez}$. 

Consider~$\eps>0$ and~$\tau<\tau'<0$. 
There exists a non decreasing function~$C\colon \xR\rightarrow\xR$ 
such that
\be\label{n109}
\begin{aligned}
\sup_{z\in [\tau',0]} \snorm{\partial_z W(z)-P_{+}W(z)}_{H^{\mu+\gamma-3-\eps}}
&\le c_1\lA \eta\rA_{H^s}+c_2\lA \px \omega\rA_{H^{\mu-1}}\\
&\quad +c_2 E(\tau)+c_3 D(\tau),
\end{aligned}
\ee
and
\be\label{n110}
\sup_{z\in [\tau',0]}\snorm{\nabla_{x,z} W(z)}_{H^{\mu}}
\le c_1 \lA \eta\rA_{H^s}+c_3\lA \px \omega\rA_{H^\mu}+c_3 E(\tau)+c_3 D(\tau),
\ee
where
\be\label{n111}
c_3\defn C(\lA \eta\rA_{\eC{\gamma}}),\quad 
c_1\defn c_3
\sup_{z\in [\tau,0]} \lA \nabla_{x,z}\varphi (z)\rA_{\eC{\gamma-1}},
\quad 
c_2\defn c_3\lA \eta\rA_{\eC{\gamma}}
.
\ee
\end{prop}
\begin{rema*}
We prove not only {\em a priori\/} estimates but also an elliptic regularity result. Namely, 
the previous statement means that if the right-hand side of \eqref{n110} is finite, then so is 
the left hand-side.
\end{rema*}
\begin{proof}
Given~$\tau\in \pol-\infty,0\por$ and~$(\mu,\sigma)\in \xR^2$, introduce
\begin{align*}
A_1(\tau;\sigma)
&\defn \sup_{z\in [\tau,0]}\snorm{\partial_z W -P_{+}W}_{H^\sigma} ,\\
A_2(\tau;\mu)
&\defn \sup_{z\in [\tau,0]}\snorm{\nabla_{x,z}W}_{H^\mu}.
\end{align*}
One denotes by~$\mathcal{A}_1$ the set of~$\mu \in \pol -\infty ,s]$ such that the 
following property holds: for all~$(\sigma,\tau,\tau')\in \xR^3$ such that
$$
\sigma\in [\mu,\mu+\gamma-3\por,\quad \tau<\tau'<0,
$$
the function $\partial_z W -P_{+}W$ belongs to $C^{0}([\tau',0];H^\sigma(\xR))$ and 
there is a non decreasing function 
$C\colon \xR_+\rightarrow \xR_+$ depending only on~$(s,\gamma,\mu,\sigma,\tau,\tau')$ 
such that
\begin{align*}
A_1(\tau';\sigma)&\le 
c_1 \lA \eta\rA_{H^s} 
+c_2 \blA \px \omega\brA_{H^{\mu-1}}
+c_3 A_{1}(\tau;-\mezl)+c_2 A_2(\tau;-\mezl),
\end{align*}
where $c_1,c_2$ and $c_3$ are as in \eqref{n111}.

Similarly, one denotes by~$\mathcal{A}_2$ the set of~$\mu \in \pol -\infty,s]$ such that the 
following property holds: for all~$\tau\in \pol-\infty,0\por$, the function 
$\nabla_{x,z}W$ belongs to $C^0([\tau,0];H^{\mu}(\xR))$ and 
there exists 
a non decreasing function 
$C\colon \xR_+\rightarrow \xR_+$ depending only on 
$(s,\gamma,\mu,\tau,\tau')$ such that
\begin{align*}
A_2(\tau';\mu)
&\le c_1\lA \eta\rA_{H^s}+
c_3 \Bigl\{\blA \px \omega\brA_{H^\mu}+A_1(\tau;-\mezl)+A_2(\tau;-\mezl)\Bigr\}\cdot
\end{align*}

The proof of Proposition~\ref{T13} is in two steps. 
The key point consists in proving that
\begin{equation}\label{A1A2}
\mathcal{A}_1=\left] -\infty,s\right],\quad \mathcal{A}_2=\left] -\infty,s\right].
\end{equation}
To prove \eqref{A1A2}, we  
proceed by means of a bootstrap argument (as in~\cite{ABZ3}).

Recall that, by notations,
$\underline{w}=(\partial_z -P_{+})W$ and
\begin{equation*}
\left\{
\begin{aligned}
&(\partial_z -P_{-})\underline{w}=f+R_0W,\\[0.5ex]
&(\partial_z-P_{+})W=\underline{w},
\end{aligned}
\right.
\end{equation*}
where $f$ is given by Lemma~\ref{good}. It follows from the estimate \eqref{2123} for $f$ and 
Lemma~\ref{T12} that, 
for any $(s,\gamma,\mu)$ as above, any $\tau<\tau'<0$ and any $\eps>0$, there exists a non decreasing function 
$C\colon \xR_+\rightarrow \xR_+$ such that
\begin{align*}
A_1(\tau';\mu+1-\eps)
&\le c_1 \lA \eta\rA_{H^s}+c_2 A_2(\tau;\mu-1-(\gamma-3))
+c_3 A_1(\tau;\mu),\\
A_2(\tau;\mu+1-\eps)&\le c_3 \lA \px \omega\rA_{H^{\mu+1-\eps}}
+c_3 A_1(\tau;\mu+1)+c_3 A_2(\tau;\mu),
\end{align*}
where $c_1,c_2$ and $c_3$ are as in \eqref{n111}. 
This implies that, for any~$\eps>0$, 
\begin{multline}\label{n113}
\mu-(\gamma-3)+\eps \in \mathcal{A}_1, ~
\mu-1-(\gamma-3)\in \mathcal{A}_2 \\
\Rightarrow ~~
\min\left\{ \mu+1-\eps-(\gamma-3),s\right\}\in \mathcal{A}_1,
\end{multline}
and
\be\label{n114}
\mu+4-\gamma+\eps \in \mathcal{A}_1,~~
\mu \in \mathcal{A}_2 ~
\Rightarrow ~ 
\min\left\{ \mu+1-\eps,s\right\}\in \mathcal{A}_2.
\ee
Now, 
let us show that \eqref{n113} and \eqref{n114} imply \eqref{A1A2}. Firstly, notice that, clearly, 
\begin{equation}\label{n115}
5/2-\gamma\in \mathcal{A}_1,\qquad -1/2\in \mathcal{A}_2.
\end{equation} 
Observe that~$5/2-\gamma<-1/2$. 
Now assume 
that~$[5/2-\gamma,\kappa]\times [5/2-\gamma,\kappa]\subset \mathcal{A}_1\times \mathcal{A}_2$ for 
some~$5/2-\gamma\le \kappa<s$, and set 
\begin{equation*}
\eps=\min \left\{ \frac{1}{4}(\gamma-3),\frac{1}{4}\right\}, \quad 
\mu=\kappa-1+2\eps,\quad \nu=\kappa+(\gamma-3)-2\eps.
\end{equation*}
Then~$\mu<\kappa$ and~$\mu+4-\gamma+\eps<\kappa$. Therefore 
$\mu+4-\gamma+\eps\in \mathcal{A}_1$ and 
$\mu\in \mathcal{A}_2$. Property~\eqref{n114} then implies that 
$\min\left\{ \mu+1-\eps,s\right\}\in \mathcal{A}_2$. 
Since~$\mu+1-\eps=\kappa+\eps$ we thus have proved that 
$\min\{\kappa+\eps,s\}\in \mathcal{A}_2$. 
Similarly,~$\nu-(\gamma-3)+\eps<\kappa$ and~$\nu-1-(\gamma-3)<\kappa$; 
so Property~\eqref{n113} implies that 
$\min\{\kappa+\eps,s\}\in \mathcal{A}_1$. We thus have proved that 
$$
[5/2-\gamma,\tilde\kappa]\times [5/2-\gamma,\tilde\kappa]\subset \mathcal{A}_1\times \mathcal{A}_2
\quad\text{with }\tilde\kappa=\min\{\kappa+\eps,s\}.
$$
In view of \eqref{n115}, this implies \eqref{A1A2}.

To conclude the proof of Proposition~\ref{T13} it is sufficient to prove that
\begin{align}
&A_2(\tau;-1/2)\les E(\tau)
+\lA \eta\rA_{H^1}\sup_{z\in I} \lA \nabla_{x,z}\varphi (z)\rA_{\eC{1}},\label{n117}\\
&A_1(\tau;-1/2)
\les D(\tau)+C(\lA \eta\rA_{\eC{2}})
\Bigl\{\lA \eta\rA_{\eC{2}} E(\tau)+
\sup_{z\in I} \lA \nabla_{x,z}\varphi (z)\rA_{L^\infty}\lA \eta\rA_{H^{1/2}}\Bigr\}.\label{n118}
\end{align}
Recall that
\begin{align*}
A_1(\tau;-1/2)&=\sup_{z\in [\tau,0]} \snorm{\partial_z W-P_{+}W}_{H^{-\mezl}},\\
A_2(\tau;-1/2)&=\sup_{z\in [\tau,0]}\Bigl\{ \snorm{\partial_zW}_{H^{-\mez}}
+\snorm{\partial_x W}_{H^{-\mez}}\Bigr\}\cdot
\end{align*}

Let us prove \eqref{n117}. Since
$\partial_z W=\partial_z \varphi-T_{\partial_z^2\varphi}\eta$ 
we have 
$$
\lA \partial_zW\rA_{H^{-\mez}}\le 
\lA \partial_z \varphi\rA_{H^{-\mez}}
+\lA \partial_z^2\varphi\rA_{L^\infty}\lA\eta\rA_{H^{-\mez}},
$$
and hence~$\sup_{z\in [\tau,0]}\snorm{\partial_zW}_{H^{-\mez}}$ is bounded 
by the right-hand side of \eqref{n117} by definition of~$E(\tau)$. 
To estimate~$\px W$, write
\begin{align*}
\partial_x W &=\partial_x\varphi -T_{\partial_z\varphi}\px\eta
-T_{\px\partial_z\varphi}\eta\\
&=\partial_x\varphi -\partial_z \varphi \px \eta
+(\partial_z \varphi-T_{\partial_z\varphi})\px\eta
-T_{\px\partial_z\varphi}\eta,
\end{align*}
so 
\begin{align*}
\lA \px W\rA_{H^{-\mez}}&\le 
\lA \partial_x\varphi -\partial_z \varphi \px \eta\rA_{H^{-\mez}}
+\lA \partial_z \varphi\px\eta \rA_{H^{-\mez}}\\
&\quad +\lA T_{\partial_z\varphi}\px\eta\rA_{H^{-\mez}}
+\lA T_{\px\partial_z\varphi}\eta\rA_{H^{-\mez}}.
\end{align*}
This implies that
$$
\lA \px W\rA_{H^{-\mez}} \le E(\tau)+K \lA \eta\rA_{H^1}\sup \lA \nabla_{x,z}\va\rA_{\eC{1}},
$$
which completes the proof of \eqref{n117}. 

Let us prove \eqref{n118}. By definition of~$P_+$ and~$W$, we have
\begin{align*}
\partial_z W-P_{+}W
&=(\partial_z -P_+)\varphi-(\partial_z -P_+)T_{\partial_z \varphi}\eta\\
&=(\partial_z -\Dx)\varphi -T_{P-\la\xi\ra}\varphi-T_{\partial_z^2\varphi}\eta
+P_+ T_{\partial_z\varphi}\eta.
\end{align*}
The first term in the right-hand side is estimated directly from the definition of~$D(\tau)$. 
To estimate the third term we write $\lA T_{\partial_z^2\varphi}\eta\rA_{H^{-\mez}}
\les \lA \pz^2\va\rA_{\eC{-1}}\lA \eta\rA_{H^{\mez}}$ and then use the 
equation \e{2120} satisfied by $\va$ to estimate $\lA \pz^2\va\rA_{\eC{-1}}$. 
To estimate the last term, 
by using \eqref{esti:quant1}, we first notice that
\begin{align*}
\lA P_+ T_{\partial_z\varphi}\eta\rA_{H^{-\mezl}}
&\les (1+M^1_0(P-|\xi|))\lA T_{\partial_z\varphi}\eta\rA_{H^{\mezl}}\\
&\les (1+M^1_0(P-|\xi|))\lA \partial_z\varphi\rA_{L^\infty}\lA\eta\rA_{H^{\mezl}}.
\end{align*}
Since $M^1_0(P-|\xi|)\le C(\lA \eta\rA_{\eC{1}})\lA \px\eta\rA_{\eC{1}}$, we obtain that
$$
\lA P_+ T_{\partial_z\varphi}\eta\rA_{H^{-\mezl}}\le C(\lA \eta\rA_{\eC{2}})
\lA \partial_z\varphi\rA_{L^\infty}\lA\eta\rA_{H^{\mezl}}.
$$
Similarly, \eqref{esti:quant0-px} implies that
$$
\lA T_{P-\la\xi\ra}\varphi\rA_{H^{-\mezl}}\les M^1_0(P-|\xi|)\lA\px\varphi\rA_{H^{-\mezl}}\le 
C(\lA \eta\rA_{\eC{2}})\lA \px\eta\rA_{\eC{1}}\lA \px\varphi\rA_{H^{-\mezl}}.
$$
Since
\begin{align*}
\lA \px\varphi\rA_{H^{-\mezl}}&\les \lA \px\varphi-\px\eta\partial_z\varphi\rA_{H^{-\mezl}}
+\lA \px\eta\rA_{\eC{1}}\lA \partial_z\varphi\rA_{H^{-\mezl}}\\
&\les (1+\lA \px\eta\rA_{\eC{1}})E(\tau),
\end{align*}
by combining the above estimates we conclude the proof of \eqref{n118}. 
This completes the proof of the induction argument and hence the proof of the proposition.
\end{proof}

\section{Tame estimates}\label{S:213-b}

In this section we prove the tame 
estimates~\eqref{2113}. 

\begin{prop}\label{T14}
$i)$ Let~$(s,\gamma,\mu)\in \xR^3$ be such that 
$$
s-\mez>\gamma>3,\quad \mez\le \mu\le s,\quad \gamma\not\in\mez\xN.
$$
Consider 
$(\eta,\psi)\in \Eg{\gamma}\cap \bigl( H^{s+\mez}(\xR)\times \h{\mez,\mu}(\xR)\bigr)$ 
and set~$\omega=\psi-T_{\B(\eta)\psi}\eta$. 
Then
$$
\B(\eta)\psi\in H^{\mu-\mez}(\xR),\quad 
V(\eta)\psi\in H^{\mu-\mez}(\xR),
$$
and there exists a non decreasing function~$C\colon \xR_+\rightarrow \xR_+$ 
depending only on~$(s,\gamma,\mu)$ such that:
\begin{multline}\label{n119}
\lA \B(\eta)\psi \rA_{H^{\mu-\mez}}+\lA V(\eta)\psi\rA_{H^{\mu-\mez}}\\
\le 
C\left( \lA \eta\rA_{\eC{\gamma}}\right)\bigl\{ \dalpha\lA \eta\rA_{H^s}+
\blA \Dxmez \omega\brA_{H^{\mu}}\bigr\}.
\end{multline}

$ii)$ Let~$(s,\gamma,\mu)\in \xR^3$ be such that 
$$
s-\mez>\gamma>3,\quad 1\le \mu\le s,\quad \gamma\not\in\mez \xN.
$$
Consider 
$(\eta,\psi)\in \Eg{\gamma}\cap \bigl( H^{s+\mez}(\xR)\times \h{\mez,\mu-\mez}(\xR)\bigr)$. 
Then 
$G(\eta)\psi\in H^{\mu-1}(\xR)$ and there exists 
a non decreasing function~$C\colon \xR_+\rightarrow \xR_+$ 
depending only on~$(s,\gamma,\mu)$ such that:
\begin{equation}\label{n120}
\lA G(\eta)\psi\rA_{H^{\mu-1}}\le C\left( \lA \eta\rA_{\eC{\gamma}}\right)\left\{\dalpha
\lA \eta\rA_{H^s}+\blA \Dxmez \psi \brA_{H^{\mu-\mez}}\right\}.
\end{equation}

\end{prop}
\begin{proof}
We begin by proving the following 
estimates.
\begin{lemm}\label{T15}
Let~$\tau<\tau'<0$ and consider~$(s,\gamma,\mu)$ as above. There exists a non decreasing function 
$C\colon \xR_+\rightarrow \xR_+$ 
such that for all~$(\eta,\psi)\in \Eg{\gamma}\cap \bigl( H^{s+\mez}(\xR)\times \h{\mez,\mu}(\xR)\bigr)$, 
\begin{multline*}
\sup_{z\in [\tau',0]} \lA \partial_z\varphi\rA_{H^{\mu-\mez}}
+\sup_{z\in [\tau',0]} 
\lA \partial_x\varphi-\partial_x \eta \partial_z\varphi\rA_{H^{\mu-\mez}}\\
\le C\left( \lA \eta\rA_{\eC{\gamma}}\right)\Bigl\{ \dalpha\lA \eta\rA_{H^s}+
\lA \partial_x \omega\rA_{H^{\mu-\mez}}+E(\tau)+D(\tau)\Bigr\},
\end{multline*}
where~$D(\tau)$ and~$E(\tau)$ are as in \eqref{n108}. 
\end{lemm}
\begin{proof}
We begin by estimating~$\partial_z\varphi$. To do so, write 
$\partial_z\varphi=\partial_z W+T_{\partial_z^2\varphi}\eta$ to obtain
\begin{equation}\label{n121}
\sup_{z\in [\tau',0]} \lA \partial_z \varphi\rA_{H^{\mu-\mez}}\les 
\sup_{z\in [\tau',0]}\lA \partial_z W\rA_{H^{\mu-\mez}} 
+\sup_{z\in [\tau',0]}\lA \partial_z^2\varphi\rA_{L^\infty}
\lA \eta\rA_{H^{\mu-\mez}}.
\end{equation}
It follows from \eqref{n110} applied with $\mu$ replaced with 
$\mu-1/2\in [0,s-1/2]$ 
that
\be\label{n121.5}
\sup_{z\in [\tau',0]}\snorm{\nabla_{x,z} W(z)}_{H^{\mu-\mez}}
\le c_1 \lA \eta\rA_{H^s}+c_3\blA \px \omega\brA_{H^{\mu-\mez}}
+c_3 E(\tau)+c_3 D(\tau),
\ee
where
$$
c_1\defn C(\lA \eta\rA_{\eC{\gamma}})
\sup_{z\in [\tau,0]} \lA \nabla_{x,z}\varphi (z)\rA_{\eC{\gamma-1}},
\quad 
c_3\defn C(\lA \eta\rA_{\eC{\gamma}}).
$$
On the other hand, since 
$\partial_z^2 \varphi=-a\partial_x^2\varphi -b\partial_x\partial_z \varphi+c\partial_z \varphi$, 
we have
$$
\sup_{z\in [\tau',0]}\lA \partial_z^2 \varphi\rA_{L^\infty}\le C(\lA \eta\rA_{\eC{2}})\sup
_{z\in [\tau',0]}\lA \nabla_{x,z}\varphi\rA_{\eC{1}}.
$$
As a result, since~$\mu\le s$, \eqref{n121} implies that
$$
\sup_{z\in [\tau',0]}\snorm{\partial_{z} \varphi}_{H^{\mu-\mez}}
\le c_1 \lA \eta\rA_{H^s}+c_3\blA \px \omega\brA_{H^{\mu-\mez}}
+c_3 E(\tau)+c_3 D(\tau),
$$
and the asserted estimate for $\pz\va$ follows from 
\eqref{1118} which implies that $c_1$ is estimated by $C(\lA \eta\rA_{\eC{\gamma}})\dalpha$. 

The estimate for~$\partial_x\varphi-\partial_z\varphi\partial_x \eta$ follows from similar arguments, the decomposition
$$
\partial_x\varphi-\partial_z\varphi\px \eta=
\px W +T_{\px\partial_z\varphi}\eta
-T_{\px\eta}\partial_z\varphi -\RBony(\partial_z\varphi,\px\eta),
$$
and the classical estimates for paraproducts (see~\eqref{Bony3} in the appendix).
\end{proof}

We now apply Lemma~\ref{T15} 
to infer the tame 
estimates \eqref{n120} and \eqref{n119}. 
Clearly, since
$$
\B(\eta)\psi=\partial_z\varphi\arrowvert_{z=0},\qquad 
V(\eta)\psi=(\partial_x \varphi-\partial_z \varphi \partial_x \eta)\arrowvert_{z=0},
$$
Lemma~\ref{T15} implies that
$\lA \B(\eta)\psi\rA_{H^{\mu-\mezl}}$ and 
$\lA V(\eta)\psi\rA_{H^{\mu-\mezl}}$ are bounded by 
$$
C\left( \lA \eta\rA_{\eC{\gamma}}\right)\Bigl\{
\dalpha\lA \eta\rA_{H^s}+
\lA \partial_x \omega\rA_{H^{\mu-\mez}}+E(-1)+D(-1)\Bigr\}.
$$
It follows from \eqref{1115} that
$$
E(-1)+D(-1)\le C(\lA \eta\rA_{\eC{2}})\blA \Dxmez\psi\brA_{L^2}.
$$
Therefore, to complete the proof of \eqref{n119}, it remains only to observe that, 
since~$\psi=\omega+T_{\B(\eta)\psi}\eta$,
\begin{equation}\label{n122}
\begin{aligned}
\blA \Dxmez \psi\brA_{L^2} 
&\les \blA \Dxmez\omega\brA_{L^2}+\blA T_{\B(\eta)\psi}\eta\brA_{H^{\mez}}\\
&\les \blA \Dxmez\omega\brA_{L^2}+\lA \B(\eta)\psi\rA_{L^\infty} \lA\eta\rA_{H^{\mez}}\\
&\les \blA \Dxmez\omega\brA_{H^\mu}+\lA \B(\eta)\psi\rA_{L^\infty} \lA\eta\rA_{H^{s}},
\end{aligned}
\end{equation}
where ~$\lA \B(\eta)\psi\rA_{L^\infty}$ 
is estimated by means of \eqref{211-1}.

Now, by using the estimate \eqref{n119}, 
the usual tame estimate for products (see \eqref{prtame}), \eqref{211-1} and the identity
$$
G(\eta)\psi=\B(\eta)\psi-(V(\eta)\psi)\partial_x\eta,
$$
we then obtain
\begin{equation}\label{n123}
\lA G(\eta)\psi\rA_{H^{\mu-1}}\le C\left( \lA \eta\rA_{\eC{\gamma}}\right)\left\{
\dalpha\lA \eta\rA_{H^s}+\blA \Dxmez \omega \brA_{H^{\mu-\mez}}\right\}.
\end{equation}
Now since~$\mu\le s$, by definition of~$\omega=\psi-T_{\B(\eta)\psi}\eta$ and \eqref{211-1}, we have
\begin{equation}\label{n124}
\blA \Dxmez \omega \brA_{H^{\mu-\mez}}\le \blA \Dxmez \psi \brA_{H^{\mu-\mez}}
+C\left( \lA \eta\rA_{\eC{\gamma}}\right)\dalpha\lA \eta\rA_{H^s}
\end{equation}
and hence~\eqref{n120} follows from \eqref{n123}. 
This completes the proof of Proposition~\ref{T14}.
\end{proof}

\section{Paralinearization of the Dirichlet-Neumann operator}\label{S:paraDN}

We here study the remainder term in the paralinearization formula
\begin{equation*}
F(\eta)\psi=G(\eta)\psi-\Bigl\{\Dx \omega-\partial_x \bigl( T_{V(\eta)\psi}\eta\bigr)\Bigr\}.
\end{equation*}
We prove an extended version of~\eqref{n2} where we add 
two extra parameters~$\mu,\sigma$.

\begin{prop}\label{T16}
Let~$(s,\mu,\gamma)\in \xR^3$ be such that
$$ 
s-\mez >\gamma>3, \quad 1\le \mu\le s,\quad\gamma\not\in\mez\xN.
$$
Assume that 
$(\eta,\psi)$ is in the set $\Eg{\gamma}$ defined after the statement of Proposition~\ref{ref:116} and that 
moreover $(\eta,\psi)\in H^{s}\times \h{\mez}$ 
is such that $\omega=\psi-T_{\B(\eta)\psi}\eta$ is in 
$\h{\mez,\mu-\mez}$. 
Then, for any~$\sigma<\mu+\gamma-3$, $F(\eta)\psi\in H^{\sigma}(\xR)$ and 
\begin{equation}\label{n125}
\lA F(\eta)\psi\rA_{H^{\sigma}}\le 
C\bigl(\lA \eta\rA_{\eC{\gamma}}\bigr)
\left\{ \dalpha  \lA \eta\rA_{H^s}+\lA \eta\rA_{\eC{\gamma}}\blA \Dxmez\omega\brA_{H^{\mu-\mez}}\right\},
\end{equation}
where~$C$ is a non decreasing function 
depending only on~$(s,\mu,\gamma,\sigma)$.
\end{prop}
\begin{rema}
For $\mu\le s$, it follows from \eqref{n124} and \eqref{n125} that
\begin{equation}\label{n126}
\lA F(\eta)\psi\rA_{H^{\sigma}}\le 
C\bigl(\lA \eta\rA_{\eC{\gamma}}\bigr)
\left\{ \dalpha  \lA \eta\rA_{H^s}+\lA \eta\rA_{\eC{\gamma}}\blA \Dxmez\psi\brA_{H^{\mu-\mez}}\right\}.
\end{equation}
\end{rema}
\begin{proof}
We use the notations and results of \S\ref{S:213}. 
\begin{lemm}\label{T17}
There holds
\be\label{n126.5}
\sup_{z\in [\tau,0]}\snorm{\nabla_{x,z} W(z)}_{H^{\mu-1}}
\le C\bigl(\lA \eta\rA_{\eC{\gamma}}\bigr)
\left\{ \dalpha \lA \eta\rA_{H^s}+\blA \Dxmez\omega\brA_{H^{\mu-\mez}}\right\},
\ee
and, for any~$\sigma<\mu+\gamma-3$ and any~$\tau\in [-2,0]$, 
we have
\begin{multline}\label{n126.7}
\sup_{z\in [\tau,0]} \snorm{\partial_z W-P_{+} W}_{H^{\sigma}}\\
\le C\bigl(\lA \eta\rA_{\eC{\gamma}}\bigr)
\left\{ \dalpha \lA \eta\rA_{H^s}+\lA \eta\rA_{\eC{\gamma}}\blA \Dxmez\omega\brA_{H^{\mu-\mez}}\right\}.
\end{multline}
\end{lemm}
\begin{proof}
The first (resp.\ second) estimate follows from \eqref{n121.5}) (resp.\ \eqref{n109}), 
the H\"older estimate \eqref{1118} 
(to bound 
the constant~$c_1$ which appears in \eqref{n121.5} and \eqref{n109}), 
the Sobolev estimate~\eqref{1115} 
(to bound $E(\tau)$ and 
$D(\tau)$) 
and the estimate~\eqref{n122} for~$\blA \Dxmez\psi\brA_{L^2}$.
\end{proof}

Given Lemma~\ref{T15} and 
Lemma~\ref{T17}, the proof of Proposition~\ref{T16} 
now follows from a close inspection of the proof
of Theorem~$1.5$ in~\cite{AM}. 
Recall that, by definition, 
\begin{equation*}
G(\eta)\psi=\bigl[(1+(\px\eta)^2)\partial_z \varphi - \px\eta \px\varphi\bigr]\big\arrowvert_{z=0}.
\end{equation*}
Write
\begin{multline}\label{n127}
(1+ (\px\eta)^2)\partial_z \varphi  - \px\eta \px \varphi \\
=\partial_z\varphi+ T_{(\px\eta)^2}\partial_z \varphi + 2T_{\partial_z \varphi \px\eta}\px \eta 
-\left( T_{ \px\eta} \px \varphi+T_{\px \varphi} \px \eta\right)+R_1,
\end{multline}
where
\begin{align*}
R_1&=\RBony(\partial_z\varphi,(\px\eta)^2)
-\RBony(\partial_x\varphi,\px\eta)\\
&\quad+T_{\partial_z\varphi}\RBony(\px\eta,\px\eta)
+2 \bigl( T_{\partial_z\varphi}T_{\px\eta}-T_{\partial_z\varphi \px \eta}\bigr)\px\eta
\end{align*}
is estimated in~$L^\infty_z(H^\sigma)$ by means of 
the paraproduct rules~\eqref{esti:quant2-func},~\eqref{Bony3} and \eqref{1118}. 
We next replace~$\partial_{z}\varphi$ by 
$\partial_{z}(W+T_{\partial_z \varphi}\eta)$ and~$\px \varphi$ by 
$\px(W+T_{\partial_z \varphi}\eta)$, 
to obtain, 
\begin{align*}
(1+(\px\eta)^2) \partial_z \varphi-\px\eta \px\varphi&=
\partial_z W+T_{(\px\eta)^2}\partial_z W  - T_{ \px \eta} \px W\\
&\quad +T_{(1+(\px\eta)^2)\pz^2\va}\eta -
T_{(\px\eta)\px\pz\va}\eta+T_{(\px\eta)\pz\va}\px\eta-T_{\px\va}\px\eta\\
&\quad +R_1+R_2
\end{align*}
with
\begin{align*}
R_2&=-(T_{(\px\eta)^2}T_{\partial_z^2\varphi}+T_{(\px\eta)^2\partial_z^2\varphi})\eta
+(T_{\px\eta}T_{\partial_z\px\varphi}-T_{(\px\eta) \px\partial_z\varphi})\eta\\
&\quad +(T_{\partial_z\varphi \px \eta}-T_{\px\eta}T_{\partial_z\varphi})\px\eta.
\end{align*}
Again, it follows from  
the paraproduct rules~\eqref{esti:quant2-func} and~\eqref{Bony3} that 
the~$L^\infty_z(H^\sigma)$-norm of~$R_2$ is estimated by the right-hand side of \eqref{n125}. 

Setting this into the right hand side of~\eqref{n127} we obtain
\begin{align*}
&(1+ (\px\eta)^2)\partial_z \varphi  - \px\eta \px \varphi \\
&\qquad =\partial_z W+T_{(\px\eta)^2}\partial_z W  - T_{ \px \eta} \px W\\
&\qquad \quad + T_{\partial_z \varphi \px\eta}\px \eta-T_{\px \varphi} \px \eta\\
&\qquad \quad +T_{(1+(\px\eta)^2)\pz^2\va}\eta -
T_{(\px\eta)\px\pz\va}\eta
+R_1+R_2,
\end{align*}
Now it follows from the elliptic equation satisfied by $\va$ that
\begin{align*}
(1+(\px\eta)^2)\pz^2\va-(\px\eta)\px\pz\va&=-\px^2\va
+(\px\eta)\px\pz\va +\pz\va \px^2\eta\\
&=-\px \bigl( \px\va -(\px\eta)\pz\va\bigr).
\end{align*}
Therefore
\begin{align*}
(1+(\px\eta)^2) \partial_z \varphi-\px\eta \px\varphi&=
\partial_z W+T_{(\px\eta)^2}\partial_z W  - T_{ \px \eta} \px W\\
&\quad-\px\Bigl( T_{\px\varphi-\partial_z \varphi\px\eta}\eta\Bigr)
+R_1+R_2.
\end{align*}

Furthermore, \eqref{n126.7} implies that
$$
\partial_z W+T_{(\px\eta)^2}\partial_z W - T_{\px \eta}\px W 
=P_+W+T_{(\px\eta)^2}P_+ W -T_{\px \eta}\px W +r_1
$$
where the~$L^\infty_z(H^\sigma)$-norm of~$r_1$ is estimated by the right-hand side of 
\eqref{n125}. 
Now write
$$
P_+W+T_{(\px\eta)^2}P_+ W -T_{\px \eta}\px W
= (\Dx +T_{\lambda-\la\xi\ra}) W +r_2,
$$
with
\begin{equation*}
\lambda=(1 + (\px\eta)^2)P -i \px\eta  \xi,
\end{equation*}
($P$ is given by \eqref{n100}) and where
$$
r_2=\bigl(T_{(\px\eta)^2}T_P-T_{(\px\eta)^2P}\bigr)W+T_{(\px\eta)^2}\bigl(\Dx-T_{|\xi|}\bigr)W.
$$
It follows from \eqref{n126.5} and \eqref{esti:quant2sharp-px} that 
the~$L^\infty_z(H^\sigma)$-norm of~$r_2$ is estimated by the right-hand side of 
\eqref{n125}. 

Now, since~$\lambda=\la\xi\ra$, by \e{n100}, and since 
$\px \varphi-\partial_z \varphi\px\eta \arrowvert_{z=0}=V$ and $W\arrowvert_{z=0}=\omega$, we conclude that
$$
(1+(\px\eta)^2) \partial_z \varphi-\px\eta \px\varphi=\Dx \omega-\px (T_V \eta)+
\big[ R_1+R_2+r_1+r_2\bigr]\big\arrowvert_{z=0}.
$$
This concludes the proof of Proposition~\ref{T16}.
\end{proof}

\section{Linearization of the Dirichlet-Neumann operator}\label{S:LinearizationDN}

In this section, we prove the estimates~\eqref{2117}. 
For later purposes, it will be convenient to 
prove the following sharp estimates which depend on an additional parameter~$\mu$.

\begin{prop}\label{T18}
Let~$(s,\mu,\gamma)\in \xR^3$ be such that
$$ 
s-\mez >\gamma>3, \quad \tdm\le \mu\le s,\quad\gamma\not\in\mez\xN.
$$
Consider 
$(\eta,\psi)\in \Eg{\gamma}\cap \bigl( H^{s+\mez}(\xR)\times \h{\mez,\mu-\mez}(\xR)\bigr)$ 
set~$\omega=\psi-T_{\B(\eta)\psi}\eta$. 
There exists a non decreasing function~$C\colon \xR_+\rightarrow \xR_+$ 
depending only on~$(s,\gamma,\mu)$ such that
\begin{align}
\lA G(\eta)\psi-\Dx \psi\rA_{H^{\mu-1}}&\le 
\Cr 
\dalpha \lA \eta \rA_{H^s}+\Cr\lA \eta \rA_{\eC{\gamma}}
\blA \Dxmez\psi\brA_{H^{\mu-\tdm}},\label{n129}\\
\lA B(\eta)\psi-\Dx \omega\rA_{H^{\mu-\mez}}
&\le 
\Cr \dalpha \lA \eta \rA_{H^s}+\Cr\lA \eta \rA_{\eC{\gamma}}
\blA \Dxmez\omega\brA_{H^{\mu}},\label{n130}\\
\lA V(\eta)\psi-\partial_x \omega \rA_{H^{\mu-\mez}}
&\le \Cr \dalpha \lA \eta \rA_{H^s}+\Cr\lA \eta \rA_{\eC{\gamma}}
\blA \Dxmez\omega\brA_{H^{\mu}},\label{n131}
\end{align}
where $\Cr=C\bigl(\lA \eta\rA_{\eC{\gamma}}\bigr)$.
\end{prop}
\begin{proof}
Abbreviate~$\B=\B(\eta)\psi$ and~$V=V(\eta)\psi$. 
In view of the definition \eqref{n1} of~$F(\eta)$, 
we can rewrite~$G(\eta)\psi-\Dx \psi$ as 
$$
G(\eta)\psi-\Dx \psi=-\Dx T_{\B}\eta-\partial_x(T_{V}\eta)+F(\eta)\psi.
$$
Using~\eqref{esti:quant0}, it follows that
$$
\lA G(\eta)\psi-\Dx \psi\rA_{H^{\mu-1}}
\les \bigl(\lA \B\rA_{L^\infty}+\lA V\rA_{L^\infty}\bigr)\lA \eta\rA_{H^{\mu}}
+\lA F(\eta)\psi\rA_{H^{\mu-1}}.
$$
Since~$\gamma-3>0$, 
the estimate~\eqref{n126} (with $(\sigma,\mu)$ replaced with $(\mu-1,\mu-1)$) 
for~$F(\eta)\psi$ 
implies that
\begin{equation}\label{n132}
\lA F(\eta)\psi\rA_{H^{\mu-1}}\le \Cr 
\dalpha \lA \eta \rA_{H^s}+\Cr\lA \eta \rA_{\eC{\gamma}}
\blA \Dxmez\psi\brA_{H^{\mu-\tdm}}.
\end{equation}
The estimate~\eqref{n129} then follows from 
the~$L^\infty$-estimate of~$(\B,V)$ (see~\eqref{211-1}).

Since~$\B-V\px\eta=G(\eta)\psi$ (c.f.\ \eqref{212}) we have 
$B-V\px\eta=\Dx\omega-\px(T_{V}\eta)+F(\eta)\psi$, so
$$
B-\Dx \omega=V\px\eta-\px(T_{V}\eta)+F(\eta)\psi.
$$
Since
$$
V\px\eta-\px(T_{V}\eta)
=T_{\px\eta}V+\RBony(V,\px\eta)-T_{\px V}\eta,
$$
we obtain
$$
B=\Dx \omega+T_{\px\eta}V-T_{\px V}\eta+\RBony(V,\px\eta)+F(\eta)\psi.
$$
The estimate \eqref{n130} follows from the tame estimate for~$V$ (see~\eqref{n119}), 
the estimate \eqref{n125} for~$F(\eta)\psi$ 
and 
the classical estimates for paraproducts (see~\eqref{esti:quant0} and \eqref{Bony3}) together with \e{211-1}.

Similarly, with regards to~$V=\px\psi-\B\px\eta$, replace~$\psi$ by~$\omega+T_\B \eta$ to obtain
\begin{align*}
V&=\px\psi-B\px\eta=\px\omega+\px(T_B\eta)-B\px\eta\\
&=\px\omega+T_{\px\B}\eta-T_{\px\eta}\B-\RBony(\B,\px\eta).
\end{align*}
Consequently, the estimate \eqref{n131} 
follows from \eqref{n119},~\eqref{esti:quant0} and \eqref{Bony3}.
\end{proof}
\begin{rema}\label{T19}
Assume that $3/2\le \mu\le s-1/2$ instead of $3/2\le \mu\le s$. Then, 
since~$\psi=\omega+T_{\B(\eta)\psi}\eta$, it follows from~\eqref{esti:quant0} and 
\eqref{211-1} that
\begin{equation}\label{n133}
\begin{aligned}
\lA \Dx\psi-\Dx \omega\rA_{H^{\mu-\mez}}&=\lA \Dx T_{\B(\eta)\psi}\eta\rA_{H^{\mu-\mez}} \les 
\lA \B(\eta)\psi\rA_{L^\infty}\lA \eta\rA_{H^{s}}\\
&\le C\left(\dbeta\right) \dalpha \lA \eta\rA_{H^{s}}.
\end{aligned}
\end{equation}
Similarly~$\blA \Dxmez\psi-\Dxmez \omega\brA_{H^{\mu}}$ 
and $\lA \partial_x\psi-\partial_x \omega\rA_{H^{\mu-1/2}}$ are bounded by the right-hand side 
of~\eqref{n133}. 
The estimates \eqref{n130}--\eqref{n131} then imply that
\begin{multline}\label{n134}
\lA B(\eta)\psi-\Dx \psi\rA_{H^{\mu-\mez}}
+\lA V(\eta)\psi-\partial_x \psi \rA_{H^{\mu-\mez}}\\
\le C\left(\dbeta\right)\Bigl\{ \dalpha \lA \eta\rA_{H^s}
+\dbeta \blA \Dxmez \psi\brA_{H^{\mu}}\Bigr\}.
\end{multline}
The previous estimates means that $\B(\eta)-\Dx$ and $V(\eta)-\px$ are operator of order $1$: 
they map $H^{\mu+\mez}(\xR)$ to $H^{\mu-\mez}(\xR)$. In sharp contrast, the estimate 
\eqref{n129} means that $G(\eta)-\Dx$ is an operator of order $0$. 
In fact even more is true: $G(\eta)-\Dx$ is a smoothing operator. 
Indeed, the proof of \eqref{n129} shows that, if we further assume that 
$\mu>s+2-\gamma$ and if we use \eqref{n125} instead of \eqref{n132}, then 
we obtain that $\lA G(\eta)\psi-\Dx \psi\rA_{H^{s-1}}$ is bounded by 
$$
C\left(\dbeta\right)
\Bigl\{ \dalpha \lA \eta\rA_{H^s}+\dbeta \blA \Dxmez \psi\brA_{H^{\mu}}\Bigr\}.
$$
\end{rema}

\section{Taylor expansions}\label{S:TaylorDN}

We here study the Taylor expansions of the Dirichlet-Neumann operator~$G(\eta)$ 
with respect to the free surface elevation~$\eta$. 
Craig, Schanz and Sulem (see \cite{CSS} and~\cite[Chapter~$11$]{SuSu}) have shown that one can expand the 
Dirichlet-Neumann operator as a sum of pseudo-differential operators and gave 
precise estimates for the remainders. 
We present now another demonstration of this property which 
gives tame estimates. 
Tame estimates are proved in \cite{CSS} and \cite{ASL,IP}.  
Our approach depends on the 
paralinearization of the Dirichlet-Neumann operator with tame estimates. 
Furthermore, the scheme of proof allows us to prove similar expansions for 
the operators~$\B(\eta)$,~$V(\eta)$. The key result of this section is the estimate 
\eqref{n137} for~$F(\eta)\psi-F_{\quadratique}(\eta)\psi$. 

Denote by~$A(\eta)$ either~$G(\eta)$ or one of the operators 
$\B(\eta)$,~$V(\eta)$ and~$F(\eta)$. 
In this section, we compare~$A(\eta)$ to~$A_{\quadratique}(\eta)$ where
\index{Dirichlet-Neumann operator!$G_{\quadratique}(\eta)\psi$}
\index{Dirichlet-Neumann operator!$B_{\quadratique}(\eta)\psi$}
\index{Dirichlet-Neumann operator!$V_{\quadratique}(\eta)\psi$}
\index{Dirichlet-Neumann operator!$F_{\quadratique}(\eta)\psi$}
\begin{align}
G_{\quadratique}(\eta)\psi&\defn \la D_x \ra\psi-\Dx(\eta\Dx\psi)-\partial_x(\eta\partial_x\psi),\notag\\[0.5ex]
\B_{\quadratique}(\eta)\psi&\defn G_{\quadratique}(\eta)\psi+\partial_x\eta \partial_x \psi,\notag\\[0.5ex]
V_{\quadratique}(\eta)\psi&\defn \partial_x \psi- \partial_x \eta \Dx \psi,\notag\\[0.5ex]
F_{\quadratique}(\eta)\psi&=
-\Dx(\eta\Dx \psi)+\Dx (T_{\Dx\psi}\eta)-\partial_x(\eta\partial_x\psi)
+\partial_x(T_{\partial_x\psi}\eta).\label{n135}
\end{align}
\begin{rema*}When we later compute the cubic resonances, 
we will be forced to study cubic approximations 
to the Dirichlet-Neumann operator. The proof of the next proposition 
contains also the analysis of the cubic terms.
\end{rema*}

\begin{prop}\label{T21}
Let~$(s,\gamma,\mu)\in \xR^3$ be such that 
$$
s-1/2>\gamma\ge 14,\quad s\ge \mu\ge 5,\quad \gamma\not\in \mez \xN,
$$ 
and 
consider~$(\eta,\psi)\in H^{s+\mez}(\xR)\times (\eC{\gamma}(\xR)\cap \h{\mez,\mu}(\xR))$ 
such that the condition~\eqref{1117} is satisfied. Then the following estimates hold.

There exists a non decreasing function~$C\colon \xR\rightarrow \xR$ such that, 
for any~$A\in \{G,B,V\}$,
\begin{multline}\label{n136}
\lA A(\eta)\psi-A_{\quadratique}(\eta)\psi\rA_{H^{\mu-1}}\\ \le 
C(\lA \eta\rA_{\eC{\gamma}} )
\lA \eta\rA_{\eC{\gamma}} \Bigl\{\dalpha \lA \eta \rA_{H^s}+\lA \eta \rA_{\eC{\gamma}}
\blA \Dxmez\psi\brA_{H^\mu}\Bigr\},
\end{multline}
and
\begin{multline}\label{n137}
\lA F(\eta)\psi-F_{\quadratique}(\eta)\psi\rA_{H^{\mu+1}}\\
\le C(\lA \eta\rA_{\eC{\gamma}} )
\lA \eta\rA_{\eC{\gamma}} \Bigl\{\dalpha \lA \eta \rA_{H^s}+\lA \eta \rA_{\eC{\gamma}}
\blA \Dxmez\psi\brA_{H^\mu}\Bigr\}\cdot
\end{multline}
\end{prop}
\begin{rema}
The estimates~\eqref{n3} and 
\eqref{n137} applied with $\mu=s-1/2$ imply that
\begin{multline}\label{n138}
\lA F(\eta)\psi-F_{\quadratique}(\eta)\psi\rA_{H^s}\\
\le 
C(\lA \eta\rA_{\eC{\gamma}} )
\lA \eta\rA_{\eC{\gamma}} \Bigl\{\dalpha \lA \eta \rA_{H^s}+\lA \eta \rA_{\eC{\gamma}}
\blA \Dxmez\omega\brA_{H^s}\Bigr\},
\end{multline}
where 
recall that~$\omega(\eta)\psi=\psi-T_{\B(\eta)\psi}\eta$. 
\end{rema}
\begin{proof}
We shall need to consider the cubic terms in the Taylor expansions 
of~$G(\eta)$,~$B(\eta)$ and~$V(\eta)$. 
Set
\index{Dirichlet-Neumann operator!$G_{(\le 3)}(\eta)\psi$}
\index{Dirichlet-Neumann operator!$B_{(\le 3)}(\eta)\psi$}
\index{Dirichlet-Neumann operator!$V_{(\le 3)}(\eta)\psi$}
\begin{align*}
G_{(\le 3)}(\eta)\psi & \defn 
G_{\quadratique}(\eta)\psi+\Dx(\eta(\Dx(\eta\Dx \psi))) +\mez \Dx 
(\eta^2\partial_x^2\psi)\\
&\quad +\mez \partial_x^2(\eta^2 \Dx\psi),\\
\B_{\cubique}(\eta)\psi&\defn G_{\cubique}(\eta)\psi+\partial_x\eta \partial_x \psi-(\partial_x\eta)^2\Dx \psi,\\
V_{\cubique}(\eta)\psi&\defn
\partial_x \psi- \partial_x \eta \B_{\quadratique}(\eta)\psi.
\end{align*}
For $k\in\{1,2,3\}$, set
\begin{equation*}
T_k\defn 
\lA \eta\rA_{\eC{\gamma}}^{k-1} \Bigl\{\dalpha \lA \eta \rA_{H^s}+\lA \eta \rA_{\eC{\gamma}}
\blA \Dxmez\psi\brA_{H^\mu}\Bigr\}.
\end{equation*}
The proof is in four steps. In the first two steps we prove the weaker estimates:
$$
\lA A(\eta)\psi-A_{\quacub}(\eta)\psi\rA_{H^{\mu-k-1}}\le 
C(\lA \eta\rA_{\eC{\gamma}} )T_k,
$$
for 
$A\in\{G,\B,V\}$ and~$k\in \{2,3\}$. 
(For~$k=2$, comparing this with \eqref{n136} we see a loss of~$2$ derivatives.) 
Then, in the second step, we prove~\eqref{n137}. 
This is the key step. Indeed, once~\eqref{n137} 
is granted, we show in 
the fourth step that one can obtain the 
optimal estimates stated in the above proposition 
for~$A(\eta)-A_{\quadratique}(\eta)$ with~$A\in\{G,B,V\}$.

\step{1}{First estimates for~$G(\eta)$}

In this step we prove that
\begin{align}
&\lA G(\eta)\psi-\Dx \psi\rA_{H^{\mu-2}}\le 
C(\lA \eta\rA_{\eC{\gamma}} )T_1,
\label{n139}\\
&\lA G(\eta)\psi-G_{\quadratique}(\eta)\psi\rA_{H^{\mu-3}}\le 
C(\lA \eta\rA_{\eC{\gamma}} )T_2,
\label{n140}\\
&\lA G(\eta)\psi-G_{\cubique}(\eta)\psi\rA_{H^{\mu-4}}\le 
C(\lA \eta\rA_{\eC{\gamma}} )T_3
\label{n141}.
\end{align}
To do so, we use the property, proved by Lannes~\cite{LannesJAMS}, that one has 
an explicit expression of the derivative of~$G(\eta)\psi$ 
with respect to~$\eta$. 
Introduce~$g\colon [0,1]\rightarrow H^{\mu-1}(\xR)$ 
defined by~$g(\lambda)=G(\lambda\eta)\psi$. 
Then
\begin{equation}\label{n142}
g'(\lambda)=-G(\lambda \eta)(\eta b_0(\lambda))-\partial_x( \eta v_0(\lambda)),
\end{equation}
where~$b_0(\lambda)\defn \B(\lambda\eta)\psi$ and~$v_0(\lambda)=V(\lambda\eta)\psi$. 
Since 
$$
b_0(\lambda)\defn \B(\lambda\eta)\psi
=\frac{g(\lambda)+\lambda \partial_x \eta  \partial_x\psi}{1+\lambda^2(\partial_x\eta)^2},
\qquad 
v_0(\lambda)=\partial_x\psi-\lambda b_0(\lambda)\partial_x\eta,
$$
it follows that~$b_0$ and~$v_0$ are~$C^1$ from~$[0,1]$ to~$H^{\mu-2}(\xR)$, with
\begin{align*}
b_0'(\lambda)&=\frac{1}{1+\lambda^2(\partial_x\eta)^2}\Bigl( g'(\lambda) +\px \eta \px \psi
-2\lambda (\px\eta)^2b_0(\lambda)\Bigr),\\
v_0'(\lambda)&=-b_0(\lambda)\px\eta-\lambda b_0'(\lambda)\px\eta.
\end{align*}

These expressions show that $g'(\lambda)$, $b_0'(\lambda)$, $v_0(\lambda)$ 
may be written as sums of expressions of the form 
$a_2(\lambda,\eta,\eta')A_2(\lambda\eta)a_1(\lambda,\eta,\eta')A_1(\lambda\eta)$ 
where $a_1,a_2$ are analytic functions of their 
argument with $a_1(\lambda,0,0)=0$ and $A_1(\eta)$, $A_2(\eta)$ belong to 
$\{ G(\eta),\B(\eta),V(\eta),\px\}$. Moreover, in the case of $g'(\lambda)$, one 
may assume that $a_2$ is constant and that $A_2(\eta)$ belongs to 
$\{G(\eta),\px\}$. 

We may thus iterate this computation, which shows that $g(\lambda)$ is $C^k$ 
with values in $H^{\mu-1-k}$, and $g^{(\ell)}(\lambda)$ is a sum of expressions of the form
\be\label{2174a}
A_{\ell+1}(\lambda\eta)\prod_{\ell'=1}^\ell a_{\ell'}(\lambda,\eta,\eta')A_{\ell'}(\lambda\eta)
\ee
where $A_{\ell'}(\eta)$ is in $\{ G(\eta),\B(\eta),V(\eta),\px\}$, $\ell'\le \ell$, 
$A_{\ell+1}(\eta)$ in $\{ G(\eta),\px\}$ and $a_{\ell'}$ are analytic functions 
vanishing at $(\eta,\eta')=(0,0)$. To compute the first terms in the Taylor expansion of 
$g$, we need to compute explicitly
\begin{align*}
g''(\lambda)&=-G(\lambda\eta)\bigl(\eta b_1 (\lambda))-\partial_x(\eta v_1(\lambda)),\\
b_1(\lambda)&=b_0'(\lambda)-B(\lambda\eta)(\eta b_0(\lambda)),\\
v_1(\lambda)&=v_0'(\lambda)-V(\lambda\eta)(\eta b_0(\lambda)).
\end{align*}
Since $g(0)=\Dx$, $\B(0)=\Dx$, $V(0)=\px$, it follows from \e{n142} and the above equalities that
$$
g'(0)=-\Dx(\eta\Dx\psi)-\partial_x(\eta\partial_x\psi),
$$
and\label{g'g''}
$$
g''(0)=2\Dx(\eta(\Dx(\eta\Dx \psi))) + \Dx 
(\eta^2\partial_x^2\psi)
+\partial_x^2(\eta^2 \Dx\psi).
$$
If \eqref{1117} is satisfied then $(\eta,\psi)$ belongs to the set $\Eg{\gamma}$ introduced after the statement of Proposition~\ref{ref:116}. 
Using the H\"older estimates~\eqref{211-1} 
we successively prove that, for $k=0,1,2$, we have 
$(\lambda\eta,\eta b_k(\lambda))\in \Egmu{\gamma-k-1}$ and 
according to \e{2174a}
\be\label{n143}
\blA g^{(k+1)}(\lambda)\brA_{\eC{\gamma-k-2}}
\le C(\dbeta)\dbeta^{k+1}\dalpha.
\ee
Using the tame estimate for product~\eqref{prtame} and the tame estimates 
for~$G(\eta)$,~$\B(\eta)$ 
and~$V(\eta)$ (see \eqref{n120}, \eqref{211-1}, and \eqref{n119} applied with 
$\mu$ replaced with $\mu-1/2$ together with \e{n124}), we obtain
$$
\blA g^{(k)}(\lambda)\brA_{H^{\mu-k-1}}\le C(\dbeta)T_k \quad\text{for }k\in\{1,2,3\}.
$$
The desired estimates~\eqref{n139}--\eqref{n141} are then obtained 
by writing that, for~$n=0,1,2$, 
\begin{equation}\label{n144}
G(\eta)\psi=g(1)
=\sum_{k=0}^n \frac{1}{k!}g^{(k)}(0)+\int_0^1\frac{(\lambda-1)^n}{n!}g^{(n+1)}(\lambda)\, d\lambda.
\end{equation}
This completes the proof of \eqref{n140} and \eqref{n141}.

Also, by using \eqref{n143} with $k=0,1$ and \eqref{n144} with $n=0,1$ we have
\begin{align}
&\lA G(\eta)\psi-\Dx\psi\rA_{\eC{\gamma-2}}\le C(\dbeta) \dbeta \dalpha,
\label{n145}\\
&\lA G(\eta)\psi-G_{\quadratique}(\eta)\psi\rA_{\eC{\gamma-3}}\le C(\dbeta) \dbeta^2 \dalpha.
\label{n146}
\end{align}
Notice that \eqref{n145} (resp.\ \eqref{n146}) holds 
for any $\gamma>4$ (resp.\ $\gamma>5$) with $\gamma\not\in\mez\xN$

\step{2}{First estimates for~$B(\eta)$ and~$V(\eta)$}

In this step we prove that
\begin{alignat}{2}
&\lA \B(\eta)\psi-\B_{\quadratique}(\eta)\psi\rA_{H^{\mu-3}}&&\le C(\dbeta)T_2,\label{n147}\\
&\lA \B(\eta)\psi-\B_{\cubique}(\eta)\psi\rA_{H^{\mu-4}}&&\le C(\dbeta)T_3,\label{n148}\\
&\lA V(\eta)\psi-V_{\quadratique}(\eta)\psi\rA_{H^{\mu-3}}&&\le C(\dbeta)T_2,\label{n149}\\
&\lA V(\eta)\psi-V_{\cubique}(\eta)\psi\rA_{H^{\mu-4}}&&\le C(\dbeta)T_3.\label{n150}
\end{alignat}

By definition of~$\B(\eta)\psi$ we have
\begin{align*}
\B(\eta)\psi&=\frac{1}{1+(\partial_x\eta)^2}\bigl(G(\eta)\psi+\partial_x \eta\partial_x \psi\bigr)\\
&=G(\eta)\psi+\partial_x \eta\partial_x \psi-(\partial_x \eta)^2 \B(\eta)\psi.
\end{align*}
Therefore
\begin{equation}\label{n151}
\B(\eta)\psi-\B_{\quadratique}(\eta)\psi 
= G(\eta)\psi -G_{\quadratique}(\eta)\psi
-(\partial_x \eta)^2 \B(\eta)\psi.
\end{equation}
The estimate \eqref{n147} for 
$\B(\eta)\psi-\B_{\quadratique}(\eta)\psi$ 
then easily follows from 
the previous estimate for 
$G(\eta)\psi -G_{\quadratique}(\eta)\psi$ (see~\eqref{n140}); 
indeed the tame estimate for products 
(see~\eqref{prtame}) and the estimates \eqref{211-1} 
and~\eqref{n119} for~$\B(\eta)\psi$ imply that
\begin{equation}\label{n152}
\begin{aligned}
&\lA (\partial_x \eta)^2 \B(\eta)\psi\rA_{H^{\mu-1}}\\
&\qquad\les 
\lA  (\partial_x \eta)^2\rA_{L^\infty}\lA  \B(\eta)\psi\rA_{H^{\mu-1}}
+\lA \B(\eta)\psi\rA_{L^\infty}\lA \partial_x \eta\rA_{L^\infty}
\lA  \partial_x \eta\rA_{H^{\mu-1}}\\
&\qquad
\le C(\dbeta) \Bigl\{\lA \eta\rA_{\eC{\gamma}}^2 \blA \Dxmez\omega\brA_{H^{\mu-\mez}}
+\dbeta\dalpha \lA \eta \rA_{H^s}\Bigr\}\\
&\qquad\le C(\dbeta)T_2.
\end{aligned}
\end{equation}
where we used \eqref{n124} in the last inequality. 
Consequently, \eqref{n147} follows from~\eqref{n140}.

To prove \eqref{n148} we begin by noting that, 
directly from the definition of~$\B(\eta)\psi$, 
the estimate \eqref{n139}Ê
implies that 
\begin{equation}\label{n153}
\lA \B(\eta)\psi-\Dx\psi\rA_{H^{\mu-2}}\le C(\dbeta)T_1,
\end{equation}
Similarly, the estimate \eqref{n145}Ê
implies that 
\begin{equation}\label{n154}
\lA \B(\eta)\psi-\Dx \psi\rA_{\eC{\gamma-2}}\le C(\dbeta) \dbeta \dalpha. 
\end{equation}
By definition $(\B(\eta)-\B_{\cubique}(\eta))\psi=
(G(\eta)-G_{\cubique}(\eta))\psi-(\px\eta)^2 [ B(\eta)-\Dx]\psi$. The first term is estimated in 
\eqref{n141} by the right hand side of \eqref{n148}. The second one is bounded using 
\eqref{n153}, \eqref{n154} and the tame estimate~\eqref{prtame}. This proves \eqref{n148}. 

Since~$V(\eta)\psi=\partial_x \psi-(\B(\eta)\psi)\partial_x\eta$, 
the estimates \eqref{n149} and~\eqref{n150} are consequences 
of the tame product rule in Sobolev spaces (see~\eqref{prtame}) and  
the estimates~\eqref{n147}, \eqref{n148}.

For later references, we also record the following estimates
\begin{align}
&\lA V(\eta)\psi-\px \psi\rA_{{\eC{\gamma-2}}}\le C(\dbeta) \dbeta \dalpha,\label{n155}\\
&\lA \B(\eta)\psi-\B_{\quadratique}(\eta)\psi\rA_{{\eC{\gamma-3}}}
\le C(\dbeta) \dbeta^2 \dalpha,\label{n156}\\
&\lA V(\eta)\psi-V_{\quadratique}(\eta)\psi\rA_{{\eC{\gamma-3}}}
\le C(\dbeta) \dbeta^2 \dalpha.\label{n157}
\end{align}
The estimates \eqref{n155} and \eqref{n157} 
follow from the definition of~$V(\eta)\psi=\px\psi-(\B(\eta)\psi)\px\eta$ 
and \eqref{n154}. The estimate~\eqref{n156} follows from~\eqref{n151} 
and \eqref{n146}. 

\step{3}{Key estimate} 

In this step we prove that
\begin{equation}\label{n158}
\lA F(\eta)\psi-F_{\quadratique}(\eta)\psi\rA_{H^{\mu+1}}
\le C\left(\dbeta\right)T_2.
\end{equation}
The proof is based on an interpolation inequality which requires 
to take into account the cubic terms. 
Introduce~$F_{\cubique}(\eta)$ defined by\index{Dirichlet-Neumann operator!$F_{(\le 3)}(\eta)\psi$}
$$
F_{\cubique}(\eta)\psi=
G_{\cubique}(\eta)\psi-\Bigl\{ \Dx (\psi-T_{\B_{\quadratique}(\eta)\psi}\eta)-
\partial_x(T_{V_{\quadratique}(\eta)\psi}\eta)\Bigr\}.
$$
\begin{lemm}There exist a constant $K>0$ 
such that for 
all $(\eta,\psi)\in H^s(\xR)\times \h{\mez,\mu-\mez}(\xR)$,
\begin{align}
&\lA F_{\quadratique}(\eta)\psi\rA_{H^{\mu+\gamma-2}}
\le K \lA\eta\rA_{\eC{\gamma}} \lA \px\psi\rA_{H^{\mu-1}},\label{n159}
\\
&\lA F_{\cubique}(\eta)\psi-F_{\quadratique}(\eta)\psi\rA_{H^{\mu+\gamma-3}}\le K T_2.
\label{n160}
\end{align}
\end{lemm}
\begin{rema*}
It follows from~\eqref{n159}, \eqref{n160} and the triangle inequality that
\begin{equation}\label{n161}
\lA F_{\cubique}(\eta)\psi\rA_{H^{\mu+\gamma-3}}\le  C(\lA \eta\rA_{\eC{\gamma}} ) T_1.
\end{equation}
\end{rema*}
\begin{proof}
Notice that one can write~$F_{\quadratique}(\eta)\psi$ under the form
\begin{align*}
F_{\quadratique}(\eta)\psi
&=
-\Dx(\eta\Dx \psi)+\Dx (T_{\Dx\psi}\eta)-\partial_x(\eta\partial_x\psi)+\partial_x(T_{\partial_x\psi}\eta)\\
&=-\Dx(T_{\eta}\Dx \psi)
-\px (T_\eta\px\psi)\\
&\quad-\Dx \RBony(\eta,\Dx\psi)-\px \RBony(\eta,\px\psi).
\end{align*}
Now the identity \eqref{A:i1} in Lemma~\ref{lemm:DxaDx} of Appendix~\ref{s2}Ê
implies that
\begin{equation}\label{n162}
\Dx T_{\eta}\Dx +\partial_x T_{\eta}\partial_x=0.
\end{equation}
Thus
\begin{equation}\label{n163}
F_{\quadratique}(\eta)\psi=-\Dx \RBony(\eta,\Dx\psi)-\px \RBony(\eta,\px\psi),
\end{equation}
and the estimate~\eqref{n159} thus follows from~\eqref{Bony3}. 

It remains to prove \eqref{n160}. 
Below, for $A\in \{G,\B,V,F\}$, we set $A_{(k)}\defn A_{(\le k)}(\eta)\psi-A_{(\le k-1)}(\eta)\psi$. 
We begin by noticing that
\begin{align*}
G_{(3)}&=-\Dx (\eta B_{(2)})-\px (\eta V_{(2)})+D,\\
D &\defn \mez \Dx (\eta^2\Dx^2 \psi)
+\mez\px (\eta^2\px \Dx \psi),
\end{align*}
which can be checked by direct computations from the 
definitions of $B_{(2)}$, $V_{(2)}$ and $G_{(3)}$. Thus,
\begin{equation}\label{n164}
F_{(3)}=-\Dx T_\eta B_{(2)} 
-\px T_\eta V_{(2)}+D+R_1,
\end{equation}
where $R_1\defn -\Dx \RBony(\eta,\B_{(2)})-\px \RBony(\eta,V_{(2)})$ is estimated by means of \eqref{Bony3}. 

Now observe that 
$$
B_{(2)}=G_{(2)}+\px \eta\px \psi
=F_{(2)}-\Dx T_{\Dx \psi}\eta-\px T_{\px\psi}\eta+\px\eta\px\psi.
$$
Setting this and $V_{(2)}=-\px\eta\Dx\psi$ into~\eqref{n164} yields
\begin{align*}
F_{(3)}
&=-\Dx T_\eta F_{(2)} +D +R_1\\
&\quad 
+\Dx T_\eta \Dx T_{\Dx\psi}\eta 
+\Dx T_\eta \px T_{\px\psi}\eta \\
&\quad-\Dx T_\eta \px \eta \px \psi +\px T_\eta \px\eta \Dx \psi.
\end{align*}
Since 
$D =-\mez (G_{\quadratique}(\eta^2)\Dx \psi-\Dx^2\psi)$ we have
$$
D=\mez\Dx T_{\Dx^2\psi}\eta^2
+\mez \px T_{\px\Dx\psi}\eta^2-\mez F_{\quadratique}(\eta^2)\Dx\psi.
$$
The cancellation~\eqref{n162} implies that
$$
\Dx T_\eta \Dx T_{\Dx\psi}\eta=-\px T_\eta \px T_{\Dx\psi}\eta.
$$
Using this identity and replacing $\eta^2$ by $2T_\eta \eta +\RBony(\eta,\eta)$, 
we obtain after some simplifications that\index{Dirichlet-Neumann operator!$F_{(3)}$}
\begin{equation}\label{n165}
\begin{aligned}
F_{(3)}&=-\Dx T_\eta F_{\quadratique}(\eta)\psi 
-\mez F_{\quadratique}(\eta^2)\Dx\psi \\[0.5ex]
&\quad -\Dx T_\eta T_{\px \eta}\px\psi +\px T_\eta T_{\px\eta}\Dx\psi\\[0.5ex]
&\quad +\Dx T_\eta T_{\px^2\psi}\eta +\Dx T_{\Dx^2\psi}T_{\eta}\eta\\[0.5ex]
&\quad +R_1+R_2,
\end{aligned}
\end{equation}
with
\begin{align*}
R_2&=\mez \Dx T_{\Dx^2\psi}\RBony(\eta,\eta)+\mez \px T_{\px\Dx\psi}\RBony(\eta,\eta)\\
&\quad +\px T_\eta \RBony(\px\eta,\Dx\psi)-\Dx T_\eta \RBony (\px\eta,\px\psi).
\end{align*}
The remainder $R_2$ is estimated by means of \eqref{Bony3}. 
The first two terms in the right-hand side of \eqref{n165} are estimated by means of the estimate~\eqref{n159} 
for $F_{\quadratique}(\eta)$. The fifth and the sixth terms are estimated by means of symbolic 
calculus (using the estimate~\eqref{esti:quant2-func} and $\Dx^2=-\px^2$). 
To conclude the proof it remains only to estimate 
the sum of the third and fourth term, denoted by $\Sigma$. Modulo a term which is estimated 
by means of \eqref{esti:quant2sharp}, $\Sigma=\Sigma'$ 
with
$$
\Sigma'=-\Dx T_{\eta \px \eta}\px\psi +\px T_{\eta \px\eta}\Dx\psi.
$$
Now the cancellation~\eqref{A:i2}Ê
in Lemma~\ref{lemm:DxaDx} implies that $\Sigma'=0$. 
This concludes the proof.
\end{proof}

It follows from \eqref{n141}, \eqref{n156}-\eqref{n157} 
and \eqref{esti:quant0} that 
$F_{\pqq}(\eta)\defn F(\eta)-F_{\cubique}(\eta)$ satisfies
\begin{equation}\label{n166}
\lA F_{\pqq}(\eta)\psi\rA_{H^{\mu-4}}\le 
C(\lA \eta\rA_{\eC{\gamma}})T_3.
\end{equation}
On the other hand, by using the triangle inequality and the estimates~\eqref{n126} 
for $F(\eta)\psi$ and \eqref{n161} for $F_{\cubique}(\eta)\psi$, 
we have
\begin{equation}\label{n167}
\lA F_{\pqq}\psi\rA_{H^{\mu+6}}
\le \lA F(\eta)\psi\rA_{H^{\mu+6}}+\lA F_{\cubique}(\eta)\psi\rA_{H^{\mu+6}}
\le C(\lA \eta\rA_{\eC{\gamma}})T_1,
\end{equation}
where, as already done, we used~\eqref{n124} and 
the fact that~$\mu+\gamma-3> \mu+7$ to 
apply~\eqref{n125} with 
$(\mu,s)$ replaced by $(\mu-1/2,s-1/2)$.

We complete the proof by means of an interpolation inequality. Namely, 
write
$$
\lA F_{\pqq}(\eta)\psi\rA_{H^{\mu+1}}
\le \lA F_{\pqq}(\eta)\psi\rA_{H^{\mu-4}} ^{1/2} 
\lA F_{\pqq}(\eta)\psi\rA_{H^{\mu+6}}^{1/2},
$$
to deduce, from \eqref{n166} and \eqref{n167},
$$
\lA F_{\pqq}(\eta)\psi\rA_{H^{\mu+1}}\le 
C(\lA \eta\rA_{\eC{\gamma}})T_1^{1/2}T_3^{1/2}=C(\lA \eta\rA_{\eC{\gamma}})T_2.
$$
Then write
$$
F(\eta)\psi-F_{\quadratique}(\eta)=F_{\pqq}(\eta)\psi +F_{\cubique}(\eta)\psi-F_{\quadratique}(\eta)\psi,
$$
and use \eqref{n160} to complete the proof of \eqref{n158}. 

\step{4}{Optimal estimates}

Now we return to the estimate of~$G(\eta)-G_{\quadratique}(\eta)$. 
By definition (see~\eqref{n135}), we have
\begin{align*}
F(\eta)\psi&=G(\eta)\psi-\Dx (\psi-T_{\B(\eta)\psi}\eta)+\px(T_V\eta),\\[0.5ex]
F_{\quadratique}(\eta)\psi&=G_{\quadratique}(\eta)\psi-\Dx\psi
+\Dx T_{\Dx\psi}\eta+\px (T_{\px\psi}\eta).
\end{align*}
Subtracting and using \eqref{n158}, \eqref{n154} and \eqref{n155}, we find that 
$G(\eta)-G_{\quadratique}(\eta)$ can be written as the sum of two differences which 
are well-estimated in~$H^{s-1}(\xR)\cup H^{\mu+1}(\xR)\subset H^{\mu-1}(\xR)$. 
This proves \eqref{n136} for~$A=G$. 

Now, using \eqref{n152} and the 
previous control of 
$G(\eta)-G_{\quadratique}(\eta)$  in~$H^{\mu-1}(\xR)$, 
an inspection of the second step yields the desired estimate for 
$B(\eta)-B_{\quadratique}(\eta)$  in~$H^{\mu-1}(\xR)$. This in turn implies the 
estimate for~$V(\eta)-V_{\quadratique}(\eta)$  in~$H^{\mu-1}(\xR)$. 
This completes the proof of \eqref{n136} and hence the proof of the proposition.
\end{proof}

\section{Smooth domains}\label{S:smooth}

In this section, 
we estimate $G(\eta)\psi$, $\B(\eta)\psi$ and $V(\eta)\psi$ in 
the case where $\psi\in H^{\mu}(\xR)$ and $\eta\in \eC{\gamma}(\xR)$ with $\gamma$ larger than $\mu$. 
We study the action of these operators and prove approximation results. 

The main new point is 
the following approximation result for $\B(\eta)\psi$:
\be\label{n168}
\lA \B(\eta)\psi-P_+(\eta)\psi\rA_{H^{\gamma-3}}
\le C\left( \lA \eta\rA_{\eC{\gamma}}\right)\lA \eta\rA_{\eC{\gamma}}
\blA \Dxmez \psi \brA_{L^2}
\ee
where $P_+(\eta)$ is given by
\be\label{n169}
P_+(\eta)=\Dx +T_{P-|\xi|}\quad\text{with}\quad 
P=\frac{1}{1+(\px\eta)^2}( i \px \eta \xi +\la \xi\ra).
\ee
The key point is that the right-hand side of \eqref{n168} 
is at least quadratic in $(\eta,\Dxmez\psi)$ and  
involves only the $L^2$-norm of $\Dxmez \psi$, 
while one bounds $\B(\eta)\psi-P_+(\eta)\psi$ in $H^{\gamma-3}(\xR)$ where $\gamma$ might be arbitrarily large. 
This is not a linearization result for $\B(\eta)\psi$ because 
$P_+(\eta)\neq \Dx$ (except for $\eta=0$). However, 
\eqref{n168} will allow us to prove a sharp linearization estimate for $G(\eta)$ as well as to bound 
$G(\eta)\psi-G_{\quadratique}(\eta)\psi$.

\begin{prop}\label{T22}
Let~$(\gamma,\mu)\in \xR^3$ be such that 
$$
\gamma\ge 3+\mez,\quad \mez \le \mu\le \gamma-2,\quad \gamma\not\in\mez\xN.
$$
$(i)$ Let $\eta\in \eC{\gamma}(\xR)$ and $\psi\in \h{\mez,\mu-\mez}(\xR)$ with the assumption that $\lA \eta\rA_{\eC{\gamma}}$ is small enough. 
Then 
$G(\eta)\psi$, $\B(\eta)\psi$ and $V(\eta)\psi$ belong to $H^{\mu-1}(\xR)$. Moreover, 
there exists a non decreasing function~$C\colon \xR_+\rightarrow \xR_+$ 
depending only on~$(\gamma,\mu)$ such that:
\begin{equation}\label{va1}
\lA G(\eta)\psi\rA_{H^{\mu-1}}
+\lA \B(\eta)\psi\rA_{H^{\mu-1}}
+\lA V(\eta)\psi\rA_{H^{\mu-1}}
\le C\left( \lA \eta\rA_{\eC{\gamma}}\right)
\blA \Dxmez \psi \brA_{H^{\mu-\mez}}.
\end{equation}

$(ii)$ Let $\eta\in \eC{\gamma}(\xR)$ and $\psi\in \h{\mez}(\xR)$ with the assumption that $\lA \eta\rA_{\eC{\gamma}}$ is small enough. Let $P_+(\eta)$ be as given by \eqref{n169}. 
Then 
$G(\eta)\psi-\Dx\psi$ and $\B(\eta)\psi-P_{+}(\eta)\psi$ belong to $H^{\gamma-3}(\xR)$. Moreover, 
there exists a non decreasing function~$C\colon \xR_+\rightarrow \xR_+$ 
depending only on~$\gamma$ such that:
\begin{equation}\label{va2}
\begin{aligned}
&\lA G(\eta)\psi-\Dx\psi\rA_{H^{\gamma-3}}
\le C\left( \lA \eta\rA_{\eC{\gamma}}\right)\lA \eta\rA_{\eC{\gamma}}
\blA \Dxmez \psi \brA_{L^2},\\
&\lA \B(\eta)\psi-P_+(\eta)\psi\rA_{H^{\gamma-3}}
\le C\left( \lA \eta\rA_{\eC{\gamma}}\right)\lA \eta\rA_{\eC{\gamma}}
\blA \Dxmez \psi \brA_{L^2}.
\end{aligned}
\end{equation}
\end{prop}
\begin{rema}$(i)$ As already mentioned in Remark~\ref{T19}, the estimate~\eqref{va2} 
means that $G(\eta)-\Dx$ is a smoothing operator. 

$(ii)$ With the assumptions and notations of statement $(ii)$, notice that 
\eqref{va2} implies that
\begin{equation}\label{n172}
\lA \B(\eta)\psi-\Dx\psi\rA_{H^{\mu-1}}
\le C\left( \lA \eta\rA_{\eC{\gamma}}\right)\lA \eta\rA_{\eC{\gamma}}
\blA \Dxmez \psi \brA_{H^{\mu-\mez}}.
\end{equation}
Indeed, it follows from \eqref{esti:quant0-px}Êthat
$$
\lA P_+(\eta)\psi-\Dx \psi \rA_{H^{\mu-1}}\le C\bigl(\lA \eta\rA_{\eC{\gamma}}\bigr)\lA \eta\rA_{\eC{\gamma}}
\blA \Dxmez \psi\brA_{H^{\mu-\mez}}.
$$
\end{rema}
\begin{proof}
Notice that statement $(i)$ is a corollary of statement $(ii)$. 
This is clear for the regularity results and the estimates for $G(\eta)\psi$ and 
$\B(\eta)\psi$, using the triangle inequality and \e{n172}. 
For $V(\eta)\psi$, this follows from the definition $V(\eta)\psi=\px\psi-(\px \eta)\B(\eta)\psi$ and the product rule \e{pr:sz} (applied with $\rho'=\mu+1>|\mu-1|=\rho$) 
which yields
$$
\lA (\px \eta)\B(\eta)\psi\rA_{H^{\mu-1}}
\les \lA \px \eta\rA_{\eC{\mu+1}}\lA \B(\eta)\psi\rA_{H^{\mu-1}}\le  
C\left( \lA \eta\rA_{\eC{\gamma}}\right)
\lA \px \eta\rA_{\eC{\gamma-1}}\blA \Dxmez \psi \brA_{H^{\mu-\mez}},
$$
where we used the estimate \e{va1} for $\B(\eta)\psi$ and the assumption $\gamma\ge \mu-2$.

To prove statement $(ii)$ we use the strategy used previously to study $G(\eta)\psi$. Recall that 
\begin{equation}\label{n173}
\left\{
\begin{aligned}
G(\eta)\psi&= (1+(\partial_x  \eta)^2)\partial_z \varphi
-\partial_x \eta \partial_x  \varphi\big\arrowvert_{z=0},\\[0.5ex]
\B(\eta)\psi&=\partial_z\varphi\arrowvert_{z=0},
\end{aligned}
\right.
\end{equation}
where $\varphi=\varphi(x,z)$ solves the Dirichlet problem:
\begin{alignat}{2}
&\partial_z^2 \varphi+a\partial_x^2\varphi +b\partial_x\partial_z \varphi-c\partial_z \varphi=0
\quad &&\text{in }\{z<0\},\label{n174}\\
&\varphi=\psi&&\text{on }\{z=0\},\label{n175}
\end{alignat}
where $a=(1+(\partial_x\eta)^2)^{-1}$, $b=-2a\partial_x \eta$, $c=a\partial_x^2\eta$. 
It follows from Proposition~\ref{ref:116} that, if $\lA \eta\rA_{\eC{\gamma}}$ is small enough, then 
there exists indeed a unique solution~$\varphi$ to \eqref{n174}--\eqref{n175}. 
Moreover, $\nabla_{x,z} \varphi$ is continuous in $z\in ]-\infty,0]$ with 
values in $H^{-1/2}(\xR)$ and there exists 
a non decreasing function~$C\colon \xR_+\rightarrow\xR_+$ independent of 
$\eta,\psi$ such that
\begin{equation}\label{n176}
\sup_{z\in ]-\infty,0]}\blA\nabla_{x,z} (\varphi(z)-e^{z\Dx}\psi)\brA_{H^{-\mezl}}
\le C(\lA \eta\rA_{\eC{\gamma}}) \lA \eta'\rA_{L^\infty}
\blA \Dxmez \psi\brA_{L^2}
\end{equation}
and
\begin{equation}\label{n177}
\sup_{z\in ]-\infty,0]}\snorm{\nabla_{x,z} \varphi(z)}_{H^{-\mezl}}
\le C(\lA \eta\rA_{\eC{\gamma}}) \blA \Dxmez \psi\brA_{L^2}.
\end{equation}
To prove statement $(ii)$ we paralinearize \eqref{n174} and 
factor out the paradifferential equation thus obtained. 
The desired result then follows from a parabolic regularity result.

We begin with the paralinearization lemma.
\begin{lemm}\label{T23}
There exists a non decreasing function~$C\colon \xR_+\rightarrow \xR_+$ such that
\be\label{n178}
\partial_z^2 \varphi +(\id+T_{a-1})\partial_x^2 \varphi +T_b \partial_x\partial_z \varphi-T_c \partial_z \varphi=f_0
\ee
with
\begin{equation}\label{n179}
\sup_{z\in ]-\infty,0]} \lA f_0(z)\rA_{H^{\gamma-3}}\le 
C(\lA \eta\rA_{\eC{\gamma}})\lA\eta\rA_{\eC{\gamma}} \blA \Dxmez\psi\brA_{L^2}.
\end{equation}
\end{lemm}
\begin{proof} We follow the beginning of the proof of Lemma~\ref{good}. Write 
\begin{align*}
&(a-1)\partial_x^2\varphi =T_{a-1}\partial_x^2\varphi 
+T_{\partial_x^2 \varphi}(a-1)+\RBony(a-1,\partial_x^2\varphi),\\[0.5ex]
&b\partial_x\partial_z \varphi =T_{b}\partial_x\partial_z\varphi 
+T_{\partial_x\partial_z\varphi} b+\RBony(b,\partial_x\partial_z\varphi),\\[0.5ex]
&c\partial_z\varphi =T_{c}\partial_z\varphi +T_{\partial_z\varphi} c+\RBony(c,\partial_z\varphi),
\end{align*}
so that \eqref{n178} holds with
\begin{align*}
f_0&\defn -\bigl(T_{\partial_x^2 \varphi}(a-1)+\RBony(a-1,\partial_x^2\varphi)\bigr)
-\bigl(T_{\partial_x\partial_z\varphi} b+\RBony(b,\partial_x\partial_z\varphi)\bigr)\\
&\quad +T_{\partial_z\varphi} c+\RBony(c,\partial_z\varphi).
\end{align*}

It follows from \eqref{esti:Tba}, 
\eqref{Bony3} and the assumption $\gamma-3>0$ that
\begin{align*}
&\lA T_{\partial_x^2\varphi}(a-1)\rA_{H^{\gamma-3}}
\les \lA a-1 \rA_{\eC{\gamma-1}} \lA \partial_x^2\varphi\rA_{H^{-3/2}},\\[0.5ex]
&\lA T_{\partial_x\partial_z\varphi}b\rA_{H^{\gamma-3}}
\les \lA b\rA_{\eC{\gamma-1}} \lA \partial_x\partial_z\varphi\rA_{H^{-3/2}},\\[0.5ex]
&\lA T_{\partial_z\varphi}c \rA_{H^{\gamma-3}}
\les \lA c\rA_{\eC{\gamma-2}} \lA \partial_z\varphi\rA_{H^{-1/2}},
\end{align*}
and
\begin{align*}
&\lA \RBony(a-1,\partial_x^2\varphi)\rA_{H^{\gamma-3}}
\les \lA a-1 \rA_{\eC{\gamma-1}} \lA \partial_x^2\varphi\rA_{H^{-3/2}},\\[0.5ex]
&\lA \RBony(b,\partial_x\partial_z\varphi)\rA_{H^{\gamma-3}}
\les \lA b\rA_{\eC{\gamma-1}} \lA \partial_x\partial_z\varphi\rA_{H^{-3/2}},\\[0.5ex]
&\lA \RBony(c,\partial_z\varphi)\rA_{H^{\gamma-3}}
\les \lA c\rA_{\eC{\gamma-2}} \lA \partial_z\varphi\rA_{H^{-1/2}}.
\end{align*}
Now use \eqref{n177} 
and write
$$
\lA a-1\rA_{\eC{\gamma-1}}+\lA b\rA_{\eC{\gamma-1}}+\lA c\rA_{\eC{\gamma-1}}\le 
C(\lA \eta\rA_{\eC{\gamma}})\lA \eta\rA_{\eC{\gamma}}
$$
to complete the proof.
\end{proof}
Let $P_{-}=P_{-}(\eta)$, $P_{+}=P_+(\eta)$ and $R_0=R_0(\eta)$ 
be as given by Lemma~\ref{T10}, so that 
$(\partial_z -P_{-})(\partial_z-P_{+})\varphi=f_0+R_0\varphi$, 
where $R_0$ is a smoothing operator, satisfying
\begin{equation*}
\lA R_0 u\rA_{H^{r+\gamma-3}}\le 
C(\lA \eta\rA_{\eC{\gamma}}) \lA \eta\rA_{\eC{\gamma}}
\lA \px u\rA_{H^{r-1}},
\end{equation*}
for any~$r\in \xR$ and any~$u\in H^r(\xR)$. The key point consists in proving that one can express, on $z=0$,  
the trace of the normal derivative $\partial_z \varphi$ in terms of the tangential derivative. 
To do so, as above, we exploit the fact that $\underline{\varphi}=\partial_z \varphi-P_{+}\varphi$ satisfies 
a parabolic equation. 
\begin{lemm}\label{T24}
For any $\tau<0$, 
the function $\underline{\varphi}\defn (\partial_z -P_{+})\varphi$ is continuous in $z\in [\tau,0]$ with values 
in $H^{\gamma-3}(\xR)$. Moreover, there exists a non decreasing function $C$ such that 
\be\label{n180}
\sup_{z\in [\tau,0]}\lA \underline{\varphi}(z)\rA_{H^{\gamma-3}}\le C\bigl(\lA \eta\rA_{\eC{\gamma}}\bigr)
\lA \eta\rA_{\eC{\gamma}}\blA \Dxmez\psi\brA_{L^2}.
\ee
\end{lemm}
\begin{proof}We prove only an {\em a priori} estimate. 
The regularity result is an immediate consequence of the method used to prove the estimate. 
We shall prove a slightly stronger result. 
Namely we shall prove that, for any $\eps \in \pol 0,1]$, 
\eqref{n180} holds with $\sup_{z\in [\tau,0]}\lA \underline{\varphi}\rA_{H^{\gamma-3}}$ 
replaced with $\sup_{z\in [\tau,0]}\lA \underline{\varphi}\rA_{H^{\gamma-2+\eps}}$.

Since
$$
(\partial_z-P_{-})\underline{\varphi}=f_0+R_0\varphi,
$$
the parabolic estimate \eqref{n107} asserts that, for any $\tau_1<\tau_2<0$ and any 
$\mu\in \xR$,
\begin{equation*}
\begin{aligned}
\lA \underline{\varphi}\rA_{L^\infty([\tau_2,0];H^{\mu+1-\epsilon})}
&\le C(\lA \eta\rA_{\eC{\gamma}}) \Bigl( \lA f_0\rA_{L^\infty([\tau_1,0];H^\mu)}
+\lA \underline{\varphi}\rA_{L^\infty([\tau_1,0];H^\mu)}\Bigr)\\
&\quad +C(\lA \eta\rA_{\eC{\gamma}})\lA \eta\rA_{\eC{\gamma}}
\lA \nabla_{x,z}\varphi\rA_{L^\infty([\tau_1,0];H^{\mu-1-(\gamma-3)})}.
\end{aligned}
\end{equation*}
Consequently, for any $\mu\le \gamma-3$,
\begin{equation*}
\begin{aligned}
\lA \underline{\varphi}\rA_{L^\infty([\tau_2,0];H^{\mu+1-\epsilon})}
&\le C(\lA \eta\rA_{\eC{\gamma}}) \Bigl( \lA f_0\rA_{L^\infty([\tau_1,0];H^{\gamma-3})}
+\lA \underline{\varphi}\rA_{L^\infty([\tau_1,0];H^\mu)}\Bigr)\\
&\quad +C(\lA \eta\rA_{\eC{\gamma}})\lA \eta\rA_{\eC{\gamma}}
\lA \nabla_{x,z}\varphi\rA_{L^\infty([\tau_1,0];H^{-1})},
\end{aligned}
\end{equation*}
so, the estimate \eqref{n179} for $f_0$ and 
the estimate \eqref{n177} imply that
\begin{align*}
\lA \underline{\varphi}\rA_{L^\infty([\tau_2,0];H^{\mu+1-\epsilon})}
&\le C(\lA \eta\rA_{\eC{\gamma}})\lA\eta\rA_{\eC{\gamma}} \blA \Dxmez\psi\brA_{L^2}\\
&\quad +C(\lA \eta\rA_{\eC{\gamma}})\lA \underline{\varphi}\rA_{L^\infty([\tau_1,0];H^\mu)}.
\end{align*}
Hence, by an immediate bootstrap argument, it is sufficient to prove that, for any $\tau<0$, 
$$
\lA \underline{\varphi}\rA_{L^\infty([\tau,0];H^{-1/2})}\le C\bigl(\lA \eta\rA_{\eC{\gamma}}\bigr)
\lA \eta\rA_{\eC{\gamma}}\blA \Dxmez\psi\brA_{L^2}.
$$
This in turn follows from the fact that 
$\underline{\varphi}=(\partial_z-\Dx )\varphi-T_{P-|\xi|}\varphi$, by definition of $P_+$, and 
the estimates 
\eqref{n177}, \eqref{n176} 
and the operator norm estimate for paradifferential operators (see \eqref{esti:quant0-px}):
$$
\lA T_{P-|\xi|}\varphi\rA_{H^{-\mez}}\les M^1_0(P-|\xi|) \lA \px \varphi\rA_{H^{-\mez}}
\le C(\lA \eta\rA_{\eC{\gamma}}) \lA \eta\rA_{\eC{\gamma}} \lA \px \varphi\rA_{H^{-\mez}}.
$$
This completes the proof of Lemma~\ref{T24}.
\end{proof}

Since $\B(\eta)\psi-P_+(\eta)\psi=(\partial_z -P_+(\eta))\varphi\arrowvert_{z=0}=\underline{\varphi}(0)$, 
it immediately follows from \eqref{n180} 
that
$$
\lA \B(\eta)\psi-P_+(\eta)\psi \rA_{H^{\gamma-3}}\le C\bigl(\lA \eta\rA_{\eC{\gamma}}\bigr)
\lA \eta\rA_{\eC{\gamma}}\blA \Dxmez\psi\brA_{L^2}.
$$
To estimate $G(\eta)-\Dx\psi$, starting from \eqref{n173}, we write
\begin{align*}
&(1+(\px\eta)^2) \partial_z \varphi-\px\eta \px\varphi=
\partial_z \varphi+T_{(\px\eta)^2}\partial_z \varphi  - T_{ \px \eta} \px \varphi+R',\\
&R'=T_{\partial_z\varphi}(\px\eta)^2+\RBony(\partial_z\varphi,(\px\eta)^2)
-T_{\px\varphi}\px\eta-\RBony(\px\varphi,\px\eta).
\end{align*}
Again, it follows from  
the paraproduct rules~\eqref{esti:Tba} and~\eqref{Bony3} that, for any $\tau<0$, 
the~$C^0([\tau,0];H^{\gamma-3})$-norm of~$R'$ is estimated by the right-hand side of \eqref{va2}. 

Furthermore, since $(1+(\px\eta)^2)P-i(\px\eta)\xi=\la\xi\ra$, by using the symbolic calculus 
estimate (see~\eqref{esti:quant2sharp}), it follows from 
\eqref{n180} that
\begin{equation*}
\partial_z \varphi+T_{(\px\eta)^2}\partial_z \varphi - T_{\px \eta}\px \varphi = \Dx \varphi +r,
\end{equation*}
where the~$C^0([\tau,0];H^{\gamma-3})$-norm of~$r$ is estimated by the right-hand side of 
\eqref{va2}. 
This concludes the proof of Proposition~\ref{T22}.
\end{proof}

We next study the Taylor expansion of the Dirichlet-Neumann operator. We 
recall that the sum of the linear part and the quadratic part is
$$
G_{\quadratique}(\eta)\psi\defn \la D_x \ra\psi-\Dx(\eta\Dx\psi)-\partial_x(\eta\partial_x\psi).
$$
We shall prove an estimate for $G(\eta)\psi-G_{\quadratique}(\eta)\psi$ similar to 
the linearization estimate \eqref{va2} proved above. 
Namely, we shall prove that $G(\eta)\psi-G_{\quadratique}(\eta)\psi$ 
is a smoothing operator, such that if $\eta\in \eC{\gamma}(\xR)$ with 
$\gamma$ large enough, then one can estimate $G(\eta)\psi-G_{\quadratique}(\eta)\psi$ 
in $H^{\gamma-4}$ by means of a low Sobolev norm of $\Dxmez \psi$ only.

\begin{prop}\label{T25}
Let~$\gamma\in \xR^3$ be such that 
$\gamma>4+\mez$, $\gamma\not\in\mez\xN$. 
Consider $\eta\in \eC{\gamma}(\xR)$ and $\psi\in \h{\mez,1}(\xR)$ 
with the assumption that $\lA \eta\rA_{\eC{\gamma}}$ is small enough. 
Then 
$G(\eta)\psi-G_{\quadratique}(\eta)\psi$ belongs to $H^{\gamma-4}(\xR)$. Moreover, 
there exists a non decreasing function~$C\colon \xR_+\rightarrow \xR_+$ 
depending only on~$\gamma$ such that
\begin{equation}\label{n182}
\lA G(\eta)\psi-G_{\quadratique}(\eta)\psi\rA_{H^{\gamma-4}}
\le C\left( \lA \eta\rA_{\eC{\gamma}}\right)\lA \eta\rA_{\eC{\gamma}}^2
\blA \Dxmez \psi \brA_{H^1}.
\end{equation}
\end{prop}
\begin{proof}
As in the proof of Proposition~\ref{T21}, there holds
$$
G(\eta)\psi-G(0)\psi=-\int_0^1 \mathcal{G}(\lambda)\,d\lambda,\quad 
 \mathcal{G}(\lambda)=G(\lambda\eta)(\eta B(\lambda\eta)\psi)+\px(\eta V(\lambda\eta)\psi).
$$
Let us fix some notations. We denote by 
$$
P_\lambda=\frac{1}{1+(\lambda\px\eta)^2}( i \lambda\px \eta \xi +\la \xi\ra),
$$
the symbol obtained by replacing $\eta$ with $\lambda\eta$ in \eqref{n169}. 
Hereafter, we denote by $C$ various constants depending only on $\lA \eta\rA_{\eC{\gamma}}$ and 
we set $\Omega\defn \lA \eta\rA_{\eC{\gamma}}^2
\blA \Dxmez \psi \brA_{H^1}$. 

Notice that $G(0)=\Dx$ and $\mathcal{G}(0)=\Dx(\eta\Dx\psi)+\partial_x(\eta\partial_x\psi)$. One has to prove that 
there exists a constant $C$ depending only on $\lA \eta\rA_{\eC{\gamma}}$ such that 
$$
\lA \mathcal{G}(\lambda)-\mathcal{G}(0)\rA_{H^{\gamma-4}}\le 
C\Omega.
$$
To prove this estimate 
we shall prove that
\be\label{n184}
\lA G(\lambda\eta)(\eta B(\lambda\eta)\psi)-\Dx(\eta\Dx\psi)
-\Dx (T_{\eta (P_\lambda-|\xi|)}\psi)\rA_{H^{\gamma-4}}\le C\Omega,
\ee
and
\begin{align}
&\lA \px(\eta V(\lambda\eta)\psi))-\px(\eta\px\psi)
+\lambda \px (T_{\eta(\px\eta)P_\lambda}\psi)\rA_{H^{\gamma-4}}\le 
C\Omega,\label{n185}\\[0.5ex]
&\Dx (T_{\eta (P_\lambda-|\xi|)}\psi)=\lambda \px (T_{\eta(\px\eta)P_\lambda}\psi).\label{n186}
\end{align}

We begin by proving \eqref{n184}. 
To do so, we use \eqref{va2} to replace $G(\lambda\eta)$ by $\Dx$ and 
$\B(\lambda\eta)$ by $P_+(\lambda\eta)$. Write
\begin{align*}
\lA G(\lambda\eta)(\eta B(\lambda\eta)\psi)-\Dx(\eta B(\lambda\eta)\psi)\rA_{H^{\gamma-4}}
&\le C \lA \eta\rA_{\eC{\gamma}}\lA \eta B(\lambda\eta)\psi\rA_{H^{1/2}}\\
&\le C \lA \eta\rA_{\eC{\gamma}}^2 \lA B(\lambda\eta)\psi\rA_{H^{1/2}}\le C\Omega,
\end{align*}
and
\begin{align*}
\lA \Dx(\eta B(\lambda\eta)\psi)-\Dx (\eta P_{+}(\lambda\eta)\psi)\rA_{H^{\gamma-4}}
&\le \blA \eta \bigl( B(\lambda\eta)\psi-P_{+}(\lambda\eta)\psi\bigr)\brA_{H^{\gamma-3}}\\
&\le \lA \eta\rA_{\eC{\gamma}} \lA B(\lambda\eta)\psi-P_{+}(\lambda\eta)\psi\rA_{H^{\gamma-3}}\\
&\le C\Omega,
\end{align*}
where we used the product rule \eqref{pr:sz}. 

Now, by definition of $P_+(\eta)$ we have
$$
\Dx (\eta P_{+}(\lambda\eta)\psi)-\Dx (\eta\Dx\psi)=\Dx (\eta T_{P_\lambda-|\xi|}\psi),
$$
so, to prove \eqref{n184} it remains only to prove that
\be
\lA \Dx (\eta T_{P_\lambda-|\xi|}\psi) -\Dx (T_{\eta (P_\lambda-|\xi|)}\psi)\rA_{H^{\gamma-4}}
\le C\Omega.
\label{n187}
\ee
Set $\wp_\lambda\defn T_{P_\lambda-|\xi|}\psi$. 
We first simplify $\Dx (\eta T_{P_\lambda-|\xi|}\psi)$ 
by paralinearizing the product $\eta \wp_\lambda$. That is, we write $\eta \wp_\lambda 
=T_\eta \wp_\lambda+(T_{\wp_\lambda}\eta +\RBony(\eta,\wp_\lambda))$ and use 
\eqref{esti:Tba} and \eqref{Bony3}Ê
to obtain that
$$
\lA T_{\wp_\lambda}\eta\rA_{H^{\gamma-3}}+\lA \RBony(\eta,\wp_\lambda)\rA_{H^{\gamma-3}}\les
\lA \wp_\lambda\rA_{H^{-\mez}}
\lA \eta\rA_{\eC{\gamma}}.
$$
Now it follows from \eqref{esti:quant0-px} that
$$
\lA \wp_\lambda\rA_{H^{-\mez}}\le C \lA \eta\rA_{\eC{\gamma}}\lA \px \psi\rA_{H^{-1/2}}
\le C \lA \eta\rA_{\eC{\gamma}}\blA \Dxmez\psi\brA_{L^2}.
$$
Therefore
$$
\Dx (\eta T_{P_\lambda-|\xi|}\psi) 
=\Dx (T_\eta T_{P_\lambda-|\xi|}\psi)+R_1
$$
with $\lA R_1\rA_{H^{\gamma-4}}\le C \Omega$. Next, 
since $\partial_\xi^k \eta =0$ for $k\ge 1$, it follows 
from symbolic calculus (see~\eqref{esti:quant2sharp-px} applied with 
$(m,m',\rho)=(0,1,\gamma-1)$) that
\begin{align*}
\Dx (\eta T_{P_\lambda-|\xi|}\psi) 
=\Dx T_{\eta (P_\lambda-|\xi|)}\psi+R_2
\end{align*}
where $R_2=R_1 +\Dx  (T_\eta T_{P_\lambda-|\xi|}-T_{\eta (P_\lambda-|\xi|)})\psi$ satisfies 
$\lA R_2\rA_{H^{\gamma-4}}\le C\Omega$. 
This proves \eqref{n187} and hence completes the proof of \eqref{n184}.

The proof of \eqref{n185} is similar. 
By definition $V(\lambda\eta)\psi=\px \psi-\lambda(\px\eta)\B(\lambda\eta)\psi$ so 
\eqref{va2} and the product rule~\eqref{pr:sz} imply that
$$
\lA \px (\eta V(\lambda\eta)\psi)-\px(\eta\px\psi)
+\lambda \px (\eta (\px\eta)P_+(\lambda\eta)\psi)\rA_{H^{\gamma-4}}\le C\Omega.
$$
Thus to obtain \eqref{n185} it is sufficient to prove that
$$
\lA \px (\eta (\px\eta)P_+(\lambda\eta)\psi)
-\px (T_{\eta (\px\eta)P_\lambda}\psi)\rA_{H^{\gamma-4}}\le C\Omega.
$$
As above, this follows from \eqref{esti:Tba}, \eqref{Bony3}Ê
and \eqref{esti:quant2sharp-px}.

To prove \eqref{n186}, notice that
$$
\eta (P_\lambda-|\xi|)=i\alpha(x)\xi-\beta(x)|\xi|,\quad 
\eta (\px \eta)P_\lambda = i\beta(x)\xi+\alpha|\xi|
$$
with
$$
\alpha=\frac{\eta (\lambda\px\eta)}{1+(\lambda\px\eta)^2},\quad \beta=\frac{\eta(\lambda\px\eta)^2}{1+(\lambda\px\eta)^2}.
$$
Therefore
$$
\Dx T_{\eta (P_\lambda-|\xi|)}=\Dx T_\alpha \px-\Dx T_\beta \Dx,\quad 
\px T_{\eta (\px \eta)P_\lambda}=\px T_\beta\px+\px T_\alpha \Dx,
$$
and the desired identity \eqref{n186}Ê
follows from Lemma~\ref{lemm:DxaDx} in 
Appendix~\ref{s2}.
\end{proof}

\begin{coro}\label{T26}
Let~$\gamma\in \xR^3$ be such that 
$\gamma>4+\mez$, $\gamma\not\in\mez\xN$. 
Consider $\eta\in \eC{\gamma}(\xR)$ and $\psi\in \h{\mez,1}(\xR)$ 
with the assumption that $\lA \eta\rA_{\eC{\gamma}}$ is small enough. 
Then 
$F(\eta)\psi-F_{\quadratique}(\eta)\psi$ belongs to $H^{\gamma-4}(\xR)$. Moreover, 
there exists a non decreasing function~$C\colon \xR_+\rightarrow \xR_+$ 
depending only on~$\gamma$ such that
\begin{equation*}
\lA F(\eta)\psi-F_{\quadratique}(\eta)\psi\rA_{H^{\gamma-4}}
\le C\left( \lA \eta\rA_{\eC{\gamma}}\right)\lA \eta\rA_{\eC{\gamma}}^2
\blA \Dxmez \psi \brA_{H^1}.
\end{equation*}
\end{coro}
\begin{proof}
By definition 
$F_{\quadratique}(\eta)\psi
=G_{\quadratique}(\eta)\psi-\Dx \psi
+\Dx T_{\Dx \psi}\eta +\px T_{\px\psi} \eta$, so
\begin{align*}
F(\eta)\psi-F_{\quadratique}(\eta)\psi&=G(\eta)\psi-G_{\quadratique}(\eta)\psi\\
&\quad+\Dx T_{\B(\eta)\psi-\Dx\psi}\eta+\px (T_{V(\eta)\psi-\px\psi}\eta).
\end{align*}
The difference $G(\eta)\psi-G_{\quadratique}(\eta)\psi$ is estimated by~\eqref{n182}. 
To estimate the last two terms in the right-hand side above, we use \eqref{esti:Tba} to deduce that
\begin{align*}
&\lA \Dx T_{\B(\eta)\psi-\Dx\psi}\eta\rA_{H^{\gamma-4}}\les 
\lA \B(\eta)\psi-\Dx\psi\rA_{H^{-1/2}}\lA \eta\rA_{\eC{\gamma}},\\
&\lA \px T_{V(\eta)\psi-\px\psi}\eta\rA_{H^{\gamma-4}}\les 
\lA V(\eta)\psi-\px\psi\rA_{H^{-1/2}}\lA \eta\rA_{\eC{\gamma}},
\end{align*}
Now write
$$
\lA \B(\eta)\psi-\Dx\psi\rA_{H^{-1/2}}\le \lA \B(\eta)\psi-P_{+}(\eta)\psi\rA_{H^{-1/2}}
+\lA P_{+}(\eta)\psi-\Dx\psi\rA_{H^{-1/2}}.
$$
The first term in the right-hand side above is estimated by means of \eqref{va2}. 
To bound the second term, 
observe that, since $P_+(\eta)-\Dx=T_{P-|\xi|}$, \eqref{esti:quant0-px} implies that
$$
\lA P_{+}(\eta)\psi-\Dx\psi\rA_{H^{-1/2}}\les C M^1_0(P-|\xi|)\lA \px \psi\rA_{H^{-1/2}}
\le C \lA\eta\rA_{\eC{\gamma}}\blA \Dxmez\psi\brA_{L^2}.
$$
On the other hand $V(\eta)\psi-\px\psi=(\px \eta)\B(\eta)\psi$ so the product rule~\eqref{pr:sz} 
implies that
$$
\lA V(\eta)\psi-\px\psi\rA_{H^{-1/2}}\les \lA \px\eta\rA_{\eC{\gamma-1}}\lA 
\B(\eta)\psi\rA_{H^{-1/2}}\le C  \lA\eta\rA_{\eC{\gamma}}\blA \Dxmez\psi\brA_{L^2}
$$
where we used the product rule \eqref{pr:sz} and the estimate~\eqref{va1} applied with $\mu=1/2$. 
This completes the proof.
\end{proof}

\chapter{Normal form for the water waves equation}\label{S:22}

The main goal of this paper is to prove that, 
given an {\em a priori} bound of some H\"older norms of 
$Z^{k'} (\eta+i\Dxmez\psi)$ for $k'\le s/2+k_0$,
we have an {\em a priori} estimate of some Sobolev norms of $Z^{k} (\eta+i\Dxmez \omega)$ for $k\le s$, 
where recall that 
$\omega=\psi-T_{\B(\eta)\psi}\eta$. 
The proof is by induction on $k\ge 0$. Each step is divided into two parts.
\begin{enumerate}
\item Quadratic approximations: in this step we paralinearize and symmetrize the equations. In addition, we 
identify the principal and subprincipal terms 
in the analysis of both the regularity and the homogeneity. 
\item Normal form: in this step we use a bilinear normal form transformation to compensate for the quadratic terms in the energy estimates.  
\end{enumerate}
Since the case $k=0$ is interesting in its own, we shall consider the case $k=0$ and the case $k>0$ 
separately. In this chapter, we consider the case $k=0$. The case $k>0$ will be considered in the next chapters. 
The overlap between this two cases will be small. Moreover, we will prove 
a slightly better result in the case $k=0$ then in the case $k>0$ (compare Proposition~\ref{T42} 
with Proposition~\ref{T65}).

\section{Quadratic approximations without losses}\label{S:paraeq}

We now consider the Craig-Sulem-Zakharov system
\begin{equation}\label{P:WW}
\left\{
\begin{aligned}
&\partial_t \eta=G(\eta)\psi,\\
&\partial_t \psi +\eta+ \frac{1}{2} 
(\partial_x \psi)^2  -\frac{1}{2(1+(\partial_x\eta)^2)}
\bigl(G(\eta)\psi+\partial_x\eta \partial_x \psi\bigr)^2= 0.
\end{aligned}
\right.
\end{equation}
In this section we use the abbreviated notations
\be\label{n187.1}
\B=\frac{G(\eta)\psi+\partial_x\eta \partial_ x\psi}{1+(\partial_x\eta)^2},\quad
V=\partial_x\psi-\B\partial_x\eta,\quad
\omega=\psi-T_B\eta.
\ee

\begin{assu}\label{T27}
Let~$T>0$ and fix~$(s,\varrho)$ such that
$$
s>\varrho+1>14,\quad \varrho\not\in\mez\xN.
$$
It is always assumed in the rest of this chapter that :

$i)$ $(\eta,\psi)\in C^0\big([0,T];H^{s}(\xR)\times \h{\mez,s-\mez}(\xR))$ is such that  
$\omega \in 
C^0\big([0,T];\h{\mez,s}(\xR))$.

$ii)$ The condition \eqref{1117} is satisfied uniformly in time. 
Namely we assume that
\be\label{n188}
\sup_{t\in [0,T]}\left\{ \lA \px \eta(t)\rA_{\eC{\varrho-1}}
+\lA \px\eta(t)\rA_{\eC{-1}}^{1/2}\lA \eta'(t)\rA_{H^{-1}}^{1/2}
\right\}
\ee
is small enough, so that we are in position to apply 
Proposition~\ref{ref:116} as well as the 
results proved in the previous chapter.
\end{assu}
\begin{rema*}
Let us comment on the smallness condition. For our purposes 
$\lA \px \eta(t)\rA_{\eC{\varrho-1}}=O(\eps t^{-1/2})$ and 
$\lA \eta'(t)\rA_{H^{-1}}\le \lA \eta\rA_{H^s}=O(\eps t^{\delta})$ for some $\delta<1/2$ so that 
\eqref{n188} will be satisfied. One can also notice that, 
for smooth solutions, we have (see~\cite{CrSu})
$$
\frac{d}{dt}\left(\int \eta^2\, dx+\int \psi G(\eta)\psi\, dx\right)=0.
$$
Now it follows from Corollary~\ref{ref:118} that
$$
0\le 
\int \psi G(\eta)\psi\, dx = \int (\Dxmez \psi) G_{1/2}(\eta)\psi\, dx\le \Ceta\Dxmezpsi ^2,
$$
so that
$$
\lA \eta\rA_{L^\infty([T_0,T];L^2)}^2\le \lA \eta_0\rA_{L^2}^2
+C(\lA \eta'_0\rA_{L^\infty})\blA \Dxmez \psi_0\brA_{L^2}^2.
$$
Thus, for \eqref{n188}Ê
to be small 
it is sufficient to require that 
$\sup_{t\in [0,T]} \lA \eta(t)\rA_{\eC{\varrho}}$, 
$\lA \eta_0\rA_{L^2}$, and $\blA \Dxmez \psi_0\brA_{L^2}$ 
are small enough.
\end{rema*}
For~$t\in [0,T]$, we set
\begin{align*}
\sobolev (t)&\defn \lA \eta(t)\rA_{H^s}+\blA \Dxmez \omega(t)\brA_{H^s},\\
\holder(t)&\defn \lA \eta(t)\rA_{\eC{\varrho}}+\blA \Dxmez \psi(t)\brA_{\eC{\varrho}}.
\end{align*}

From~\eqref{211-1}, \eqref{2113} 
and \eqref{n3} we know that
\begin{equation}\label{n189}
\begin{aligned}
&\lA \B\rA_{H^{s-1}}+\lA V\rA_{H^{s-1}}
\le C\left(\holder\right)\sobolev ,\\
&\lA \B\rA_{\eC{\varrho-1}}+\lA V\rA_{\eC{\varrho-1}}\le C\left(\holder\right)\holder.
\end{aligned}
\end{equation}

We start with some basic remarks about the Taylor coefficient 
$\ma$ which is defined as follows.

\begin{nota}
Define
\be\label{n190}
\ma = 1+ \partial_t \B +V  \partial_x \B.
\ee
\end{nota}
If $(\eta,\psi)\in C^0\big([0,T];H^{s}(\xR)\times \h{\mez,s-\mez}(\xR))$ 
solves \eqref{P:WW} then
\begin{align*}
&(\eta,\psi)\in C^1\big([0,T];H^{s-1}(\xR)\times \h{\mez,s-\tdm}(\xR)),\\
&(\B,V)\in C^0\big([0,T];H^{s-1}(\xR)\times H^{s-1}(\xR)).
\end{align*}
In addition, 
it follows from the shape derivative formula for the Dirichlet-Neumann 
(see \cite{LannesLivre}) that $G(\eta)\psi\in C^1\big([0,T];H^{s-1}(\xR))$ together with
\be\label{n187.2}
\partial_t G(\eta)\psi
=G(\eta)\bigl(\partial_t \psi-(B(\eta)\psi)\partial_t \eta\bigr)-\partial_x\bigl((V(\eta)\psi)\partial_t\eta\bigr).
\ee
Then it follows from the definition~\eqref{n187.1} that 
$\partial_t \B \in C^0\big([0,T];H^{s-2}(\xR))$. Consequently, 
$\ma$ is well-defined and belongs to $C^0\big([0,T];H^{s-2}(\xR))$. 
It is known (see~\cite{Bertinoro,LannesJAMS}) that 
$\ma=-\partial_y P\arrowvert_{y=\eta}$ where~$P$ is the pressure. 
Here, we shall use the following identity for $\ma$ which is proved in the 
appendix (see~\eqref{formule:a}):
\begin{equation}\label{n191}
\ma=\frac{1}{1+(\px\eta)^2}
\left(1+V\px \B - \B \px V-\mez G(\eta)V^2 -\mez G(\eta)\B^2-G(\eta)\eta\right).
\end{equation}

\begin{lemm}
$i)$ For any $\gamma>3$, there exists a nondecreasing function~$C$ such that, 
\be\label{n192}
\lA a-1\rA_{\eC{1}}\le C\big(\lA \eta\rA_{\eC{\gamma}}\big) \Big[ \lA \eta\rA_{\eC{\gamma}}+\blA \Dxmez \psi\brA_{\eC{\gamma-\mez}}\Big].
\ee
Using the notation $\holder$, this means that $\lA \ma -1\rA_{\eC{1}}\le C(\holder)\holder$. 

$ii)$ There exists a nondecreasing function $C$ such that
\begin{align}
&\lA \partial_t \ma -\px^2\psi \rA_{L^{\infty}}\le C(\holder)\holder^2,\label{n193}\\
&\lA \ma -1+\Dx\eta\rA_{\eC{1}}\le C(\holder)\holder^2.\label{n193.1}
\end{align}
\end{lemm}
\begin{proof}
Let us prove \e{n192}. By \e{n191}, we know that
\begin{align*}
\lA a-1\rA_{\eC{1}}\le C\big(\lA \eta\rA_{\eC{1}}\big) 
\Big[& \lA \px\eta\rA_{\eC{1}}^2+
\lA V\rA_{\eC{1}}\lA \px \B\rA_{\eC{1}}+\lA \B\rA_{\eC{1}}\lA \px V\rA_{\eC{1}}\\
&\quad+\lA G(\eta)V^2\rA_{\eC{1}}+\lA G(\eta)\B^2\rA_{\eC{1}}+\lA G(\eta)\eta\rA_{\eC{1}}\Big].
\end{align*}
By \e{1139} applied with $\gamma$ replaced by $\gamma-1$, we may write
\begin{align*}
\lA G(\eta)\eta\rA_{\eC{1}}&\le C\big( \lA \eta\rA_{\eC{\gamma-1}}\big)\lA \eta\rA_{\eC{\gamma-1}},\\
\lA G(\eta)\B^2\rA_{\eC{1}}&\le C\big( \lA \eta\rA_{\eC{\gamma-1}}\big)\lA \B\rA_{\eC{\gamma-1}}^2,\\
\lA G(\eta)V^2\rA_{\eC{1}}&\le C\big( \lA \eta\rA_{\eC{\gamma-1}}\big)\lA V\rA_{\eC{\gamma-1}}^2,
\end{align*}
where we used that $\eC{\gamma-1}$ is an algebra to obtain $\lA \B^2\rA_{\eC{\gamma-1}}\les 
\lA \B\rA_{\eC{\gamma-1}}^2$, $\lA V^2\rA_{\eC{\gamma-1}}\les 
\lA V\rA_{\eC{\gamma-1}}^2$. 

Since $\eC{\gamma-2}$ is an algebra, we get from the definitions~\e{n187.1} of $V,\B$, 
$$
\lA V\rA_{\eC{1}}+\lA \px V\rA_{\eC{1}}\le \blA \Dxmez \psi\brA_{\eC{\gamma-\mez}}+\lA \eta\rA_{\eC{\gamma}}
\lA \B\rA_{\eC{\gamma-1}},
$$
and
$$
\lA \B\rA_{\eC{\gamma-1}}\le C\big(\lA \eta\rA_{\eC{\gamma}}\big)\Big[ 
\lA G(\eta)\psi\rA_{\eC{\gamma-1}}+\blA \Dxmez \psi\brA_{\eC{\gamma-\mez}}\Big].
$$
Combining the inequalities and \e{1139}, we get finally \e{n192}. 


The proof of the second estimate is similar. 
By using the identity \eqref{n191} and \eqref{n187.2} applied 
with $\psi$ replaced with $V^2$, $\B^2$ or $\eta$,  
together with the following  
expressions (see \eqref{n190} and Lemma~\ref{L:A.4.1} in Appendix~\ref{S:A.4})
$$
\partial_t\B=-V\px\B+\ma-1,\quad 
\partial_t V=-V\px V-\ma \px\eta,\quad 
\partial_t \eta=G(\eta)\psi,
$$
we obtain that $\partial_t (a+G(\eta)\eta)$ is bounded by 
$C(\holder)\holder^2$. Using again~\eqref{n187.2} 
to compute $\partial_t G(\eta)\eta$ we find that 
$\partial_t G(\eta)\eta-G(\eta)\partial_t \eta$  is bounded by 
$C(\holder)\holder^2$. Since 
$G(\eta)\partial_t \eta=G(\eta)G(\eta)\psi$, we deduce from \e{n145} that 
modulo quadratic terms which are estimated as above, 
$G(\eta)\partial_t\eta$ is given by $\Dx^2\psi$. 

Eventually it follows from the identity \eqref{n191} and the estimates \eqref{211-1} that 
$$
\lA a-1+G(\eta)\eta\rA_{\eC{1}}\le C(\holder)\holder^2.
$$
So \eqref{n193.1} 
follows from \eqref{n145}.
\end{proof}

Notice that \eqref{n192} implies 
that~$\ma$ is a positive function under a smallness assumption: 

\begin{coro}
If $\holder$ is small enough then  
\begin{equation}\label{n194}
\ma(t,x)\ge 1/2,\quad \forall (t,x)\in [0,T]\times \xR.
\end{equation}
\end{coro}
\begin{assu}\label{T28}
Hereafter, it is 
assumed that $\holder$ is small enough, so that~\eqref{n194} holds.
\end{assu}
\begin{rema}
Wu proved that~$\ma$ is a positive function (see~\cite{WuJAMS,WuInvent} 
and also~\cite{LannesJAMS}) without smallness assumption. 
\end{rema}
\begin{nota}\label{T29}
Given two functions~$f,g$ defined on the time interval~$[0,T]$, we write
\begin{equation}\label{n195}
\md{f}{g}{\sigma},
\end{equation}
to say that there exists an increasing function~$C$, independent of~$(\eta,\psi,T)$ 
such that for all~$t\in [0,T]$,
$$
\lA f(t)-g(t)\rA_{H^\sigma}\le C\bigl(\holder(t)\bigr)\holder(t)^{2}\sobolev (t).
$$
We say then that $f$ is equal to $g$ 
modulo admissible cubic terms.  
\end{nota}

We write now the water waves system 
as a paradifferential system of quasi-linear dispersive equations. 
This will allow us to get energy estimates for the 
good unknowns $\eta$ and $\omega$.

\begin{prop}\label{T30}
Use Notation~$\ref{T29}$ and Assumptions~$\ref{T27}$ and 
$\ref{T28}$. Introduce 
$$
\alpha=\sqrt{\ma}-1,\quad 
\vU^1=\eta+T_\alpha \eta,\quad 
\vU^2=\Dxmez \omega.
$$
Then 
\begin{equation}\label{n196}
\left\{
\begin{aligned}
&\partial_{t}\vU^1+T_{V}\partial_x \vU^1 - (\id+T_{\alpha})\Dxmez \vU^2 = F^1,\\
&\partial_t U^2+\Dxmez T_{V\la\xi\ra^{-1/2}}\px U^2 + \Dxmez\left( (\id+T_\alpha)\vU^1\right) =F^2,
\end{aligned}
\right.
\end{equation}
for some source terms~$F^1,F^2$ satisfying
\begin{align}
&\md{F^1}{F_{\quadratique}(\eta)\psi-\mez T_{\px^2\psi}\eta}{s}, 
\label{n197}
\\
&\md{F^2}{\mez \Dxmez \RBony(\Dx\psi,\Dx\omega)
-\mez\Dxmez\RBony(\px\psi,\px\omega)}{s},\label{n198}
\end{align}
where $F_{\quadratique}(\eta)\psi$ is given by~\eqref{n135}.
\end{prop}
\begin{proof} 

The proof is in two steps. 

\step{1}{Paralinearization of the equations}

We begin the proof of Proposition~\ref{T30} by proving that
\begin{equation}\label{n199}
\left\{
\begin{aligned}
&\partial_t \eta+T_V \px\eta- \Dx\omega =f^1,\\
&\partial_t \omega+T_V  \px \omega +(\id+T_{\ma-1}) \eta =f^2,
\end{aligned}
\right.
\end{equation}
with 
\begin{align}
&\md{f^1}{F_{\quadratique}(\eta)\psi-T_{\px^2\psi}\eta}{s},
\label{n200}
\\
&\md{f^2}{\mez\RBony(\Dx\psi,\Dx\omega)
-\mez\RBony(\px\psi,\px\omega)}{s+\mezl}.\label{n201}
\end{align}

The first half of this result is already proved. 
Indeed, by definition \e{n1} of~$F(\eta)\psi$, 
the first equation of \eqref{n199} holds with~$f^1\defn F(\eta)\psi-T_{\px V}\eta$. 
Consequently, the previous 
estimates for~$F(\eta)\psi-F_{\quadratique}(\eta)\psi$ 
(see~\eqref{n138}) and~$V(\eta)\psi-\px\psi$ 
(see~\eqref{n155}) 
imply \eqref{n200}. 

To prove \eqref{n201}, we use the elementary identity
\begin{equation*}
\mez (\partial_x \psi)^2 -\mez \frac{\left( \partial_x \eta \partial_x \psi + G(\eta)\psi \right)^2}{1+(\partial_x\eta)^2} 
= \mez V^2 +\B V\partial_x \eta- \mez \B^2,
\end{equation*}
which is proved in the appendix (see~\eqref{iden:eqpsi}). 
The paralinearization formula~$ab=T_a b+T_b a +\RBony(a,b)$ then implies that
\begin{align*}
&\mez (\partial_x \psi)^2 
-\mez \frac{\left( \partial_x \eta   \partial_x \psi + G(\eta)\psi \right)^2}{1+(\partial_x\eta)^2} \\
&\qquad\qquad=T_V V -T_{\B}\B +T_{V \partial_x \eta}\B+T_{\B}V \partial_x \eta\\
&\qquad\qquad\quad+ \mez \RBony(V,V)+\RBony(\B,V \partial_x\eta)-\mez \RBony(\B,\B).
\end{align*}
By using the identity~$\B-V \partial_x \eta=\partial_t\eta$ (see~\eqref{B2-0}), one obtains
$$
-T_{\B}\B +T_{\B}V \partial_x \eta
=-T_{\B}\partial_t\eta.
$$
On the other hand, starting from the definition of~$V=\px\psi-\B\px\eta$ 
we have
\begin{align*}
T_V V &=T_{V}( \partial_x\psi -\B\partial_x\eta)\\
&=T_V\left(\partial_x\psi-T_\B\partial_x\eta-T_{\partial_x\eta}\B-\RBony(\B,\partial_x\eta)\right)\\
&=T_V\partial_x (\psi-T_\B\eta)
+T_V   T_{\partial_x\B} \eta-T_{V}  T_{\partial_x\eta}\B-T_V \RBony(\B,\partial_x\eta)\\ 
&=T_V\partial_x \omega+
T_V  T_{\partial_x\B}  \eta-T_{V}  T_{\partial_x\eta}\B-T_V \RBony(\B,\partial_x\eta).
\end{align*}
Consequently,
$$
T_V V+T_{V \partial_x \eta}\B=T_V\partial_x \omega+T_V   T_{\partial_x\B}  \eta
+(T_{V \partial_x\eta}-T_{V}  T_{\partial_x\eta})\B-T_V \RBony(\B,\partial_x\eta).
$$
By writing~$\partial_t\psi-T_\B\partial_t \eta=\partial_t \omega +T_{\partial_t \B}\eta$ 
and using \eqref{P:WW}, the expression of $a-1$ in terms of $B,V$ given in \e{n190} and the preceding expressions 
we thus end up with
$$
\partial_t \omega +T_V \partial_x \omega +(\id+T_{\ma-1})\eta=f^2,
$$
where
\begin{align*}
f^2&=(T_{V}  T_{\partial_x\eta}-T_{V \partial_x\eta})\B
+(T_{V \partial_x\B}-T_{V}  T_{\partial_x\B})\eta\\
&\quad +\mez \RBony(\B,\B)-\mez\RBony(V,V)+T_V\RBony(\B,\partial_x\eta)
-\RBony(\B,V\px\eta).
\end{align*}
The end of the proof is simple: 
$(i)$ we use the paralinearization theorem 
to estimate all the remainders 
$\RBony(a,b)$; $(ii)$ we use the symbolic calculus theorem 
to estimate the two terms of the form 
$T_aT_b-T_{ab}$. More precisely, it follows from the 
symbolic calculus (see~\eqref{esti:quant2-func}) 
that
\begin{align*}
\lA (T_{V\partial_x\eta}-T_{V}T_{\partial_x\eta})\B\rA_{H^{s+\mez}}
&\les \lA V\rA_{\eC{\tdm}}\lA \partial_x \eta\rA_{\eC{\tdm}} 
\lA \B\rA_{H^{s-1}},\\
\lA (T_{V\partial_x\B}-T_{V}  T_{\partial_x\B})\eta\rA_{H^{s+\mez}}
&\les \lA V\rA_{\eC{1/2}}\lA \partial_x \B\rA_{\eC{1/2}}\lA \eta\rA_{H^s}.
\end{align*}
On the other hand \eqref{esti:quant0} and \eqref{Bony3} imply that
\begin{align*}
\lA T_V\RBony(\B,\partial_x\eta)\rA_{H^{s+\mez}}&\les
\lA V\rA_{L^\infty}\lA \RBony(\B,\partial_x\eta)\rA_{H^{s+\mez}}\\
&\les \lA V\rA_{L^\infty} \lA \partial_x \eta\rA_{\eC{\tdm}}\lA \B\rA_{H^{s-1}}.
\end{align*}
By using \eqref{Bony3} with~$\beta=s-1$ and~$\alpha=3/2$, 
we obtain that
\begin{align*}
\lA \RBony(\B,V\partial_x\eta)\rA_{H^{s+\mez}}\les \lA V\px\eta\rA_{\eC{\tdm}}
\lA \B\rA_{H^{s-1}}.
\end{align*}
Using the bounds~\eqref{n189}Ê
for~$\B$ and~$V$, we thus 
find that
\begin{align*}
&\lA (T_{V\partial_x\eta}-T_{V}T_{\partial_x\eta})\B\rA_{H^{s+\mez}}\les \holder^2\sobolev ,\\
&\lA (T_{V\partial_x\B}-T_{V}  T_{\partial_x\B})\eta\rA_{H^{s+\mez}}\les \holder^2\sobolev ,\\
&\lA T_V\RBony(\B,\partial_x\eta)\rA_{H^{s+\mez}}\les \holder^2\sobolev .
\end{align*}

It immediately follows from the previous analysis that
$$
\lmd{f^2}{\mez\RBony(\B,\B)-\mez\RBony(V,V)}{{s+\mez}}.
$$ 
Now write
$$
\RBony(\B,\B)
=\RBony(\B-\Dx\psi,\B)+\RBony(\Dx\psi,\B-\Dx\omega)+\RBony(\Dx\psi,\Dx\omega),
$$
and
$$
\RBony(V,V)=\RBony(V-\px\psi,V)
+\RBony(\px\psi,V-\px\omega)+\RBony(\px\psi,\px\omega).
$$
Using Proposition~\ref{T18}, \eqref{n154}, \eqref{n155} and 
\eqref{Bony3} we obtain
\begin{equation}\label{n202}
\lmd{f^2}{\mez\RBony(\Dx\psi,\Dx\omega)
-\mez\RBony(\px\psi,\px\omega)}{{s+\mez}},
\end{equation}
as asserted.

\step{2}{Symmetrization}

Since~$\partial T_a b=T_{\partial a}b+T_a\partial b$ 
with~$\partial = \partial_t$ or 
$\partial=\partial_x$, we find that
\begin{align*}
(\partial_t+T_V\px)U^1&=(\partial_t+T_V\px)(\eta+T_\alpha \eta)\\
&=(\id+T_\alpha)(\partial_t+T_V\px)\eta
+\Bigl\{ T_{\partial_t \alpha}+T_VT_{\px \alpha} 
+\bigl[ T_V,T_\alpha\bigr]\px \Bigr\}\eta
\end{align*}
and hence \eqref{n196} holds with
$$
F^1\defn (\id +T_\alpha)f^1+\Bigl\{ T_{\partial_t \alpha}+T_VT_{\px \alpha} 
+\bigl[ T_V,T_\alpha\bigr]\px \Bigr\}\eta,
$$
where~$f^1$ is given by \eqref{n199}. 

Clearly, from the assumption~$\ma \ge 1/2$ and the estimate of the~$\eC{1}$-norm of~$\ma$ 
(see~\eqref{n192}) we obtain that the~$\eC{1}$-norm 
of~$\alpha=\sqrt{\ma}-1$ is bounded by 
\begin{equation}\label{n203}
\lA \alpha\rA_{\eC{1}}\le C(\holder)\holder.
\end{equation} 

Recall that we have proved that 
$\lA f^1\rA_{H^s}\le C(\holder)\holder \sobolev$ so, using 
\eqref{esti:quant0}, the previous estimate for~$\alpha$ implies that  
\begin{equation}\label{n204}
\lmd{T_\alpha f^1}{0}{s}.
\end{equation}
Using the symbolic calculus estimates 
\eqref{esti:quant0} and \eqref{esti:quant2} applied with~$\rho=1$, 
we next deduce that
$$
\md{T_VT_{\px \alpha}\eta}{0}{s},\quad 
\md{\bigl[ T_V,T_\alpha\bigr]\px \eta}{0}{s}.
$$
Together with~\eqref{n204} this implies that
$$
\md{F^1}{f^1+T_{\partial_t \alpha}\eta}{s}.
$$
Now \eqref{n193} implies that~$\lmd{T_{\partial_t \alpha}\eta}{\mez T_{\px^2\psi}\eta}{s}$, so \e{n200}Ê
yields
the claim
$$
\md{F^1}{F_{\quadratique}(\eta)\psi-\mez T_{\px^2\psi}\eta}{s}.
$$

It remains to prove the second identity \eqref{n198}. 
Since
\begin{align*}
\partial_t U^2+\Dxmez T_{V\la\xi\ra^{-1/2}}\px U^2
&=\Dxmez (\partial_t\omega+T_{V}\px\omega)\\
&=\Dxmez \left(f^2-(\id+T_{\ma-1})\eta\right),
\end{align*}
and since, by definition of~$\alpha=\sqrt{\ma}-1$,
\begin{align*}
(\id+T_\alpha)(\id+T_\alpha)&=\id +T_\alpha T_\alpha +2T_\alpha\\
&=\id+T_{\alpha^2+2\alpha}+(T_\alpha T_\alpha -T_{\alpha^2})\\
&=\id+T_{\ma-1}+(T_\alpha T_\alpha -T_{\alpha^2}),
\end{align*}
we find that the second equation in \eqref{n196} holds with
$$
F^2=\Dxmez f^2+\Dxmez (T_\alpha T_\alpha -T_{\alpha^2})\eta.
$$
It follows from 
\eqref{esti:quant2} (applied with~$\rho=1/2$) that
$$
\blA \Dxmez (T_\alpha T_\alpha -T_{\alpha^2})\eta\brA_{H^{s}}
\les \lA \alpha\rA_{C^{1/2}}^2\lA \eta\rA_{H^s}.
$$
This implies, since, as already mentioned, the $\eC{1}$ norm of $\alpha$ 
is bounded by $C(\holder)\holder$, that 
$$
\lmd{\Dxmez (T_\alpha T_\alpha -T_{\alpha^2})\eta}{0}{s},
$$
and hence~$\lmd{F^2}{\Dxmez f^2}{s}$. 
The identity \eqref{n198} then follows from \eqref{n202}. 
\end{proof}

\section{Quadratic and cubic terms in the equations}\label{S:321}
Previously in \S\ref{S:paraeq} we paralinearized the water waves 
equations and identified the quadratic terms, 
with tame estimates for the remainders. 
Our next goal is to prove that one can further simplify the equations. 
We want either to eliminate the quadratic terms from the equations, or 
to eliminate the cubic terms from the energy estimates. In this section we introduce some notations. 
The strategy of the proof is explained in Section~\ref{S:I3} of the chapter of introduction. 

Set
$$
\vu =\begin{pmatrix} \vu^1 \\ \vu^2\end{pmatrix}
=\begin{pmatrix} \eta \\ \Dxmez\psi\end{pmatrix},
\quad \vU=\begin{pmatrix} \vU^1 \\ \vU^2\end{pmatrix}
=\begin{pmatrix} \eta+T_\alpha \eta \\ \Dxmez \omega
\end{pmatrix},
$$
with~$\alpha=\sqrt{\ma}-1$ where $\ma$ is as given by \eqref{n190} (see also \eqref{n191}). 
Assuming that Assumptions~\ref{T27} and \ref{T28} hold, 
our goal is to estimate the Sobolev norms $H^s$ of $\vU$ given an {\em a priori} 
estimate of some H\"older norm $\eC{\varrho}$ of $\vu$. Recall that we 
fixed $s$ and $\varrho$ such that
$$
s>\varrho+1>14,\quad \varrho\not\in\mez\xN.
$$

In this section, we introduce some notations in order to 
rewrite the water waves system under the form
\begin{equation}\label{quadratic}
\partial_t \vU +D\vU +Q(\vu)\vU+S(\vu)\vU+C(\vu)\vU=G,
\end{equation}
where $G$ is a cubic term of order $0$, satisfying
\begin{equation}\label{n205}
\lA G\rA_{H^s}\le C(\lA \vu\rA_{C^\varrho})\lA \vu\rA_{C^\varrho}^2\lA \vU\rA_{H^s}
\end{equation}
and where $(\vu,\vU)\mapsto Q(\vu)\vU$ and $(\vu,\vU)\mapsto S(\vu)\vU$ are bilinear while  
$C(\vu)\vU$ contains cubic and higher order terms. In addition
\begin{itemize}
\item $\vU\mapsto Q(\vu)\vU$ and $\vU\mapsto C(\vu)\vU$ are linear operators of order $1$ with tame 
dependence on $\vu$: this means that 
for any $\mu\in \xR$, if $\vu\in C^{\varrho}(\xR)$ then 
$\vU\mapsto Q(\vu)\vU \in \mathcal{L}(H^\mu,H^{\mu-1})$ and 
$\vU\mapsto C(\vu)\vU \in \mathcal{L}(H^\mu,H^{\mu-1})$,
 together with the estimates
\begin{align*}
&\lA Q(\vu)\rA_{\mathcal{L}(H^\mu,H^{\mu-1})}
\le C(\lA \vu\rA_{C^\varrho})\lA \vu\rA_{C^\varrho}, \\
&\lA C(\vu)\rA_{\mathcal{L}(H^\mu,H^{\mu-1})}
\le C(\lA \vu\rA_{C^\varrho})\lA \vu\rA_{C^\varrho}^2,
\end{align*}
for some nondecreasing function $C$ depending only on 
$\varrho$ and $\mu$. 

\item the linear operator $\vU\mapsto S(\vu)U$ is a smoothing operator with tame 
dependence on $u$: this means that for any $m\ge 0$ there exists $\rho>0$ 
such that, for any $\mu\in \xR$, if $\vu\in C^{\rho}(\xR)$ then 
$\vU\mapsto S(\vu)\vU \in \mathcal{L}(H^\mu,H^{\mu+m})$ together with the estimates
$$
\lA S(\vu)\rA_{\mathcal{L}(H^\mu,H^{\mu+m})}
\le C(\lA \vu\rA_{C^\rho})\lA \vu\rA_{C^\rho},
$$
for some nondecreasing function $C$ depending only on 
$m,\rho,\mu$. 
\end{itemize}

To do so, we rewrite the conclusion of 
Proposition~\ref{T30} as
$$
\partial_t \vU +D\vU +A_1=F,
$$
where $F=(F^1,F^2)$ was computed in the proof of Proposition~\ref{T30}, and where
\be\label{n206}
D=\begin{pmatrix} 0 & -\Dxmez \\ \Dxmez & 0\end{pmatrix}, \qquad
A_1=\begin{pmatrix}
T_{V}\partial_x \vU^1- T_{\alpha}\Dxmez \vU^2\\[1ex]
\Dxmez T_{V\la\xi\ra^{-1/2}}\px \vU^2
+ \Dxmez T_{\alpha}\vU^1
\end{pmatrix}.
\ee
We set
$$
\mathcal{G}=F+A_0+\mathcal{S},
$$
with 
\be\label{n207}
A_0=\begin{pmatrix}\mez T_{\px^2\psi}\eta\\0\end{pmatrix},~
\mathcal{S}=-
\begin{pmatrix}
F_{\quadratique}(\eta)\psi\\
\mez \Dxmez \RBony(\Dx\psi,\Dx\omega)
-\mez\Dxmez\RBony(\px\psi,\px\omega),
\end{pmatrix}
\ee
where~$F_{\quadratique}(\eta)\psi$ is given by (see~\eqref{n163})
$$
F_{\quadratique}(\eta)\psi=-\Dx \RBony(\eta,\Dx\psi)-\px \RBony(\eta,\px\psi).
$$
Then we may rewrite the equation for $\vU$ as 
$$
\partial_t \vU +D\vU +A_1+A_0+\mathcal{S}=\mathcal{G},
$$
where $\lmd{\mathcal{G}}{0}{s}$ by 
Proposition~\ref{T30}.

For later purposes, we write the explicit expression 
of $\mathcal{G}=(\mathcal{G}^1,\mathcal{G}^2)$:
\begin{equation}\label{n208}
\begin{aligned}
\mathcal{G}^1&=(\id+T_\alpha)F(\eta)\psi - F_{\quadratique}(\eta)\psi+
T_{\partial_t\alpha-\px V+\mez \px^2\psi}\eta\\[0.5ex]
&\quad +\Bigl\{ -T_\alpha T_{\px V}+T_V T_{\px\alpha}\eta+\bigl[T_V,T_\alpha\bigr]\Bigr\}\eta,\\[0.5ex]
\mathcal{G}^2&=\Dxmez \Bigl( \mez \RBony(\B,\B)-\mez\RBony(\Dx\psi,\Dx\omega)\Bigr)\\
&\quad -\Dxmez \Bigl( \mez \RBony(V,V)-\mez\RBony(\px\psi,\px\omega)\Bigr)\\
&\quad +\Dxmez \bigl((T_{V}  T_{\partial_x\eta}-T_{V \partial_x\eta})\B
+(T_{V \partial_x\B}-T_{V}  T_{\partial_x\B})\eta\bigr)\\
&\quad +\Dxmez T_V\RBony(\B,\partial_x\eta)-\Dxmez\RBony(\B,V\px\eta)\\
&\quad +\Dxmez (T_\alpha T_\alpha -T_{\alpha^2})\eta.
\end{aligned}
\end{equation}

\subsection*{Definition of $Q(\vu)$ and $C(\vu)$.} 
Here we define the terms $Q(\vu)$ and $C(\vu)$ which appear in \eqref{quadratic}. 
They arise when one splits $A_0$ and $A_1$ to isolate quadratic terms. 
Write 
$A_1=Q_1+C$ where
\begin{align}
Q_1&=\begin{pmatrix}
T_{\px\psi}\partial_x \vU^1- T_{-\mez\Dx\eta}\Dxmez \vU^2\\
\Dxmez T_{\px\psi\la\xi\ra^{-1/2}}\px \vU^2
+ \Dxmez T_{-\mez\Dx\eta}\vU^1
\end{pmatrix},\notag\\
C&=\begin{pmatrix}
T_{V-\px\psi}\partial_x \vU^1- T_{\alpha+\mez\Dx\eta}\Dxmez \vU^2 \\
\Dxmez T_{(V-\px\psi)\la\xi\ra^{-1/2}}\px \vU^2
+ \Dxmez T_{\alpha+\mez\Dx\eta}\vU^1
\end{pmatrix}.\label{n209}
\end{align}
Moreover, the quadratic term~$Q_1$ can be written under a form 
involving only the unknowns $\vu$ and $\vU$. We have
\begin{equation*}
Q_1= Q_1(\vu)\vU\defn\begin{pmatrix}
T_{\px\Dx^{-\mez}\vu^2}\partial_x \vU^1- T_{-\mez\Dx\vu^1}\Dxmez \vU^2\\
\Dxmez T_{\px\Dx^{-\mez}\vu^2\la\xi\ra^{-1/2}}\px \vU^2
+ \Dxmez T_{-\mez\Dx\vu^1}\vU^1
\end{pmatrix}.
\end{equation*}
We write below $C$ as given by~\eqref{n209}  
under the form $C(\vu)\vU$. 
This is an abuse of notations since 
$C$ cannot be directly written under the form of a function of 
$\vu=(\eta,\Dxmez\psi)$ and $\vU$. Instead, $C$ 
is an operator acting on $\vU$ whose coefficients depend on $(\eta,\psi)$. 
This abuse of notations will not 
introduce confusion since the estimates for this operator will always involved only 
$\vu$ and $\vU$. This is because the nonlinear estimates we proved for the 
Dirichlet-Neumann operator involved only $\Dxmez\psi$ and never $\psi$ itself.

Similarly, write
$$
\mez T_{\px^2\psi}\eta=-\mez T_{\Dx^\tdm \vu^2}\eta=
-\mez T_{\Dx^\tdm \vu^2}\vU^1+\mez T_{\Dx^\tdm \vu^2}T_\alpha\eta.
$$
to decompose $A_0$ as a sum $A_0=Q_0+C_0$ of a quadratic term and a cubic term. The cubic term 
$C_0$, being of order $0$ will contribute to the remainder $G$ in equation~\e{quadratic}. Eventually, we set
\begin{equation}\label{n210}
Q=Q_1+Q_0=
\begin{pmatrix}
T_{\px\Dx^{-\mez}\vu^2}\partial_x \vU^1-\mez T_{\Dx^\tdm \vu^2}\vU^1
- T_{-\mez\Dx\vu^1}\Dxmez \vU^2\\
\Dxmez T_{\px\Dx^{-\mez}\vu^2\la\xi\ra^{-1/2}}\px \vU^2
+ \Dxmez T_{-\mez\Dx\vu^1}\vU^1
\end{pmatrix}.
\end{equation}

\subsection*{Definition of $S(\vu)$.} 
Here we define the term $S(\vu)$ which appears in \eqref{quadratic}. To do so, 
with regards to $\mathcal{S}$, write $\eta=\vU^1-T_{\alpha}\eta$ and
\begin{alignat*}{2}
&\Dx \psi=\Dxmez\vu^2,\quad &&\px \psi=\px\Dx^{-\mez}\vu^2,\\
&\Dx\omega=\Dxmez\vU^2,\quad &&\px\omega=\px\Dx^{-\mez}\vU^2,
\end{alignat*}
to obtain 
$\mathcal{S}=S(\vu)\vU+\widetilde{\mathcal{S}}$ 
where $S(\vu)\vU=\left(\begin{smallmatrix}(S(\vu)\vU)^1\\(S(\vu)\vU)^2\end{smallmatrix}\right)$ 
with
\begin{align*}
(S(\vu)\vU)^1&=\Dx \RBony(\Dxmez\vu^2,\vU^1)
+\px \RBony(\px\Dx^{-\mez}\vu^2,\vU^1),\\
(S(v)f)^2&=-\mez \Dxmez \RBony(\Dxmez\vu^2,\Dxmez\vU^2)\\
&\quad +\mez\Dxmez\RBony(\px\Dx^{-\mez}\vu^2,\px\Dx^{-\mez}\vU^2),
\end{align*}
and
\begin{equation}\label{n211}
\widetilde{S}=\begin{pmatrix}
-\Dx \RBony(\Dx\psi,T_\alpha \eta)-\px \RBony(\px\psi,T_\alpha\eta)\\
0
\end{pmatrix}.
\end{equation}

\subsection*{Definition of $G$.} It follows from 
the computations above that 
\eqref{quadratic} holds with 
\be\label{n212}
G=\mathcal{G}-\widetilde{\mathcal{S}}-C_0
\ee
where $\mathcal{G}$ is given by \eqref{n208}, $\widetilde{\mathcal{S}}$ is 
given by \eqref{n211} 
and $C_0=\big(\mez T_{\Dx^\tdm \vu^2}T_\alpha\eta,0\big)$ 
arises when we rewrite $A_0$ in terms of $\vu$ and $\vU$. 
We have proved in Proposition~\ref{T30} that $\lmd{\mathcal{G}}{0}{s}$. On the other hand, 
it follows from \eqref{Bony3}Ê
(resp.\ \e{esti:quant0}) 
and the estimate~\eqref{n203} for $\alpha$ that 
$\lmd{\widetilde{\mathcal{S}}}{0}{s}$ (resp.\ $\lmd{C_0}{0}{s}$). 
This proves that $\lmd{G}{0}{s}$ as asserted in~\eqref{n205}.

\section{Quadratic normal form: strategy of the proof}\label{S:3.2.2}

To help the reader, let us reproduce here the explanations given 
in Section~\ref{S:I3} of the introduction. 
We want to implement the normal form approach by introducing a quadratic perturbation of 
$U$ of the form
$$
\Phi=\vU+E(\vu)\vU,
$$
where~$(\vu,\vU)\mapsto E(\vu)\vU$ is bilinear and chosen in such a way that the quadratic terms 
in the equation 
for $\Phi$ do not contribute to a Sobolev energy estimate. 

Writing
$$
\partial_t \Phi=\partial_t \vU +E(\partial_t \vu)\vU+E(\vu)\partial_t \vU.
$$
and replacing~$\partial_t \vU$ by~$-D\vU-(Q(\vu)+S(\vu))\vU$, we obtain that modulo cubic terms,
\begin{align*}
\partial_t \Phi&=-D\vU-(Q(\vu)+S(\vu))\vU-E(D\vu)\vU-E(\vu)D\vU
\\
&=-D\Phi+DE(\vu)\vU-(Q(\vu)+S(\vu))\vU
-E(D\vu)\vU-E(\vu)D\vU.
\end{align*}
It is thus tempting to seek~$E$ under the form $E=E_1+E_2$ such that 
\begin{align}
Q(\vu)U+E_1(D\vu)U+E_1(\vu)DU&=DE_1(\vu)U,\label{P:1}\\
S(\vu)U+E_2(D\vu)U+E_2(\vu)DU&=DE_2(\vu)U,\label{P:2}
\end{align}
to eliminate the quadratic terms in the equation for $\Phi$. 
However, one cannot solve these two equations directly for two different reasons. 
The equation \eqref{P:1} 
leads to a loss of derivative: for a general~$\vu\in H^\infty$ and~$s\ge 0$, 
it is not possible to eliminate the quadratic terms~$Q(\vu)U$ by means of a bilinear Fourier multiplier 
$E_1$ such that~$U\mapsto E_1(\vu)\vU$ is bounded from~$H^s$ to~$H^s$. 
Instead we shall 
add other quadratic terms to the equation to 
compensate the worst terms. 
More precisely, our strategy consists in seeking a bounded bilinear Fourier multiplier~$\tilde{E_1}$ 
(such that~$\vU\mapsto \tilde{E_1}(\vu)\vU$ is bounded from~$H^s$ to~$H^s$) 
such that the operator $B_1(\vu)$ given by
\begin{equation}\label{n214}
B_1(\vu)\vU\defn D\tilde{E_1}(\vu)\vU-\tilde{E_1}(D\vu)\vU-\tilde{E_1}(\vu)D\vU,
\end{equation}
satisfies
\begin{equation*}
\RE\langle Q(\vu)\vU-B_1(\vu)\vU,\vU\rangle_{H^s\times H^s}=0.
\end{equation*}
The key point is that one can find~$B_1(\vu)$ such that~$\vU\mapsto B_1(\vu)\vU$ 
is bounded from~$H^s$ to~$H^s$. 
This follows from the fact that, 
while~$\vU\mapsto Q(\vu)\vU$ is an operator of order~$1$, the operator 
$Q(\vu)+Q(\vu)^*$ is an operator of order~$0$. 
Once~$B_1$ is so determined, we find a bounded bilinear 
transformation~$\tilde{E_1}$ such that \eqref{n214} is satisfied. 
We here use the fact that~$Q$ is a paradifferential operator 
so that one has some restrictions on the support of the symbols. 

As explained in the introduction, 
the problem~\eqref{P:2} leads to another technical issue. Again, we shall verify that one can 
find $\tilde{E}_2(\vu)$ such that 
$$
\blA \tilde{E}_2(\vu)\brA_{\Fl{H^s}{H^s}}\le K \lA \vu\rA_{\eC{\varrho}}.
$$
and such that the operator $B_2(\vu)$ defined by
\begin{equation*}
B_2(\vu)\vU\defn D\tilde{E_2}(\vu)\vU-\tilde{E_2}(D\vu)\vU-\tilde{E_2}(\vu)D\vU,
\end{equation*}
satisfies
\begin{equation}\label{n215}
\RE\langle S(\vu)\vU-B_2(\vu)\vU,\vU\rangle_{H^s\times H^s}=0.
\end{equation}

\section{Paradifferential operators}\label{S:2.2.4}

Below we shall consider the equation
\begin{equation}\label{n216}
E(D\vu)\vU+E(\vu)D\vU-D \bigl[E(\vu)\vU\bigr]=\Pi (\vu)\vU,
\end{equation}
where~$(\vu,\vU)\mapsto E(\vu)\vU$ and~$(\vu,\vU)\mapsto \Pi(\vu)\vU$ 
are bilinear operators of the form
\begin{align*}
E(\vu)\vU&=\sum_{1\le k\le 2}\frac{1}{(2\pi)^2}\int
e^{ix(\xip+\xii)} \widehat{\vu^k}(\xip)A^k(\xip,\xii)\widehat{\vU}(\xii) \,d\xip \, d\xii,\\
\Pi(\vu)\vU&=\sum_{1\le k\le 2}\frac{1}{(2\pi)^2}\int
e^{ix(\xip+\xii)} \widehat{\vu^k}(\xip)M^k(\xip,\xii)\widehat{\vU}(\xii) \,d\xip \, d\xii,
\end{align*}
where $A^k$ and $M^k$ are $2\times 2$ matrices of symbols. 

We shall consider the problem~\eqref{n216} in two different cases according to the 
frequency interactions which are permitted in $E(\vu)\vU$ and $\Pi(\vu)\vU$. 
These cases are the following.
\begin{enumerate}[(i)]
\item The case where $\Pi(\vu)U$ is a paraproduct of the form $T_a b$. 
Namely, the case where there exists a constant $c\in \pol 0,1/2\por$ such that
$$
\supp M^k
\subset \Bigl\{ (\xip,\xii)\in\xR^2\, : \, \la\xii\ra\ge 1,~\la \xip\ra\le c \la\xii\ra\Bigr\}.
$$
\item The case where $\Pi(\vu)\vU$ is a remainder of the form 
$\RBony(a,b)$. Which means that there exists a constant $C>0$ such that
$$
\supp M^k\subset \Bigl\{ (\xip,\xii)\in\xR^2\,:\, 
\la \xip+\xii\ra \le C ( 1+\min (\la \xip\ra,\la\xii\ra))\Bigr\}.
$$
\end{enumerate}

There is another important property of the symbols 
which have to be taken into account. Indeed, when solving the equation 
$E(D\vu)\vU+E(\vu)D\vU-D \bigl[E(\vu)\vU\bigr]=\Pi (\vu)\vU$, we will have to invert a matrix which yields 
a small divisors issue. Here this problem arrises only for low frequencies. Therefore, we need to quantify the 
order of vanishing of the symbols on $\xip=0$, $\xii=0$ or $\xip+\xii=0$. 
For the analysis of the first case, for instance, since 
$|\xii|\ge 1$ and $|\xip+\xii|\ge 1/2$ on the support of $M^k$, 
it is sufficient to quantify the order of vanishing in $\xip$. 
We are thus lead to the following definition.

\begin{defi}
Let~$(m,\gamma,\nu)\in [0,+\infty\por^3$. 
One denotes\index{Symbols!$S_\nu^{m,\gamma}$} by~$S_\nu^{m,\gamma}$ 
the space 
of functions $(\xip,\xii)\mapsto A(\xip,\xii)$ 
with values in~$2\times 2$ matrices,~$C^\infty$ for 
$(\xip,\xii)\in  (\xR\setminus \{0\})\times\xR$ and satisfying
\begin{equation}\label{n217}
\exists c \in \pol 0,1/2\por \text{ such that} \quad 
\supp A(\xip,\xii)
\subset \Bigl\{ (\xip,\xii)\, : \, \la\xii\ra\ge 1,~\la \xip\ra\le c \la\xii\ra\Bigr\},
\end{equation}
and, for all~$(\alpha,\beta)\in\xN^2$,
\begin{equation}\label{n218}
\la \partial_{\xip}^\alpha \partial_{\xii}^{\beta}A(\xip,\xii)
\ra\le C_{\alpha \beta}\la \xip\ra^{-\alpha+\nu}\langle \xii\rangle^{m-\beta}\langle \xip\rangle^{\gamma}.
\end{equation}
If~$a=a(\xip,\xii)$ is a scalar valued function, we shall say that 
$a\in S^{m,\gamma}_\nu$ if~$a I_2\in S^{m,\gamma}_\nu$ where~$I_2$ the identity matrix.
\end{defi}

To analyze the remainders terms it is convenient to introduce the following definition.

\begin{defi}\label{T31}
Let~$(m,\nu_1,\nu_2)\in [0,+\infty\por^3$. 
One denotes\index{Symbols!$SR^m_{\nu_1,\nu_2}$} by~$SR^m_{\nu_1,\nu_2}$ 
the space of functions~$(\xip,\xii)\mapsto R(\xip,\xii)$ with values 
in~$2\times 2$ matrices,~$C^\infty$ for 
$(\xip,\xii)\in (\xR\setminus \{0\})\times(\xR\setminus \{0\})$ and satisfying
\be\label{n219}
\exists C>0 \text{ s.t.} \quad 
\supp R(\xip,\xii)\subset \Bigl\{ (\xip,\xii)\, : \, 
\la\xip+\xii\ra\le C(1+\min (\la\xip\ra,\la\xii\ra))\Bigr\},
\ee
and
\begin{equation}\label{n220}
\la \partial_{\xip}^\alpha \partial_{\xii}^{\beta}R(\xip,\xii)
\ra\le C_{\alpha \beta}\la \xip\ra^{-\alpha+\nu_1}\la \xii\ra^{-\beta+\nu_2} (1+ \la \xip\ra+\la\xii\ra)^{m}.
\end{equation}
\end{defi}

\begin{nota}
Given a scalar function $v$, 
a matrix~$A$ in one of these two classes of symbols and $f$ with values in $\xC^2$, 
we set\index{Pseudo-differential operators!$\Op^{\Bony}[v,A]f$}
\begin{equation}\label{n221}
\Op^{\Bony}[v,A]f
=\frac{1}{(2\pi)^2}\int
e^{ix(\xip+\xii)} \widehat{v}(\xip)A(\xip,\xii)\widehat{f}(\xii) \,d\xip \, d\xii.
\end{equation}
When there is no risk of confusion, we will use the notation $\Op^{\Bony}[v,A]f$ 
also for scalar symbols~$A$ and scalar unknowns~$f$. 
\end{nota}

\begin{prop}\label{T32}
$i)$ 
Given $m\in\xR$, one denotes 
\index{Symbols!$SR^m_{reg}$} by~$SR^m_{reg}$ 
the space of functions~$(\xip,\xii)\mapsto R(\xip,\xii)$ with values 
in~$2\times 2$ matrices,~$C^\infty$ for 
$(\xip,\xii)\in \xR^2$ satisfying~\eqref{n219} and 
\begin{equation}\label{n222}
\la \partial_{\xip}^\alpha \partial_{\xii}^{\beta}R(\xip,\xii)
\ra\le C_{\alpha \beta}(1+ \la \xip\ra+\la\xii\ra)^{m-\alpha-\beta}.
\end{equation}
Then for any $a\in [0,+\infty[$ and any $\sigma\in [0,+\infty\por$ such that $a+\sigma>m$, 
\be\label{n222a}
\lA \Op^{\Bony}[v,R]f\rA_{H^{\sigma+a-m}}\le K \lA v\rA_{H^\sigma}\lA f\rA_{\eC{a}}
\ee
and
\be\label{n222b}
\lA \Op^{\Bony}[v,R]f\rA_{H^{\sigma+a-m}}\le K \lA v\rA_{\eC{a}}\lA f\rA_{H^{\sigma}}.
\ee

$ii)$ Let $m$ in $\xR$ and let $\nu_1,\nu_2$ in $]0,+\infty[$. 
Consider two real numbers $a\in [0,+\infty[$ and 
$\sigma\in [0,+\infty\por$ such that $a+\sigma>m+\nu_1+\nu_2$. If $R$ belongs to $SR^m_{\nu_1,\nu_2}$ then 
$$
\lA \Op^{\Bony}[v,R]f\rA_{H^{\sigma+a-m-\nu_1-\nu_2}}\le K \lA v\rA_{H^\sigma}\lA f\rA_{\eC{a}}
$$
and
$$
\lA \Op^{\Bony}[v,R]f\rA_{H^{\sigma+a-m-\nu_1-\nu_2}}\le K \lA v\rA_{\eC{a}}\lA f\rA_{H^{\sigma}}.
$$
\end{prop}
\begin{proof}
$i)$ Notice that \e{n219} implies that there holds 
$C_0^{-1}\langle \xip\rangle \le \langle \xii\rangle\le C_0\langle\xip\rangle$ on the support of $R(\xip,\xii)$. 

Consider a dyadic decomposition of the identity (see Appendix~\ref{S:A.3}) and write 
\begin{align*}
\Delta_j \Op^\Bony[v,R]f&=\sum_{k>0}\sum_{\ell>0}\Delta_j \Op^\Bony \big[ \Delta_k v,R\big] \Delta_\ell f\\
&\quad +\sum_{k>0}\Delta_j \Op^\Bony \big[ \Delta_k v,R\big]S_0 f\\
&\quad+
\sum_{\ell>0}\Delta_j \Op^\Bony \big[ S_0 v,R\big]\Delta_\ell f \\
&\quad +\Delta_j\Op^\Bony[S_0v,R]S_0f.
\end{align*}
By using the previous remark and \e{n219} one can assume that 
$|k-\ell|\le N_0$ and $j\ge k-N_0$ in the first sum and the two other sums are non zero only if $j\le N_0$, $k\le N_0$, $\ell\le N_0$. 

The summand of the first sum can be written
$$
A_j^{k,\ell}=\Delta_j \int K_{k,\ell}(x-y_1,x-y_2)\Delta_k v(y_1)\Delta_\ell f(y_2)\, dy_1 dy_2
$$
with
$$
K_{k,\ell}=\frac{2^{k+\ell}}{(2\pi)^2}\int e^{i(2^k z_1\xip+2^\ell z_2 \xii)}\varphi(\xip)\varphi(\xii) R\big(2^k\xip,2^\ell \xii\big) 
\, d\xip d\xii.
$$
Since $R$ satisfies \e{n222}, the partial derivatives of the non oscillating term 
are $O(1)$ (since $|k-\ell|\le N_0$), whence the estimate
$$
\la K_{k,\ell}(z_1,z_2)\ra \le C_N 2^{k+\ell+km}\Big( 1+2^k |z_1|+2^\ell |z_2|\Big)^{-N}
$$
for any $N$. Therefore
$$
\big\vert A_j^{k,\ell}\big\vert  \le \lA \Delta_j f\rA_{L^\infty} \int 2^k \big( 1+2^{k} | x-y_1 |\big)^{-N} \la \Delta_k v(y_1)\ra\, dy_1 \cdot 2^{km}
$$
so
$$
\blA A_j^{k,\ell}\brA_{L^2} \le C 2^{km} \lA \Delta_j f\rA_{L^\infty}\lA \Delta_k v\rA_{L^2}\le C 
2^{-\ell a-k\sigma+km}c_k \lA f\rA_{\eC{a}}\lA v\rA_{H^\sigma}.
$$
Since we sum for $|k-\ell|\le N_0$, $k\ge j$, we obtain for $a+\sigma>m$
$$
\sum_{k,\ell} \blA  A_j^{k,\ell}\brA_{L^2} \le c_j 2^{-j(a+\sigma)}\lA f\rA_{\eC{a}}\lA v\rA_{H^\sigma}2^{jm}.
$$
The analysis of the other terms is trivial. This proves \e{n222a}. The proof of \e{n222b} is similar.

$ii)$ Since we assume that \e{n222} holds, if $\la\xip\ra\gg 1$ or $\la \xii\ra\gg 1$, the other term is large, and they are 
of comparable size. Then we have
$$
\Op^\Bony[v,R]f=\Op^\Bony\big[ \widetilde{S}_0 v,R\big] \bigl( \widetilde{S}_0 f)+\Op^\Bony\big[ v, \widetilde{R}\big] f
$$
where $\widetilde{S}_0$ cut-offs on a ball with a large enough radius and where 
$\widetilde{R}$ is in $SR_{reg}^{m+\nu_1+\nu_2}$. It suffices to study the first term, in which we decompose
$$
v=\sum_{k<N_0} \Delta_k v,\quad f=\sum_{\ell <N_0}\Delta_\ell f.
$$
If we set
$$
A^{k,\ell}= \int K_{k,\ell}(x-y_1,x-y_2)\Delta_k v(y_1)\Delta_\ell f(y_2)\, dy_1 dy_2
$$
then the kernel $K_{k,\ell}$ satisfies
$$
\la K_{k,\ell}(z_1,z_2)\ra\le C 2^{k(1+\nu_1)+\ell(1+\nu_2)}\Big(1+2^k | z_1 |+2^\ell | z_2|\Big)^{-N}
$$
whence
$$
\blA A^{k,\ell}\brA_{L^2}\le C \lA \Delta_\ell f\rA_{L^\infty}\lA \Delta_k v\rA_{L^2}2^{k\nu_1 +\ell \nu_2}
\le C \lA f\rA_{L^\infty}\lA v\rA_{L^2}c_k 2^{k\nu_1+\ell \nu_2}.
$$
To be able to sum on $k<0$, $\ell<0$, we need the assumption $\nu_1>0$, $\nu_2>0$. 
We then obtain that $\blA \Op^\Bony\big[ \widetilde{S}_0 v,R\big] \bigl( \widetilde{S}_0 f)\brA_{L^2}
\le C\lA f\rA_{L^\infty}\lA v\rA_{L^2}$ (together with a similar estimate in $\lA f\rA_{L^2}\lA v\rA_{L^\infty}$).
\end{proof}

In the rest of this section, we study the case where $A\in S^{m,\gamma}_\nu$. 
In particular, we shall prove that, for all 
$v\in \eC{\rho}\cap L^2(\xR)$ and all $A\in S^{m,\gamma}_\nu$, 
the operator $\Op^\Bony[v,A]$ is 
well-defined and bounded from $H^{\mu+m}(\xR)$ 
to $H^\mu(\xR)$ for any $\mu\in\xR$. 
To prove this result, we first notice that~$\Op^{\Bony}[v,A]$ is a pseudo-differential operator. Indeed, 
$$
\Op^{\Bony}[v,A]f
=\frac{1}{2\pi}\int_{\xR} e^{ix \xi} a(x,\xi)\widehat{f}(\xi) \,d\xi,
$$
where the symbol~$a$ is defined by
\begin{equation}\label{n223}
a(x,\xi)=\frac{1}{2\pi}\int e^{ix\xip} \widehat{v}(\xip)A(\xip,\xi)\, d\xip.
\end{equation}
Since $v\in L^2(\xR)$ and $A(\cdot,\xi)$ is bounded, 
$a(\cdot,\xi)$ is well-defined and belongs to $L^2(\xR;dx)$ by Plancherel's theorem.

The following two lemmas state that, in fact, if~$A\in S^{m,\gamma}_\nu$, then 
$a$ is a paradifferential symbol of order~$m$ and 
regularity~$C^{\rho-\gamma-\nu}$. 
We first consider the case $\nu=0$ and then the case $\nu>0$.

\begin{lemm}\label{T33}
$(i)$ Let~$(m,\gamma)\in [0,+\infty\por^2$, 
$A\in S^{m,\gamma}_0$ and consider a scalar function 
$v\in C^\rho\cap L^2$ where~$\rho$ is such that 
$\rho\ge \gamma,\rho\not\in\xN,\rho-\gamma\not\in\xN$. 
Then, for all~$\beta\in\xN$ and for all 
$\epsilon\in \pol0,1]$, there exists a constant~$K$ such that 
the symbol~$a$ defined by~\eqref{n223} satisfies
$$
\sup_{\xi}
\lA \langle\xi\rangle^{\beta-m}\partial_{\xi}^\beta a(\cdot,\xi)\rA_{\eC{\rho-\gamma}}
\le K
\Bigl\{ \lA v\rA_{\eC{\rho}}+ \frac{1}{\epsilon}
 \lA v\rA_{\eC{\rho}}^{1-\epsilon}\lA v\rA_{L^2}^\epsilon\Bigr\}\cdot
$$
$(ii)$ Let~$(m,\gamma,\nu)\in  [0,+\infty\por^3$ and assume that $\nu>0$. 
Consider $A\in S^{m,\gamma}_\nu$ and a scalar function 
$v\in C^\rho\cap L^{2}(\xR)$ where~$\rho$ is such that 
$\rho\ge \gamma+\nu,\rho\not\in\xN,\rho-\gamma-\nu\not\in\xN$. 
Then, for all~$\beta\in\xN$, there exists a constant~$K$ such that 
the symbol~$a$ defined by~\eqref{n223} satisfies
$$
\sup_{\xi}
\lA \langle\xi\rangle^{\beta-m}\partial_{\xi}^\beta a(\cdot,\xi)\rA_{\eC{\rho-\gamma-\nu}}
\le K \lA v\rA_{\eC{\rho}}.
$$

\end{lemm}
\begin{proof}Let us prove statement $(i)$. 
Consider the dyadic decomposition of the identity 
$\id=\Phi(D_x)+\sum_{j=1}^\infty \Delta_j$ introduced in~\eqref{defi:LP}. 
We have to prove, for~$j\in \xN^*$ and~$\beta\in\xN$,
$$
\lA \Delta_j\partial_{\xi}^\beta a(\cdot,\xi)\rA_{L^\infty(dx)}
\le K \lA v\rA_{\eC{\rho}} \langle \xi\rangle^{m-\beta}2^{-j(\rho-\gamma)},
$$
and an analogous estimate for the low frequencies. 
One can assume without loss of generality that~$\beta=0$. 

Consider~$j\in \xZ$ and a~$C^\infty$ function~$\tilde{\phi}$ with compact support 
such that~$\tilde{\phi}=1$ 
on the support of~$\phi$ and~$\tilde{\phi}=0$ on a neighborhood of the origin. 
Then 
\begin{align*}
\Delta_j a(x,\xi)&=
\frac{1}{2\pi}\int e^{ix\xip}\phi(2^{-j}\xip) \widehat{v}(\xip)
A(\xip,\xi)\, d\xip\\
&=\frac{2^j}{2\pi}\int e^{i2^j(x-x')\xip}\tilde{\phi}(\xip) \Delta_j v(x') A(2^j\xip,\xi)
\, dx' \,d\xip\\
&=2^j \int E_j(2^j(x-x'),\xi)\Delta_j v(x')\, dx'
\end{align*}
where
\begin{equation}\label{n224}
E_j(z,\xi)=\frac{1}{2\pi}\int e^{iz\xip}\tilde{\phi}(\xip) A(2^j\xip,\xi)\, d\xip.
\end{equation}
Then, integrating by parts, the inequalities~\eqref{n218} and the support 
condition \eqref{n217}Ê
imply that for all~$n\in\xN$ there is a constant 
$C_n$ such that, for all~$(z,\xi)\in \xR^2$ and all~$j\in \xZ$,
$$
\la z^n E_j(z,\xi)\ra \le C_n \langle 2^{j\gamma}\rangle \langle \xi\rangle^{m}.
$$
Consequently, the kernel satisfies 
$\lA E_j(\cdot,\xi)\rA_{L^1(dz)}\le K \langle 2^{j\gamma}\rangle
 \langle\xi\rangle^{m}$. 
 For~$j>0$, we deduce that
$$
\lA \Delta_j a(\cdot,\xi)\rA_{L^\infty(dx)}
\les \langle\xi\rangle^{m} \lA \Delta_j v\rA_{L^\infty}2^{j\gamma}
\les \langle\xi\rangle^{m} 2^{-j(\rho-\gamma)}\lA v\rA_{\eC{\rho}}.
$$

On  the other hand, for~$j<0$, write
\begin{equation}\label{n225}
\lA \Delta_j a(\cdot,\xi)\rA_{L^\infty}
\le K \langle\xi\rangle^m  \lA \Delta_j v\rA_{L^\infty}
= K \langle\xi\rangle^m \lA \Delta_j v\rA_{L^\infty}^{1-\epsilon}
\lA \Delta_j v\rA_{L^\infty}^\epsilon.
\end{equation}
Estimating $\lA \Delta_j v\rA_{L^\infty}\les 2^{j/2}\lA \Delta_j v\rA_{L^2}$, we get
%
$$
\lA \Delta_j a(\cdot,\xi)\rA_{L^\infty}
\le K 2^{j\epsilon/2} \langle\xi\rangle^m 
\lA \Delta_j v\rA_{L^\infty}^{1-\epsilon}\lA \Delta_j v\rA_{L^2}^\epsilon.
$$
Since $a(\cdot,\xi)\in L^2(\xR)$ and since 
$\sum_{-\infty}^{-1} 2^{j\epsilon/2}=O(\epsilon^{-1})$, summing on~$j<0$ 
(using Remark~\ref{rema:LF}), we obtain that
$$
\lA  \Phi(D_x) a(\cdot,\xi)\rA_{\eC{\rho-\gamma}}
\les \lA \Phi(D_x) a(\cdot,\xi)\rA_{L^\infty}
\les \frac{\langle\xi\rangle^m}{\epsilon}
\lA v\rA_{L^\infty}^{1-\epsilon}\lA  v\rA_{L^2}^\epsilon,
$$
which completes the proof of statement $(i)$.

We now prove statement $(ii)$. 
Since~$S^{m,\gamma}_\nu\subset S^{m,\gamma+\nu}_0$, 
the analysis of the high-frequency component follows from the previous proof. 
It remains only to 
bound the low-frequency component. 
Namely, it remains to estimate $\blA \Phi(D_x) \partial_\xi^\beta a(\cdot,\xi)\brA_{L^\infty}$. Again, it is sufficient to 
consider the case $\beta=0$. As above,
\begin{align*}
\Phi(D_x) a(x,\xi)
=\int E(x-x',\xi)(\Phi(D_x)v)(x')\, dx'
\end{align*}
with
$$
E(z,\xi)
=\frac{1}{2\pi}\int e^{iz\xip}\tilde{\Phi}(\xip)A(\xip,\xi)\,d\xip,
$$
where~$\tilde\Phi\in C^\infty_0(\xR)$ satisfies 
$\tilde\Phi=1$ on the support of~$\Phi$. To conclude the proof, 
we have to estimate the $L^1(\xR;dz)$-norm of $E(\cdot,\xi)$. 
This will follow from the following fact:
if~$g=g(\xi)$ is a compactly supported function,~$C^\infty$ for 
$\xi\in \xR\setminus\{0\}$ and such that its derivatives satisfy
$$
\la g(\xi)\ra \le \la \xi\ra^{\nu},\quad
\la g'(\xi)\ra \le \la \xi\ra^{\nu-1},\quad 
\la g''(\xi)\ra \le \la \xi\ra^{\nu-2},  
$$
with~$\nu>0$, then its inverse Fourier transform~$\tilde g$ belongs to~$L^1(\xR)$. 
\end{proof}

We thus have proved that 
$\Op^{\Bony}[v,A]U
=\frac{1}{2\pi}\int_{\xR} e^{ix \xi} a(x,\xi)\widehat{U}(\xi) \,d\xi$ 
where $a$ is a paradifferential symbol. We now claim that 
$\Op^\Bony[v,A]$ is a paradifferential operator. 
\begin{lemm}\label{T34}
Let~$(m,\gamma,\nu)\in  [0,+\infty\por^3$. 
Consider $A\in S^{m,\gamma}_\nu$ and a scalar function 
$v\in C^\rho\cap L^{2}(\xR)$ with  
$\rho\ge \gamma+\nu,\rho\not\in\xN,\rho-\gamma-\nu\not\in\xN$. Then
$$
\Op^{\Bony}[v,A]=T_a +R,
$$
where $T_a$ is the paradifferential operator with symbol $a$ given by \eqref{n223} and $R$ is a smoothing 
operator of order $m-(\rho-\gamma-\nu)$, satisfying
$$
\lA R f \rA_{H^{\mu-m+(\rho-\gamma-\nu)}}\le K
\Bigl( \sup_{\la\xi\ra \ge 1/2~}
\lA \langle\xi\rangle^{\beta-m}\partial_{\xi}^\beta a(\cdot,\xi)\rA_{\eC{\rho-\gamma-\nu}}
\Bigr)\lA f\rA_{H^\mu}.
$$
\end{lemm}
\begin{proof}
By virtue of the support condition~\eqref{n217}, there exists 
a $C^\infty$ function $\Theta$  satisfying the same properties as $\theta$ does 
in Definition~\ref{defi:theta}, except 
that 
\begin{alignat*}{5}
\Theta(\xip,\xii)&=1 \quad &&\text{if}\quad &&\la\xip\ra\le \widetilde{\eps}_1(1+\la \xii\ra) 
\quad&&\text{and }
&&\la\xii\ra \ge 2,\\
\Theta(\xip,\xii)&=0 \quad &&\text{if}\quad &&\la\xip\ra\geq \widetilde{\eps}_2(1+\la\xii\ra)
\quad&&\text{or }
&&\la\xii\ra \le 1,
\end{alignat*}
for some $0<\widetilde{\eps}_1<\eps_1<\eps_2<\widetilde{\eps}_2<1/2$ with the additional 
assumption that $c<\widetilde{\eps}_2$ where $c$ is the small constant which appears 
in \eqref{n217}. 
Denote by $T_a^\Theta$ the operator defined by 
$$
T_a^\Theta f =\frac{1}{(2\pi)^2}\int e^{ix(\xip+\xii)}\Theta(\xip,\xii)\widehat{a}(\xip,\xii)
\widehat{f}(\xii)\, d\xip \, d\xii,
$$
where
$\widehat{a}(\xip,\xii)=\int e^{-ix\xip}a(x,\xii)\, dx$. Now, 
$\Op^\Bony[v,A]=T_a^\Theta$, 
which is better written as 
$\Op^\Bony[v,A]=T_a+R$ with $R\defn T_a^\Theta-T_a$. 
Since $\theta$ are $\Theta$ are two admissible cut-off functions 
(in the sense of Remark~\ref{rema:cutoff}) it follows 
from~\cite[Prop. $5.1.17$]{MePise} 
that $R=T_a^\Theta-T_a=T_a^\Theta-T_a^\theta$ is of order $m-r$ if 
$a$ is a symbol of order $m$ in $\xi$ with regularity $\eC{r}$ in $x$. 
\end{proof}
We conclude this part by establishing two identities.

\begin{lemm}\label{T35}
Let~$(m,\gamma,\nu)\in  [0,+\infty\por^3$. 
Consider $A\in S^{m,\gamma}_\nu$ and a real-valued function 
$v\in C^\rho\cap L^{2}(\xR)$ with $\rho\ge \gamma+\nu,\rho\not\in\xN,\rho-\gamma-\nu\not\in\xN$. 
Then
$$
\bigl(\Op^{\Bony}[v,A]\bigr)^*=\Op^{\Bony}[v,B],
$$
with
$B(\xip,\xii)=\overline{A^{T}(-\xip,\xip+\xii)}$ where~$A^{T}$ is the transpose of~$A$.
\end{lemm}
\begin{proof}
We have
$$
\widehat{\Op^{\Bony}[v,A]W}(\eta)
=\frac{1}{2\pi}\int\widehat{v}(\xip)A(\xip,\eta-\xip)\widehat{W}(\eta-\xip)\, d\xip,
$$
so that
\begin{align*}
\left\langle \Op^{\Bony}[v,A]^*U,W\right\rangle
&=\frac{1}{(2\pi)^2} \int \widehat{U}(\xii)
\overline{\widehat{v}(\xip) A(\xip,\xii-\xip) \widehat{W}(\xii-\xip)} \,d\xii \,d\xip \\
&=\frac{1}{(2\pi)^2} \int \widehat{v}(\xip)\overline{A^{T}(-\xip,\xii)}\widehat{U}(\xii-\xip)
\overline{\widehat{W}(\xii)}\, d\xii\, d\xip \\
&=\frac{1}{2\pi}\int \widehat{\Op^{\Bony}[v,B]U}(\xii) \overline{\widehat{W}(\xii)}\, d\xii,
\end{align*}
with~$B(\xip,\xii)=\overline{A^{T}(-\xip,\xip+\xii)}$.
\end{proof}

We shall also use the identity 
\begin{equation}\label{n226}
x\px \Op^{\Bony}[v,A]f=\Op^{\Bony}[ x\px v,A]f+\Op^{\Bony}[v,A](x\px f)
-\Op^{\Bony}[v,\xi\cdot\nabla_\xi A]f.
\end{equation}
Indeed, this follows from an integration by parts, using 
$$
x\px e^{ix(\xip+\xii)}=\xip \partial_{\xip} e^{ix(\xip+\xii)} + \xii \partial_{\xii} e^{ix(\xip+\xii)}.
$$
In particular,
\begin{equation}\label{n227}
x\px T_a b=T_{x\px a}b+T_a(x\px b)+S_\Bony(a,b),
\end{equation}
where $S_\Bony(a,b)=\Op^{\Bony}[a,R]b$ with $R=-\xi\cdot\nabla_{\xi}\theta$ 
where $\theta$ is given by Definition~\ref{defi:theta}.

\section{The main equations}

We continue our normal form analysis by studying the equation
\begin{equation}\label{E}
E(Dv)f+E(v)Df-D \bigl[E(v)f\bigr]=\Pi (v)f,
\end{equation}
where $v=(v^1,v^2)$, $f=(f^1,f^2)$ and $(v,f)\mapsto E(v)f$ and~$(v,f)\mapsto \Pi(v)f$ 
are bilinear operators of the form
\begin{alignat*}{2}
&E(v)f=\Op^{\Bony}\big[v^1,A^1\big]f
&&+\Op^{\Bony}\big[v^2,A^2\big]f,\\
&\Pi(v)f=\Op^{\Bony}\big[v^1,M^1\big]f
&&+\Op^{\Bony}\big[v^2,M^2\big]f.
\end{alignat*}
We first consider the case where $(M^1,M^2)\in 
S^{m,\gamma}_{\nu}\times S^{m,\gamma}_{\nu}$. 
\begin{prop}\label{T36}
Let~$(m,\gamma)\in ([0,+\infty\por)^2$,~$\nu\in [1,+\infty\por$ and consider 
$(M^1,M^2)$ in $S^{m,\gamma}_{\nu}\times S^{m,\gamma}_{\nu}$. 
Then there exist 
$A^1\in S^{m,\gamma}_{\nu-1/2}$ and~$A^2\in S^{m,\gamma}_{\nu-1/2}$ 
such that
$$
E(v)f=\Op^{\Bony}[v^1,A^1]f
+\Op^{\Bony}[v^2,A^2]f,
$$
satisfies~\eqref{E} and 
$(M^1,M^2)\mapsto (A^1,A^2)$ is continuous 
from~$S^{m,\gamma}_\nu\times S^{m,\gamma}_{\nu}$ 
to~$S^{m,\gamma}_{\nu-1/2}\times S^{m,\gamma}_{\nu-1/2}$.
\end{prop}
\begin{proof}
We have
\begin{align*}
DE(v)f&=
\Op^{\Bony}\big[v^1,D(\xip+\xii)A^1(\xip,\xii)\big]f
+\Op^{\Bony}\big[v^2,D(\xip+\xii)A^2(\xip,\xii)\big]f,\\[0.5ex]
E(Dv)f&=
\Op^{\Bony}\big[-\Dxmez v^2,A^1(\xip,\xii)\big]f
+\Op^{\Bony}\big[\Dxmez v^1,A^2(\xip,\xii)\big]f,\\[0.5ex]
E(v)Df&=
\Op^{\Bony}\big[v^1,A^1(\xip,\xii)D(\xii)\big]f
+\Op^{\Bony}\big[v^2,A^2(\xip,\xii)D(\xii)\big]f,
\end{align*}
where~$D(\xi)=\left(\begin{smallmatrix} 0 & -1\\ 
1  & 0\end{smallmatrix}\right)\la\xi\ra^\mez$ is the matrix-valued symbol of the operator~$D$. 
To solve
\begin{equation*}
E(Dv)f+E(v)Df-D\bigl[E(v)f\bigr]=\Pi(v)f
=\Op^{\Bony}\big[v^1,M^1\big]f
+\Op^{\Bony}\big[v^2,M^2\big]f,
\end{equation*}
we thus have to solve
\begin{equation}\label{n229}
\left\{
\begin{aligned}
&&-D(\xip+\xii)A^1+A^1D(\xii)+\la\xip\ra^\mez A^2=M^1,\\[0.5ex]
&&-D(\xip+\xii)A^2+A^2D(\xii)-\la\xip\ra^\mez A^1=M^2.
\end{aligned}
\right.
\end{equation}

Denote by~$a^k_{ij}$ (resp.\ $m^k_{ij})$,~$1\le i,j\le 2$, the coefficients of the matrix~$A^k$ (resp.\ $M^k$),~$k=1,2$. 
To solve~\eqref{n229}, we have to solve two~$4\times 4$ systems 
for the~$8$ unknowns~$a^{k}_{ij}$. 
To simplify the computations, 
it is convenient 
to observe that this~$8\times 8$ system can be decoupled into two other 
$4\times 4$ systems: one system for 
$(a^2_{11},a^1_{12},a^1_{21},a^2_{22})$ and another system for~$(a^1_{11},a^2_{12},a^1_{22},a^2_{21})$. 
They read
\begin{alignat*}{4}
&\la \xip+\xii\ra^{\mez} a^1_{21}
~&&+\la \xip\ra^{\mez} a^2_{11}  ~&&+\la \xii\ra^{\mez} a^1_{12}~&&=m^1_{11},\\
-&\la \xip+\xii\ra^\mez a^2_{11}
~&&-\la \xip\ra^\mez  a^1_{21} ~&&+ \la \xii\ra^\mez a^2_{22} ~&&=m^2_{21},\\
-&\la \xip+\xii\ra^\mez a^1_{12}
~&&+\la \xip\ra^\mez  a^2_{22}~&&-\la \xii\ra^\mez a^1_{21}~&&=m^1_{22},\\
&\la \xip+\xii\ra^\mez a^2_{22} 
~&&-\la \xip\ra^\mez a^1_{12} ~&&-\la \xii\ra^\mez a^2_{11} ~&&=m^2_{12},
\end{alignat*}
and
\begin{alignat}{4}
-&|\xip+\xii|^\mez a^1_{11} ~&&+ |\xip|^\mez a^2_{21}
~&&+|\xii|^\mez a^1_{22} ~&&=m^1_{21},\label{LF:2}\\
&|\xip+\xii|^\mez a^2_{21} &&-|\xip|^\mez a^1_{11} 
&&+|\xii|^\mez a^2_{12} &&=m^2_{11},\label{LF:1}\\
&|\xip+\xii|^\mez a^1_{22}&&+ |\xip|^\mez a^2_{12}
&&-|\xii|^\mez a^1_{11} &&=m^1_{12},\label{LF:3}\\
-&|\xip+\xii|^\mez a^2_{12} &&-|\xip|^\mez a^1_{22}
&&-|\xii|^\mez a^2_{21}&&=m^2_{22}.\label{LF:4}
\end{alignat}
Clearly, these two systems are equivalent and 
it is enough to solve one of them. 

Let us solve~\eqref{LF:2}--\eqref{LF:4}. 
By using \eqref{LF:2} and \eqref{LF:4} one can determine~$a^2_{12}$ and 
$a^1_{11}$ by means of~$a^1_{22}$ and~$a^2_{21}$. 
It remains only a~$2\times 2$ system for~$(a^1_{22},a^2_{21})$. 
Set 
$$
\delta\defn\la\xip+\xii\ra-\la\xip\ra-\la\xii\ra,\quad 
D\defn \delta^2-4\la \xip\ra\la \xii\ra.
$$ 
It is found that
\begin{alignat*}{5}
&\delta a^2_{21}
&&-2\la\xip\ra^\mez \la\xii\ra^\mez a^1_{22}
&&=-&&\la\xip\ra^\mez m^1_{21}&&+\la\xii\ra^\mez m^2_{22}
+|\xip+\xii|^\mez m^2_{11},\\
&\delta a^1_{22}
&&-2\la\xip\ra^\mez \la\xii\ra^\mez a^2_{21}
&&=&&\la\xip\ra^\mez m^2_{22}&&-\la\xii\ra^\mez m^1_{21}
+|\xip+\xii|^\mez m^1_{12},
\end{alignat*}
thus
\begin{equation}\label{n234}
\begin{aligned}
a^2_{21}&=\frac{\delta}{D}\bigl(|\xip+\xii|^\mez m^2_{11}
-\la\xip\ra^\mez m^1_{21}+\la\xii\ra^\mez m^2_{22}\bigr)\\
&\quad +\frac{2}{D}\la\xip\ra^\mez\la\xii\ra^\mez
\bigl(|\xip+\xii|^\mez m^1_{12}+\la\xip\ra^\mez m^2_{22}-\la\xii\ra^\mez m^1_{21}\bigr),\\
a^1_{22}&=\frac{\delta}{D}\bigl(|\xip+\xii|^\mez m^1_{12}
+\la\xip\ra^\mez m^2_{22}-\la\xii\ra^\mez m^1_{21}\bigr)\\
&\quad+\frac{2}{D}\la\xip\ra^\mez\la\xii\ra^\mez
\bigl(|\xip+\xii|^\mez m^2_{11}-\la\xip\ra^\mez m^1_{21}+\la\xii\ra^\mez m^2_{22}\bigr),\\
a^2_{12}&=-\frac{1}{|\xip+\xii|^\mez}
\bigl( |\xip|^\mez a^1_{22}+|\xii|^\mez a^2_{21}+m^2_{22}\bigr),\\
a^1_{11}&=\frac{1}{|\xip+\xii|^\mez}
\bigl( |\xip|^\mez a^2_{21}+|\xii|^\mez a^1_{22}-m^1_{21}\bigr).
\end{aligned}
\end{equation}

We here give simplified expressions for~$\delta$ and~$D$ on the support of the symbols 
$m^k_{ij}$. Notice that, by definition of the spaces 
$S^{m,\gamma}_\nu$, we have~$\la\xip\ra<\la\xii\ra/2$ 
on the support of the symbols~$m^k_{ij}$. 
We then observe that
\begin{alignat*}{4}
&\xip\xii>0\quad &&\Rightarrow\quad &&\delta=0
~&&\text{and}~ D=-4\la \xip\ra \la \xii\ra ,\\
&\xip\xii<0 \text{ and }\la\xip\ra<\la\xii\ra 
\quad 
&&\Rightarrow \quad&&\delta=-2\la\xip\ra ~
&&\text{and}~ D
=-4\la \xip\ra \la \xip+\xii\ra.
\end{alignat*}

Thus, for all~$(\xip,\xii)\in \xR^2$, if 
$\la\xip\ra\le \la\xii\ra/2$ then $\la D\ra \ge  \la \xip\ra\la\xii\ra$. 
Consequently, since $\la \xii\ra\sim \langle \xii\rangle$ on the supports of $m^k_{ij}$, we 
have 
$\la D\ra \ge  \la \xip\ra \langle \xii\rangle$ on the supports of $m^k_{ij}$.

Now, since $m^k_{ij}\in S^{m,\gamma}_\nu$ for some $\nu\ge 1/2$ by assumptions, one 
can write~$m^k_{ij}=\la \xip\ra^{1/2}\tilde{m}^k_{ij}$ 
with~$\tilde{m}^k_{ij}\in S^{m,\gamma}_{\nu-1/2}$. 
Furthermore there exists a~$C^\infty$ function 
$\tilde\theta\colon \xR^2\rightarrow \xR$ satisfying
\begin{equation}\label{n235}
\tilde\theta(\xip,\xii)=0\quad\text{for } \la\xip\ra\ge \mez\la\xii\ra \text{ or }\la \xii\ra\le \mez,
\end{equation}
such that~$\tilde{m}^k_{ij}=\tilde{\theta}(\xip,\xii)\tilde{m}^k_{ij}$. 

Introduce the coefficients
\begin{alignat*}{3}
c_1&\defn \tilde\theta\frac{\delta}{D} \la\xip\ra^\mez \la\xip+\xii\ra^\mez ,\quad 
&&c_2\defn \tilde\theta\frac{\delta}{D} \la\xip\ra^{\mez}\la\xii\ra^{\mez} ,\quad 
&&c_3\defn \tilde\theta\frac{\delta}{D} \la\xip\ra, \\
c_4&\defn \tilde\theta\frac{2}{D}\la \xip\ra \la \xii\ra^\mez \la\xip+\xii\ra^\mez,\quad 
&&c_5\defn \tilde\theta\frac{2}{D}\la \xip\ra^\tdm \la \xii\ra^\mez,
&&c_6\defn \tilde\theta\frac{2}{D}\la \xip\ra  \la \xii\ra,
\end{alignat*}
In view of the support restrictions \eqref{n235} and the simplified 
expressions for~$\delta$ and~$D$ given above, these coefficients 
belong to~$S^{0,0}_{0}$. 

Thus, for any coefficient~$c_\ell$ ($\ell=1,\ldots,6$) and any symbol~$\tilde{m}^k_{ij}$, one has 
$c_\ell \tilde{m}^k_{ij}\in S^{m,\gamma}_{\nu-1/2}$. 
Now, using the formulas~\eqref{n234}, we obtain 
that~$a^2_{21}$ and~$a^1_{22}$ can be written as linear combinations 
of terms of the form~$c_\ell \tilde{m}^k_{ij}$. This implies that 
the symbols~$a^2_{21}$ and~$a^1_{22}$  belong to~$S^{m,\gamma}_{\nu-1/2}$. 
This in turn implies that~$a^2_{12},a^1_{11}$ belong to~$S^{m,\gamma}_{\nu-1/2}$, which 
concludes the proof. 
\end{proof}

We next consider 
the following problem:
\begin{equation*}
E_R(Dv)f+E_R(v)Df-D \bigl[E_R(v)f\bigr]=S (v)f,
\end{equation*}
where we recall that 
$S(v)f=\left(\begin{smallmatrix}(S(v)f)^1\\(S(v)f)^2\end{smallmatrix}\right)$ 
with
\begin{align*}
(S(v)f)^1&=\Dx \RBony(\Dxmez v^2, f^1)
+\px \RBony(\px\Dx^{-\mez}v^2,f^1),\\
(S(v)f)^2&=-\mez \Dxmez \RBony(\Dxmez v^2,\Dxmez f^2)\\
&\quad +\mez\Dxmez\RBony(\px\Dx^{-\mez} v^2,\px\Dx^{-\mez} f^2).
\end{align*}
We shall see that it is useful to split $S(v)f$ into two parts. 
Introduce
\be\label{n235.1}
S^\sharp(v)f=\begin{pmatrix}
(S(v)f)^1\\
0
\end{pmatrix},\quad 
S^\flat(v)f=
\mez\begin{pmatrix}
0\\
(S(v)f)^2
\end{pmatrix}.
\ee
These two operators are different because $S^\flat(v)f$ satisfies $S^\flat(v)f=S^\flat(f)v$, 
while $S^\sharp(v)f$ does not satisfy this symmetry.

Our purpose is to study the equations
\begin{align*}
&E^\sharp(Dv)f+E^\sharp(v)Df-D \bigl[E^\sharp(v)f\bigr]=S^\sharp (v)f,\\
&E^\flat (Dv)f+E^\flat(v)Df-D \bigl[E^\flat(v)f\bigr]=S^\flat (v)f.
\end{align*}
The next proposition states that one can solve these equations, and that the solutions  
$E^\sharp(v)$ and $E^\flat(v)f$ are smoothing operators depending tamely on $v$. 
Recall that the spaces of symbols $SR^m_{\nu_1,\nu_2}$ have been introduced in Definition~\ref{T31}.

\begin{prop}\label{T37}
There exist four matrices of symbols $R^{\sharp,1}$, $R^{\sharp,2}$, $R^{\flat,1}$, $R^{\flat,2}$ in 
$SR^{1}_{0,0}$ such that 
the following properties hold.

$i)$ Let $(\mu,\rho)\in \xR\times \xR_+$ be such that $\mu+\rho>1$. 
The bilinear operators given by 
\begin{align*}
&(v,f)\mapsto E^\sharp(v)f= \Op^\Bony[v^1,R^{\sharp,1}]f+\Op^\Bony[v^2,R^{\sharp,2}]f, \\
&(v,f)\mapsto E^\flat(v)f= \Op^\flat[v^1,R^{\flat,1}]f+\Op^\Bony[v^2,R^{\flat,2}]f,
\end{align*}
are well-defined for any $(v,f)$ in $(\eC{\rho}\cap L^2(\xR))\times H^\mu(\xR)$ or in 
$H^{\mu}(\xR)\times (\eC{\rho}(\xR)\cap L^2(\xR))$. 

$ii)$ There holds
\begin{align}
&E^\sharp(Dv)+E^\sharp(v)D-D E^\sharp(v)=S^\sharp(v),\label{n236}\\
&E^\flat(Dv)+E^\flat(v)D-D E^\flat(v)=S^\flat(v).\label{n237}
\end{align}

$iii)$ The following estimates hold.  

$\bullet$ For all $(\mu,\rho)\in \xR\times \xR_+$ such that $\mu+\rho>1$ and $\rho\not\in\mez\xN$, 
there exists a positive constant $K$ such that, for any 
$f\in \eC{\rho}(\xR)\cap L^2(\xR)$ and any $v\in H^\mu(\xR)$
\begin{align}
&\blA E^\sharp(v)f \brA_{H^{\mu+\rho-1}}\le 
K\lA f\rA_{\eC{\rho}}\lA v\rA_{H^\mu},\label{n238}\\
&\blA E^\flat(v)f \brA_{H^{\mu+\rho-1}}\le 
K\bigl(\lA f\rA_{\eC{\rho}}+\lA \mathcal{H}f\rA_{\eC{\rho}}\bigr)\lA v\rA_{H^\mu},\label{n238.1}
\end{align}
where $\mathcal{H}v$ is the Hilbert transform of $v$.

$\bullet$ for all $(\mu,\rho)\in \xR\times \xR_+$ 
such that $\mu+\rho>1$ and $\rho\not\in\mez\xN$ 
there exists a positive constant $K$ such that
\begin{equation}\label{n239}
\blA E^\sharp(v)f \brA_{H^{\mu+\rho-1}}+\blA E^\flat(v)f \brA_{H^{\mu+\rho-1}}\le 
K\left(\lA v\rA_{\eC{\rho}}+\lA \mathcal{H}v\rA_{\eC{\rho}}\right)
\lA f\rA_{H^\mu}.
\end{equation}

$iv)$ The operators $\RE E^\sharp(v)=\mez(E^\sharp(v)+E^\sharp(v)^*)$ 
and $\RE E^\flat(v)$ satisfy
\begin{align}
&\RE E^\sharp(Dv)+\RE E^\sharp(v)D-D \RE E^\sharp(v)=\RE S^\sharp (v),\label{n240}\\
&\RE E^\flat(Dv)+\RE E^\flat(v)D-D \RE E^\flat(v)=\RE S^\flat(v).\label{n241}
\end{align}
Moreover, 
for all $(\mu,\rho)\in \xR\times \xR_+$ such that $\mu+\rho>1$ and $\rho\not\in\mez\xN$, 
there exists a positive constant $K$ such that for any $f\in H^\mu(\xR)$ and any function 
$v\in \eC{\rho}(\xR)\cap L^2(\xR)$ such that $\widehat{v}(\xi)=0$ for $\la \xi\ra\ge 1$, 
\begin{equation}\label{n242}
\blA \RE E^\sharp(v)f \brA_{H^{\mu+\rho-1}}+
\blA \RE E^\flat(v)f \brA_{H^{\mu+\rho-1}}\le 
K\lA v\rA_{\eC{\rho}}\lA f\rA_{H^\mu}.
\end{equation}
\end{prop}
\begin{remas}
Some technical remarks are in order. 
Had we instead obtained symbols $R^{\sharp,1}$, $R^{\sharp,2}$, $R^{\flat,1}$, $R^{\flat,2}$ in 
$SR^{1}_{\nu_1,0}$ for some $\nu_{1}>0$, then we would have obtained the bound
\begin{equation*}
\blA E^\sharp(v)f \brA_{H^{\mu+\rho-\nu_1-1}}
+\blA E^\flat(v)f \brA_{H^{\mu+\rho-\nu_1-1}}\le 
K \lA v\rA_{\eC{\rho}}\lA f\rA_{H^\mu},
\end{equation*}
that is, up to the harmless loss of $\nu_1$ derivative, the estimate~\eqref{n239} without the 
extra term 
$\lA \mathcal{H}v\rA_{\eC{\rho}}$. 
However we shall see that our symbols only belong to $SR^{1}_{0,0}$ (see~\eqref{n248}). 
For such symbols, in general, one cannot expect a better estimate than 
\eqref{n239}. 
For our purpose, it is crucial to have an estimate which involves only $\lA v\rA_{\eC{\rho}}$. 
To overcome this difficulty, the key point is that, on the one hand, the right-hand side of \eqref{n238} 
does not involve $\lA v\rA_{\eC{\rho}}\lA f\rA_{H^\mu}$ and on the other hand 
the estimates~\eqref{n239} and \eqref{n242} are sharp. The latter estimates 
will be used in the proof of Proposition~\ref{T65}. Finally 
an estimate analogous to \eqref{n238} for $E^\flat (v)f$ does not hold. 
We shall circumvent this by using the symmetry $S^\flat(v)f=S^\flat(f)v$, so that the estimate 
\eqref{n238.1} is enough for $E^\flat (v)f$. As already mentioned, this is the reason why it 
is convenient to split $S(v)f$ as the sum of $S^\sharp(v)f$ and $S^\flat(v)f$.
\end{remas}
\begin{proof}
The proof is divided into two parts. 
We first study $E^\sharp(v)$, then we study $E^\flat(v)$. 

\step{1}{Analysis of $E^\sharp(v)$}

Set $\zeta(\xip,\xii)=1-\theta(\xip,\xii)-\theta(\xii,\xip)$ \index{Pseudo-differential operators!$\zeta$, cut-off function}
where $\theta$ is the cutoff function used in the definition of paradifferential operators 
(see Definition~\ref{defi:theta}). Then
$$
\RBony(a,b)=\frac{1}{(2\pi)^2}\int
e^{ix(\xip+\xii)} \zeta(\xip,\xii)\widehat{a}(\xip)\widehat{b}(\xii) \,d\xip \, d\xii.
$$
Introduce
\begin{equation}\label{n243}
m^2_{11}(\xip,\xii)
= \la \xip\ra^{-\mez}\bigl( \la \xip+\xii\ra \la \xip\ra-(\xip+\xii)\xip\bigr) \zeta(\xip,\xii).
\end{equation}
Then 
$$
S^\sharp(v)f=\Op^\Bony[v^2,M^2]f\quad\text{with}\quad 
M^2=\begin{pmatrix} m^2_{11} & 0 \\ 0 & 0\end{pmatrix}.
$$
We seek $E^\sharp(v)f$ under the form $ \Op^\Bony[v^1,R^{\sharp,1}]f
+\Op^\Bony[v^2,R^{\sharp,2}]f$ satisfying \eqref{n236}. 
Denote by $r^k_{ij}$ the coefficients 
of the matrix $R^{\sharp,k}$. It follows from the proof of Proposition~\ref{T36} that, to 
solve~\eqref{n236}, it suffices to set 
$r^2_{11}=r^1_{12}=r^1_{21}=r^2_{22}=0$ and to solve
\begin{alignat*}{4}
-&|\xip+\xii|^\mez r^1_{11} ~&&+ |\xip|^\mez r^2_{21}
~&&+|\xii|^\mez r^1_{22} ~&&=0,\\
&|\xip+\xii|^\mez r^2_{21} &&-|\xip|^\mez r^1_{11} 
&&+|\xii|^\mez r^2_{12} &&=m^2_{11},\\
&|\xip+\xii|^\mez r^1_{22}&&+ |\xip|^\mez r^2_{12}
&&-|\xii|^\mez r^1_{11} &&=0,\\
-&|\xip+\xii|^\mez r^2_{12} &&-|\xip|^\mez r^1_{22}
&&-|\xii|^\mez r^2_{21}&&=0.
\end{alignat*}
As already seen in the proof of Proposition~\ref{T36}, we have 
\begin{equation*}
\begin{aligned}
r^2_{21}&=\frac{\delta}{D}|\xip+\xii|^\mez m^2_{11},\\
r^1_{22}&=\frac{2}{D}\la\xip\ra^\mez\la\xii\ra^\mez |\xip+\xii|^\mez m^2_{11},\\
r^2_{12}&=-\frac{2\la\xip\ra+\delta}{D}\la\xii\ra^\mez
 m^2_{11},\\
r^1_{11}&= \frac{\delta+2\la\xii\ra }{D} |\xip|^\mez m^2_{11}.
\end{aligned}
\end{equation*}
where 
$\delta\defn\la\xip+\xii\ra-\la\xip\ra-\la\xii\ra$ and 
$D\defn \delta^2-4\la \xip\ra\la \xii\ra$. 

Notice that on 
the support of $m^2_{11}$ we have $(\xip+\xii)\xip\le 0$ so that $\xip\xii\le 0$ and $\la \xip\ra\le \la\xii\ra$. 
Then $\la \xip+\xii\ra=\la \la\xii\ra-\la\xip\ra\ra=\la \xii\ra-\la\xip\ra$ and we have 
$$
\delta=-2\la\xip\ra\quad\text{and}\quad 
D=4\la \xip\ra\big( \la \xip\ra-\la\xii\ra\big)=
-4\la \xip\ra \la\xip+\xii\ra.
$$
This allows us to simplify the computations. It is found that
\begin{equation*}
\begin{aligned}
r^2_{21}
&=\frac{1}{2\la \xip+\xii\ra^\mez} m^2_{11},\\
r^{1}_{22}
&=-\frac{\la\xii\ra^\mez}{2 \la\xip\ra^\mez\la \xip+\xii\ra^\mez}
m^2_{11},\\
r^2_{12}
&=0,\\
r^1_{11}
&=-\frac{1}{2\la\xip\ra^\mez}m^2_{11},
\end{aligned}
\end{equation*}
so
\begin{equation*}
\begin{aligned}
r^2_{21}
&=\mez \la\xip\ra^\mez \la \xip+\xii\ra^\mez\left(1-\frac{\xip+\xii}{\la\xip+\xii\ra}\frac{\xip}{\la \xip\ra}\right)\zeta(\xip,\xii),\\
r^{1}_{22}
&=-\mez|\xii|^\mez\la \xip+\xii\ra^\mez\left(1-\frac{\xip+\xii}{\la\xip+\xii\ra}\frac{\xip}{\la \xip\ra}\right)\zeta(\xip,\xii),\\
r^2_{12}
&=0,\\
r^1_{11}&=-\mez\la \xip+\xii\ra\left(1-\frac{\xip+\xii}{\la\xip+\xii\ra}\frac{\xip}{\la \xip\ra}\right)\zeta(\xip,\xii).
\end{aligned}
\end{equation*}

We thus obtain the desired result \eqref{n236} with
\begin{equation}\label{n248}
\begin{aligned}
R^{\sharp,1}&\defn 
-\mez  \left(1-\frac{(\xip+\xii)}{|\xip+\xii|}\frac{\xip}{|\xip|}\right)\zeta(\xip,\xii) 
\begin{pmatrix} |\xip+\xii|  & 0 \\
0 & \la \xii\ra^\mez \la \xip+\xii\ra^\mez 
\end{pmatrix},\\
R^{\sharp,2}&\defn \phantom{-}\mez  \left(1-\frac{(\xip+\xii)}{|\xip+\xii|}\frac{\xip}{|\xip|}\right)\zeta(\xip,\xii) 
\begin{pmatrix} 0 & 0 \\
|\xip|^\mez|\xip+\xii|^\mez & 0 
\end{pmatrix}.
\end{aligned}
\end{equation}
Set $\mathcal{H}=-i\px  \Dx^{-1}$. Then $\Op^\Bony\bigl[v^1,R^{\sharp,1}\bigr]f$ is given by 
$$
\mez \begin{pmatrix}
-\Dx R_{\Bony}(v^1, f^1)+\mathcal{H}\Dx R_{\Bony}(\mathcal{H}v^1, f^1)\\[0.5ex]
-\Dxmez  R_{\Bony}(v^1,\Dxmez f^2)
+\mathcal{H}\Dxmez R_{\Bony}(\mathcal{H}v^1, \Dxmez f^2)
\end{pmatrix},
$$
and $\Op^\Bony\bigl[v^2,R^{\sharp,2}\bigr]f$ is given by
$$
\mez \begin{pmatrix}
0\\[0.5ex]
\Dxmez R_{\Bony}(\Dxmez v^2, f^1)-\mathcal{H}\Dxmez 
R_{\Bony}(\mathcal{H}\Dxmez v^2, f^1)
\end{pmatrix}.
$$
To prove~\eqref{n238} we have to estimate various terms of the form
$$
A_1 \Dx^a R_{\Bony}(A_2 \Dx^b V, \Dx^c F),\quad A_j\in \{ \mathcal{H},\id\},
\quad a+b+c=1,\quad c\in \{0,\mez\},\quad a,b\ge 0.
$$
Since $\RBony(a,b)=\RBony(b,a)$, 
the estimate~\eqref{Bony3} and the fact that $\mathcal{H}$ is bounded on 
Sobolev spaces imply that
\begin{equation}\label{n248.5}
\begin{aligned}
&\blA A_1 \Dx^a R_{\Bony}(A_2 \Dx^b V, \Dx^c F)
\brA_{H^{\mu+\rho-1}}\\
&\qquad \les \blA R_{\Bony}(A_2 \Dx^b V, \Dx^c F)
\brA_{H^{\mu+\rho-1+a}}\\
&\qquad \les \blA A_2 \Dx^b V\brA_{H^{\mu-1+a+c}}\blA \Dx^c F\brA_{\eC{\rho-c}}\\
&\qquad \les \lA V\rA_{H^{\mu}}\lA  F\rA_{\eC{\rho}}
\end{aligned}
\end{equation}
where we used \eqref{esti:Dxmez-Crho} in the third inequality. 
This proves that
\begin{align*}
\blA \Op^\Bony\bigl[v^1,R^{\sharp,1}\bigr]f \brA_{H^{\mu+\rho-1}}
&\les \lA v^1\rA_{H^\mu}\lA f\rA_{\eC{\rho}},\\
\blA \Op^\Bony\bigl[v^2,R^{\sharp,2}\bigr]f \brA_{H^{\mu+\rho-1}}
&\les \lA v^2\rA_{H^\mu}\lA f\rA_{\eC{\rho}},
\end{align*}
which imply \eqref{n238}. Similarly, we have
\begin{align}
\blA \Op^\Bony\bigl[v^1,R^{\sharp,1}\bigr]f \brA_{H^{\mu+\rho-1}}
&\les \left(\lA v^1\rA_{\eC{\rho}}+\lA \mathcal{H}v^1\rA_{\eC{\rho}}\right)
\lA f\rA_{H^\mu},\notag\\
\blA \Op^\Bony\bigl[v^2,R^{\sharp,2}\bigr]f \brA_{H^{\mu+\rho-1}}
&\les \lA v^2\rA_{\eC{\rho}}\lA f\rA_{H^\mu},\label{n249}
\end{align}
which proves \eqref{n239}.  

It remains to prove statement $iv)$. Notice that, since $D^*=-D$, \eqref{n236}Ê
implies
that 
$$
E^\sharp(Dv)^*+E^\sharp(v)^* D-D E^\sharp(v)^*=S^\sharp(v)^*.
$$
This and \eqref{n236} implies that $\RE E^\sharp(v)$ satisfies \eqref{n240}. 
We now have to prove that
\begin{equation}\label{n250}
\blA \RE E^\sharp(v)f \brA_{H^{\mu+\rho-1}}
\le 
K\lA v\rA_{\eC{\rho}}\lA f\rA_{H^\mu},
\end{equation}
provided that the Fourier transform of $v$ is supported in the unit ball. 
To do so we begin by noting that Lemma~\ref{T35} implies that
\begin{equation*}
\begin{aligned}
\RE E^\sharp(v)&=\Op^\Bony\bigl[v^1,R^{\sharp,1}(\xip,\xii)
+R^{\sharp,1}(-\xip,\xip+\xii)^T\bigr]\\
&\quad +\Op^\Bony\bigl[v^2,R^{\sharp,2}(\xip,\xii)+R^{\sharp,2}(-\xip,\xip+\xii)^T\bigr].
\end{aligned}
\end{equation*}
We begin by proving that 
\be\label{n251}
\blA \Op^\Bony\bigl[v^1,R^{\sharp,1}(\xip,\xii)
+R^{\sharp,1}(-\xip,\xip+\xii)^T\bigr]\brA_{\Fl{H^\mu}{H^{\mu+\rho-1}}}\les \lA v^1\rA_{\eC{\rho}}.
\ee
Below we use the following notation : given a scalar symbol $p=p(\xip,\xii)$ we denote by 
$\widetilde{p}$ the symbol defined by $\widetilde{p}(\xip,\xii)=p(-\xip,\xip+\xii)$. 

To prove \eqref{n251} we write 
$R^{\sharp,1}$ under the form 
$R^{\sharp,1}=\mez\zeta\left(\begin{smallmatrix} a & 0 \\ 0 & b\end{smallmatrix}\right)$. Then
$$
R^{\sharp,1}(\xip,\xii)+R^{\sharp,1}(-\xip,\xip+\xii)^T
=\mez \zeta\begin{pmatrix} a+\widetilde{a} & 0 \\ 0 & b+\widetilde{b}\end{pmatrix}
+\mez\bigl(\widetilde{\zeta}-\zeta\bigr)\begin{pmatrix} \widetilde{a} & 0 \\ 0 & \widetilde{b}\end{pmatrix},
$$
with
\begin{align}
a&=-\la \xip+\xii\ra\left(1-\frac{\xip+\xii}{\la\xip+\xii\ra}\frac{\xip}{\la \xip\ra}\right), \notag\\ 
\widetilde{a}&=
-\la \xii\ra\left(1+\frac{\xii}{\la \xii\ra}\frac{\xip}{\la \xip\ra}\right),
\label{n252}\\
b&=-|\xii|^\mez\la \xip+\xii\ra^\mez\left(1-\frac{\xip+\xii}{\la\xip+\xii\ra}\frac{\xip}{\la \xip\ra}\right),\notag \\
\widetilde{b}
&=-|\xip+\xii|^\mez\la \xii\ra^\mez\left(1+\frac{\xii}{\la\xii\ra}\frac{\xip}{\la \xip\ra}\right),\label{n253}
\end{align}
so that
\begin{align*}
a+\widetilde{a}&=-\la \xip+\xii\ra-\la\xii\ra+|\xip| ,\\
b+\widetilde{b}&=-2\la \xii\ra^\mez \la \xip+\xii\ra^\mez
+\la \xii\ra^\mez \la \xip+\xii\ra^\mez\left(\frac{\xip+\xii}{|\xip+\xii|}-\frac{\xii}{|\xii|}\right)\frac{\xip}{|\xip|}
.
\end{align*}
As above it follows from~\eqref{Bony3}, \eqref{esti:Dxpsiz0} and \eqref{esti:Dxmez-Crho} that
$$
\Big\Vert \Op^\Bony\Bigl[v^1,\mez \zeta\left(\begin{smallmatrix} 
-\la \xip+\xii\ra-\la\xii\ra+|\xip| & 0 \\ 
0 & -2\la \xii\ra^\mez \la \xip+\xii\ra^\mez \end{smallmatrix}\right)\Bigr]\Big\Vert
_{\Fl{H^\mu}{H^{\mu+\rho-1}}}\les \lA v^1\rA_{\eC{\rho}}.
$$
Set
\be\label{n254}
\beta(\xip,\xii)=\mez \zeta(\xip,\xii)\left(\frac{\xip+\xii}{|\xip+\xii|}-\frac{\xii}{|\xii|}\right)\frac{\xip}{|\xip|}
\la \xii\ra^\mez \la \xip+\xii\ra^\mez.
\ee
We have to prove that similarly
\be\label{n255}
\big\Vert \Op^\Bony\bigl[v^1, \beta\bigr]\Vert
_{\Fl{H^\mu}{H^{\mu+\rho-1}}}\les \lA v^1\rA_{\eC{\rho}}.
\ee
Notice that on the support of 
$$
\left(\frac{\xip+\xii}{|\xip+\xii|}-\frac{\xii}{|\xii|}\right)
$$
we have $|\xip|>|\xii|$. Introduce now $\Upsilon \in C^\infty(\xR^2\setminus\{0\})$, to be chosen later on, 
such that 
$\Upsilon(\xip,\xii)=1$ for $\la\xip\ra\ge \la\xii\ra$ and $\Upsilon(\xip,\xii)=0$ for $\la \xip\ra<\la\xii\ra/2$. 
Then
$$
\left(\frac{\xip+\xii}{|\xip+\xii|}-\frac{\xii}{|\xii|}\right)=\left(\frac{\xip+\xii}{|\xip+\xii|}-\frac{\xii}{|\xii|}\right)
\Upsilon(\xip,\xii)
$$
and we can decompose $\beta$ as 
\begin{align*}
\beta&=\frac{\xip+\xii}{|\xip+\xii|^\mez} \beta_1+\la \xip+\xii\ra^\mez\beta_2 \quad\text{where}\\
\beta_1&=\mez \zeta(\xip,\xii)\frac{\xip}{|\xip|}
\la \xip\ra^\uq \la \xii\ra^\uq\left( \frac{\la\xii\ra^{\uq}}{\la \xip\ra^{\uq}}\Upsilon(\xip,\xii)\right) ,\\
\beta_2&=-\mez \zeta(\xip,\xii)\frac{\xii}{|\xii|}\frac{\xip}{|\xip|}
\la \xip\ra^\uq \la \xii\ra^\uq\left( \frac{\la\xii\ra^{\uq}}{\la \xip\ra^{\uq}}\Upsilon(\xip,\xii)\right).
\end{align*}
Then $\Op^\Bony\bigl[v^1, \beta\bigr]=\mathcal{H}\Dxmez \Op^\Bony[ v^1,\beta_1]+
\Dxmez \Op^\Bony[v^1,\beta_2]$. We claim that $\Upsilon$ can be so chosen 
that $\beta_1\in SR^{0}_{1/4,1/4}$ and similarly $\beta_2\in SR^{0}_{1/4,1/4}$ 
so the result~\eqref{n255} follows from statement $ii)$ in Proposition~\ref{T32}. To do so 
we consider a function $\upsilon\in C^{\infty}(\xR)$ such that 
$\upsilon (t)=1$ for $| t |\le 1$ and $\upsilon(t)=0$ for $|t|\ge 2$. Then we set 
$\Upsilon(\xip,\xii)=\upsilon(\xii/\xip)$ and it is easily verified that
$$
\la \partial_{\xip}^\alpha\partial_{\xii}^\beta
\left(  \frac{\la\xii\ra^{\uq}}{\la \xip\ra^{\uq}}\Upsilon(\xip,\xii)\right)\ra\le C_{\alpha,\beta} \la \xip\ra^{ -\alpha}
\la \xii\ra^{-\beta}.
$$
This concludes the proof of \eqref{n255}.

To prove \eqref{n251} it remains only to prove that
\be\label{n256}
\Big\Vert \Op^\Bony\Bigl[v^1,\bigl(\widetilde{\zeta}-\zeta\bigr)
\left(\begin{smallmatrix}\widetilde{a} & 0 \\ 0 & \widetilde{b}\end{smallmatrix}\right)\Bigr]\Big\Vert
_{\Fl{H^\mu}{H^{\mu+\rho-1}}}\les \lA v^1\rA_{\eC{\rho}}.
\ee
Here we use our assumption on the spectrum of $v$ to write $v=\chi(D_x)v$ for some 
function $\chi$ in $C^\infty_0(\xR)$. Then 
$$
\Op^\Bony\Bigl[v^1,\bigl(\widetilde{\zeta}-\zeta\bigr)
\left(\begin{smallmatrix}\widetilde{a} & 0 \\ 0 & \widetilde{b}\end{smallmatrix}\right)\Bigr]
=\Op^\Bony\Bigl[v^1,\chi(\xip)\bigl(\widetilde{\zeta}-\zeta\bigr)
\left(\begin{smallmatrix}\widetilde{a} & 0 \\ 0 & \widetilde{b}\end{smallmatrix}\right)\Bigr]
$$
Since $\theta(-\xip,\xii)=\theta(\xip,\xii)=\theta(\xip,-\xii)$ we have
$$
\widetilde{\zeta}(\xip,\xii)=\zeta(-\xip,\xip+\xii)=\zeta(\xip,\xip+\xii)=\zeta(\xip,\xii)+\xip\zeta'(\xip,\xii),
$$
where
$\zeta'(\xip,\xii)=\int_0^1\partial_{\xii} \zeta(\xip, y \xip+\xii)\, dy $ is such that 
$\chi(\xip)\zeta'(\xip,\xii)$ 
belongs to the symbol class $SR^{-1}_{reg}$ introduced in the statement of Proposition~\ref{T32} 
(in fact this symbol belongs to $SR^{-\infty}_{reg}$ since it has compact support in $(\xip,\xii)$, which 
also insures that \e{n219}Ê
holds). 

Therefore directly from the definition \eqref{n252} of 
$\widetilde{a}$ we have $\chi(\xip)\bigl(\widetilde{\zeta}-\zeta\bigr)\widetilde{a}\in SR^0_{1,1}$. 
Statement $ii)$ in Proposition~\ref{T32} then implies that
$$
\blA \Op^\Bony\bigl[v^1,\chi(\xip)\bigl(\widetilde{\zeta}-\zeta\bigr)\widetilde{a}\bigr]\brA_{\Fl{H^\mu}{H^{\mu+\rho-1}}}
\les \lA \widetilde{\chi}(D_x)v^1\rA_{\eC{\rho+1}}\les \lA v^1\rA_{\eC{\rho}},
$$
where $\widetilde{\chi}\in C^\infty_0(\xR)$ is equal to one on the support of $\chi$.  
Similarly
$$
\Op^\Bony\bigl[\chi(D_x)v^1,\bigl(\widetilde{\zeta}-\zeta\bigr)\widetilde{b}\,\bigr]
=\Dxmez \Op^\Bony\bigl[v^1,b'\bigr]
$$ 
with 
$$
b'=-\mez \left(\xip\la \xii\ra^\mez + \la\xii\ra^\mez \frac{\xii}{\la\xii\ra}\la\xip\ra\right)
\chi(\xip)\zeta'(\xip,\xii)
\in SR^{-1}_{1,1/2}.
$$
Statement $ii)$ in Proposition~\ref{T32} implies that 
$\blA \Op^\Bony\bigl[v^1,b'\, \bigr]\brA_{\Fl{H^\mu}{H^{\mu+\rho-1/2}}}\les \lA v^1\rA_{\eC{\rho}}$. 
This proves \eqref{n256} and hence this completes the proof of \eqref{n251}.

To complete the proof of \eqref{n250} it remains to prove that  
$$
\blA \Op^\Bony\bigl[v^2,R^{\sharp,2}(\xip,\xii)+R^{\sharp,2}(-\xip,\xip+\xii)^T\bigr]\brA_{\Fl{H^\mu}{H^{\mu+\rho-2}}}
\les \lA v^2\rA_{\eC{\rho}}.
$$
In view of \eqref{n249}, to prove this estimate 
it is sufficient to prove that
\be\label{n257}
\blA \Op^\Bony\bigl[v^2,R^{\sharp,2}(-\xip,\xip+\xii)^T\bigr]\brA_{\Fl{H^\mu}{H^{\mu+\rho-2}}}\les \lA v^2\rA_{\eC{\rho}}.
\ee
Since
$$
R^{\sharp,2}(-\xip,\xip+\xii)^T=
\mez  \left(1+\frac{\xii}{\xii}\frac{\xip}{|\xip|}\right)\zeta(\xip,\xip+\xii) 
\begin{pmatrix} 0 & |\xip|^\mez|\xii|^\mez \\
0 & 0 
\end{pmatrix},
$$
and since $\chi(\xip)\zeta(\xip,\xip+\xii) $ has compact support, 
we have $\chi(\xip)R^{\sharp,2}(-\xip,\xip+\xii)^T\in SR^{0}_{1/2,1/2}$ so 
\eqref{n257} follows from Proposition~\ref{T32}.

\step{2}{Analysis of $E^\flat(v)$}

Introduce
\be\label{n258}
m^2_{22}(\xip,\xii)
=-\mez \frac{ \la \xip+\xii\ra^\mez}{\la \xip\ra^{\mez} \la \xii\ra^{\mez} }
\bigl( \la \xip\ra \la \xii\ra+\xip\xii\bigr)\zeta(\xip,\xii),
\ee
so that $S^\flat(v)f=\Op^\Bony[v^2,M^2]$ with 
$M^2=\begin{pmatrix} 0 & 0 \\ 0 & m^2_{22}\end{pmatrix}$. 

We seek $E^\flat(v)f$ under the form 
$\Op^\Bony[v^2,R^{\flat,1}]f+\Op^\Bony[v^2,R^{\flat,2}]f$ 
satisfying \eqref{n237}. 
We still denote by $r^k_{ij}$ the coefficients 
of the matrix $R^{\flat,k}$. Again, 
it follows from the proof of Proposition~\ref{T36} that, to 
solve~\eqref{n236}, it suffices to set 
$r^2_{11}=r^1_{12}=r^1_{21}=r^2_{22}=0$ and to solve
\begin{alignat*}{4}
-&|\xip+\xii|^\mez r^1_{11} ~&&+ |\xip|^\mez r^2_{21}
~&&+|\xii|^\mez r^1_{22} ~&&=0,\\
&|\xip+\xii|^\mez r^2_{21} &&-|\xip|^\mez r^1_{11} 
&&+|\xii|^\mez r^2_{12} &&=0,\\
&|\xip+\xii|^\mez r^1_{22}&&+ |\xip|^\mez r^2_{12}
&&-|\xii|^\mez r^1_{11} &&=0,\\
-&|\xip+\xii|^\mez r^2_{12} &&-|\xip|^\mez r^1_{22}
&&-|\xii|^\mez r^2_{21}&&=m^2_{22}.
\end{alignat*}
As already seen in the proof of Proposition~\ref{T36}, we have 
\begin{align*}
r^2_{21}&=\frac{\delta}{D}\la\xii\ra^\mez m^2_{22}
+\frac{2}{D}\la\xip\ra\la\xii\ra^\mez
 m^2_{22},\\
r^1_{22}&=\frac{\delta}{D}\la\xip\ra^\mez m^2_{22}
+\frac{2}{D}\la\xip\ra^\mez\la\xii\ra m^2_{22},\\
r^2_{12}&=-\frac{1}{|\xip+\xii|^\mez}
\bigl( |\xip|^\mez r^1_{22}+|\xii|^\mez r^2_{21}+m^2_{22}\bigr),\\
r^1_{11}&=\frac{1}{|\xip+\xii|^\mez}
\bigl( |\xip|^\mez r^2_{21}+|\xii|^\mez r^1_{22}\bigr),
\end{align*}
where 
$\delta\defn\la\xip+\xii\ra-\la\xip\ra-\la\xii\ra$ and 
$D\defn \delta^2-4\la \xip\ra\la \xii\ra$. 

Consequently,
\begin{align*}
r^{2}_{21}
&=\frac{(\delta+2\la\xip\ra)\la\xii\ra^\mez}{D}m^2_{22},\\
r^1_{22}
&=\frac{(\delta+2\la\xii\ra)\la\xip\ra^\mez}{D}m^2_{22},\\
r^2_{12}
&=-\frac{\delta\la\xip+\xii\ra^\mez}{D}m^2_{22},\\
r^1_{11}
&=\frac{\la \xip\ra^\mez\la\xii\ra^\mez}{\la\xip+\xii\ra^\mez}
\frac{2\delta+2\la\xip\ra+2\la\xii\ra}{D}m^2_{22}.
\end{align*}
On the support of $m^2_{22}$ there holds 
$\xip\xii>0$ and we have $\delta=0$ and $D=-4\la \xip\ra \la \xii\ra$. 
Therefore
\begin{alignat*}{2}
r^1_{11}
&=-\frac{\la\xip+\xii\ra^\mez}{2\la\xip\ra^\mez\la\xii\ra^\mez}m^2_{22},
\qquad && r^1_{22}=-\frac{1}{2\la\xip\ra^\mez}m^2_{22},\\
r^2_{12}
&=0, && r^{2}_{21}=-\frac{1}{2\la\xii\ra^\mez}m^2_{22},
\end{alignat*}
We next give a simplified expression for $m^2_{22}$ based on the identity
$$
\frac{\la\xip\ra\la\xii\ra+\xip\xii}{\la \xip\ra\la\xii\ra}
=(\sign (\xip)+\sign(\xii))\sign(\xip+\xii) 
= \left(\frac{\xip}{\la\xip\ra}+\frac{\xii}{\la\xii\ra}\right)\frac{\xip+\xii}{\la\xip+\xii\ra}\cdot
$$
Then, by definition of $m^2_{22}$ (cf.\ \eqref{n258}), 
we have
$$
m^2_{22}=-\mez \la \xip+\xii\ra^\mez \la \xip\ra^{\mez} \la \xii\ra^{\mez}
\left(\frac{\xip}{\la\xip\ra}+\frac{\xii}{\la\xii\ra}\right)\frac{\xip+\xii}{\la\xip+\xii\ra} \zeta(\xip,\xii).
$$
Therefore
\be\label{n258.3}
\begin{aligned}
r^1_{11}
&=\uq \left(\frac{\xip}{\la\xip\ra}+\frac{\xii}{\la\xii\ra}\right)(\xip+\xii) \zeta,
\\ 
r^1_{22}
&=\uq 
\la \xip+\xii\ra^\mez \la \xii\ra^{\mez}
\left(\frac{\xip}{\la\xip\ra}+\frac{\xii}{\la\xii\ra}\right)\frac{\xip+\xii}{\la\xip+\xii\ra}
\zeta,\\
r^2_{12}
&=0, \\
r^{2}_{21}&=\uq \la \xip+\xii\ra^\mez \la \xip\ra^{\mez}
\left(\frac{\xip}{\la\xip\ra}+\frac{\xii}{\la\xii\ra}\right)\frac{\xip+\xii}{\la\xip+\xii\ra} \zeta,
\end{aligned}
\ee
Then $\Op^\Bony\bigl[v^1,R^{\flat,1}\bigr]f$ is given by 
$$
\uq \begin{pmatrix}
\mathcal{H}\Dx R_{\Bony}(\mathcal{H}v^1,f^1)
+\mathcal{H}\Dx R_{\Bony}(v^1, \mathcal{H}f^1)\\[0.5ex]
\mathcal{H}\Dxmez  R_{\Bony}(\mathcal{H} v^1, \Dxmez f^2)
+\mathcal{H}\Dxmez R_{\Bony}( v^1,\mathcal{H}\Dxmez f^2)
\end{pmatrix}
$$
and $\Op^\Bony\bigl[v^2,R^{\flat,2}\bigr]f$ is given by 
$$
\uq \begin{pmatrix}
0\\[0.5ex]
\mathcal{H}\Dxmez  R_{\Bony}(\mathcal{H}\Dxmez  v^2, f^1)
+\mathcal{H}\Dxmez R_{\Bony}(\Dxmez v^2,\mathcal{H}f^1)
\end{pmatrix}.
$$
Then it follows from~\eqref{Bony3} and 
\eqref{esti:Dxmez-Crho-b} that
\begin{align*}
&\blA E^\flat(v)f \brA_{H^{\mu+\rho-1}}\le 
K\bigl(\lA v\rA_{\eC{\rho}}+\lA \mathcal{H}v\rA_{\eC{\rho}}\bigr)\lA f\rA_{H^\mu},\\
&\blA E^\flat(v)f \brA_{H^{\mu+\rho-1}}\le 
K\bigl(\lA f\rA_{\eC{\rho}}+\lA \mathcal{H}f\rA_{\eC{\rho}}\bigr)\lA v\rA_{H^\mu}.
\end{align*}

It remains to prove statement $iv)$. As in the previous step, since 
$D^*=-D$, \eqref{n237}Êimplies
that 
$$
E^\flat(Dv)^*+E^\flat(v)^* D-D E^\flat(v)^*
=S^\flat(v)^*.
$$
This and \eqref{n237} implies that $\RE E^\flat(v)$ satisfies \eqref{n241}. 
We now have to prove
\begin{equation*}
\blA \RE E^\flat(v)f \brA_{H^{\mu+\rho-1}}\le 
K\lA v\rA_{\eC{\rho}}\lA f\rA_{H^\mu}.
\end{equation*}
Again, Lemma~\ref{T35} implies that
\begin{equation*}
\begin{aligned}
\RE E^\flat(v)&=\Op^\Bony\bigl[v^1,R^{\flat,1}(\xip,\xii)
+R^{\flat,1}(-\xip,\xip+\xii)^T\bigr]\\
&\quad +\Op^\Bony\bigl[v^2,R^{\flat,2}(\xip,\xii)+R^{\flat,2}(-\xip,\xip+\xii)^T\bigr].
\end{aligned}
\end{equation*}
The $\Fl{H^\mu}{H^{\mu+\rho-1}}$-norm of $\Op^\Bony\bigl[v^1,R^{\flat,1}(\xip,\xii)
+R^{\flat,1}(-\xip,\xip+\xii)^T\bigr]$ is estimated from the fact that
$$
R^{\flat,1}(\xip,\xii)+R^{\flat,1}(-\xip,\xip+\xii)
=\begin{pmatrix}a & 0\\ 0 & b\end{pmatrix}
$$
with
\begin{align}
a&=\uq \zeta \left(|\xip|+|\xii|+\xip\frac{\xii}{|\xii|}\right)
-\uq (\widetilde{\zeta}-\zeta)\xii\frac{\xip}{|\xip|}
+\uq \widetilde{\zeta}\xii \frac{\xip+\xii}{|\xip+\xii|},\label{n258.1}\\
b&=\mez \beta
+\uq \bigl( \zeta+\widetilde{\zeta}\bigr)
\frac{\xip+\xii}{|\xip+\xii|}
\frac{\xii}{|\xii|}|\xii|^\mez |\xip+\xii|^\mez,\label{n258.2}
\end{align}
where $\beta$ is given by \e{n254}. 
We estimate the $\Fl{H^\mu}{H^{\mu+\rho-1}}$-norms of 
$\Op^\Bony[v^1,a]$ and $\Op^\Bony[v^1,b]$ separately. 

Let us estimate the $\Fl{H^\mu}{H^{\mu+\rho-1}}$-norm of 
$\Op^\Bony[v^1,a]$. 
To do so it is convenient to rewrite the third term in the right hand side of \eqref{n258.1} as 
$$
\widetilde{\zeta}\xii \frac{\xip+\xii}{|\xip+\xii|}=(\widetilde{\zeta}-\zeta)\xii \frac{\xip+\xii}{|\xip+\xii|}+
\zeta\xii \frac{\xip+\xii}{|\xip+\xii|},
$$
so that $a=a_1+a_2$ with
\begin{align*}
a_1&\defn \uq \zeta \left(|\xip|+|\xii|+\xip\frac{\xii}{|\xii|}+\xii \frac{\xip+\xii}{|\xip+\xii|}\right),\\
a_2&\defn -\uq (\widetilde{\zeta}-\zeta)\xii\frac{\xip}{|\xip|}
+\uq(\widetilde{\zeta}-\zeta)\xii \frac{\xip+\xii}{|\xip+\xii|}.
\end{align*}
We begin by estimating the contribution due to $a_1$. 
To do so we notice that
\begin{multline*}
\Op^\Bony\left[v^1,\zeta \left(|\xip|+|\xii|+\xip\frac{\xii}{|\xii|}\right)+\xii \frac{\xip+\xii}{|\xip+\xii|}\right]
f\\
=\RBony(\Dx v^1,f)+\RBony(v^1,\Dx f)
+\RBony(D_x v^1,\mathcal{H}f)
+\mathcal{H}\RBony(v^1,D_x f),
\end{multline*}
where $D_x=-i\px$, 
and then we use arguments similar to those used to prove~\eqref{n248.5}. 
To estimate the contribution due to $a_2$, notice that we have already seen that 
$\chi(\xip)(\widetilde{\zeta}-\zeta)$ belongs to $SR^{-1}_{1,0}$ so that
$$
\chi(\xip)(\widetilde{\zeta}-\zeta)\xii\frac{\xip}{|\xip|}\in SR^{-1}_{1,1},\quad 
\chi(\xip)(\widetilde{\zeta}-\zeta)\xii\in SR^{-1}_{1,1}
$$
and hence one may apply the arguments used to prove \eqref{n256}. 

One can estimate the $\Fl{H^\mu}{H^{\mu+\rho-1}}$-norm of 
$\Op^\Bony[v^1,b]$ in a similar way (using \eqref{n255} to estimate the contribution due to 
$\mez \beta$).

The $\Fl{H^\mu}{H^{\mu+\rho-1}}$-norm 
of $\Op^\Bony[v^2,R^{\flat,2}(\xip,\xii)+R^{\flat,2}(-\xip,\xip+\xii)^T]$ is estimated by similar arguments.
\end{proof}

We need also the following variant of Proposition~\ref{T37}.

\begin{prop}\label{T37-add}
Consider a real number $\beta$ in $[0,\infty[$. 
There exist four matrices of symbols $R_\beta^{\sharp,1}$, $R_\beta^{\sharp,2}$, 
$R_\beta^{\flat,1}$, $R_\beta^{\flat,2}$ in 
$SR^{1}_{0,0}$ such that 
the following properties hold.

$i)$ Let $(\mu,\rho)\in \xR\times \xR_+$ with $\mu+\rho>1$. 
The bilinear operators given by 
\begin{align*}
&(v,f)\mapsto E_\beta^\sharp(v)f= \Op^\Bony\big[v^1,R_\beta^{\sharp,1}\big]f
+\Op^\Bony\big[v^2,R_\beta^{\sharp,2}\big]f, \\
&(v,f)\mapsto E_\beta^\flat(v)f= \Op^\flat\big[v^1,R_\beta^{\flat,1}\big]f
+\Op^\Bony\big[v^2,R_\beta^{\flat,2}\big]f,
\end{align*}
are well-defined for any $(v,f)$ in $(\eC{\rho}\cap L^2(\xR))\times H^\mu(\xR)$ or in 
$H^{\mu}(\xR)\times \eC{\rho}(\xR)$. 

$ii)$ There holds
\begin{align}
&E_\beta^\sharp(Dv)+E_\beta^\sharp(v)D-D E_\beta^\sharp(v)
=\mathfrak{S}_\beta^\sharp(v),\label{add-1}\\
&E_\beta^\flat(Dv)+E_\beta^\flat(v)D-D E_\beta^\flat(v)=\mathfrak{S}_\beta^\flat(v),\label{add-2}
\end{align}
where $\mathfrak{S}_\beta^\sharp$ and $\mathfrak{S}_\beta^\flat$ are such that
\begin{align}
&\RE \langle S^\sharp(v) f-\mathfrak{S}_\beta^\sharp(v)f,f\rangle_{H^\beta\times H^\beta}=0,
\label{add-3}
\\
&\RE \langle S^\flat(v) f-\mathfrak{S}_\beta^\flat(v)f,f\rangle_{H^\beta\times H^\beta}=0,\label{add-4}
\end{align}
for any 
$f\in H^{\beta}(\xR)^2$, and satisfy
\begin{align}
&\blA \mathfrak{S}_\beta^\sharp(v)\brA_{\Fl{H^{\mu}}{H^{\mu+\rho-1}}}
\le K \lA v \rA_{\eC{\rho}},\label{add-5}\\
&\blA \mathfrak{S}_\beta^\flat(v)\brA_{\Fl{H^{\mu}}{H^{\mu+\rho-1}}}
\le K \lA v \rA_{\eC{\rho}}\label{add-6}.
\end{align}
$iii)$ The following estimates hold. 
For all $(\mu,\rho)\in \xR\times \xR_+$ such that $\mu+\rho>1$ and $\rho\not\in\mez\xN$, 
there exists a positive constant $K$ such that
\begin{align}
&\blA E_\beta^\sharp(v)f \brA_{H^{\mu+\rho-1}}\le 
K\lA v\rA_{\eC{\rho}}\lA f\rA_{H^\mu},\label{add-7}\\
&\blA E_\beta^\flat(v)f \brA_{H^{\mu+\rho-1}}\le 
K\lA v\rA_{\eC{\rho}}\lA f\rA_{H^\mu}.\label{add-8}
\end{align}
\end{prop}

\begin{proof}
We begin by studying $E_\beta^\sharp(v)$ under the additional assumption that 
$\widehat{v}(\xi)=0$ for $\la\xi\ra\ge 1$. 
We have $S^\sharp(v)=\Op^\Bony[v^2,M^2]f$ with 
$M^2=\begin{pmatrix} m^2_{11} & 0 \\ 0 & 0\end{pmatrix}$ 
where $m^2_{11}$ is given by \eqref{n243}. 
Introduce the following weight
$$
w(\xip,\xii)=\frac{\langle \xip +\xii\rangle^{2\beta}}{\langle \xip +\xii\rangle^{2\beta}
+\langle \xii\rangle^{2\beta}}
$$
and set
$$
\mathfrak{M}(\xip,\xii)\defn w(\xip,\xii) M^2(\xip,\xii)
+ w(-\xip,\xip+\xii)M^2(-\xip,\xip+\xii)^T,
$$
so that $\mathfrak{S}_\beta^\sharp(v)
=\Op^\Bony[v^2,\mathfrak{M}]$ satisfies \eqref{add-3}. Let us prove 
the estimate \eqref{add-5}. To do so introduce $R_w(v,f)=\Op^\Bony[v,w\zeta]f$ where $\zeta$ is the cut-off function 
$1-\theta(\xip,\xii)-\theta(\xii,\xip)$. Then Proposition~\ref{T32} implies that 
$R_w(v,f)$ satisfies the same estimates as $\RBony(v,f)$ does. Now $\Op^\Bony[v^2,wM^2]$ is given by
$$
\begin{pmatrix}
\Dx R_w(\Dxmez v^2, f^1)
+\px R_w(\px\Dx^{-\mez}v^2,f^1)\\
0
\end{pmatrix}
$$
and hence $\Op^\Bony[v^2,wM^2]$ satisfies the same estimate as $S(v)$ does. 
Proceeding similarly, one estimates $\Op^\Bony[v^2,w(-\xip,\xip+\xii)M^2(-\xip,\xip+\xii)^T\bigr]$ 
which completes the proof of \eqref{add-5}.

Then to solve \eqref{add-1} it is sufficient to seek $\mathfrak{E}(v)$ such that
\begin{equation}\label{n266.1}
\mathfrak{E}(Dv)+\mathfrak{E}(v)D-D \mathfrak{E}(v)
=\Op^\Bony\left[v^2,w M^2\right]
\end{equation}
and then to set $E_\beta^\sharp(v)=\mathfrak{E}(v)+\bigl(\mathfrak{E}(v)\bigr)^*$. 
Now we recall that $E^\sharp(v)=\Op^\Bony[v^1,R^{\sharp,1}]+
\Op^\Bony[v^2,R^{\sharp,2}]$, 
as given by Proposition~\ref{T37}, solves
\begin{equation*}
E^\sharp(Dv)+E^\sharp(v)D-D E^\sharp(v)=\Op^\Bony[v^2,M^2].
\end{equation*}
Therefore
$$
\mathfrak{E}(v)\defn \Op^\Bony\left[v^1,wR^{\sharp,1}\right]+
\Op^\Bony\left[v^2,wR^{\sharp,2}\right].
$$
satisfies \eqref{add-1}. Therefore one obtains the desired result with 
$E_\beta^\sharp(v)=\Op^\Bony[v^1,R_\beta^{\sharp,1}]+
\Op^\Bony[v^2,R_\beta^{\sharp,2}]$ where
$$
R_\beta^{\sharp,k}(\xip,\xii)= w(\xip,\xii) R^{\sharp,k}(\xip,\xii)+w(-\xip,\xip+\xii) R^{\sharp,k}(-\xip,\xip+\xii)^T.
$$
We have symbols of exactly the same form as those found in the proof of Proposition~\ref{T37} 
except that the cut-off function $\zeta$ is replaced with $w\zeta$. Thus $E_\beta^\sharp(v)$ 
satisfies the same estimates as $\RE E^\sharp(v)$ does. In particular, for any 
function $\chi$ in $C^\infty_0(\xR)$ such that $\chi(\xi)=0$ for $\la\xi\ra\ge 1/2$, there holds
$$
\blA E_\beta^\sharp\bigl(\chi(D_x)v\bigr)f \brA_{H^{\mu+\rho-1}}\le 
K\lA v\rA_{\eC{\rho}}\lA f\rA_{H^\mu}.
$$

This completes the analysis of $E_\beta^\sharp(v)$ in the case when the spectrum 
of $v$ is contained in the unit ball. Now consider a general function $v\in \eC{\rho}(\xR)\cap L^2(\xR)$. 
Introduce a function $\chi$ in $C^\infty_0(\xR)$ such that $\chi(\xi)=0$ for $\la\xi\ra\ge 1/2$ and $\chi(\xi)=1$ on a neighborhood of the origin. We then set
$$
E_\beta^\sharp(v)=E^\sharp\bigl((1-\chi(D_x))v\bigr)+E_\beta^\sharp\bigl(\chi(D_x)v\bigr),
$$
where $E_\beta^\sharp\bigl(\chi(D_x)v\bigr)$ is as given by the previous step and where 
$E^\sharp$ is given by Proposition~\ref{T37}. It follows from \e{n236} and the previous analysis 
that \e{add-1} and \e{add-3} are satisfied. On the other hand, \e{n239}Ê
and the fact that the 
$(1-\chi(D_x))\mathcal{H}$ is bounded on H\"older spaces $\eC{\rho}$ (with $\rho\not\in\xN$) imply that
$$
\blA E_\beta^\sharp\bigl((1-\chi(D_x))v\bigr)f \brA_{H^{\mu+\rho-1}}\le 
K\lA v\rA_{\eC{\rho}}\lA f\rA_{H^\mu}.
$$
We thus obtain \e{add-7} by combining the two previous inequalities.

The analysis of $E_\beta^\flat(v)$ is similar.
\end{proof}

\section{System for the new unknown}

Recall that
\begin{equation}\label{n259}
\md{\partial_t\vU+DU+Q(\vu)\vU+C(\Ur)\vU+S(\vu)\vU}{0}{s}.
\end{equation}
As explained above, our first task is to 
prove that there exists an operator of order $0$, denoted $B(v)$, such that
\begin{equation*}
\RE \langle Q(v)f-B(v)f,f\rangle_{H^s\times H^s}=0,
\end{equation*}
where $\langle\cdot,\cdot\rangle_{H^s\times H^s}$ denotes 
the scalar product in $H^s(\xR)^2$. 

\begin{lemm}\label{T38}
There exists~$B^1\in S^{0,0}_2$ and~$B^2\in S^{0,0}_{3/2}$ such that 
for all $v=(v^1,v^2)\in C^\rho(\xR)^2$
$$
B(v)\defn \Op^{\Bony}[v^1,B^1]+\Op^{\Bony}[v^2,B^2],
$$
satisfies~$B(v)=B(v)^*$ and 
$\RE \langle Q(v)f-B(v)f,f\rangle_{H^s\times H^s}=0$ for any 
$f\in H^{s+1}(\xR)^2$.
\end{lemm}
\begin{proof}
Write
\begin{align*}
&2\RE \langle Q(v)f-B(v)f,f\rangle_{H^s\times H^s}\\
&\qquad
=2\RE \langle \Lambda^s (Q(v)f-B(v)f),\Lambda^s f\rangle_{L^2\times L^2}\\
&\qquad
=\langle \Lambda^s (Q(v)f-B(v)f),\Lambda^s f\rangle_{L^2\times L^2}+
\langle \Lambda^s f,\Lambda^s (Q(v)f-B(v)f)\rangle_{L^2\times L^2}\\
&\qquad=\langle \bigr(\Lambda^{2s} Q(v)+Q(v)^*\Lambda^{2s}\bigr)f,f\rangle_{L^2\times L^2}
-\langle \bigr(\Lambda^{2s} B(v)
+B(v)^*\Lambda^{2s}\bigr)f,f\rangle_{L^2\times L^2}
\end{align*}
where $\Lambda=(\id-\Delta)^{1/2}$. Since we seek~$B(v)$ such that~$B(v)=B(v)^*$, this means that we have 
to solve
\begin{equation}\label{n260}
\Lambda^{2s} B(v)+B(v)\Lambda^{2s}=\Lambda^{2s} Q(v)+Q(v)^*\Lambda^{2s}.
\end{equation}
We first rewrite~$Q(v)f$ 
as~$\Op^{\Bony}[v^1,Q^1]f+\Op^{\Bony}[v^2,Q^2]f$. 
Recall from \e{n210} that
$$
Q(\vu)\vU=\begin{pmatrix}
T_{\px\Dx^{-\mez}\vu^2}\partial_x \vU^1-\mez T_{\Dx^\tdm \vu^2}\vU^1- T_{-\mez\Dx\vu^1}\Dxmez \vU^2\\
\Dxmez T_{\px\Dx^{-\mez}\vu^2\la\xi\ra^{-1/2}}\px \vU^2
+ \Dxmez T_{-\mez\Dx\vu^1}\vU^1
\end{pmatrix}.
$$
Then set
\begin{equation}\label{n261}
\begin{aligned}
Q^1&=\mez \la\xip\ra \theta(\xip,\xii)\begin{pmatrix} 0 & \la\xii\ra^{\mez} \\
-\la \xip+\xii\ra^{\mez} & 0 \end{pmatrix},\\
Q^2&=\xip \la \xip\ra^{-\mez}\theta(\xip,\xii)
\begin{pmatrix} -\xii -\mez\xip& 0 \\ 0 & -\la \xip+\xii\ra^\mez \xii\la \xii\ra^{-\mez}
\end{pmatrix},
\end{aligned}
\end{equation}
where~$\theta$ is given by Definition~\ref{defi:theta}. 
We have
\begin{align*}
\Op^{\Bony}[v^1,Q^1]f&=
\begin{pmatrix}
0 & - T_{-\mez\Dx v^1}\Dxmez f^2\\
\Dxmez T_{-\mez\Dx v^1}f^1 & 0
\end{pmatrix},\\
\Op^{\Bony}[v^2,Q^2]f&=
\begin{pmatrix} T_{\px\Dx^{-\mezl}v^2} \px f^1 
-\mez T_{\Dx^\tdm \vu^2}f^1& 0 \\
0 & \Dx^{\mezl} T_{(\px\Dx^{-\mezl}v^2)\la \xi\ra^{-1/2}}\px f^2\end{pmatrix}.
\end{align*}

Then~$Q^1\in S^{1/2,0}_{1}$ and
$$
\Lambda^{2s} \Op^{\Bony}[v^1,Q^1]
+\left(\Op^{\Bony}[v^1,Q^1]\right)^* \Lambda^{2s} 
=\Op^{\Bony}[v^1,\mathcal{Q}^1] 
$$
where~$\mathcal{Q}^1$ is given by (see~Lemma~\ref{T35})
\begin{equation*}
\begin{aligned}
\mathcal{Q}^1(\xip,\xii)&=\langle\xip+\xii\rangle^{2s}Q^1(\xip,\xii)
+\langle\xii\rangle^{2s}Q^1(-\xip,\xip+\xii)^{T}
\\
&=\mez \langle\xip+\xii\rangle^{2s} \la\xip\ra \theta(\xip,\xii)
\begin{pmatrix} 0 & \la\xii\ra^{\mezl} \\
-\la \xip+\xii\ra^{\mezl} & 0 \end{pmatrix}\\
&\quad +\mez \langle\xii\rangle^{2s}\la\xip\ra\theta(-\xip,\xip+\xii)
\begin{pmatrix} 0 &   -\la \xii\ra^\mez\\ 
\la\xip+\xii\ra^{1/2} &0 
\end{pmatrix}.
\end{aligned}
\end{equation*}
Since $\theta$ is even in $\xip$ (by assumption~\eqref{sym:chipsi}) and since
$$
\theta(\xip,\xip+\xii)=\theta(\xip,\xii)+\xip\int_0^1 \frac{\partial \theta}{\partial \xii}(\xip,\xii+y\xip)\, dy, 
\quad\text{and}\quad 
 \frac{\partial \theta}{\partial \xii} \in S^{-1,0}_0,
$$
we obtain that~$\mathcal{Q}^1\in S^{2s-1/2,0}_{2}$.

Similarly~$Q^2\in S^{1,0}_{1/2}$ and 
$
\Lambda^{2s} \Op^{\Bony}[v^2,Q^2]
+\left(\Op^{\Bony}[v^2,C_2]\right)^*\Lambda^{2s} 
=\Op^{\Bony}[v^2,\mathcal{Q}^2]$ 
where~$\mathcal{Q}^2\in S^{2s,0}_{3/2}$ is given by
\begin{equation*}
\begin{aligned}
\mathcal{Q}^2&=\langle\xip+\xii\rangle^{2s}Q^2(\xip,\xii)
+\langle\xii\rangle^{2s}Q^2(-\xip,\xip+\xii)^{T}\\
&=\langle\xip+\xii\rangle^{2s}\xip \la \xip\ra^{-\mez}\theta(\xip,\xii)
\begin{pmatrix} -\xii -\mez \xip& 0 \\ 0 & -\la \xip+\xii\ra^\mez \xii\la \xii\ra^{-\mez}
\end{pmatrix}\\
&\quad+\langle\xii\rangle^{2s}\xip \la \xip\ra^{-\mez}\theta(-\xip,\xip+\xii)
\begin{pmatrix} \mez \xip+\xii & 0 \\ 
0 & \la \xii\ra^\mez (\xip+\xii)\la \xip+\xii\ra^{-\mez}
\end{pmatrix}.
\end{aligned}
\end{equation*}

Now set
\be\label{n261.1}
B^1\defn \frac{1}{\langle \xip +\xii\rangle^{2s}+\langle \xii\rangle^{2s}}
\mathcal{Q}^1\in S^{-1/2,0}_2
\ee
and
\be\label{n261.2}
B^2\defn \frac{1}{\langle \xip +\xii\rangle^{2s}+\langle \xii\rangle^{2s}}
\mathcal{Q}^2\in S^{0,0}_{3/2}.
\ee
Then~$B(v)$ solves \eqref{n260}. Moreover, 
since~$\overline{\mathcal{Q}^k(-\xip,\xip+\xii)^{T}}
=\mathcal{Q}^k(\xip,\xii)^{T}$ for~$k=1,2$ and since 
$\langle \xip +\xii\rangle^{2s}+\langle \xii\rangle^{2s}=\langle (-\xip) 
+(\xip+\xii)\rangle^{2s}+\langle \xip+ \xii\rangle^{2s}$ we check that~$\Op^{\Bony}[v^1,B^1]$ 
and~$\Op^{\Bony}[v^2,B^2]$ are self-adjoint, so is~$B(v)$. 
\end{proof}

We next study the equation
$$
E_A(Dv)+E_A(v)D-D E_A(v)=-B(v)
$$
where $B(v)$ is given by the previous lemma.

\begin{lemm}\label{T40}
There exist 
$A^1,A^2$ in $S^{0,1/2}_{1}$ such that, for 
all $v\in \eC{3}\cap L^2(\xR)$ the operator 
$E_{A}(v)=\Op^\Bony\big[v^1,A^1\big]+\Op^\Bony\big[v^2,A^2\big]$ satisfies
\begin{equation}\label{n262}
E_A(Dv)+E_A(v)D-D E_A(v)=-B(v)
\end{equation}
and such that the following properties hold.

$i)$ Let $\mu$ be a given real number. There exists $K>0$ such that, 
for any scalar function $w\in \eC{1}(\xR)$, any $v=(v^1,v^2)\in \eC{3}\cap L^2(\xR)$ 
and any $f=(f^1,f^2)\in H^\mu(\xR)$,
$$
\blA \left[ T_w I_2, E_{A}(v)\right] f\brA_{H^{\mu+1}} 
\le K \lA w\rA_{\eC{1}}\lA v\rA_{\eC{3}}\lA f\rA_{H^\mu},
$$
where $I_2=\left(\begin{smallmatrix} 1 & 0 \\ 0 & 1\end{smallmatrix}\right)$.

$ii)$ Let $\mu$ be a given real number. 
There exists $K>0$ such that, 
for any $v=(v^1,v^2)\in \eC{3}\cap L^2(\xR)$ 
and any $f=(f^1,f^2)\in H^\mu(\xR)$,
\begin{equation}\label{n263}
\blA E_{A}(v)f\brA_{H^{\mu}}
\le K \lA v\rA_{\eC{3}}\lA f\rA_{H^{\mu}}.
\end{equation}
\end{lemm}
\begin{proof}
Since $B^1\in S^{0,0}_2\subset S^{0,1/2}_{3/2}$ 
and~$B^2\in S^{0,0}_{3/2}\subset S^{0,1/2}_{3/2}$, 
the fact that there exist $A_1$ and~$A_2$ in $S^{0,1/2}_{1}$ 
such that $E_{A}(v)$ satisfies \eqref{n262} 
follows from Proposition~\ref{T36}. Now, Lemma~\ref{T34} applied with $\rho=3-\epsilon$ 
(with $\epsilon\in \pol 0,1/2 \por$) 
implies that, if $v\in \eC{\rho}(\xR)\cap L^2(\xR)$, modulo 
a smoothing operator, $E_A(\vu)$ is a paradifferential operator 
whose matrix-valued symbol $a$, given by \eqref{n223}, 
has semi-norms in $\Gamma^{0}_{\rho-3/2}$ estimated by statement $(ii)$ in 
Lemma~\ref{T33}: 
this means that $E_A(v)$ can be written as 
$E_{A}(v)=T_a+R$ with
\begin{align*}
&\sup_{\la\xi\ra \ge 1/2~}
\lA \langle\xi\rangle^{\beta}\partial_{\xi}^\beta a(\cdot,\xi)\rA_{\eC{\rho-\tdm}}
\le K \lA v\rA_{\eC{3}}\langle \xi\rangle^{-\beta}, \\
&\lA R  f\rA_{H^{\mu+\rho-\tdm}}\le K\lA v\rA_{\eC{3}}\lA f\rA_{H^\mu}.
\end{align*}
Since $\rho-3/2\ge 1$, the statements $i)$ and $ii)$ now 
follow from Theorem~\ref{theo:sc0} in Appendix~\ref{s2}.
\end{proof}

We next prove an analogous result for the 
quadratic term $S(v)$. 
\begin{lemm}
There exist two matrices of symbols $R^{1}$, $R^{2}$ in 
$SR^{1}_{0,0}$ such that $E_R(v)=\Op^\Bony[v^1,R^1]+
\Op^\Bony[v^2,R^2]$ satisfies the following properties.

$i)$ There holds
\begin{equation}\label{n264}
E_{R}(Dv)+E_{R}(v)D-D E_{R}(v)=\mathfrak{S}(v)
\end{equation}
where $\mathfrak{S}$ is such that
\be\label{n265}
\RE \langle S(v)f-\mathfrak{S}(v)f,f\rangle_{H^s\times H^s}=0,
\ee
 for any 
$f\in H^{s}(\xR)^2$, and satisfies
\be\label{n265.1}
\lA \mathfrak{S}(v)\rA_{\Fl{H^{\mu}}{H^{\mu+\rho-1}}}\le K \lA v \rA_{\eC{\rho}}.
\ee

$ii)$ For all $(\mu,\rho)\in \xR\times \xR_+$ such that $\mu+\rho>1$ and $\rho\not\in\xN$, 
there exists a positive constant $K$ such that
\begin{equation}\label{n266}
\blA E_R(v)f \brA_{H^{\mu+\rho-1}}\le 
K\lA v\rA_{\eC{\rho}}\lA f\rA_{H^\mu}.
\end{equation}
\end{lemm}
\begin{proof}
Set $E_R(v)=E_s^\sharp(v)+E_s^\flat(v)$ where 
$E_s^\sharp(v)$ and $E_s^\flat(v)$ are as given by Proposition~\ref{T37-add} with $\beta=s$.
\end{proof}

The main result of this chapter is the following proposition.

\begin{prop}\label{T42}
Use Notation~$\ref{T29}$ and Assumptions~$\ref{T27}$ and 
$\ref{T28}$.
The new unknown
$$
\Phi=\vU+E_A(\vu)\vU-E_R(\vu)\vU
$$
satisfies
$$
\md{\partial_t\Phi+D\Phi+(Q(\vu)-B(\vu))\Phi
+(S(\vu)-\mathfrak{S}(\vu))\Phi
+C(\Ur)\Phi}{0}{s}.
$$
\end{prop}
\begin{proof}

Set $E=E_A-E_R$. 
Since
\begin{align*}
\partial_t \Phi&=\partial_t \vU +E(\partial_t \vu)\vU+E(\vu)\partial_t \vU,\\
D\Phi&=D\vU+DE(\vu)\vU,
\end{align*}
by using~\eqref{n262} and \eqref{n264} we find that
$$
\partial_t\Phi+D\Phi
=\partial_t\vU+D\vU+B(\vu)\vU+\mathfrak{S}(\vu)\vU+E(\partial_t\vu+D\vu)\vU
+E(\vu)(\partial_t\vU+D\vU).
$$
Thus, 
$$
\md{\partial_t\Phi+D\Phi+(Q(\vu)-B(\vu))\Phi+(S(\vu)-\mathfrak{S}(\vu))\Phi+C(\vu)\Phi
}{\mathcal{F}}{s}
$$
with
\begin{align*}
\mathcal{F}&=\bigl(Q(\vu)+S(\vu)+C(\vu)\bigr)E(\vu)\vU-B(\vu)E(\vu)\vU
-\mathfrak{S}(\vu)E(\vu)\vU\\
&\quad+E(\partial_t\vu+D\vu)\vU+E(\vu)(\partial_t\vU+D\vU).
\end{align*}
Since $\lA E(\vu)\rA_{\Fl{H^s}{H^{s}}}\le C\lA \vu\rA_{\eC{\varrho}}$ 
it follows from~\eqref{n259} that
$$
\md{E(\vu)(\partial_t\vU+D\vU)}{-E(\vu)\bigl(Q(\vu)\vU+C(\Ur)\vU+S(\vu)\vU\bigr)}{s}
$$
and hence $\lmd{\mathcal{F}}{\mathcal{F}_1+\mathcal{F}_2}{s}$ with 
\begin{align*}
\mathcal{F}_1&=\bigl[A(\vu),E(\vu)\bigr]\vU+\bigl[S(\vu),E(\vu)\bigr]\vU
-\bigl(B(\vu)+\mathfrak{S}(\vu)\bigr)E(\vu)\vU\\
\mathcal{F}_2&=E(\partial_t\vu+D\vu)\vU,
\end{align*} 
where recall that $A(\vu)=Q(\vu)+C(\vu)$. 

We now have to prove that $\lmd{\mathcal{F}_1}{0}{s}$ and 
$\lmd{\mathcal{F}_2}{0}{s}$. 

For this proof, 
we say that 
an operator $f\mapsto P(\vu)f$ is of order $m$ if there exists $\mu_0\in \xR$ such that for any real number $\mu\ge \mu_0$, 
it is bounded from $H^{\mu}$ to $H^{\mu-m}$ together with the estimate
$$
\lA P(\vu)\rA_{\Fl{H^\mu}{H^{\mu-m}}}\le C\lA \vu\rA_{\eC{\varrho}}
$$
for some constant $C$ depending only on $\lA \vu\rA_{\eC{\varrho}}$. 
We shall use the fact that 
if $P(\vu)$ is of order $m$ and $L(\vu)$ is of order $-m$ for some $m\in [0,1]$, then
$$
\md{P(\vu)L(\vu)\vU}{0}{s},
$$
provided that $s$ is large enough (for our purposes, it is easily verified that the requirement that $s$ is large enough will hold true under our assumption on $s$ imposed in Assumption~\ref{T27}). 
With this definition, $A(\vu)=Q(\vu)+C(\vu)$ is of order $1$ 
(this is most easily seen by using the expression~\eqref{n206} for 
$A=A(\vu)$, 
the rule~\eqref{esti:quant1}, the estimates 
\eqref{n189} for $\lA V\rA_{\eC{0}}$ and \eqref{n203} for $\lA \alpha\rA_{\eC{0}}$). 
Lemma~\ref{T40} implies that 
$E_A(\vu)$ is of order~$0$ (see~\eqref{n263}). 
Similarly, since $B(\vu)=\Op^\Bony[\vu^1,B^1]+\Op^\Bony[\vu^2,B^2]$ with 
$B^1,B^2$ in $S^{0,1/2}_{3/2}$, Lemma~\ref{T34}, Lemma~\ref{T33} 
(see statement $(ii)$) and \eqref{esti:quant1} 
imply that $B(\vu)$ is of order~$0$. 
The estimate \eqref{Bony3} 
implies that 
$S(\vu)$ is of order $3/2-\varrho$ provided that $\varrho$ is large enough. 
Similarly, \eqref{n265.1} and \eqref{n266} imply that $\mathfrak{S}$ and 
$E_R(\vu)$ are of order $1-\varrho$. 
We shall only use the fact that, with our assumption on 
$\varrho$, $S(\vu)$ and $\mathfrak{S}(\vu)$ are 
of order $0$ while $E_R(\vu)$ is of order $-1$. 

Since $E(\vu)$, $B(\vu)$, $\mathfrak{S}(\vu)$ and~$S(\vu)$ are of order~$0$, we obtain that
\begin{alignat*}{2}
&\md{S(\vu)E(\vu)\vU}{0}{s}, \qquad &&\md{E(\vu)S(\vu)\vU}{0}{s},\\
&\md{B(\vu)E(\vu)\vU}{0}{s}, &&\md{\mathfrak{S}(\vu)E(\vu)\vU}{0}{s}.
\end{alignat*}
Now we claim that $\lmd{[A(\vu),E(\vu)]\vU}{0}{s}$. To prove this result 
we estimate separately the contribution due to $E_A$ and the contribution due to $E_R$. 
Firstly, notice that since $E_R(\vu)$ is of order $-1$ and since $A(\vu)$ is of order $1$ we have
\begin{align*}
&\md{A(\vu)E_R(\vu)\vU}{0}{s},\\
&\md{E_R(\vu)A(\vu)\vU}{0}{s},
\end{align*}
which imply that $\lmd{[A(\vu),E_R(\vu)]\vU}{0}{s}$. Now we claim that similarly
\begin{equation}\label{n268}
\md{A(\Ur)E_A(\vu)\vU-E_A(\vu)A(\Ur)\vU}{0}{s}.
\end{equation}
This we prove by using symbolic calculus. We need some preparation and 
introduce $\widetilde{A}(\Ur)$ defined by
$$
\widetilde{A}(\Ur)=A(\Ur)-T_{V}\px -T_\alpha D.
$$
Directly from the definition of $A(\Ur)$, one can check 
that~$\widetilde{A}(\Ur)$ is an operator 
of order~$0$, so that
\begin{align*}
&\md{\widetilde{A}(\Ur)E_A(\vu)\vU}{0}{s},\\
&\md{E_A(\vu)\widetilde{A}(\Ur)\vU}{0}{s}.
\end{align*}
It remains to estimate the commutators of $E_A(\vu)$ with 
$T_V\px$ and $T_\alpha D$. 
Since~$T_V\px$ has a scalar symbol, it follows from statement $i)$ in Lemma~\ref{T40} that
$$
\md{T_V \px (E_A(\vu)\vU)}{E_A(\vu)T_V \px \vU}{s}.
$$
To estimate~$\bigl[T_\alpha D,E_A(\vu)\bigr]$, 
we use instead the equation~\eqref{n262} 
satisfied by~$E_A$ to obtain:
$$
T_\alpha D E_A(\vu)\vU=T_\alpha\Big(E_A(\vu)D\vU+E_A(D\vu)\vU+B(\vu)\vU\Bigr).
$$
Since~$T_\alpha$, $E_A(D\vu)$ and~$B(\vu)$ are of order~$0$ we directly find that 
$$
\md{T_\alpha E_A(D\vu)\vU+T_\alpha B(\vu)\vU}{0}{s}.
$$
Since~$\alpha$ is a scalar function, 
we can apply statement $i)$ in Lemma~\ref{T40} to obtain 
$$
\md{T_\alpha E_A(\vu)D\vU}{E_A(\vu)T_\alpha D\vU}{s}.
$$
This proves the claim~\eqref{n268} which completes the proof of 
$\lmd{\mathcal{F}_1}{0}{s}$.

It remains to prove 
that $\lmd{\mathcal{F}_2}{0}{s}$ where $\mathcal{F}_2=E(\partial_t\vu+D\vu)\vU$. This will follow 
from the operator norm estimate of $E(v)$ (see~\eqref{n263} and \eqref{n266}) 
and the estimate 
of the $C^3$-norm of $\partial_t\vu +D\vu$. 
The key point is that, 
since 
$$
\partial_t\vu +D\vu =\begin{pmatrix} \partial_t \eta-\Dx \psi
\\ \Dxmez (\partial_t \psi+\eta)\end{pmatrix}
$$
directly from \eqref{P:WW} and the definition of $\B(\eta)\psi$ 
we have
\begin{equation}\label{n269}
\partial_t\vu +D\vu =\begin{pmatrix} 
G(\eta)\psi-\Dx\psi\\
\Dxmez \bigl( -\mez (\px\psi)^2+\mez (1+(\partial_x\eta)^2)(\B(\eta)\psi)^2\bigr)
\end{pmatrix}.
\end{equation}
Then \eqref{211-1}, \eqref{n145} and \eqref{esti:Dxmez-Crho} imply that 
\be\label{n269.1}
\lA \partial_t \vu+D\vu\rA_{\eC{3}}\le C(\lA \vu\rA_{\eC{6}})\lA \vu\rA_{\eC{6}}^2.
\ee
As above mentioned, \eqref{n263} and \eqref{n266} then imply that 
$\lmd{\mathcal{F}_2}{0}{s}$.

This completes the proof of Proposition~\ref{T42}.
\end{proof}

\section{Energy estimate}\label{S:227}

\begin{prop}
Let~$T>T_0>0$ and fix~$(s,\gamma)$ such that
$$
s>\gamma+\mez>14,\quad \gamma\not\in\mez\xN.
$$
There exists a constant $C>0$ such that 
for any $\delta>0$, for any $N_1$, there exists $\eps_0$ such that 
for all $\eps\in ]0,\eps_0]$, for all $M_1>0$, 
if a solution $(\eta,\psi)$ to \eqref{P:WW}Ê
satisfy 
the following assumptions 

$i)$ $(\eta,\psi)\in C^0\bigl([T_0,T];H^{s}(\xR)\times \h{\mez,s-\mez}(\xR)\bigr)$ and 
$\omega \in 
C^0\big([T_0,T];\h{\mez,s}(\xR)\bigr)$, 

$ii)$ for any $t\in [T_0,T]$, 
$\lA \eta(t)\rA_{\eC{\gamma}} +\blA \Dxmez\psi(t)\brA_{\eC{\gamma-\mez}}
\le N_1\eps t^{-\mez}$,

$iii)$ $\lA \eta(T_0)\rA_{H^s}+\blA \Dxmez \omega(T_0) \brA_{H^s}\le M_1 \eps$,

then for any $t\in [T_0,T]$,
\be\label{n269.5}
\lA \eta(t)\rA_{H^s}+\blA \Dxmez \omega(t)\brA_{H^s}\le C M_1 \eps t^\delta.
\ee
\end{prop}
\begin{proof}
By using mollifiers and standard arguments, 
it is sufficient to prove this result under the additional assumptions that 
$\eta\in C^1([T_0,T];H^{s+1}(\xR))$ and   
$\omega \in 
C^1([T_0,T];\h{\mez,s+1}(\xR))$.

Set $\varrho=\gamma-1/2$. Then it is obvious that 
\begin{align*}
\holder(t)&\defn \lA \eta(t)\rA_{\eC{\varrho}}+\blA \Dxmez\psi(t)\brA_{\eC{\varrho}}
=\lA u(t)\rA_{\eC{\varrho}}\\
&\le \lA \eta(t)\rA_{\eC{\gamma}} +\blA \Dxmez\psi(t)\brA_{\eC{\gamma-\mez}}.
\end{align*}
As already mentioned in the remark made after the statement of Assumption~\ref{T27}, 
it follows from the assumptions $ii)$ and $iii)$ above that, if $\eps$ is small enough, then 
for any $t$ in $[T_0,T]$, 
$$
\lA \px \eta(t)\rA_{\eC{\gamma-1}} +\lA \px\eta(t)\rA_{\eC{-1}}^{1/2}\lA \eta'(t)\rA_{H^{-1}}^{1/2}
\le \eps.
$$
Therefore Assumptions \ref{T27} and \ref{T28} are satisfied (we can replace 
the time interval $[0,T]$ by $[T_0,T]$ without causing confusion since 
the equation \eqref{P:WW} is invariant by translation in time). 
Thus we may apply  Proposition~\ref{T42} which implies that 
$\Phi=\vU+E_A(\vu)\vU-E_R(\vu)\vU$ 
satisfies
$$
\partial_t\Phi+D\Phi+(Q(\vu)-B(\vu))\Phi
+(S(\vu)-\mathfrak{S}(\vu))\Phi
+C(\Ur)\Phi=\Gamma
$$
for some source term~$\Gamma$ such that 
$\lA \Gamma\rA_{H^s}\le C(\lA \vu\rA_{\eC{\varrho}})\lA \vu\rA_{\eC{\varrho}}^2\lA \vU\rA_{H^s}$. 
If~$\lA \vu\rA_{\eC{\varrho}}$ is small enough, it follows from~\eqref{n263} and \eqref{n266} that
\be\label{n270}
\mez \lA \vU\rA_{H^s}\le \lA \Phi\rA_{H^s}\le \tdm \lA \vU\rA_{H^s}.
\ee
Similarly as already seen (cf.\ \eqref{n50}) we have
\be\label{n270.1}
\mez \lA \vU\rA_{H^s}\le \lA \eta\rA_{H^s}+\blA \Dxmez \omega\brA_{H^s}
\le \tdm \lA \vU\rA_{H^s}.
\ee
Therefore,
\be\label{n271}
\lA \Gamma\rA_{H^s}\le C(\lA \vu\rA_{\eC{\varrho}})\lA \vu\rA_{\eC{\varrho}}^2\lA \Phi\rA_{H^s}.
\ee
We want to estimate $ \lA \eta\rA_{H^s}+\blA \Dxmez \omega\brA_{H^s}$. 
In view of \eqref{n270} and \eqref{n270.1} it is sufficient to 
estimate the $L^2$-norm of 
$\dot{\Phi}=\Lambda^s\Phi$ where $\Lambda=(\id-\Delta)^{1/2}$. 
This unknown satisfies
\be\label{n272}
\partial_t\dot{\Phi}+D\dot{\Phi}+L(\vu)\dot{\Phi}
+C(\Ur)\dot{\Phi}=\Gamma'
\ee
where
\begin{align*}
&L(\vu)=\Lambda^s (Q(\vu)-B(\vu)) \Lambda^{-s}+\Lambda^s (S(\vu)-\mathfrak{S}(\vu))\Lambda^{-s},\\
&\Gamma'=\Lambda^s \Gamma+ \bigl[C(\vu),\Lambda^s\bigr]\Phi.
\end{align*}
To estimate the $L^2$-norm of $\dot{\Phi}$ we take the $L^2$-scalar product of 
\eqref{n272} with $\dot{\Phi}$. The key point is that, 
by definition of $B(\vu)$ and $\mathfrak{S}(\vu)$, we have 
$\RE \langle L(\vu)\dot{\Phi},\dot{\Phi}\rangle =0$ where 
$\langle \cdot,\cdot\rangle$ is the $L^2$-scalar product. 

We need also to estimate the $L^2$-norm of the term 
$\bigl[C(\vu),\Lambda^s\bigr]\Phi$ as well as $\RE \langle C(\vu)\dot{\Phi},\dot{\Phi}\rangle$. 
Both estimates rely on the fact that, 
directly from the definition~\eqref{n209} of $C(\vu)$, the estimates 
\eqref{n155}Ê
and \eqref{n193.1} imply that $C(\vu)$ is a matrix of paradifferential operators whose symbols 
are estimated in the symbol class $\Gamma^1_{1}$ by $C(\lA \vu\rA_{\eC{\varrho}})\lA \vu\rA_{\eC{\varrho}}^2$. 
Therefore it follows from \eqref{esti:quant2} that 
$\lA \bigl[C(\vu),\Lambda^s\bigr]\Phi\rA_{L^2}$ is bounded by 
$C(\lA \vu\rA_{\eC{\varrho}})\lA \vu\rA_{\eC{\varrho}}^2\lA \Phi\rA_{H^s}$. 

On the other hand, it follows from Lemma~\ref{ref:A56}Ê
in Appendix \ref{S:A4} that 
\begin{equation}\label{n273}
\big\lvert \RE \langle C(\vu)\dot{\Phi},\dot{\Phi}\rangle \bigr\rvert
\le C(\lA \vu\rA_{\eC{\varrho}})\lA \vu\rA_{\eC{\varrho}}^2\blA \dot{\Phi}\brA_{L^2}^2.
\end{equation}
Therefore, it follows from~\eqref{n271} and \eqref{n273} that
\begin{equation}\label{n273a}
\blA \dot{\Phi}(t)\brA_{L^2}^2\le \blA \dot{\Phi}(T_0)\brA_{L^2}^2
+\int_{T_0}^t C(\lA \vu(\tau)\rA_{\eC{\varrho}})
\lA \vu(\tau) \rA_{\eC{\varrho}}^2 \blA \dot{\Phi}(\tau)\brA_{L^2}^2\, d\tau,
\ee
and hence
$$
\blA \dot{\Phi}(t)\brA_{L^2}^2\le \blA \dot{\Phi}(T_0)\brA_{L^2}^2
+K\int_{T_0}^t \frac{\eps^2}{\tau} \blA \dot{\Phi}(\tau)\brA_{L^2}^2\, d\tau,
$$
for some constant $K$ depending on the constant $N_1$ which appears in assumption $ii)$. 
The Gronwall lemma then yields 
$\blA \dot{\Phi}(t)\brA_{L^2}^2\le \blA \dot{\Phi}(T_0)\brA_{L^2}^2 t^{\eps^2 K}$. 

Since 
$\blA \dot{\Phi}\brA_{L^2}\sim \lA \eta\rA_{H^s}+\blA \Dxmez \omega\brA_{H^s}$, 
this gives the asserted estimate~\eqref{n269.5}.
\end{proof}
\begin{rema*}Notice that \e{n273a}Ê
implies the estimate \e{i4} asserted in the introduction, as explained at the end of 
Section~\ref{S:I3} of the introduction.
\end{rema*}

\chapter{Commutation of the Z-field with the equations}\label{S:23}

We begin the analysis of the Sobolev estimates for $Z^k U$ by 
establishing some identities which allow us to commute $Z^k$ with
the equations (recall that $Z=t\partial_t+2x\px$). 
This problem has already been obtained by Wu~\cite{Wu09} and 
Germain-Masmoudi-Shatah in \cite{GMS2}. 
We shall prove sharp 
tame estimates tailored to our purposes. 
To find the quadratic terms in 
the equations satisfied by $Z^kU$, 
the main difficulty consists in estimating 
$Z^kF(\eta)\psi-Z^kF_{\quadratique}(\eta)\psi$, $Z^kG(\eta)\psi-Z^k \Dx\psi$, $Z^kV(\eta)\psi-Z^k\px\psi$, 
$Z^k\B(\eta)\psi-Z^k\Dx\psi$ and $Z^k (\ma-1)$. 
These will be the main goals of this chapter.

The plan of this chapter is as follows. In section~\ref{S:231} 
we compute $ZG(\eta)\psi$. 
We then establish some identities which allow us to commute the~$Z$ field with $\B(\eta)\psi$, $V(\eta)\psi$ and 
$F(\eta)\psi$. Next we estimate the cubic terms.

\section{Action of the~Z-field on the Dirichlet-Neumann operator}\label{S:231}

The goal of this section is to compute the action of the vector field 
$Z$ on $G(\eta)\psi$. We use the abbreviated notation 
\begin{equation*}
\B=\B(\eta)\psi=\frac{G(\eta)\psi+\partial_x\eta \partial_x \psi}{1+(\partial_x\eta)^2},\quad
V=V(\eta)\psi=\partial_x \psi-\B \partial_x\eta.
\end{equation*}
We notice also that the time $t$ plays here the role of 
a parameter (as soon as we assume we may take derivatives relatively to it) that 
will not be written explicitly.

\begin{prop}\label{T43}
Let $(\eta,\psi)$ be in $\eC{\gamma}\times\h{\mez,1}$, 
with $\gamma$ in $]2,+\infty[\setminus \mez \xN$. Assume that 
$\etapetit <\delta$, where $\delta$ is the constant in $i)$ of Proposition~\ref{ref:116}, 
that $Z\psi\in \h{\mez}$, $\partial_t\px^\alpha \eta$, $Z(\px^\alpha\eta)\in L^\infty$ for $0\le \alpha\le 1$, 
and that $Z\eta$ is in $\eC{\gamma-1}$. Then
\begin{equation}\label{231}
Z G(\eta)\psi=   G(\eta)\big(Z\psi-(\B(\eta)\psi) Z\eta) -\partial_x( (Z\eta)V(\eta)\psi )+\Rzero{G}(\eta)\psi
\end{equation}
where
$$
\Rzero{G}(\eta)\psi=2\left[G(\eta)(\eta \B(\eta)\psi)-\eta G(\eta)\B(\eta)\psi\right]
+2(V(\eta)\psi)\partial_x \eta-2G(\eta)\psi.
$$
\end{prop}
Let us introduce the following notation, where $\eta'$ stands for $\px\eta$,
\be\label{232}
\ba
P&=(1+\eta'^2)\pz^2+\px^2-\px \eta'\pz -\pz\eta'\px,\\
\uZ&=t\partial_t +2x\px +(2z+2\eta-(Z\eta))\pz.
\ea
\ee
The operator $P$ is the Laplace operator $\px^2+\partial_y^2$ written in 
$(x,z)$-coordinates (see~\eqref{111}). In the same way, $\uZ$ is the vector field 
$t\partial_t +2(x\px+y\partial_y)$ written in $(x,z)$-coordinates. 
As $\bigl[ \Delta, t\partial_t +2(x\px+y\partial_y)\bigr]=4\Delta$, we have
\be\label{232a}
[P,\uZ]=4P.
\ee

To prove Proposition~\ref{T43}, we shall show that, under its assumptions, 
if $\va$ is the unique solution in $E$ to $P\va=0$, $\va\az=\psi$ provided 
by $i)$ of Proposition~\ref{ref:116}, then $\uZ \va$, which according to \eqref{232a} 
solves $P(\uZ \va)=0$, belongs to~$E$, so is the unique solution of that elliptic equation 
in $E$ with boundary data $(\uZ\va)\az$. It follows then 
from the definition \eqref{1136} of $G(\eta)$ that 
$$
G(\eta)\bigl( (\uZ \va)\az\bigr) =\Bigl[\bigl( (1+\eta'^2)\pz -\eta'\px\bigr)\uZ\va\Bigr]\Big\arrowvert_{z=0}.
$$
Computing explicitly both sides from $\psi$, $G(\eta)\psi$, $\B$, $V$, we shall get \eqref{231}. 

We start proving the regularity properties if $\uZ \va$ indicated above.

\begin{lemm}\label{T44}
Let $(\eta,\psi)$ be in $\eC{\gamma}\times\h{\mez,1}$ (at fixed $t$), 
with $\gamma$ in $]2,+\infty[\setminus \mez \xN$ and $\etapetit$ small enough. 
Assume moreover that 
$Z\eta$, $\partial_t\eta'$, $Z\eta'$ are in $L^\infty$ and that $\partial_t\psi$, $Z\psi$ 
are in $\h{\mez}(\xR)$. Then the unique solution $\va$ in the space $E$ of 
$P\va=0$, $\va\az=\psi$ provided by $i)$ of Proposition~\ref{ref:116} 
satisfies $\nabla_{x,z}\va\in E$, $\uZ \va \in E$ (at fixed $t$).
\end{lemm}
\begin{proof}
Assume given an action $(\lambda,f)\rightarrow M_\lambda f$ of some 
abelian group $\Lambda$ on the space of real valued functions defined on 
$\{ (t,x,z)\,;\,z<0\}$, sending~$E$ into~$E$. Assume also 
that there is some continuous function $\lambda\rightarrow m(\lambda)$, 
$\xR_+^*$-valued, such that
$$
\px \big[ M_\lambda f\bigr]=m(\lambda)M_\lambda(\px f),\quad 
\pz \big[ M_\lambda f\bigr]=m(\lambda)M_\lambda(\pz f)
$$
and that
$$
M_\lambda (f_1 f_2)=(M_\lambda f_1)(M_\lambda f_2),\quad 
(M_\lambda f)\az =M_\lambda (f\az).
$$

Let $\va$ be a solution in $E$ of $P\va=0$. Then, using 
the preceding properties of~$M_\lambda$,
\begin{align*}
P(M_\lambda \va)=m(\lambda)^2 M_\lambda \Bigl[
&\bigl(1+(M_\lambda^{-1}\eta')^2\bigr)\pz^2\va+\px^2\va\\
&-\px\bigl( \bigl(M_\lambda^{-1}\eta'\bigr)\pz\va\bigr)w-\pz \bigl( \bigl(M_\lambda^{-1}\eta'\bigr)\px\va\bigr)\Bigr].
\end{align*}
If, in the \rhsvirgule, we substitute $\eta'$ to $M_\lambda^{-1}\eta'$ 
(resp.\ $\eta'^{2}$ to $\bigl(M_\lambda^{-1}\eta'\bigr)^2$), we 
make appear $P\va=0$. Consequently, we may rewrite the preceding 
relation as 
$$
P\bigl(M_\lambda\va\bigr)=\pz h_1^\lambda+\px h_2^\lambda.
$$
with
\begin{align*}
h_1^\lambda&=m(\lambda)^2 M_\lambda \Bigl[ 
\bigl( \bigl(M_\lambda^{-1}\eta'\bigr)^2-\eta'^2\bigr)\pz\va 
-\bigl( \bigl(M_\lambda^{-1}\eta'\bigr)-\eta'\bigr)\px\va\Bigr],\\
h_2^\lambda&=-m(\lambda)^2 M_\lambda \Bigl[ 
\bigl( \bigl(M_\lambda^{-1}\eta'\bigr)-\eta'\bigr)\pz\va\Bigr].
\end{align*}
Using again that $P\va=0$ and that $M_\lambda$ commutes to restriction to 
$z=0$, we obtain finally
\begin{align*}
&P\bigl(M_\lambda\va-\va\bigr)=\pz h_1^\lambda+\px h_2^\lambda,\\
&\bigl(M_\lambda\va-\va\bigr)\big\az=M_\lambda\psi-\psi.
\end{align*}
Since $\va$ is in $E$, $h_1^\lambda$, $h_2^\lambda$ are in 
$L^2(]-\infty,0[\times\xR)$. Since $\eta'$ 
is in $\eC{\gamma-1}$ and since by the equation 
$\pz^2\va$ is in $L^2(]-\infty ,0[;H^{-1}(\xR))$, 
the same is true for $\pz h_1^\lambda$. Since moreover, at fixed $\lambda$, 
$M_\lambda\va-\va$ is in $E$, we may apply inequality \eqref{1112} which implies that
\be\label{232b}
\blA \nabla_{x,z}\bigl( M_\lambda\va-\va\bigr)\brA_{L^2L^2} 
\le C \Bigl[ \blA \Dxmez \bigl(M_\lambda\psi-\psi\bigr)\brA_{L^2}+
\blA h^\lambda\brA_{L^2L^2}\Bigr]
\ee
with a constant $C$ independent of $\lambda$ staying in a compact subset of 
$\Lambda$. We apply this inequality first with 
$\Lambda=\xR$, $M_\lambda$ being the action by translation 
relatively to the $x$-variable, so that $m(\lambda)\equiv 1$. 
Then $M_0=\id$ and we get
\begin{align*}
&\blA \Dxmez \bigl(M_\lambda\psi-\psi\bigr)\brA_{L^2}\le \blA \Dxmez \psi\brA_{H^1}\la \lambda\ra,\\
&\blA h^\lambda\brA_{L^2L^2}\le C \lA \nabla \va\rA_{L^2L^2}\la \lambda\ra,
\end{align*}
where, for the second estimate, we used that $\eta'$ is lipschitz relatively to $x$. 
We deduce from \eqref{232b} 
$$
\blA \nabla_{x,z}\bigl( \va(t,x+\lambda,z)-\va(t,x,z)\bigr)\brA_{L^2_xL^2_z}
\le C \la \lambda\ra\Big[ \blA \Dxmez \psi\brA_{H^1}+\lA \nabla_{x,z} \va\rA_{L^2L^2}\Big].
$$
It follows that $\nabla_{x,z}(\px\va)$ is in $L^2L^2$. 
Using the equation $P\va=0$, we obtain as well 
$\pz^2\va\in L^2L^2$ so that 
$\nabla_{x,z}\va$ is in $E$. 

Applying the same reasoning to time translations, we get that $\partial_t\va$ is in $E$. 

Let us prove now that $\uZ\va$ belongs to $E$. Denote $Z_0=t\partial_t+(2x\px+2z\pz)$ 
so that $(\uZ-Z_0)\va=(2\eta-(Z\eta))\pz\va$ is in $L^2(]-\infty,0[\times\xR)$ (at fixed $t$) 
as well as its $(x,z)$-gradient by what we just saw. This shows that $(\uZ-Z_0)\va$ is in $E$, so that 
we just need to check that $Z_0\va$ is in $E$, so that 
$\nabla_{x,z}Z_0\va$ is in $L^2(]-\infty,0[\times\xR)$. We use estimate \eqref{232b} 
where $M_\lambda$ is the action of $\xR_+^*$ on functions given by 
$M_\lambda\va (t,x,z)=\va (\lambda t,\lambda^2x,\lambda^2z)$ and where 
$m(\lambda)=\lambda^2$. 
Then $\nabla\Bigl( \frac{M_\lambda\va-\va}{\lambda-1}\Bigr)$ 
converges in the sense of distributions to $\nabla Z_0\va$ when $\lambda$ goes to $1$, and the assumptions 
$Z\psi\in \h{\mez}$, $Z\eta'\in L^\infty$, $\nabla \va\in L^2L^2$ show that, 
when $\lambda$ stays in a compact neighborhood of $1$, the \rhs of \eqref{232b} 
is bounded from above by $C\la \lambda-1\ra$ 
(Notice that the action by $M_\lambda$ on functions of $(t,x)$ has $Z$ as infinitesimal 
generator). Dividing \eqref{232b} by $\lambda-1$, we conclude that 
$\nabla_{x,z}(Z_0\va)$ is in $L^2(]-\infty,0[\times\xR)$ as wanted.
\end{proof}

\begin{proof}[Proof of Proposition~\ref{T43}]
We notice first that by the definition \eqref{1136} of $G(\eta)$ and the one of $\B$, $\pz\va\az=\B(\eta)\psi$, so that
$$
\uZ \va\az=Z\psi+(2\eta-(Z\eta))(\B(\eta)\psi).
$$
As $G(\eta)\psi$ is in $H^{1/2}$ as a function of $x$ by Proposition~\ref{T14}, we see 
that under the assumptions of the statement, $\B$ belongs to $H^{1/2}(\xR)$, 
so that $\uZ\va\az$ is in $\h{1/2}$. Moreover, by Lemma~\ref{T44}, $\uZ\va$ is in $E$. 
By uniqueness of solutions in $E$ to $P(\uZ \va)=0$, 
$\uZ \va\az\in \h{1/2}$ given by Proposition~\ref{ref:116}, we deduce that
\be\label{232c}
G(\eta)\bigl[ Z\psi+(2\eta-(Z\eta))\B\bigr]
=\Bigl[ (1+\eta'^2)\pz (\uZ \va)-\eta'\px (\uZ \va)\Bigr]\Big\az.
\ee
Let us deduce \eqref{231} from this equality. From the definition~\e{232} of $\uZ$ we get
$$
\pz \bigl(\uZ \va\bigr)=Z(\pz\va)+2\pz\va+(2z+2\eta-(Z\eta))\pz^2\va.
$$
Multiplying by $(1+\eta'^2)$ and using that $P\va=0$ to express the $\pz^2\va$ term, 
we get
\begin{align*}
(1+\eta'^2)\pz (\uZ \va)&=(1+\eta'^2)Z(\pz \va)
+2(1+\eta'^2)\pz\va\\
&\quad +(2z+2\eta-(Z\eta))\Big[ \px (\eta'\pz\va-\px\va)+\pz (\eta'\px\va)\Big].
\end{align*}
We compute from that expression the \rhs of \eqref{232c} remembering that 
$\pz\va\az=\B$ and that $V=(\px\va-\eta'\pz\va)\az$.

We obtain
\be\label{232d}
\ba
G(\eta)\bigl[Z\psi+(2\eta-(Z\eta))\B\bigr]&=(1+\eta'^2)Z\B +2(1+\eta'^2)\B\\
&\quad +(2\eta-(Z\eta))\bigl[-\px V+\eta'\px \B\bigr]\\
&\quad -\eta'\px \bigl[ Z\psi+(2\eta-(Z\eta))\B\bigr].
\ea
\ee
We are left with transforming this expression into \eqref{231}. We notice 
first that $\pz\va$ satisfies $P(\pz\va)=0$, $\pz\va\az=\B$ and that 
by Lemma~\ref{T44}, $\pz\varphi$ is in $E$, while $\B$ has been seen to 
belong to $\h{\mez}$. We may thus apply again the uniqueness result of Proposition~\ref{ref:116} 
to conclude that
$$
G(\eta)\B=\Bigl[\bigl( (1+\eta'^2)\pz -\eta'\px \bigr)(\pz\va)\Bigr]\Big\az.
$$
Expressing in the right hand side of this equality $(1+\eta'^2)\pz^2\va$ 
from the equation $P\va=0$, we get
\be\label{232e}
G(\eta)\B=-\px V.
\ee
Using that formula, and, by definition of $\B$
$$
(1+\eta'^2)\B=G(\eta)\psi+\eta'(\px\psi)
$$
we rewrite \eqref{232d} after simplifications as
\begin{align*}
G(\eta)\bigl[Z\psi-\B (Z\eta)\bigr]&=Z \bigl[ G(\eta)\psi\bigr] +2\bigl[ \eta G(\eta)\B-G(\eta)(\eta \B)\bigr]
\\
&\quad+2G(\eta)\psi+(Z\eta)(\px V)+(Z\eta')(\px\psi-\eta'\B).
\end{align*}
Expressing in the last term $\px\psi$ from $V+\B(\px\eta)$, by definition of $V$, 
we get \eqref{231}. This concludes the proof.
\end{proof}

\section{Other identities}\label{S:232}
Next  we notice that properties of 
$Z\B(\eta)\psi$ and~$ZV(\eta)\psi$ can be deduced using 
$\B(\eta)\psi-(V(\eta)\psi)\px\eta=G(\eta)\psi$ and the previous 
calculation result for~$ZG(\eta)\psi$. 
The conclusion is given by the following lemma.

\begin{lemm}\label{T45}
Use the same notations and assumptions as in Proposition~$\ref{T43}$. Then
\begin{align}
Z B(\eta)\psi& = \B(\eta)\big(Z\psi-(\B(\eta)\psi) Z\eta\big)+\Rzero{B}(\eta)\psi,\label{N300} \\[0.5ex]
Z V(\eta)\psi& = V(\eta)\big(Z\psi-(\B(\eta)\psi) Z\eta\big)+\Rzero{V}(\eta)\psi,\label{N301}
\end{align}
with
\begin{align}
\Rzero{B}(\eta)\psi&=\frac{1}{1+(\px\eta)^2}\Big[-4(V(\eta)\psi)\px\eta
+\bigl((\px \eta) (\px \B(\eta)\psi)-(\px V(\eta)\psi)\bigr)Z\eta\Bigr]\label{N302}\\
&\quad +\frac{1}{1+(\px\eta)^2}\Rzero{G}(\eta)\psi,\notag
\end{align}
and
\be\label{N303}
\Rzero{V}(\eta)\psi=-(\Rzero{B}(\eta)\psi)\px\eta+(\px \B(\eta)\psi) Z\eta -2V(\eta)\psi,
\ee
where recall that~$\Rzero{G}(\eta)\psi$ is given by~\eqref{231}.
\end{lemm}
\begin{proof}We abbreviate $B=B(\eta)\psi$, $V=V(\eta)\psi$ and $\Rzero{G}=\Rzero{G}(\eta)\psi$. 

Starting from~$B-V\partial_x \eta=G(\eta)\psi$, we have
$$
ZB-(ZV)\partial_x \eta-V Z\px \eta=ZG(\eta)\psi.
$$
Since~$ZV=Z(\px\psi-\B\px\eta)$, 
we have
$$
ZB-(ZV)\partial_x \eta=(1+(\px\eta)^2)Z\B - (Z\px\psi)\px \eta+\B\px\eta Z\px \eta
,
$$
so
$$
(1+(\px\eta)^2)Z\B=ZG(\eta)\psi + (Z\px\psi)\px\eta-\B\px\eta Z\px \eta+VZ\px\eta.
$$
Now, according to the identity~\eqref{231} for~$ZG(\eta)\psi$, we obtain
\begin{align*}
(1+(\px\eta)^2)Z\B&=G(\eta)(Z\psi-\B Z\eta) -\partial_x( (Z\eta)V )+\Rzero{G}\\
&\quad+ (Z\px\psi) \px\eta-\B\px\eta Z\px \eta+VZ\px\eta. 
\end{align*}
Then, it is a simple calculation using~$Z\partial_x =\partial_x Z -2 \partial_x$ to verify 
that
\begin{align*}
(1+(\px\eta)^2)Z\B&=G(\eta)(Z\psi-\B Z\eta) -Z\eta \px V +\Rzero{G}\\
&\quad+ V(-2\px\eta) +\px \eta \px Z\psi -2\px \eta \px\psi\\
&\quad -\B\px\eta\px Z\eta+2\B(\px\eta)^2
\end{align*}
so
\begin{align*}
(1+(\px\eta)^2)Z\B
&=G(\eta)(Z\psi-\B Z\eta) -Z\eta \px V +\Rzero{G}\\
&\quad +\px \eta\px (Z\psi-\B Z\eta)+(\px \eta)(\px \B)Z\eta\\
&\quad-2V\px\eta-2\px\psi\px\eta+2\B (\px\eta)^2.
\end{align*}
On the other hand, by definition of $B(\eta)$, we have
$$
\B(\eta)(Z\psi-\B Z\eta)=\frac{1}{1+(\px\eta)^2}
\Bigl(G(\eta)(Z\psi-\B Z\eta)+\px \eta\px (Z\psi-\B Z\eta)\Bigr).
$$
Thus, we obtain that $\Rzero{B}(\eta)\psi$ is given by
$$
\frac{1}{1+(\px\eta)^2}\Big[-Z\eta \px V +\Rzero{G}+\px\eta(\px \B)Z\eta
-2V\px \eta-2\px\psi\px\eta+2\B (\px\eta)^2\Bigr].
$$
Since 
$$
-2\px\psi\px\eta+2\B (\px\eta)^2=-2(\px\psi-\B\px\eta)\px\eta=-2V\px\eta
$$
by definition of $V$, this yields the desired result~\eqref{N300}.

It remains to prove \eqref{N301}. Starting from~$V=\px\psi-\B \px\eta$, we have
\begin{align*}
ZV&=Z(\px\psi-\B\px\eta)\\
&=\px Z\psi -2\px\psi-(Z\B)\px\eta-\B Z\px\eta\\
&=\px Z\psi -2\px\psi-(Z\B)\px\eta-\B\px Z\eta+2\B \px\eta\\
&=\px(Z\psi-\B Z\eta)+(\px \B) Z\eta-(Z\B) \px\eta-2V.
\end{align*}
Since
$$
V(\eta)(Z\psi-\B Z\eta)=\px(Z\psi-\B Z\eta)-\bigl(B(\eta)(Z\psi-\B Z\eta)\bigr)\px\eta,
$$
the desired result follows from~\eqref{N300}.
\end{proof}

The previous identities have been stated in a way which is 
convenient to compute $ZF(\eta)\psi$. Our last identity is about~$ZF(\eta)\psi$ where recall that 
$$
F(\eta)\psi = G(\eta)\psi-\left( \Dx \bigl(\psi-T_{\B(\eta)\psi}\eta \bigr) - \partial_x(T_{V(\eta)\psi} \eta) \right).
$$
\begin{lemm}\label{T46}
Use the same notations and assumptions as in Proposition~$\ref{T43}$. 
There holds
\begin{align*}
ZF(\eta)\psi 
&= 
F(\eta)\bigl(Z\psi-(\B(\eta)\psi) Z\eta\bigr)-2F(\eta)\psi\\
&\quad -\Dx T_{Z\eta}\B(\eta)\psi -\px\bigl(T_{Z\eta} V(\eta)\psi\bigr)\\
&\quad -\Dx \RBony(\B(\eta)\psi,Z\eta)-\px \RBony(Z\eta,V(\eta)\psi)\\
&\quad+2 G(\eta)(\eta \B(\eta)\psi)-2 \eta G(\eta)\B(\eta)\psi \\
&\quad +\Dx T_{\Rzero{B}(\eta)\psi}\eta+2(V(\eta)\psi)\px \eta +\px (T_{\Rzero{V}(\eta)\psi}\eta)\\
&\quad+2\Dx S_\Bony(\B(\eta)\psi,\eta)+2\px S_\Bony(V(\eta)\psi,\eta),
\end{align*}
where $S_\Bony$ is given by \eqref{n227}; $\Rzero{B}$ and 
$\Rzero{V}$ 
are given by \eqref{N302} and \eqref{N303} and $\RBony(a,b)=ab-T_a b-T_b a$.
\end{lemm}
\begin{proof}
We write simply $A$ instead of 
$A(\eta)\psi$ for $A\in \{ B,V,\Rzero{G},\Rzero{B},\Rzero{V}\}$.
 
Recall that
$$
Z G(\eta)\psi=   G(\eta)(Z\psi-\B Z\eta) -\partial_x( (Z\eta)V )+\Rzero{G},
$$
with
$$
\Rzero{G}=2\left[G(\eta)(\eta \B)-\eta G(\eta)\B\right]+2V \partial_x \eta-2G(\eta)\psi.
$$
Consequently,
$$
ZF(\eta)\psi= G(\eta)(Z\psi-\B Z\eta) -\px( (Z\eta)V )+\Rzero{G}
- Z \Dx \bigl(\psi-T_{\B}\eta \bigr) + Z \px(T_{V} \eta).
$$
We shall study the terms separately.

Start with~$Z \Dx \bigl(\psi-T_{\B}\eta \bigr)$. Since 
$Z\Dx =\Dx Z -2\Dx$, we have
$$
Z \Dx \bigl(\psi-T_{\B}\eta \bigr) = \Dx Z (\psi-T_{\B}\eta \bigr) -2\Dx(\psi-T_{\B}\eta \bigr),
$$
By using the 
following consequence of \eqref{n227}:
\be\label{N315}
Z(T_a b)=T_{Za}b+T_a Zb+2S_\Bony(a,b),
\ee
we find that
\begin{align*}
Z \Dx \bigl(\psi-T_{\B}\eta \bigr) &= \Dx (Z \psi-T_{\B}Z \eta \bigr)
-\Dx T_{ZB}\eta-2\Dx(\psi-T_{\B}\eta \bigr)\\
&\quad-2\Dx S_\Bony(\B,\eta).
\end{align*}
Now set
$$
C\defn\B(\eta)(Z\psi-\B Z\eta),\quad W\defn V(\eta)(Z\psi-\B Z\eta),
$$
to obtain, by definition of~$F(\eta)$,
$$
G(\eta)(Z\psi-\B Z\eta) =\Dx (Z\psi-\B Z\eta-T_C\eta )-\partial_x(T_W \eta)+F(\eta)(Z\psi-\B Z\eta).
$$
Writing~$\Dx (Z\psi-\B Z\eta)$ under the form
$$
\Dx (Z\psi-\B Z\eta)=\Dx  (Z\psi-T_{\B} Z\eta)-\Dx (T_{Z\eta}\B)-\Dx \RBony(\B,Z\eta),
$$
and combining the previous identities, we conclude
\begin{equation}\label{N316}
\begin{aligned}
ZF(\eta)\psi&= -\Dx (T_{Z\eta}\B)-\partial_x( (Z\eta)V )\\
&\quad +2\Dx(\psi-T_{\B}\eta \bigr)-2G(\eta)\psi\\
&\quad +2\left[G(\eta)(\eta \B)-\eta G(\eta)\B\right]\\
&\quad +\Dx T_{ZB}\eta+2V \partial_x \eta-\Dx T_C\eta-\partial_x(T_W \eta)+ Z \partial_x(T_{V} \eta)\\
&\quad -\Dx R_\Bony(\B,Z\eta)+2\Dx S_\Bony(\B,\eta)+F(\eta)(Z\psi-\B Z\eta).
\end{aligned}
\end{equation}
To simplify this expression, we use three facts. Firstly, 
by definition of~$F(\eta)$, we have
$$
2\Dx(\psi-T_{\B}\eta \bigr)-2G(\eta)\psi=2\partial_x(T_V\eta)-2F(\eta)\psi.
$$
Secondly, we paralinearize the product~$(Z\eta)V$ to obtain
\begin{align*}
\px( (Z\eta)V )&=\px(T_{Z\eta}V+T_{V}Z\eta+\RBony(Z\eta,V))\\
&=T_{Z\eta}\px V +T_{\px Z\eta}V+\px (T_V Z\eta)
+\px R_\Bony(Z\eta,V).
\end{align*}
Thirdly, 
since~$Z\px -\px Z=-2\px$, \eqref{N315} implies that
$$
Z\px(T_{V} \eta)+2\px (T_V\eta)-\px (T_V Z\eta)=\px (T_{ZV}\eta)
+2\px S_\Bony(V,\eta).
$$
Now substitute the above relations into 
\eqref{N316} and 
simplify. We conclude that
\begin{equation*}
\begin{aligned}
Z F(\eta)\psi&=-\Dx T_{Z\eta}\B -\px(T_{Z\eta} V)\\
&\quad +2 G(\eta)(\eta \B)-2 \eta G(\eta)\B 
+\Dx T_{Z\B-C}\eta+2V\px \eta +\px (T_{ZV-W}\eta) \\
&\quad -\Dx \RBony(\B,Z\eta)-\px \RBony(Z\eta,V)
+2\Dx S_\Bony(\B,\eta)+2\px S_\Bony(V,\eta)\\
&\quad -2F(\eta)\psi+F(\eta)(Z\psi-B Z\eta).
\end{aligned}
\end{equation*}
The desired result then follows from \eqref{N300} and \eqref{N301}.
\end{proof}

\section{Estimates for the action of iterated vector fields}

In this section, we shall estimate the action of iterated vector fields $Z$ on the Dirichlet-Neumann operator $G(\eta)$, and on 
related operators. We shall express these actions in terms of convenient classes of multilinear operators. 

We denote by $\Er$ the algebra of operators generated by the operators of multiplication by analytic functions 
$(\eta,\eta')\rightarrow a(\eta,\eta')$ (defined on a neighborhood of zero), by the operators
\be\label{2311a}
G(\eta)\Dx^{-\mez}\langle \OD\rangle ^{-\mez},\quad \B(\eta)\Dx^{-\mez}\langle \OD\rangle^{-\mez}, \quad 
V(\eta)\Dx^{-\mez}\langle \OD\rangle^{-\mez}, \quad b_0(D_x)
\ee
where $b_0(D_x)$ is any Fourier multiplier, continuous, smooth outside zero, and satisfying estimates 
$\bla \partial_{\xi}^\alpha b_0(\xi)\bra =O\big(|\xi|^{c-\alpha}\langle \xi\rangle^{-c}\big)$ for some $c>0$ or 
$|\partial_\xi^\alpha b_0(\xi)|=O\big(\langle \xi\rangle^{-\alpha}\big)$. Notice that all these operators 
are of order zero i.e.\ if $\eta$ is in $\eC{\gamma}$ and if $\mu\ge 0$ is such that $\gamma>\mu+\tdm$, the first of these operators 
acts from $H^\mu$ to $H^\mu$ by Proposition~\ref{ref:116}. By the definition \e{211} of $\B(\eta)$ and $V(\eta)$, the same holds true 
for the second and third one. By Corollary~\ref{ref:118} we have also boundedness from $\eC{\gamma-1}$ to itself. 

We denote by $\widetilde{\Er}$ the right ideal of $\Er$ given by these elements of $\Er$ that may be written as linear combinations of 
$G(\eta)\Dx^{-\mez}\langle \OD\rangle^{-\mez}E$ and $b_0(D_x)E$ where $E$ is in $\Er$ and $b_0(D_x)$ is a Fourier multiplier as above, 
with $c\ge 1/2$. 

\begin{defi}\label{ref:235A}
Let $p\in \xN$, $q\in \xZ$, $p+q\ge 0$, $N\in\xN$. One denotes by $\mc{q}{p}{N}$ the vector space generated by operators 
of the form
\be\label{2311b}
E_0\circ \Big[ \big(Z^{p_1}b_{q_1}(\OD)a_1\big)E_1\Big]\circ 
\Big[ \big(Z^{p_2}b_{q_2}(\OD)a_2\big)E_2\Big]\circ \cdots \circ \Big[ \big(Z^{p_{N'}}b_{q_{N'}}(\OD)a_{N'}\big)E_{N'}\Big]
\ee
where $N'\ge N$, $b_j(\OD)$, $j=1,\ldots, N'$, is a smooth Fourier multiplier of order $q_j$, $E_j$ 
is in $\Er$ for $1\le j\le N'$, $a_j$ is some analytic function of $(\eta,\eta')$ vanishing at $(\eta,\eta')=(0,0)$, and 
the integers $p_j,q_j$ satisfy the inequalities
\be\label{2311c}
\sum_{r=1}^{N'}(p_r+q_r)\le p+q,\quad 
\sum_{r=1}^{N'}p_r\le p,\quad p_r+q_r\ge 0,\quad 
q_r\ge -1,~r=1,\ldots, N'.
\ee
We set $\Cr_{q}^{p}$ for $\mc{q}{p}{0}$. We denote by 
$\tmc{q}{p}{N}$ the subspace of $\mc{q}{p}{N}$ generated by operators of the form \e{2311b} where $E_0$ is in $\widetilde{\Er}$.
\end{defi}

We study first the composition of an element of $\mc{q}{p}{N}$ and of $(\px,Z)$-derivatives.

\begin{prop}\label{ref:235B}
Let $C$ be an element of $\mc q p N$, $\ell$, $k$ be in $\xN$. There are elements $C_{j,h}^i$ of 
$\mc{q+\ell-h-j}{p+k-i}{N}$ for $i+j\le k$, $h\le \ell$, $i,j,h$ in $\xN$ such that 
\be\label{2311d}
\px^\ell Z^k C=\sum_{\substack{i+j \le k \\ h\le \ell}}C_{j,h}^i \px^{j+h}Z^i.
\ee
Moreover, if $C$ is in $\tmc q p N$, then $C_{j,h}^{i}$ is in $\tmc{q+\ell-h-j}{p+k-i}{N}$.
\end{prop}

We consider first the case when $\ell+k=1$ and $C$ is in $\Er$. 

\begin{lemm}\label{ref:235C}
Let $E$ be in $\Er$. Then
\be\label{2311e}
\ba
ZE&=EZ +C_0^1+C_{-1}^1\px,\\
\px E&=E\px +C_1^0
\ea
\ee
where $C_q^p$ are in $\Cr_q^p$. If $E$ is in $\widetilde{\Er}$, the first equality holds with $C_q^p$ in 
$\widetilde{\Cr}_q^p$. 
\end{lemm}
\begin{proof}
Consider the case when $E=G(\eta)\Dx^{-\mez}\langle \OD\rangle^{-\mez}\in \widetilde{\Er}$. 
Writing $G(\eta)=E\Dxmez \langle \OD\rangle^\mez$ and decomposing $\Dxmez\langle \OD\rangle^\mez=b_0'(\OD)+
b_0''(\OD)\px$ where $b_0',b_0''$ are symbols satisfying the same conditions 
as $b_0$ in \e{2311a}, with $b_0''=0$ close to zero, we may write
\be\label{2311f}
G(\eta)=E'+E''\px
\ee
with $E',E''$ in $\widetilde{\Er}$. Write
\be\label{2311g}
[Z,E]=[ Z, G(\eta)]\Dx^{-\mez}\langle \OD\rangle^{-\mez}+E b_0(\OD)
\ee
for some Fourier multiplier $b_0(\OD)$ as in \e{2311a}, and express $[Z,G(\eta)]$ using \e{231} and the fact that $G(\eta)\B=-\px V$ 
i.e.\ 
\be\label{2311gbis}
\ba
\big[Z,G(\eta)\big]\widetilde{\psi}&=-G(\eta)\big( (Z\eta)\B(\eta)\widetilde{\psi}\big)-\px \big((Z\eta)V(\eta)\widetilde{\psi}\big)\\
&\quad +2 G(\eta)\big( \eta \B(\eta)\widetilde{\psi}\big)+2 \px \bigl( \eta V(\eta)\widetilde{\psi}\big)-2G(\eta)\widetilde{\psi}.
\ea
\ee
If we express $\widetilde{\psi}=\Dx^{-\mez}\langle \OD\rangle^{-\mez}\psi$ and use \e{2311f}, \e{2311g}, we see finally that $[Z,E]\psi$ 
may be written from expression
\be\label{2311h}
\ba
&\widetilde{E}_0\px\big((Z\eta)E_0\psi\big),\quad &&\widetilde{E}_0\big((Z\eta)E_0\psi\big),\\
&\widetilde{E}_0\px (E_0\psi),\quad &&\widetilde{E}_0\psi,
\ea
\ee
where $\widetilde{E}_0$ is in $\widetilde{\Er}$ and $E_0$ is in $\Er$. 

Since, on the other hand
\begin{equation}\label{2311ha}
\begin{aligned}
\big[ \px , G(\eta)\big]&=2\eta'\eta''B(\eta)-\eta''\px,\\
[\px,\B(\eta)]&=-\frac{2\eta'\eta''}{(1+\eta'^2)^2}G(\eta)+\frac{\eta''(1-\eta'^2)}{(1+\eta'^2)^2}\px
+\frac{1}{1+\eta'^2}[\px,G(\eta)],\\
[\px,V(\eta)]&=-\eta' [\px ,\B(\eta)]-\eta'' \B(\eta),
\end{aligned}
\end{equation}
we see that, if $E_0$ is in $\Er$, $[\px,E_0]$ may be written as a linear combination of quantities
$$
E_0',\quad \px \big(a(\eta,\eta')\big)E_0'
$$
with $E_0'$ in $\Er$, so that the second equality in \e{2311e} holds. 

Plugging this information in \e{2311h}, we see finally that $[Z,E]\psi$ is a linear combination of 
quantities of the following type
\be\label{2311i}
\ba
&\widetilde{E}_0\big( (Z\eta)E_0'\px\psi\big),\quad &&\widetilde{E}_0(\px \psi),\\
&\widetilde{E}_0\big( (\px Z\eta)E_0'\psi\big),\quad &&\widetilde{E}_0\big( (Z\eta)(\px a)E_0'\psi\big),\\
&\widetilde{E}_0 \big( (Z\eta)E_0'\psi\big),\quad &&\widetilde{E}_0\big( (\px a)E_0'\psi\big),\quad \widetilde{E}_0 \psi,
\ea
\ee
where $a$ is some analytic function of $(\eta,\eta')$, $\widetilde{E}_0 $ is in $\widetilde{\Er}$ 
and $E_0'$ is in $\Er$. We may write $\eta$ in the above formulas as $\eta=b_0'(\OD)\eta+b_0''(\OD)\eta'$ where 
$b_0',b_0''$ are Fourier multipliers of order $-1$. It follows then from Definition~\ref{ref:235A} that the quantities on the first 
line of \e{2311i} may be written $C_{-1}^1\px \psi$ 
with $C_{-1}^1$ in $\widetilde{\Cr}_{-1}^1$. Those one the second and third lines are of the form $C_0^1$ with 
$C_0^1$ in $\widetilde{C}_0^1$. This gives the first formula in \e{2311e} when 
$E=G(\eta)\Dx^{-\mez}\langle \OD\rangle^{-\mez}$. If $E$ is the operator $b_0(\OD)$ in \e{2311a}, the same conclusion holds. 

Consider next the case when $E=B(\eta)\Dx^{-\mez}\langle \OD\rangle^{-\mez}$ or $E=V(\eta)\Dx^{-\mez}\langle \OD\rangle^{-\mez}$. 
We may express $\B(\eta)$, $V(\eta)$ from $G(\eta)$ and explicit quantities, which shows that 
$[Z,E]$ may still be written from expressions \e{2311i}, but with $\widetilde{E}_0 $ in $\Er$ instead of 
$\widetilde{\Er}$. We thus get an expression $C_0^1+C_{-1}^1\px$, 
with $C_q^p$ in $\Cr_q^p$. 

We have thus shown both equalities \e{2311e} when $E$ is any of the expressions \e{2311a}. If 
$E$ is a general element of $\Er$, the conclusion follows by composition.
\end{proof}

\begin{rema*}If $E$ is in $\widetilde{\Er}$, the expressions obtained above for $[\px,G(\eta)]$, 
$[\px , \B(\eta)]$, $[\px,V(\eta)]$ show that $[\px,E]$ will not be in $\widetilde{\Cr}_0^0$ in general. Nevertheless we 
may write
$$
\px E =\px \chi(\OD)E+(1-\chi)(\OD)E\px +(1-\chi)(\OD)[\px,E]
$$
which shows that 
\be\label{2311j}
\px E=E'\px +E''
\ee
with $E',E''$ in $\widetilde{\Er}$. 
\end{rema*} 

\begin{proof}[Proof of Proposition~\ref{ref:235B}]
We notice first that it follows from Definition~\ref{ref:235A} that, by concatenation of expressions \e{2311b}, 
$\mc q p N\circ \mc {q'}{p'}{N'}\subset \mc{q+q'}{p+p'}{N+N'}$. Let us prove that
\be\label{2311k}
\ba
&\big[\px, \mc q p N\big]\subset \mc{q+1}{p}{N},\\
&\big[Z, \mc q p N\big]\subset \mc{q}{p+1}{N}+\mc {q-1}{p+1}{N}\circ \px.
\ea
\ee
It is enough to consider operators of the form \e{2311b} and to argue by induction on $N'$. If $N'=0$, we just get an element 
$E_0$ of $\Er$, with $p=q=0$, and the conclusion follows from \e{2311e}. Assume that the conclusion has been proved with $N'$ replaced 
by $N'-1$ in \e{2311b} and for any $p,q$ with $p+q\ge 0$. We may write \e{2311b} as 
$E_0\circ \big(\big(Z^{p_1}b_{q_1}(\OD)a_1\big)\circ C\big)$ where $C$ is an element of 
$\Cr_{q-q_1}^{p-p_1}$  which is the product of $N'-1$ factors, so to which the induction assumption applies. 
We write 
\begin{align*}
\Big[ Z, E_0\circ \big(Z^{p_1}b_{q_1}(\OD)a_1\big)\circ C\Big]
&=\big[ Z,E_0\big] \circ \big(Z^{p_1}b_{q_1}(\OD)a_1\big)\circ C\\
&\quad + E_0\circ \big(Z^{p_1+1}b_{q_1}(\OD)a_1\big)\circ C
+ E_0\circ \big(Z^{p_1}b_{q_1}(\OD)a_1\big)\circ [Z,C].
\end{align*}
The assumption of induction implies that the last two terms belong to 
$\Cr_{q}^{p+1}+\Cr_{q-1}^{p+1}\circ \px$. By \e{2311e}, the first term in the \rhs 
may be written
\begin{align*}
&C_0^1\circ \big(Z^{p_1}b_{q_1}(\OD)a_1\big)\circ C 
+C_{-1}^1 \circ \big(\px Z^{p_1}b_{q_1}(\OD)a_1\big)\circ C
\\
&\quad +\big(Z^{p_1}b_{q_1}(\OD)a_1\big)\circ [\px,C]+\big(Z^{p_1}b_{q_1}(\OD)a_1\big)\circ C\circ \px.
\end{align*}
By the assumption of induction, the composition rule and \e{2311e}, the first three terms belong to 
$\Cr_q^{p+1}$. The last term is in 
$\Cr_q^p\circ \px\subset \Cr_{q-1}^{p+1}\circ \px$. This gives the second inclusion 
in \e{2311k}. The proof of the first inclusion \e{2311k} is similar. Formula \e{2311d} follows then by induction, using \e{2311k} and the fact that 
$[Z,\px]=-2\px$. 
\end{proof}

We shall use the preceding results to obtain bounds for the action of vector fields on operator of the form $G(\eta)$, $\B(\eta)$, \ldots 
Let us define some norms.

\begin{defi}\label{ref:236}
Given $T>0$, $n\in \xN$ and $\sigma\in [0,+\infty\por$, one denotes by 
$\eCZ{n,\sigma}$ (resp.\ $\eHZ{n,\sigma}$) the space of functions 
$f\colon [0,T]\times\xR\rightarrow \xC$ such that for any integer $p$ in $[0,n]$, one has 
$Z^p f\in \eC{0}([0,T];\eC{\sigma+n-p}(\xR))$ (resp.\ $Z^p f\in\eC{0}([0,T];H^{\sigma+n-p}(\xR))$). 
One uses the notations\index{Norms!$\triple{\cdot}{*,*}$}\index{Norms!$\double{\cdot}{*,*}$}
\begin{alignat*}{2}
\triple{f(t)}{n,\sigma}&=\sum_{p=0}^n\lA Z^p f(t)\rA_{\eC{\sigma+n-p}(\xR)}, \qquad 
\triple{f}{n,\sigma}&&=\sup_{t\in [0,T]}\triple{f(t)}{n,\sigma},\\
\double{f(t)}{n,\sigma}&=\sum_{p=0}^n\lA Z^p f(t)\rA_{H^{\sigma+n-p}(\xR)}, \qquad 
\double{f}{n,\sigma}&&=\sup_{t\in [0,T]}\double{f(t)}{n,\sigma}.
\end{alignat*}
We shall use the variants \index{Function spaces!$\deCZ{n,\sigma}$}$\deCZ{n,\sigma}$ (resp.\ \index{Function spaces!$\deHZ{n,\sigma}$}$\deHZ{n,\sigma}$) for the spaces defined as above, but with 
$\eC{\sigma+n-p}(\xR)$ (resp.\ $H^{\sigma+n-p}(\xR)$) replaced by 
$\C{\mez,\sigma+n-p}(\xR)$ (resp.\ $\h{\mez,\sigma+n-p}(\xR)$). The norms on these spaces 
are $\btriple{\Dxmez f}{n,\sigma}$ (resp.\ $\bdouble{\Dxmez f}{n,\sigma}$).
\end{defi}

We gather here some elementary estimates which 
follow from the definition of $\triple{\cdot}{n,\sigma}$. 
\begin{lemm}
Consider $(n,\sigma_1)\in \xN^2$ and $\sigma\in [0,+\infty\por$. 

$i)$ For any $f\in \eCZ{n+\sigma_1,\sigma}$,
\begin{equation}\label{ns1s2}
\triple{f}{n,\sigma_1+\sigma}\le \triple{f}{n+\sigma_1,\sigma}.
\end{equation}
$ii)$ There exists a constant $c$ such that for 
any $f,g$ in $\eCZ{n,\sigma}$,
\begin{equation}\label{prod:eCZ}
\triple{f g}{n,\sigma}\le c\triple{f}{n,\sigma}\triple{g}{n,\sigma}.
\end{equation}
$iii)$ For any $F\in C^\infty(\xR^N)$ satisfying $F(0)=0$, 
there exists a nondecreasing function 
$C\colon\xR_+\rightarrow\xR_+$ such that for any $f\in \eCZ{n,\sigma}^N$, one has
\begin{equation}\label{F:eCZ}
\triple{F(f)}{n,\sigma}\le C(\triple{f}{n,\sigma})\triple{f}{n,\sigma}.
\end{equation}
\end{lemm}

The bounds involving the preceding norms that we shall obtain below will be deduced from estimates for the action of an element of 
$\Cr_q^p$ on a function given in the following lemma.

\begin{lemm}\label{ref:235D}
Let $\gamma\in ]2,+\infty[\setminus \mez \xN$, $\mu'\in [0,1[$. 

$i)$ Take $\ell,k',p,N$ in $\xN$, 
$q$ in $\xZ$ with $p+q\ge 0$ and $C$ an element of $\Cr_q^p[N]$. 
For any $N'\ge N$, any integer $h$ with $0\le h\le \ell$, any $i',j'$ with $i'+j'\le k'$ define,
\be\label{2311l}
\ba
\Ir(N',h,i',j')=\Big\{ &(p_1,\ldots, p_{N'},q_1,\ldots , q_{N'})\in \xN ^{N'}\times \xZ^{N'}\, ;\, \\
& \sum_{r=1}^{N'}(p_r+q_r)+(i'+j'+h)\le p+q+k'+\ell\\
& \sum_{r=1}^{N'}p_r+i'\le p+k'\\
&p_r+q_r\ge 0,~q_r\ge -1,~r=1,\ldots, N'\Big\}.
\ea
\ee
For $I$ an element of $\Ir(N',h,i',j')$ and $(\eta,\widetilde{\psi})$ two functions, smooth 
enough so that the norms below are finite, set $M_{I,2}(\eta,\widetilde{\psi})$ for the minimum of the following 
quantities
\be\label{2311m}
\ba
&\prod_{r=1}^{N'}\blA Z^{p_r}\langle \OD\rangle^{q_r}\eta\brA_{\eC{\gamma}}\blA \px^{j'+h} Z^{i'}\widetilde{\psi}\brA_{H^{\mu'}}\\
&\bigg(\prod_{r\neq r'}\blA Z^{p_r}\langle \OD\rangle^{q_r}\eta\brA_{\eC{\gamma}}\bigg)\blA Z^{p_{r'}}\langle \OD\rangle^{q_{r'}}\eta\brA_{H^{\mu'+1}}
\blA \px^{j'+h} Z^{i'}\widetilde{\psi}\brA_{\eC{\gamma-1}},\quad 1\le r'\le r.
\ea
\ee
Then 
\be\label{2311n}
\blA \px^{\ell} Z^{k'}C\widetilde{\psi}\brA_{H^{\mu'}}\le C(\eta)
\sum_{\substack{N'\ge N\\ \text{finite}}}\sum_{\substack{h\le \ell \\ i'+j'\le k'}}
\sum_{I\in \Ir(N',h,i',j')}M_{I,2}(\eta,\widetilde{\psi})
\ee
where the first sum is finite and where $C(\eta)$ depends only on $\lA \eta\rA_{\eC{\gamma}}$. If $\sigma'$ is a real 
number with $0<\sigma'<1$, $\sigma'\neq \mez$, if we define $M_{I,\infty}(\eta,\widetilde{\psi})$ by the minimum 
of the quantities obtained replacing $H^{\mu'}$ by $C^{\sigma'}$ and $H^{\mu'+1}$ by $\eC{\sigma'+1}$ in 
\e{2311m} we have also
\be\label{2311o}
\blA \px^{\ell} Z^{k'}C\widetilde{\psi}\brA_{\eC{\sigma'}}\le C(\eta)
\sum_{\substack{N'\ge N\\ \text{finite}}}\sum_{\substack{k\le \ell \\ i'+j'\le k'}}
\sum_{I\in \Ir(N',h,i',j')}M_{I,\infty}(\eta,\widetilde{\psi}).
\ee

$ii)$ Assume that $C$ is in $\tmc q p N$. Then 
$\blA \px^\ell Z^{k'}\Dx^{-\mez} C\widetilde{\psi}\brA_{H^{\mu'-\mez}}$ 
is bounded from above by the 
\rhs of \e{2311n} and, for any $\theta>0$, $\blA \px^\ell Z^{k'} \Dx^{-\mez+\theta}C\widetilde{\psi}\brA_{\eC{\sigma'-\mez+\theta}}$ 
is bounded from above by the 
\rhs of \e{2311o}.
\end{lemm}
\begin{proof}
$i)$ Apply \e{2311d} to write
\be\label{2311obis}
\px^\ell Z^{k'} C\widetilde{\psi}=\sum_{\substack{i'+j' \le k' \\ h\le \ell}}C_{j',h}^{i'} \px^{j'+h}Z^{i'}\widetilde{\psi}
\ee
with $C_{j',h}^{i'}\in \mc{q+\ell-h-j'}{p+k'-i'}{N}$. Let us bound
$$
\blA C_{j',h}^{i'} \px^{j'+h}Z^{i'}\widetilde{\psi}\brA_{H^{\mu'}}.
$$
By Definition~\ref{ref:235A}, $C_{j',h}^{i'}$ may be written from expressions of the form~\e{2311b} with $N'\ge N$ and with the indices 
$(p_1,\ldots,p_{N'};q_1,\ldots,q_{N'})$ satisfying inequalities \e{2311l}. Since $\gamma>\mu'+\tdm$, the operators 
$E_0,\ldots,E_{N'}$ in \e{2311b} are bounded in $H^{\mu'}$ and in $\eC{\gamma-1}$ (see Proposition~\ref{ref:116} and Corollary~\ref{ref:118}). 
Moreover, by property \e{pr:sz} of the appendix, we have the estimate $\lA bv\rA_{H^{\mu'}}\les \lA b\rA_{\eC{\gamma-1}}\lA v\rA_{H^{\mu'}}$. 
We apply this to bound the action of \e{2311b} on $\px^{j'+h}Z^{i'}\widetilde{\psi}$. If we estimate the 
$Z^{p_r}b_{q_r}(\OD)a_r$ terms in $\eC{\gamma-1}$ and $\px^{j'+h}Z^{i'}\widetilde{\psi}$ in $H^{\mu'}$, we get a bound by
\be\label{2311p}
C(\eta)\prod_{r=1}^{N'}\blA Z^{p_r}b_{q_r}(\OD)a_r\brA_{\eC{\gamma-1}}\blA \px^{j'+h}Z^{i'}\psi\brA_{H^{\mu'}}. 
\ee
On the other hand, if we estimate the $Z^{p_{r'}}\langle \OD\rangle^{q_{r'}}a_{r'}$-factor 
in $H^{\mu'}$ and the other ones in $\eC{\gamma-1}$, we get as well a bound
\be\label{2311q}
C(\eta)\prod_{\substack{1\le r\le N' \\ r\neq r'}}\blA Z^{p_r}b_{q_r}(\OD)a_r\brA_{\eC{\gamma-1}}\blA Z^{p_{r'}}b_{q_{r'}}(\OD)a_{r'}\brA_{H^{\mu'}}
\blA \px^{j'+h}Z^{i'}\psi\brA_{\eC{\gamma-1}}
\ee
with a constant $C(\eta)$ depending only on $\lA \eta\rA_{\eC{\gamma}}$. 
Let us remark that we have the estimates
\be\label{2311r}
\begin{split}
\blA Z^{p_r}b_{q_r}(\OD)a_r\brA_{\eC{\gamma-1}}\le C(\eta)\sum_{\substack{p_{r_1}+\cdots +p_{r_\ell}\le p_r\\
\sum (p_{r_j}+q_{r_j})\le p_r+q_r\\
p_{r_j}+q_{r_j}\ge 0,~q_{r_j}\ge -1}}\prod_{j=1}^\ell \blA Z^{p_{r_j}}\langle \OD\rangle^{q_{r_j}}\eta\brA_{\eC{\gamma}}\\
\ba
\blA Z^{p_r}b_{q_r}(\OD)a_r\brA_{H^{\mu'}}\le C(\eta)\sum_{\substack{p_{r_1}+\cdots +p_{r_\ell}\le p_r\\
\sum (p_{r_j}+q_{r_j})\le p_r+q_r\\
p_{r_j}+q_{r_j}\ge 0,~q_{r_j}\ge -1}}\min_{1\le j'\le \ell} &\prod_{\substack{1\le j\le \ell \\ j\neq j'}}
\blA Z^{p_{r_j}}\langle \OD\rangle^{q_{r_j}}\eta\brA_{\eC{\gamma}}\\
&\qquad\times \blA Z^{p_{r_{j'}}}\langle \OD\rangle^{q_{r_{j'}}}\eta\brA_{H^{\mu'+1}}.
\ea
\end{split}
\ee
Actually, we notice first that $[Z,b_{q_r}(\OD)]=\widetilde{b_{q_r}}(\OD)$ for another symbol of the same order as~$b_{q_r}$. 
Consequently, we may as well estimate the norm of $b_{q_r}(\OD)Z^{p'_r}a_r$ for $p'_r\le p_r$. if $q_r\ge 0$, we are reduced 
to estimating $\blA \px^{q'_r}Z^{p'_r}a_r\brA_{\eC{\gamma-1}}$ and $\blA \px^{q'_r}Z^{p'_r}a_r\brA_{H^{\mu'}}$ for 
$q'_r\le q_r$, $p'_r\le p_r$. Since $a_r$ is an analytic function of $\eta,\eta'$, we express the quantities inside the norm as 
a sum of expressions $\tilde{a}_r(\eta,\eta') \big(\px^{q'_1}Z^{p'_{r_1}}\widetilde{\eta}\big)\cdot\big(\px^{q'_{r_\ell}}Z^{p'_{r_\ell}}\widetilde{\eta}\big)$ 
where $\widetilde{a}_r$ is some new analytic function, $q'_1+\cdots +q'_{r_\ell}\le q'_r$, 
$p'_1+\cdots +p'_{r_\ell}\le p'_r$, and $\widetilde{\eta}=\eta$ or $\eta'$. Using that $\eC{\gamma-1}$ is an algebra, we obtain the first estimate. The second one follows 
from the inequality $\lA ab\rA_{H^{\mu'}}\le C\lA a\rA_{\eC{\gamma-1}}\lA b\rA_{H^{\mu'}}$ which holds since $\gamma-1>\mu'\ge 0$.

Consider now the case $q_r=-1$, so that $p_r\ge 1$ and we have to estimate $\blA Z^{p_r}a_r\brA_{\eC{\gamma-2}}$ and 
$\blA Z^{p_r}a_r\brA_{H^{\mu'-1}}$. As $\eC{\gamma-2}$ is also an algebra, the first estimate \e{2311r} follows. The second is a 
consequence of the inclusions $\eC{\gamma-1}\cdot H^{\mu'-1}\subset H^{\mu'-1}$ and $\eC{\gamma-2}\cdot H^{\mu'}\subset H^{\mu'-1}$ which are true 
since $\gamma>2>\mu'+1$. 

We plug estimates \e{2311r} in \e{2311p}, \e{2311q} and obtain the bound \e{2311m}. The inequalities
\e{2311l} follow from \e{2311c}, where we replace $(p,q)$ by $(p+k'-i',q+\ell-h-j')$ and from the conditions on the indices 
in the \rhs of \e{2311r}. Estimate \e{2311o} 
is obtained in the same way.

$ii)$ If we cut-off $C$ for non zero frequencies, then the estimate follows from $i)$. Consequently, we have to study 
$\blA Z^{k'}\Dx^{-\mez}\chi(\OD)C\wpsi\brA_{L^2}$ and $\blA Z^{k'}\Dx^{-\mez+\theta}\chi(\OD)C \wpsi\brA_{L^\infty}$, 
where $\chi$ is in $C^\infty_0(\xR)$, $\chi\equiv 1$ close to zero. 
By \e{2311obis}, and the fact that $[Z,\chi(\OD)]=\chi_1(\OD)$ for some 
$C^\infty_0(\Rs)$ function $\chi_1$, we are reduced to the study of $\Dx^{-\mez+\theta}\chi(\OD)C_{j'}^{i'}\px^{j'}Z^{i'}\wpsi$, where 
according to the last statement in Proposition~\ref{ref:235B}, we may assume that $C_{j'}^{i'}$ belongs to 
$\tmc{q-j'}{p+k'-i'}{N}$. This means that this operator may be written as a linear 
combination of expressions \e{2311b}, with 
$N'\ge N$, indices $(p_1,\ldots,p_{N'},q_1,\ldots ,q_{N'})$ satisfying \e{2311c} and $E_0$ in $\widetilde{\Er}$, i.e.\ $E_0=G(\eta)\Dx^{-\mez}\langle \OD\rangle^{-\mez}E$ 
or $E_0=b_0(\OD)E$, where $E$ is in $\Er$ and $b_0(\OD)$ is a Fourier multiplier homogeneous 
of degree larger or equal to $1/2$ close to zero. It follows from Proposition~\ref{ref:116} that $\Dx^{-\mez}\chi(\OD)E_0$ is bounded on 
$L^2$ and $\Dx^{-\mez+\theta}\chi(\OD)E_0$ is bounded on H\"older spaces if $\theta>0$. Consequently, estimates \e{2311p}, \e{2311q} still 
hold for the building blocks of $\Dx^{-\mez}\chi(\OD)C_{j'}^{i'}\px^{j'}Z^{i'}\wpsi$, which gives the wanted Sobolev estimate. The case of the H\"older 
bound is similar for 
$\Dx^{-\mez+\theta}\chi(\OD)C_{j'}^{i'}\px^{j'}Z^{i'}\wpsi$, $\theta>0$.
\end{proof}

We may prove now the main result of this section, which gives estimates for the action of $Z^k$ on $G(\eta)\psi$, 
$\B(\eta)\psi$, $V(\eta)\psi$. 

\begin{prop}\label{ref:235E}
Let $\gamma,\gamma_0$ be given in $]0,+\infty[\setminus \mez\xN$ with $\gamma\ge \gamma_0>2$ and 
let $s_0,s_1,s$ be integers satisfying 
$$
s\ge s_1\ge s_0\ge \mez (s+2\gamma-1).
$$
Let $k$ be in $\xN^*$, $\mu$ in $\xR_+$ with $\mu+k\le s-1$. Let $(\psi,\eta)$ be in $\Eg{\gamma}$ 
and in $\h{\mez,k,\mu+\mez}\times H^{k,\mu+1}$, smooth enough so that the norms in the inequality below 
are all finite. Let $A(\eta)$ be one of the operators 
$G(\eta)$, $\B(\eta)$, $V(\eta)$. There is a non increasing function $C(\cdot)$ such that, for any $(\eta,\psi)$ as above
\be\label{2311s}
\ba
&\blA \bigl( Z^k A(\eta)-A(\eta)(Z-2)^k \bigr)\psi\brA_{H^\mu}\\
&\qquad\qquad \le \indicator{\xR_+}(\mu+k-s_0+\gamma_0)
C\bigl(\lA \eta\rA_{\eC{\gamma}}\big) \blA \Dxmez \psi\brA_{\eC{\gamma}}\blA Z^k \eta\brA_{H^{\mu+1}}\\
&\qquad \qquad \quad +C\bigl( \triple{\eta}{\bar{k},s_0-\bar{k}}\bigr)\triple{\eta}{\bar{k},s_0-\bar{k}}\bdouble{\Dxmez\psi}{k-1,\mu+\tdm}\\
&\qquad \qquad \quad +C\bigl( \triple{\eta}{\bar{k},s_0-\bar{k}}\bigr)\btriple{\Dxmez \psi}{\min (\mu+k-s_0+\gamma_0,k),\gamma}\double{\eta}{k-1,\mu+2}\\
&\qquad \qquad \quad +\indicator{\xR_+^*}\big([\mu]-(\gamma-\gamma_0)\big) 
C\bigl( \triple{\eta}{\bar{k},s_0-\bar{k}}\bigr)
\btriple{\Dxmez \psi}{\min(\mu+k-s_0+\gamma_0,k),\gamma}\double{\eta}{k,\mu}
\ea
\ee
where we have denoted by $\triple{\cdot}{*,*}$, $\double{\cdot}{*,*}$ the norms defined in 
Definition~$\ref{ref:236}$, $\indicator{\xR_+}$ is the indicator function of $\xR_+$, 
$\btriple{\Dxmez\psi}{\mu+k-s_0-\gamma_0,\gamma}$ should be understood as zero 
if $\mu+k-s_0+\gamma_0<0$ and where $\bar{k}=\min(k,s_0)$.
\end{prop}
\begin{rema*}
The key properties in \e{2311s} is the fact that the terms involving $kZ$-derivatives of $\eta$ in the \rhs are multiplied 
by specific factors, well tailored for the induction argument that will be used in section~\ref{S:236} and in Chapter~\ref{S:24}.
\end{rema*}
\begin{proof}
Let us show first that 
\be\label{2311w}
\bigl( Z^k A(\eta)-A(\eta)(Z-2)^k \bigr)\psi=\sum_{\substack{i\le k-1\\ i+j\le k}}C_j^i \px^j Z^i \Dx^{\mez}\langle \OD\rangle^{\mez}\psi
\ee
where $C_j^i$ belongs to $\mc{-j}{k-i}{1}$. 

Consider first the case $k=1$, $A(\eta)=G(\eta)$. Then
$$
\Bigl(ZG(\eta)-G(\eta)(Z-2)\Bigr)\psi=\Bigl([Z,G(\eta)]+2G(\eta)\Bigr)\psi
$$
may be computed from 
\e{2311gbis} as a sum of expressions of type
$$
\widetilde{E}_0 \px \Big[ \big(Z^\alpha \eta\big) E_0 \Dxmez \langle \OD\rangle^{\mez}\psi\Big],\quad 
\widetilde{E}_0 \Big[ \big( Z^\alpha \eta\big) E_0 \Dxmez \langle \OD\rangle^\mez \psi\Big]
$$
with $E_0$, $\widetilde{E}_0$ in $\Er$, $\alpha=0,1$. This, together with the second commutation relation \e{2311e} shows that 
$\big[ ZG(\eta)-G(\eta)(Z-2)\big]\psi$ may be written as $C_0^0 \Dxmez \langle \OD\rangle^\mez \psi +C_1^0\px \Dxmez \langle \OD\rangle^\mez \psi$ with 
$C_0^0$ in $\mc{0}{1}{1}$, $C_1^0$ in $\mc{-1}{1}{1}$. If now $A(\eta)$ is equal to 
$B(\eta)=(1+\eta'^2)^{-1}G(\eta)+\eta'(1+\eta'^2)^{-1}\px$, we see that 
$\big[ Z A(\eta)-A(\eta)(Z-2)\big]\psi$ is the sum of the product of the \rhs of \e{2311w} with $k=1$ by $(1+\eta'^2)^{-1}$, which is still of the same form, 
and of the quantities 
$$
-2(1+\eta'^2)^{-2}\eta'(Z\eta')G(\eta)\psi,\quad Z\big( \eta'(1+\eta'^2)^{-1}\big)\px \psi
$$
which may be written as $C_0^0 \Dxmez \langle \OD\rangle^\mez\psi$ for some $C_0^0$ in $\mc{0}{1}{1}$. Consequently, \e{2311w} with $k=1$ holds as well 
when $A(\eta)=\B(\eta)$. The same conclusion holds for $V(\eta)=\px-\eta'\B(\eta)$ since $Z\px-\px(Z-2)=0$. We have thus proved \e{2311w} when $k=1$. 
Let us prove that this equality holds for any $k$ by induction. We write from \e{2311w}
\begin{align*}
\bigl( Z^{k+1} A(\eta)-A(\eta)(Z-2)^{k+1} \bigr)\psi&=
\sum_{\substack{i\le k-1\\ i+j\le k}}Z C_j^i \px^j Z^i \big( \Dxmez \langle \OD\rangle^{\mez}\psi\big)\\
&\quad + \bigl( Z A(\eta)-A(\eta)(Z-2)\bigr)(Z-2)^k\psi.
\end{align*}
It follows from 
\e{2311d}Ê
with $k=1$, $\ell=0$, that the first sum is of the form of the \rhs of \e{2311w} with $k$ replaced by 
$(k+1)$. Moreover, the last term may be written
$$
\big(C_1^0 \px +C_0^0\big)\Dxmez \langle \OD\rangle^{\mez}(Z-2)^k \psi
$$
with $C_1^0$ in $\Cr_{-1}^1$, $C_0^0$ in $\Cr_0^1$. Commuting $\Dxmez\langle \OD\rangle^{\mez}$ to the powers 
of $(Z-2)$, we see that we get again a contribution of the wanted form.

We may now prove \e{2311s}. We write $\mu=[\mu]+\mu'$ with $\mu'\in [0,1[$. According to \e{2311w}, we have to bound for any 
$\ell=0,\ldots, [\mu]$,
\be\label{2311w1}
\blA \px^\ell C_j^i \px^j Z^i\Dxmez \langle \OD\rangle^{\mez}\psi\brA_{H^{\mu'}}
\ee
where $i\le k-1$, $i+j\le k$, $C_j^i$ in $\mc{-j}{k-i}{1}$. We apply estimate 
\e{2311n} with $\gamma$ replaced by $\gamma_0$, $k'=0$, $p=k-i$, $q=-j$, $N\ge 1$, $\wpsi=\px^jZ^i \Dxmez \langle \OD\rangle^\mez \psi$. 

We obtain a bound in terms of a sum for $N'\ge 1$, $h\le \ell$ of the minimum of quantities \e{2311m}Ê
where we set $j'=i'=0$ i.e.
\begin{align}
&\prod_{r=1}^{N'}\blA Z^{p_r}\langle \OD\rangle^{q_r}\eta\brA_{\eC{\gamma_0}}\blA \px^{h+j}Z^i \Dxmez \langle \OD\rangle^\mez \psi\brA_{H^{\mu'}},\label{2311x} \\
&\biggl(\prod_{r\neq r'}\blA Z^{p_r}\langle \OD\rangle^{q_r}\eta\brA_{\eC{\gamma_0}}\biggr)
\blA Z^{p_{r'}}\langle \OD\rangle^{q_{r'}}\eta\brA_{H^{\mu'+1}}
\blA \px^{h+j}Z^i \Dxmez \langle \OD\rangle^{\mez} \psi\brA_{\eC{\gamma_0-1}},\label{2311y}
\end{align}
where the indices have to obey the restrictions deduced from \e{2311l}, namely
\be\label{2311z}
\ba
& \sum_{r=1}^{N'}(p_r+q_r)+(i+j+h)\le k+\ell\\
& \sum_{r=1}^{N'}p_r+i\le k\\
&p_r+q_r\ge 0,~q_r\ge -1,~r=1,\ldots, N'.
\ea
\ee
To finish the proof of estimate~\e{2311s} we have to bound \e{2311w1} by one of the four terms $I$, $II$, $III$, $IV$ of the 
\rhs of \e{2311s}. We distinguish several cases. 

\underline{Case 1:} For any $r=1,\ldots,N'$, $p_r+q_r\le s_0-\gamma_0$. 

In this case, we use \e{2311x}. Since $p_r\le k$, we may bound 
$\blA Z^{p_r}\langle \OD\rangle^{q_r}\eta\brA_{\eC{\gamma_0}}$ by 
$\triple{\eta}{\bar{k},s_0-\bar{k}}$. Moreover, since the exponent $i$ in \e{2311w1} is smaller than $k-1$, and since \e{2311z} implies 
$i+j+h\le k+l\le k+[\mu]$, the last factor in \e{2311x} is bounded by $\bdouble{\Dxmez\psi}{k-1,\mu+\tdm}$. We see that we obtain a bound by~$II$. 

From now on we may assume that there is some $r$, say $r=1$, with $p_1+q_1>s_0-\gamma_0$. Notice that \e{2311z} implies then that 
for $r>1$
\be\label{2311zz}
p_r+q_r\le k+\ell-p_1-q_1<k+[\mu]-(s_0-\gamma_0)\le s_0-\gamma_0
\ee
where the last inequality follows from the assumptions $k+\mu\le s-1$ and the inequalities between $s$ and $s_0$. 

\underline{Case 2:} $p_1=k$ and $j+h\le \gamma-\gamma_0$, $q_r\le \gamma-\gamma_0$, $r>1$. 

Since $p_1=k$, the second inequality \e{2311z} implies that $i=0$, $p_r=0$ for $r>1$. We use the bound \e{2311y} with $r'=1$. For 
$r>1$, we estimate $\blA Z^{p_r}\langle \OD\rangle^{q_r}\eta\brA_{\eC{\gamma_0}}=\blA \langle \OD\rangle^{q_r}\eta\brA_{\eC{\gamma_0}}\le \lA \eta\rA_{\eC{\gamma}}$ according 
to the assumption on $q_r$. In the same way $\blA \px^{j+h}\Dxmez \langle \OD\rangle^{\mez}\psi\brA_{\eC{\gamma_0-1}}$ is bounded from 
$\blA \Dxmez \psi\brA_{\eC{\gamma-\mez}}$. If we notice that $\blA Z^{p_1}\langle \OD\rangle^{q_1}\eta\brA_{H^{\mu'+1}}\le \blA Z^{k}\eta\brA_{H^{\mu+1}}$, 
using that the first relation \e{2311z} implies $q_1\le \ell\le [\mu]$, we conclude that we obtain a bound by $I$. 
The cut-off for $k+\mu-s_0+\gamma_0\ge 0$ comes from the fact that by \e{2311z} and our assumption on $p_1,q_1$, we have 
$s_0-\gamma_0<p_1+q_1\le [\mu]+k$.

\underline{Case 3:} $p_1=k$ and either $j+h>\gamma-\gamma_0$ or there is $r>1$ with $p_r+q_r>\gamma-\gamma_0$. 

We notice that, as $p_1=k$, inequalities \e{2311z} implies $q_r\le [\mu]$ for any $r$ and $j+h\le [\mu]$. The assumptions of this case imply that $\gamma-\gamma_0<[\mu]$ 
so that the cut-off condition 
in the term $IV$ in the \rhs of \e{2311s} holds. We notice also that $q_1<[\mu]$: if not, the first inequality \e{2311z} and 
$p_1=k$, would imply that $j+h=0$ and $q_r=0$ for $r>1$, which would contradict the assumptions of this case. It follows that, in \e{2311y} 
with $r'=1$, $\blA Z^{p_1}\langle \OD\rangle^{q_1}\eta\brA_{H^{\mu'+1}}\le \double{\eta}{k,\mu}$. Moreover using \e{2311zz}, we estimate for 
$r>1$ $\blA Z^{p_r}\langle \OD\rangle^{q_r}\eta\brA_{\eC{\gamma_0}}$ by $\triple{\eta}{\bar{k},s_0-\bar{k}}$. Finally, since \e{2311z} implies that
$$
i+j+h\le k+\ell-(p_1+q_1)\le k+[\mu]-(s_0-\gamma_0)
$$
taking into account the assumption made after the conclusion of case $1$, we may bound $\blA \px^{h+j}Z^i\Dxmez \langle \OD\rangle^{\mez}\psi\brA_{\eC{\gamma_0-1}}$ 
by $\btriple{ \Dxmez \psi}{\min(k+[\mu]-s_0+\gamma_0,k),\gamma_0-\mez}$. We obtain a contribution to the term $IV$ in \e{2311s}.

\underline{Case 4:} $p_1<k$.

We use \e{2311y} with $r'=1$. As above, the last factor in this inequality is bounded from above by $\btriple{ \Dxmez \psi}{\min(k+[\mu]-s_0+\gamma_0,k),\gamma_0-\mez}$ 
and for $r>1$, $\blA Z^{p_r}\langle \OD\rangle^{q_r}\eta\brA_{\eC{\gamma_0}}\les \triple{\eta}{\bar{k},s_0-\bar{k}}$. 
Since \e{2311z} implies $p_1+q_1\le [\mu]+k$ and since $p_1<k$, $\blA Z^{p_1}\langle \OD\rangle^{q_1}\eta\brA_{H^{\mu'+1}}$ is smaller than 
$\double{\eta}{k-1,\mu+2}$. We thus get a contribution to term $III$ in \e{2311s}. 

This concludes the proof.
\end{proof}

\begin{coro}\label{ref:235E1}
Under the assumptions of Proposition~\ref{ref:235E} and if moreover $\gamma\ge 4$
\be\label{2311aa}
\ba
&\bdouble{\big(A(\eta)-A(0)\big)\psi\bigr)}{k,\mu}\\
&\qquad\qquad \le C\bigl(\lA \eta\rA_{\eC{\gamma}}\bigr) \lA \eta\rA_{\eC{\gamma}}\blA \Dxmez Z^k\psi\brA_{H^{\mu+\mez}}\\
&\qquad\qquad \quad+\indicator{\xR_+^*}([\mu]-(\gamma-\gamma_0))
C\big(\triple{\eta}{s_0,0}\big)\triple{\eta}{s_0,0}\bdouble{\Dxmez\psi}{k,\mu-\mez}\\
&\qquad \qquad \quad +\indicator{\xR_+}(\mu+k-s_0+\gamma_0)C\bigl(\lA \eta\rA_{\eC{\gamma}}\bigr) \blA \Dxmez \psi\brA_{\eC{\gamma}}\blA Z^k \eta\brA_{H^{\mu+1}}\\
&\qquad \qquad \quad +C\bigl( \triple{\eta}{s_0,0}\bigr)\triple{\eta}{s_0,0}\bdouble{\Dxmez\psi}{k-1,\mu+\tdm}\\
&\qquad \qquad \quad +C\bigl( \triple{\eta}{s_0,0}\bigr)\btriple{\Dxmez \psi}{\mu+k-s_0+\gamma_0,\gamma}\double{\eta}{k-1,\mu+2}\\
&\qquad \qquad \quad +\indicator{\xR_+^*}([\mu]-(\gamma-\gamma_0))
C\bigl( \triple{\eta}{s_0,0}\bigr)\btriple{\Dxmez \psi}{\mu+k-s_0+\gamma_0,\gamma}\double{\eta}{k,\mu}.
\ea
\ee
\end{coro}
\begin{proof}
We have to bound $\blA Z^k (A(\eta)-A(0))\psi\brA_{H^\mu}$. Since $A(0)=\Dx$ if $A=G$ or $\B$ and $V(0)=\px$, we have 
$Z^k A(0)=A(0)(Z-2)^k$. It follows that $Z^k (A(\eta)-A(0))-(A(\eta)-A(0))(Z-2)^k$ is estimated by \e{2311s}. We just need to study 
\be\label{2311ab}
\blA (A(\eta)-A(0))(Z-2)^k\psi\brA_{H^\mu}.
\ee
Assume first that $\mu\ge s_0-\gamma_0$. When $A(\eta)=G(\eta)$, apply \e{n129} with $(\mu,s)$ replaced by $(\mu+1,\mu+1)$ and $\gamma$ replaced by $\gamma_0$. We obtain a bound by 
$$
C\big(\lA \eta\rA_{\eC{\gamma_0}}\big)\Big[ 
\blA \Dxmez (Z-2)^k\psi\brA_{\eC{\gamma_0-\mez}}
\lA \eta\rA_{H^{\mu+1}}+\lA \eta\rA_{\eC{\gamma_0}}\blA \Dxmez (Z-2)^k\psi\brA_{H^{\mu-\mez}}\Big].
$$ 
The last term is bounded from above by the contributions $I+II$ of the 
right hand side of \e{2311aa}. The first term may be controlled by $V$ since 
$k\le \mu+k-s_0+\gamma_0$ because of our assumption on $\mu$. 
When $A(\eta)=\B(\eta)$ or $V(\eta)$, we argue in the same way applying \e{n134} with $(\mu,s)$ replaced by $(\mu+1/2,\mu+1)$. 

Assume now that $\mu<s_0-\gamma_0$. Set $\widetilde{\psi}=(Z-2)^k\psi$. We want to estimate for $0\le \ell\le [\mu]$
$$
\blA \px^\ell \bigl(A(\eta)-A(0)\bigr)\wpsi\brA_{H^{\mu'}}\le \blA  \bigl(A(\eta)-A(0)\bigr)\px^\ell\wpsi\brA_{H^{\mu'}}
+\blA \big[ \px^\ell,A(\eta)\big] \wpsi\brA_{H^{\mu'}}
$$
with $\mu'=\mu-[\mu]$. The first term in the \rhs may be estimated when $A(\eta)=G(\eta)$ from \e{va2} since $\mu'\le \gamma-3$ 
for $\gamma\ge 4$, so by $I+IV$. If $A=\B$ or $V$, the bound follows from the one of $G$, the expressions of $\B$, $V$ in terms of $G$ and 
the law product $\eC{\gamma-1}\cdot H^{\mu'}\subset H^{\mu'}$. 

Consider now the second term. According to \e{2311ha}, $\big[ \px^\ell, A(\eta)\big]\wpsi$ is a linear combination of quantities of the form 
$$
a(\eta')L\big( \px^{\ell_1}\eta',\ldots, \px^{\ell_N}\eta'\big) \widetilde{A}(\eta)\px^{\ell_{N+1}}\wpsi
$$
where $N\in \xN^*$, $\ell_j\in \xN$ with $\ell_1+\cdots+\ell_{N+1}=\ell\le [\mu]$, 
$\ell_1+\cdots+\ell_N>0$, $L$ is a multilinear form in its arguments, $\widetilde{A}(\eta)$ is taken among $G(\eta)$, $\B(\eta)$, $\px$, and 
$a(\eta')$ is some analytic functions of $\eta'$. Using again the product law $\eC{\gamma-1}\cdot H^{\mu'}\subset H^{\mu'}$, we bound 
the $H^{\mu'}$-norm of the above expression by
$$
C\big(\blA \eta'\brA_{\eC{\gamma_0-1+[\mu]}}\big) \blA \eta'\brA_{\eC{\gamma_0-1+[\mu]}} \blA \widetilde{A}(\eta)\px^{\ell_{N+1}}(Z-2)^k\psi\brA_{H^{\mu'}}.
$$
We use that $\ell_{N+1}\le [\mu]-1$ and \e{1116a} to estimate the last factor by $\bdouble{\Dxmez \psi}{k,\mu-\mez}$. Since 
$\gamma_0+[\mu]<s_0$, we see that we obtain finally a bound by term $II$ in the \rhs of \e{2311aa} when $[\mu]>\gamma-\gamma_0$. If 
$[\mu]\le \gamma-\gamma_0$, we use instead the bound provided by $I$ and $IV$, remembering that we are in the case $\mu<s_0-\gamma_0$. 
This concludes the proof.
\end{proof}

Next we state a corollary of the previous estimate under a form which is convenient for later purposes.

\begin{prop}\label{T56}
$i)$ Under the assumptions of Proposition~$\ref{ref:235E}$ and if moreover $\gamma\ge 4$
\be\label{2311aa-bis}
\ba
\bdouble{A(\eta)\psi}{k,\mu}
&\le C\bigl(\lA \eta\rA_{\eC{\gamma}}\bigr)\blA \Dxmez Z^k\psi\brA_{H^{\mu+\mez}}\\
&\quad+\indicator{\xR_+^*}([\mu]-(\gamma-\gamma_0))
C\big(\triple{\eta}{s_0,0}\big)\triple{\eta}{s_0,0}\bdouble{\Dxmez\psi}{k,\mu-\mez}\\
& \quad +\indicator{\xR_+}(\mu+k-s_0+\gamma_0)C\bigl(\lA \eta\rA_{\eC{\gamma}}\bigr) \blA \Dxmez \psi\brA_{\eC{\gamma}}\blA Z^k \eta\brA_{H^{\mu+1}}\\
&\quad +C\bigl( \triple{\eta}{s_0,0}\bigr)\bdouble{\Dxmez\psi}{k-1,\mu+\tdm}\\
& \quad +C\bigl( \triple{\eta}{s_0,0}\bigr)\btriple{\Dxmez \psi}{\mu+k-s_0+\gamma_0,\gamma}\double{\eta}{k-1,\mu+2}\\
&\quad +\indicator{\xR_+^*}([\mu]-(\gamma-\gamma_0))
C\bigl( \triple{\eta}{s_0,0}\bigr)\btriple{\Dxmez \psi}{\mu+k-s_0+\gamma_0,\gamma}\double{\eta}{k,\mu}.
\ea
\ee

$ii)$ Under the assumptions of Proposition~$\ref{ref:235E}$ and if moreover $\gamma\ge 4$
\begin{equation}\label{z:Q}
\bdouble{A(\eta)\psi}{k,\mu} 
\le \Cr \bdouble{\Dxmez \psi}{k,\mu+\mez}
+\Cr \btriple{\Dxmez\psi}{\mu+k-s_0+\gamma_0,\gamma}\double{\eta}{k,\mu+1},
\end{equation}
where $\Cr=C( \triple{\eta}{s_0,0} )$. 
\end{prop}
\begin{proof}
The first inequality follows from \e{2311aa} and the triangle inequality and 
the second inequality follows from \e{2311aa-bis} and the definitions of the norms $\triple{\cdot}{*;*}$ 
and $\double{\cdot}{*,*}$. 
\end{proof}
\begin{rema}
The key point is that, in the right-hand side 
of \eqref{2311aa-bis}, \eqref{z:Q} when say $k\sim s$, the factors estimated in H\"older norms contain at most 
$s/2+{\rm Cst}$ $Z$-derivatives.
\end{rema}

The method of proof used above provides as well H\"older estimates.
\begin{prop}\label{T52}
Let $\gamma\in\xN$ with $\gamma\ge 4$. There exists $\eps_0>0$ such that 
for all integer $k\in [0,\gamma-4]$ and all numbers $\sigma$ in $]3,\gamma-k]$, 
$\sigma\not\in \mez\xN$, there exists an increasing function $C\colon \xR_+\rightarrow 
\xR_+$ such that, for all $T>0$, all $\psi$ in $\C{\mez,k,\sigma+\mez}([0,T]\times \xR)$ and all $\eta$ in $C^{k,\sigma+1}([0,T]\times\xR)\cap C^{0,\gamma+1}([0,T]\times \xR)$ 
satisfying 
$\sup_{t\in [1,T]}\lA \eta(t)\rA_{\eC{\gamma+1}}\le \eps_0$, one has
\begin{equation}\label{Z:1}
\blA Z^k A(\eta)\psi-A(\eta)(Z-2)^k \psi\brA_{\eC{\sigma}}
\le C\bigl( \triple{\eta}{k,\sigma+1}\bigr) 
\triple{\eta}{k,\sigma+1}\btriple{\Dxmez \psi}{k-1,\sigma+\tdm}
\end{equation}
and
\be\label{Z:2}
\triple{A(\eta)\psi}{k,\sigma}
\le C\bigl( \triple{\eta}{k,\sigma+1}\bigr) 
\btriple{\Dxmez \psi}{k,\sigma+\mez}
\ee
for any $A\in \{G,B,V\}$.
\end{prop}
\begin{proof}
Write $\sigma=[\sigma]+\sigma'$ with $\sigma'\in ]0,1[$. From expression 
\e{2311w}, we see that it is enough to bound for $\ell=0,\ldots,[\sigma]$
$$
\blA \px^{\ell}C_j^i \px^j Z^i \Dxmez \langle D_x\rangle^\mez \psi \brA_{\eC{\sigma}}
$$
with $i\le k-1$, $i+j\le k$, $C_{j}^i$ in $\mc{-j}{k-i}{1}$. We apply estimate 
\e{2311o} with $\gamma$ replaced by $\gamma_0$, $\gamma_0>2$ close to 
$2$, $k'=0$, $p=k-i$, $q=-j$, $\wpsi=\px^j Z^i \Dxmez \langle D_x\rangle^{\mez}\psi$. We 
obtain a bound in terms of the minimum of the quantities
\be\label{ts}
\ba
&\prod_{r=1}^{N'}\blA Z^{p_r}\langle D_x\rangle^{q_r} \eta\brA_{\eC{\gamma_0}}
\blA \px^{h+j} Z^{i}\Dxmez \langle D_x\rangle^{\mez} \psi\brA_{\eC{\sigma'}},\\
&\bigg(\prod_{r\neq r'}\blA Z^{p_r}\langle D_x\rangle^{q_r} \eta\brA_{\eC{\gamma_0}}\bigg)
\blA Z^{p_{r'}}\langle D_x\rangle^{q_{r'}} \eta\brA_{\eC{\sigma'+1}}
\blA \px^{h+j} Z^{i}\Dxmez \langle D_x\rangle^{\mez} \psi\brA_{\eC{\gamma_0-1}},\\
\ea
\ee
where the exponents satisfy \e{2311z}. 

If for $r=1,\ldots,N'$ we have $p_r+q_r+\gamma_0\le \sigma+1$, we use the first bound. 
Since $h+j+i\le k+\ell\le k+[\sigma]$ and $i\le k-1$, we get the wanted inequality \e{Z:1}. 

If for some $r'$, for instance $r'=1$, $p_1+q_1+\gamma_0>\sigma+1$, then for all 
$r\ge 2$
$$
p_r+q_r\le k+\ell+\gamma_0-\sigma-1\le k+\gamma_0-1\le k+\sigma+1-\gamma_0
$$
since, taking $\gamma_0$ close enough to $2$, we may assume $2\gamma_0\le \sigma+2$. Similarly, $i+j+h\le k+\sigma+1-\gamma_0$. We use the second bound \e{ts} 
with $r'=1$. Since $p_1+q_1\le k+\ell\le k+[\sigma]$ and $i\le k-1$, we obtain 
\e{Z:1}. Estimate~\e{Z:2} follows from \e{Z:1} and Corollary~\ref{ref:118}.
\end{proof}

The second objective of this section is to obtain estimates for the remainder in the Taylor development at zero of $\eta\rightarrow G(\eta)$. 

Let us introduce a notation: if $\Zr$ denotes the couple $(Z,\px)$, and if $k$ is in $\xN$, we set $\Zr^k$ for the family $(Z^{k'}\px^{k''}u)_{k'+k''\le k}$. 

\begin{prop}\label{ref:235F}
Let $m$ be in $\xN$, $m\ge 3$. There is a positive constant $\alpha$ such that the following holds: For any family 
$(A_j(\eta))_{1\le j\le m}$ of operators with $A_2(\eta),\ldots,A_m(\eta)$ taken among $a(\eta,\eta')G(\eta)$, 
$a(\eta,\eta')\B(\eta)$, $a(\eta,\eta')V(\eta)$, $a(\eta,\eta')\px$, where $a$ is an analytic function of $(\eta,\eta')$ vanishing at zero and such that 
$A_1(\eta)=G(\eta)$ or $\px$, for any $k\in \xN$, any $d\in\xN$, for any $u=\Dxmez \psi+i\eta$ such that 
$\sup_{t\in [0,T]}\lA \Zr^k u(t,\cdot)\rA_{\eC{d+\alpha}}$ and $\sup_{t\in [0,T]}\lA \Zr^k u(t,\cdot)\rA_{H^{d+\alpha}}$ are finite, 
the following estimates for $R_0(\eta)\defn \Dx^{-\mez}A_1(\eta)\circ \cdots \circ A_m(\eta)$ holds
\be\label{2311alpha}
\ba
&\blA \Zr^k R_0(\eta)\psi\brA_{H^d}\le C[u] \sum_{\substack{k_1+\cdots +k_4\le k\\ k_1,k_2,k_3\le k_4}}
\prod _{j=1}^3\blA \Zr^{k_j}u\brA_{\eC{d+\alpha}}\blA \Zr^{k_4}u\brA_{H^{d+\alpha}},\\
&\blA \Zr^k \Dx^\theta R_0(\eta)\psi\brA_{\eC{d}}\le C[u] \sum_{k_1+\cdots +k_4\le k}
\prod _{j=1}^4\blA \Zr^{k_j}u\brA_{\eC{d+\alpha}}\quad (\theta>0)
\ea
\ee
where $C[u]$ depends only on $\blA \Zr^{(k-1)_+}u\brA_{\eC{d+\alpha}}$ for the first estimate, and on 
$\blA \Zr^{(k-1)_+}u\brA_{\eC{d+\alpha}}$ and on a bound for $\blA \eta'\brA_{H^{-1}}^{1-2\theta'}\blA \eta'\brA_{\eC{-1}}^{2\theta'}$ for some 
$\theta'\in ]0,\theta[$ for the second one.
\end{prop}
\begin{proof}
We may write each of the operators $A_j$ under the form 
$A_j(\eta)=E_j(\eta)\Dxmez\langle \OD\rangle^{\mez}$ with $E_j$ in $\Er$ and $E_1$ in $\widetilde{\Er}$. 
For $j=1,\ldots,m-1$, we decompose $A_j(\eta)=E_j'(\eta)\px+E_j''(\eta)$, with $E_j'$, $E_j''$ in $\Er$, and 
in $\widetilde{\Er}$ if $j=1$. Then 
$$
A_1(\eta)\circ\cdots \circ A_m(\eta)=\prod_{j=1}^{m-1}\bigl( E_j'(\eta)\px+E_j''(\eta)\big)E_m(\eta)\Dxmez \langle \OD\rangle^\mez.
$$
Using the second commutation relation \e{2311e} and the fact that $E_1',E_1''$ are in $\widetilde{\Cr}_0^0$ and $E_j',E_j''$, $j=2,\ldots, m-1$, $E_m$ are in 
$\mc{0}{0}{1}$, we see that $A_1(\eta)\circ\cdots \circ A_m(\eta)$ may be written as a linear combination of operators 
$C(\eta)\px^{\ell'}\Dxmez \langle \OD\rangle^\mez$ where $\ell'\le m-1$ and $C$ is in $\tmc{m-1-\ell'}{0}{m-1}$. We have to estimate, 
in order to study the first inequality \e{2311alpha}, 
$\blA Z^{k'}\Dx^{-\mez}A_1(\eta)\circ\cdots \circ A_m(\eta)\psi\brA_{H^{d+k''}}$ for any decomposition $k=k'+k''$, so to bound 
for $\ell=0,\ldots, d+k''$,
$$
\blA \px^\ell Z^{k'}\Dx^{-\mez}C(\eta)\wpsi\brA_{L^2}
$$
where $\wpsi=\px^{\ell'}\Dxmez\langle \OD\rangle^\mez \psi$. By $ii)$ of Lemma~\ref{ref:235D} (applied with $\mu'=\mez$), we may bound this by the \rhs 
of \e{2311n} i.e.\ by a finite sum indexed by $N'\ge m-1\ge 3$, $i'$, $j'$ with $i'+j'\le k'$ and $h\le \ell$, of the minimum between the quantities 
\e{2311m}, namely
\be\label{2311beta}
\ba
&\prod_{r=1}^{N'}\blA Z^{p_r}\langle \OD\rangle^{q_r}\eta\brA_{\eC{\gamma}}\blA \px^{j'+h} Z^{i'}\px^{\ell'}\Dxmez \langle \OD\rangle^\mez \psi\brA_{H^{\mez}}\\
&\bigg(\prod_{r\neq r'}\blA Z^{p_r}\langle \OD\rangle^{q_r}\eta\brA_{\eC{\gamma}}\bigg)\blA Z^{p_{r'}}\langle \OD\rangle^{q_{r'}}\eta\brA_{H^{\tdm}}
\blA \px^{j'+h} Z^{i'}\px^{\ell'}\Dxmez \langle \OD\rangle^\mez\brA_{\eC{\gamma-1}},
\ea
\ee
where the exponents satisfy the following inequalities
\be\label{2311gamma}
\ba
& \sum_{r=1}^{N'}(p_r+q_r)+(i'+j'+\ell'+h)\le m-1+k'+\ell\\
& \sum p_r+i'\le k',\\
&p_r+q_r\ge 0,~q_r\ge -1.
\ea
\ee
Set $p_0=i'$, $q_0=j'+h+\ell'$, and for $r=0,\ldots,N'$, $k_r=p_r+(q_r-d-m-2)_+$. Then
\be\label{2311delta}
\ba
&\blA Z^{p_r}\langle \OD\rangle^{q_r}\eta\brA_{\eC{\gamma}}\le \blA \Zr^{k_r}u\brA_{\eC{d+\alpha}}\\
&\blA \px^{j'+h}Z^{i'}\px^{\ell'}\Dxmez \langle \OD\rangle^{\mez}\psi\brA_{H^\mez}\le \blA \Zr^{k_0}u\brA_{H^{d+\alpha}}\\
&\blA Z^{p_r}\langle \OD\rangle^{q_r}\eta\brA_{H^\tdm}\le \blA \Zr^{k_r}u\brA_{H^{d+\alpha}}\\
&\blA \px^{j'+h}Z^{i'}\px^{\ell'}\Dxmez \langle \OD\rangle^{\mez}\psi\brA_{\eC{\gamma-1}}\le \blA \Zr^{k_0}u\brA_{\eC{d+\alpha}}.
\ea
\ee
for some $\alpha$ depending only on $\gamma$ and $m$. We notice that if $q_r<d+m+2$, $k_r=p_r\le k'\le k$ and if 
$q_r\ge d+m+2$, $k_r=p_r+q_r-d-m-2\le k-3$ by \e{2311gamma}. We check similarly that $k_0\le k$. Moreover, there is at most one 
$r$ for which $k_r=k$. In the expressions \e{2311beta}, we use \e{2311delta} to bound $N'-3$ factors by $\blA \Zr^{k_r}u\brA_{\eC{d+\alpha}}$, choosing 
those $r$ for which $k_r\le (k-1)_+$, so by 
$\blA \Zr^{(k-1)_+}u\brA_{\eC{d+\alpha}}$. We use the first (resp.\ the second) estimate 
\e{2311beta} when the largest $k_r$ is obtained for 
$r=0$ (resp.\ $r=r'$). Taking \e{2311delta} into account, 
we obtain in all cases a bound
$$
C\bigl(\blA \Zr^{k-1}u\brA_{\eC{d+\alpha}}\big)\prod_{r=1}^3 \blA \Zr^{k_r}u\brA_{\eC{+\alpha}}\blA \Zr^{k_4}u\brA_{H^{d+\alpha}}
$$
with $k_1,k_2,k_3\le k_4$, after renumbering of the $k_j$'s. It follows from \e{2311gamma} that $\sum_1^4(p_r+q_r)\le m-1+k+d$ and 
$\sum_1^4 p_r\le k$. The last inequality implies $\sum_1^4 k_r\le k$ if $q_r-d-m-2\le 0$ for $r=1,\ldots,4$. If there is at least one $r$ for which $q_r-d-m-2>0$ 
we get $\sum_1^4 k_r\le \sum_1^4(p_r+q_r)-d-m+1\le k$. We have obtained the conditions on the summation indices 
in the first inequality. The second inequality is proved in the same way. 
\end{proof}

Let us now state and prove corollaries of the preceding results that will be used in the rest of this paper. We take for $\alpha$ 
the constant given by Proposition~\ref{ref:235F} when $m=3$. We take $s_0$ an integer. We assume that we are given $(\eta,\psi)$ and 
$d\in \xR_+$ with $\eta\in H^{s_0,d+\alpha}\cap C^{s_0,d+\alpha}$ and $\psi$ in $\h{\mez,s_0,d+\alpha}\cap \C{\mez,s_0,d+\alpha}$. 
Then $u=\Dxmez \psi+i\eta$ will satisfy, on the interval $[T_0,T]$ on which it is defined, for any $k\le s_0$,
$$
\sup_{[T_0,T[}\blA \Zr^k u(t,\cdot)\brA_{H^{d+\alpha}}<+\infty, 
\quad \sup_{[T_0,T]}\blA \Zr^k u(t,\cdot)\brA_{\eC{d+\alpha}}<+\infty.
$$

\begin{coro}\label{ref:235G}
Assume that $(\eta,\psi)$ is a solution of the water waves system~\e{121}, satisfying the above smoothness properties. 
Then $u=\Dxmez\psi+i\eta$ satisfies the equation
\be\label{2311eta}
D_t u=\Dxmez u +\mQ+\mC+\widetilde{R}_0(\mU)
\ee
where $\mU=(u,\overline{u})$ and 
\begin{align*}
\mQ&=-\frac{i}{8}\Dxmez \Bigl[ \bigl( \OD \Dx^{-\mez} (u+\bar{u})\bigr)^2
+\bigl(\Dxmez (u+\bar{u})\bigr)^2\Bigr]\\
&\quad +\frac{i}{4} \Dx \bigl( (u-\bar{u})\Dxmez (u+\bar{u})\bigr)-\frac{i}{4}\OD \bigl( (u-\bar{u})\OD \Dx^{-\mez}
(u+\bar{u})\bigr),
\end{align*}
$\mC$ stands for the cubic contribution 
\begin{align*}
\mC&=\frac{1}{8}\Dxmez \Bigl[ \big(\Dxmez (u+\bar{u})\big)\Dx \Big( (u-\bar{u})\Dxmez (u+\bar{u})\Big)\Big]\\
&\quad -\frac{1}{8}\Dxmez \Big[ \big( \Dxmez (u+\bar{u})\Big((u-\bar{u})\Dx^{\tdm}(u+\bar{u})\Big)\Big]\\
&\quad -\frac{1}{8}\Dx \Big[ (u-\bar{u})\Dx \Big((u-\bar{u})\Dxmez (u+\bar{u})\Big)\Big]\\
&\quad +\frac{1}{16}\Dx \Bigl[ (u-\bar{u})^2\Dx^{\tdm}(u+\bar{u})\Bigr]\\
&\quad +\frac{1}{16}\Dx^2 \Bigl[ (u-\bar{u})^2\Dxmez (u+\bar{u})\Bigr]
\end{align*}
Moreover, the remainder $\widetilde{R}_0(\mU)$ satisfies the following bounds: one may write 
$\widetilde{R}_0(\mU)=\Dxmez\pmR$, where for any $k\le s_0$
\be\label{r1}
\blA \Zr^k \pmR\brA_{H^d}\le C_k[u]\sum_{\substack{k_1+\cdots+k_4\le k\\ k_1,k_2,k_3\le k_4}}
\prod_{j=1}^{3}\blA \Zr^{k_j}u\brA_{\eC{d+\alpha}}\blA \Zr^{k_4}u\brA_{H^{d+\alpha}}
\ee
with a constant $C_k[u]$ depending only on $\blA \Zr^{(k-1)_+}u\brA_{\eC{d+\alpha}}$. 
Moreover, for $\theta>0$ small, we get also H\"older estimates
\be\label{r2}
\blA \Zr^k \Dx^\theta \pmR\brA_{\eC{d}}\le C_k[u]\sum_{k_1+\cdots+k_4\le k}
\prod_{j=1}^{4}\blA \Zr^{k_j}u\brA_{\eC{d+\alpha}}
\ee
where $C_k[u]$ depends only on $\blA \Zr^{(k-1)_+}u\brA_{\eC{d+\alpha}}$ and on a bound for 
$\blA \eta'\brA_{H^{-1}}^{1-2\theta'}\blA \eta'\brA_{\eC{-1}}^{2\theta'}$ for some $\theta'\in ]0,\theta[$.
\end{coro}
\begin{proof}
We apply formula \e{n144} with $n=2$. We get
\begin{equation}\label{2311phi}
G(\eta)\psi=
\sum_{k=0}^2 \frac{1}{k!}g^{(k)}(0)+\int_0^1\frac{(\lambda-1)^2}{2}g^{(3)}(\lambda)\, d\lambda
\end{equation}
where $g(\lambda)=G(\lambda\eta)\psi$. We have seen that $g^{(3)}(\lambda)$ has the structure given by formula 
\e{2174a} i.e.\ the structure of the expressions considered in Proposition~\ref{ref:235F} (up to an extra uniform dependence 
on the parameter $\lambda\in [0,1]$). By Proposition~\ref{ref:235F} the integrated term in \e{2311phi} may thus be written 
$\widetilde{R}_0^1(\mU)=\Dxmez \widetilde{R}_0^{'1}$, with $\widetilde{R}_0^{'1}$ satisfying the inequalities of the statement. 

Let us study the Taylor expansion in \e{2311phi}. The expressions of $g(0)$, $g'(0)$, $g''(0)$ obtained page \pageref{g'g''} show that
\begin{align*}
\partial_t \eta=G(\eta)\psi&=\Dx \psi -\Dx(\eta\Dx\psi)-\partial_x(\eta\partial_x\psi)\\
&\quad +\Dx(\eta(\Dx(\eta\Dx \psi))) + \mez\Dx 
(\eta^2\partial_x^2\psi)
+\mez \partial_x^2(\eta^2 \Dx\psi)\\
&\quad +\Dxmez \widetilde{R}_0^{'1}.
\end{align*}
The second equation in \e{121} implies, when combined with the above expansion of $G(\eta)\psi$, that
\begin{align*}
\partial_t\psi&=-\eta-\mez (\px\psi)^2+\mez ( \Dx \psi)^2-(\Dx\psi)
\big[ \Dx (\eta\Dx\psi)+\eta \px^2\psi\big]\\
&\quad+a(\eta')P\Big[ \eta',G(\eta)\psi,\eta'\px\psi,\Dx\psi,\Dx(\eta\Dx\psi),\eta\px^2\psi,C_3(\eta,\psi),\Dxmez \widetilde{R}_0^{'1}\Big]
\end{align*}
where $P$ is a polynomial, sum of components that are homogeneous at least of degree $4$ and $C_3$ is the cubic term in the expansion 
of $G(\eta)\psi$, and where $a$ is some analytic function of $\eta'$. 

Since we have seen that $\widetilde{R}_0^{'1}$ satisfies \e{r1}, \e{r2}, Leibniz formula 
shows that the last term in the above equation satisfies similar bounds, replacing eventually $\alpha$ by 
some larger value. Computing from the above expressions $\partial_tu$, we get \e{2311eta}. This concludes the proof. 
\end{proof}

\section{Nonlinear estimates}

Our next goal is to estimate the action of $Z^k$ on various remainder terms. 
This task is quite technical and requires some preparation. We gather here various 
estimates which are extensively used in the sequel. Namely, we estimate 
$\double{\zeta F}{K,\nu}$, $\double{T_\zeta F}{K,\nu}$,  $\double{T_F \zeta}{K,\nu}$ 
and $\double{\RBony(\zeta,F)}{K,\nu}$. 

Recall that, for any real number $s\ge 0$,
\begin{align*}
&\lA \zeta F\rA_{H^s}\les \lA \zeta\rA_{L^\infty}\lA F\rA_{H^s}
+\lA F\rA_{L^\infty}\lA \zeta\rA_{H^s},\\
&\lA \zeta F\rA_{H^s}\les \lA \zeta\rA_{C^{s+1}}\lA F\rA_{H^s}.
\end{align*}
We need similar estimates for $\double{\zeta F}{K,\nu}$. We shall prove 
that, for any $s\ge 2$ and any $(K,\nu)\in\xN\times [0,+\infty\por$ such that $\nu+K\le s-2$, 
there holds
\begin{align}
&\double{\zeta F}{K,\nu}\les \triple{\zeta}{\frac{s}{2},0}
\double{F}{K,\nu}+\triple{F}{\frac{s}{2},0}\double{\zeta}{K,\nu},
\label{N320}\\
&\double{\zeta F}{K,\nu}\les \triple{\zeta}{s+1,0}\double{F}{K,\nu}.\label{N321}
\end{align}
These estimates can be deduced from the following result: 
for any real number $m\in [0,+\infty\por$ and any $(K,\nu)\in\xN\times [0,+\infty\por$,
\be\label{N322}
\double{\zeta F}{K,\nu}\les \triple{\zeta}{m,0} \double{F}{K,\nu}
+\triple{F}{\nu+K-m+2,0}\double{\zeta}{K,\nu},
\ee
where we use the convention that $\triple{F}{\nu+K-m+2,0}=0$ for $\nu+K-m+2<0$. Indeed, by applying 
\eqref{N322} with $m=s/2$ (resp.\ $m=s+1$) 
one recovers \eqref{N320} (resp.\ \eqref{N321}). 
Moreover, \eqref{N322} is convenient to prove estimate by induction on $K$ since
$$
\triple{ZF}{\nu+K-m+2,0}\le \triple{F}{\nu+(K+1)-m+2,0}.
$$
We begin by proving estimates similar to \eqref{N322} 
for $\double{T_\zeta F}{K,\nu}$,  $\double{T_F \zeta}{K,\nu}$ 
and $\double{\RBony(\zeta,F)}{K,\nu}$ as well as 
for $S_\Bony(\zeta,F)$ where $S_\Bony$ is defined by \eqref{n227}.

Recall that the notations $\triple{\cdot}{r,\sigma}$ and $\double{\cdot}{r,\nu}$ are defined 
for any real number $r$ (see Notation~\ref{T55}) so that 
$\triple{\cdot}{r,\sigma}\equiv 0$ and $\double{\cdot}{r,\nu}\equiv 0$ for $r<0$.

\begin{prop}\label{pr:ZTR}
Consider $m\in \xR$, $K\in \xN$ and $\nu\in ]0,+\infty\por$. 
Below one uses the conventions that 
\be\label{n322.5}
\triple{\zeta}{m,0}=0\text{ for }m <0, \qquad 
\triple{F}{\nu+K-m+1,0}=0\text{ for }\nu+K-m+1<0.
\ee

$(i)$ There exists a positive constant $c$ 
such that,
\be\label{n323b}
\double{T_{\zeta}F}{K,\nu} \le c 
\triple{\zeta}{m,0}\double{F}{K,\nu}
+c\triple{F}{\nu+K-m+1,0}\double{\zeta}{K,0}.
\ee

$(ii)$ There exists a positive constant $c$ 
such that,
\be\label{n324b}
\double{T_{F}\zeta}{K,\nu} \le c 
\triple{\zeta}{m,0}\double{F}{K,0}
+c\triple{F}{\nu+K-m+1,0}\double{\zeta}{K,\nu}.
\ee

$(iii)$ For any $a$ in $[0,+\infty[$ there exists a positive constant $c$ 
such that,
\be\label{n325b}
\double{\RBony(\zeta,F)}{K,\nu+a} \le c 
\triple{\zeta}{m,a}\double{F}{K,\nu}
+c\triple{F}{K-m,a}\double{\zeta}{K,\nu}.
\ee

$(iv)$ Let $S_\Bony(a,b)=\Op^{\Bony}[a,R]b$ with $R=-2\xi\cdot\nabla\theta$ 
where $\theta$ is given by Definition~\ref{defi:theta}. Then 
for any $a$ in $[0,+\infty[$ there exists a positive constant $c$ 
such that,
\be\label{n326b}
\double{S_\Bony(\zeta,F)}{K,\nu+a} \le c 
\triple{\zeta}{m,a}\double{F}{K,\nu}
+c\triple{F}{K-m,a}\double{\zeta}{K,\nu}.
\ee
\end{prop}
\begin{proof}

Let us prove statement $(i)$. By definition
$$
\double{T_{\zeta}F}{K,\nu}=\sum_{\ell=0}^K\blA Z^\ell T_\zeta F\brA_{H^{\nu+K-\ell}}.
$$
It follows from~\eqref{n227} that one can write $Z^\ell (T_\zeta F)$ as a linear combination of terms of 
the form $T^{(n_3)}(Z^{n_1}\zeta) Z^{n_2}F$ where $n_1+n_2+n_3\le \ell$ 
and where we used the following notation: $T^{(n)}(v)f=\Op^\Bony[v,(-2\xi\cdot\nabla)^n\theta]f$ 
where $\theta=\theta(\xip,\xii)$ is the cutoff function used in the definition of paradifferential operators 
(see Definition~\ref{defi:theta}), $\xi\cdot\nabla=\xip\partial_{\xip}+\xii\partial_{\xii}$ 
and where $\Op^\Bony[v,A]f$ is as defined 
in \S\ref{S:2.2.4} (so that $T^{(0)}(a)b$ is the 
paraproduct $T_a b$).

We thus have to prove that, for any $\ell \le K$ and any $(n_1,n_2,n_3)\in \xN^3$ such that 
$n_1+n_2+n_3\le \ell$,
\be\label{n327}
\blA T^{(n_3)}(Z^{n_1}\zeta)Z^{n_2} F\brA_{H^{\nu+K-\ell}}\les  \triple{\zeta}{m,0}\double{F}{K,\nu}
+c\triple{F}{\nu+K-m+1,0}\double{\zeta}{K,0}.
\ee

Notice that, 
for any $n\in \xN$, $T^{(n)}(v)f$ satisfies the same estimates as $T_v f$ does. For 
$n=0$ this is obvious since $T^{(0)}(v)=T_v$. For $n>0$, 
with the notation of Proposition~\ref{T32},  
one has $(-2\xi\cdot\nabla)^n\theta\in SR_{reg}^0$ 
(the condition~\eqref{n219} is satisfied since $\langle \xip\rangle\sim \langle \xii\rangle$ on 
the support of $\nabla_\xi\theta$). Then Proposition~\ref{T32}Ê
implies that, for any $\sigma\in ]0,+\infty[$ and any real numbers $\rho,\rho'$ 
such that $\rho'>\rho> 0$,
\begin{align}
&\blA T^{(n)}(v)f\brA_{H^\sigma}\les \lA v\rA_{L^\infty}\lA f\rA_{H^\sigma},\label{n328}\\
&\blA T^{(n)}(v)f\brA_{H^\rho}\le K \lA v\rA_{L^2}\lA f\rA_{\eC{\rho'}}.\label{n329}
\end{align}
For $n=0$, these estimates follow from the paraproduct rules \e{esti:quant0} and \e{esti:Tba}.

We now prove \eqref{n327}. Either $n_1\le m$ or $n_1>m$. We first consider the case where 
$n_1\le m$. Since $\nu+K-\ell\ge \nu> 0$ we may use \eqref{n328} to write
$$
\blA T^{(n_3)}(Z^{n_1}\zeta) Z^{n_2}F\brA_{H^{\nu+K-\ell}}
\les \lA Z^{n_1}\zeta\rA_{L^\infty}\lA Z^{n_2}F\rA_{H^{\nu+K-\ell}}.
$$
Now write 
$\lA Z^{n_1}\zeta\rA_{L^\infty}\le \triple{\zeta}{n_1,0}\le \triple{\zeta}{m,0}$ 
and
$$
\lA Z^{n_2}F\rA_{H^{\nu+K-\ell}}\le \double{F}{n_2,\nu+K-\ell}
\le \double{F}{n_2+K-\ell,\nu}\le \double{F}{K,\nu},
$$
by definition of the norms $\triple{\cdot}{n,\sigma}$ and $\double{\cdot}{n,\sigma}$. 
This proves \eqref{n327} for $n_1\le m$. 

We next consider the case where $n_1\ge m$. We apply \eqref{n329} to obtain that
$$
\blA T^{(n_3)}(Z^{n_1}\zeta) Z^{n_2}F\brA_{H^{\nu+K-\ell}}
\les \lA Z^{n_1}\zeta\rA_{L^2}\lA Z^{n_2}F\rA_{\eC{\nu+K-\ell+1}}.
$$
Since $n_1\le \ell\le K$, notice that  $\lA Z^{n_1}\zeta\rA_{L^2}\le \double{\zeta}{K,0}$. 
On the other hand
\begin{alignat}{2}
\blA Z^{n_2}F\brA_{\eC{\nu+K-\ell+1}}
&\le  \blA Z^{n_2}F\brA_{\eC{\nu+K-n_1-n_2+1}} \quad &&\text{since }n_1+n_2\le \ell \notag \\
&\le  \sum_{p=0}^{n_2} \blA Z^p F\brA_{\eC{\nu+K-n_1+1-p}} && \notag\\
&\le \sum_{p=0}^{n_2} \blA Z^p F\brA_{\eC{\nu+K-m+1-p}} &&\text{since }n_1\ge m.\label{n330}
\end{alignat}
Now observe that, since $n_1\ge m$, $n_1+n_2\le \ell$ and $\ell\le K$, one has
$$
m+n_2-K-1\le m+\ell -n_1-K-1 = 
(m-n_1)+(\ell-K)-1\le -1\le \nu
$$
and hence $n_2\le \nu+K-m+1$. 
Setting this into \eqref{n330} yields
$$
\blA Z^{n_2}F\brA_{\eC{\nu+K-\ell+1}}\le \sum_{p=0}^{\nu+K-m+1} \blA Z^p F\brA_{\eC{\nu+K-m+1-p}}
=\triple{F}{\nu+K-m+1,0},
$$
which completes the proof of statement $(i)$.

Statement $(ii)$ is a corollary of statement $(i)$. Indeed, \eqref{n323b} applied with $(\zeta,F)$ 
replaced with $(F,\zeta)$ implies that
$$
\double{T_{F}\zeta}{K,\nu} \les 
\triple{F}{m,0}\double{\zeta}{K,\nu}
+\triple{\zeta}{\nu+K-m+1,0}\double{F}{K,0}.
$$
By using this estimate with $m$ replaced with $\nu+K-m+1$ we obtain \eqref{n324b}. 

Finally we shall prove statement $(iii)$ by using 
arguments similar to those used in the proof of statement $(i)$. 

Set $\chi(\xip,\xii)\defn 1-\theta(\xip,\xii)-\theta(\xii,\xip)$ 
where $\theta$ is the cutoff function given by \eqref{defi:theta}. 
Then $\RBony(\zeta,F)=\Op^\Bony[\zeta,\chi(\xip,\xii)]F$. 
Thus $Z^\ell \RBony(\zeta,F)$ is a linear combination of terms of 
the form $R^{(n_3)}(Z^{n_1}\zeta,Z^{n_2}F)$ where $n_1+n_2+n_3\le \ell$ 
and where we used the following notation: $R^{(n)}(v,f)
=\Op^\Bony[v,(-2\xi\cdot\nabla)^n\chi]f$.

We thus have to prove that, for any $\ell \le K$ and any $(n_1,n_2,n_3)\in \xN^3$ such that 
$n_1+n_2+n_3\le \ell$,
\be\label{n332}
\blA R^{(n_3)}(Z^{n_1}\zeta,Z^{n_2} F)\brA_{H^{\nu+K-\ell+a}}
\les  \triple{\zeta}{m,a}\double{F}{K,\nu}
+c\triple{F}{K-m,a}\double{\zeta}{K,\nu}.
\ee

For any $n\in \xN$, $R^{(n)}(v,f)$ satisfies the same estimates as $\RBony(v,f)$ does. 
Indeed, with the notations of Proposition~\ref{T32}, 
one has $(-2\xi\cdot\nabla)^n\chi\in SR_{reg}^0$ for any $n\ge 0$. Consequently, 
for any real numbers $\sigma,a$ in $[0,+\infty\por$ such that $\sigma+a>0$, there holds
\begin{align}
&\blA R^{(n)}(v,f)\brA_{H^{\sigma+a}}\les \lA v\rA_{\eC{a}}
\lA f\rA_{H^\sigma},\label{n333}\\
&\blA R^{(n)}(v,f)\brA_{H^{\sigma+a}}\les \lA v\rA_{H^{\sigma}}\lA f\rA_{\eC{a}}.\label{n334}
\end{align}

We now prove \eqref{n332}. Either $n_1\le m$ or $n_1>m$. We first consider the case where 
$n_1\le m$. Then we use \eqref{n333}, $n_2\le \ell\le K$ and 
$n_1\le m$ to write
\begin{align*}
\blA R^{(n_3)}(Z^{n_1}\zeta,Z^{n_2}F)\brA_{H^{\nu+K-\ell+a}}
&\les \blA R^{(n_3)}(Z^{n_1}\zeta,Z^{n_2}F)\brA_{H^{\nu+K-n_2+a}}\\
&\les \lA Z^{n_1}\zeta\rA_{\eC{a}}\lA Z^{n_2}F\rA_{H^{\nu+K-n_2}}\\
&\les \triple{\zeta}{m,a}\double{F}{K,\nu}.
\end{align*}
On the other hand, if $n_1\ge m$ then
\begin{align*}
\blA R^{(n_3)}(Z^{n_1}\zeta,Z^{n_2}F)\brA_{H^{\nu+K-\ell+a}}
&\les \blA R^{(n_3)}(Z^{n_1}\zeta,Z^{n_2}F)\brA_{H^{\nu+K-n_1+a}}\\
&\les \lA Z^{n_1}\zeta\rA_{H^{\nu+K-n_1}}\lA Z^{n_2}F\rA_{\eC{a}}\\
&\les \double{\zeta}{K,\nu}\triple{F}{n_2,a}\les \double{\zeta}{K,\nu}\triple{F}{K-m,a}
\end{align*}
where we used in the last inequality that $n_1+n_2\le \ell \le K$ and hence $n_2\le K-m$ 
since $n_1\ge m$. This proves \eqref{n332} and hence completes the proof of statement $(iii)$. 

The proof of statement $(iv)$ is analogous to the proof of statement $(iii)$. 
Indeed, by definition $S_\Bony(\zeta,F)=\Op^\Bony[v,(-\xi\cdot\nabla)\theta]F$ 
and hence $Z^{\ell}S_\Bony(\zeta,F)$ is a linear combination of terms of the form 
$\Op^\Bony[Z^{n_2}\zeta,(\xi\cdot\nabla)^{n_3}\theta]Z^{n_1}F$ with $n_1+n_2+n_3\le \ell$ and
$n_3\ge 1$. As already mentioned, 
one has $(-2\xi\cdot\nabla)^n\theta\in SR_{reg}^0$ for $n>0$, so 
Proposition~\ref{T32}Ê
implies that $\Op^\Bony[v,(\xi\cdot\nabla)^{n}\theta]f$ satisfies the same estimates~\eqref{n333} 
and \eqref{n334} as $R^{(n)}(v,f)$ does. 
\end{proof} 
\begin{rema*}
For further references, let us state and prove an estimate analogous to \e{n328}-\e{n329} in H\"older spaces. Consider a positive real number $\sigma$ with $\sigma\not\in \xN$. 
Then
\be\label{n615-wa}
\triple{T_\zeta F}{n,\sigma}\les \triple{\zeta}{n,1}\triple{F}{n,\sigma}.
\ee
To see this, using elementary arguments similar to those used in the proof of 
statement $i)$ of Proposition~\ref{pr:ZTR}, one needs only to prove that, for any $n\in \xN$ 
and for any real number $\sigma$ in $[0,+\infty[$, one has
\be\label{n615-wab}
\blA T^{(n)}(v)f\brA_{\eC{\sigma}}\le K \lA v\rA_{\eC{1}}\lA f\rA_{\eC{\sigma}}.
\ee
For $n=0$, this follows from the paraproduct rule \e{Tab:Crho}. For $n>0$, using the notations and 
the observations made in the proof of Proposition~\ref{pr:ZTR}, notice that 
$T^{(n)}(v)=\Op^\Bony[v,R]$ where $R=(-2\xi\cdot\nabla)^n\theta$ belongs to 
$SR^0_{reg}$. Now the wanted estimate follows easily from the estimate of the kernel $K_{k,\ell}$ made in the proof 
of Proposition~\ref{T32}.
\end{rema*}

The previous proposition 
has the following corollary.

\begin{coro}\label{pr:Z}
Consider $m\in \xR$, $K\in \xN$ and $\nu\in ]0,+\infty\por$. 
There exists a positive constant $c$ 
such that,
\begin{equation}\label{n335b}
\begin{aligned}
\double{\zeta F}{K,\nu} \le c 
\triple{\zeta}{m,0}\double{F}{K,\nu}
+c\triple{F}{\nu+K-m+1,0}\double{\zeta}{K,\nu},
\end{aligned}
\end{equation}
where $\triple{\zeta}{m,0}=0$ for $m <0$ and $\triple{F}{\nu+K-m+1,0}=0$ for $\nu+K-m+1<0$, 
by convention.
\end{coro}
\begin{proof}
Write $\zeta F=T_\zeta F+T_{F}\zeta +\RBony(\zeta,F)$ and apply Proposition~\ref{pr:ZTR}. 
\end{proof}

For further references, we shall also need more precise estimates.

\begin{prop}\label{pr:ZTR-sharp}
Consider $m\in \xR$, $K\in \xN$ and $\nu\in ]0,+\infty\por$. 
One uses 
the conventions in \e{n322.5} and denotes by $\indicator{\xR_+}$ the 
indicator function of $\xR_+$.

$(i)$ There exists a positive constant $c$ 
such that,
\be\label{n323c}
\ba
\double{T_{\zeta}F}{K,\nu} &\le c 
\triple{\zeta}{m,0}\double{F}{K-1,\nu+1}
+c\triple{F}{\nu+K-m+1,0}\double{\zeta}{K-1,0}\\
&\quad +c\indicator{\xR_+}(m)\lA \zeta\rA_{L^\infty}\blA Z^K F\brA_{H^\nu}
+c \indicator{\xR_+}(K-m)\lA F\rA_{\eC{\nu+1}}\blA Z^K\zeta\brA_{L^2}.
\ea
\ee

$(ii)$ For any real number $a$ in $[0,+\infty[$ there exists a positive constant $c$ 
such that,
\be\label{n325c}
\ba
\double{\RBony(\zeta,F)}{K,\nu+a} &\le c
\triple{\zeta}{m,a}\double{F}{K-1,\nu+1}+c \indicator{\xR_+}(m)\lA \zeta\rA_{\eC{a}}\blA Z^K F\brA_{H^\nu}\\
&\quad+c\triple{F}{K-m,a}\double{\zeta}{K-1,\nu+1}
+c \indicator{\xR_+}(K-m)\lA F\rA_{\eC{a}}\blA Z^K\zeta\brA_{H^\nu}.
\ea
\ee

$(iii)$ There exists a positive constant $c$ 
such that,
\be\label{n335c}
\ba
\double{\zeta F}{K,\nu} &\le c 
\triple{\zeta}{m,0}\double{F}{K-1,\nu+1}
+c\triple{F}{\nu+K-m+1,0}\double{\zeta}{K-1,\nu+1}\\[0.5ex]
&\quad +c\indicator{\xR_+}(m)\lA \zeta\rA_{\eC{1}}\blA Z^K F\brA_{H^\nu}\\[0.5ex]
&\quad +c\indicator{\xR_+}(m-\nu-1)\lA \zeta\rA_{\eC{\nu+1}}\blA Z^KF\brA_{L^2}\\[0.5ex]
&\quad +c\indicator{\xR_+}(K-m)
\lA F\rA_{\eC{\nu+1}}\blA Z^K\zeta\brA_{L^2}\\[0.5ex]
&\quad +c\indicator{\xR_+}(\nu+K-m+1)
\lA F\rA_{\eC{1}}\blA Z^K \zeta\brA_{H^\nu}.
\ea
\ee
\end{prop}
\begin{rema*}
Assume $m\ge 1$. 
By using in addition the obvious inequalities
\begin{align*}
&\indicator{\xR_+}(m)\lA \zeta\rA_{\eC{1}}\blA Z^K F\brA_{H^\nu}
\le \triple{\zeta}{m,0}\double{F}{K,\nu},\\
&\indicator{\xR_+}(m-\nu-1)\lA \zeta\rA_{\eC{\nu+1}}\le \indicator{\xR_+}(m)\lA \zeta\rA_{\eC{m}},
\end{align*} 
it follows from \e{n335c} that
$$
\ba
\double{\zeta F}{K,\nu} &\le c 
\triple{\zeta}{m,0}\double{F}{K,\nu}
+c\triple{F}{\nu+K-m+1,0}\double{\zeta}{K-1,\nu+1}\\
&\quad +c\indicator{\xR_+}(K-m)
\lA F\rA_{\eC{\nu+1}}\blA Z^K\zeta\brA_{L^2}\\
&\quad +c\indicator{\xR_+}(\nu+K-m+1)
\lA F\rA_{\eC{1}}\blA Z^K \zeta\brA_{H^\nu}.
\ea
$$
Let $b>1$ be any fixed real number. 
Using the obvious inequalities
\be\label{o325}
\ba
&\lA F\rA_{\eC{\nu+1}}\le \lA F\rA_{\eC{b}}+\indicator{\xR_+} (\nu+1-b)
\lA F\rA_{\eC{\nu+1}},\\
&\indicator{\xR_+}(K-m)\lA F\rA_{\eC{\nu+1}}
\le \lA F\rA_{\eC{\nu+K-m+1}},\\
&\indicator{\xR_+}(K-m)\le \indicator{\xR_+}(\nu+K-m+1),
\ea
\ee
one has the following corollary
\be\label{n335e}
\ba
\double{\zeta F}{K,\nu} &\le c 
\triple{\zeta}{m,0}\double{F}{K,\nu}
+c\triple{F}{\nu+K-m+2,0}\double{\zeta}{K-1,\nu+1}\\
&\quad +c\indicator{\xR_+}(\nu+K-m+1)\lA F\rA_{\eC{b}}\blA Z^K \zeta\brA_{H^\nu}\\
&\quad +c\indicator{\xR_+}(\nu+1-b)\lA F\rA_{\eC{\nu+K-m+1}}\blA Z^K\zeta\brA_{L^2}.
\ea
\ee
Similarly, by using \e{o325}, we deduce from 
\e{n323c} that 
\be\label{n369.5}
\ba
\double{T_{\zeta}F}{K,\nu} &\les 
\triple{\zeta}{m,0}\double{F}{K-1,\nu+1}
+\indicator{\xR_+}(m)\lA \zeta\rA_{L^\infty}
\blA Z^K F\brA_{H^\nu}\\
&\quad+\triple{F}{\nu+K-m+1,0}\double{\zeta}{K-1,0}\\
&\quad +\indicator{\xR_+}(\nu+K-m+1)\lA F\rA_{\eC{b}}
\blA Z^K\zeta\brA_{L^2}\\
&\quad +\indicator{\xR_+} (\nu+1-b)\lA F\rA_{\eC{\nu+K-m+1}}\blA Z^K\zeta\brA_{L^2}.
\ea
\ee

\end{rema*}
\begin{proof}
Let us prove \e{n323c}. Write 
$\double{T_{\zeta}F}{K,\nu} =\blA Z^K\bigl(T_\zeta F\bigr)\brA_{H^\nu}
+\double{T_\zeta F}{K-1,\nu+1}$. It follows from \e{n323b} that
$$
\double{T_{\zeta}F}{K-1,\nu+1} \le c 
\triple{\zeta}{m,0}\double{F}{K-1,\nu+1}
+c\triple{F}{\nu+K-m+1,0}\double{\zeta}{K-1,0},
$$
which is smaller than the right hand side  of \e{n323c}. To estimate the 
$H^\nu$-norm of $Z^K\bigl(T_\zeta F\bigr)$ we write, using the 
notations introduced in the proof of Proposition~\ref{pr:ZTR}, that
$$
\blA Z^K\bigl(T_\zeta F\bigr)\brA_{H^\nu}
\les \sum_{n_1+n_2+n_3\le K} I(n_1,n_2,n_3)\quad\text{with }
 I(n_1,n_2,n_3)\defn \blA T^{(n_3)}(Z^{n_1}\zeta)
Z^{n_2}F\brA_{H^\nu}.
$$
We split the sum into two pieces, according to $n_1\le m$ or $n_1>m$. We further 
split the first (resp.\ second) sum into two pieces, according to $n_1=0$ or $0<n_1\le m$ (resp.\ 
$n_1=K$ or $m<n_1<K$). The same arguments used to prove \e{n323b} imply that 
$$
\sum_{\substack{n_1+n_2+n_3\le K \\ 0<n_1\le m}}
I(n_1,n_2,n_3)\les \triple{\zeta}{m,0}\double{F}{K-1,\nu+1},
$$
and
$$
\sum_{\substack{n_1+n_2+n_3\le K \\ m<n_1<K}}
I(n_1,n_2,n_3)\les \triple{F}{\nu+K-m+1,0}\double{\zeta}{K-1,0}.
$$
Moreover, the paraproduct rules~\e{esti:quant0} and \e{esti:Tba} imply that
\begin{align*}
&I(K,0,0)=\blA T_{Z^K\zeta}F\brA_{H^\nu}\les \blA Z^K\zeta\brA_{L^2}\lA F\rA_{\eC{\nu+1}},\\
&I(0,K,0)=\blA T_\zeta Z^K F\brA_{H^\nu}\les \lA \zeta\rA_{L^\infty}\blA Z^KF\brA_{H^\nu}.
\end{align*}
The first (resp.\ second) of the two previous inequalities is to be taken into account only for $K> m$ (resp.\ $m\ge 0$), 
we obtain the 
desired result \e{n323c}; indeed for $K\le m$ (resp.\ $m<0$, the sum $\sum_{\substack{m<n_1<K}}
I(n_1,n_2,n_3)$ (resp.\ $\sum_{n_1\le m}$ vanishes).

The proof of \e{n325c} is similar. 

To prove \e{n335c} 
we write $\zeta F= T_\zeta F+T_F \zeta +\RBony(\zeta,F)$. 
The first (resp.\ third) term is estimated by means of \e{n323c} (resp.\ \e{n325c}). 
The second term is estimated by means of 
\e{n323c} applied with $(\zeta,F,m)$ replaced with $(F,\zeta,\nu+K-m+1)$. 
\end{proof}

We shall also need the following estimates. 
\begin{lemm}
Consider an integer $n$ in $\xN^*$ and a positive real number $\mu$. 
Then, for any integer $m$ such that $m\ge 2$ and $2m>n+\mu+2$, there exists a 
positive constant $c$ such that
\be\label{n336}
\ba
&\blA Z^n (T_a T_b-T_{ab})f\brA_{H^\mu}\\[1ex]
&\qquad\le c\lA a \rA_{\eC{2}}\lA b\rA_{\eC{2}}\blA Z^n f\brA_{H^{\mu-2}}\\
&\qquad\quad +c\indicator{\xR_+}(n-m)\Bigl( \blA Z^n a \brA_{L^2} \lA b\rA_{L^\infty}
+\lA a \rA_{L^\infty}\blA Z^n b\brA_{L^2}\Bigr) \lA f\rA_{\eC{\mu+1}}\\
&\qquad\quad + c\triple{a}{m,0}\triple{b}{m,0}\double{f}{n-1,\mu-1}\\
&\qquad\quad +c\indicator{\xR_+}(n-m)\Bigl( \triple{a}{m,0}\double{b}{n-1,0}+\double{a}{n-1,0}\triple{b}{m,0}\Bigr)\triple{f}{m,0}.
\ea
\ee
\end{lemm}
\begin{proof}
Using the notations introduced in this section, 
$Z^n (T_a T_b-T_{ab})f $ can be written as $Z^n (T_a T_b-T_{ab})f = R_0+R_1+R_2+R_3+R_4$ 
where
\begin{align*}
R_0&=(T_a T_b-T_{ab})Z^n f,\\
R_1&=\Bigl( T_{Za}T_b  +T_{a}T_{Zb}-T_{(Za)b}-T_{a(Zb)}\Bigr) Z^{n-1}f,\\
R_2&=\Bigl(T^{(1)}(a)T_b+T_a T^{(1)}(b)-T^{(1)}(ab)\Bigr) Z^{n-1}f,
\end{align*}
$R_3$ is a linear combination of terms of the form
$$
T^{(\ell_1)}\bigl(Z^{n_1}a\bigr)T^{(\ell_2)}\bigl(Z^{n_2}b\bigr) Z^{n_3}f,
\quad \ell_1+\ell_2+n_1+n_2+n_3\le n, \quad n_3\le n-2,
$$
and $R_4$ is a linear combination of terms of the form
$$
T^{(\ell)}\bigl( (Z^{m_1}a)(Z^{m_2}b)\bigr)Z^{m_3}f,\quad 
\ell+m_1+m_2+m_3\le n,\quad m_3\le n-2.
$$

$\bullet$ The terms $R_0$ and $R_1$ are estimated by means 
of the symbolic calculus rule~\e{esti:quant2-func} which yields
\begin{align*}
&\lA (T_a T_b-T_{ab})Z^n f\rA_{H^\mu}
\les \lA a \rA_{\eC{2}}\lA b\rA_{\eC{2}}\blA Z^n f\brA_{H^{\mu-2}},\\
&\blA \bigl(T_{Za}T_b  -T_{(Za)b}\bigr)Z^{n-1}f\brA_{H^{\mu}}\les
\lA Za\rA_{\eC{1}}\lA b\rA_{\eC{1}} \blA Z^{n-1}f\brA_{H^{\mu-1}},\\ 
&\blA \bigl(T_{a}T_{Zb}  -T_{a(Zb)}\bigr)Z^{n-1}f\brA_{H^{\mu}}\les
\lA a\rA_{\eC{1}}\lA Zb\rA_{\eC{1}} \blA Z^{n-1}f\brA_{H^{\mu-1}}.
\end{align*}
So $\lA R_0\rA_{H^\mu}$ is controlled by the first term in the right hand side  of \e{n336}. 
Since $\lA Za\rA_{\eC{1}}\le \triple{a}{m,0}$ and 
$\lA Zb\rA_{\eC{1}}\le \triple{b}{m,0}$ for $m\ge 2$, and since 
$ \blA Z^{n-1}f\brA_{H^{\mu-1}}\le \double{f}{n-1,\mu-1}$, we verify that 
$R_1$ is controlled by the third term in the right hand side  of \e{n336}. 

$\bullet$ Let us estimate the $H^\mu$-norm of $R_2$. 
Since $T^{(1)}(a)=\Op^\Bony[a,(-2\xi\cdot\nabla)\theta]$ with $(-2\xi\cdot\nabla)\theta\in SR_{reg}^0$, as already seen, Proposition~\ref{T32}Ê
implies that 
$$
\blA T^{(1)}(a)T_b Z^{n-1}f\brA_{H^\mu}
\les \lA a\rA_{\eC{1}} \blA T_b Z^{n-1}f\brA_{H^{\mu-1}}
\les \lA a\rA_{\eC{1}} \lA b\rA_{L^\infty} \blA Z^{n-1} f\brA_{H^{\mu-1}}.
$$
By applying the same estimates for the two other terms which enter in the definition of $R_2$, 
we conclude that the $H^\mu$-norm of $R_2$ is controlled by the 
third term in the right hand side  of \e{n336}.

$\bullet$ Let us estimate $R_3$. Set $A=T^{(\ell_1)}\bigl(Z^{n_1}a\bigr)T^{(\ell_2)}\bigl(Z^{n_2}b\bigr) Z^{n_3}f$. 
We shall split the analysis in several cases. 

If $n_1\le m$ and $n_2\le m$, we write
\begin{align*}
\lA A\rA_{H^\mu}&\les \blA Z^{n_1}a\brA_{L^\infty} \blA Z^{n_2}b\brA_{L^\infty}
\blA Z^{n_3} f\brA_{H^\mu}\\
&\les \triple{a}{m,0}\triple{b}{m,0}\double{f}{n-2,\mu}.
\end{align*}
since $n_3\le n-2$. Since $\double{f}{n-2,\mu}\le \double{f}{n-1,\mu-1}$, this proves that 
the $H^\mu$-norm of $A$ is controlled by the third term in the right hand side  of \e{n336}.

If $m<n_1<n$, then we notice that the assumption $2m>n+\mu+2$ implies that 
$$
m>2m-n_1\ge n+\mu+2-n_1\ge n_2+(n_3+\mu+1)+1
$$
so $n_2\le m-1$ and $n_3+\mu+1\le m-1$. Consequently, the estimates 
\e{n328}--\e{n329} and \e{n615-wab} imply that 
\begin{align*}
\lA A\rA_{H^\mu} &\les \blA Z^{n_1}a\brA_{L^2} \blA Z^{n_2}b\brA_{L^\infty}
\blA Z^{n_3}f\brA_{\eC{\mu+1}}\\
&\les \double{a}{n-1,0}\triple{b}{m,0}\triple{f}{m,0},
\end{align*}
so the $H^\mu$-norm of $A$ is controlled by the fourth term in the right hand side  of \e{n336}. 
The analysis of the case $m<n_2<n$ is similar. 

Assume that $n>m$ and $n_1=n$. Then $\ell_1=\ell_2=n_2=n_3$ and hence the paraproduct rules \e{esti:quant0} 
and \e{esti:Tba} imply that
$$
\lA A\rA_{H^\mu}\les \lA Z^n a\rA_{L^2}\lA b\rA_{L^\infty} \lA f\rA_{\eC{\mu+1}},
$$
so the $H^\mu$-norm of $A$ is controlled by the second term in the right hand side  of \e{n336}. 

This proves that the $H^\mu$-norm of $R_3$ is controlled by the right hand side  of \e{n336}. 
The analysis of $R_4$ is similar.
\end{proof}

\section{Estimate of the remainder terms}\label{S:236}

The goal to this section is 
to prove various estimates required when estimating the remainder terms. 

To estimate the remainder terms, we shall need to exploit repeatedly the fact that the commutator $[G(\eta),\eta]$ 
is of order $0$.  Similarly, when studying the 
linearization estimates, we have seen that $G(\eta)-\Dx$ is of order $0$ (while $\B(\eta)-\Dx$ 
and $V(\eta)-\px$ are of order $1$). We shall need to exploit this fact too. 

We need to estimate $Z^k [G(\eta),\eta]$ and $Z^k (G(\eta)-\Dx)$. 
The analysis of both $Z^k [G(\eta),\eta]$ and $Z^k (G(\eta)-\Dx)$ will be by induction on $k$, using the fact 
that one 
can compute explicitly $Z[G(\eta),\eta]$ and $Z (G(\eta)-\Dx)$. In both formula we shall see that 
the commutator $[G(\eta),Z\eta]$ appears. More generally, to control 
$Z[G(\eta),Z^p\eta]$ for some integer $p\in \xN$, one needs to control 
$[G(\eta),Z^{p+1}\eta]$. 
We thus begin by studying these operators. Below, for $p\in \xN$, we denote by 
$J(\eta,Z^p\eta)$ the commutator defined by
$$
J(\eta,Z^p\eta)f\defn G(\eta)\bigl((Z^p \eta)f\bigr)-(Z^p\eta)G(\eta)f.
$$

In this section, 
we use various inequalities in some H\"older spaces $\eC{\varrho}(\xR)$. 
We shall freely use the fact that, for our purposes, one can assume that 
$\varrho\not \in \mez\xN$ up to replacing $\varrho$ with 
$\varrho+\delta$ for some $\delta\ll 1$.

\begin{nota}\label{T55}
The notation $\triple{f}{r,\gamma}$ has been introduced for $r\in\xN$ and $\gamma \in [0,+\infty\por$. 
For the purpose of the next results, it is convenient to 
extend it to the case when $r$ is any real number. This is done as follows: $(i)$ for $r<0$ one sets 
$\triple{f}{r,\gamma}=0$ for any $f$ and any $\gamma$, 
and $(ii)$ for $r\ge 0$, one sets $\triple{f}{r,\gamma}\defn \triple{f}{[r],\gamma}$ where $[r]$ 
is the largest integer smaller or equal to $r$. One defines similarly $\double{f}{r,\gamma}$ for any real number $r\in\xN$.
\end{nota}

\begin{prop}\label{T60}
There exists $\eps_0>0$ small enough and there exist $\gpr$ with 
$\gpr\not\in\mez\xN$ and $\np$ large enough such that, 
for any $(s,s_1,s_0)\in\xN^3$ satisfying
\begin{equation*}
s\ge s_1\ge s_0\ge \mez (s+2\gpr),
\end{equation*}
for any integer $p$ in $[0,s_1]$, any integer 
$K$ in $[0,s_1-p]$ 
and 
any real number $\mu$ in $[4,s-K-p-1]$ there exists a nondecreasing function 
$C$ such that, for any $T>0$ and any smooth 
functions $(\eta,f)$ such that $\sup_{t\in [0,T]}\lA \eta(t)\rA_{\eC{s_0}}\le \eps_0$, 
\begin{equation}\label{n337}
\double{J(\eta,Z^p\eta)f}{K,\mu} 
\le \Cr \triple{\eta}{s_0,0}\double{f}{K,\mu}
+\Cr \triple{f}{\mu+K+p-s_0+\np,\gpr}\double{\eta}{K+p,\mu+1},
\end{equation}
where $\Cr=C( \triple{\eta}{s_0,0} )$. 
\end{prop}
\begin{rema*}
This estimate is not optimal with respect to the factors estimated 
in H\"older norms. The key point is that it is optimal 
with respect to the factors estimated in Sobolev norms.
\end{rema*}
\begin{proof}
For technical reasons, instead of proving \eqref{n337}, it is convenient to prove that, for $\npr$ large enough, 
\begin{equation}\label{n338}
\double{J(\eta,Z^p\eta)f}{K,\mu} 
\le \Cr \triple{\eta}{s_0,0}\double{f}{K,\mu}
+\Cr \triple{f}{\mu+K+\tilde p-s_0+\npr,\gpr}\double{\eta}{K+p,\mu+1},
\end{equation}
where $\tilde{p}=\max (p,1)$. It is clear that this estimate is equivalent to \eqref{n337}. 

Hereafter, we freely use the following estimates
\be\label{n339}
\begin{aligned}
&\triple{Z^p u}{n,\sigma}\le \triple{u}{n+p,\sigma},
\quad \triple{u_1 u_2}{n,\sigma}\les \triple{u_1}{n,\sigma}\triple{u_2}{n,\sigma},\\
&\triple{u}{n,\sigma+m}\le 
\triple{u}{n+m,\sigma},\quad 
\double{u}{n,\sigma+m}\le \double{u}{n+m,\sigma}.
\end{aligned}
\ee

The proof is by induction on $K$. 

\step{1}{Initialization}

We first prove \eqref{n338} for $K=0$. 
We prove that, with $\npr=5$ and $\gpr$ large enough, 
for any $p\in [0,s_1]$ and any $\mu\in [-1/2,s-p-1]$, there holds
\begin{equation}\label{n340}
\double{J(\eta,Z^p\eta)f}{0,\mu} 
\le \Cr \triple{\eta}{s_0,0}\double{f}{0,\mu}
+\Cr \triple{f}{\mu+\tilde p-s_0+\npr,\gpr}\double{\eta}{p,\mu+1},
\end{equation}
where $\Cr=C( \triple{\eta}{s_0,0})$.

To prove \eqref{n340} it is sufficient to prove that $J(\eta,\tilde\eta)f\defn [G(\eta),\tilde\eta]f$ 
satisfies
\begin{equation}\label{n341}
\lA J(\eta,\tilde\eta)f\rA_{H^\mu} 
\le \Cr \lA \tilde\eta\rA_{\eC{s_0-p}} \lA f\rA_{H^\mu}
+\Cr \lA f\rA_{\eC{\mu+\tilde p-s_0+\npr+\gpr}}\bigl(\lA\tilde\eta\rA_{H^{\mu+1}}
+\lA \eta\rA_{H^{\mu+1}}\bigr),
\end{equation}
where $\Cr=C( \lA \eta\rA_{\eC{s_0}},\lA \tilde\eta\rA_{\eC{s_0-p}})$ and where it is understood that $\lA \tilde\eta\rA_{\eC{a}}$ if $a<0$.

It follows from Proposition~\ref{T22} and the product estimate~\eqref{pr:sz} that, 
for any $\mu\ge -1/2$,
$$
\lA G(\eta)(\tilde\eta f)\rA_{H^\mu}
\le C\bigl(\lA \eta\rA_{\eC{\mu+5}}\bigr)\lA \tilde\eta f\rA_{H^{\mu+1}}
\le C\bigl(\lA \eta\rA_{\eC{\mu+5}}\bigr)\lA \tilde\eta\rA_{H^{\mu+1}}\lA  f\rA_{\eC{\mu+2}}.
$$
Similarly, by using~\eqref{pr:sz}, \eqref{211-1} and \eqref{esti:Dxmez-Crho}, we obtain that for any $\mu\ge -1/2$,
$$
\lA \tilde\eta G(\eta)f\rA_{H^\mu}\le \lA \tilde\eta\rA_{H^{\mu}}\lA G(\eta)f\rA_{\eC{\mu+\mez}}
\le C\bigl(\lA \eta\rA_{\eC{\mu+5}}\bigr)\lA \tilde\eta\rA_{H^{\mu}}\lA  f\rA_{\eC{\mu+4}}.
$$
This implies that
$$
\lA J(\eta,\tilde\eta)f\rA_{H^\mu}
\le C\bigl(\lA \eta\rA_{\eC{\mu+5}}\bigr)\lA \tilde\eta\rA_{H^{\mu+1}}\lA  f\rA_{\eC{\mu+4}},
$$
which in turn implies \eqref{n341} provided that $s_0-\tilde p\le \npr$ and 
$\gpr$ is large enough (indeed, we then have 
$\mu+5\le s-p+4\le s-\tilde p+5\le s-s_0+\npr+5\le s_0$ for $\gpr$ large enough). 

It remains to prove 
\eqref{n341} for $s_0-\tilde p> \npr$. We further split the analysis into two parts. 
Consider first the case where $\mu+\tilde p-s_0+\npr\ge 0$. Write
\be\label{n342}
J(\eta,\tilde{\eta})f=\Dx (\tilde{\eta} f)-\tilde{\eta}\Dx f +(G(\eta)-\Dx)(\tilde{\eta} f)-
\tilde{\eta} (G(\eta)f-\Dx f).
\ee
The first term is estimated by means of \eqref{L312:2}Ê
in Lemma~\ref{Lemm:A7} which yields that
$$
\lA \Dx (\tilde{\eta} f)-\tilde{\eta}\Dx f\rA_{H^{\mu}} \les 
 \lA \tilde\eta\rA_{\eC{1}}\lA f\rA_{H^\mu}+\lA f\rA_{\eC{2}}
\lA \tilde\eta \rA_{H^{\mu+1}}.
$$
The second term in the right-hand side of \eqref{n342} is estimated by means of the tame 
product rule \eqref{prtame}Ê
and the estimate~\eqref{n129} (applied with 
$(s,\gamma,\mu)$ replaced with $(\mu+1,3+\epsilon,\mu+1)$, recalling that 
$\mu\ge 4$ by assumption) 
for the operator norm of $G(\eta)-\Dx$. It is found that
$$
\lA (G(\eta)-\Dx)(\tilde{\eta} f)\rA_{H^\mu}\le  C\bigl(\lA \eta\rA_{\eC{4}}\bigr)\Bigl\{
\lA \tilde\eta f\rA_{\eC{4}}\lA \eta\rA_{H^{\mu+1}}
+\lA \eta\rA_{\eC{4}}\lA \tilde\eta f\rA_{H^\mu}\Bigr\}
$$
so
\begin{align*}
\lA (G(\eta)-\Dx)(\tilde{\eta} f)\rA_{H^\mu}\le 
C\bigl(\lA (\eta,\tilde\eta)\rA_{\eC{4}}\bigr)\Bigl\{&
\lA f\rA_{\eC{4}}\lA \eta\rA_{H^{\mu+1}}+\lA \tilde\eta\rA_{L^\infty}\lA f\rA_{H^\mu}
\\
&\quad
+\lA \tilde\eta\rA_{H^\mu}\lA f\rA_{L^\infty}\Bigr\}.
\end{align*}
The third term in the right-hand side of \eqref{n342} is estimated by means of the tame 
product rule \eqref{prtame}Êand the estimates~\eqref{n129} (applied with 
$(s,\gamma,\mu)$ replaced with $(\mu+1,3+\epsilon,\mu+1)$). 
It is found that
\begin{align*}
&\lA \tilde{\eta} (G(\eta)f-\Dx f)\rA_{H^\mu}\\
&\qquad\le
\lA \tilde{\eta}\rA_{L^\infty}\lA G(\eta)f-\Dx f\rA_{H^\mu}
+\lA \tilde{\eta}\rA_{H^\mu}\lA G(\eta)f-\Dx f\rA_{L^\infty}\\
&\qquad\le   C\bigl(\lA \eta\rA_{\eC{4}}\bigr) \lA \tilde{\eta}\rA_{L^\infty}\Bigl\{
\lA f\rA_{\eC{4}}\lA \eta\rA_{H^{\mu+1}}
+\lA \eta\rA_{\eC{4}}\lA f\rA_{H^\mu}\Bigr\}\\
&\qquad\quad +C\bigl(\lA \eta\rA_{\eC{4}}\bigr)\lA  \tilde{\eta}\rA_{H^\mu}\lA f\rA_{\eC{4}},
\end{align*}
where we estimated $\lA G(\eta)f-\Dx f\rA_{L^\infty}$ by means of the triangle inequality 
and \eqref{211-1} and \eqref{esti:Dxpsiz0}. 

Since $s_0-\tilde p> \npr$ by assumption, for $\npr\ge 5$ we have $s_0-p>4$ and hence $\lA \tilde\eta\rA_{\eC{4}}\le 
\lA \tilde\eta\rA_{\eC{s_0-p}}$. Since $s_0\ge 5$ by assumption we have 
$\lA \eta\rA_{\eC{4}}\le 
\lA \eta\rA_{\eC{s_0}}$. Eventually, 
for $\mu+p-s_0+\npr\ge 0$, we have $\lA f\rA_{\eC{4}}\le \lA f\rA_{\eC{\mu+p-s_0+\npr+\gpr}}$ and hence the desired result \eqref{n341}Ê
follows from \eqref{n342} and 
the previous estimates. 

We now consider the last case where $s_0-\tilde p>\npr$ and $\mu+p-s_0+\npr< 0$. 
We use again the decomposition \eqref{n342}. However, we now estimate the first term 
in the right-hand side of \eqref{n342} by means of \eqref{L312:iii}. This yields
$$
\lA \Dx (\tilde{\eta} f)-\tilde{\eta}\Dx f\rA_{H^{\mu}} \les 
\lA \tilde\eta\rA_{\eC{\mu+2}}\lA f\rA_{H^\mu}.
$$
We now estimate the second term in the right-hand side of \eqref{n342} 
by means of the product rule \eqref{pr:sz}Ê
and the estimate~\eqref{va2} (applied with 
$\gamma=\mu+5$) 
for the operator norm of $G(\eta)-\Dx$. It is found that
\begin{align*}
\lA (G(\eta)-\Dx)(\tilde{\eta} f)\rA_{H^\mu}&\le  C\bigl(\lA \eta\rA_{\eC{\mu+5}}\bigr)
\lA \tilde\eta f\rA_{H^{1/2}}\\
&\le C\bigl(\lA \eta\rA_{\eC{\mu+5}}\bigr)
\lA \tilde\eta\rA_{\eC{1}}\lA f\rA_{H^{1/2}}.
\end{align*}
Similarly, 
\begin{align*}
\lA \tilde{\eta} (G(\eta)f-\Dx f)\rA_{H^\mu}&\le
\lA \tilde{\eta}\rA_{\eC{\mu+1}} \lA G(\eta)f-\Dx f\rA_{H^\mu}\\
&\le  C\bigl(\lA \eta\rA_{\eC{\mu+5}}\bigr)\lA \tilde{\eta}\rA_{\eC{\mu+1}}
\lA  f\rA_{H^{1/2}}.
\end{align*}
For $\npr\ge 5$ and $\mu+p-s_0+\npr< 0$ we have  
$\mu+5\le s_0-p\le s_0$ 
so that \eqref{n341}~follows from \eqref{n342} and 
the previous estimates.

\step{2}{H\"older estimates}

We shall need to estimate 
$\triple{J(\eta,\eta)f}{n,\sigma}$ and 
$\triple{J(\eta,Z\eta)f}{n,\sigma}$. 
For our purpose, it is sufficient to have a non optimal estimate in H\"older spaces, that is an estimate 
which involves $\triple{f}{n,\sigma+1}$ (which amounts to lose one derivative, while 
$J(\eta,\eta)$ and $J(\eta,Z\eta)$ are  
expected to be of order~$0$). We claim that for $p=0$ or $p=1$ and for 
any integer $n$ in $[0,s_0-p-5]$ and 
any real number $\sigma$ in $\pol 3,s_0-p-n-1]\setminus \mez\xN$,
\begin{equation}\label{n343}
\triple{J(\eta,Z^p\eta)f}{n,\sigma} \le C\bigl(\triple{\eta}{n,\sigma+1}\bigr)
\triple{\eta}{n+p,\sigma+1}\triple{f}{n,\sigma+1}.
\end{equation}
Directly from the definition of $J(\eta,Z^p\eta)$, it follows 
from the triangle inequality, 
the product rule~\eqref{prod:eCZ} and the estimate~\eqref{Z:2} for $\triple{G(\eta)f}{n,\sigma}$ that
\begin{align*}
\triple{J(\eta,Z^p\eta)f}{n,\sigma} &\le \btriple{G(\eta)\bigl((Z^p \eta)f\bigr)}{n,\sigma}
+\btriple{(Z^p\eta)G(\eta)f}{n,\sigma}\\
&\le C\bigl(\triple{\eta}{n,\sigma+1}\bigr)
\btriple{(Z^p \eta)f}{n,\sigma+1}\\
&\quad+\triple{Z^p\eta}{n,\sigma}C\bigl(\triple{\eta}{n,\sigma+1}\bigr)\triple{f}{n,\sigma+1}\\
&\le C\bigl(\triple{\eta}{n,\sigma+1}\bigr)
\triple{Z^p\eta}{n,\sigma+1}\triple{f}{n,\sigma+1}
\end{align*}
which implies the desired result~\eqref{n343}.

\step{3}{Induction}

So far we have proved that \e{n338} holds for $K=0$. 
To prove \eqref{n338} for $K>0$ we proceed by induction on $K$. Assuming 
that \e{n338} holds at rank $K$, we want to prove that,  
\begin{equation}\label{n344}
\double{J(\eta,Z^p\eta)f}{K+1,\mu} 
\le \Cr \triple{\eta}{s_0,0}\double{f}{K+1,\mu}
+\Cr \triple{f}{\mu+K+1+\tilde p-s_0+\npr,\gpr}\double{\eta}{K+p+1,\mu+1},
\end{equation}
where $\tilde{p}=\max (p,1)$. Notice that
\be\label{n345}
\double{J(\eta,Z^p\eta)f}{K+1,\mu}
\le \lA J(\eta,Z^p\eta)f\rA_{H^{\mu+K+1}}+\double{Z J(\eta,Z^p\eta)f}{K,\mu}.
\ee
The first term in the right hand side  of \e{n345} is estimated by \e{n340}. 
To estimate the second term, again, the key point is that one can express $Z J(\eta,Z^p\eta)f$ as a sum of 
terms which are estimated 
either by the induction hypothesis of by a previous estimate. 
By using the operators $J(\eta,\eta)=[G(\eta),\eta]$ and $J(\eta,Z\eta)=[G(\eta),Z\eta]$ 
and by using the identity $G(\eta)\B(\eta)\psi=-\px V(\eta)\psi$ (see Remark~\ref{rema:C3}), 
notice that one can rewrite the 
identity~\eqref{231} for $ZG(\eta)f$ under the form
\be\label{n346}
\begin{aligned}
ZG(\eta)\psi&=G(\eta)(Z\psi-2\psi)
-J(\eta,Z\eta)\B(\eta)\psi+2J(\eta,\eta)\B(\eta)\psi\\
&\quad -(\px Z\eta)V(\eta)\psi+2(\px\eta)V(\eta)\psi.
\end{aligned}
\ee
Then it is easily verified that
\begin{equation}\label{n347}
\begin{aligned}
Z J(\eta,Z^p\eta)f&=\mathcal{J}^1+\cdots +\mathcal{J}^{10}\\
&=J(\eta,Z^p\eta)(Zf-2f)+J(\eta,Z^{p+1}\eta)f\\
&\quad -J(\eta,Z\eta)\B(\eta)((Z^p\eta)f)+(Z^p\eta) J(\eta,Z\eta)\B(\eta)f\\
&\quad +2J(\eta,\eta)\B(\eta)((Z^p\eta)f)-2(Z^p\eta)J(\eta,\eta)\B(\eta)f\\
&\quad -(\px Z\eta)V(\eta)((Z^p\eta)f)+(Z^p\eta)(\px Z\eta)V(\eta)f\\
&\quad +2(\px\eta)V(\eta)((Z^p\eta)f)-2(Z^p\eta)(\px\eta)V(\eta)f.
\end{aligned}
\end{equation}

We now consider an integer $K$ in $[0,s_1-1]$ and assume that \eqref{n338} holds for any 
integer $p$ in $[0,s_1-K]$ and any real number $\mu$ in $[4,s-K-p-1]$. Our goal is to prove 
that \eqref{n344} holds for any $p$ in $[0,s_1-K-1]$ and any real number $\mu$ in $[4,s-K-p-2]$. 
To do so, in view of \eqref{n345}, it is sufficient to prove 
that, for any $i=1,\ldots, 10$,
\begin{equation}\label{n348}
\double{\mathcal{J}^i}{K,\mu}
\le \Cr \triple{\eta}{s_0,0}\double{f}{K+1,\mu}
+\Cr \triple{f}{\mu+K+1+\tilde p-s_0+\npr,\gpr}\double{\eta}{K+p+1,\mu+1},
\end{equation}
for any $p$ in $[0,s_1-K-1]$ and any real number $\mu$ in $[4,s-K-p-2]$. 

Given \eqref{n338}, it is clear that \eqref{n348} holds 
for $i=1$ or $i=2$. To estimate the other terms, we need 
some further preliminary estimates.

\noindent{\em Preliminary estimates}

In order to estimate $\mathcal{J}^3$, $\mathcal{J}^5$, $\mathcal{J}^7$ and $\mathcal{J}^9$ 
(see \eqref{n347}), we have to estimate
$\double{A(\eta)((Z^p\eta)f)}{K,\mu}$ for $A\in \{ \B,V\}$. We claim that
\be\label{n349}
\ba
\double{A(\eta)((Z^p\eta)f)}{K,\mu}&\le 
\Cr \triple{\eta}{s_0,0}\double{f}{K,\mu+1}\\
&\quad+\Cr \triple{f}{\mu+K+p-s_0+\npr,\gpr}\double{\eta}{K+p,\mu+1}.
\ea
\ee
To prove \eqref{n349}, use \eqref{z:Q} to obtain that
\be\label{n350}
\ba
\double{A(\eta)((Z^p\eta)f)}{K,\mu}&\le 
\Cr \double{(Z^p\eta)f}{K,\mu+1}\\
&\quad+\Cr \triple{ (Z^p\eta)f}{\mu+K-s_0+4,\gpr}\double{\eta}{K,\mu+1}.
\ea
\ee
Firstly, notice that \eqref{n335b} applied with $m=s_0-p$ implies that
\begin{equation*}
\begin{aligned}
\double{(Z^p\eta) f}{K,\mu+1} &\les 
\triple{Z^p\eta}{s_0-p,0}\double{f}{K,\mu+1}
+\triple{f}{\mu+K+p-s_0+3,0}\double{Z^p\eta}{K,\mu+1}\\
& \les 
\triple{\eta}{s_0,0}\double{f}{K,\mu+1}
+\triple{f}{\mu+K+p-s_0+3,0}\double{\eta}{K+p,\mu+1},
\end{aligned}
\end{equation*}
and hence $\double{(Z^p\eta)f}{K,\mu+1}$ is bounded by the right-hand side of \eqref{n349}. 
Secondly, observe that
\begin{align*}
\triple{ (Z^p\eta)f}{\mu+K-s_0+4,\gpr}
&\les \triple{ Z^p\eta}{\mu+K-s_0+4,\gpr}\triple{ f}{\mu+K-s_0+4,\gpr}\\
&\les \triple{\eta}{s_0,0}\triple{ f}{\mu+K-s_0+\npr,\gpr}
\end{align*}
since $\mu+K+p-s_0+4+\gpr\le s-s_0+4+\gpr\le s_0$ and since 
$4\le \npr$ by assumptions. This completes the proof of \eqref{n349}.

We need also to estimate $\triple{A(\eta)((Z^p\eta)f)}{\mu+K+1-s_0+\npr,\gpr}$ 
for $A\in \{\B,V\}$. 
To do so, write
\begin{equation}\label{n350a}
\begin{aligned}
&\triple{A(\eta)((Z^p\eta)f)}{\mu+K+1-s_0+\npr,\gpr}\\
&\qquad\le C\bigl( \triple{\eta}{\mu+K+1-s_0+\npr,\gpr+1}\bigr)
\triple{(Z^p\eta)f}{\mu+K+1-s_0+\npr,\gpr+1}\\
&\qquad\le C\bigl( \triple{\eta}{s_0,0}\bigr)\triple{f}{\mu+K+1-s_0+\npr,\gpr+1}\\
&\qquad\le C\bigl( \triple{\eta}{s_0,0}\bigr)\triple{f}{\mu+(K+1)+\tilde p-s_0+\npr,\gpr}
\end{aligned}
\end{equation}
where we used \eqref{Z:2}, \eqref{n339}, 
$\tilde p\ge 1$ and $\mu+K+p-s_0+\npr+\gpr+2\le s_0$ for $s_0\ge 1/2(s+2\gpr)$ with 
$\gpr$ large enough. 

Similarly we have that
\begin{equation}\label{n350b}
\triple{A(\eta)((Z^p\eta)f)}{\mu+K-s_0+4,0}
\le C\bigl( \triple{\eta}{s_0,0}\bigr)\triple{f}{\mu+K-s_0+\npr,0}
\end{equation}
for $\npr\ge 9$, where we used \eqref{Z:2} and $\mu+K+p-s_0+9\le s_0$ (for $\gpr$ large enough).

\noindent{\em Estimate of }$\mathcal{J}^3$ and $\mathcal{J}^5$. By the induction hypothesis one can 
apply \eqref{n338} with $p=1$ and $f$ replaced with $\B(\eta)((Z^p\eta)f)$ to obtain that
\begin{multline}\label{n352}
\double{J(\eta,Z\eta)\B(\eta)((Z^p\eta)f)}{K,\mu}
\le \Cr \triple{\eta}{s_0,0}\double{\B(\eta)((Z^p\eta)f)}{K,\mu}+\\
\Cr \triple{\B(\eta)((Z^p\eta)f)}{\mu+K+\tilde 1-s_0+\npr,\gpr}\double{\eta}{K+1,\mu+1}.
\end{multline}
The first (resp.\ second) term in the right-hand side of \eqref{n352} 
is estimated by means of \eqref{n349} (resp.\ \eqref{n350a}). 
This gives
\begin{multline}\label{n353}
\double{J(\eta,Z\eta)\B(\eta)((Z^p\eta)f)}{K,\mu}\le 
\Cr \triple{\eta}{s_0,0}\double{f}{K+1,\mu}\\
+\Cr \triple{f}{\mu+K+1+\tilde p-s_0+\npr,\gpr}\double{\eta}{K+p+1,\mu+1}.
\end{multline}
Thus we verify that \eqref{n348} holds for $i=3$. The proof for $i=5$ is similar. 

\noindent{\em Estimate of }$\mathcal{J}^4$ and $\mathcal{J}^6$. 
The product rule \eqref{n335b} (applied with $m=s_0-p$) 
implies that
\begin{multline}\label{n354}
\double{(Z^p\eta) J(\eta,Z\eta)\B(\eta)f}{K,\mu}
\les \Cr \triple{\eta}{s_0,0}\double{J(\eta,Z\eta)\B(\eta)f}{K,\mu}
\\
+\Cr \triple{J(\eta,Z\eta)\B(\eta)f}{\mu+K+p-s_0+2,0}\double{\eta}{K+p,\mu+1}.
\end{multline}
The first term in the right-hand side of \eqref{n354} is estimated by means of \eqref{n352} 
(with $Z^p\eta$ replaced with $1$). With regards to the second term, using 
\e{n343} and \e{Z:2}, we obtain for any $\epsilon\in ]0,1[$,
\begin{align*}
\triple{J(\eta,Z\eta)\B(\eta)f}{\mu+K+p-s_0+2,0}
&\le \triple{J(\eta,Z\eta)\B(\eta)f}{\mu+K+p-s_0+2,4-\epsilon}\\
&\le \Cr \triple{\B(\eta)f}{\mu+K+p-s_0+2,5-\epsilon}\\
&\le \Cr \triple{f}{\mu+K+p-s_0+2,6}\\
&\le \Cr \triple{f}{\mu+K+\tilde p-s_0+\npr,\gpr}.
\end{align*}
This proves that \eqref{n348} holds for $i=4$. The proof for $i=6$ is similar.

\noindent{\em Estimate of }$\mathcal{J}^7$ and $\mathcal{J}^9$. 
The product rule \eqref{n335b} implies that
\begin{multline*}
\double{(\px Z\eta)V(\eta)((Z^p\eta)f)}{K,\mu}
\les \triple{\px Z\eta}{s_0-2,0}\double{V(\eta)((Z^p\eta)f)}{K,\mu}\\
+\triple{V(\eta)((Z^p\eta)f)}{\mu+K-s_0+4,0}\double{\px Z\eta}{K,\mu}.
\end{multline*}
Since $\triple{\px Z\eta}{s_0-2,0}\le \triple{\eta}{s_0,0}$ and 
$\double{\px Z\eta}{K,\mu}\le \double{\eta}{K+1,\mu+1}$, in view of \eqref{n349} and \eqref{n350b} 
we verify that \eqref{n348} holds for $i=7$. The proof for $i=9$ is similar. 

The estimates for $i=8$ and $i=10$ are simpler. This completes the proof.
\end{proof}

\begin{coro}\label{T62}
There exists $\eps_0>0$ small enough and there exist $\gi$ with 
$\gi\not\in\mez\xN$ and $\npi$ large enough such that, 
for any $(s,s_1,s_0)\in\xN^3$ satisfying
\begin{equation*}
s\ge s_1\ge s_0\ge \mez (s+2\gi),
\end{equation*}
for any integer $K$ in $[0,s_1]$ 
and any real number $\mu$ in $[4,s-K-1]$ there exists a nondecreasing function 
$C$ such that, for any $T>0$ and any smooth 
functions $(\eta,f)$ such that $\sup_{t\in [0,T]}\lA \eta(t)\rA_{\eC{s_0}}\le \eps_0$, 
\begin{equation}\label{n356}
\double{G(\eta)f-\Dx f}{K,\mu} 
\le \Cr \triple{\eta}{s_0,0}\double{f}{K,\mu}
+\Cr \triple{f}{\mu+K-s_0+\npi,\gi}\double{\eta}{K,\mu+1},
\end{equation}
where $\Cr=C( \triple{\eta}{s_0,0} )$. 
\end{coro}
\begin{proof}Again, the proof proceeds by induction on $K$. 
It follows from Proposition~\ref{T18} and Proposition~\ref{T22} that \eqref{n356} is true 
for $K=0$. 

Since $[Z,\Dx]=-2\Dx$, it follows from \eqref{n346} that
\begin{align*}
Z \bigl(G(\eta)f-\Dx f\bigr)&=\bigl(G(\eta)-\Dx)(Zf-2f)-J(\eta,Z\eta)\B(\eta)f\\
&\quad+2J(\eta,\eta)\B(\eta)f -(\px Z\eta)V(\eta)f+2(\px\eta)V(\eta)f.
\end{align*}
So the desired result follows from the estimates already established in the last step of the proof 
of Proposition~\ref{T60}.
\end{proof}

We are now in position to estimate $ZF(\eta)\psi-ZF_{\quadratique}(\eta)\psi$. 

\begin{prop}\label{T62.5}
There exists $\eps_0>0$ small enough and there exist $\gii',\gii$ with 
$\gii\not\in\mez\xN$, $\gii>\gii'$ and $\nii$ large enough such that, 
for any $(s,s_1,s_0)\in\xN^3$ satisfying
\begin{equation}\label{n356.3}
s\ge s_1\ge s_0\ge \mez (s+2\gii),
\end{equation}
for $k$ in $[0,s_1]$ 
and 
any real number $\mu$ in $[4,s-k]$ there exists a nondecreasing function 
$C$ such that, for any $T>0$ and any smooth 
functions $(\eta,f)$ such that $\sup_{t\in [0,T]}\lA \eta(t)\rA_{\eC{s_0}}\le \eps_0$, 
\be\label{n357}
\ba
&\double{F(\eta)\psi-F_{\quadratique}(\eta)\psi}{k,\mu} \\
&\qquad\le \Cri \lA \eta\rA_{\eC{\gii}}^2\blA \Dxmez Z^k \psi\brA_{H^{\mu-\mez}}\\
&\qquad\quad + \indicator{\xR_+}(\mu+k-s_0+\nii)\Cri \lA \eta\rA_{\eC{\gii}} \blA \Dxmez \psi\brA_{\eC{\gii}} 
\blA Z^k\eta\brA_{H^\mu}\\
&\qquad\quad +\Crs\triple{\eta}{s_0,0}^2
\bdouble{\Dxmez\psi}{k-1,\mu+\mez}\\
&\qquad\quad +\indicator{\xR_+}(\mu-\gamma_2')
\Crs\triple{\eta}{s_0,0}^2
\bdouble{\Dxmez\psi}{k,\mu-\tdm}\\
&\qquad\quad+\Crs \triple{\eta}{s_0,0} \btriple{\Dxmez\psi}{\mu+k-s_0+\nii,\gii}
\double{\eta}{k-1,\mu+1}\\
&\qquad\quad +\indicator{\xR_+}(\mu-\gii')
\Crs \triple{\eta}{s_0,0} \btriple{\Dxmez\psi}{\mu+k-s_0+\nii,\gii}
\double{\eta}{k,\mu-1},
\ea
\ee
where $\Cri=C(\lA \eta\rA_{\eC{\gii}})$, 
$\Crs=C( \triple{\eta}{s_0,0})$, and 
$\indicator{\xR_+}$ is the indicator function of $\xR_+$.
\end{prop}

\begin{proof}
Let $\gamma$ be large enough, 
$(\np,\gpr)$ and $(\npi,\gi)$ be as given by the statements of Proposition~\ref{T56}, Proposition~\ref{T60} and Corollary~\ref{T62}, respectively. 
Then $\nii$, $\gii,\gii'$ will be chosen so that 
$$
\gii'=\gii-1,\quad \nii\ge \max(N,\np,\npi,5),\quad \gii\ge \max(\gamma+7/2,\gpr,\gi,\nii+1)
$$ 
(with $\gii\not\in\mez\xN$).

The proof proceeds by induction. Notice first that \eqref{n357}~holds 
for $k=0$: if $\mu-s_0+\nii\ge 0$, we apply \e{n137} with $(\mu,s)$ 
replaced by $(\mu-1,\mu-1)$ and get that the left hand side is bounded 
by the first and second terms in the right hand side. If $\mu-s_0+\nii<0$, we use Corollary 
\ref{T26} with $\gamma=\mu+4$. If moreover, $\mu\le \gii-4$ we 
obtain a bound by the first term in the right hand side  of \e{n357}. If 
$\mu>\gii-4$, we use the fourth term in that right hand side  to get that bound (taking $\gii>\gii'+4$).

Hereafter we fix an integer $k$ in $[0,s_1-1]$ and 
we assume that for any real number $\mu$ in $[4,s-k]$ the estimate \eqref{n357} holds.  
Our goal is to prove that \eqref{n357} holds at rank $k+1$. 
Since $\double{F(\eta)\psi-F_{\quadratique}(\eta)\psi}{k+1,\mu} $ is smaller than 
$$
\lA F(\eta)\psi-F_{\quadratique}(\eta)\psi\rA_{H^{\mu+k+1}}
+\double{Z\bigl(F(\eta)\psi-F_{\quadratique}(\eta)\psi\bigr)}{k,\mu},
$$
this reduces to proving that, for any $\mu\in [4,s-k-1]$,
\be\label{n358}
\ba
&\double{Z(F(\eta)\psi-F_{\quadratique}(\eta)\psi)}{k,\mu}  \\
&\qquad\le \Cri \lA \eta\rA_{\eC{\gii}}^2\blA \Dxmez Z^{k+1} \psi\brA_{H^{\mu-\mez}}\\
&\qquad\quad + \indicator{\xR_+}(\mu+k+1-s_0+\nii)\Cri \lA \eta\rA_{\eC{\gii}} \blA \Dxmez \psi\brA_{\eC{\gii}} 
\blA Z^{k+1}\eta\brA_{H^\mu}\\
&\qquad\quad +\Crs\triple{\eta}{s_0,0}^2
\bdouble{\Dxmez\psi}{k,\mu+\mez}\\
&\qquad\quad +\indicator{\xR_+}(\mu-\gamma_2')
\Crs\triple{\eta}{s_0,0}^2
\bdouble{\Dxmez\psi}{k+1,\mu-\tdm}\\
&\qquad\quad+\Crs \triple{\eta}{s_0,0} \btriple{\Dxmez\psi}{\mu+k+1-s_0+\nii,\gii}
\double{\eta}{k,\mu+1}\\
&\qquad\quad +\indicator{\xR_+}(\mu-\gii')
\Crs \triple{\eta}{s_0,0} \btriple{\Dxmez\psi}{\mu+k+1-s_0+\nii,\gii}
\double{\eta}{k+1,\mu-1},
\ea
\ee
provided that $\nii,\gii,\gii'$ are large enough. 

To prove \eqref{n358}, we express 
$Z(F(\eta)-F_{\quadratique}(\eta))\psi$ as the sum of 
$(F(\eta)-F_{\quadratique}(\eta))(Z-2)\psi$, which we are going 
to estimate by the induction hypothesis, and other terms which are estimated either by the induction hypothesis 
or by means of the previous results. 

Recall that, by Lemma~\ref{T46}
\be\label{n359}
\begin{aligned}
ZF(\eta)\psi 
&= 
F(\eta)\bigl(Z\psi-2\psi\bigr)-F(\eta)\bigl((Z\eta)\B(\eta)\psi\bigr)\\
&\quad -\Dx T_{Z\eta}\B(\eta)\psi -\px\bigl(T_{Z\eta} V(\eta)\psi\bigr)\\
&\quad+2 G(\eta)(\eta \B(\eta)\psi)-2 \eta G(\eta)\B(\eta)\psi+2(V(\eta)\psi)\px \eta  \\
&\quad-\Dx \RBony(\B(\eta)\psi,Z\eta)-\px \RBony(Z\eta,V(\eta)\psi)\\
&\quad+\Dx T_{\Rzero{B}(\eta)\psi}\eta+\px (T_{\Rzero{V}(\eta)\psi}\eta)\\
&\quad+2\Dx S_\Bony(\B(\eta)\psi,\eta)+2\px S_\Bony(V(\eta)\psi,\eta)
\end{aligned}
\ee
where $S_\Bony$ is given by \eqref{n227}; $\Rzero{B}$ and 
$\Rzero{V}$ 
are given by \eqref{N302} and \eqref{N303} and $\RBony(a,b)=ab-T_a b-T_b a$.

On the other hand, remembering that according to \e{n135}
\begin{equation}\label{n360}
F_{\quadratique}(\eta)\psi=
-\Dx(\eta\Dx \psi)+\Dx (T_{\Dx\psi}\eta)-\partial_x(\eta\partial_x\psi)
+\partial_x(T_{\partial_x\psi}\eta),
\end{equation}
by using 
$[Z,\Dx]=-2\Dx$, $[Z,\px]=-2\px$ and \eqref{n227} one gets that
$$
\begin{aligned}
ZF_{\quadratique}(\eta)\psi&=F_{\quadratique}(Z\eta)\psi+F_{\quadratique}(\eta)Z\psi
-4F_{\quadratique}(\eta)\psi\\
&\quad +2\Dx S_\Bony(\Dx\psi,\eta)+2\px S_\Bony(\px\psi,\eta),
\end{aligned}
$$
which is better written under the form
\begin{multline}\label{n362}
ZF_{\quadratique}(\eta)\psi-F_{\quadratique}(\eta)(Z\psi-2\psi)\\
=F_{\quadratique}(Z\eta)\psi-2F_{\quadratique}(\eta)\psi+2\Dx S_\Bony(\Dx\psi,\eta)
+2\px S_\Bony(\px\psi,\eta),
\end{multline}
We have already seen (see~\eqref{n163}) that one can either write 
$F_{\quadratique}(\eta)\psi$ under the form \eqref{n360} or 
under the form
\be\label{n363}
F_{\quadratique}(\eta)\psi=-\Dx \RBony(\eta,\Dx\psi)-\px \RBony(\eta,\px\psi).
\ee
In the right-hand side of \eqref{n362} we use 
\eqref{n363} to express $F_{\quadratique}(Z\eta)\psi$ 
and \eqref{n360} to express $-2F_{\quadratique}(\eta)\psi$. It is found that
\be\label{n364}
\begin{aligned}
&ZF_{\quadratique}(\eta)\psi\\
&\quad=F_{\quadratique}(\eta)(Z\psi-2\psi)\\
&\quad\quad-\Dx \RBony(Z\eta,\Dx\psi)-\px \RBony(Z\eta,\px\psi)\\
&\quad\quad -2\Bigl(-\Dx(\eta\Dx \psi)+\Dx (T_{\Dx\psi}\eta)-\partial_x(\eta\partial_x\psi)
+\partial_x(T_{\partial_x\psi}\eta)\Bigr)\\
&\quad\quad +2\Dx S_\Bony(\Dx\psi,\eta)+2\px S_\Bony(\px\psi,\eta).
\end{aligned}
\ee

Now by combining \eqref{n359} and \eqref{n364}, we conclude that
$$
Z\bigl(F(\eta)\psi -F_{\quadratique}(\eta)\psi\bigr)
=\mathcal{F}^0+\cdots +\mathcal{F}^6
$$
where
\begin{align*}
\mathcal{F}^0&=\bigl(F(\eta)-F_{\quadratique}(\eta)\bigr)\bigl(Z\psi-2\psi\bigr),\\
\mathcal{F}^1&=-F(\eta)\bigl((Z\eta)\B(\eta)\psi\bigr),\\
\mathcal{F}^2&= -\Dx T_{Z\eta}\B(\eta)\psi -\px\bigl(T_{Z\eta} V(\eta)\psi\bigr),\\
\mathcal{F}^3&=2 G(\eta)(\eta \B(\eta)\psi)-2 \eta G(\eta)\B(\eta)\psi+2(V(\eta)\psi)\px \eta \\
&\quad-2\Dx(\eta\Dx \psi)-2\px(\eta \px\psi),\\
\mathcal{F}^4&=-\Dx \RBony(\B(\eta)\psi,Z\eta)-\px \RBony(Z\eta,V(\eta)\psi)\\
&\quad+\Dx \RBony(Z\eta,\Dx\psi)+\px \RBony(Z\eta,\px\psi),\\
\mathcal{F}^5&=2\Dx S_\Bony(\B(\eta)\psi-\Dx\psi,\eta)
+2\px S_\Bony(V(\eta)\psi-\px\psi,\eta),\\
\mathcal{F}^6&=\Dx T_{\Rzero{B}(\eta)\psi}\eta+\px (T_{\Rzero{V}(\eta)\psi}\eta)
+2\Dx (T_{\Dx\psi}\eta)
+2\partial_x(T_{\partial_x\psi}\eta).\\
\end{align*}

To prove \eqref{n358}, we have to prove that, for any $\mu\in [4,s-k-1]$ 
and any $0\le i\le 6$, 
\be\label{n366}
\ba
\double{\mathcal{F}^i}{k,\mu} 
&\le 
\Cri \lA \eta\rA_{\eC{\gii}}^2\blA \Dxmez Z^{k+1} \psi\brA_{H^{\mu-\mez}}\\
&\quad + \indicator{\xR_+}(\mu+k+1-s_0+\nii)\Cri \lA \eta\rA_{\eC{\gii}} \blA \Dxmez \psi\brA_{\eC{\gii}} 
\blA Z^{k+1}\eta\brA_{H^\mu}\\
&\quad +\Crs\triple{\eta}{s_0,0}^2
\bdouble{\Dxmez\psi}{k,\mu+\mez}\\
&\quad +\indicator{\xR_+}(\mu-\gamma_2')
\Crs\triple{\eta}{s_0,0}^2
\bdouble{\Dxmez\psi}{k+1,\mu-\tdm}\\
&\quad+\Crs \triple{\eta}{s_0,0} \btriple{\Dxmez\psi}{\mu+k+1-s_0+\nii,\gii}
\double{\eta}{k,\mu+1}\\
&\quad +\indicator{\xR_+}(\mu-\gii')
\Crs \triple{\eta}{s_0,0} \btriple{\Dxmez\psi}{\mu+k+1-s_0+\nii,\gii}
\double{\eta}{k+1,\mu-1}.
\ea
\ee

The estimate \eqref{n366} for $i=0$ follows from the induction hypothesis, by applying 
\eqref{n357} with $\psi$ replaced with $Z\psi-2\psi$. We shall 
estimate the other terms separately.

\step{0}{Preliminary}

We shall need to estimate $\triple{G(\eta)f-\Dx f}{n,\sigma}$, 
$\triple{\B(\eta)f-\Dx f}{n,\sigma}$ and 
$\triple{V(\eta)f-\px f}{n,\sigma}$. 
We claim that, for any integer $n$ in $[0,s_0-5]$ 
and any $\sigma$ in $\pol3,s_0-1-k]$ with $\sigma\not\in \mez\xN$,
\begin{multline}\label{n367}
\triple{G(\eta)f-\Dx f}{n,\sigma}+
\triple{\B(\eta)f-\Dx f}{n,\sigma}+
\triple{V(\eta)f-\px f}{n,\sigma}\\
\le C\bigl(\triple{\eta}{n,\sigma+2}\bigr)\triple{\eta}{n,\sigma+2}
\btriple{\Dxmez f}{n,\sigma+\tdm}.
\end{multline}
(The key point is that the right-hand side is at least quadratic; there is a loss 
of one derivative since we estimate the $\triple{\cdot}{n,\sigma}$-norm 
of $A(\eta)f-A(0)f$ by means of the $\triple{\cdot}{n,\sigma+2}$-norm of 
$f$ while $A(\eta)-A(0)$ is of order $1$, but this loss 
is harmless for our purposes.) 

To fix matters, we prove \eqref{n367} for $G(\eta)f-\Dx f$ only. 
Write
$$
\triple{G(\eta)f-\Dx f}{n,\sigma}
=\sum_{k=0}^n\blA Z^k (G(\eta)f-\Dx f)\brA_{\eC{\sigma+n-k}}.
$$
Now $Z^k \Dx =\Dx (Z-2)^k$ so
$$
Z^k (G(\eta)f-\Dx f)
=\Bigl(Z^k G(\eta) f-G(\eta)(Z-2)^k\Bigr) +(G(\eta)-\Dx)(Z-2)^k f.
$$
It follows from \eqref{Z:1} that
$$
\blA Z^k G(\eta)\psi-G(\eta)(Z-2)^k \psi\brA_{\eC{\sigma}}
\le C\bigl( \triple{\eta}{k,\sigma+1}\bigr) 
\triple{\eta}{k,\sigma+1}\btriple{\Dxmez \psi}{k-1,\sigma+\tdm}
$$
On the other hand \eqref{n145} implies 
that, for any $\sigma>3$ with $\sigma\not\in\mez\xN$,
\begin{multline*}
\blA (G(\eta)-\Dx)(Z-2)^k f\brA_{\eC{\sigma+n-k}}\\
\le C(\lA \eta\rA_{\eC{\sigma+n-k+2}})\lA \eta\rA_{\eC{\sigma+n-k+2}}
\blA \Dxmez (Z-2)^k f\brA_{\eC{\sigma+n-k+\tdm}}.
\end{multline*}
Since
$$
\lA \eta\rA_{\eC{\sigma+n-k+2}}\le \triple{\eta}{n,\sigma+2},\quad 
\blA \Dxmez (Z-2)^k f\brA_{\eC{\sigma+n-k+\tdm}}
\le \btriple{\Dxmez f}{n,\sigma+\tdm},
$$
this completes the proof of \eqref{n367}.

We shall also use the following corollary of \e{2311aa}: 
let $A(\eta)$ be one of the operators 
$G(\eta)$, $\B(\eta)$, $V(\eta)$, then 
\be\label{n368}
\ba
\double{(A(\eta)-A(0))\psi}{k,\mu}
&\le  \Crs \triple{\eta}{s_0,0}
\bdouble{\Dxmez\psi}{k,\mu+\mez}\\
&\quad+\Crs 
\btriple{\Dxmez\psi}{\mu+k+1-s_0+\nii,\gii}\double{\eta}{k,\mu+1}.
\ea
\ee

\step{1}{Estimate of $\mathcal{F}^1$.}

To estimate $\mathcal{F}^1$ we first claim that 
\eqref{n357} implies that,
\begin{align*}
\bdouble{F(\eta)\wpsi}{k,\mu} &\le 
\Cri \lA \eta\rA_{\eC{\gii}}\blA Z^k \wpsi\brA_{H^{\mu}}\\
&\quad +\Crs\triple{\eta}{s_0,0}
\bdouble{\wpsi}{k-1,\mu+1}\\
&\quad+\indicator{\xR_+}(\mu-\gamma_2')
\Crs\triple{\eta}{s_0,0}
\bdouble{\wpsi}{k,\mu-1}\\
&\quad+\Crs \btriple{\wpsi}{\mu+k-s_0+\nii,\gii+\mez}
\double{\eta}{k,\mu}.
\end{align*}
To prove this estimate, using \eqref{n357} and the triangle inequality, it is sufficient to 
prove that $\double{F_{\quadratique}(\eta)\psi}{k,\mu}$ is bounded by the right-hand side of the 
above inequality. This in turn follows from~\eqref{n363} and \e{n325c} applied with 
$(m,a,\nu)=(s_0-2,2,\mu-1)$.

To estimate $\double{\mathcal{F}^1}{K,\mu}$ we now apply the previous 
estimate with $\wpsi$ replaced with $(Z\eta)\B(\eta)\psi$. This yields
\be\label{n368.5}
\ba
\double{F(\eta)\bigl((Z\eta)\B(\eta)\psi\bigr)}{k,\mu} &\le 
\Cri \lA \eta\rA_{\eC{\gii}}\blA Z^k \bigl((Z\eta)\B(\eta)\psi\bigr)\brA_{H^{\mu}}\\
&\quad +\Crs\triple{\eta}{s_0,0}
\bdouble{(Z\eta)\B(\eta)\psi}{k-1,\mu+1}\\
&\quad+\indicator{\xR_+}(\mu-\gamma_2')
\Crs\triple{\eta}{s_0,0}
\bdouble{(Z\eta)\B(\eta)\psi}{k,\mu-1}\\
&\quad+\Crs \triple{(Z\eta)\B(\eta)\psi}{\mu+k-s_0+\nii,\gii+\mez}
\double{\eta}{k,\mu}.
\ea
\ee
To estimate the first term in the right-hand side of the above inequality we 
use the product rule \eqref{n335e} 
with $m=s_0-1$, $b=\gii$, $\zeta=Z\eta$ and $F=\B(\eta)\psi$, we find that 
$$
\ba
\double{(Z\eta)\B(\eta)\psi}{k,\mu}
&\les \triple{\eta}{s_0,0}\double{\B(\eta)\psi}{k,\mu}\\
&\quad +\indicator{\xR_+}(\mu+k-s_0+3) \lA B(\eta)\psi\rA_{\eC{\gii}}
\blA Z^{k+1}\eta\brA_{H^\mu}\\
&\quad+ \triple{\B(\eta)\psi}{\mu+k-s_0+3,0}\double{\eta}{k,\mu+1}\\
&\quad +\indicator{\xR_+}(\mu+1-\gii)
  \triple{\B(\eta)\psi}{\mu+k-s_0+3,0}
\double{\eta}{k+1,\mu-1}.
\ea
$$
Now $\double{\B(\eta)\psi}{k,\mu}$ is estimated by means of Proposition~\ref{T56}. 
On the other hand, 
$$
\btriple{(Z\eta)\B(\eta)\psi}{\mu+k-s_0+\nii,\gii+\mez}
\les \triple{Z\eta}{\mu+k-s_0+\nii,\gii+\mez}
\triple{\B(\eta)\psi}{\mu+k-s_0+\nii,\gii+\mez}.
$$
If $\gii\ge \nii+1$ then 
\be\label{n369}
\mu+k+1-s_0+\nii+\gii\le s-s_0+\nii+\gii\le s_0+\nii-\gii\le s_0-1.
\ee
Therefore 
$\triple{Z\eta}{\mu+k-s_0+\nii,\gii+\mez}\le \triple{\eta}{s_0,0}$. Moreover \eqref{n369} 
implies that 
we can apply 
Proposition~\ref{T52}  to bound $\triple{\B(\eta)\psi}{\mu+k-s_0+\nii,\gii+\mez}$ 
(and hence $\triple{\B(\eta)\psi}{\mu+k-s_0+3,0}$). 
This completes the estimate of the first and last term 
in the \rhs of \e{n368.5}. 

It remains to estimate the second and third terms in the \rhs 
of \e{n368.5}. Both terms are estimated similarly and 
we consider the third one only. To estimate this term 
we use the product rule 
\e{n335b} (instead of the product rule \e{n335e} used above) 
applied with $m=s_0-1$. 
This yields
\begin{align*}
\double{(Z\eta)\B(\eta)\psi}{k,\mu-1}
&\les \triple{Z\eta}{s_0-1,0}
\double{\B(\eta)\psi}{k,\mu-1}
+\triple{\B(\eta)\psi}{\mu+k-s_0+1,0}
\double{Z\eta}{k,\mu-1}\\
&\les \triple{\eta}{s_0,0}
\double{\B(\eta)\psi}{k,\mu-1}
+\triple{\B(\eta)\psi}{\mu+k-s_0+1,0}
\double{\eta}{k+1,\mu-1}.
\end{align*}
Then we use \e{z:Q} (resp.\ \e{Z:2}) 
to estimate the $\double{\cdot}{k,\mu-1}$-norm 
(resp.\ $\triple{\cdot}{\mu+k-s_0+1,0}$-norm) 
of $\B(\eta)\psi$.

We conclude that \eqref{n366} holds for $i=1$.

\step{2}{Estimate of $\mathcal{F}^2$.} 

Write
\begin{align*}
-\mathcal{F}^2&= \Dx T_{Z\eta}\B(\eta)\psi +\px\bigl(T_{Z\eta} V(\eta)\psi\bigr)\\
&=\bigl[ \Dx,T_{Z\eta}\bigr]\B+T_{\px Z\eta}V
+T_{Z\eta}\Dx\B +T_{Z\eta} \px V
\\
&=\bigl[\Dx,T_{Z\eta}\bigr] \Dx \psi 
+T_{\px Z\eta}\px\psi
+\widetilde{\mathcal{F}}^2,
\end{align*}
with
\begin{align*}
\widetilde{\mathcal{F}}^2&=T_{Z\eta}\bigl( \Dx \B(\eta)\psi +\px V(\eta)\psi\bigr)\\
&\quad+\bigl[ \Dx,T_{Z\eta}\bigr] (\B(\eta)\psi -\Dx\psi)+T_{\px Z\eta}(V(\eta)\psi-\px\psi).
\end{align*}
Since $\Dx^2=-\px^2$, it follows from identity~\eqref{A:i1} that
$$
\bigl[\Dx,T_{Z\eta}\bigr] \Dx \psi +T_{\px Z\eta}\px\psi
=\Dx T_{Z\eta}\Dx \psi+\px(T_{Z\eta}\px\psi)=0.
$$
It remains to estimate $\widetilde{\mathcal{F}}^2$ 
(which is equal to $-\mathcal{F}^2$ in view of the above cancellation). 
The estimates for $\B(\eta)\psi$ and~$V(\eta)\psi$ would be 
insufficient to control $\Dx \B(\eta)\psi+\px V(\eta)\psi$. 
We remedy this by using 
the identity~$\px V(\eta)\psi=-G(\eta)\B(\eta)\psi$ 
(see \eqref{232e}) and hence
$$
\Dx \B(\eta)\psi+\px V(\eta)\psi =\Dx \B (\eta)\psi-G(\eta)\B(\eta)\psi.
$$
Therefore, we conclude that
\begin{align*}
-\mathcal{F}^2&=\mathcal{F}^2_a+\mathcal{F}^2_b+\mathcal{F}^2_c\\
&=-T_{Z\eta}\bigl(G(\eta)-\Dx)\B(\eta)\psi
+\bigl[ \Dx,T_{Z\eta}\bigr] (\B(\eta)-\Dx)\psi\\
&\quad+T_{\px Z\eta}(V(\eta)-\px)\psi.
\end{align*}
These three terms are estimated by similar arguments. 

Let us estimate $\mathcal{F}^2_a=-T_{Z\eta}\bigl(G(\eta)-\Dx)\B(\eta)\psi$. 
Set $A(\eta)\psi=\bigl(G(\eta)-\Dx)\B(\eta)\psi$. 
We shall use a corollary of the 
estimate \eqref{n369.5} whose statement is recalled here
$$
\ba
\double{T_{\zeta}F}{K,\nu} &\les 
\triple{\zeta}{m,0}\double{F}{K-1,\nu+1}
+\indicator{\xR_+}(m)\lA \zeta\rA_{L^\infty}
\blA Z^K F\brA_{H^\nu}\\
&\quad+\triple{F}{\nu+K-m+1,0}\double{\zeta}{K-1,0}\\
&\quad +\indicator{\xR_+}(\nu+K-m+1)\lA F\rA_{\eC{b}}
\blA Z^K\zeta\brA_{L^2}\\
&\quad +\indicator{\xR_+} (\nu+1-b)\lA F\rA_{\eC{\nu+K-m+1}}\blA Z^K\zeta\brA_{L^2}.
\ea
$$
By using the obvious inequalities 
$$
\indicator{\xR_+}(m)\lA \zeta\rA_{L^\infty}
\blA Z^K F\brA_{H^\nu}
\le \triple{\zeta}{m,0}\double{F}{K,\nu},\quad 
\double{F}{K-1,\nu+1}\le \double{F}{K,\nu},
$$
this yields
$$
\ba
\double{T_{\zeta}F}{K,\nu} &\les 
\triple{\zeta}{m,0}\double{F}{K,\nu}\\
&\quad+\triple{F}{\nu+K-m+1,0}\double{\zeta}{K-1,0}\\
&\quad +\indicator{\xR_+}(\nu+K-m+1)\lA F\rA_{\eC{b}}
\blA Z^K\zeta\brA_{L^2}\\
&\quad +\indicator{\xR_+} (\nu+1-b)\lA F\rA_{\eC{\nu+K-m+1}}\blA Z^K\zeta\brA_{L^2}.
\ea
$$
By applying this estimate with 
$(K,\nu,m,b)$ replaced by $(k,\mu,s_0-1,\gii)$, 
we then obtain that
\be\label{n370}
\ba
\double{\mathcal{F}^2_a}{k,\mu}
&\les
\triple{Z\eta}{s_0-1,0}\double{A(\eta)\psi}{k,\mu}\\
&\quad +\indicator{\xR_+}(\mu+k-s_0+2)
\lA A(\eta)\psi\rA_{\eC{\gii}}\blA Z^{k+1}\eta\brA_{H^\mu}\\
&\quad
+\triple{A(\eta)\psi}{\mu+k-s_0+2,0}\double{Z\eta}{k-1,0}\\
&\quad +\indicator{\xR_+}(\mu-\gii+1)\triple{A(\eta)\psi}{\mu+k-s_0+2,0}
\double{\eta}{k+1,\mu-1}.
\ea
\ee
Now it follows from \eqref{n356} that
\begin{equation*}
\double{A(\eta)\psi}{k,\mu} 
\le \Cr \triple{\eta}{s_0,0}\double{\B(\eta)\psi}{k,\mu}
+\Cr \triple{\B(\eta)\psi}{\mu+k-s_0+\npi,\gi}\double{\eta}{k,\mu+1},
\end{equation*}
and \eqref{Z:2} and \eqref{z:Q} imply that
\begin{align*}
&\triple{\B(\eta)\psi}{\mu+k-s_0+\npi,\gi}
\le C\bigl( \triple{\eta}{\mu+k-s_0+\npi,\gi+1}\bigr) 
\btriple{\Dxmez \psi}{\mu+k-s_0+\npi,\gi+\mez},\\
&\bdouble{\B(\eta)\psi}{k,\mu} 
\le \Cr \bdouble{\Dxmez \psi}{k,\mu+\mez}
+\Cr \btriple{\Dxmez\psi}{\mu-s_0+4+k,\gamma}\double{\eta}{k,\mu+1},
\end{align*}
Therefore
\begin{equation}\label{n371}
\double{A(\eta)\psi}{k,\mu} 
\le \Cr \triple{\eta}{s_0,0} \bdouble{\Dxmez \psi}{k,\mu+\mez}
+\Cr \btriple{\Dxmez \psi}{\mu+k-s_0+\nii,\gii}\double{\eta}{k,\mu+1}.
\end{equation}

On the other hand, it follows from \eqref{n367} that 
\be\label{n372-1}
\begin{aligned}
&\triple{A(\eta)\psi}{\mu+k-s_0+2,0}\\
&\qquad \le \triple{A(\eta)\psi}{\mu+k-s_0+2,\sigma_0}\\
&\qquad\le C\bigl(\triple{\eta}{\mu+k-s_0+2,\sigma_0+2}\bigr)
\triple{\eta}{\mu+k-s_0+2,\sigma_0+2}
\btriple{\Dxmez \psi}{\mu+k-s_0+2,\sigma_0+\tdm}
\end{aligned}
\ee
where the index $\sigma_0$ appears in the first inequality because 
\eqref{n367} is proved only for $\sigma$ larger than some 
number $\sigma_0$ large enough. 
Now, by assumption on $\mu$ we have $\mu\le s-k-1$ and 
by assumption on $(s,s_0)$ we have 
$s\le 2s_0-2\gii$. Thus, if $\gii$ is large enough 
(namely for $2\gii\ge \sigma_0+4$) we have 
$\mu+k\le s\le 2s_0-2\gii\le 2s_0-\sigma_0-4$ 
and hence 
$$
\triple{\eta}{\mu+k-s_0+2,\sigma_0+2}\le \triple{\eta}{s_0,0}.
$$
Thus \e{n372-1} implies that
\be\label{n372}
\triple{A(\eta)\psi}{\mu+k-s_0+2,0}
\le \Cr \triple{\eta}{s_0,0}\btriple{\Dxmez\psi}{\mu+k+1-s_0+\nii,\gii},
\ee
Setting \eqref{n371} and \eqref{n372}~into \eqref{n370}, we obtain that 
$\double{\mathcal{F}^2_a}{k,\mu}$ is estimated by the right-hand side 
of \eqref{n366}.

\step{3}{Analysis of $\mathcal{F}^3$.} 

Write
$$
G(\eta)(\eta \B)-\eta G(\eta)\B=
G(\eta)(\eta \Dx\psi)-\eta G(\eta)\Dx\psi+J(\eta,\eta)(\B(\eta)\psi-\Dx\psi),
$$
where recall that by definition $J(\eta,\eta)f=G(\eta)(\eta f)-\eta G(\eta)f$. Then
\begin{align*}
&G(\eta)(\eta \B(\eta)\psi)- \eta G(\eta)\B(\eta)\psi \\
&\qquad=\Dx (\eta  \Dx \psi)+ \eta \px^2\psi\\
&\qquad\quad +(G(\eta)-\Dx )\bigl( \eta \Dx \psi\bigr) -\eta (G(\eta)-\Dx)\Dx\psi\\
&\qquad\quad +J(\eta,\eta)(\B(\eta)-\Dx)\psi,
\end{align*}
where we used $\Dx^2=-\px^2$. 
By replacing $V(\eta)\psi$ by $V(\eta)\psi-\px\psi+\px\psi$, we conclude that $\mathcal{F}^3$ 
satisfies
\begin{align*}
\mathcal{F}^3
&=2(G(\eta)-\Dx )\bigl( \eta \Dx \psi\bigr) -2\eta (G(\eta)-\Dx)\Dx\psi+2(\px\eta)(V(\eta)-\px)\psi\\
&\quad +2J(\eta,\eta)(\B(\eta)-\Dx)\psi,
\end{align*}

The first three terms in the right-hand side above 
are estimated as $\mathcal{F}^2_a$ 
(except that we use Proposition~\ref{T56} for estimating 
products instead of using \e{n323c} for estimating paraproducts). 

To estimate $\double{J(\eta,\eta)(\B(\eta)-\Dx)\psi}{k,\mu}$, we first use 
\eqref{n337} to obtain that
\begin{multline*}
\double{J(\eta,\eta)(\B(\eta)-\Dx)\psi}{k,\mu} 
\le \Cr \triple{\eta}{s_0,0}\double{(\B(\eta)-\Dx)\psi}{k,\mu}\\
+\Cr \triple{(\B(\eta)-\Dx)\psi}{\mu+k-s_0+\np,\gpr}\double{\eta}{k,\mu+1},
\end{multline*}
The term $\double{(\B(\eta)-\Dx)\psi}{k,\mu}$ is estimated by means 
of \eqref{n368}. Now notice that $\gamma_0>3$ and 
$\mu+k-s_0+\np\le s-1-s_0+\np \le s_0 -3$ (also, up to replacing $\gamma_0$ by $\gamma_0+\delta$, $\delta\ll 1$, one can assume without loss of generality that $\gamma_0\not\in\xN$). 
So, we can apply \eqref{n367} to estimate $\triple{(\B(\eta)-\Dx)\psi}{\mu+k-s_0+\np,\gpr}$.

\step{4}{Analysis of $\mathcal{F}^i$ for $4\le i\le 6$.} 

By definition
$$
\mathcal{F}^4
=-\Dx \RBony\big(\B(\eta)\psi-\Dx\psi,Z\eta\big)-\px \RBony\big(Z\eta,V(\eta)\psi-\px\psi\big).
$$
So \eqref{n366} for $i=4$ follows from the estimate \eqref{n325b} and 
the estimates \eqref{n367} and \eqref{n368}.

Similarly, \eqref{n366} for $i=5$ follows from the estimate \eqref{n326b} 
and the estimates  \eqref{n367} and \eqref{n368}.

Finally, it remains to estimate $\mathcal{F}^6$. 
We estimate $\Dx T_{\Rzero{B}(\eta)\psi+2\Dx\psi}\eta$ 
and $\px (T_{\Rzero{V}(\eta)\psi+2\px\psi}\eta)$ separately. 
To fix matters we consider the first term only (the second term is estimated similarly). 
One has to take care of the fact 
that $\Rzero{B}(\eta)\psi$ involves one $Z$-derivative 
acting on~$\eta$. We thus use the sharp product estimate 
\eqref{n323c} with $m=\mu+k-s_0+2$ to obtain that
\begin{align*}
\double{\Dx T_{\Rzero{B}(\eta)\psi+2\Dx\psi}\eta}{k,\mu} 
&\le \double{T_{\Rzero{B}(\eta)\psi+2\Dx\psi}\eta}{k,\mu+1} \\
&\les \triple{\Rzero{B}(\eta)\psi+2\Dx\psi}{\mu+k-s_0+2,0}
\double{\eta}{k-1,\mu+2}\\
&\quad +\triple{\eta}{s_0,0}
\double{\Rzero{B}(\eta)\psi+2\Dx\psi}{k-1,0}\\
&\quad +\lA \Rzero{B}(\eta)\psi+2\Dx\psi\rA_{L^\infty}
\blA Z^k\eta\brA_{H^\mu}\\
&\quad +\indicator{\xR_+}(s_0-\mu-2)\lA \eta\rA_{\eC{\mu+1}}
\blA Z^k\bigl(\Rzero{B}(\eta)\psi+2\Dx\psi\bigr)\brA_{L^2}.
\end{align*}
It follows from the 
definition~\eqref{N302} of $\Rzero{B}(\eta)\psi$ and the definition~\e{231}Ê
of 
$\Rzero{G}(\eta)\psi$ that 
\begin{align*}
\Rzero{B}(\eta)\psi+2\Dx\psi&=I+II+III\\
I&=-2\big(G(\eta)-\Dx\big)\psi\\
II&=\frac{2}{1+(\eta'^2)}\Big( \bigl[ G(\eta),\eta\bigr]\B(\eta)\psi
- \eta' (V(\eta)\psi)\Big)\\
III&=-\frac{1}{1+\eta'^2}\big(\px (V(\eta)\psi)-\eta' \px (\B(\eta)\psi)\big)
Z\eta.
\end{align*}
All the terms in the right hand side are quadratic and can be estimated as above; 
let us mention that we do not need to use the fact that 
$\bigl[ G(\eta),\eta\bigr]\B(\eta)\psi$ is a commutator (it is sufficient 
to estimate $G(\eta)(\eta\B(\eta)\psi)$ and $\eta  G(\eta)
\B(\eta)\psi$ separately) and that 
\begin{alignat*}{3}
&  \triple{I}{\mu+k-s_0+2,0} 
\quad&&\text{is estimated by } &&\e{n367}\\
&\triple{II}{\mu+k-s_0+2,0},~\triple{III}{\mu+k-s_0+2,0}
\quad&&\text{are estimated by } &&\e{Z:2} \text{ and }\e{prod:eCZ}\\
&\double{I}{k-1,0}\quad&&\text{is estimated by } && \e{n368} \\
&\double{II}{k-1,0},~\double{III}{k-1,0},~\blA Z^k II\brA_{L^2}\quad&&\text{are estimated by } && \e{z:Q},~\e{n335b},~\e{Z:2} \\
&\lA I\rA_{L^\infty} \quad&&\text{is estimated by } && \e{n145}\\
&\lA II\rA_{L^\infty},~\lA III\rA_{L^\infty} \quad&&\text{is estimated by } && \e{211-1} \\
&\blA Z^k I\brA_{L^2} \quad&&\text{is estimated by } && \e{2311aa}\\
&\blA Z^k III\brA_{L^2} \quad&&\text{is estimated by } && 
\e{n335e} \text{ with }\zeta=Z\eta, \e{z:Q}, \e{Z:2}.
\end{alignat*}
Then \eqref{n366} for $i=6$ follows from arguments similar to the observations made above \e{n335e}. 
This completes the proof.
\end{proof}

\chapter{Energy estimates for the Z-field system}\label{S:24}

Combining the results obtained so far, we prove in this chapter the Sobolev estimates for the action of the $Z$-vector field on the 
solution we are looking for.

\section{Notations}\label{S:241}

We start by recalling or fixing some notations.

We fix real numbers $a$ and $\gamma$ with
$$
\gamma\not\in\mez\xN,\quad a\gg \gamma\gg 1.
$$
(In particular, we assume that $\gamma$ is large relatively to the fixed positive 
constants $\gii'$, $\nii$ given by Proposition~\ref{T62.5}).  
Given these two numbers, we fix three integers $s,s_0,s_1$ in $\xN$ such that
$$
s-a\ge s_1\ge s_0\ge \frac{s}{2}+\gamma.
$$
We also fix an integer $\rho$ larger than $s_0$. 
Our goal is to estimate the norm
\be\label{n403}
M_s^{(s_1)}(t)=\sum_{p=0}^{s_1} \Bigl( 
\blA Z^p \eta(t)\brA_{H^{s-p}}
+\blA \Dxmez Z^p \omega(t)\brA_{H^{s-p}}\Bigr),
\ee
assuming some control of the H\"older norms
$$
\blA \Dxmez \psi(t)\brA_{\eC{\gamma}}
+\lA \eta(t)\rA_{\eC{\gamma}}
$$
and
$$
N_\rho^{(s_0)}(t)=\sum_{p=0}^{s_0} \Bigl( 
\blA Z^p \eta(t)\brA_{\eC{\rho-p}}
+\blA \Dxmez Z^p \psi(t)\brA_{\eC{\rho-p}}\Bigr).
$$
We want to prove the following theorem.
\begin{theo}\label{T63}
There is a constant $B_2>0$ and for any constants $B_\infty>0$, 
$B_\infty'>0$, there is $\eps_0$ such that the following holds: 
Let $T>T_0$ be a number such that equation \eqref{121} 
with Cauchy data satisfying \eqref{126} has a solution satisfying the regularity properties of Proposition~\ref{ref:121} 
on $[T_0,T]\times \xR$ and such that 

$i)$ For any $t\in [T_0,T[$, and any $\eps\in ]0,\eps_0]$, 
\be\label{134bis}
\blA \Dxmez \psi(t)\brA_{\eC{\gamma}}
+\lA \eta(t)\rA_{\eC{\gamma}}\le B_\infty\eps t^{-\mez}.
\ee 

$ii)$ For any $t\in [T_0,T[$, any $\eps\in ]0,\eps_0]$
\be\label{135bis}
N_\rho^{(s_0)}(t)\le B_\infty \eps t^{-\mez+B_\infty' \eps^2}.
\ee
Then, there is an increasing sequence $(\delta_k)_{0\le k\le s_1}$ 
depending only on $B_\infty'$ and $\eps$ with $\delta_{s_1}<1/32$ such that for any $t$ in $[T_0,T[$, any 
$\eps$ in $]0,\eps_0]$, any $k\le s_1$,
\be\label{136bis}
M_s^{(k)}(t)\le \mez B_2\eps t^{\delta_k}.
\ee
\end{theo}
\begin{rema*}This is Theorem~\ref{ref:131} except that we 
replaced \eqref{134} by \eqref{134bis}, 
which we can freely do replacing $\gamma$ by $\gamma+\mez$. 
\end{rema*}
\begin{proof}[Proof of Theorem~\ref{T63}]
We fix an integer $\beta$ such that
\be\label{n403a}
\gii'-1\ge \beta\ge 4,
\ee
where $\gii'$ is a fixed large enough positive number given by Proposition~\ref{T62.5}. 
Since we assumed that $\gamma$ is large relatively to $\gii'$, we can assume that 
$\gamma-4\ge \beta$. Moreover, since $s-s_1\ge a\ge \gamma$, this yields 
that $\beta\le s-s_1$. 
Introduce the set
\be\label{n404}
\mathcal{P}=\bigl\{ (\alpha,n)\in \xN\times \xN\,;\, 0\le n\le s_1,~0\le \alpha\le s-n-\beta\bigr\}.
\ee
For any $(\alpha,n)$ in $\mathcal{P}$ we set
\be\label{n405}
Y_{(\alpha,n)}\defn 
\blA \px^\alpha Z^n \eta\brA_{H^\beta}+\blA  \Dxmez \px^\alpha Z^n \omega\brA_{H^\beta}
+\blA \Dxmez \px^\alpha Z^n  \psi\brA_{H^{\beta-\mez}}.
\ee
Since
$$
\sum_{\substack{0\le n\le k\\ 0\le \alpha\le s-n-\beta}}Y_{(\alpha,n)}
=\sum_{n=0}^{k}
\Bigl\{ \blA Z^{n}\eta\brA_{H^{s-n}}+\blA \Dxmez Z^{n}\omega\brA_{H^{s-n}}
+\blA \Dxmez Z^{n}\psi\brA_{H^{s-n-\mez}}\Bigr\}
$$
we have
\be\label{n410}
M_s^{(k)}\le \sum_{\substack{0\le n\le k\\ 0\le \alpha\le s-n-\beta}}Y_{(\alpha,n)}.
\ee

We shall proceed by induction. This requires to introduce a 
bijective map, denoted by $\Lambda$, from $\mathcal{P}$ to $\{0,1,\ldots,\# \mathcal{P}-1\}$. 
For $(\alpha,n)\in \mathcal{P}$, we set
$$
\Lambda(\alpha,n)=\sum_{p=0}^{n-1}(s+1-\beta-p)+\alpha,
$$
with the convention that $\sum_{p=0}^{-1} (s+1-\beta-p)=0$ so that $\Lambda(\alpha,0)=\alpha$. Then we define the following order on 
$\mathcal{P}$: 
$$
(\alpha',n')\prec (\alpha,n) \Leftrightarrow \Lambda(\alpha',n')< \Lambda(\alpha,n).
$$
So, there holds $(\alpha',n')\prec (\alpha,n)$ if and  only if either 
$n'<n$ or [$n'=n$ and $\alpha'<\alpha$].

Given an integer $K$ in $\{0,\ldots,\# \mathcal{P}-1\}$ we set 
$$
\mathcal{P}_K=\{ (\alpha,n)\in \xN\times \xN\,;\, \Lambda(\alpha,n)\le K\}.
$$
We also set $\mathcal{P}_{-1}=\emptyset$ and we introduce, for $K$ in $\{0,\ldots,\# \mathcal{P}\}$, 
\be\label{n412}
\mathcal{M}_{K}\defn \sum_{(\alpha',n')\in \mathcal{P}_{K-1}}Y_{(\alpha',n')},
\ee
where, by convention, $\mathcal{M}_0=0$.

We use the forthcoming Corollary~\ref{ref:242a} that will be established in the next section. Since assumption 
\e{135bis} shows that $N_\rho^{(s_0)}(t)$ stays uniformly bounded by $1$ is $\eps$ is small enough, inequality~\e{n514bb} shows that
\be\label{249a}
\ba
\mathcal{M}_{K+1}(t)\le C_K \Big[ &M_s^{(s_1)}(T_0)+(1+\mathcal{N}_K(t))\mathcal{M}_K(t)\\
&+\int_{T_0}^t \lA u(t',\cdot)\rA_{\eC{\gamma}}^2\mathcal{M}_{K+1}(t')\, dt'\\
&+\int_{T_0}^t \mathcal{N}_K(t')^2\mathcal{M}_K(t')\, dt'\Big]
\ea
\ee
for some constant $C_K$. In the definition \e{n513} of $\mathcal{N}_K$, we shall relate $\nu$ to the size $\eps$ of the Cauchy data by $\nu=\sqrt{\eps}$. 
We shall construct inductively an increasing sequence of constants $(B_{2,K})_K$ and of small exponents $(\widehat{\delta}_K)_K$ such that for 
any $t$ in $[T_0,T]$
\be\label{2410}
\mathcal{M}_K(t)\le \eps B_{2,K}t^{\widehat{\delta}_K}.
\ee
Since $\mathcal{M}_0\equiv 0$ by assumption, we may take $B_{2,0}=0$, $\widehat{\delta}_0=0$. 
Assume that the estimate has been obtained at rank $K$. This induction assumption, together with 
\e{135bis} implies that
\be\label{2410a}
\mathcal{N}_K(t)\le \eps \Big[ B_\infty+\frac{1}{\nu}\widetilde{B}_K(\nu)\Big]t^{-\mez+\gamma_K(\eps,\nu)}
\ee
where, if $\eps$ is small enough so that $B_\infty'\eps^2<\mez$, we may take
\be\label{2410b}
\ba
\widetilde{B}_K(\nu)&=B_\infty^{1-\nu}B_{2,K}^\nu\\
\gamma_K(\eps,\nu)&=\frac{\nu}{2}+(1-\nu)B_\infty'\eps^2+\nu \widehat{\delta}_K.
\ea
\ee
Our choice $\nu=\sqrt{\eps}$ implies in particular that, by \e{2410a}, $\mathcal{N}_K(t)$ 
is uniformly bounded so that \e{249a} may be rewritten, up to a modification of $C_K$, and making use of \e{134bis},
\begin{align*}
\mathcal{M}_{K+1}(t)\le C_K\Big[ &M_s^{(s_1)}(T_0)+\mathcal{M}_K(t)\\
&+\eps^2\int_{T_0}^t \mathcal{M}_{K+1}(t')\, \frac{dt'}{t'}\\
&+\int_{T_0}^t\mathcal{N}_K(t')^2\mathcal{M}_K(t')\, dt'\Big].
\end{align*}
Using Gronwall inequality for a non decreasing function $\alpha(\cdot)$ under the form
$$
y(t)\le \alpha(t)+\int_{T_0}^t \beta(\tau)y(\tau)\, d\tau ~\Rightarrow 
y(t)\le \alpha(t)\exp\left(\int_{T_0}^t \beta(\tau)\, d\tau\right)
$$
we get
\be\label{2410c}
\ba
\mathcal{M}_{K+1}(t)\le C_K\Big[ &M_s^{(s_1)}(T_0)+\sup_{T_0\le t'\le t}\mathcal{M}_K(t')\\
&+\int_{T_0}^t\mathcal{N}_K(t')^2\mathcal{M}_K(t')\, dt'\Big]t^{\eps^2 C_K}.
\ea
\ee
We may take a large enough constant $A$ so that 
$M_s^{(s_1)}(T_0)\le A\eps$ since the Cauchy data are $O(\eps)$. Using the induction assumption \e{2410}, we deduce from 
\e{2410c} and \e{2410a}
\be\label{2410d}
\ba
\mathcal{M}_{K+1}(t)\le \eps C_K t^{\eps^2 C_K}\Big[ A&+B_{2,K}t^{\widehat{\delta}_K}\\
&+B_{2,K}\eps^2 \frac{\big(B_\infty+\frac{1}{\nu}\widetilde{B}_K(\nu)\big)^2}{2\gamma_K(\eps,\nu)+\widehat{\delta}_K}
t^{2\gamma_K(\eps,\nu)+\widehat{\delta}_K}\Big].
\ea
\ee
Our choice $\nu=\sqrt{\eps}$ implies that $\gamma_K(\eps,\nu)$ given by \e{2410b} is bounded from below by $\mez \sqrt{\eps}$, so that the last coefficient in the above inequality 
is uniformly bounded. 

We find a new constant $B_{2,K+1}\ge B_{2,K}$ such that
\be\label{2410e}
\mathcal{M}_{K+1}(t)\le \eps B_{2,K+1}t^{\widehat{\delta}_{K+1}}
\ee
if we define 
$$
\widehat{\delta}_{K+1}=2\gamma_K(\eps,\nu)+\widehat{\delta}_K+\eps^2 C_K.
$$
The expression \e{2410b} of $\gamma_K$ shows that $\widehat{\delta}_{K+1}=O(\sqrt{\eps})$. 
We have obtained the bound \e{2410} at rank $K+1$. 

To finish the proof of Theorem~\ref{T63}, we are left with deducing from the above estimates inequality 
\e{136bis}. For $k\le s_1$, we define $K=\Lambda(s-k-\beta,k)$, $\delta_k=\widehat{\delta}_{K+1}$. Then 
by \e{n410} and \e{n412}, $M_s^{(k)}(t)\le \mathcal{M}_{K+1}(t)$. Estimate \e{136bis} thus follows from \e{2410e} if we take 
$B_2$ larger than $2B_{2,K+1}$ for any $K\le \#\mathcal{P}-1$. Notice that this constant is independent of $B_\infty$, $B_\infty'$ 
if $\eps$ is small enough: actually the only dependence of $B_{2,K+1}$ on $B_\infty$ could come only from the coefficient 
of $t^{2\gamma_K(\eps,\nu)+\widehat{\delta}_K}$ in the right hand side of \e{2410d}. But taking $\eps$ small enough in function of 
$B_\infty$, we may assume that this coefficient is smaller than a power of $B_{2,K}$. This concludes the proof of the theorem, 
assuming that Corollary \ref{ref:242a} holds. The rest of this chapter will be devoted to the proof of that corollary (actually of the proposition that 
will imply it) using a normal forms method. 
\end{proof}

\section{Normal form for the Z-systems}

From now on, we fix $K$ in $\{0,\ldots,\#\mathcal{P}-1\}$ and denote by 
$(\alpha,n)$ is the unique couple in $\mathcal{P}$ such that $\Lambda(\alpha,n)=K$. 
Then by the definition \e{n412}
\be\label{2411}
\mathcal{M}_{K+1}=Y_{(\alpha,n)}+\mathcal{M}_K.
\ee
We keep the notations introduced in section~\ref{S:321}. In particular, 
$$
\vu =\begin{pmatrix} \vu^1 \\ \vu^2\end{pmatrix}
=\begin{pmatrix} \eta \\ \Dxmez\psi\end{pmatrix},
\quad \vU=\begin{pmatrix} \vU^1 \\ \vU^2\end{pmatrix}
=\begin{pmatrix} \eta+T_{\sqrt{\ma}-1} \eta \\ \Dxmez \omega
\end{pmatrix},
$$
where $\ma$ is the Taylor coefficient given by \eqref{n190}.

As already mentioned in the remark made after the statement of Assumption~\ref{T27}, 
it follows from the assumptions of Theorem~\ref{T63} that, if $\eps$ is small enough, then
 the condition \eqref{1117} is satisfied uniformly in time. The other smallness conditions which appear in the previous chapters 
 are trivially satisfied under the only assumption \eqref{134bis}: namely, the smallness 
condition in Assumption~\ref{T28} 
which insures that the Taylor coefficient is bounded from below by $1/2$ and the smallness 
condition that $\etapetit$ is small enough which was used to justify the identity \eqref{231} as well as 
its corollaries. Thus we may apply the previous results. 

\begin{prop}\label{T65}
There exists a function $\Phi$ of the form
\be\label{n505}
\ZPhi{\alpha,n}\defn \px^\alpha Z^n\vU 
+\underset{0\le n_1+n_2\le n, ~0\le \alpha_1+\alpha_2\le \alpha}{\sum\sum\sum\sum}
E_{n_1 n_2 \alpha_1 \alpha_2} (\px^{\alpha_1} Z^{n_1} \vu)\px^{\alpha_2} Z^{n_2} \vU
\ee
where $E_{n_1 n_2 \alpha_1 \alpha_2}$ are bilinear operators, 
explicitly defined in the proof, such that the following properties hold

$i)$ $\ZPhi{\alpha,n}$ satisfies an equation of the form
\be\label{n510}
\partial_t\ZPhi{\alpha,n}+D\ZPhi{\alpha,n}+L(\vu)\ZPhi{\alpha,n}
+C(\vu)\ZPhi{\alpha,n}=\ZGamma{\alpha,n},
\ee
where $L(\vu)$ and $\ZGamma{\alpha,n}$ satisfy the following properties:

$\bullet$ $(v,f)\mapsto L(v)f$ is a bilinear mapping well defined for any 
$(v,f)$ in $\eC{2}(\xR)\times H^{\beta}(\xR)$ with values in $H^{\beta-1}(\xR)$. Moreover, for any $v$ in $\eC{2}(\xR)$, 
$L(v)$ satisfies $\RE\langle L(v)f,f\rangle_{H^\beta\times H^\beta}=0$ for any $f\in H^{\beta+1}(\xR)$. 

$\bullet$ $\Gamma$ is a cubic term satisfying  the following property: 
there exists a non decreasing function $C_K$ 
such that, for any $\nu\in ]0,1]$, 
\be\label{n512}
\blA \ZGamma{\alpha,n}\brA_{H^\beta}\le
C_0(\lA \vu\rA_{\eC{\gamma}})\lA \vu\rA_{\eC{\gamma}}^2Y_{(\alpha,n)}
+C_K(\avant)\avant^2 \Avant,
\ee
where
\be\label{n513}
\avant=N_\rho^{(s_0)}+\frac{1}{\nu}\bigl(N_\rho^{(s_0)}\bigr)^{1-\nu}\bigl(\Avant\bigr)^\nu.
\ee

$ii)$ 
There exists $\kappa_0>0$ and a non decreasing function $C_K(\cdot)$ such that, 
if $\lA \vu\rA_{\eC{\gamma}}\le \kappa_0$ then
\be\label{n514a}
\ba
Y_{(\alpha,n)}&\le 5\lA \Phi\rA_{H^\beta}+C_K\big(N_\rho^{(s_0)}\bigr)(1+\mathcal{N}_K) \mathcal{M}_K,\\
\lA \Phi\rA_{H^\beta}&\le 2 Y_{(\alpha,n)}+C_K\big(N_\rho^{(s_0)}\bigr) \mathcal{N}_K \mathcal{M}_K.
\ea
\ee
There exist $\kappa_0>0$ and $K_0>0$ 
such that if $N_\rho^{(s_0)}(T_0)\le \kappa_0$ then 
\be\label{n514b}
\lA \Phi\rA_{H^\beta}(T_0)\le K_0 M_s^{(s_1)}(T_0).
\ee
\end{prop}

Let us deduce from the above estimates the inequality that has been used in the previous section to prove Theorem~\ref{T63}. 
\begin{coro}\label{ref:242a}
Under the assumptions of the proposition, for any $K=0,\ldots,\#\mathcal{P}-1$ there is a non-decreasing 
function $C_K(\cdot)$ such that for any $\nu$ in $]0,1]$, any $t$ in $[T_0,T]$,
\be\label{n514bb}
\ba
\mathcal{M}_{K+1}(t)&\le 5K_0M_s^{(s_1)}(T_0)+C_K\big(N_\rho^{(s_0)}(t)\big)\big(1+\mathcal{N}_K(t)\big)\mathcal{M}_K(t)\\
&\quad +\int_{T_0}^t C_K\big(N_\rho^{(s_0)}(t')\big)\lA u(t',\cdot)\rA_{\eC{\gamma}}^2\Mr_{K+1}(t')\, dt'\\
&\quad +\int_{T_0}^t C_K\big( N_\rho^{(s_0)}(t')\big)\Nr_K(t')^2\Mr_{K}(t')\, dt'
\ea
\ee
(setting $\Nr_0\equiv 0$, $\Mr_0\equiv 0$ when $K= 0$).
\end{coro}
\begin{proof}
By assumption $\RE \langle D\Phi+L(\vu)\Phi,\Phi\rangle_{H^\beta\times H^\beta}=0$. 
Moreover, by Lemma~\ref{ref:A56}Ê
in 
Appendix~\ref{S:A4},
$$
\RE \langle C(\vu)\Phi,\Phi\rangle_{H^\beta\times H^\beta}\le C_0\big(\lA \vu\rA_{\eC{\gamma}}\big) \lA u\rA_{\eC{\gamma}}^2
\lA \Phi\rA_{H^\beta}^2.
$$
We may therefore compute $\frac{d}{dt}\lA \Phi(t,\cdot)\rA_{H^\beta}^2$ using \e{n510} and conclude, integrating the resulting expression from $T_0$ to $t$, 
that
\begin{align*}
\lA \Phi(t,\cdot)\rA_{H^\beta}^2&\le \lA \Phi(T_0,\cdot)\rA_{H^\beta}^2+\int_{T_0}^t C_0 \big(\lA u(t',\cdot)\rA_{\eC{\gamma}}\big) 
\lA u(t',\cdot)\rA_{\eC{\gamma}}^2\lA \Phi(t',\cdot)\rA_{H^\beta}^2\, dt'\\
&\quad +\int_{T_0}^t \lA \Gamma(t',\cdot)\rA_{H^\beta}\lA \Phi(t',\cdot)\rA_{H^\beta}\, dt'.
\end{align*}
We deduce from this inequality
\be\label{n500.5}
\ba
\lA \Phi(t,\cdot)\rA_{H^\beta}&\le \lA \Phi(T_0,\cdot)\rA_{H^\beta}+\int_{T_0}^t C_0 \big(\lA u(t',\cdot)\rA_{\eC{\gamma}}\big) 
\lA u(t',\cdot)\rA_{\eC{\gamma}}^2\lA \Phi(t',\cdot)\rA_{H^\beta}\, dt'\\
&\quad +\int_{T_0}^t \lA \Gamma(t',\cdot)\rA_{H^\beta}\, dt'.
\ea
\ee
By \e{n512}Êand the bound $Y_{(\alpha,n)}\le \Mr_{K+1}$ provided by \e{2411}, we get
\begin{align*}
\lA \Gamma(t',\cdot)\rA_{H^\beta}&\le C_0 \big(\lA u(t',\cdot)\rA_{\eC{\gamma}}\big) 
\lA u(t',\cdot)\rA_{\eC{\gamma}}^2 \Mr_{K+1}(t')\\
&\quad +C_K\big(\Nr_{K}(t')\big)\Nr_{K}(t')^2\Mr_{K}(t').
\end{align*}
If follows from the inequalities \e{n514a} and from \e{2411} that 
\begin{align*}
&\Mr_{K+1}(t)\le 5 \lA \Phi(t,\cdot)\rA_{H^\beta}+C_K\big( N_\rho^{(s_0)}(t)\big)(1+\Nr_K(t))\Mr_K(t),\\
&\lA \Phi(t',\cdot)\rA_{H^\beta}\le 2\Mr_{K+1}(t')+C_K\big( N_\rho^{(s_0)}(t')\big)\Nr_K(t')\Mr_K(t')
\end{align*}
for new values of $C_K(\cdot)$. We bound in the first inequality above $\lA \Phi(t,\cdot)\rA_{H^\beta}$ from 
\e{n500.5}, where we control in the right hand side $\lA \Phi(t',\cdot)\rA_{H^\beta}$ and $\lA \Gamma(t',\cdot)\rA_{H^\beta}$ using the estimates 
just obtained. We get
\begin{align*}
\mathcal{M}_{K+1}(t)&\le 5 \lA \Phi(T_0,\cdot)\rA_{H^\beta}+C_K\big(N_\rho^{(s_0)}(t)\big)\big(1+\mathcal{N}_K(t)\big)\mathcal{M}_K(t)\\
&\quad +\int_{T_0}^t C_K\big(\lA u(t',\cdot)\rA_{\eC{\gamma}}\big)\lA u(t',\cdot)\rA_{\eC{\gamma}}^2\Mr_{K+1}(t')\, dt'\\
&\quad +\int_{T_0}^t C_K\big( N_\rho^{(s_0)}(t')\big)\Nr_K(t')^2\Mr_{K}(t')\, dt'
\end{align*}
(using that $\lA u\rA_{\eC{\gamma}}$ may be estimated from $\Nr_K$, and changing again the value of the constants). 
Combining this and \e{n514b}, we get \e{n514bb}.
\end{proof}

We now have to prove Proposition~\ref{T65}. 
Let us describe the strategy of the proof. The proof is divided into four steps. We first write 
the equation for $\px^\alpha Z^n U$ under the form
\begin{equation}\label{n515}
\bigl(\partial_t+D
+Q(\vu) +S(\vu)+C(\vu)\bigr)\px^\alpha Z^n U 
=\ZmG{\alpha,n}+\ZmF{\alpha,n},
\end{equation}
where $\ZmG{\alpha,n}$ is a cubic term, $\ZmF{\alpha,n}$ is a quadratic term and 
where $Q(\vu)$, $S(\vu)$ and $C(\vu)$ are as defined in Section~\ref{S:321}. 
As a preparation for the next step, we rewrite this equation under the form
\begin{equation}\label{n516}
\bigl(\partial_t+D
+Q(\vu) +S^\sharp(\vu)+2S^\flat(\vu)+C(\vu)\bigr)\px^\alpha Z^n U 
=\ZmG{\alpha,n}' +\ZmF{\alpha,n}'',
\end{equation}
where, again, $\ZmG{\alpha,n}'$ is a cubic term, $\ZmF{\alpha,n}''$ is a quadratic term 
and where $S^\sharp(\vu)$ and $S^\flat(\vu)$ are as defined in \eqref{n235.1}, 
so that $S(\vu)=S^\sharp(\vu)+S^\flat(\vu)$. 
The main difference between the quadratic terms 
$\ZmF{\alpha,n}$ (which appears in \eqref{n515})Ê
and 
$\ZmF{\alpha,n}''$ (which appears in \eqref{n516}) is that we shall show in 
the second step that one can eliminate $\ZmF{\alpha,n}''$ by a bilinear normal form 
which produces cubic terms satisfying \eqref{n512}---whereas eliminating $\ZmF{\alpha,n}$ 
would produce a cubic term whose $L^2$-norm is estimated by
$$
C(\lA \vu\rA_{\eC{\gamma}})(\lA \vu\rA_{\eC{\gamma}}
+\lA \mathcal{H}\vu\rA_{\eC{\gamma}})^2\lA \px^\alpha Z^n\vU\rA_{L^2}
+C(\avant)\avant^2 \Avant.
$$
In the third step we follow the strategy already explained in \S\ref{S:3.2.2}. 
We shall prove that one can add a quadratic term in the equation which 
compensates for the most singular quadratic term. 
Eventually, in the fourth step we estimate various terms.

\begin{proof}The proof is divided into four steps. Let us mention that, for this proof, we write simply $C(\cdot)$ instead of $C_K(\cdot)$.

\step{1}{Equation for $\px^\alpha Z^n\vU$}

Using the notations of \S\ref{S:321} for 
the operators $Q(\vu)$, $S(\vu)$ and $C(\vu)$, we have
\begin{equation*}
\partial_t\vU+D \vU+Q(\vu)\vU+S(\vu)\vU
+C(\vu)\vU=G,
\end{equation*}
where $G=(G^1,G^2)$ is given by (see~\e{n212})
\be\label{n517.6}
\ba
G^1&=
(\id+T_\alpha)F(\eta)\psi - F_{\quadratique}(\eta)\psi+
T_{\partial_t\alpha-\px V+\mez \px^2\psi}\eta\\[0.5ex]
&\quad +\Bigl\{ -T_\alpha T_{\px V}+T_V T_{\px\alpha}\eta+\bigl[T_V,T_\alpha\bigr]-\mez T_{\Dx^\tdm \vu^2}T_\alpha\Bigr\}\eta,\\[0.5ex]
&\quad +\Dx \RBony(\Dx\psi,T_\alpha \eta)+\px \RBony(\px\psi,T_\alpha\eta),
\ea
\ee
and
\be\label{n517.8}
\ba
G^2&=\Dxmez \Bigl( \mez \RBony(\B,\B)-\mez\RBony(\Dx\psi,\Dx\omega)\Bigr)\\
&\quad -\Dxmez \Bigl( \mez \RBony(V,V)-\mez\RBony(\px\psi,\px\omega)\Bigr)\\
&\quad +\Dxmez \bigl((T_{V}  T_{\partial_x\eta}-T_{V \partial_x\eta})\B
+(T_{V \partial_x\B}-T_{V}  T_{\partial_x\B})\eta\bigr)\\
&\quad +\Dxmez T_V\RBony(\B,\partial_x\eta)-\Dxmez\RBony(\B,V\px\eta)\\
&\quad +\Dxmez (T_\alpha T_\alpha -T_{\alpha^2})\eta,
\ea
\ee
where we still denote by $\alpha$ the coefficient $\sqrt{\ma}-1$ where $\ma$ is the Taylor coefficient.

To compute the equations satisfied by $Z^nU$ we use two 
calculus results. Firstly,
\be\label{n518}
Z\partial_t =\partial_t Z -\partial_t,\quad ZD=DZ -D,
\ee
and secondly, given 
$A(v)=\Op^{\Bony}[v^1,A^1]+\Op^\Bony[v^2,A^2]$ for some matrix-valued symbol $A^1,A^2$ in some class 
$S^{m,\gamma}_\nu$ we have 
(see~\eqref{n226})
$$
Z A(v)f =A(v)Zf+A(Zv)f+A'(v)f
$$
where $A'(v)f=\Op^{\Bony}[v^1,A^{',1}]f+\Op^\Bony[v^2,A^{',2}]$ with 
$A^{',r}=-2\xi \cdot \nabla_{\xi}  A^{r}$ for $r=1,2$. Notice that $A^{',1},A^{',2}$ 
belong to $S^{m,\gamma}_\nu$ if 
$A^1,A^2$ belongs to $S^{m,\gamma}_\nu$. 

In particular it follows from \e{n261} that 
$$
ZQ(\vu)=Q(Z\vu)+Q(\vu)Z+Q'(\vu)Ê\quad\text{where } Q'(\vu)=\Op^\Bony[\vu,Q'],
\quad Q'\in S^{1,0}_{1/2}.
$$
Similarly, $ZS(\vu)=S(Z\vu)+S(\vu)Z+S'(\vu)$ where 
$S'(\vu)=\Op^\Bony[\vu,R']$ with $R'=-2\xi\cdot\nabla_\xi R$ 
where $R$ (resp.\ $R'$) is given by \e{n547} below with $\ell=0$ (resp.\ $\ell=1$). 

Consequently, by induction on $n\in\xN$, we have
\begin{equation}\label{n520}
\partial_t Z^{n}\vU +DZ^n\vU +Q(\vu)Z^n\vU +S(\vu)Z^n\vU
+C(\vu)Z^n\vU
=\ZG{n} +\ZF{n},
\end{equation}
where $\ZF{n}$ (resp.\ $\ZG{n}$) 
is a quadratic (resp.\ cubic) term defined by induction:
\begin{align}
\ZG{n}&\defn Z\ZG{n-1}+ \ZG{n-1}\notag\\
&\quad+C(\vu)Z^{n}\vU-ZC(\vu)Z^{n-1}\vU-C(\vu)Z^{n-1}\vU
,\notag\\[0.5ex]
\ZF{n}&\defn Z\ZF{n-1}+\ZF{n-1}-Q(Z\vu)Z^{n-1}\vU\label{n522}\\
&\quad -Q(\vu)Z^{n-1}\vU-Q'(\vu)Z^{n-1}\vU\notag\\
&\quad -S(Z\vu)Z^{n-1}\vU-S(\vu)Z^{n-1}\vU-S'(\vu)Z^{n-1}\vU,\notag
\end{align}
with, by definition, $\ZG{0}=G$ and $\ZF{0}=0$.

Observe that one can write $\ZF{n}$ under the form
$$
\ZF{n}=
\sum_{i\in I(n)}m(i)Q^{(n_{3})}(Z^{n_{1}} \vu) Z^{n_{2}}\vU
+\sum_{I(n)}m(i)S^{(n_{3})}(Z^{n_{1}} \vu) Z^{n_{2}}\vU
$$
where $m(i)\in \xN$ and where we used the following notations : 
$$
I(n)=\left\{\, i=(n_{1},n_{2},n_{3})\in \xN^3\,;\, n_{1}+n_{2}+n_{3}\le n \text{ and }n_{2}<n\,\right\},
$$
and $Q^{(n_{3})}$ and $S^{(n_{3})}$ are defined by 
\begin{align*}
Q^{(n_{3})}(v)&=\Op^{\Bony}[v^1,Q^{(n_{3}),1}]+\Op^\Bony[v^2,Q^{(n_{3}),2}],\\ 
S^{(n_{3})}(v)&=\Op^{\Bony}[v^1,R^{(n_{3}),1}]+\Op^\Bony[v^2,R^{(n_{3}),2}],
\end{align*} 
where for $A=Q$ or $A=R$, and for $k=1,2$, $A^{(n_{3}),k}$ is defined by induction: 
$$
A^{(0),k}=A^k,\quad 
A^{(a+1),k}=-2\xi \cdot \nabla_{\xi}  A^{(a),k}.
$$

Applying $\px^\alpha$ to \eqref{n520}Ê
we conclude that
\begin{equation}\label{n523}
\bigl(\partial_t+D
+Q(\vu) +S(\vu)+C(\vu)\bigr)\px^\alpha Z^n U
=\ZmG{\alpha,n} +\ZmF{\alpha,n},
\end{equation}
where $\ZmF{\alpha,n}$ (resp.\ $\ZmG{\alpha,n}$) 
is a quadratic (resp.\ cubic) term defined by
\be\label{n524}
\begin{aligned}
\ZmG{\alpha,n}&\defn \px^\alpha \ZG{n}+
C(\vu)\px^\alpha Z^{n}\vU-\px^\alpha C(\vu)Z^n \vU,\\
\ZmF{\alpha,n}&\defn \px^\alpha \ZF{n}+Q(\vu)\px^\alpha Z^n\vU
-\px^\alpha Q(\vu)Z^n\vU\\
&\quad+S(\vu)\px^\alpha Z^n \vU-\px^\alpha S(\vu)Z^n \vU.
\end{aligned}
\ee

Observe that one can write $\ZmF{\alpha,n}$ under the form
\be\label{n525}
\begin{aligned}
\ZmF{\alpha,n}&=
\sum_{j\in \Jalpha}m(j)Q^{(n_{3})}(\px^{\alpha_1} Z^{n_{1}} \vu) \px^{\alpha_2} Z^{n_{2}}\vU\\
&\quad +\sum_{j\in \Jalpha}m(j)S^{(n_{3})}(\px^{\alpha_1} Z^{n_{1}} \vu) \px^{\alpha_2} Z^{n_{2}}\vU
\end{aligned}
\ee
where $m(j)\in \xN$ and $\Jalpha$ is the set of those $(\alpha_1,\alpha_2,n_{1},n_{2},n_{3})\in \xN^5$ such that
\be\label{n526}
\alpha_1+\alpha_2=\alpha,~ n_{1}+n_{2}+n_{3}\le n,~\alpha_2+n_{2}<\alpha+n.
\ee

There are two terms in the right hand side  of \eqref{n525} which involve $\px^\alpha Z^n \vu$. 
Namely, when  $(\alpha_1,\alpha_2,n_{1},n_{2},n_{3})=(\alpha,0,n,0,0)$ 
we have \begin{align*}
&Q^{(n_{3})}(\px^{\alpha_1} Z^{n_{1}} \vu) \px^{\alpha_2} Z^{n_{2}}\vU=
Q(\px^\alpha Z^n\vu)\vU,\\
&S^{(n_{3})}(\px^{\alpha_1} Z^{n_{1}} \vu) \px^{\alpha_2} Z^{n_{2}}\vU=S(\px^\alpha Z^n\vu)\vU.
\end{align*} 
We shall see that one cannot eliminate these quadratic terms by the same method. 
So we need to transform further the equation. 

Notice that if $j=(\alpha_1,\alpha_2,n_{1},n_{2},n_{3})=(\alpha,0,n,0,0)$ 
then the coefficient $m(j)$ in 
\eqref{n525} is equal to $-1$. Thus we may rewrite 
the equation~\eqref{n523} as
\begin{equation}\label{n530}
\bigl(\partial_t+D
+Q(\vu) +S(\vu)+C(\vu)\bigr)\px^\alpha Z^n U 
+S(\px^\alpha Z^n\vu)\vU
=\ZmG{\alpha,n} +\ZmF{\alpha,n}',
\end{equation}
where
\begin{align*}
\ZmF{\alpha,n}'&=
\sum_{j\in \Jalpha}m(j)Q^{(n_{3})}(\px^{\alpha_1} Z^{n_{1}} \vu) \px^{\alpha_2} Z^{n_{2}}\vU\\
&\quad
+\sum_{j\in \Jpr}m(j)S^{(n_{3})}(\px^{\alpha_1} Z^{n_{1}} \vu) \px^{\alpha_2} Z^{n_{2}}\vU
\end{align*}
where $m(j)\in \xN$ and
\be\label{n532}
\Jpr=\left\{\,(\alpha_1,\alpha_2,n_{1},n_{2},n_{3})\in \Jalpha \,;\, 
\alpha_1+n_1<\alpha+n\,\right\}.
\ee

Eventually, we split $S(\vu)$ as $S(\vu)=S^\sharp(\vu)+S^\flat(\vu)$ where these operators 
are defined by \eqref{n235.1}. Since $S^\flat(v)f=S^\flat(f)v$, we have
\begin{align*}
S(\vu)\px^\alpha Z^n U +S(\px^\alpha Z^n\vu)\vU
&=S^\sharp(\vu)\px^\alpha Z^n U+S^\sharp(\px^\alpha Z^n\vu)\vU\\
&\quad+S^\flat(\vu)\px^\alpha Z^n\vU+S^\flat(\vU)\px^\alpha Z^n\vu.
\end{align*}
Now we write the second and last terms in the right hand side above as
\begin{align*}
S^\sharp(\px^\alpha Z^n\vu)\vU&=S^\sharp(\px^\alpha Z^n\vU)\vu
+\bigl(S^\sharp(\px^\alpha Z^n\vu)\vU-S^\sharp(\px^\alpha Z^n\vU)\vu\bigr),\\
S^\flat(\vU)\px^\alpha Z^n\vu&=S^\flat(\vu)\px^\alpha Z^n\vU
+\bigl(S^\flat(\vU)\px^\alpha Z^n\vu-S^\flat(\vu)\px^\alpha Z^n\vU\bigr),
\end{align*}
to obtain that
\begin{equation}\label{n534}
\bigl(\partial_t+D
+Q(\vu) +S^\sharp(\vu)+2S^\flat(\vu)+C(\vu)\bigr)\px^\alpha Z^n U 
=\ZmG{\alpha,n}' +\ZmF{\alpha,n}''
\end{equation}
where
\be\label{n537}
\begin{aligned}
\ZmG{\alpha,n}'&=\ZmG{\alpha,n}-
\bigl(S^\sharp(\px^\alpha Z^n\vu)\vU-S^\sharp(\px^\alpha Z^n\vU)\vu\bigr)\\
&\quad-\bigl(S^\flat(\vU)\px^\alpha Z^n\vu-S^\flat(\vu)\px^\alpha Z^n\vU\bigr),\\
\ZmF{\alpha,n}''&=\ZmF{\alpha,n}'-S^\sharp(\px^\alpha Z^n\vU)\vu.
\end{aligned}
\ee
Hereafter, we use the notation
\be\label{n538}
N(u)=Q(\vu) +S^\sharp(\vu)+2S^\flat(\vu)+C(\vu).
\ee
Then \e{n534} reads
\begin{equation}\label{n535}
\bigl(\partial_t+D+N(\vu)\bigr)\px^\alpha Z^n U 
=\ZmG{\alpha,n}' +\ZmF{\alpha,n}''.
\end{equation}
For further references, let us prove that, for any $\mu\in\xR$,
\be\label{n539}
\lA N(\vu)\rA_{\Fl{H^{\mu+1}}{H^\mu}}\le C(\lA u\rA_{\eC{\gamma}})\lA u\rA_{\eC{\gamma}}.
\ee
Indeed, directly from the definition \e{n210} (resp.\ \e{n209}) 
for $Q(\vu)$ (resp.\ $C(\vu)$), and using 
the rule~\eqref{esti:quant1}, the estimates 
\eqref{n189} for $\lA V\rA_{\eC{0}}$ and \eqref{n203} for $\lA \alpha\rA_{\eC{0}}$, we check that
$$
\lA (Q(\vu)+C(\vu)) w\rA_{H^{\mu}}\le K \lA \vu\rA_{\eC{\gamma}}\lA w\rA_{H^{\mu+1}},
$$
provided that $\gamma$ is large enough. 
On the other hand, directly from the definition~\e{n235.1} of $S^\sharp(\vu)$ and $S^\flat(\vu)$, 
it follows from \e{Bony3}Ê
that, for any $\rho\not\in\mez\xN$ and any 
$\mu\in\xR$ such that $\mu+\rho>1$,
\be\label{n539c}
\blA (S^\sharp(\vu)+2S^\flat(\vu)) w\brA_{H^{\mu+\rho}}
\le K \lA \vu\rA_{\eC{\rho}}\lA w\rA_{H^{\mu+\tdm}}.
\ee
This proves \e{n539}. Similarly, for any positive real number $\rho$ with $\rho\not\in\mez\xN$, we have
\be\label{n539a}
\lA N(\vu)\rA_{\Fl{\eC{\rho+1}}{\eC{\rho}}}\le C(\lA u\rA_{\eC{\gamma}})\lA u\rA_{\eC{\gamma}}.
\ee
\step{2}{First normal form}

We next seek a nonlinear change of unknown which removes 
the quadratic term $\ZmF{\alpha,n}''$ in the right-hand side of~\eqref{n535}. 
To do so, we shall prove that for any $\ell\in \xN$ 
there exist bilinear transforms 
$(v,f)\mapsto \EA{\ell}(v)f$ and 
$(v,f)\mapsto \ER{\ell}(v)f$ such that
\begin{align*}
D \EA{\ell}(v)f&=\EA{\ell}(Dv)f+\EA{\ell}(v)Df+Q^{(\ell)}(v)f,\\
D \ER{\ell}(v)f&=\ER{\ell}(Dv)f+\ER{\ell}(v)Df+S^{(\ell)}(v)f.
\end{align*}

We begin by studying the operators $Q^{(\ell)}(v)f$ and $S^{(\ell)}(v)f$. 
For further references, we state the following lemma.
\begin{lemm}
Let $\ell \in \xN$. For all $\mu\in \xR$ and all $\rho\in [4,+\infty\por$ 
there exists a constant $K$ such that
\begin{alignat}{2}
&\blA Q^{(\ell)}(v)f\brA_{H^{\mu-1}}&&\le K \lA v\rA_{\eC{2}}\lA f\rA_{H^{\mu}},\label{n540}\\
&\blA Q^{(\ell)}(v)f\brA_{H^{\rho-2}}&&\le K \lA v\rA_{L^{2}}\lA f\rA_{\eC{\rho}},\label{n541}\\
&\blA S^{(\ell)}(v)f\brA_{H^{\mu+2}}&&\le K \lA v\rA_{\eC{4}}\lA f\rA_{H^{\mu}},\label{n542}\\
&\blA S^{(\ell)}(v)f\brA_{H^{\rho-2}}&&\le K \lA v\rA_{L^{2}}\lA f\rA_{\eC{\rho}},\label{n543}
\end{alignat}
whenever these terms are well-defined.
\end{lemm}
\begin{proof}
For $\ell=0$ we have $Q^{(0)}(v)f=Q(v)f$ and the estimates~\eqref{n540}--\eqref{n541} 
follow from the definition of $Q(v)f$ (see~\eqref{n210}), the usual estimates for 
paraproducts (see~\eqref{esti:quant0} and \eqref{esti:Tba}) and the H\"older estimates~\eqref{esti:Dxpsiz0} and
 \eqref{esti:Dxmez-Crho} proved in Appendix~\ref{S:A.3}. 

For $\ell>0$, introduce $\displaystyle{\theta^{(\ell)}=\Bigl(1+\frac{2}{3}\xi\cdot\nabla_{\xi}\Bigr)^\ell \theta}$ 
where~$\theta$ is given by Definition~\ref{defi:theta}. We claim that 
$Q^{(\ell)}(v)=\Op^\Bony\big[v^1,Q^{(\ell),1}\big]+\Op^\Bony\big[v^2,Q^{(\ell),2}\big]$ with
\begin{equation}\label{n545}
\begin{aligned}
Q^{(\ell),1}&=(-3)^\ell\mez \la\xip\ra \theta^{(\ell)}(\xip,\xii)\begin{pmatrix} 0 & \la\xii\ra^{\mez} \\
-\la \xip+\xii\ra^{\mez} & 0 \end{pmatrix},\\
Q^{(\ell),2}&=(-3)^\ell \xip \la \xip\ra^{-\mez}\theta^{(\ell)}(\xip,\xii)
\begin{pmatrix} -\xii -\mez\xip & 0 \\ 0 & -\la \xip+\xii\ra^\mez \xii\la \xii\ra^{-\mez}
\end{pmatrix}.
\end{aligned}
\end{equation}
For $\ell=0$ this is true by definition of 
the symbols $Q^1$ and $Q^2$ as defined in \eqref{n261}. For $\ell>0$ 
this is proved by induction, since 
$Q^{(\ell+1),k}=-2\xi\cdot\nabla_\xi Q^{(\ell),k}$ for $k=1,2$. 
It follows from~\eqref{n545} that $Q^{(\ell)}(v)$ is a 
paradifferential operator of exactly the same form as $Q(v)$, except that 
the cut-off function $\theta$ is replaced with $\theta^{(\ell)}$. Since $\theta^{(\ell)}$ 
is an admissible cut-off function 
(satisfying similar assumptions to those imposed on $\theta$, see Remark~\ref{rema:cutoff}), then 
$Q^{(\ell)}(v)f$ satisfies the same estimates as $Q(v)f$ does. 
This proves \eqref{n540}--\eqref{n541}.

The estimates~\eqref{n542}--\eqref{n543}Êare proved by using similar arguments. 
Indeed, it follows from \e{n243}, \e{n235.1}, and \e{n258} that
$$
S^{(\ell)}(v)=\Op^\Bony\left[v^2,\begin{pmatrix} 
m^{(\ell),2}_{11} & 0 \\ 0 & m^{(\ell),2}_{22}\end{pmatrix}\right]
$$
where
\begin{equation}\label{n547}
\begin{aligned}
m^{(\ell),2}_{11}
&=(-3)^\ell\zeta^{(\ell)}(\xip,\xii) 
\la \xip\ra^{-\mez}\bigl( \la \xip+\xii\ra \la \xip\ra-(\xip+\xii)\xip\bigr)
\quad \\
m^{(\ell),2}_{22}
&=(-3)^\ell\zeta^{(\ell)}(\xip,\xii\Bigl( -\mez \la \xip+\xii\ra^\mez
\bigl( \la \xip\ra \la \xii\ra+\xip\xii\bigr) \la \xip\ra^{-\mez} \la \xii\ra^{-\mez}\Bigr),
\end{aligned}
\end{equation}
with $\displaystyle{\zeta^{(\ell)}=\Bigl(1+\frac{2}{3}\xi\cdot\nabla_{\xi}\Bigr)^\ell \zeta}$ 
where $\zeta(\xip,\xii)=1-\theta(\xip,\xii)-\theta(\xii,\xip)$. Notice that 
$\zeta^{(\ell)}(\xip,\xii)=1-\theta^{(\ell)}(\xip,\xii)-\theta^{(\ell)}(\xii,\xip)$. Since 
$\theta^{(\ell)}$ is an admissible cut-off function, we are in position to apply 
the usual estimates for the remainders (see~\eqref{Bony3}). 
\end{proof}

Next we notice that, for any $\ell\in \xN$, 
it follows from Proposition~\ref{T36} 
 and the structure of $Q^{(\ell)}$ given in \eqref{n545} 
that there exists a pair of 
matrix-valued symbols $P_\ell=(P_\ell^1,P_\ell^2)\in S^{1,0}_{0}\times S^{1,0}_0$ 
such that, for all $v=(v^1,v^2)\in \eC{\rho}\cap L^2(\xR)$ (with $\rho$ large enough)
\be\label{n549}
\EA{\ell}(v)=\Op^\Bony[v^1,P_\ell^1]+\Op^\Bony[v^2,P_\ell^2]
\ee
satisfies
\begin{equation}\label{n551}
D \EA{\ell}(v)=\EA{\ell}(Dv)+\EA{\ell}(v)D+Q^{(\ell)}(v).
\end{equation}
We gather the properties satisfied by $\EA{\ell}(v)$ in the next lemma.

\begin{lemm}\label{T67}
Let $\ell\in\xN$.

$i)$ Let $\mu$ be a given real number. There exists $K>0$ such that, 
for any scalar function $w\in \eC{2}(\xR)$, any $v=(v^1,v^2)\in \eC{5}\cap L^2(\xR)$ 
and any $f=(f^1,f^2)\in H^\mu(\xR)$, any $\nu\in ]0,1[$,
\be\label{n552}
\lA \left[ T_w I_2 , \EA{\ell}(v)\right] f\rA_{H^\mu} 
\le K \lA w\rA_{\eC{1}} \Bigl\{\lA v\rA_{\eC{5}}+\frac{1}{\nu}\lA v\rA_{\eC{5}}^{1-\nu}\lA v\rA_{L^2}^\nu\Bigr\} 
\lA f\rA_{H^\mu},
\ee
where $I_2=\left(\begin{smallmatrix} 1 & 0 \\ 0 & 1\end{smallmatrix}\right)$.

$ii)$ Let $\mu$ be a given real number. 
There exists $K>0$ such that, 
for any $v=(v^1,v^2)$ in $\eC{4}\cap L^2(\xR)$ 
and any $f=(f^1,f^2)$ in $H^\mu(\xR)$, any $\nu \in ]0,1[$, 
\be\label{n553}
\blA \EA{\ell}(v)f\brA_{H^{\mu-1}}
\le K \Bigl\{\lA v\rA_{\eC{4}}+\frac{1}{\nu}\lA v\rA_{\eC{4}}^{1-\nu}\lA v\rA_{L^2}^\nu\Bigr\}\lA f\rA_{H^{\mu}}.
\ee
$iii)$ Let $\rho\in [3/2,+\infty\por$. 
There exists $K>0$ such that, 
for any $v=(v^1,v^2)$ in $L^2(\xR)$ 
and any $f=(f^1,f^2)$ in $\eC{\rho}(\xR)$,
\be\label{n554}
\lA \EA{\ell}(v)f\rA_{H^{\rho-\tdm}}\le K \lA f\rA_{\eC{\rho}}\lA v\rA_{L^2}.
\ee
\end{lemm}
\begin{proof}
We recall that $\EA{\ell}(v)$ is given by \eqref{n549} where 
$P_\ell^1$ and $P_\ell^2$ belong to $S^{1,0}_{0}$. 
It follows from Lemma~\ref{T34} 
that $\EA{\ell}(v)$ is a paradifferential operator of order $1$, modulo a smoothing operator, whose symbol 
has semi-norms estimated by means of statement $i)$ in Lemma~\ref{T33}. 
The assertions in statements $i)$ and $ii)$ thus follow from~Theorem~\ref{theo:sc0}. 
We shall give another proof of these results 
which will also prove~statement $iii)$. 

Let us introduce a class of symbols. 
Given $(j_{1},j_{2},j_{3})\in \xR^3$, one denotes by $\Sell{j_{1}}{j_{2}}{j_{3}}$ 
the class of scalar 
symbols $m(\xip,\xii)$, $C^\infty$ for 
$(\xip,\xii)$ in $(\xR\setminus \{0\})\times\xR$ which are linear combinations 
of symbols of the form
$$
p_{1}(\xip)p_{2}(\xii)p_{3}(\xip+\xii)\theta^{(\ell)}(\xip,\xii)
$$
with $\displaystyle{\theta^{(\ell)}=\Bigl(1+\frac{2}{3}\xi\cdot\nabla_{\xi}\Bigr)^\ell \theta}$ 
where~$\theta$ is given by Definition~\ref{defi:theta}, and $p_{r}(\lambda \xi)= \lambda^{j_{r}}p_{r}(\xi)$ for 
all $r\in\{1,2,3\}$, all $\lambda >0$ and all $\xi\neq 0$.

Given two functions $a=a(x)$ and $b=b(x)$, one denotes by $T^{(\ell)}_a b$ the paraproduct given 
by replacing the cut-off function $\theta$ by $\theta^{(\ell)}$ in 
the definition~\eqref{eq.para} of $T_a b$. 
If $m\in \Sell{j_{1}}{j_{2}}{j_{3}}$ then
$$
\frac{1}{(2\pi)^2}\int e^{ix(\xip+\xii)} \widehat{a}(\xip)m(\xip,\xii)\widehat{b}(\xii) \,d\xip \, d\xii=
p_3(D_x)T^{(\ell)}_{p_1(D_x)a}p_2(D_x)b.
$$
By virtue of the support properties of $\theta^{(\ell)}$, we have
$$
p_3(D_x)T^{(\ell)}_{p_1(D_x)v}p_2(D_x)f=\tilde{p_3}(D_x)T^{(\ell)}_{p_1(D_x)v}\tilde{p_2}(D_x)f,
$$
where $\tilde{p_2}(\xi)$ and $\tilde{p_3}(\xi)$ vanish on a neighborhood of $\xi=0$ and are equal 
to $p_2(\xi)$ and $p_3(\xi)$, respectively, for $\la\xi\ra$ large enough. Consequently, 
it follows from~\eqref{esti:Tba} that, to prove statement 
$iii)$ of the lemma, it is sufficient to prove that the matrices  
$P^1_\ell=(a^{\ell,1}_{ij})_{1\le i,j\le 2}$ and $P^2_\ell=(a^{\ell,2}_{ij})_{1\le i,j\le 2}$ 
are such that, for all $(i,j,k)\in \{1,2\}^3$, the coefficient $a^{\ell,k}_{ij}$ belongs to some class 
$\Sell{j_1}{j_2}{j_3}$ with $j_1\ge 0$ and $j_2+j_3\le 1$ (the values of $j_1,j_2,j_3$ 
might depend on $(i,j,k)$). 

Consider the symbols $Q^{(\ell),1}$ and $Q^{(\ell),2}$ as defined in \eqref{n545}. They are of the form
\begin{align*}
Q^{(\ell),1}
=\begin{pmatrix} 0 & m^{\ell,1}_{12} \\ m^{\ell,1}_{21} & 0 \end{pmatrix},
\quad
Q^{(\ell),2}=\begin{pmatrix}Êm^{\ell,2}_{11} & 0 \\ 0 & m^{\ell,2}_{22}\end{pmatrix}
\end{align*}
where, for any $(i,j,k)\in\{1,2\}^3$, 
\be\label{n555}
m^{\ell,k}_{ij}\in S_{\ell}(j_{1},j_{2},j_{3}) \quad \text{with }
j_{1}\ge \mez,\quad j_{2}\ge 0,\quad j_{3}\ge 0,\quad 
j_{1}+j_{2}+j_{3}=\frac{3}{2}.
\ee
Below we write simply $m^k_{ij}$ (resp.\ $a^k_{ij}$) instead of 
$m^{\ell,k}_{ij}$ 
(resp.Ê
$a^{\ell,k}_{ij}$). The symbols $a^k_{ij}$ are determined explicitly in the proof of Proposition~\ref{T36}: We have $a^{2}_{11}=a^{1}_{12}=a^{1}_{21}=a^{2}_{22}=0$ and
\begin{equation}\label{n557}
\begin{aligned}
a^2_{21}&=\frac{\delta}{D}\bigl(|\xip+\xii|^\mez m^2_{11}
-\la\xip\ra^\mez m^1_{21}+\la\xii\ra^\mez m^2_{22}\bigr)\\
&\quad +\frac{2}{D}\la\xip\ra^\mez\la\xii\ra^\mez
\bigl(|\xip+\xii|^\mez m^1_{12}+\la\xip\ra^\mez m^2_{22}-\la\xii\ra^\mez m^1_{21}\bigr),\\
a^1_{22}&=\frac{\delta}{D}\bigl(|\xip+\xii|^\mez m^1_{12}
+\la\xip\ra^\mez m^2_{22}-\la\xii\ra^\mez m^1_{21}\bigr)\\
&\quad+\frac{2}{D}\la\xip\ra^\mez\la\xii\ra^\mez
\bigl(|\xip+\xii|^\mez m^2_{11}-\la\xip\ra^\mez m^1_{21}+\la\xii\ra^\mez m^2_{22}\bigr),\\
a^2_{12}&=-\frac{1}{|\xip+\xii|^\mez}
\bigl( |\xip|^\mez a^1_{22}+|\xii|^\mez a^2_{21}+m^2_{22}\bigr),\\
a^1_{11}&=\frac{1}{|\xip+\xii|^\mez}
\bigl( |\xip|^\mez a^2_{21}+|\xii|^\mez a^1_{22}-m^1_{21}\bigr).
\end{aligned}
\end{equation}
Recall also that
\begin{alignat*}{4}
&\xip\xii>0\quad &&\Rightarrow\quad &&\delta=0
\quad&&\text{and}\quad D=-4\la \xip\ra \la \xii\ra ,\\
&\xip\xii<0 \text{ and }\la\xip\ra<\la\xii\ra 
\quad 
&&\Rightarrow \quad&&\delta=-2\la\xip\ra\quad
&&\text{and}\quad D
=-4\la \xip\ra \la \xip+\xii\ra.
\end{alignat*}
Denote by $\indicator{A}$ the indicator function of the set 
$A$. Then
\begin{align*}
\frac{\delta}{D}\theta^{(\ell)}
&=\indicator{\{\xip\xii<0\}} \frac{1}{2\la \xip+\xii\ra}\theta^{(\ell)},\\
\frac{2\la\xip\ra^\mez\la\xii\ra^\mez}{D}\theta^{(\ell)}
&=
-\mez \frac{1}{\la\xip\ra^\mez}\Bigl( 
\indicator{\{\xip\xii>0\}} \frac{1}{\la\xii\ra^\mez}
+\indicator{\{\xip\xii<0\}}\frac{\la\xii\ra^\mez}{\la\xip+\xii\ra}\Bigr)\theta^{(\ell)}.
\end{align*}
Since
\begin{align*}
\indicator{\{\xip\xii>0\}}=\mez +\mez \sign (\xip)\sign(\xii),\quad
\indicator{\{\xip\xii<0\}}
=\mez -\mez \sign (\xip)\sign(\xii),
\end{align*}
and since $\sign$ is homogeneous of order $0$, it follows that
$$
\frac{\delta}{D}\theta^{(\ell)}\in \Sell{0}{0}{-1},\quad 
\frac{\la\xip\ra^\mez\la\xii\ra^\mez}{D}\theta^{(\ell)}
\in \Sell{-1/2}{-1/2}{0}+\Sell{-1/2}{1/2}{-1}.
$$
Consequently, it follows from~\e{n555} and \eqref{n557} that $a^k_{ij}$ is a sum of terms which belong to 
classes $\Sell{j_1}{j_2}{j_3}$ with $j_{1}+j_{2}+j_{3}=3/2-1/2=1$ and $j_{1}\ge 0$. 
This concludes the proof.
\end{proof}

\begin{lemm}\label{T68}
For any $\ell\in\xN$ there exist two matrix-valued symbols $R_\ell^1,R_\ell^2$ in $SR^{3/2}_{0,0}$
such that, for all $v\in \eC{4}\cap L^2(\xR)$
$$
\ER{\ell}(v)=\Op^\Bony[v^1,R_\ell^1]+\Op^\Bony[v^2,R_\ell^2]
$$
satisfies
\begin{equation}\label{n560}
D \ER{\ell}(v)=\ER{\ell}(Dv)+\ER{\ell}(v)D+S^{(\ell)}(v),
\end{equation}
and such that the following estimates hold.

$i)$ For all $(\mu,\rho)\in \xR\times \xR_+$ such that $\mu+\rho>1$ and $\rho\not\in\mez\xN$, 
there exists a positive constant $K$ such that, 
for any $v=(v^1,v^2)\in \eC{\rho}\cap L^2(\xR)$ 
and any $f=(f^1,f^2)\in H^\mu(\xR)$,
\be\label{n561}
\blA \ER{\ell}(v)f\brA_{H^{\mu+\rho-1}}
\le K (\lA v\rA_{\eC{\rho}}+\lA \mathcal{H}v\rA_{\eC{\rho}})\lA f\rA_{H^{\mu}}.
\ee

$ii)$ For all $(\mu,\rho)\in \xR\times \xR_+$ such that $\mu+\rho>1$ and $\rho\not\in\mez\xN$, 
there exists a positive constant $K$ such that, 
for any $v=(v^1,v^2)\in H^\mu(\xR)$ and any $f=(f^1,f^2)\in \eC{\rho}(\xR)\cap L^2(\xR)$,
\be\label{n562}
\blA \ER{\ell}(v)f\brA_{H^{\mu+\rho-1}}\le K (\lA f\rA_{\eC{\rho}}
+\lA \mathcal{H}f\rA_{\eC{\rho}})\lA v\rA_{H^{\mu}}.
\ee
\end{lemm}
\begin{rema*}We shall use later that (see~\e{lemm:Tpm2-b}) for 
any $\rho\not\in\xN$, 
there exists $K>0$ and for any $\nu>0$, any $v\in \eC{\rho}\cap L^2$,
\be\label{n563}
\lA \mathcal{H}v\rA_{\eC{\rho}}\le K \Big[\lA v\rA_{\eC{\rho}}+\frac{1}{\nu}
\lA v\rA_{\eC{\rho}}^{1-\nu}\lA v\rA_{L^2}^\nu\Bigr].
\ee
\end{rema*}
\begin{proof}
For $\ell=0$ we have $S^{(0)}(v)f=S(v)f$ and hence 
$\ER{0}(v)f=E^\sharp(v)f+E^\flat(v)f$ with the operators 
given by Proposition~\ref{T37}. The asserted estimates thus follow from Proposition~\ref{T37}.

For $\ell>0$, 
we have seen in \eqref{n547} that the symbols of $S^{(\ell)}(v)$ are obtained from the symbols of $S(v)$ 
by replacing $\theta$ with $\theta^{(\ell)}$ (and multiplying by $(-3)^\ell$). 
Therefore, $R_\ell^1$ and $R_\ell^2$ are deduced from 
$R_0^1\defn R^{\sharp,1}+R^{\flat,1}$ and $R_0^2\defn R^{\sharp,2}+R^{\flat,2}$ 
(which are given by~\eqref{n248} and \eqref{n258.3}) by the same modifications. 
Since $\theta^{(\ell)}$ is an admissible cut-off function (see Remark~\ref{rema:cutoff}), 
this shows that $\ER{\ell}(v)f$ satisfies the same estimates as $\ER{0}(v)f$ does. 
\end{proof}

We shall use also the operator $E^\sharp(v)$ introduced in Proposition~\ref{T37}. 
satisfying 
\be\label{n564}
E^\sharp(Dv)+E^\sharp(v)D-D E^\sharp(v)=S^\sharp(v)
\ee
and
\begin{equation}\label{n565}
\blA E^\sharp(v)f \brA_{H^{\mu+\rho-1}}\le 
K\lA f\rA_{\eC{\rho}}\lA v\rA_{H^\mu}.
\end{equation}
for any $(\mu,\rho)\in \xR\times \xR_+$ such that $\mu+\rho>1$ and $\rho\not\in\mez\xN$.

Then \eqref{n551}, \eqref{n560}, and \eqref{n564} imply that
\begin{align*}
&(\partial_t+D)\left(
\sum_{\Jalpha}m(j)\EA{n_{3}}(\px^{\alpha_1} Z^{n_{1}} \vu)\px^{\alpha_2} Z^{n_{2}}\vU \right)
\\
&\quad +(\partial_t+D)\left(\sum_{\Jpr}m(j)\ER{n_{3}}(\px^{\alpha_1} Z^{n_{1}} \vu)\px^{\alpha_2} Z^{n_{2}}\vU
\right)\\
&\quad +(\partial_t+D)\left(
E^\sharp(\px^\alpha Z^n\vu)\vU\right)=
\ZmF{\alpha,n}''+\ZmR{\alpha,n}
\end{align*}
where $\ZmF{\alpha,n}''$ is as given by \eqref{n537} and
\be\label{n569}
\ba
\ZmR{\alpha,n}\defn 
&\sum_{\Jalpha}m(j)
\EA{n_{3}}\bigl((\partial_t +D)\px^{\alpha_1} Z^{n_{1}}\vu\bigr)\px^{\alpha_2} Z^{n_{2}}\vU\\
&\quad +\sum_{\Jalpha}m(j) \EA{n_{3}}(\px^{\alpha_1} Z^{n_{1}}\vu)
\bigl((\partial_t+D)\px^{\alpha_2} Z^{n_{2}}\vU\bigr)
\\
&\quad+\sum_{\Jpr}m(j)
\ER{n_{3}}\bigl((\partial_t +D)\px^{\alpha_1} Z^{n_{1}}\vu\bigr)\px^{\alpha_2} Z^{n_{2}}\vU
\\
&\quad+\sum_{\Jpr}m(j)\ER{n_{3}}(\px^{\alpha_1} Z^{n_{1}}\vu)\bigl((\partial_t+D)\px^{\alpha_2} Z^{n_{2}}\vU\bigr)\\
&\quad+E^\sharp((\partial_t+D)\px^\alpha Z^n\vu)\vU
+E^\sharp(\px^\alpha Z^n\vu)(\partial_t\vU+D\vU).
\ea
\ee

This implies that
\be\label{n567}
\ba
\ZWPhi{\alpha,n}&\defn \px^\alpha Z^n\vU 
-\sum_{\Jalpha}m(j)\EA{n_{3}}(\px^{\alpha_1} Z^{n_{1}} \vu)\px^{\alpha_2} Z^{n_{2}}\vU \\
&\quad-\sum_{\Jpr}m(j)\ER{n_{3}}(\px^{\alpha_1} Z^{n_{1}} \vu)\px^{\alpha_2} Z^{n_{2}}\vU\\
&\quad-E^\sharp(\px^\alpha Z^n\vu)\vU
\ea
\ee
satisfies
\begin{equation*}
\partial_t\ZWPhi{\alpha,n}+D\ZWPhi{\alpha,n}=
\partial_t \px^\alpha Z^n\vU+D \px^\alpha Z^n\vU - \ZmF{\alpha,n}''-\ZmR{\alpha,n}.
\end{equation*}

Therefore, \eqref{n535} implies that
\begin{equation}\label{n570}
\bigl(\partial_t  +D+N(\vu)\bigr)\ZWPhi{\alpha,n}
=\ZWGamma{\alpha,n},
\end{equation}
where $N(\vu)$ is given by \e{n538} and 
\begin{equation}\label{n571}
\ZWGamma{\alpha,n}=\ZmG{\alpha,n}'-\ZmR{\alpha,n}+N(\vu)\bigl(\ZWPhi{\alpha,n}-\px^\alpha Z^{n}\vU\bigr).
\end{equation}

We shall estimate $\ZWGamma{\alpha,n}$ in the last step of the proof. 
This is the most technical part of the proof.

\step{3}{Second normal form}

We start with the following result, which is analogous to Lemma~\ref{T40}.
\begin{lemm}\label{T69}
There exist 
$A_0^1,A_0^2$ in $S^{0,1/2}_{1}$ such that, for 
all $v\in \eC{3}\cap L^2(\xR)$ the operator 
$E_{A_0}(v)=\Op^\Bony[v^1,A_0^1]+\Op^\Bony[v^2,A_0^2]$ satisfies
\be\label{n572}
D E_{A_0}(v)=E_{A_0}(Dv)+E_{A_0}(v)D+B(v),
\ee
where the operator $B(v)$ satisfies $B(v)=B(v)^*$ and 
\be\label{n573}
\RE \langle Q(v)f-B(v)f,f\rangle_{H^\beta\times H^\beta}=0
\ee
for any 
$f\in H^{\beta+1}(\xR)^2$, and such that the following properties hold.

$i)$ Let $\mu$ be a given real number. There exists $K>0$ such that, 
for any scalar function $w\in \eC{2}(\xR)$, any $v=(v^1,v^2)\in \eC{3}\cap L^2(\xR)$ 
and any $f=(f^1,f^2)\in H^\mu(\xR)$,
\be\label{n574}
\lA \left[ T_w I_2, E_{A_0}(v)\right] f\rA_{H^{\mu+1}} 
\le K \lA w\rA_{\eC{1}}\lA v\rA_{\eC{3}}\lA f\rA_{H^\mu},
\ee
where $I_2=\left(\begin{smallmatrix} 1 & 0 \\ 0 & 1\end{smallmatrix}\right)$.

$ii)$ Let $\mu$ be a given real number. 
There exists $K>0$ such that, 
for any $v=(v^1,v^2)\in \eC{3}\cap L^2(\xR)$ 
and any $f=(f^1,f^2)\in H^\mu(\xR)$,
\be\label{n575}
\blA E_{A_0}(v)f\brA_{H^{\mu}}
\le K \lA v\rA_{\eC{3}}\lA f\rA_{H^{\mu}}.
\ee
\end{lemm}
\begin{proof}
This is Lemma~\ref{T40} applied with $s$ replaced by $\beta$.
\end{proof}

Consider now the operator $E_\beta^\sharp(v)$ and 
$E_\beta^\flat(v)$ as given by Proposition~\ref{T37-add}. 
It follows from this proposition that 
\be\label{n577}
\ba
&E_\beta^\sharp(Dv)+E_\beta^\sharp(v)D-D E_\beta^\sharp(v)
=\mathfrak{S}^\sharp(v),\\
&E_\beta^\flat(Dv)+E_\beta^\flat(v)D-D E_\beta^\flat(v)=\mathfrak{S}^\flat(v),
\ea
\ee
where $\mathfrak{S}^\sharp$ and $\mathfrak{S}^\flat$ are such that
\begin{align}
&\RE \langle S^\sharp(v) f-\mathfrak{S}^\sharp(v)f,f\rangle_{H^\beta\times H^\beta}=0,
\label{add-3b}
\\
&\RE \langle S^\flat(v) f-\mathfrak{S}^\flat(v)f,f\rangle_{H^\beta\times H^\beta}=0,\label{add-4b}
\end{align}
for any 
$f\in H^{\beta}(\xR)^2$, and satisfies
\be\label{n579}
\ba
&\blA \mathfrak{S}^\sharp(v)\brA_{\Fl{H^{\mu}}{H^{\mu+\rho-1}}}
\le K \lA v \rA_{\eC{\rho}},\\
&\blA \mathfrak{S}^\flat(v)\brA_{\Fl{H^{\mu}}{H^{\mu+\rho-1}}}
\le K \lA v \rA_{\eC{\rho}}.
\ea
\ee
Moreover, for all $(\mu,\rho)\in \xR\times \xR_+$ such that $\mu+\rho>1$ and $\rho\not\in\mez\xN$, 
there exists a positive constant $K$ such that
\begin{equation}\label{n242-b}
\ba
&\blA E_\beta^\sharp(v)f \brA_{H^{\mu+\rho-1}}\le 
K\lA v\rA_{\eC{\rho}}\lA f\rA_{H^\mu},\\
&\blA E_\beta^\flat(v)f \brA_{H^{\mu+\rho-1}}\le 
K\lA v\rA_{\eC{\rho}}\lA f\rA_{H^\mu}.
\ea
\ee

Set 
$$
E(v)=E_{A_0}(v)-E_\beta^\sharp(v)-2 E_\beta^\flat(v).
$$
Then \e{n572} and \e{n577} imply that
\be\label{n580}
D E(v)-E(Dv)-E(v)D=B(v) +\mathfrak{S}^\sharp(v)+2\mathfrak{S}^\flat(v).
\ee
Moreover \e{n575} and \e{n242-b} imply that
\be\label{n581}
\lA E(v)\rA_{\Fl{H^{\mu}}{H^\mu}}\le K \lA v\rA_{\eC{3}}.
\ee

Now set
\be\label{Change}
\ZPhi{\alpha,n}=\ZWPhi{\alpha,n}+E(\vu)\px^\alpha Z^n\vU.
\ee
It follows from~\eqref{n580} that
\be\label{n582}
\begin{aligned}
(\partial_t+D)\ZPhi{\alpha,n}
&=(\partial_t+D)\ZWPhi{\alpha,n}\\
&\quad +E(\partial_t \vu+D\vu)\px^\alpha Z^n \vU+E(\vu)(\partial_t +D)\px^\alpha Z^n\vU\\
&\quad +\bigl( B(\vu)+\mathfrak{S}^\sharp(\vu)+2 \mathfrak{S}^\flat(\vu)\bigr)\px^\alpha Z^n\vU.
\end{aligned}
\ee
Recall that $\ZWPhi{\alpha,n}$ satisfies
$$
(\partial_t  +D)\ZWPhi{\alpha,n}=-N(\vu)\ZWPhi{\alpha,n}+\ZWGamma{\alpha,n}.
$$
Now write $\ZWPhi{\alpha,n}=\ZPhi{\alpha,n}-E(\vu)\px^\alpha Z^n\vU$ in the right hand side 
of the above identity and set the result into \eqref{n582}, to obtain that
\be\label{n584}
\begin{aligned}
(\partial_t+D)\ZPhi{\alpha,n}
&=-N(\vu)\ZPhi{\alpha,n}+\ZWGamma{\alpha,n}\\
&\quad+N(\vu)E(\vu)\px^\alpha Z^n\vU 
\\[0.5ex]
&\quad +E(\partial_t \vu+D\vu)\px^\alpha Z^n \vU+E(\vu)(\partial_t +D)\px^\alpha Z^n\vU\\
&\quad +\bigl(B(\vu) +\mathfrak{S}^\sharp(\vu)+2\mathfrak{S}^\flat(\vu)\bigr)\px^\alpha Z^n\vU.
\end{aligned}
\ee
Eventually we use \eqref{n535} to substitute $(\partial_t +D)\px^\alpha Z^n\vU$, 
which appears in the fifth term of the right hand side 
of \eqref{n584}, by
$$
(\partial_t +D)\px^\alpha Z^n\vU=
-N(\vu)\px^\alpha Z^n U
+\ZmG{\alpha,n}' +\ZmF{\alpha,n}'',
$$
and we write $\px^\alpha Z^n\vU=\ZPhi{\alpha,n}+\bigl(\px^\alpha Z^n\vU-\ZPhi{\alpha,n}\bigr)$ 
in the last term of the right hand side of \eqref{n584}. By so doing it is found that
\be\label{n590}
\partial_t\ZPhi{\alpha,n}+D\ZPhi{\alpha,n}+L(\vu)\ZPhi{\alpha,n}
+C(\vu)\ZPhi{\alpha,n}=\ZGamma{\alpha,n},
\ee
where
\be\label{n591}
L(\vu)\defn Q(\vu) +S^\sharp(\vu)+2S^\flat(\vu)
-\bigl( B(\vu) +\mathfrak{S}^\sharp(\vu)+2\mathfrak{S}^\flat(\vu)\bigr),
\ee
and where 
\be\label{n595}
\ZGamma{\alpha,n}=\ZWGamma{\alpha,n}+(1)+(2)+(3)+(4)
\ee
with
\begin{align*}
(1)&=N(\vu)E(\vu)\px^\alpha Z^n\vU -E(\vu)N(\vu)\px^\alpha Z^n U,\\[0.5ex]
(2)&=E(\partial_t \vu+D\vu)\px^\alpha Z^n \vU,\\[0.5ex]
(3)&=E(\vu)\ZmG{\alpha,n}' +E(\vu)\ZmF{\alpha,n}'',\\
(4)&=\bigl(B(\vu) +\mathfrak{S}^\sharp(\vu)+2\mathfrak{S}^\flat(\vu)\bigr)
\bigl(\px^\alpha Z^n\vU-\ZPhi{\alpha,n}\bigr).
\end{align*}

It follows from \e{n573}, \e{add-3b}, and \e{add-4b} that the 
operator $L(v)$ defined by \e{n591} satisfies
$\RE \langle L(v)f,f\rangle_{H^\beta\times H^\beta}=0$ for any 
$f$ in $H^{\beta+1}(\xR)$. Consequently, 
to complete the proof of the proposition, it remains only 
to prove the estimates~\e{n512} and \e{n514a}--\e{n514b}. 

\step{4}{Proof of the estimates~\e{n512} and \e{n514a}--\e{n514b}}

We begin by estimating the term $(1)$ which appears in \eqref{n595}. 
\begin{lemm}\label{T70}
There holds
$$
\lA (1)\rA_{H^\beta}\le C(\lA \vu\rA_{\eC{\gamma}})\lA \vu\rA_{\eC{\gamma}}^2\lA \px^\alpha Z^n\vU\rA_{H^\beta}.
$$
\end{lemm}
\begin{rema*}We shall later estimate $\lA \px^\alpha Z^n\vU\rA_{H^\beta}$ in terms 
of $Y_{(\alpha,n)}$ and $\Avant$.
\end{rema*}
\begin{proof}
This is proved by means of the arguments used in 
the proof of Proposition~\ref{T42}. For the sake of clarity we recall the proof. 

Recall from~\eqref{n581} that
$\lA E(\vu)\rA_{\Fl{H^\beta}{H^\beta}}\le C\lA \vu\rA_{\eC{\gamma}}$. Also, directly 
from the definition \eqref{n235.1} of $S^\sharp(\vu)$ and $S^\flat(\vu)$ we have
$$
\blA S^\sharp(\vu)\brA_{\Fl{H^\beta}{H^\beta}}+\blA S^\flat(\vu)\brA_{\Fl{H^\beta}{H^\beta}}
\le C\lA \vu\rA_{\eC{\gamma}}.
$$
Therefore
$$
\blA \bigl(S^\sharp(\vu)+2S^\flat(\vu)\bigr)E(\vu)\brA_{\Fl{H^\beta}{H^\beta}}
\le C\lA \vu\rA_{\eC{\gamma}}^2,
$$
and similarly
$$
\blA E(\vu)\bigl(S^\sharp(\vu)+2S^\flat(\vu)\bigr)\brA_{\Fl{H^\beta}{H^\beta}}
\le C\lA \vu\rA_{\eC{\gamma}}^2.
$$
It remains to estimate the operator norm of the 
commutator $[ A(\vu),E(\vu)]$ where we recall that $A(\vu)=Q(\vu)+C(\vu)$ 
where $Q(\vu)$ (resp.\ $C(\vu)$) is given by \eqref{n210} (resp.\ \e{n209}). 
We claim that
\begin{equation}\label{n600}
\lA [A(\vu),E(\vu)]\rA_{\Fl{H^\beta}{H^\beta}}\le C \lA \vu\rA_{\eC{\gamma}}^2.
\end{equation}
By definition $E(\vu)=E_{A_0}(\vu)+E_R(\vu)$ with $E_R(\vu)=-E_\beta^\sharp(v)-2 E_\beta^\flat(v)$. 
To prove~\eqref{n600}, we first 
observe that,
\begin{equation*}
\lA A(\vu)\rA_{\Fl{H^\beta}{H^{\beta-1}}}\le C\lA \vu\rA_{\eC{\gamma}},\quad 
\lA  E_R(\vu)\rA_{\Fl{H^\beta}{H^{\beta+1}}}\les \lA \vu\rA_{\eC{\gamma}},
\end{equation*}
where $C$ depends only on $\lA \vu\rA_{\eC{\gamma}}$. This implies that
$\lA  E_R(\vu)A(\vu)\rA_{\Fl{H^\beta}{H^\beta}}
\le C \lA \vu\rA_{\eC{\gamma}}^2$. 
Similarly, one has $\lA A(\vu) E_R(\vu) \rA_{\Fl{H^\beta}{H^\beta}}
\le C \lA \vu\rA_{\eC{\gamma}}^2$. This obviously implies \eqref{n600}. 
Thus it remains only to prove that
\begin{equation}\label{n602}
\lA [A(\vu),E_{A_0}(\vu)]\rA_{\Fl{H^\beta}{H^\beta}}\le C \lA \vu\rA_{\eC{\gamma}}^2.
\end{equation}
This we now prove by using the commutator estimate~\eqref{n574} together 
with the following remark. Introduce 
$$
\widetilde{A}(\Ur)=A(\Ur)-T_{V}\px -T_\alpha D.
$$
Directly from the definition of $A(\Ur)$ (recalling again that $A(\vu)=Q(\vu)+C(\vu)$ 
where $Q(\vu)$ (resp.\ $C(\vu)$) is given by \eqref{n210} (resp.\ \e{n209})), one can check 
that $\widetilde{A}(\Ur)$ is of order $0$ and satisfies
\be\label{n603}
\blA \widetilde{A}(\vu)\brA_{\Fl{H^\beta}{H^{\beta}}}\le C \lA \vu\rA_{\eC{\gamma}},
\ee
for some constant $C$ depending only on $\lA \vu\rA_{\eC{\gamma}}$. 
By combining this estimate with \eqref{n575} we get 
$$
\blA E_{A_0}(\vu)\widetilde{A}(\vu)\brA_{\Fl{H^\beta}{H^\beta}}
+\blA \widetilde{A}(\vu)E_{A_0}(\vu) \brA_{\Fl{H^\beta}{H^\beta}}\le C \lA \vu\rA_{\eC{\gamma}}^2,
$$
which obviously implies that 
$\blA \bigl[\widetilde{A}(\vu),E_{A_0}(\vu)\bigr]\brA_{\Fl{H^\beta}{H^\beta}}
\le C \lA \vu\rA_{\eC{\gamma}}^2$. 
So to prove \eqref{n602} it remains only 
to estimate the commutators of $E_{A_0}(\vu)$ with 
$T_V\px$ and $T_\alpha D$. 

Since~$T_V\px=T_{V (i\xi)}$ is a paradifferential operator with a scalar symbol and since 
the $C^1$-norm of $V$ is estimated by $C\lA \vu\rA_{\eC{\gamma}}$ 
for some constant $C$ depending only on $\lA \vu\rA_{\eC{\gamma}}$ 
(see~\eqref{n189}), 
it follows from statement $i)$ in Lemma~\ref{T69} that
$$
\blA \bigl[ T_V \px,  E_{A_0}(\vu)\bigr]\brA_{\Fl{H^\beta}{H^\beta}}\le C \lA \vu\rA_{\eC{\gamma}}^2,
$$
for some constant $C$ depending only on $\lA \vu\rA_{\eC{\gamma}}$.
To estimate~$\bigl[T_\alpha D,E_{A_0}(\vu)\bigr]$, 
use instead the equation~\eqref{n262} 
satisfied by~$E_{A_0}$ to obtain:
$$
T_\alpha D E_{A_0}(\vu)\vU=T_\alpha\Big(E_{A_0}(\vu)D\vU+E_{A_0}(D\vu)\vU
+B(\vu)\vU\Bigr).
$$
Notice that 
\be\label{n607}
\lA B(\vu)\rA_{\Fl{H^\beta}{H^\beta}}\le C \lA \vu\rA_{\eC{\gamma}}.
\ee
Indeed, 
$B(\vu)=\Op^{\Bony}[\vu^1,B^1]+\Op^\Bony[\vu^2,B^2]$ where 
$B^1$ and $B^2$ are given by \eqref{n261.1} and \eqref{n261.2} with $s$ 
replaced by $\beta$; so assertion $(ii)$ in Lemma~\ref{T33}, 
Lemma~\ref{T34} and \eqref{esti:quant1} imply the wanted estimate. Also, \eqref{n575} implies that 
$\lA E_{A_0}(D\vu)\rA_{\Fl{H^\beta}{H^\beta}}\le C \lA \vu\rA_{\eC{\gamma}}$. 
Consequently, 
since $\lA \alpha\rA_{\eC{1}}\le C\lA \vu\rA_{\eC{\gamma}}$ (see~\eqref{n203}) 
we have $\lA T_\alpha\rA_{\Fl{H^\beta}{H^\beta}}\le C \lA \vu\rA_{\eC{\gamma}}$ and hence 
$$
\lA T_\alpha E_{A_0}(D\vu)\rA_{\Fl{H^\beta}{H^\beta}}
+\lA T_\alpha B(\vu)\rA_{\Fl{H^\beta}{H^\beta}}\le C \lA \vu\rA_{\eC{\gamma}}^2,
$$
for some constant $C$ depending only on $\lA \vu\rA_{\eC{\gamma}}$. Moreover, 
since~$\alpha$ is a scalar function, it follows from the above mentioned estimate 
$\lA \alpha\rA_{\eC{1}}\le C\lA \vu\rA_{\eC{\gamma}}$
and statement $i)$ in Lemma~\ref{T69} that
$$
\lA \left[ T_\alpha, E_{A_0}(\vu)\right]D\rA_{\Fl{H^\beta}{H^\beta}}\le C \lA \vu\rA_{\eC{\gamma}}^2,
$$
for some constant $C$ depending only on $\lA \vu\rA_{\eC{\gamma}}$. This proves 
~\eqref{n602} and hence completes the proof of the lemma.
\end{proof}

\begin{lemm}\label{T70.2}
There holds
\begin{align*}
&\lA (2)\rA_{H^\beta}\le C(\lA \vu\rA_{\eC{\gamma}})
\lA \vu\rA_{\eC{\gamma}}^2\lA \px^\alpha Z^n\vU\rA_{H^\beta},\\
&\lA (3)\rA_{H^\beta}\le C(\lA \vu\rA_{\eC{\gamma}})
\lA \vu\rA_{\eC{\gamma}}\Bigl\{\blA \ZmG{\alpha,n}'\brA_{H^\beta}
+\blA \ZmF{\alpha,n}''\brA_{H^\beta}\Bigr\}.
\end{align*}
\end{lemm}
\begin{proof}
This follows from the estimates \eqref{n581} 
and \eqref{n269.1}.
\end{proof}

\begin{lemm}\label{T70.4}
$i)$ For any $(\alpha',n')$ such that $\alpha'+n'\le s_0$, there holds
$$
\blA \px^{\alpha'}Z^{n'}\bigl(\sqrt{a}-1\bigr)\brA_{\eC{1}}
\le C\bigl(N_\rho^{(s_0)}\bigr)N_\rho^{(s_0)}.
$$

$ii)$ If $(\alpha',n')\prec (\alpha,n)$ then 
$$
\blA \px^{\alpha'}Z^{n'} \bigl(\sqrt{a}-1\bigr)\brA_{L^2}\le C\bigl(N_\rho^{(s_0)}\bigr)
\Avant.
$$
$iii)$ There holds
$$
\blA \px^\alpha Z^n (\sqrt{\ma}-1) \brA_{L^2}\le C(\lA u\rA_{\eC{\gamma}})
Y_{(\alpha,n)} +C\bigl(N_\rho^{(s_0)}\bigr)\Avant.
$$
\end{lemm}
\begin{rema*}
Here we use the assumption $\beta>2$.
\end{rema*}
\begin{proof}
Recall that the Taylor coefficient $\ma$ 
can be written under the form (see~\eqref{formule:a} in Appendix~\ref{S:A.4}):
\begin{equation*}
\ma=\frac{1}{1+(\px\eta)^2}
\left(1+V\px \B - \B \px V-\mez G(\eta)V^2 -\mez G(\eta)\B^2-G(\eta)\eta\right),
\end{equation*}
where we used the abbreviated notations
$\B=\B(\eta)\psi$ and $V=V(\eta)\psi$. The assertion in statement $i)$, 
which is 
equivalent to saying that $\triple{\sqrt{a}-1}{s_0,1}$ is estimated by 
$C\bigl(N_\rho^{(s_0)}\bigr)N_\rho^{(s_0)}$,  then immediately follows 
from the estimates~\e{prod:eCZ}, \e{F:eCZ} and from Proposition~\ref{T52}. The assertions in 
statements $ii)$ and $iii)$ follow from the product rule~\e{n335c}, Proposition~\ref{T52} and 
Proposition~\ref{T56}.
\end{proof}

Below we freely use the following lemma. 
\begin{lemm}\label{T71}
Recall that we fixed $(\alpha,n)$ and $K$ such that $\Lambda(\alpha,n)=K$ and recall that $\mathcal{M}_K$ is defined by \e{n412}. There holds
\begin{align}
&\lA \px^\alpha Z^n \vu\rA_{H^{\beta-\mez}}\les 
Y_{(\alpha,n)} +\Avant,\label{n610}\\
&\lA \px^\alpha Z^n \vU\rA_{H^{\beta}}\le C(\lA u\rA_{\eC{\gamma}})
Y_{(\alpha,n)} +C\bigl(N_\rho^{(s_0)}\bigr)\Avant.\label{n610-b}
\end{align}
If $(\alpha',n')\prec (\alpha,n)$ then 
\begin{align}
&\blA \px^{\alpha'} Z^{n'} \vu\brA_{H^{\beta-\mez}}\les 
\Avant,\label{n611}\\
&\blA \px^{\alpha'} Z^{n'} \vU\brA_{H^{\beta}}\le C\bigl(N_\rho^{(s_0)}\bigr)
\Avant.\label{n612}
\end{align}
For any $(\alpha',n')$ such that $\alpha'+n'\le s_0$, there holds
\be\label{n615}
\blA \px^{\alpha'} Z^{n'}\vU\brA_{\eC{\beta+3}}\le C\bigl(N_\rho^{(s_0)}\bigr)
N_\rho^{(s_0)}.
\ee
\end{lemm}
\begin{proof}
The estimates~\e{n610} and \e{n611} follow directly from 
the definitions of $Y_{(\alpha,n)}$ and $\Avant$, and the fact that 
$\bigl[Z,\Dxmez\bigr]=-\Dxmez$. 

For further references, we shall prove \e{n612} and the following estimate
\be\label{n610-c}
\ba
\lA \px^\alpha Z^n \vU-\px^\alpha Z^n \begin{pmatrix} \eta \\ \Dxmez \omega\end{pmatrix}
\rA_{H^\beta}& \le C(\lA u\rA_{\eC{\gamma}})\lA u\rA_{\eC{\gamma}}
Y_{(\alpha,n)}\\
&\quad +C\bigl(N_\rho^{(s_0)}\bigr)N_\rho^{(s_0)}\Avant,
\ea
\ee
which immediately implies \e{n610-b}. 
We shall see that the estimates~\e{n610-c} and~\e{n612} follow from the definition of $\vU$. Indeed, 
$$
\vU=\begin{pmatrix} \eta \\ \Dxmez \omega\end{pmatrix}
+\begin{pmatrix} T_{\sqrt{a}-1}\eta\\ 0\end{pmatrix}.
$$
So, to prove \e{n610-c} it is sufficient to estimate the $H^\beta$-norm of 
$\px^{\alpha} Z^{n} \bigl(T_{\sqrt{a}-1}\eta\bigr)$. 
To do so, we write 
$\blA  \px^{\alpha} Z^{n} \bigl(T_{\sqrt{a}-1}\eta\bigr)\brA_{H^\beta}
\le \bdouble{T_{\sqrt{a}-1}\eta}{n,\alpha+\beta}$ and use 
the estimate~\e{n369.5} 
applied with $(K,\nu,m,b)$ replaced by $(n,\alpha+\beta,s_0,\gamma)$, which gives 
(bounding all the indicator functions by~$1$)
$$
\ba
\double{T_{\sqrt{a}-1}\eta}{n,\alpha+\beta} &\les 
\triple{\sqrt{a}-1}{s_0,0}\double{\eta}{n-1,\alpha+\beta+1}
+\lA \sqrt{a}-1\rA_{L^\infty}
\blA Z^n \eta\brA_{H^{\alpha+\beta}}\\
&\quad+\triple{\eta}{n+\alpha+\beta-s_0+1,0}
\double{\sqrt{a}-1}{n-1,0}\\
&\quad +
\lA \eta\rA_{\eC{\gamma}}
\blA Z^n(\sqrt{a}-1)\brA_{L^2}\\
&\quad +\lA \eta\rA_{\eC{\alpha+\beta-s_0+1}}\blA Z^n(\sqrt{a}-1)\brA_{L^2}.
\ea
$$
Since $\alpha+\beta+n\le s\le 2s_0-1$, we can use 
the inequality 
$$
\lA \eta\rA_{\eC{\alpha+\beta+n-s_0+1}}\le 
\triple{\eta}{n+\alpha+\beta-s_0+1,0}\le 
\lA \eta\rA_{s_0,0}
$$ 
in the third and last terms of the right hand side. On the other hand, since $(\alpha+1,n-1)\prec (\alpha,n)$ for $n\ge 1$ and 
since $\double{\eta}{n-1,\alpha+\beta+1}=0$ by convention for $n=0$, we have
$$
\double{\eta}{n-1,\alpha+\beta+1}\le \mathcal{M}_K.
$$
Also, one has 
$\blA Z^n \eta\brA_{H^{\alpha+\beta}}\le 
Y_{(\alpha,n)}+\mathcal{M}_K$. 
The wanted estimate~\e{n610-c} thus follows from statements $i)$, $ii)$, and $iii)$ in 
Lemma~\ref{T70.4}. The proof of \e{n612} is similar: 
we estimate $\px^{\alpha'} Z^{n'} \bigl(T_{\sqrt{a}-1}\eta\bigr)$ by means of the estimate 
~\e{n323b} and statements $i)$ and $ii)$ in Lemma~\ref{T70.4}. 

Let us prove~\e{n615}. We shall prove a stronger result. Namely, we prove that
\be\label{n615-t}
\triple{\vU-\vu}{s_0,\beta+3}\le C\bigl(N_\rho^{(s_0)}\bigr)
\Bigl(N_\rho^{(s_0)}\Bigr)^2.
\ee
We shall use the estimate~\e{n615-wa} whose statement is recalled here
\be\label{n615-w}
\triple{T_\zeta F}{n,\sigma}\les \triple{\zeta}{n,1}\triple{F}{n,\sigma}.
\ee
We decompose $\vU$ as 
\be\label{n615-a0}
U=u+\begin{pmatrix} T_{\sqrt{a}-1}\eta \\ -\Dxmez T_B \eta\end{pmatrix}.
\ee
So, to prove \e{n615-t}, it is sufficient to prove that
\be\label{n615-y}
\forall \zeta\in \{ \sqrt{a}-1,B\},
\quad \triple{T_\zeta \eta}{s_0,\beta+3+\mez}\le  
C\bigl(N_\rho^{(s_0)}\bigr)\Bigl(N_\rho^{(s_0)}\Bigr)^2.
\ee
This in turn follows from \e{n615-w} and the estimate for 
$\B$ (resp.\ $\sqrt{a}-1$) given by Proposition~\ref{T52} (resp.\ Lemma~\ref{T70.4} $i)$).
\end{proof}
\begin{rema*}We also have the following estimate, analogous to \e{n610-c}
\be\label{n610-d}
\ba
\lA \px^\alpha Z^n \vu-\px^\alpha Z^n \begin{pmatrix} \eta \\ \Dxmez \omega\end{pmatrix}
\rA_{H^{\beta-\mez}}& \le C(\lA u\rA_{\eC{\gamma}})\lA u\rA_{\eC{\gamma}}
Y_{(\alpha,n)}\\
&\quad +C\bigl(N_\rho^{(s_0)}\bigr)N_\rho^{(s_0)}\Avant,
\ea
\ee
The proof is similar to the proof of \e{n610-c}, using that 
$\vu=\begin{pmatrix} \eta \\ \Dxmez \omega\end{pmatrix}
+\begin{pmatrix} 0 \\ \Dxmez T_{\B}\eta\end{pmatrix}$.
\end{rema*}

We next estimate the source terms 
$\ZmF{\alpha,n}''$ and $\ZmG{\alpha,n}'$ given by \e{n537}. 
\begin{lemm}\label{T72}
There holds
$$
\lA \ZmF{\alpha,n}''\rA_{H^\beta}
\le 
C(\lA \vu\rA_{\eC{\gamma}})\lA \vu\rA_{\eC{\gamma}}Y_{(\alpha,n)}+C\bigl(N_\rho^{(s_0)}\bigr) \avant \Avant .
$$
\end{lemm}
\begin{proof}
By definition, one can write $\ZmF{\alpha,n}''$ under the form
\begin{align*}
\ZmF{\alpha,n}''&=
\sum_{j\in \Jpr}m(j)Q^{(n_{3})}(\px^{\alpha_1} Z^{n_{1}} \vu) \px^{\alpha_2} Z^{n_{2}}\vU\\
&\quad
+\sum_{j\in \Jpr}m(j)S^{(n_{3})}(\px^{\alpha_1} Z^{n_{1}} \vu) \px^{\alpha_2} Z^{n_{2}}\vU
\\
&\quad -Q(\px^\alpha Z^n\vu)\vU-S^\sharp(\px^\alpha Z^n\vu)\vU,
\end{align*}
where $m(j)\in \xN$ and
$$
\Jpr=\left\{\,(\alpha_1,\alpha_2,n_{1},n_{2},n_{3})\in \Jalpha \,;\, 
\alpha_1+n_1<\alpha+n\,\right\}.
$$
Below we freely use the fact that, by definition of $\Jalpha$ (see~\e{n526}), 
if $(\alpha_1,\alpha_2,n_{1},n_{2},n_{3})$ is in $\Jalpha$ then $\alpha_2+n_2<\alpha+n$. 

Let us split $\Jpr$ into two parts: 
set $\Jpr=J_1'\cup J_2'$ where
\begin{equation}\label{defi:J1J2}
\begin{aligned}
J_1'&=\left\{\, j=(\alpha_{1},\alpha_{2},n_{1},n_{2},n_{3})\in \Jpr\,;\, 
\alpha_{1}+n_1\le \mez (\alpha+n)\,\right\},\\
J_2'&=\left\{\, j=(\alpha_{1},\alpha_{2},n_{1},n_{2},n_{3})\in \Jpr\,;\, 
\alpha_{1}+n_1> \mez (\alpha+n)\,\right\}.
\end{aligned}
\end{equation}

We begin by estimating 
$$
\sum_{j\in J_1'}
m(j)Q^{(n_{3})}(\px^{\alpha_1} Z^{n_{1}} \vu) \px^{\alpha_2} Z^{n_{2}}
+\sum_{j\in J_1'}m(j)S^{(n_{3})}(\px^{\alpha_1} Z^{n_{1}} \vu) \px^{\alpha_2} Z^{n_{2}}\vU.
$$
If $j$ belongs to $J_1'$ and $A$ denotes $Q^{(n_{3})}$ (resp.\ $S^{(n_{3})}$) 
then we use \eqref{n540} (resp.\ \eqref{n542}) to obtain
$$
\lA A(\px^{\alpha_1} Z^{n_{1}} \vu) \px^{\alpha_2} Z^{n_{2}}\vU\rA_{H^\beta}
\le K \lA \px^{\alpha_1} Z^{n_{1}} \vu\rA_{\eC{4}}\lA \px^{\alpha_2} Z^{n_{2}}\vU\rA_{H^{\beta+1}}.%
$$
If $j\in J_1'$ and $(\alpha_2,n_2)\neq (\alpha-1,n)$ then one uses \e{n612} to find that 
$$
\lA \px^{\alpha_2} Z^{n_{2}}\vU\rA_{H^{\beta+1}}\le 
\lA \px^{\alpha_2} Z^{n_{2}}\vU\rA_{H^{\beta}}
+\lA \px^{\alpha_2+1} Z^{n_{2}}\vU\rA_{H^{\beta}}
\le 
C\bigl(N_\rho^{(s_0)}\bigr) \mathcal{M}_K,
$$
where we used the fact that if $j\in J_1'\subset J$ then $(\alpha_2,n_2)\prec (\alpha,n)$ and 
$\alpha_2\le \alpha$, so that the assumption that $(\alpha_2,n_2)\neq (\alpha-1,n)$ implies that $(\alpha_2+1,n)\prec (\alpha,n)$. 
On the other hand 
$$
\lA \px^{\alpha_1} Z^{n_{1}} \vu\rA_{\eC{4}}\le N_\rho^{(s_0)},
$$
since $\alpha_1+n_1+4\le\mez (\alpha+n)+4\le \sd +4\le s_0$ by assumption on $s_0$. 

If $j\in J_1'$ and $(\alpha_2,n_2)=(\alpha-1,n)$ then $(\alpha_1,n_1)=(1,0)$ so 
$\lA \px^{\alpha_1} Z^{n_{1}} \vu\rA_{\eC{2}}\le \lA \vu\rA_{\eC{\gamma}}$. 
On the other hand, \e{n610-b} and \e{n612} imply that
\begin{equation*}
\begin{aligned}
\lA \px^{\alpha_2} Z^{n_{2}}\vU\rA_{H^{\beta+1}}&\le 
\lA \px^{\alpha} Z^{n}\vU\rA_{H^\beta}+\lA \px^{\alpha-1}Z^{n}\vU\rA_{H^\beta}\\
&\le 
C\bigl(\lA u\rA_{\eC{\gamma}}\bigr)Y_{(\alpha,n)}+C\bigl(N_\rho^{(s_0)}\bigr)\mathcal{M}_K.
\end{aligned}
\end{equation*}

We now estimate 
$$
\sum_{j\in J_2'}
m(j)Q^{(n_{3})}(\px^{\alpha_1} Z^{n_{1}} \vu) \px^{\alpha_2} Z^{n_{2}}
+\sum_{j\in J_2'}m(j)S^{(n_{3})}(\px^{\alpha_1} Z^{n_{1}} \vu) \px^{\alpha_2} Z^{n_{2}}\vU.
$$
If $j$ belongs to $J_2'$ and $A$ denotes either 
$Q^{(n_{3})}$ or $S^{(n_{3})}$ then we use \eqref{n541} or \eqref{n543} to obtain
$$
\lA A(\px^{\alpha_1} Z^{n_{1}} \vu) \px^{\alpha_2} Z^{n_{2}}\vU\rA_{H^\beta}
\le K \lA \px^{\alpha_1} Z^{n_{1}} \vu\rA_{L^2}\lA \px^{\alpha_2} Z^{n_{2}}\vU\rA_{\eC{\beta+3}}.
$$
For any $j\in J_2'\subset \Jpr$ we have $\alpha_1+n_1<\alpha+n$ and 
$\alpha_1\le \alpha$, $n_1\le n$ so that $(\alpha_1,n_1)\prec (\alpha,n)$. Since $\beta\ge 1/2$, \e{n611} implies that
$$
\lA \px^{\alpha_1} Z^{n_{1}} \vu\rA_{L^2}\le 
\lA \px^{\alpha_1} Z^{n_{1}} \vu\rA_{H^{\beta-\mez}}\le \mathcal{M}_K.
$$
Since $\alpha_2+n_2\le \mez(\alpha+n)\le s_0$ for $j\in J_2'$, \e{n615} implies that 
\be\label{n615-bis}
\lA \px^{\alpha_2} Z^{n_{2}}\vU\rA_{\eC{\beta+3}}\le C\bigl(N_\rho^{(s_0)}\bigr)
N_\rho^{(s_0)}.
\ee

It remains to estimate 
$Q(\px^\alpha Z^n\vu)\vU$ and $S^\sharp(\px^\alpha Z^n\vu)\vU$. Using \eqref{n541}, we find that
$$
\blA Q(\px^\alpha Z^n\vu)\vU\brA_{H^\beta}\les \blA \px^\alpha Z^n\vu\brA_{L^2}
\lA \vU\rA_{\eC{\beta+3}}.
$$
It follows from \eqref{n610} 
and the assumption $\beta\ge 1/2$ that $\blA \px^\alpha Z^n\vu\brA_{L^2}\les Y_{(\alpha,n)}+\Avant$. 
On the other hand, we claim that
\be\label{n616}
\lA \vU\rA_{\eC{\beta+3}}\le C(\lA u\rA_{\eC{\gamma}}) 
\lA \vu\rA_{\eC{\gamma}}.
\ee
For further references, we shall prove a stronger estimate:
\be\label{n616bis}
\lA \vU-\vu\rA_{\eC{\beta+3}}\le C(\lA u\rA_{\eC{\gamma}}) 
\lA \vu\rA_{\eC{\gamma}}^2.
\ee
To prove this claim, recall that
\be\label{n616-a}
U=u+\begin{pmatrix} T_{\sqrt{a}-1}\eta \\ -\Dxmez T_B \eta\end{pmatrix}.
\ee
So to prove \e{n616} it is enough to prove that
\be\label{n616-b}
\blA T_{\sqrt{a}-1}\eta\brA_{\eC{\beta+3}}
+\blA \Dxmez T_B \eta\brA_{\eC{\beta+3}}
\le C(\lA u\rA_{\eC{\gamma}})\lA u\rA_{\eC{\gamma}}^2.
\ee
It follows from \e{Tab:Crho} that 
$\blA T_{\sqrt{a}-1}\eta\brA_{\eC{\beta+3}}\les 
\lA\sqrt{a}-1\rA_{L^\infty}\lA \eta\rA_{\eC{\beta+3}}$. Similarly, 
for any $r>1/2$, it follows from \e{esti:Dxmez-Crho} and \e{Tab:Crho} that 
$$
\blA \Dxmez T_B \eta\brA_{\eC{\beta+3}}\les 
\blA T_{\B}\eta\brA_{\eC{\beta+3+r}}\les 
\lA \B\rA_{L^\infty}\lA \eta\rA_{\eC{\beta+3+r}}.
$$
So \e{n616-b} follows from the assumption $\gamma>\beta+4$ and 
the estimate (see \e{n203} and \e{n189})
$$
\lA \sqrt{a}-1\rA_{L^\infty}+\lA \B\rA_{L^\infty}\le 
C(\lA u\rA_{\eC{\gamma}})\lA u\rA_{\eC{\gamma}}.
$$ 
This completes the proof of \e{n616}.

The estimate for $S^\sharp(\px^\alpha Z^n\vu)\vU$ is similar.
\end{proof}

\begin{lemm}\label{T74}
There holds
\be\label{n616-c}
\lA \ZmG{\alpha,n}'\rA_{H^\beta}\le 
C(\lA \vu\rA_{\eC{\gamma}})\left\{\lA \vu\rA_{\eC{\gamma}}^2 Y_{(\alpha,n)}
+\avant^2\Avant\right\}.
\ee
\end{lemm}
\begin{proof}
It follows from \e{n537} that
\be\label{n537-bis}
\begin{aligned}
\ZmG{\alpha,n}'=\ZmG{\alpha,n}&-
\bigl(S^\sharp(\px^\alpha Z^n \vu)(\vU-\vu)+S^\sharp(\px^\alpha Z^n (\vu-\vU))\vu\bigr)\\
&-\bigl(S^\flat(\vU-\vu)\px^\alpha Z^n \vu+S^\flat(\vu)\px^\alpha Z^n (\vu-\vU)\bigr).
\end{aligned}
\ee

To estimate the last two terms in the right hand side  of \e{n537-bis}, 
we use the estimates (for $\rho\not\in\mez\xN$)
\begin{align*}
&\blA S^\sharp(v)w\brA_{H^{\mu+\rho}}+\blA S^\flat(v) w\brA_{H^{\mu+\rho}}
\le K \lA v\rA_{\eC{\rho}}\lA w\rA_{H^{\mu+\tdm}},\\
&\blA S^\sharp(v)w\brA_{H^{\mu+\rho}}+\blA S^\flat(v) w\brA_{H^{\mu+\rho}}
\le K \lA w\rA_{\eC{\rho}}\lA v\rA_{H^{\mu+\tdm}},
\end{align*}
which readily follow from the definition~\e{n235.1}, and the estimates 
\e{n610-c} and \e{n616bis} for $\vu-\vU$. 

Let us show that the estimate for $\ZmG{\alpha,n}$ 
follows from the results proved in $\S$\ref{S:23}. 
The key point is to estimate the $\double{\cdot}{n,\alpha+\beta}$-norm 
of $G^1$ and $G^2$ given by \e{n517.6}Êand \e{n517.8}. 

Rewrite $G^1$ (as given by \e{n517.6}) as  
\be\label{n517.6a}
\ba
G^1&=
F(\eta)\psi - F_{\quadratique}(\eta)\psi+
T_{\partial_t\sqrt{a}-\px V+\mez \px^2\psi}\eta\\[0.5ex]
&\quad +T_{\sqrt{a}-1}F(\eta)\psi \\
&\quad +\Bigl\{ -T_{\sqrt{a}-1}
T_{\px V}+T_V T_{\px\sqrt{a}}\eta+\bigl[T_V,T_{\sqrt{a}-1}\bigr]
-\mez T_{\Dx^\tdm \vu^2}T_{\sqrt{a}-1}\Bigr\}\eta,\\[0.5ex]
&\quad +\Dx \RBony(\Dx\psi,T_{\sqrt{a}-1} \eta)
+\px \RBony(\px\psi,T_{\sqrt{a}-1}\eta),
\ea
\ee

$\bullet$ The $\double{\cdot}{n,\alpha+\beta}$-norm of $F(\eta)\psi - F_{\quadratique}(\eta)\psi$ is estimated by means of Proposition~\ref{T62.5} applied with $(k,\mu)=(n,\alpha+\beta)$ which yields
\be\label{n616-d}
\ba
&\double{F(\eta)\psi-F_{\quadratique}(\eta)\psi}{n,\alpha+\beta} \\
&\qquad\le \Cri \lA \eta\rA_{\eC{\gii}}^2\blA \Dxmez Z^n \psi\brA_{H^{\alpha+\beta-\mez}}\\
&\qquad\quad + \indicator{\xR_+}(\alpha+\beta+n-s_0+\nii)\Cri \lA \eta\rA_{\eC{\gii}} \blA \Dxmez \psi\brA_{\eC{\gii}} 
\blA Z^n\eta\brA_{H^{\alpha+\beta}}\\
&\qquad\quad +\Crs\triple{\eta}{s_0,0}^2
\bdouble{\Dxmez\psi}{n-1,\alpha+\beta+\mez}\\
&\qquad\quad +\indicator{\xR_+}(\alpha+\beta-\gamma_2')
\Crs\triple{\eta}{s_0,0}^2
\bdouble{\Dxmez\psi}{n,\alpha+\beta-\tdm}\\
&\qquad\quad+\Crs \triple{\eta}{s_0,0} \btriple{\Dxmez\psi}{\alpha+\beta+n-s_0+\nii,\gii}
\double{\eta}{n-1,\alpha+\beta+1}\\
&\qquad\quad +\indicator{\xR_+}(\alpha+\beta-\gii')
\Crs \triple{\eta}{s_0,0} \btriple{\Dxmez\psi}{\alpha+\beta+n-s_0+\nii,\gii}
\double{\eta}{n,\alpha+\beta-1},
\ea
\ee
where $\Cri=C(\lA \eta\rA_{\eC{\gii}})$, 
$\Crs=C( \triple{\eta}{s_0,0})$, and 
$\indicator{\xR_+}$ is the indicator function of $\xR_+$.
The first four terms in the 
right hand side  of \e{n616-d}Ê
are clearly controlled by the right hand side 
of \e{n616-c}. To estimate the last but one term in 
the right hand side  of \e{n616-d}, notice that,  since $\gamma$ has been chosen large relatively to $\gii$ and $\nii$, since $\alpha+\beta+n\le s$ and since 
$s\le 2s_0-\gamma$, we have
$\alpha+\beta+n-s_0+\nii+\gii\le s_0$ and hence 
$\btriple{\Dxmez\psi}{\alpha+\beta+n-s_0+\nii,\gii}\le N_\rho^{(s_0)}$ for any $\rho\ge \gamma$. It remains to estimate the last term in the right hand side  of \e{n616-d}. Notice 
that, because of the indicator function, it is 
non zero only for $\alpha+\beta\ge \gii'$. Since $\beta\le \gii'-1$ 
by assumption~\e{n403a}  
on $\beta$, this means that the last term is non zero only for $\alpha>0$. 
Now for $\alpha>0$ we have $\double{\eta}{n,\alpha+\beta-1}\le \Avant$ and hence 
the last term in the right hand side  of \e{n616-d} is also controlled 
by the right hand side  of \e{n616-c}. 

$\bullet$ We now estimate the $\double{\cdot}{n,\alpha+\beta}$-norm of 
$T_{\partial_t\sqrt{a}-\px V+\mez \px^2\psi}\eta$. To do so, 
we first check that one has the following estimates
\be\label{n616-e1}
\ba
&\btriple{\partial_t a- \px^2\psi}{s_0,1}\le C\bigl(N_\rho^{(s_0)}\bigr)\Bigl(N_\rho^{(s_0)}\Bigr)^2,\\
&\blA \px^{\alpha'}Z^{n'} \bigl(\partial_t a- \px^2\psi\bigr)\brA_{L^2}\le C\bigl(N_\rho^{(s_0)}\bigr)
N_\rho^{(s_0)}\Avant 
\quad \text{ for }(\alpha',n')\prec (\alpha,n),
\\
&\blA \px^\alpha Z^n (\partial_t a- \px^2\psi) \brA_{L^2}\le C(\lA u\rA_{\eC{\gamma}})\lA u\rA_{\eC{\gamma}}
Y_{(\alpha,n)} +C\bigl(N_\rho^{(s_0)}\bigr)N_\rho^{(s_0)}\Avant.
\ea
\ee
To prove these estimates, we use the arguments used in the proof of \e{n193}: we differentiate in time 
the identity~\eqref{formule:a} for $\ma$ (by using the rule~\e{n187.2}) and then we use 
Lemma~\ref{L:A.4.1}. This gives that $\partial_t \ma -\px^2\psi$ is an explicit 
sum of quadratic terms which are estimated as in Lemma~\ref{T70.4}. 
Next, \e{n616-e1} readily implies that 
$\partial_t\sqrt{a}-\mez\px^2\psi$ satisfies  
\be\label{n616-e2}
\ba
&\btriple{\partial_t\sqrt{a}-\mez\px^2\psi}{s_0,1}\le C\bigl(N_\rho^{(s_0)}\bigr)\Bigl(N_\rho^{(s_0)}\Bigr)^2,\\
&\blA \px^{\alpha'}Z^{n'} \bigl(\partial_t\sqrt{a}-\mez\px^2\psi\bigr)\brA_{L^2}\le C\bigl(N_\rho^{(s_0)}\bigr)
N_\rho^{(s_0)}\Avant\quad \text{ for }(\alpha',n')\prec (\alpha,n),\\
&\blA \px^\alpha Z^n (\partial_t\sqrt{a}-\mez\px^2\psi) \brA_{L^2}\le C(\lA u\rA_{\eC{\gamma}})\lA u\rA_{\eC{\gamma}}
Y_{(\alpha,n)} +C\bigl(N_\rho^{(s_0)}\bigr)N_\rho^{(s_0)}\Avant.
\ea
\ee
On the other hand, the estimates 
\e{n367} and \e{2311aa-bis} imply that $\px V-\px^2\psi$ satisfies 
\be\label{n616-e3}
\ba
&\btriple{\px V-\px^2\psi}{s_0,1}\le C\bigl(N_\rho^{(s_0)}\bigr)\Bigl(N_\rho^{(s_0)}\Bigr)^2,\\
&\blA \px^{\alpha'}Z^{n'} \bigl(\px V-\px^2\psi\bigr)\brA_{L^2}\le C\bigl(N_\rho^{(s_0)}\bigr)
N_\rho^{(s_0)}\Avant\quad \text{ for }(\alpha',n')\prec (\alpha,n),\\
&\blA \px^\alpha Z^n (\px V-\px^2\psi) \brA_{L^2}\le C(\lA u\rA_{\eC{\gamma}})\lA u\rA_{\eC{\gamma}}
Y_{(\alpha,n)} +C\bigl(N_\rho^{(s_0)}\bigr)N_\rho^{(s_0)}\Avant.
\ea
\ee
Set
$$
\zeta\defn \partial_t \sqrt{a}-\px V+\mez \px^2\psi.
$$
Then, by using the triangle inequality, \e{n616-e2} and \e{n616-e3} imply that
\be\label{n616-e}
\ba
&\triple{\zeta}{s_0,1}\le C\bigl(N_\rho^{(s_0)}\bigr)\Bigl(N_\rho^{(s_0)}\Bigr)^2,\\
&\blA \px^{\alpha'}Z^{n'} \zeta\brA_{L^2}\le C\bigl(N_\rho^{(s_0)}\bigr)
N_\rho^{(s_0)}\Avant\quad \text{ for }(\alpha',n')\prec (\alpha,n),\\
&\blA \px^\alpha Z^n \zeta\brA_{L^2}\le C(\lA u\rA_{\eC{\gamma}})\lA u\rA_{\eC{\gamma}}
Y_{(\alpha,n)} +C\bigl(N_\rho^{(s_0)}\bigr)N_\rho^{(s_0)}\Avant.
\ea
\ee
Now, to estimate the 
$\double{\cdot}{n,\alpha+\beta}$-norm of $T_{\zeta}\eta$, we apply 
the estimate~\e{n369.5} with $m=s_0$, 
$\nu=\alpha+\beta$ and $b=\gamma$. This yields
$$
\ba
\double{T_{\zeta}\eta}{n,\alpha+\beta} &\les 
\triple{\zeta}{s_0,0}\double{\eta}{n-1,\alpha+\beta+1}
+\lA \zeta\rA_{L^\infty}
\blA Z^n \eta\brA_{H^{\alpha+\beta}}\\
&\quad+\triple{\eta}{\alpha+\beta+n-s_0+1,0}\double{\zeta}{n-1,0}\\
&\quad +\lA \eta\rA_{\eC{\gamma}}
\blA Z^n\zeta\brA_{L^2}\\
&\quad +\indicator{\xR_+} (\alpha+\beta+1-\gamma)
\lA \eta\rA_{\eC{\alpha+\beta+n-s_0+1}}\blA Z^n\zeta\brA_{L^2}.
\ea
$$
In view of \e{n616-e}, the first four terms in the right hand side are 
clearly controlled by the right hand side of \e{n616-c}. 
Again, to bound the last term, we notice that is non zero only for $\alpha >0$ since 
$\beta+1-\gamma<0$ by assumption. Now, for $\alpha>0$, we have 
$(0,n)\prec (\alpha,n)$ and hence $\blA Z^n\zeta\brA_{L^2}$ is estimated by the 
second inequality in \e{n616-e}. On the other hand, again, we 
$\lA \eta\rA_{\eC{\alpha+\beta+n-s_0+1}}\le N_\rho^{(s_0)}$ by assumptions 
on $\alpha,\beta,n,s,s_0,\rho$. 

$\bullet$ Now we estimate the $\double{\cdot}{n,\alpha+\beta}$-norm of 
$T_{\sqrt{a}-1} F(\eta)\psi$. We apply the estimate \e{n369.5} with 
$(K,\nu,m,b)$ replaced by $(n,\alpha+\beta,s_0,\beta+2)$. This gives
$$
\ba
\double{T_{\sqrt{a}-1}F(\eta)\psi}{n,\alpha+\beta} &\les 
\triple{\sqrt{a}-1}{s_0,0}\double{F(\eta)\psi}{n-1,\alpha+\beta+1}
+\lA \sqrt{a}-1\rA_{L^\infty}
\blA Z^n F(\eta)\psi\brA_{H^{\alpha+\beta}}\\
&\quad+\triple{F(\eta)\psi}{\alpha+\beta+n-s_0+1,0}\double{\zeta}{n-1,0}\\
&\quad +\lA F(\eta)\psi\rA_{\eC{\beta+2}}
\blA Z^n(\sqrt{a}-1)\brA_{L^2}\\
&\quad +\indicator{\xR_+} (\alpha-1)\lA F(\eta)\psi\rA_{\eC{\alpha+\beta+n-s_0+1}}
\blA Z^n(\sqrt{a}-1)\brA_{L^2}.
\ea
$$
The first and second term in the right hand side 
are estimated by means 
of the previous estimates for 
$\sqrt{a}-1$ (see Lemma~\ref{T70.4}) and 
$F(\eta)\psi$ (see \e{n358}, which easily implies an estimate for 
$\double{F(\eta)\psi}{k,\mu}$, using the triangle inequality and 
the fact that one can estimate $\double{F_{\quadratique}(\eta)\psi}{k,\mu}$ 
directly from~\eqref{n363} and \e{n325c}). Again, notice that the last term is non zero only for $\alpha>0$. Then $(0,n)\prec (\alpha,n)$ and 
$\blA Z^n(\sqrt{a}-1)\brA_{L^2}$ is controlled by Lemma~\ref{T70.4} $ii)$. On the other hand, by assumptions on 
$\alpha,\beta,s,n,s_0$ we have $\alpha+\beta+n-s_0+1\le s_0$. Therefore, 
it remains only to bound $\triple{F(\eta)\psi}{s_0,0}$ 
and $\lA F(\eta)\psi\rA_{\eC{\beta+2}}$. 
Both estimates are easily obtained writing
$$
F(\eta)\psi=\Bigl( G(\eta)\psi-\Dx\psi\Bigr)-\Bigl(\Dx T_{\B(\eta)\psi}\eta 
+\px T_{V(\eta)\psi}\eta\Bigr).
$$
The $\triple{\cdot}{s_0,0}$-norm (resp.\ $\lA \cdot\rA_{\eC{\beta+2}}$) 
norm of the first term is estimated by \e{n367} (resp.\ \e{n145}). 
The $\triple{\cdot}{s_0,0}$-norm (resp.\ $\lA \cdot\rA_{\eC{\beta+2}}$) 
norm of the second term is estimated by \e{n615-w} and \e{Z:2} 
(resp.\ \e{Tab:Crho} and \e{211-1}). 

The last terms in the third and fourth lines of the right hand side of \e{n517.6a}Ê
are estimated by means of 
\e{n323c}, \e{n325c}, \e{n336}, and Lemma~\ref{T70.4}. 

We now estimate $G^2$ which is given by \e{n517.8}. 
To estimate the first and second terms in the right hand side  of \e{n517.8}, we use 
the estimates~\e{n367} and \e{2311aa-bis} for the 
$\triple{\cdot}{n,\sigma}$ and $\double{\cdot}{K,\nu}$ norms 
of $B(\eta)-\Dx$ and $V(\eta)-\px$. Then the desired estimates follow 
from \e{n325c} and \e{n610-d}. 

The third and fifth terms in the right hand side  of \e{n517.8} are estimated by means 
of \e{n336}, Proposition~\ref{T52}, Proposition~\ref{T56} and Lemma~\ref{T70.4}. 
The fourth term is estimated by means of \e{n323c}, \e{n325c}, 
Proposition~\ref{T52}, and Proposition~\ref{T56}.

To complete the study of $\ZmG{\alpha,n}$ we have to study 
the terms involving the operator $C(\vu)$ in \e{n522} and \e{n524}. We obtain the wanted estimates 
by using the estimates~\e{n367} and \e{2311aa-bis} for the estimates of the 
$\triple{\cdot}{n,\sigma}$ and $\double{\cdot}{K,\nu}$ norms 
of $V(\eta)-\px$, statement $iv)$ in Lemma~\ref{T70.4} (which implies similar estimates 
for $(\sqrt{a}-1)+\mez\Dx\eta$) and the rules \e{n328}, \e{n329}.
\end{proof}

It follows from Lemma~\ref{T70.2}, Lemma~\ref{T72}, and Lemma~\ref{T74} that 
the $H^\beta$-norm of the terms $(2)$ and $(3)$ in \e{n595} are controlled by the 
right hand side  of \e{n512}. Since we have already estimated the term $(1)$ in Lemma~\ref{T70}, 
to complete the proof, it remains only to prove the estimate~\e{n514a}--\e{n514b} and to 
estimate the $H^\beta$-norms of the term $(4)$ and $\ZWGamma{\alpha,n}$ which 
appear in \e{n595}.

\begin{lemm}\label{T75}
$i)$ 
There holds
\begin{align*}
&\lA (\partial_t+D)\px^\alpha Z^n \vu\rA_{L^2}\le 
C(\lA u\rA_{\eC{\gamma}}) \lA u\rA_{\eC{\gamma}} Y_{(\alpha,n)} +
C\bigl(N_\rho^{(s_0)}\bigr)N_\rho^{(s_0)}\Avant,\\
&\lA (\partial_t+D) \px^\alpha Z^n \vU\rA_{L^2}\le C(\lA u\rA_{\eC{\gamma}})
\lA u\rA_{\eC{\gamma}} Y_{(\alpha,n)} +C\bigl(N_\rho^{(s_0)}\bigr)N_\rho^{(s_0)}\Avant.
\end{align*}
$ii)$ If $(\alpha',n')\prec (\alpha,n)$ then 
\begin{align*}
&\blA (\partial_t+D)\px^{\alpha'} Z^{n'} \vu\brA_{L^2}\le 
C\bigl(N_\rho^{(s_0)}\bigr)N_\rho^{(s_0)}\Avant,\\
&\blA (\partial_t+D)\px^{\alpha'} Z^{n'} \vU\brA_{L^2}\le C\bigl(N_\rho^{(s_0)}\bigr)N_\rho^{(s_0)}\Avant.
\end{align*}
$iii)$ If $\alpha'+n'\le s_0$ then 
\begin{align*}
&\blA (\partial_t+D)\px^{\alpha'} Z^{n'} \vu\brA_{\eC{4}}\le  
C\bigl(N_\rho^{(s_0)}\bigr)\Bigl(N_\rho^{(s_0)}\Bigr)^2,\\
&\blA (\partial_t+D)\px^{\alpha'} Z^{n'} \vU\brA_{\eC{4}}\le  
C\bigl(N_\rho^{(s_0)}\bigr)\Bigl(N_\rho^{(s_0)}\Bigr)^2.
\end{align*}
\end{lemm}
\begin{proof}
Notice that the third (resp.\ the fourth) estimate is an obvious consequence of the 
first (resp.\ the second) estimate 
since 
$\lA u\rA_{\eC{\gamma}}\le N_\rho^{(s_0)}$ and since 
$Y_{(\alpha',n')}\le \mathcal{M}_K$ for $(\alpha',n')\prec (\alpha,n)$. 

To prove the first estimate, recall that
\be\label{n616-t}
\partial_t\vu +D\vu =\begin{pmatrix} 
G(\eta)\psi-\Dx\psi\\
\Dxmez \bigl( -\mez (\px\psi)^2+\mez (1+(\partial_x\eta)^2)(\B(\eta)\psi)^2\bigr)
\end{pmatrix}.
\ee
Therefore, using \e{n518} to commute $\px^{\alpha}Z^{n}$ with $\partial_t+D$, 
the first estimate follows from Corollary~\ref{ref:235E1}, Proposition~\ref{T56}, 
Proposition~\ref{T52}, 
and the product rule \e{n335e}. 

To prove the second estimate, 
we notice that for any $(\alpha,n)\in \mathcal{P}$,
\be\label{n616-t'}
(\partial_t+D)\px^{\alpha} Z^{n} U =-N(\vu)\px^{\alpha} Z^{n} U 
+\ZmGb{\alpha,n}' +\ZmFb{\alpha,n}'',
\ee
where $\ZmGb{\alpha,n}'$ and $\ZmFb{\alpha,n}''$ are given by \e{n537}. According to 
\e{n539} applied with $\mu=0$, the first term in 
the right hand side  clearly satisfies the wanted estimate. 
Thus the second estimate in the lemma follows from Lemma~\ref{T72} and Lemma~\ref{T74}.

Let us prove the estimates in statement $iii)$. 
Using again \e{n518} to commute $\px^{\alpha}Z^{n}$ with $\partial_t+D$, notice that it is enough to prove that 
\begin{align}
&\triple{(\partial_t+D)\vu}{s_0,4}\le  
C\bigl(N_\rho^{(s_0)}\bigr)\Bigl(N_\rho^{(s_0)}\Bigr)^2,\label{n616-u}\\
&\triple{(\partial_t+D)\vU}{s_0,4}\le  
C\bigl(N_\rho^{(s_0)}\bigr)\Bigl(N_\rho^{(s_0)}\Bigr)^2.\label{n616-v}
\end{align}
The estimate \e{n616-u} follows from \e{n616-t}, Proposition~\ref{T52} and the 
product rule~\e{prod:eCZ}. To prove \e{n616-v}, we use the estimate \e{n615-w} whose statement is recalled here: 
\be\label{n616-w}
\triple{T_\zeta F}{n,\sigma}\les \triple{\zeta}{n,1}\triple{F}{n,\sigma}.
\ee
Remembering the decomposition~\e{n616-a} of $\vU$ as $\vu+\vU'$ with 
$\vU'=(T_{\sqrt{a}-1}\eta,\Dxmez T_B\eta)$, and using \e{n616-u}, it is enough to prove that
\be\label{n616-x}
\triple{(\partial_t+D)\vU'}{s_0,4}\le  
C\bigl(N_\rho^{(s_0)}\bigr)\Bigl(N_\rho^{(s_0)}\Bigr)^2.
\ee
Since it is enough to prove that the right-hand side is quadratic in $N_\rho^{(s_0)}$, to 
prove \e{n616-x}, it is sufficient to estimate separately $\partial_t \vU'$ and $D\vU'$. Thus, it is sufficient to prove that
\be\label{n616-y}
\forall \zeta\in \{ \sqrt{a}-1,\partial_t\sqrt{a},B,\partial_t B\},~
\forall F\in\{ \eta,\partial_t\eta\},\quad \triple{T_\zeta F}{s_0,4+\mez}\le  
C\bigl(N_\rho^{(s_0)}\bigr)\Bigl(N_\rho^{(s_0)}\Bigr)^2.
\ee
In view of \e{n616-w}, this reduces to proving that
\begin{alignat*}{2}
&\forall \zeta\in \{ \sqrt{a}-1,\partial_t\sqrt{a},B,\partial_t B\},\quad 
&&\triple{\zeta}{s_0,1}\le C\bigl(N_\rho^{(s_0)}\bigr)N_\rho^{(s_0)},\\
&\forall F\in\{ \eta,\partial_t\eta\}, \quad &&
\triple{F}{s_0,4+\mez }\le C\bigl(N_\rho^{(s_0)}\bigr)N_\rho^{(s_0)}.
\end{alignat*}
The second estimate is clear for $F=\eta$. Since $\partial_t\eta=G(\eta)\psi$, it follows from 
Proposition \ref{T52} for $F=\partial_t\eta$. On the other hand, for $\zeta=\B$ (resp.\ $\zeta=\sqrt{a}-1$) 
the first estimate follows from Proposition \ref{T52} (resp.\ Lemma~\ref{T70.4}). For 
$\zeta=\partial_t\sqrt{a}$, the first estimate follows from \e{n616-e} and Proposition~\ref{T52} (to estimate 
$\triple{\px V}{s_0,1}$). Eventually, for $\zeta=\partial_t \B$, 
we use that, by definition of $a$, 
$\partial_t \B=-V\px \B+\ma-1$ so that the wanted estimate follows from \e{prod:eCZ}, Proposition~\ref{T52} 
and Lemma~\ref{T70.4}.
\end{proof}

Introduce
\be\label{n617a}
X_{(\alpha,n)}\defn 
\blA \px^{\alpha}Z^{n}\vU\brA_{H^\beta}
+\blA \px^{\alpha}Z^{n}\vu\brA_{H^{\beta-\mez}},
\ee
and
\be\label{n617b}
\ba
M_K&\defn \mathcal{M}_K+
\sum_{(\alpha',n')\prec (\alpha,n)}X_{(\alpha',n')},\\
N_K&\defn N_\rho^{(s_0)}+\frac{1}{\nu}\Bigl( N_\rho^{(s_0)}\Bigr)^{1-\nu}M_K^\nu.
\ea
\ee
Recall that we want to prove that
\be\label{n617c}
\blA \ZGamma{\alpha,n}\brA_{H^\beta}\le
C(\lA \vu\rA_{\eC{\gamma}})\lA \vu\rA_{\eC{\gamma}}^2Y_{(\alpha,n)}
+C(\avant)\avant^2 \Avant.
\ee
According to Lemma~\ref{T71} and Lemma~\ref{T75} we have
\be\label{n617d}
\ba
&X_{(\alpha,n)}\le  C\bigl(\lA \vu\rA_{\eC{\gamma}}\big)Y_{(\alpha,n)}
+C\bigl(N_\rho^{(s_0)}\bigr)\mathcal{M}_K,\\
&\lA (\partial_t+D)\px^\alpha Z^n \vu\rA_{L^2}
+\lA (\partial_t+D) \px^\alpha Z^n \vU\rA_{L^2}
\le  C\bigl(\lA \vu\rA_{\eC{\gamma}}\big)Y_{(\alpha,n)}
+C\bigl(N_\rho^{(s_0)}\bigr)\mathcal{M}_K,\\
&M_K\le  C\bigl(N_\rho^{(s_0)}\bigr)\mathcal{M}_K
\ea
\ee
and it follows from the third inequality above that
\be\label{n617e}
\ba
N_K&=N_\rho^{(s_0)}+\frac{1}{\nu}\Bigl( N_\rho^{(s_0)}\Bigr)^{1-\nu}M_K^\nu\\
&\le N_\rho^{(s_0)}+\frac{1}{\nu}\Bigl( N_\rho^{(s_0)}\Bigr)^{1-\nu}
\Bigl(C\bigl(N_\rho^{(s_0)}\bigr)\mathcal{M}_K \Bigr)^\nu\\
&\le C\bigl(N_\rho^{(s_0)}\bigr)\mathcal{N}_K
\ea
\ee
by definition~\e{n513} of $\mathcal{N}_K$. Consequently, to prove \e{n617c} it is sufficient to prove that
\be\label{n618}
\ba
\blA \ZGamma{\alpha,n}\brA_{H^{\beta}}&\le 
C(\lA \vu\rA_{\eC{\gamma}})\lA \vu\rA_{\eC{\gamma}}^2 X_{(\alpha,n)}
+\CNK N_K^2 M_K\\
&\quad + C(\lA \vu\rA_{\eC{\gamma}})\lA \vu\rA_{\eC{\gamma}}
\Bigl\{ \lA (\partial_t+D)\px^\alpha Z^n \vu\rA_{L^2}
+\lA (\partial_t+D) \px^\alpha Z^n \vU\rA_{L^2}\Bigr\}.
\ea
\ee
(Let us mention that the factor $\lA \vu\rA_{\eC{\gamma}}$ multiplying the bracket in the second line is linear in $\lA \vu\rA_{\eC{\gamma}}$ instead of being quadratic since 
$(\partial_t+D) \px^\alpha Z^n \vU$ is at least quadratic, see~\e{n616-t'}.)

Next we prove that $\blA \ZWGamma{\alpha,n}\brA_{H^\beta}$ 
is estimated by the right hand side  of \e{n618}. Recall that $\ZWGamma{\alpha,n}$ 
is given by (see \e{n571})
\begin{equation}\label{n620}
\ZWGamma{\alpha,n}=\ZmG{\alpha,n}'-\ZmR{\alpha,n}+N(\vu)
\bigl(\ZWPhi{\alpha,n}-\px^\alpha Z^{n}\vU\bigr),
\end{equation}
with $N(\vu)=Q(\vu)+S^\sharp(\vu)+2S^\flat(\vu)+C(\vu)$ and 
where $\ZmG{\alpha,n}'$ is given by \e{n537}, $\ZmR{\alpha,n}$ is given by~\e{n569}, 
and $\ZWPhi{\alpha,n}$ is given by \e{n567}. We shall 
use a cancellation between the second term in the right hand side  of \e{n569} and the last term in the right hand side  
of \e{n620}. To do so we write according to \e{n616-t'}
\be\label{n622}
(\partial_t+D)\px^{\alpha_2} Z^{n_2} U=-N(\vu)\px^{\alpha_2} Z^{n_2} U
+\ZmGb{\alpha_2,n_{2}}' +\ZmFb{\alpha_2,n_{2}}'',
\ee
where $\ZmGb{\alpha_2,n_{2}}'$ and $\ZmFb{\alpha_2,n_{2}}''$ are obtained 
by replacing $(\alpha,n)$ by $(\alpha_2,n_{2})$ in the definition \e{n537} 
of $\ZmG{\alpha,n}'$ and $\ZmF{\alpha,n}''$. 
We substitute \e{n622} in the second term in the right hand side  of \e{n569} 
to obtain, using the definition~\e{n567} of $\ZWPhi{\alpha,n}$, 
${\displaystyle{\ZWGamma{\alpha,n}=
\ZmG{\alpha,n}'+\sum_{q=1}^{9} \ZWGamma{\alpha,n}^q}}$ with
\begin{align*}
\ZWGamma{\alpha,n}^1&= -\sum_{\Jalpha}m(j)
\EA{n_{3}}\bigl((\partial_t +D)\px^{\alpha_1} Z^{n_{1}}\vu\bigr)\px^{\alpha_2} Z^{n_{2}}\vU\\
\ZWGamma{\alpha,n}^2&=-\sum_{\Jpr}m(j)
\ER{n_{3}}\bigl((\partial_t +D)\px^{\alpha_1} Z^{n_{1}}\vu\bigr)\px^{\alpha_2} Z^{n_{2}}\vU
\\
\ZWGamma{\alpha,n}^3&=-\sum_{\Jpr}m(j)\ER{n_{3}}(\px^{\alpha_1} Z^{n_{1}}\vu)\bigl((\partial_t+D)\px^{\alpha_2} Z^{n_{2}}\vU\bigr)\\
\ZWGamma{\alpha,n}^4&=-E^\sharp((\partial_t+D)\px^\alpha Z^n\vu)\vU\\
\ZWGamma{\alpha,n}^5&=
-E^\sharp(\px^\alpha Z^n\vu)(\partial_t\vU+D\vU)\\
\ZWGamma{\alpha,n}^6&=+\sum_{\Jalpha}m(j) \Bigl[ \EA{n_{3}}(\px^{\alpha_1} Z^{n_{1}}\vu), 
N(\vu)\Bigr]\px^{\alpha_2}Z^{n_2}\vU\\
\ZWGamma{\alpha,n}^7&= -N(\vu)E^\sharp(\px^\alpha Z^n\vu)\vU\\
\ZWGamma{\alpha,n}^8&=-N(\vu)
\sum_{\Jpr}m(j)\ER{n_{3}}(\px^{\alpha_1} Z^{n_{1}} \vu)\px^{\alpha_2} Z^{n_{2}}\vU\\
\ZWGamma{\alpha,n}^{9}&=
-\sum_{\Jalpha}m(j) \EA{n_{3}}(\px^{\alpha_1} Z^{n_1}\vu)\Bigl\{\ZmGb{\alpha_2,n_2}'
+\ZmFb{\alpha_2,n_2}''\Bigr\}\cdot
\end{align*}

We shall further split the sum over $J$ (resp.\ $\Jpr$) into two pieces according 
to the splitting of $\Jalpha$ as $\Jalpha=J_1\cup J_2$ (resp.\ $\Jpr=J_1'\cup J_2'$) where 
\begin{equation}\label{J1J2}
\begin{aligned}
J_1&=\left\{\, j=(\alpha_{1},\alpha_{2},n_{1},n_{2},n_{3})\in \Jalpha\,;\, 
\alpha_{1}+n_1\le \mez (\alpha+n)\,\right\},\\
J_2&=\left\{\, j=(\alpha_{1},\alpha_{2},n_{1},n_{2},n_{3})\in \Jalpha\,;\, 
\alpha_{1}+n_1> \mez (\alpha+n)\,\right\}
\end{aligned}
\end{equation}
(resp.\ $J_1'$ and $J_2'$ are given by \e{defi:J1J2} so that 
$J_1'=J_1\cap J'$ and $J_2'=J_2\cap J'$). 
Notice that, if $(\alpha,n)=(0,0)$ then $J=\emptyset=J'$. Therefore, 
if $j\in J_1$ then $(\alpha_1,n_1)\prec (\alpha,n)$. 

Using obvious notations, we write $\ZWGamma{\alpha,n}^q=
\ZWGamma{\alpha,n}^q_1+\ZWGamma{\alpha,n}^q_2$ for 
$q\in \{ 1,2,3,6,8,9\}$.

We shall use the following notation: for $r$ in $[0,+\infty[$ we set
\be\label{n790}
\lA v\rA_{\eC{r}\cap L^2}\defn \lA v\rA_{\eC{r}}+\frac{1}{\nu}\lA v\rA_{\eC{r}}^{1-\nu}\lA v\rA_{L^2}^\nu,
\ee
where recall that $\nu$ is a fixed small positive number, 
and the constants involved are independent of $\nu$.

\underline{Estimates of $\ZWGamma{\alpha,n}^1_1$, $\ZWGamma{\alpha,n}^2_1$ and $\ZWGamma{\alpha,n}^3_1$}

Let us prove that
\be\label{n794}
\blA \ZWGamma{\alpha,n}^1_1\brA_{H^\beta}
+\blA \ZWGamma{\alpha,n}^2_1\brA_{H^\beta}
\le C\bigl(\lA \vu\rA_{\eC{\gamma}}\bigr)\lA \vu\rA_{\eC{\gamma}}^2 X_{(\alpha,n)}+\CNK N_K^2 M_K.
\ee

If $A$ denotes $\EA{n_{3}}$ or $\ER{n_{3}}$ then 
the estimates \e{n553} and \e{n561} (together with \e{n563}) imply that
\be\label{n795}
\ba
&\lA A\bigl((\partial_t +D)\px^{\alpha_1} Z^{n_{1}}\vu\bigr)\px^{\alpha_2} Z^{n_{2}}\vU\rA_{H^\beta}\\
&\qquad\qquad\les \blA (\partial_t +D)\px^{\alpha_1} Z^{n_{1}}\vu\brA_{\eC{4}\cap L^2}
\blA \px^{\alpha_2} Z^{n_{2}}\vU\brA_{H^{\beta+1}}.
\ea
\ee

Remembering that, by definition, 
$\ZWGamma{\alpha,n}^1_1$, $\ZWGamma{\alpha,n}^2_1$ and $\ZWGamma{\alpha,n}^3_1$ are sums of terms indexed by either $J_1$ or $J_1'$, we are going 
to use a dichotomy already used in the proof of Lemma~\ref{T72}. Either $(\alpha_2+1,n_2)\prec (\alpha,n)$ 
or $\alpha\ge 1$ and $(\alpha_2,n_2)=(\alpha-1,n)$. 

If $(\alpha_2+1,n_2)\prec (\alpha,n)$, writing 
$$
\lA \px^{\alpha_2} Z^{n_{2}}\vU\rA_{H^{\beta+1}}\le 
\lA \px^{\alpha_2} Z^{n_{2}}\vU\rA_{H^{\beta}}
+\lA \px^{\alpha_2+1} Z^{n_{2}}\vU\rA_{H^{\beta}}
$$
we see that the second factor in the right hand side of \e{n795} 
is bounded by $M_K$, by definition~\e{n617b} of $M_K$. 
To bound the 
first factor in the right hand side of \e{n795}, 
we first recall from Lemma~\ref{T75} that for any $j$ in $J_1$,
\be\label{n800}
\blA (\partial_t +D)\px^{\alpha_1} Z^{n_{1}}\vu\brA_{\eC{4}}
\le C\bigl(N_\rho^{(s_0)}\bigr)\Bigl(N_\rho^{(s_0)}\Bigr)^2
\ee
where we used the fact that, for $j\in J_1$, and our assumptions 
on $\alpha,n,s,s_0$, we have $\alpha_1+n_1\le s_0$. Secondly, we have
\be\label{n801}
\blA (\partial_t +D)\px^{\alpha_1} Z^{n_{1}}\vu\brA_{L^2}
\le C\bigl(N_\rho^{(s_0)}\bigr)N_\rho^{(s_0)} \mathcal{M}_K,
\ee
where, to obtain~\e{n801}, we used the above mentioned observation that $(\alpha_1,n_1)\prec (\alpha,n)$ for $j\in J_1$. 
By combining \e{n800} and \e{n801} we obtain
$$
\blA (\partial_t +D)\px^{\alpha_1} Z^{n_{1}}\vu\brA_{\eC{4}\cap L^2}
\le \CNK N_K^2
$$
by definition~\e{n617b} of $N_K$ and definition~\e{n790} of the norm $\lA \cdot\rA_{\eC{4}\cap L^2}$. This proves the wanted estimate.

Consider now the case when $(\alpha_2,n_2)=(\alpha-1,n)$. Then $(\alpha_1,n_1)=(1,0)$ and we have to estimate the $H^\beta$-norms of 
$$
\EA{0}\bigl((\partial_t +D)\px \vu\bigr)\px^{\alpha-1} Z^{n}\vU,\quad 
\ER{0}\bigl((\partial_t +D)\px \vu\bigr)\px^{\alpha-1} Z^{n}\vU.
$$
Both terms are estimated similarly. Let us consider the first one. Using the estimate 
\e{n553-bis} below, we have
\be\label{n802}
\blA \EA{0}\bigl((\partial_t +D)\px \vu\bigr)\px^{\alpha-1} Z^{n}\vU
\brA_{H^\beta}\le \blA (\partial_t +D) \vu\brA_{\eC{5}}\blA \px^{\alpha-1} Z^{n}\vU\brA_{H^{\beta+1}}.
\ee
(The key difference with \e{n795}Ê
is that the right hand side of the above inequality 
does not involve the $L^2$-norm of $(\partial_t +D) \vu$.) As above, using 
the notations \e{n617a} and \e{n617b}, one has
\be\label{n803}
\blA \px^{\alpha-1} Z^{n}\vU\brA_{H^{\beta+1}}\le 
\blA \px^{\alpha} Z^{n}\vU\brA_{H^{\beta}}
+\blA \px^{\alpha-1} Z^{n}\vU\brA_{H^{\beta}}\le X_{(\alpha,n)}+M_K.
\ee
On the other hand, according to \e{n616-t}, \e{n145}, \e{211-1} and the product rule 
in H\"older spaces, we have
\be\label{n804}
\blA (\partial_t +D) \vu\brA_{\eC{5}}\le C\bigl(\lA \vu\rA_{\eC{\gamma}}\bigr)
\lA \vu\rA_{\eC{\gamma}}^2
\ee 
provided that $\gamma$ is large enough. Plugging \e{n803} and 
\e{n804} in \e{n802}Ê
we obtain that the $H^\beta$-norm of 
$\EA{0}\bigl((\partial_t +D)\px \vu\bigr)\px^{\alpha-1} Z^{n}\vU$ is bounded by the 
right hand side of \e{n794}.

The $\lA \cdot\rA_{H^\beta}$-norm of $\ZWGamma{\alpha,n}^3_1$ is estimated by similar arguments.

\underline{Estimates of $\ZWGamma{\alpha,n}^1_2$, 
$\ZWGamma{\alpha,n}^2_2$, $\ZWGamma{\alpha,n}^3_2$, and $\ZWGamma{\alpha,n}^4$}

Consider $j\in J'$. Let us estimate the $H^\beta$-norms of 
$$
\EA{n_{3}}\bigl((\partial_t +D)\px^{\alpha_1} Z^{n_{1}}\vu\bigr)\px^{\alpha_2} Z^{n_{2}}\vU,\quad 
\ER{n_{3}}\bigl((\partial_t +D)\px^{\alpha_1} Z^{n_{1}}\vu\bigr)\px^{\alpha_2} Z^{n_{2}}\vU.
$$
If $A$ denotes $\EA{n_{3}}$ (resp.\ $\ER{n_{3}}$, ) then 
the estimate \e{n554} (resp.\ \e{n562}) implies that
\begin{align*}
&\lA A\bigl((\partial_t +D)\px^{\alpha_1} Z^{n_{1}}\vu\bigr)\px^{\alpha_2} Z^{n_{2}}\vU\rA_{H^\beta}\\
&\qquad\qquad\les \blA (\partial_t +D)\px^{\alpha_1} Z^{n_{1}}\vu\brA_{L^2}
\blA \px^{\alpha_2} Z^{n_{2}}\vU\brA_{\eC{\beta+3}\cap L^2}.
\end{align*}
To estimate the right hand side of the above inequality, we recall that we consider the case $j\in J'$ and notice that 
$(\alpha_1,n_1)\prec (\alpha,n)$ when $j\in J'$ (since by definition~\e{n532} of $J'$ we have $\alpha_1+n_1<\alpha+n$, $\alpha_1\le \alpha$, $n_1\le n$ for $j\in J'$). 
Then the second factor in the right hand side above is estimated by means of \e{n612} and \e{n615}, while the first factor is estimated by Lemma~\ref{T75} $ii)$. 
This proves that the right hand side of the above inequality is bounded by $
\CNK N_K^2 M_K$.

This proves the desired estimate of the $H^\beta$-norm of 
$\ZWGamma{\alpha,n}^2_2$. To estimate the $H^\beta$-norm of 
$\ZWGamma{\alpha,n}^1_2$, it remains to consider the case when 
$j\in J\setminus J'$, that the case $j=(\alpha_1,\alpha_2,n_1,n_2,n_3)=(\alpha,0,n,0,0)$. 
Let us study this term, together with $\ZWGamma{\alpha,n}^4$. Here we notice 
that the $H^\beta$-norm of 
$\EA{0}\bigl((\partial_t +D)\px^\alpha Z^n \vu\bigr)\vU$ (resp.\ 
$E^\sharp\bigl((\partial_t +D)\px^\alpha Z^n \vu\bigr)\vU$) 
is estimated by means of \e{n554} (resp.\ \e{n565}) and $i)$ in 
Lemma~\ref{T75}.

The $H^\beta$-norm of $\ZWGamma{\alpha,n}^3_2$ is estimated by similar arguments.

\underline{Estimate of $\ZWGamma{\alpha,n}^5$}

We claim that
$$
\blA \ZWGamma{\alpha,n}^5\brA_{H^\beta}
\le C(\lA u\rA_{\eC{\gamma}})\lA u\rA_{\eC{\gamma}}^2
X_{(\alpha,n)}.
$$
To see this we use the estimate \e{n565} which implies that 
$$
\blA \ZWGamma{\alpha,n}^5\brA_{H^\beta}
\les \lA \partial_t \vU+D\vU\rA_{\eC{\tdm}}
\blA \px^\alpha Z^n\vu\brA_{H^{\beta-\mez}}.
$$
Since $\blA \px^\alpha Z^n\vu\brA_{H^{\beta-\mez}}\le X_{(\alpha,n)}$ 
by definition~\e{n617a} of $X_{(\alpha,n)}$, it remains only to prove that
\be\label{n810}
\lA \partial_t \vU+D\vU\rA_{\eC{\tdm}}\le C(\lA u\rA_{\eC{\gamma}})\lA u\rA_{\eC{\gamma}}^2.
\ee
Since $\lA \partial_t \vu+D\vu\rA_{\eC{\tdm}}\le C(\lA u\rA_{\eC{\gamma}})\lA u\rA_{\eC{\gamma}}^2$ (as already seen in \e{n269.1}), remembering \e{n616-a}, 
to prove \e{n810} it is sufficient to prove that
\be\label{n812}
\lA \partial_t T_{\sqrt{a}-1}\eta\rA_{\eC{\tdm}}
+\blA \partial_t \Dxmez T_B \eta\brA_{\eC{\tdm}}
\le C(\lA u\rA_{\eC{\gamma}})\lA u\rA_{\eC{\gamma}}
\ee
and
\be\label{n813}
\blA \Dxmez T_{\sqrt{a}-1}\eta\brA_{\eC{\tdm}}
+\lA \Dx T_B \eta\rA_{\eC{\tdm}}
\le C(\lA u\rA_{\eC{\gamma}})\lA u\rA_{\eC{\gamma}}.
\ee
The second estimate is obvious: For any $r>0$, it follows 
from \e{esti:Dxmez-Crho} and \e{esti:Dxpsiz0} that 
\begin{align*}
&\blA \Dxmez T_{\sqrt{a}-1}\eta\brA_{\eC{\tdm}}\les 
\blA T_{\sqrt{a}-1}\eta\brA_{\eC{2+r}}\les 
\lA \alpha\rA_{L^\infty}\lA \eta\rA_{\eC{2+r}},\\
&\lA \Dx T_B \eta\rA_{\eC{\tdm}}\les 
\blA T_{\B}\eta\brA_{\eC{\frac{5}{2}+r}}\les 
\lA \B\rA_{L^\infty}\lA \eta\rA_{\eC{\frac{5}{2}+r}},
\end{align*}
so \e{n813} follows from the estimate 
$\lA \alpha\rA_{L^\infty}+\lA \B\rA_{L^\infty}\le 
C(\lA u\rA_{\eC{\gamma}})\lA u\rA_{\eC{\gamma}}$ 
(see \e{n203} and \e{n189}). 

Let us prove~\e{n812}. In view of \e{esti:quant0} we have 
\begin{align*}
&\lA \partial_t T_{\sqrt{a}-1}\eta\rA_{\eC{\tdm}}
+\blA \partial_t \Dxmez T_B \eta\brA_{\eC{\tdm}}\\
&\qquad \qquad \les \Bigl( \lA \sqrt{a}-1\rA_{L^\infty}
+\lA \partial_t (\sqrt{a}-1)\rA_{L^\infty} 
+ \lA B\rA_{L^\infty}
+\lA \partial_t B\rA_{L^\infty} \Bigr)
\Bigl( \lA \eta\rA_{\eC{3}}+\lA \partial_t \eta\rA_{\eC{3}}\Bigr).
\end{align*}
Since $\partial_t \eta=G(\eta)\psi$, it follows from \e{211-1} that 
$\lA \partial_t \eta\rA_{\eC{3}}\le C(\lA u\rA_{\eC{\gamma}})\lA u\rA_{\eC{\gamma}}$. 
On the other hand, \e{n203} and \e{211-1}Êimply that 
$\lA \sqrt{a}-1\rA_{L^\infty}+\lA B\rA_{L^\infty}\le C(\lA u\rA_{\eC{\gamma}})\lA u\rA_{\eC{\gamma}}$. It remains only to prove that
$$
\lA \partial_t \alpha \rA_{L^\infty}+\lA \partial_t \B\rA_{L^\infty}
\le C(\lA u\rA_{\eC{\gamma}})\lA u\rA_{\eC{\gamma}}.
$$
Now notice that \e{n193}Ê
immediately implies that 
$\lA \partial_t a \rA_{L^\infty}\le C(\lA u\rA_{\eC{\gamma}})\lA u\rA_{\eC{\gamma}}$, which implies 
the wanted estimate for $\partial_t \alpha$ since $\alpha=\sqrt{a}-1$ and since 
$\ma$ is bounded from below by $1/2$ by assumption (see~\e{n194}). 
Now, to estimate $\partial_t \B$ we use that
$\partial_t \B=-V\px \B+\ma-1$ by definition of $\ma$, so
$\lA \partial_t \B\rA_{L^\infty}\le \lA V\rA_{L^\infty}\lA \px \B\rA_{L^\infty} 
+\lA \ma -1\rA_{L^\infty}$. The first term in the right hand side  is estimated by \e{211-1} while 
$\lA \ma -1\rA_{L^\infty}$ is estimated by means of \e{n192}. 
This completes the proof of the claim.

\underline{Estimate of $\ZWGamma{\alpha,n}^6$}

We divide the analysis into two cases: either $\alpha=0$ or $\alpha\ge 1$. 
If $\alpha\ge 1$, we decompose $J$ as $J''_1\cup J_2''\cup \{ j_{1},j_{2}\}$ where 
\be\label{n820}
\ba
&j_1=(\alpha,0,n,0,0), \quad 
j_2=(1,\alpha-1,0,n,0),\\
&J_1''=J_1\setminus  \{ j_{2}\} ,\quad 
J_2''=J_2\setminus  \{j_{1}\}.
\ea
\ee
If $\alpha=0$ we decompose $J$ as $J_1\cup J_2''\cup \{ j_{1}\}$. Below we consider the case $\alpha\ge 1$ and the 
proof for the case $\alpha=0$ will be included in this analysis since we shall use the assumption 
$\alpha\ge 1$ only to give sense to $j_2$.

The estimate for the sum over 
$J''_2$ is straightforward: It follows from 
\e{n539a} and \e{n554}Ê
that
\be\label{Gamma6-1}
\ba
&\blA \EA{n_{3}}(\px^{\alpha_1} Z^{n_{1}}\vu)\bigl(N(\vu)\px^{\alpha_2}Z^{n_2}\vU\bigr)\brA_{H^\beta}\\
&\qquad \les 
\blA \px^{\alpha_1} Z^{n_{1}}\vu\brA_{L^2}
\blA N(\vu)\px^{\alpha_2}Z^{n_2}\vU\brA_{\eC{\beta+2}}\\
&\qquad \le
C(\lA u\rA_{\eC{\gamma}})\lA u\rA_{\eC{\gamma}}
\blA \px^{\alpha_1} Z^{n_{1}}\vu\brA_{L^2}
\blA\px^{\alpha_2}Z^{n_2}\vU\brA_{\eC{\beta+3}}.
\ea
\ee
So \e{n611} and \e{n615}Ê
imply that
$$
\blA \EA{n_{3}}(\px^{\alpha_1} Z^{n_{1}}\vu)N(\vu)\px^{\alpha_2}Z^{n_2}\vU\brA_{H^\beta}
\le C\bigl(N_\rho^{(s_0)}\bigr)\bigl(N_\rho^{(s_0)}\bigr)^2 \mathcal{M}_K.
$$
Now $N(\vu)\px^{\alpha_2}\EA{n_{3}}(\px^{\alpha_1} Z^{n_{1}}\vu)Z^{n_2}\vU$ is estimated by 
parallel arguments. 
This obviously implies the wanted estimate for the commutator. 

If $(\alpha_1,\alpha_2,n_1,n_2,n_3)=j_1$ then $(\alpha_2,n_2)=(0,0)$. Thus, 
it follows from the first inequality in \e{Gamma6-1}, \e{n539a},
\e{n610} and the assumption $\beta\ge 1/2$ that
\begin{align*}
&\blA \EA{n_{3}}(\px^{\alpha_1} Z^{n_{1}}\vu)N(\vu)\px^{\alpha_2}Z^{n_2}\vU\brA_{H^\beta}\\
&\qquad \le C(\lA u\rA_{\eC{\gamma}})\lA u\rA_{\eC{\gamma}}^2\Bigl( Y_{(\alpha,n)}
+\mathcal{M}_K\Bigr).
\end{align*}
We estimate 
$N(\vu)\px^{\alpha_2}\EA{n_{3}}(\px^{\alpha_1} Z^{n_{1}}\vu)Z^{n_2}\vU$ by similar arguments. 
This obviously implies the wanted estimate for the commutator. 

If $\alpha \ge 1$ and 
$(\alpha_1,\alpha_2,n_1,n_2,n_3)=j_2$, we have to estimate 
$\blA \bigl[\EA{0}(\px \vu),N(\vu)\bigr]\px^{\alpha-1}Z^{n}\vU\brA_{H^\beta}$. 
We claim that
\be\label{claim:EA0}
\blA \bigl[\EA{0}(\px \vu),N(\vu)\bigr]\brA_{\Fl{H^{\beta+1}}{H^{\beta}}}\le 
C \lA u\rA_{\eC{\gamma}}^2.
\ee
Let us assume this claim. Then 
\be\label{claim:EA0-suite}
\blA \bigl[\EA{0}(\px \vu),N(\vu)\bigr]\px^{\alpha-1}Z^{n}\vU\brA_{H^\beta}
\le C \lA u\rA_{\eC{\gamma}}^2\blA \px^{\alpha-1}Z^{n}\vU\brA_{H^{\beta+1}}.
\ee
We then write that, obviously,
$$
\blA \px^{\alpha-1}Z^{n}\vU\brA_{H^{\beta+1}}
\le \blA \px^{\alpha}Z^{n}\vU\brA_{H^{\beta}}
+\blA \px^{\alpha-1}Z^{n}\vU\brA_{H^{\beta}}
\le X_{(\alpha,n)}+M_K.
$$

The proof of the claim~\e{claim:EA0} is then based on the following lemma. 

\begin{lemm}\label{T80}
Let $\mu$ be a given real number. 

$i)$ There exists $K>0$ such that, 
for any scalar function $w\in \eC{2}(\xR)$, any $v=(v^1,v^2)\in \eC{6}\cap L^2(\xR)$ 
and any $f=(f^1,f^2)\in H^\mu(\xR)$,
\be\label{n552-bis}
\lA \left[ T_w I_2 , \EA{0}(\px v)\right] f\rA_{H^\mu} 
\le K \lA w\rA_{\eC{1}} \lA v\rA_{\eC{6}}
\lA f\rA_{H^\mu},
\ee
where $I_2=\left(\begin{smallmatrix} 1 & 0 \\ 0 & 1\end{smallmatrix}\right)$.

$ii)$ There exists $K>0$ such that, 
for any $v=(v^1,v^2)$ in $\eC{5}\cap L^2(\xR)$ 
and any $f=(f^1,f^2)$ in $H^\mu(\xR)$,
\be\label{n553-bis}
\blA \EA{0}(\px v)f\brA_{H^{\mu-1}}
\le K \lA v\rA_{\eC{5}}\lA f\rA_{H^{\mu}}.
\ee
\end{lemm}
\begin{proof}
We recall that $\EA{0}(v)$ is given by 
$\Op^\Bony\bigl[ v^1, P_\ell^1\bigr]+\Op^\Bony\bigl[ v^2, P_\ell^2\bigr]$ where 
$P_\ell^1$ and $P_\ell^2$ belong to $S^{1,0}_{0}$ (see~\e{n549}). 
Therefore
$$
\EA{0}(\px v)=\Op^\Bony\bigl[ v^1, i\xip P_\ell^1\bigr]
+\Op^\Bony\bigl[ v^2, i\xip P_\ell^2\bigr].
$$
Since $i\xip P_\ell^1$ and $i\xip P_\ell^2$ belong to $S^{1,0}_{1}$, 
it follows from Lemma~\ref{T34} 
that $\EA{0}(\px v)$ is a paradifferential operator of order $1$, whose symbol 
has semi-norms estimated by means of statement $ii)$ in Lemma~\ref{T33}. 
The assertions in the lemma then follows from Theorem~\ref{theo:sc0}. 
\end{proof}

Next we proceed as in the proof of Lemma~\ref{T70}. Firstly, we introduce
$$
\widetilde{N}(\vu)=N(\vu)-T_{V}\px -T_\alpha D. 
$$
Directly from the definition~\e{n538} of $N(\vu)$, using \e{n539c} and \e{n603}, one can check 
that
$$
\blA \widetilde{N}(\vu)\brA_{\Fl{H^{\beta}}{H^\beta}}\le C \lA \vu\rA_{\eC{\gamma}},
$$
for some constant $C$ depending only on $\lA \vu\rA_{\eC{\gamma}}$. 
By combining this estimate with \eqref{n553-bis} we get 
$$
\blA \EA{0}(\px \vu)\widetilde{N}(\vu)\brA_{\Fl{H^{\beta+1}}{H^\beta}}
+\blA \widetilde{N}(\vu)\EA{0}(\px \vu) \brA_{\Fl{H^{\beta+1}}{H^\beta}}\le C \lA \vu\rA_{\eC{\gamma}}^2,
$$
which obviously implies that 
$$
\blA [\widetilde{N}(\vu),\EA{0}(\vu)]\brA_{\Fl{H^{\beta+1}}{H^\beta}}\le C \lA \vu\rA_{\eC{\gamma}}^2.
$$ 
So to prove \eqref{claim:EA0} it remains only 
to estimate the commutators of $\EA{0}(\px\vu)$ with 
$T_V\px$ and $T_\alpha D$. The commutator with $T_V\px$ 
is estimated by means of statement $i)$ in the above lemma. 
To estimate the commutator with $T_\alpha D$ we use again statement $i)$ in the above lemma to estimate 
the commutator $[T_\alpha,\EA{0}]$ and we use 
the equation~\e{n551} 
satisfied by $\EA{0}(v)$ to estimate $[D,\EA{0}(\px \vu)]$: 
Indeed, \e{n551} implies that
$$
\bigl[ D,\EA{0}(\px \vu)\bigr]=\EA{0}(D\px \vu)+Q^{(0)}(\px\vu)
$$
and hence $[D,\EA{0}(\px\vu)]$ is an operator of order $1$ which is estimated by means of the estimate~\e{n540} and statement $ii)$ in the above lemma. 

Now let us assume that $j\in J_1''$. Then we claim that
\be\label{claim:EA0-ter}
\blA \bigl[\EA{n_3}(\px^{\alpha_1}Z^{n_1}\vu),N(\vu)\bigr]\brA_{\Fl{H^{\beta+1}}{H^{\beta}}}\le 
C N_K^2.
\ee
This is proved exactly as we proved \e{claim:EA0}, excepted that we use 
Lemma~\ref{T67}Ê
instead of Lemma~\ref{T80}. 
Then \e{claim:EA0-ter} implies that
\be\label{claim:EA0-suite2}
\ba
&\blA \bigl[\EA{n_3}(\px^{\alpha_1}Z^{n_1}\vu),N(\vu)\bigr]\px^{\alpha_2}Z^{n_2}\vU\brA_{H^\beta}\\
&\qquad\qquad\le C N_K^2\blA \px^{\alpha_2}Z^{n_2}\vU\brA_{H^{\beta+1}}.
\ea
\ee
We then write that, as already seen, if $(\alpha_2,n_2)\prec (\alpha,n)$, 
$\alpha_2\le \alpha$, and 
$(\alpha_2,n_2)\neq (\alpha-1,n)$ then $(\alpha_2+1,n_2)\prec (\alpha,n)$ so
$$
\blA \px^{\alpha_2}Z^{n_2}\vU\brA_{H^{\beta+1}}
\le \blA \px^{\alpha_2}Z^{n_2}\vU\brA_{H^{\beta}}+\blA \px^{\alpha_2+1}Z^{n_2}\vU\brA_{H^{\beta+1}}\le M_K.
$$
This completes the proof.

\underline{Estimate of $\ZWGamma{\alpha,n}^7$}

Remembering the estimate (see \e{n539}) 
$\blA N(\vu)\brA_{\Fl{H^{\beta+1}}{H^\beta}}
\le C(\lA \vu\rA_{\eC{\gamma}})\lA \vu\rA_{\eC{\gamma}}$, 
we have
$$
\blA \ZWGamma{\alpha,n}^7\brA_{H^\beta}\le 
C(\lA \vu\rA_{\eC{\gamma}})\lA \vu\rA_{\eC{\gamma}} 
\blA E^\sharp(\px^\alpha Z^n\vu)\vU\brA_{H^{\beta}}.
$$
Now \e{n565} implies that
$$
\blA \ZWGamma{\alpha,n}^7\brA_{H^\beta}\le 
C(\lA \vu\rA_{\eC{\gamma}})\lA \vu\rA_{\eC{\gamma}} 
\lA \vU\rA_{\eC{\tdm}}\blA \px^\alpha Z^n\vu\brA_{H^{\beta-\mez}}.
$$
By definition of $X_{(\alpha,n)}$ there holds 
$\blA \px^\alpha Z^n\vu\brA_{H^{\beta-\mez}}\le X_{(\alpha,n)}$. 
So the estimate~\e{n616} for $\lA \vU\rA_{\eC{\tdm}}$ implies that
$$
\blA \ZWGamma{\alpha,n}^7\brA_{H^\beta}\le 
C(\lA \vu\rA_{\eC{\gamma}})\lA \vu\rA_{\eC{\gamma}} ^2X_{(\alpha,n)}.
$$

The estimates for $\ZWGamma{\alpha,n}^8$, 
and $\ZWGamma{\alpha,n}^9_1$ are obtained by similar arguments to those used previously. 
Also, to estimate $\ZWGamma{\alpha,n}^9_2$, using \e{n554}, all we need to prove is that
$$
\lA \ZmGb{\alpha_2,n_2}'
+\ZmFb{\alpha_2,n_2}''\rA_{\eC{4}}\le C\bigl( N_\rho^{(s_0)}\bigr)N_\rho^{(s_0)}.
$$
Here one notices that, while it could be long to estimate these terms separately, one can readily estimate 
the sum writing that, by \e{n622},
$$
\ZmGb{\alpha_2,n_2}'+\ZmFb{\alpha_2,n_2}''=(\partial_t+D)\px^{\alpha_2}Z^{n_2}\vU 
+N(\vu)\px^{\alpha_2}Z^{n_2}\vU.
$$
The first term in the right hand side  is estimated by Lemma~\ref{T75} since, 
as we study $\ZWGamma{\alpha,n}^9_2$, the condition 
$\alpha_2+n_2\le s_0$ holds. The second term is estimated by means of 
\e{n539a} and \e{n615}. This completes the estimate of 
$\ZWGamma{\alpha,n}$.

To complete the proof of $i)$ of Proposition~\ref{T65}, we still need to estimate 
the $H^\beta$-norm of the term $(4)$ in \e{n595}. Moreover, we have to prove the 
bounds~\e{n514a}, \e{n514b} of statement $ii)$ of that proposition. 
These estimates will be deduced from 
the following result.

\begin{lemm}There holds
\be\label{n850}
\blA \px^\alpha Z^n\vU-\ZPhi{\alpha,n}\brA_{H^\beta}
\le C(\lA \vu\rA_{\eC{\gamma}})\lA \vu\rA_{\eC{\gamma}} X_{(\alpha,n)}
+C(N_\rho^{(s_0)})N_K M_K.
\ee
\end{lemm}
\begin{proof}
It follows from the definition \e{Change} of $\Phi$ and the definition~\e{n567} 
of $\ZWPhi{\alpha,n}$ that
\be\label{n851}
\ba
\ZPhi{\alpha,n}- \px^\alpha Z^n\vU&\defn E(\vu)\px^\alpha Z^n\vU\\
&\quad
-\sum_{\Jalpha}m(j)\EA{n_{3}}(\px^{\alpha_1} Z^{n_{1}} \vu)\px^{\alpha_2} Z^{n_{2}}\vU \\
&\quad-\sum_{\Jpr}m(j)\ER{n_{3}}(\px^{\alpha_1} Z^{n_{1}} \vu)\px^{\alpha_2} Z^{n_{2}}\vU\\
&\quad-E^\sharp(\px^\alpha Z^n\vu)\vU.
\ea
\ee
Then we use arguments similar to those used previously. 
The first (resp.\ last) term 
in the right hand side  of \e{n851} is estimated by means of \e{n581} (resp.\ \e{n565}). 
To estimate the second term in the right hand side  
of \e{n851}, we decompose $J$ as $J''_1\cup J_2''\cup \{ j_{1},j_{2}\}$ (see \e{n820}) 
and then use the estimates \e{n553} (for $j\in J_1''$), 
\e{n553-bis} (for $j=j_2$), 
\e{n554} (for $j\in J_2''\cup \{j_1\}$). The estimate of the third term 
in the right hand side  of \e{n851} is similar; we decompose $J'$ as $J_1'\cup J_2'$ where 
$J_1'$ and $J_2'$ are defined by \e{defi:J1J2} and we 
use the estimates \e{n561}Ê
and \e{n562} (since $j\neq j_1$ for $j\in J'$ and since $\alpha_2+n_2<\alpha +n$ for any $j\in J'$, the terms $\lA \mathcal{H}v\rA_{\eC{\rho}}$ and $\lA \mathcal{H}f\rA_{\eC{\rho}}$ which appear in \e{n561} and 
\e{n562} lead to terms which are estimated by means of $N_K$).
\end{proof}

This lemma and the estimates~\e{n579} (applied with some $\rho>1$), the operator 
norm estimates \e{n607} (resp.\ \e{n579}) for $B(\vu)$ (resp.\ 
$ \mathfrak{S}^\sharp(\vu)$ and $ \mathfrak{S}^\flat(\vu)$) 
readily imply the wanted estimate of the $H^\beta$-norm of the term $(4)$ which 
appears in \e{n595}. 

Let us prove~\e{n514a}--\e{n514b}. 
Recall that (see \e{n405} and \e{n617a}), by notations,
\begin{align*}
Y_{(\alpha,n)}&\defn 
\blA \px^\alpha Z^n \eta\brA_{H^\beta}+\blA  \Dxmez \px^\alpha Z^n \omega\brA_{H^\beta}
+\blA \Dxmez \px^\alpha Z^n  \psi\brA_{H^{\beta-\mez}},\\
X_{(\alpha,n)}&\defn 
\blA \px^{\alpha}Z^{n}\vU\brA_{H^\beta}
+\blA \px^{\alpha}Z^{n}\vu\brA_{H^{\beta-\mez}}.
\end{align*}
It is convenient to set
$$
A_{(\alpha,n)}\defn \Bigl(\blA \px^\alpha Z^n \eta\brA_{H^\beta}^2+\blA  \Dxmez \px^\alpha Z^n \omega\brA_{H^\beta}^2\Bigr)^\mez.
$$
Hereafter we denote by $\Crg$ (resp.\ $\Crs$) various constants depending only on $\lA \vu\rA_{\eC{\gamma}}$ (resp.\ $N_\rho^{(s_0)}$). 

It follows from \e{n610-d} that
\be\label{n901}
\blA \px^{\alpha}Z^{n}\vu\brA_{H^{\beta-\mez}}\le 
A_{(\alpha,n)}+\Crg \lA \vu\rA_{\eC{\gamma}}Y_{(\alpha,n)}+\Crs N_\rho^{(s_0)}\mathcal{M}_K.
\ee
Using the obvious inequalities
\be\label{n901a}
\blA \px^\alpha Z^n \eta\brA_{H^\beta}
+\blA  \Dxmez \px^\alpha Z^n \omega\brA_{H^\beta}\le 2A_{(\alpha,n)}
\ee
and
\be\label{n901b}
\ba
\blA \Dxmez \px^\alpha Z^n  \psi\brA_{H^{\beta-\mez}}
&\les \blA  \px^\alpha Z^n \Dxmez \psi\brA_{H^{\beta-\mez}} 
+ \sum_{n'<n}\blA  \px^{\alpha} Z^{n'} 
\Dxmez \psi\brA_{H^{\beta-\mez}}\\
&\les  \blA  \px^\alpha Z^n \vu\brA_{H^{\beta-\mez}} 
+ \mathcal{M}_K.
\ea
\ee
It follows from \e{n901} that
\be\label{n902}
Y_{(\alpha,n)}\le 3 A_{(\alpha,n)}+\Crg \lA \vu\rA_{\eC{\gamma}}Y_{(\alpha,n)}+
\Crs \big[ 1+N_\rho^{(s_0)}\big] \mathcal{M}_K.
\ee
Then for $\lA \vu\rA_{\eC{\gamma}}$ small enough we have
\be\label{n903}
Y_{(\alpha,n)}\le 4 A_{(\alpha,n)}+\Crs \big[ 1+N_\rho^{(s_0)}\big]\mathcal{M}_K.
\ee

On the other hand \e{n610-c} implies that
\be\label{n904}
A_{(\alpha,n)}\le \blA \px^{\alpha}Z^{n}\vU\brA_{H^{\beta}}+ 
\Crg \lA \vu\rA_{\eC{\gamma}}Y_{(\alpha,n)}+\Crs N_\rho^{(s_0)}\mathcal{M}_K,
\ee
and using \e{n850} to estimate $\blA \px^{\alpha}Z^{n}\vU\brA_{H^{\beta}}$ by means of 
$\lA \Phi\rA_{H^\beta}$, we find that
\be\label{n905}
A_{(\alpha,n)}\le \lA \Phi\rA_{H^\beta}+ 
\Crg \lA \vu\rA_{\eC{\gamma}}\bigl(Y_{(\alpha,n)}+X_{(\alpha,n)}\bigr)+\Crs N_K M_K.
\ee
Using the first bound in \e{n617d} to estimate $X_{(\alpha,n)}$ in the right hand side  of the previous inequality, 
we find that
\be\label{n906}
A_{(\alpha,n)}\le \lA \Phi\rA_{H^\beta}+ 
\Crg \lA \vu\rA_{\eC{\gamma}}Y_{(\alpha,n)}+\Crs N_K M_K.
\ee
Then \e{n903} and \e{n906} imply that
\be\label{n907}
Y_{(\alpha,n)}\le 4 \lA \Phi\rA_{H^\beta}
+\Crg \lA \vu\rA_{\eC{\gamma}}Y_{(\alpha,n)}+\Crs \big[1+N_K\big] M_K,
\ee
and hence, provided that $\Crg \lA \vu\rA_{\eC{\gamma}}$ is small enough, 
\be\label{n908}
Y_{(\alpha,n)}\le 5\lA \Phi\rA_{H^\beta}+\Crs  \big[1+N_K\big] M_K.
\ee
Finally, it follows from \e{n617d} and \e{n617e} 
that the same inequality holds 
with $M_K$ (resp.\ $N_K$) replaced by $\mathcal{M}_K$ (resp.\ $\mathcal{N}_K$):
\be\label{n908-b}
Y_{(\alpha,n)}\le 5\lA \Phi\rA_{H^\beta}+\Crs \big[1+\mathcal{N}_K\big] \mathcal{M}_K.
\ee
This establishes the first inequality of \e{n514a}.

Let us prove \e{n514b}. The estimate \e{n850} implies that
\be\label{n909}
\lA \Phi\rA_{H^\beta}\le \blA \px^{\alpha}Z^{n}\vU\brA_{H^\beta}
+\Crg  \lA \vu\rA_{\eC{\gamma}} X_{(\alpha,n)}+\Crs N_K M_K,
\ee
and the estimate \e{n610-c} implies that
\be\label{n910}
\blA \px^{\alpha}Z^{n}\vU\brA_{H^{\beta}}\le 
A_{(\alpha,n)}+\Crg \lA \vu\rA_{\eC{\gamma}}Y_{(\alpha,n)}+\Crs N_\rho^{(s_0)}\mathcal{M}_K.
\ee
Now \e{n617d} and \e{n908}Ê
imply that 
\be\label{n911}
X_{(\alpha,n)}\le \Crg \lA \Phi\rA_{H^\beta}+\Crs \big[1+N_K\big] M_K.
\ee
Then \e{n908-b}, \e{n909}, \e{n910}, and \e{n911}Ê
imply that
\be\label{n911a}
\lA \Phi\rA_{H^\beta}\le A_{(\alpha,n)}+\Crg \lA \vu\rA_{\eC{\gamma}}\lA \Phi\rA_{H^\beta}
+\Crs N_K M_K.
\ee
If $\Crg \lA \vu\rA_{\eC{\gamma}}$ is small enough, we conclude that
\be\label{n912}
\lA \Phi\rA_{H^\beta}\le 2 A_{(\alpha,n)}+\Crs N_K M_K,
\ee
so we have, according to \e{n617d} and \e{n617e},
\be\label{n912-b}
\lA \Phi\rA_{H^\beta}\le 2 A_{(\alpha,n)}+\Crs \mathcal{N}_K \mathcal{M}_K.
\ee
Using the obvious inequality $A_{(\alpha,n)}
\le Y_{(\alpha,n)}$, this yields the second estimate 
of \e{n514b} and hence completes the proof 
of \e{n514b}.

Next, we shall use \e{n912-b} at time $T_0$. 
Our goal is to deduce from this estimate that
\be\label{n912-f}
\lA \Phi\rA_{H^\beta}(T_0)\les M_s^{(s_1)}(T_0)
\ee
provided that $N_\rho^{(s_0)}(T_0)$ is small enough. 

Plugging \e{n903}Ê
into \e{n901} we find that
\be\label{n901-b}
\blA \px^{\alpha}Z^{n}\vu\brA_{H^{\beta-\mez}}\le 
\bigl(1+4\Crg \lA \vu\rA_{\eC{\gamma}}\bigr)A_{(\alpha,n)}+\Crs N_\rho^{(s_0)}\mathcal{M}_K.
\ee
Obviously, we have
\be\label{n903a}
A_{(\alpha,n)}\le 
\sum_{(\alpha',n')\in \mathcal{P}} A_{(\alpha',n')}\le M_s^{(s_1)}
\ee
by definition~\e{n403} of the norm $M_s^{(s_1)}$ (recall that $\mathcal{P}$ 
is defined by \e{n404}). 
Similarly, 
$$
\mathcal{M}_K\le \sum_{K'=0}^{\#\mathcal{P}} \mathcal{M}_{K'}
\le M_s^{(s_1)}+\sum_{(\alpha',n')\in \mathcal{P}}
\blA \Dxmez \px^{\alpha'}Z^{n'}\psi\brA_{H^{\beta-\mez}}.
$$
Now, since $\Dxmez Z^{n'} \psi$ is a linear combination of terms of $Z^k\Dxmez \psi$, 
$0\le k\le n'$, we have
$$
\sum_{(\alpha',n')\in \mathcal{P}}\blA \Dxmez 
\px^{\alpha'}Z^{n'}\psi\brA_{H^{\beta-\mez}}
\les \sum_{(\alpha',n')\in \mathcal{P}}\blA  
\px^{\alpha'}Z^{n'}\Dxmez\psi\brA_{H^{\beta-\mez}}
\le \sum_{(\alpha',n')\in \mathcal{P}}\blA  
\px^{\alpha'}Z^{n'}\vu\brA_{H^{\beta-\mez}},
$$
since $u=(\eta,\Dxmez\psi)$ by definition of $\vu$. So the previous bound for $\mathcal{M}_K$ implies that
\be\label{n903b}
\mathcal{M}_K\les M_s^{(s_1)}+\sum_{(\alpha',n')\in \mathcal{P}}
\blA \px^{\alpha'}Z^{n'}\vu\brA_{H^{\beta-\mez}}.
\ee
Plugging \e{n903a} and \e{n903b}Ê
into \e{n901-b} we conclude that
\begin{align*}
\blA \px^{\alpha}Z^{n}\vu\brA_{H^{\beta-\mez}}
&\le 
\bigl(1+4\Crg \lA \vu\rA_{\eC{\gamma}}\bigr)A_{(\alpha,n)}+\Crs N_\rho^{(s_0)}\mathcal{M}_K\\
&\les \bigl(1+4\Crg \lA \vu\rA_{\eC{\gamma}}+\Crs N_\rho^{(s_0)}\bigr) M_s^{(s_1)} 
+\Crs N_\rho^{(s_0)} \sum_{(\alpha',n')\in \mathcal{P}}
\blA \px^{\alpha'}Z^{n'}\vu\brA_{H^{\beta-\mez}}.
\end{align*}
Since $\lA \vu\rA_{\eC{\gamma}}\le  N_\rho^{(s_0)}$, this simplifies to
$$
\blA \px^{\alpha}Z^{n}\vu\brA_{H^{\beta-\mez}}
\les \Sigma\defn  \bigl(1+\Crs N_\rho^{(s_0)}\bigr) M_s^{(s_1)}+\Crs N_\rho^{(s_0)} \sum_{(\alpha',n')\in \mathcal{P}}
\blA \px^{\alpha'}Z^{n'}\vu\brA_{H^{\beta-\mez}}.
$$
Taking the sum of the inequality thus obtained for $(\alpha,n)\in \mathcal{P}$, we conclude that
$$
\sum_{(\alpha,n)\in \mathcal{P}}\blA \px^{\alpha}Z^{n}\vu\brA_{H^{\beta-\mez}}\les 
\# \mathcal{P} \Sigma \les 
\bigl(1+\Crs N_\rho^{(s_0)}\bigr) M_s^{(s_1)}
+\Crs N_\rho^{(s_0)}\sum_{(\alpha',n')\in \mathcal{P}}
\blA \px^{\alpha'}Z^{n'}\vu\brA_{H^{\beta-\mez}}.
$$
So $a\defn \sum_{(\alpha,n)\in \mathcal{P}}\blA \px^{\alpha}Z^{n}\vu\brA_{H^{\beta-\mez}}$ satisfies 
\be\label{n920}
a \les 
\bigl(1+\Crs N_\rho^{(s_0)}\bigr) M_s^{(s_1)}
+\Crs N_\rho^{(s_0)} a. 
\ee
For $N_\rho^{(s_0)}(T_0)$ small enough, this yields that
$$
\sum_{(\alpha,n)\in \mathcal{P}}\blA \px^{\alpha}Z^{n}\vu(T_0)\brA_{H^{\beta-\mez}}\les M_s^{(s_1)}(T_0).
$$
Plugging this estimate in \e{n903b} and then \e{n903a} into \e{n912-b} we obtain the wanted estimate \e{n912-f} and hence the desired result~\e{n514b}.
 
This concludes the proof of the proposition.
\end{proof}

\chapter{Appendices}

\renewcommand{\thesection}{A.\arabic{section}}
\renewcommand{\theequation}{\thesection.\arabic{equation}}

\section{Paradifferential calculus}\label{s2}

We recall here some definitions and results concerning Bony's paradifferential calculus. We refer to the original 
articles of Bony \cite{Bony} and Meyer~\cite{Meyer} as well as to the books of H\"ormander \cite{Hormander}, M\'etivier \cite{MePise} and Taylor \cite{Taylor}. 

\smallbreak

We denote by~$\eC{0}(\xR)$ the space of bounded continuous functions. 
For any $\rho\in\xN$, we denote 
by~$\eC{\rho}(\xR)$\index{Function spaces!$\eC{\rho}(\xR)$, H\"older spaces} the space of~$\eC{0}(\xR)$ functions 
whose derivatives of order less or equal to $\rho$ are in~$\eC{0}$. 
For any $\rho\in ]0,+\infty[\setminus \xN$, we denote 
by~$\eC{\rho}(\xR)$ the 
space of bounded functions whose derivatives of order~$[\rho]$ 
are uniformly H\"older continuous with 
exponent~$\rho- [\rho]$.

\begin{defi}
Consider $\rho$ in $[0,+\infty[$ and $m$ in $\xR$. One denotes by $\Gamma_{\rho}^{m}(\xR)$ \index{Symbols!$\Gamma_{\rho}^m(\xR)$} the space of
locally bounded functions~$a(x,\xi)$
on~$\xR\times(\xR\setminus 0)$,
which are~$C^\infty$ functions of $\xi$ outside the origin and
such that, for any~$\alpha\in\xN$ and any~$\xi\neq 0$, the function
$x\mapsto \partial_\xi^\alpha a(x,\xi)$ belongs to~$\eC{\rho}(\xR)$ and there exists a constant
$C_\alpha$ such that,
\begin{equation}\label{para:10}
\forall\la \xi\ra\ge \mez,\quad \lA \partial_\xi^\alpha a(\cdot,\xi)\rA_{\eC{\rho}}
\le C_\alpha
(1+\la\xi\ra)^{m-\la\alpha\ra}.
\end{equation}
\end{defi}

Given a symbol~$a$, to define 
the paradifferential operator~$T_a$ we need to introduce a cutoff function $\theta$.

\begin{defi}\label{defi:theta}
Fix $\theta \in C^\infty(\xR\times\xR)$ \index{Pseudo-differential operators!$\theta$, cut-off function} 
satisfying the three following properties. 
\begin{enumerate}[(i)]
\item \label{item:P1} There exists $\eps_1,\eps_2$ satisfying $0<2\eps_1<\eps_2<1/2$ such that
\begin{alignat*}{5}
\theta(\xip,\xii)&=1 \quad &&\text{if}\quad &&\la\xip\ra\le \eps_1\la \xii\ra \quad&&\text{and }
&&\la\xii\ra \ge 2,\\
\theta(\xip,\xii)&=0 \quad &&\text{if}\quad &&\la\xip\ra\geq \eps_2\la\xii\ra
\quad&&\text{or }
&&\la\xii\ra \le 1.
\end{alignat*}
\item \label{item:P2} For all $(\alpha,\beta)\in\xN^2$, there is $C_{\alpha,\beta}$ such that
$$
\forall (\xip,\xii)\in\xR^2,\quad 
\big\lvert \partial_{\xip}^{\alpha}\partial_{\xii}^\beta \theta(\xip,\xii)\big\rvert
\le C_{\alpha,\beta} (1+\la\xii\ra)^{-|\alpha|-|\beta|}.
$$
\item $\theta$ satisfies the following symmetry conditions:
\begin{equation}\label{sym:chipsi}
\theta(\xip,\xii)=\theta(-\xip,-\xii)=\theta(-\xip,\xii). 
\end{equation}
\end{enumerate}
\end{defi}
\begin{rema*}
Notice that $\theta(\xip,\xii)=0$ for $\la \xii\ra$ small enough. This choice 
(different from~\cite{MePise}) plays a key role in our analysis since we have to handle symbols 
which are homogeneous in $\xii$ and hence not regular for $\xii=0$. 
\end{rema*}

As an example, fix $\eps_1,\eps_2$ such that $0<2\eps_1<\eps_2<1/2$ and a function 
$f$ in $C^\infty_0(\xR)$ satisfying $f(t)=f(-t)$, $f(t)=1$ for $|t|\le 2\eps_1$ and 
$f(t)=0$ for $|t|\ge \eps_2$. Then set
$$
\theta(\xip,\xii)=(1-f(\xii))f\left(\frac{\xip}{\xii}\right) .
$$
Properties $(i)$, $(ii)$ and $(iii)$ are clearly satisfied. 

The paradifferential operator~$T_a$ with symbol $a$ is defined by
\index{Pseudo-differential operators! $T_a$, Paradifferential operator}
\begin{equation}\label{eq.para}
\widehat{T_a u}(\xi)=(2\pi)^{-1}
\int \theta(\xi-\eta,\eta)\widehat{a}(\xi-\eta,\eta)\widehat{u}(\eta)\, d\eta,
\end{equation}
where
$\widehat{a}(\theta,\xi)=\int e^{-ix \theta}a(x,\xi)\, dx$
is the Fourier transform of~$a$ with respect to $x$.

\begin{rema}\label{rema:real}
It follows from~\eqref{sym:chipsi} that, if~$a$ and~$u$ are real-valued functions, so is~$T_a u$.
\end{rema}
\begin{rema}\label{rema:cutoff}
One says that $\Theta=\Theta(\xip,\xii)$ is an admissible cut-off function 
if $\Theta$ satisfies the properties $\eqref{item:P1}$ and $\eqref{item:P2}$ in Definition~\ref{defi:theta}. 
All the results given in this appendix remain true for any admissible cut-off function (except Remark~\ref{rema:real}). 
\end{rema}

We shall use quantitative results from \cite{MePise}. 
To do so, we introduce the following semi-norms.
\begin{defi}
For~$m\in\xR$,~$\rho\ge 0$ and~$a\in \Gamma^m_{\rho}(\xR)$, we set
\begin{equation}\label{defi:norms}
M_{\rho}^{m}(a)= 
\sup_{\la\alpha\ra\le 2+\rho ~}\sup_{\la\xi\ra \ge 1/2~}
\lA (1+\la\xi\ra)^{\la\alpha\ra-m}\partial_\xi^\alpha a(\cdot,\xi)\rA_{\eC{\rho}(\xR)}.
\end{equation}
\end{defi}

The main features of symbolic calculus for paradifferential operators 
are given by the following theorem.
\begin{defi}\label{defi:order}
Let~$m$ in $\xR$.
An operator~$T$ is said of order~$\leo m$ if, for any~$\mu\in\xR$,
it is bounded from~$H^{\mu}$ to~$H^{\mu-m}$.
\end{defi}
\begin{theo}\label{theo:sc0}
Let~$m\in\xR$. 

$(i)$ If~$a \in \Gamma^m_0(\xR)$, 
then~$T_a$ is of order~$\leo m$. 
Moreover, for any~$\mu\in\xR$ there exists~$K>0$ such that
\begin{equation}\label{esti:quant1}
\lA T_a \rA_{\Fl{H^{\mu}}{H^{\mu-m}}}\le K M_{0}^{m}(a).
\end{equation}
$(ii)$ Let~$(m,m')\in\xR^2$ and~$\rho\in (0,+\infty)$. 
If~$a\in \Gamma^{m}_{\rho}(\xR), b\in \Gamma^{m'}_{\rho}(\xR)$ then 
$T_a T_b -T_{a\sharp b}$ is of order~$\leo m+m'-\rho$ where
\be\label{defi:sharp}
a\sharp b=
\sum_{\la \alpha\ra < \rho} \frac{1}{i^{\la\alpha\ra} \alpha !} \partial_\xi^{\alpha} a \partial_{x}^\alpha b.
\ee
Moreover, for any~$\mu\in\xR$ there exists $K>0$ such that
\begin{equation}\label{esti:quant2sharp}
\lA T_a T_b  - T_{a\sharp b}   \rA_{\Fl{H^{\mu}}{H^{\mu-m-m'+\rho}}}\le 
K M_{\rho}^{m}(a)M_{\rho}^{m'}(b).
\end{equation}

In particular, if~$\rho\in \pol 0,1]$, 
$a\in \Gamma^{m}_{\rho}(\xR), b\in \Gamma^{m'}_{\rho}(\xR)$ then
\begin{equation}\label{esti:quant2}
\lA T_a T_b  - T_{a b}   \rA_{\Fl{H^{\mu}}{H^{\mu-m-m'+\rho}}}\le 
K M_{\rho}^{m}(a)M_{\rho}^{m'}(b).
\end{equation}
$(iii)$ Let~$a\in \Gamma^{m}_{1}(\xR)$. Denote by 
$(T_a)^*$ the adjoint operator of~$T_a$ and 
by~$\overline{a}$ the complex-conjugated of~$a$. Then 
$(T_a)^* -T_{\overline{a}}$ is of order~$\leo m-1$. 
Moreover, for any~$\mu$ in $\xR$ there exists a constant~$K$ such that
\begin{equation}\label{esti:quant3}
\lA (T_a)^*   - T_{\overline{a}}   \rA_{\Fl{H^{\mu}}{H^{\mu-m+1}}}\le 
K M_{1}^{m}(a).
\end{equation}
\end{theo}
\begin{rema}\label{rema:quantpx}
One can improve \eqref{esti:quant1} by noting that the estimates for 
$T_a u$ involves only the norm of $\px u$ and not $u$ itself. 
Indeed, introduce~$\tilde{\kappa}=\tilde{\kappa}(\xi)$ such that 
$\tilde{\kappa}(\xi)=1$ for~$\la \xi\ra\ge 1/3$ and~$\tilde{\kappa}(\xi)=0$ for~$\la\xi\ra\le 1/4$. 
Then, by assumption on the cutoff function $\theta$, we 
have $T_a=T_a \tilde{\kappa}(D_x)$ and hence \eqref{esti:quant1} implies that
\begin{equation}\label{esti:quant0-px}
\lA T_a  u \rA_{H^{\mu-m}}
\le K M_{0}^{m}(a) \lA \partial_x u \rA_{H^{\mu-1}}.
\end{equation}
since $\lA \tilde{\kappa}(D_x) u \rA_{H^{\mu}}\le \lA \partial_x u \rA_{H^{\mu-1}}$. 
Similarly, \eqref{esti:quant2sharp} implies that
\be\label{esti:quant2sharp-px}
\lA T_a T_b -T_{a\sharp b}u \rA_{H^{\mu-m-m'+\rho}}
\le K M_{\rho}^{m}(a)M_{\rho}^{m'}(b)\lA \partial_x u \rA_{H^{\mu-1}}.
\ee
\end{rema}

If~$a=a(x)$ is a function of~$x$ only, then $T_a$ is called a paraproduct. 
\index{Pseudo-differential operators! $T_a$, paraproduct} 
It follows from \eqref{esti:quant1} that if~$a\in L^\infty(\xR)$ then~$T_a$ is an operator of order 
$0$, together with the estimate
\begin{equation}\label{esti:quant0}
\forall\sigma\in \xR,\quad \lA T_a u\rA_{H^\sigma}\les \lA a\rA_{L^\infty}\lA u\rA_{H^\sigma}.
\end{equation}
A paraproduct with an $L^\infty$-function acts on any H\"older space: for any $\rho$ in $\xR_+^*\setminus \xN$ we have
\begin{equation}\label{Tab:Crho}
\forall\sigma\in \xR,\quad \lA T_a u\rA_{\eC{\rho}}
\les \lA a\rA_{L^\infty}\lA u\rA_{\eC{\rho}}.
\end{equation}
If $a=a(x)$ and $b=b(x)$ are two functions then \e{defi:sharp} simplifies to 
$a\sharp b=ab$ and hence \e{esti:quant2sharp} implies that, for any $\rho>0$,
\begin{equation}\label{esti:quant2-func}
\lA T_a T_b  - T_{a b}   \rA_{\Fl{H^{\mu}}{H^{\mu-m-m'+\rho}}}\le 
K \lA a \rA_{\eC{\rho}}\lA b\rA_{\eC{\rho}},
\end{equation}
provided that $a$ and $b$ belong to $\eC{\rho}(\xR)$.

\begin{defi}
Given two functions~$a,b$ defined on~$\xR$ we define the remainder\index{Pseudo-differential operators!$\RBony(a,u)$, 
remainder} 
\begin{equation}\label{defi:RBony}
\RBony(a,u)=au-T_a u-T_u a.
\end{equation}
\end{defi}

We record here two estimates about the remainder~$\RBony(a,b)$ 
(see chapter 2 in~\cite{BCD}).

\begin{theo}
Let~$\alpha\in \xR_+$ and $\beta\in \xR$ be such that~$\alpha+\beta>0$. Then
\begin{align}
&\lA \RBony(a,u) \rA _{H^{\alpha + \beta-\frac{1}{2}}(\xR)} 
\leq K \lA a \rA _{H^{\alpha}(\xR)}\lA u\rA _{H^{\beta}(\xR)},\label{Bony} \\ 
&\lA \RBony(a,u) \rA _{H^{\alpha + \beta}(\xR)}
\leq K \lA a \rA _{\eC{\alpha}(\xR)}\lA u\rA _{H^{\beta}(\xR)}.\label{Bony3}
\end{align}
\end{theo}

We next recall a well-known property of products 
of functions in Sobolev spaces (see chapter 8 in \cite{Hormander}) 
that can be obtained from~\eqref{esti:quant0} 
and~\eqref{Bony3}: 
If~$u_1,u_2 \in H^{s}(\xR)\cap L^{\infty}(\xR)$ and $s> 0$ then
\begin{equation}\label{prtame}
\lA u_1 u_2 \rA_{H^{s}}\le K \lA u_1\rA_{L^\infty}\lA u_2\rA_{H^{s}}+K \lA u_2\rA_{L^\infty}\lA u_1\rA_{H^s}.
\end{equation}
Similarly, recall that, for~$s>0$ and $F\in C^\infty(\xC^N)$ such that~$F(0)=0$, 
there exists a non-decreasing function~$C\colon\xR_+\rightarrow\xR_+$ 
such that
\begin{equation}\label{esti:F(u)}
\lA F(U)\rA_{H^s}\le C\bigl(\lA U\rA_{L^\infty}\bigr)\lA U\rA_{H^s},
\end{equation}
for any~$U\in (H^s(\xR)\cap L^\infty(\xR))^N$. 

One has also the following result: 
for any $(r,\rho,\rho')\in [0,+\infty\por^3$ such that $\rho'>\rho\ge r$, there exists a constant $K>0$ such that
\begin{equation}\label{esti:Tba}
\lA T_a u\rA_{H^{\rho-r}}\le K \lA a\rA_{H^{-r}}\lA u\rA_{\eC{\rho'}}.
\end{equation}
One can use this estimate to study the regularity of the product $fg$ when $g$ is in some H\"older space. 
Writing $fg=T_fg +T_gf+\RBony(f,g)$, 
it follows from \eqref{esti:quant0}, \eqref{esti:Tba} applied 
with $r=0$ and \eqref{Bony3} that, for any real numbers $\rho'>\rho\ge 0$, 
the product is continuous from $H^\rho\times \eC{\rho'}$ to 
$H^\rho$. By duality, the estimate \eqref{pr:sz} is true for any $(\rho,\rho')\in \xR\times \xR_+$ such that $\rho'>|\rho|$. Therefore, 
\begin{equation}\label{pr:sz}
\forall (\rho,\rho')\in \xR\times \xR_+ \text{ such that }\rho'>|\rho|,\quad 
\lA fg\rA_{H^\rho}\les \lA f\rA_{H^\rho}\lA g\rA_{\eC{\rho'}}.
\end{equation}
The estimate is obvious for $\rho'=\rho\in\xN$. When $\rho'\not\in\xN$ one has to allow a small loss.

Here is a couple of identities which are used to simplify many expressions 
(see the proof of \eqref{n159}, \eqref{n160}, 
the proof of Lemma~\ref{T10} and the proof of Proposition~\ref{T22}).
\begin{lemm}\label{lemm:DxaDx}
For any function $a=a(x)$ in $L^\infty(\xR)$ 
and any function $u$ in $L^2(\xR)$, one has
\begin{align}
&\Dx T_a\Dx u +\px T_a\px u=0,\label{A:i1}\\
&\Dx T_a \px u-\px T_a \Dx u=0.\label{A:i2}
\end{align}
\end{lemm}
\begin{proof}
There holds
$$
\Dx T_{a}\Dx u+\partial_x T_{a}\partial_x u
=\frac{1}{(2\pi)^2}\int
e^{ix(\xip+\xii)} \widehat{a}(\xip)p(\xip,\xii)\widehat{u}(\xii) \,d\xip \, d\xii,
$$
with
$$
p(\xip,\xii)
=
\bigl( \la \xip+\xii\ra\la \xii\ra -(\xip+\xii)\xii \bigr)\theta(\xip,\xii),
$$
where $\theta(\xip,\xii)$ is as given by Definition~\ref{defi:theta}. 
Now, on the support of $\theta$ we have 
$\la \xip\ra\le \la \xii\ra$ and hence $(\xip+\xii)\xii\ge 0$. As a result $p=0$, which proves \eqref{A:i1}. 

Set $\Sigma\defn \Dx T_a \px u-\px T_a \Dx u$. Since 
$\px \Sigma = \Dx \bigl(\px T_a \px u+\Dx T_a \Dx u\bigr)$ the idenity \eqref{A:i1} implies that
$\px\Sigma=0$ and hence $\Sigma=0$ since $\Sigma\in H^{-2}(\xR)$. This proves \eqref{A:i2}.
\end{proof}

We also need a commutator estimate to control 
the commutator of $\Dx$ and a paraproduct. 
\begin{lemm}\label{Lemm:A7}
$(i)$ For any~$\mu\in \xR$ there exists a positive constant~$K$ such that 
for all~$a\in \eC{1}(\xR)$ and all~$f\in H^\mu(\xR)$, 
\begin{equation}\label{L312:1}
\blA T_a \Dx f -\Dx (T_a f) \brA_{H^\mu}\le K \lA a\rA_{\eC{1}}\lA f\rA_{H^\mu}.
\end{equation}

$(ii)$ For any~$\eps>0$ and any~$\sigma\in \pol 0,+\infty\por~$ there exists a positive constant~$K$ such that 
for all~$a\in \eC{1}(\xR)\cap H^{\sigma+1}(\xR)$ and all~$f\in \eC{1+\eps}(\xR)\cap H^\sigma(\xR)$, 
\begin{equation}\label{L312:2}
\blA a \Dx f -\Dx (af )\brA_{H^\sigma}\le K \lA a\rA_{\eC{1}}\lA f\rA_{H^\sigma}+K\lA f\rA_{\eC{1+\eps}}
\lA a \rA_{H^{\sigma+1}}.
\end{equation}

$(iii)$ For any~$\eps>0$ and any~$\sigma\in [1,+\infty\por~$ there exists a positive constant~$K$ such that 
for all~$a\in \eC{\sigma+1+\eps}(\xR)$ 
and all~$f\in H^\sigma(\xR)$, 
\begin{equation}\label{L312:iii}
\blA a \Dx f -\Dx (af )\brA_{H^\sigma}\le K \lA a\rA_{\eC{\sigma+1+\eps}}
\lA f\rA_{H^\sigma}.
\end{equation}
\end{lemm}
\begin{rema}
The estimate \eqref{L312:2} is not optimal (see \cite{LannesJFA} for sharp results).
\end{rema}
\begin{proof} To prove~\eqref{L312:1}, 
write~$\bigl[ T_a, \Dx \bigr]f=\left[T_{a},T_{|\xi|}\right]f+T_{a} (\Dx -T_{|\xi|})f 
-(\Dx -T_{|\xi|}) T_{a} f$, 
and use the bounds
\begin{alignat*}{2}
&\forall \sigma\in\xR,&
\lA \left[T_{a},T_{|\xi|}\right]g \rA_{H^{\sigma}}&\les 
\lA a\rA_{\eC{1}}\lA g\rA_{H^\sigma},\\
&\forall(\sigma,\sigma')\in\xR^2,\quad &
\lA \Dx g-T_{|\xi|}g\rA_{H^{\sigma'}}&\les 
\lA g\rA_{H^\sigma},
\end{alignat*}
where 
the first estimate follows 
from~\eqref{esti:quant2} applied with~$\rho=1$. 

To prove~\eqref{L312:2}, rewrite the commutator~$\bigl[ a, \Dx \bigr]f$ as
\begin{equation*}
\left[T_{a},\Dx \right]f+(a-T_{a})\Dx f-\Dx (a-T_{a})f.
\end{equation*}
The first term is estimated by~\eqref{L312:1}. To estimate 
the last two terms, we use the bound
\begin{equation}\label{L312:3}
\forall r \in\pol-1,+\infty\por,\quad
\blA a g-T_{a}g\brA_{H^{r+1}} 
\les 
\lA a\rA_{\eC{1}}\lA g\rA_{H^r}+\lA g\rA_{L^\infty}
\lA a\rA_{H^{r+1}},
\end{equation}
which follows from the paradifferential rules~\eqref{esti:quant0} and \eqref{Bony3} (by writing 
$a g-T_{a}g=T_{g}a+\RBony(a,g)$). By using~\eqref{L312:3} 
with~$r= \sigma$ or~$r=\sigma-1>-1$, we find that
\begin{align*}
&\blA \Dx (a f-T_{a}f)\brA_{H^{\sigma}} 
\le \blA (a f-T_{a}f)\brA_{H^{\sigma+1}} 
\les 
\lA a\rA_{\eC{1}}\lA f\rA_{H^\sigma}+\lA f\rA_{L^\infty}
\lA a\rA_{H^{\sigma+1}},\\
&\blA a \Dx f-T_{a}\Dx f\brA_{H^{\sigma}} 
\les 
\lA a\rA_{\eC{1}}\lA \Dx f\rA_{H^{\sigma-1}}+\lA \Dx f\rA_{L^\infty}
\lA a\rA_{H^{\sigma}}.
\end{align*}
Since~$\Dx$ is bounded from~$\eC{1+\eps}$ to~$\eC{0}$ 
(see \eqref{esti:Dxpsiz0}), 
this completes the proof of~\eqref{L312:2}.

To prove statement $(iii)$, notice that \eqref{esti:Tba} and \eqref{Bony3} imply that
$$
\lA a \Dx f -\Dx (af )- [T_a,\Dx]f\rA_{H^{\sigma}}\les 
\lA f\rA_{H^1}\lA a\rA_{\eC{\sigma+1+\eps}}.
$$
Then \eqref{L312:iii} follows from \eqref{L312:1}.
\end{proof}

In Chapters~\ref{S:22} and \ref{S:24}, 
when studying the quadratic normal forms, we see that 
there is a small divisor issue at low frequencies. To help the reader, we end this section with two pictures which describe the support properties of the function 
$\theta$ as well as the function $\zeta$ defined by 
$\zeta(\xip,\xii)=1-\theta(\xip,\xii)-\zeta(\xii,\xip)$ (so that 
the remainder $\RBony$ defined by \e{defi:RBony} is a bilinear Fourier multiplier 
with symbol $\zeta$).

\begin{figure}[ht]
\centerline{
\begin{tikzpicture}[scale=1]
\draw[->] (0,0) -- (0,1)  ;
\node at (0,0.7) [right] {$\xii$} ;
\draw[->] (0,0) -- (1,0) node[below] {$\xip$};
\draw (-4,0) -- (4,0) ;
\draw (0,-4) -- (0, 4) ;
\filldraw[fill=gray!20] (-2,4) -- (-0.5,1) -- (0.5,1) -- (2,4) ;
\filldraw[fill=gray!80] (-1.33,4) -- (-0.667,2) -- (0.667,2) -- (1.33,4) ;
\filldraw[fill=gray!40] (-2,-4) -- (-0.5,-1) -- (0.5,-1) -- (2,-4) ;
\filldraw[fill=gray!80] (-1.33,-4) -- (-0.667,-2) -- (0.667,-2) -- (1.33,-4) ;
\draw (-4,-4) -- (4,-4)  -- (4,4) -- (-4,4) --  (-4,-4) ;
\end{tikzpicture}
}
\caption{The support of the cut-off function $\theta(\xip,\xii)$ is in grey. 
The set of points $(\xip,\xii)$ where $\theta(\xip,\xii)=1$ is in darker grey.}
\end{figure}
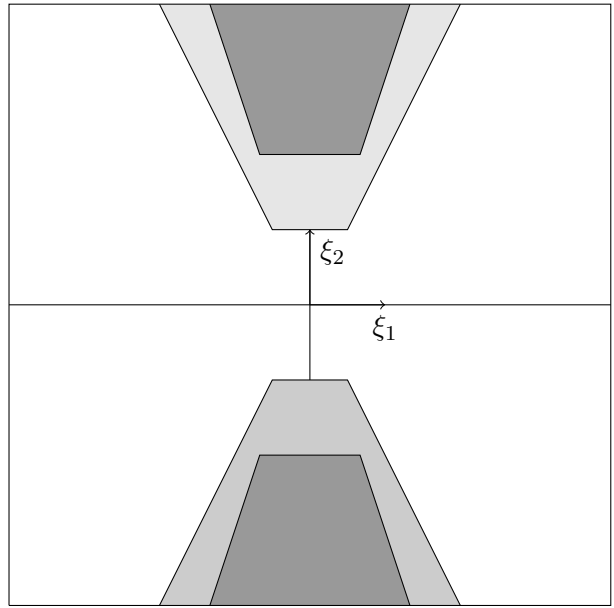

\begin{figure}[ht]
\centerline{
\begin{tikzpicture}[scale=1]
\filldraw[fill=gray!80] (-4,-4) -- (4,-4)  -- (4,4) -- (-4,4) --  (-4,-4) ;
\filldraw[fill=gray!40,draw=black] (-2,4) -- (-0.5,1) -- (0.5,1) -- (2,4) ;
\filldraw[fill=white,draw=black] (-1.33,4) -- (-0.667,2) -- (0.667,2) -- (1.33,4) ;
\filldraw[fill=gray!40,draw=black] (-2,-4) -- (-0.5,-1) -- (0.5,-1) -- (2,-4) ;
\filldraw[fill=white,draw=black] (-1.33,-4) -- (-0.667,-2) -- (0.667,-2) -- (1.33,-4) ;
\filldraw[fill=gray!40,draw=black] (4,-2) -- (1,-0.5) -- (1,0.5) -- (4,2) ;
\filldraw[fill=white,draw=black] (4,-1.33) -- (2,-0.667) -- (2,0.667) -- (4,1.33) ;
\filldraw[fill=gray!40,draw=black] (-4,-2) -- (-1,-0.5) -- (-1,0.5) -- (-4,2) ;
\filldraw[fill=white,draw=black] (-4,-1.33) -- (-2,-0.667) -- (-2,0.667) -- (-4,1.33) ;
\draw (-4,-4) -- (4,-4)  -- (4,4) -- (-4,4) --  (-4,-4) ;
\end{tikzpicture}
}
\caption{The support of $\zeta(\xip,\xii)=1-\theta(\xip,\xii)-\theta(\xii,\xip)$ is in grey. 
The set of points $(\xip,\xii)$ where $\zeta(\xip,\xii)=1$ is in darker grey.}
\end{figure}
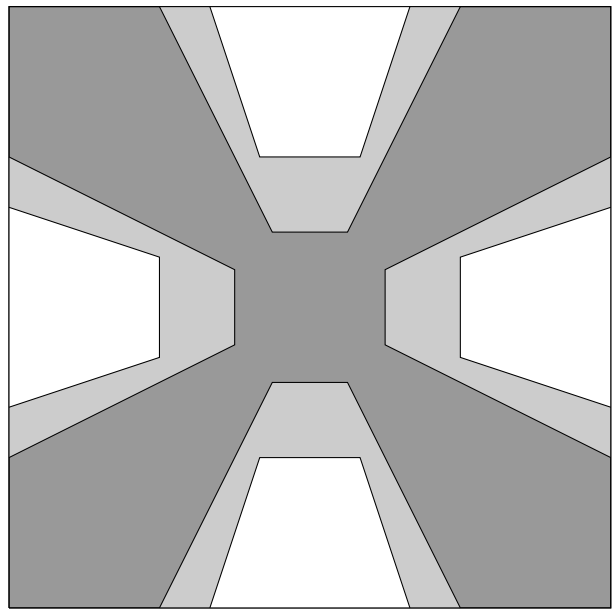

\clearpage

\section{Estimates in H\"older spaces}\label{S:A.3}
Here we gather H\"older estimates. 
It is convenient to work in the Zygmund spaces~$C^\rho_*$,~$\rho\in \xR$ whose definition is recalled here.

Choose a function~$\Phi\in C^\infty_0(\{\xi;\,\la\xi\ra\le 1\})$ which is equal to~$1$ when~$\la \xi\ra\le 1/2$ and set 
$\phi(\xi)=\Phi(\xi/2)-\Phi(\xi)$ which is supported in the annulus~$\{\xi;\, 1/2\le \la \xi\ra\le 2\}$. 
Then we have\index{Pseudo-differential operators!$\Delta_j$, Littlewood Paley decomposition} for~$\xi\in\xR$,
\begin{equation}\label{defi:LP}
1=\Phi(\xi)+\sum_{j=0}^{\infty} \phi(2^{-j}\xi),
\end{equation}
which we shall use to decompose temperate distributions. 
We set 
$\Delta_j=\phi(2^{-j}D_x)$ for~$j\in \xZ$. We also use in the paper the notation 
$S_0v$ instead of $\Phi(D_x)v$.
\begin{rema}\label{rema:LF}
For~$\mu\in\xR$, 
if~$u\in H^\mu(\xR)$ then the series~$\sum_{j=-\infty}^{-1}\Delta_ju$ 
converges to~$\Phi(D_x)u$.
\end{rema}
\begin{defi}[Zygmund spaces]\label{defi:Zygmund}
For any $s\in\xR$, we define 
$C_*^s(\xR)$\index{Function spaces!$C_*^s(\xR)$, Zygmund spaces} as the 
space of temperate distributions~$u$ such that
\be\label{a31}
\lA u\rA_{C_*^s}\defn \lA \Phi(D_x)u\rA_{L^\infty}+\sup_{j\ge 0 } 2^{js}\lA \Delta_j u\rA_{L^\infty}<+\infty.
\ee
\end{defi}
We recall the following result (see~\cite[Prop. $4.1.16$]{MePise}).
\begin{prop}\label{prop:Zygmund}
If $s>0$ and $s\not\in\xN$ then $C_*^s(\xR)=C^{s}(\xR)$ and the norms 
$\lA \cdot\rA_{C_*^s}$ and $\lA \cdot\rA_{\eC{s}}$ are equivalent. 
\end{prop}

\begin{prop}
$i)$ Let~$\gamma\in \pol 0,+\infty\por$ with~$\gamma\not\in \mez\xN$. 
There exists a constant~$K$ such that, for all~$z\le 0$ and all 
$f\in \C{\mez,\gamma+\mez}$ satisfying 
$\Dx f\in L^2(\xR)$, 
\begin{equation}\label{esti:Dxpsiz0}
 \blA \Dx f\brA_{\eC{\gamma}}\le K \blA \Dxmez f\brA_{\eC{\gamma+\mez}}.
\end{equation}
$ii)$ 
For all~$\gamma\in \pol 1/2,+\infty\por\setminus \mez\xN$, 
there exists $K>0$ such that, for all 
$f\in C^\gamma$ satisfying 
$\Dxmez f\in L^2(\xR)$, 
\begin{align}
\blA \Dxmez f\brA_{\eC{\gamma-\mez}}&\le K \blA f\brA_{\eC{\gamma}},
\label{esti:Dxmez-Crho}\\
\blA \Dx^{-\mez}\px f\brA_{\eC{\gamma-\mez}}&\le K \blA f\brA_{\eC{\gamma}}.
\label{esti:Dxmez-Crho-b}
\end{align}

$iii)$ Let $r\in \xR$, $\gamma\in \pol0,+\infty\por$ with $\gamma\not\in\xN$. 
There exists a constant $K$ such that, for any $\nu \in ]0,1]$ and 
for all~$f \in\eC{\gamma}(\xR)\cap H^{r}(\xR)$, there holds
\begin{equation}\label{lemm:Tpm2-b}
\lA \mathcal{H}f\rA_{\eC{\gamma}}\le K \Big[\lA f\rA_{\eC{\gamma}}
+\frac{1}{\nu}\lA f\rA_{\eC{\gamma}}^{1-\nu}\lA f\rA_{H^r}^\nu\Big].
\end{equation}
where $\mathcal{H}= \Dx^{-1}\px$ is the Hilbert transform.
\end{prop}
\begin{proof}
Let us prove \e{lemm:Tpm2-b}. 
Consider $j$ in $\xZ$ and a~$C^\infty$ function~$\tilde{\phi}$ with compact support 
such that~$\tilde{\phi}\equiv 1$ 
on the support of~$\phi$ and~$\tilde{\phi}\equiv 0$ on a neighborhood of the origin. 
Then 
\begin{align*}
\Delta_j   \mathcal{H} f&=
\frac{1}{2\pi}\int e^{ix\xi}\phi(2^{-j}\xi) \widehat{f}(\xi)
\la \xi\ra^{-1} (i\xi)\, d\xi\\
&=\frac{2^j}{2\pi}\int e^{i2^j(x-x')\xi}\tilde{\phi}(\xi) \Delta_j f(x') \la \xi\ra^{-1} (i\xi)\, dx' \,d\xi\\
&=2^j \int E(2^j(x-x'))\Delta_j f(x')\, dx'
\end{align*}
where
\begin{equation*}
E(y)=\frac{1}{2\pi}\int e^{iy\xi}\tilde{\phi}(\xi)  \la \xi\ra^{-1} (i\xi)\, d\xi.
\end{equation*}
Since $E(y)$ and $y^2 E(y)$ are bounded (using an integration by parts), 
we have $E\in L^1(\xR)$. 
This implies that
\be\label{a36}
\blA \Delta_j  \mathcal{H} f\brA_{L^\infty}\les \lA \Delta_j f\rA_{L^\infty},
\ee
and hence, using \e{a31}, for $\gamma$ in $]0,+\infty[\setminus \xN$, we have
\be\label{a37}
\sup_{j\ge 0}2^{j\gamma}
\blA \Delta_j  \mathcal{H} f\brA_{L^\infty}\les \sup_{j\ge 0} 2^{j\gamma} \lA \Delta_j f\rA_{L^\infty}\les \lA f\rA_{\eC{\gamma}}.
\ee

It remains to estimate the low frequency. Consider $j\le 0$. 
Using \e{a36}, the estimate 
$\lA \Delta_j u\rA_{L^\infty}
\les 2^{j/2}\lA \Delta_j u\rA_{L^2}$ and the fact that $\mathcal{H}$ is bounded on $L^2$, we get
\begin{align*}
\lA \Delta_j \mathcal{H}f\rA_{L^\infty}
&=\lA \Delta_j \mathcal{H}f\rA_{L^\infty}^{1-\nu}
\lA \Delta_j \mathcal{H}f\rA_{L^\infty}^\nu\\
&\les  2^{\nu j/2} \lA \Delta_j f\rA_{L^\infty}^{1-\nu}\lA \Delta_j f\rA_{L^2}^\nu.
\end{align*}
Now, for any $r\in\xR$ and $j\le 0$, 
$\lA \Delta_j f\rA_{L^2}\les \lA \Delta_j f\rA_{H^r}\le 
\lA f\rA_{H^r}$. Thus 
$$
\lA \Delta_j \mathcal{H}f\rA_{L^\infty}
\les  2^{\nu j/2} \lA f\rA_{\eC{\gamma}}^{1-\nu}\lA f\rA_{H^r}^\nu.
$$
Since $\sum_{j\le 0}2^{\nu j/2}=O(\frac{1}{\nu})$ it follows from Remark \ref{rema:LF} that 
$\lA 
\Phi(D_x)\mathcal{H}f\rA_{L^\infty}\les \lA f\rA_{\eC{\gamma}}^{1-\nu}\lA f\rA_{H^r}^\nu$. 
The wanted estimate \e{lemm:Tpm2-b} then follows from \e{a31} and \e{a37}.

The proof of \e{esti:Dxpsiz0}, \e{esti:Dxmez-Crho} and 
\e{esti:Dxmez-Crho-b} are similar. Let us prove \e{esti:Dxpsiz0}. 
For~$j$ in $\xZ$, write
$$
\Dx \Delta_j f=\int E_j(x-x') ( \Dxmez f)(x') \, dx'
$$
where
$$
E_j(y)=\frac{1}{2\pi}\int e^{iy\xi} \la\xi\ra^\mez \phi(2^{-j}\xi)\, d\xi
$$
satisfies 
$\lA E\rA_{L^1}\les 2^{j/2}$. Consequently,
\be\label{a38}
\blA \Dx \Delta_j f \brA_{L^\infty}\les 2^{j/2}
\blA \Dx^{\mez}\Delta_j f\brA_{L^\infty}.
\ee
Since~$\sum_{j\le 0} 2^{j/2}<+\infty$ (and using Remark~\ref{rema:LF}) we thus have 
\be\label{a39}
\blA \Phi(D_x) \Dx f\brA_{L^\infty}
\les \blA \Dx^{\mez} f\brA_{L^\infty}.
\ee
By combining \e{a38}, \e{a39} and using \e{a31} we obtain the wanted estimate.
\end{proof}

\section{Identities}\label{S:A.4}

Consider a smooth solution $(\eta,\psi)$ of the water waves system
\begin{equation}\label{system:A4}
\left\{
\begin{aligned}
&\partial_t \eta=G(\eta)\psi,\\
&\partial_t \psi + \eta+ \frac{1}{2}(\partial_x \psi)^2
-\frac{1}{2(1+(\partial_x\eta)^2)}\bigl(G(\eta)\psi+\partial_x  \eta \partial_x \psi\bigr)^2= 0.
\end{aligned}
\right.
\end{equation}
Set
\begin{equation*}
\B=\frac{G(\eta)\psi+\partial_x\eta \partial_x\psi}{1+(\partial_x\eta)^2},\quad
V=\partial_x\psi-(\B(\eta)\psi)\partial_x\eta,
\end{equation*}
and $\ma=1+\partial_t\B+V\px\B$. 
\begin{lemm}\label{L:A.4.1}
There hold
\begin{align}
&\partial_t \eta=\B-V\px\eta,\label{B2-0}\\
&\partial_t\psi+\eta+\mez V^2+\B V\partial_x \eta-\mez \B^2=0,\label{B2-1}\\
&\partial_t \psi-B\partial_t \eta=-\eta-\mez V^2-\mez B^2,\label{B2-2}\\
&\partial_t \psi+V\px\psi=-\eta+\mez V^2+\mez \B^2,\label{B2-3}\\
&\partial_t V+V\partial_x V+a\partial_x \eta=0.\label{B2-4}
\end{align}
\end{lemm}
\begin{proof}
The equation~\eqref{B2-0} follows from~$\B-V\px\eta=G(\eta)\psi$ (see~\eqref{212}). 
To prove~\eqref{B2-1}, we begin by noticing that
$$
(\partial_x\psi)^2=(V+\B\partial_x\eta)^2=V^2+\B^2 (\partial_x\eta)^2+2 \B V\partial_x\eta.
$$
Since
\begin{equation}\label{B2-10}
\frac{\left( \partial_x \eta \partial_x \psi + G(\eta)\psi \right)^2}{1+(\partial_x\eta)^2}
=(1+(\partial_x\eta)^2)\B^2,
\end{equation}
this yields
\begin{equation}\label{iden:eqpsi}
\mez (\partial_x \psi)^2 -\mez \frac{\left( \partial_x \eta \partial_x \psi 
+ G(\eta)\psi \right)^2}{1+(\partial_x\eta)^2} 
= \mez V^2 +\B V\partial_x \eta- \mez \B^2,
\end{equation}
so~\eqref{B2-1} follows from the second equation of \eqref{system:A4}. Now~\eqref{B2-2} 
can be verified by a direct calculation, 
using~$\partial_t\eta=\B-V\px\eta$ and~\eqref{B2-1}. In the same way, \e{B2-1} and the definition 
of $V$ imply \e{B2-3}.

To prove~\eqref{B2-4}, write
\begin{align*}
&\partial_t V +V\px V\\
&\qquad=(\partial_t +V\px)(\px\psi-\B\px\eta)\\
&\qquad=\px(\partial_t\psi +V\px\psi)-(\partial_t \B+V\px \B)\px\eta
-\B\px( \partial_t \eta+V\px\eta)+R,
\end{align*}
where $R=-\px V \px\psi+\B\px V\px\eta=-(\px V)V $. 
Then, it follows from \eqref{B2-0} and \eqref{B2-3} that
$$
\partial_t V+V\px V=\px\Bigl( -\eta+\mez V^2+\mez B^2\Bigr)
-(\partial_t \B+V\px\B)\px\eta-\B\px\B+ R.
$$
Now, observing that $R+1/2 \px V^2=0$ 
and simplifying,
$$
\partial_t V+V\px V 
+(1+\partial_t \B+V\px \B)\px\eta
=0,
$$
which completes the proof of \eqref{B2-4} since 
$\ma=1+\partial_t \B+V\px \B$.
\end{proof}
\begin{lemm}\label{lemm:B7}
There holds
\begin{equation}\label{formule:a}
\ma=\frac{1}{1+(\px\eta)^2}
\left(1+V\px \B - \B \px V-\mez G(\eta)V^2 -\mez G(\eta)\B^2-G(\eta)\eta\right).
\end{equation}
\end{lemm}
\begin{proof}
Starting from $B-V\partial_x\eta=G(\eta)\psi$ we have 
$$
\partial_t B -(\partial_tV)\partial_x\eta-V\partial_t\partial_x\eta
=\partial_t G(\eta)\psi.
$$
We then use the identity $\partial_t V+V\partial_x V+a\partial_x \eta=0$ 
(see~\eqref{B2-4})Ê
to obtain that
$$
\partial_t B+a(\partial_x \eta)^2 +V\partial_xV\partial_x\eta
-V\partial_t\partial_x\eta
=\partial_t G(\eta)\psi,
$$
and hence, using that by definition of $\ma$ we have $\partial_t\B=\ma-1-V\px\B$,
$$
(1+(\partial_x\eta)^2)a=1+\partial_t G(\eta)\psi
+V\partial_t\partial_x\eta +V\px\B-V\partial_xV\partial_x\eta.
$$
Now we have (see~\cite{LannesJAMS,LannesLivre} or 
the proof of Lemma~\ref{T44})
\begin{equation}\label{formule:dtDN}
\partial_tG(\eta)\psi=G(\eta)(\partial_t\psi-B\partial_t \eta)-\partial_x(V\partial_t\eta),
\end{equation}
and hence
$$
(1+(\partial_x\eta)^2)a=1+G(\eta)(\partial_t\psi-B\partial_t \eta)-(\px V)(\partial_t\eta+V\px\eta)+V\px\B.
$$
Since $\partial_t\eta+V\px\eta=\B$ (see~\eqref{B2-0}), to conclude it remains only to use \eqref{B2-2}. 
\end{proof}
\begin{rema}\label{rema:C3}
One can further simplify \eqref{formule:a} to obtain
\begin{equation*}
\ma=\frac{1}{1+(\px\eta)^2}
\left(1+V G(\eta)V +\B G(\eta)\B-\mez G(\eta)V^2 -\mez G(\eta)\B^2-G(\eta)\eta\right).
\end{equation*}
Indeed
\begin{equation}\label{GBGV}
G(\eta)\B=-\partial_x V,\quad G(\eta)V=\partial_x \B.
\end{equation}
We have already seen the first formula (see~\eqref{232e}) and the proof of the second is similar: 
it relies on the uniqueness result result of Proposition~\ref{ref:116} and 
the 
fact that 
$\px \phi$ is the harmonic extension of~$V=\partial_x\phi\arrowvert_{y=\eta}$. 
Therefore, by definition of 
the Dirichlet-Neumann operator, 
\begin{align*}
G(\eta)V&=\left(\partial_y \partial_x \phi
-\partial_x \eta\partial_x^2\phi\right)\big\arrowvert_{y=\eta}\\
&=\left(\partial_x\partial_y\phi +\partial_x \eta\partial_y^2\phi \right)
\big\arrowvert_{y=\eta}\\
&=\partial_x \left(\partial_y\phi\right)\big\arrowvert_{y=\eta}
\end{align*}
and hence~$G(\eta)V=\px \B$ since 
$\partial_y\phi$ is the harmonic extension 
of~$\B=\partial_y\phi\arrowvert_{y=\eta}$. 
\end{rema}

\section{Local existence results}\label{S:A4}

The goal of this appendix is to show that Proposition~\ref{ref:121} is just a restatement of 
Theorem~$4.35$ in the book of Lannes~\cite{LannesLivre}, and to prove also a local propagation of Sobolev estimates for the 
action of vector fields on a solution of the water waves equation.

To help the reader we recall the equations and the statement of Proposition~\ref{ref:121}. 
We consider the system  
\begin{equation}\label{a-121}
\left\{
\begin{aligned}
&\partial_t \eta=G(\eta)\psi,\\
&\partial_t \psi + \eta+ \frac{1}{2}(\partial_x \psi)^2
-\frac{1}{2(1+(\partial_x\eta)^2)}\bigl(G(\eta)\psi+\partial_x  \eta \partial_x \psi\bigr)^2= 0.
\end{aligned}
\right.
\end{equation} 

\begin{prop}\label{a-ref:121}
Let  $\gamma$ be in $]7/2,+\infty[\setminus \mez\xN$, $s\in \xN$ with $s>2\gamma-1/2$. 
There are $\delta_0>0$, $T>1$ such that for any couple $(\eta_0,\psi_0)$ 
in $H^{s}(\xR)\times \h{\mez,\gamma}(\xR)$ satisfying
\be\label{a-123}
\psi_0-T_{\B(\eta_0)\psi_0}\eta_0 \in \h{\mez,s}(\xR), \quad 
\lA \eta_0\rA_{\eC{\gamma}}+\blA \Dxmez \psi_0\brA_{\eC{\gamma-\mez}}<\delta_0,
\ee
equation \eqref{a-121} with Cauchy data $\eta\arrowvert_{t=1}=\eta_0$, $\psi\arrowvert_{t=1}
=\psi_0$ has a unique solution $(\eta,\psi)$ which is continuous on $[1,T]$ with values in 
\be\label{a-124}
\left\{\, (\eta,\psi)\in H^s(\xR)\times \h{\mez,\gamma}(\xR)\,;\, \psi-
T_{B(\eta)\psi}\eta\in \h{\mez,s}(\xR)\,\right\}.
\ee
Moreover, if the data are $O(\eps)$ on the indicated spaces, then $T\ge c/\eps$.
\end{prop}
\begin{proof}
Let us check that the assumptions of 
Theorem~$4.35$ of Lannes~\cite{LannesLivre}Ê
are satisfied under the hypothesis of Proposition~\ref{a-ref:121}. 
We have to check that the finiteness of the quantity $(4.66)$ of \cite{LannesLivre} with 
$t_0=\gamma-3/2$, $N=s$, which may be written
\be\label{a41}
\blA \Dxmez \psi_0\brA_{H^\gamma}^2+\lA \eta_0\rA_{H^s}^2+\sum_{|\alpha|\le s}
\blA \Dxmez \big(\px^\alpha\psi_0-(\B(\eta_0)\psi_0)\px^\alpha \eta_0\big)\brA_{L^2}
\ee
is actually equivalent to 
\be\label{a42}
\psi_0\in \h{\mez,\gamma},\quad \eta_0\in H^s,\quad \psi_0-T_{\B(\eta_0)\psi_0}\eta_0\in \h{\mez,s}.
\ee
We also need to verify assumption $(4.69)$ of \cite{LannesLivre}, which follows from the inequality
\be\label{a43}
\lA a_0-1\rA_{L^\infty}<\mez
\ee
where $a_0$ is given by \e{n190} (see also \e{n191}) with $(\eta,\psi)$ replaced by 
$(\eta_0,\psi_0)$. Let us write for $|\alpha|\le s$, setting 
$B_0=\B(\eta_0)\psi_0$, 
\be\label{a44}
\ba
\px^\alpha\big[ \Dxmez \big(\psi_0-T_{\B_0}\eta_0\big)\big]&=\Dxmez \big( (\px^\alpha \psi_0)-\B_0(\px^\alpha \eta_0)\big)
-\Dxmez \big[ \px^\alpha,T_{\B_0}\big]\eta_0\\
&\quad +\Dxmez \big( T_{\px^\alpha\eta_0}\B_0+\RBony(\px^\alpha\eta_0,\B_0)\big).
\ea
\ee

Both assumptions \e{a41} and \e{a42} imply that $\Dxmez\psi_0$ belongs to 
$\eC{\gamma-\mez}$ and that $\px\eta_0$ is in $\eC{\gamma-1}$, so that 
by \e{1139}, $G(\eta_0)\psi_0$ and so 
$\B_0$ are in $\eC{\gamma-1}$. Since $\gamma-1>1$, the symbolic calculus 
of appendix \ref{s2} shows that $[\px^\alpha,T_{B_0}]$ sends $H^s$ to $H^{s-\alpha+1}\subset H^{1/2}$ for 
$\alpha\le s$, so that the commutator term in \e{a44} belongs to $L^2$. 
The boundeness properties of the remainder given in \e{Bony3} show in the same way that 
$\RBony(\px^\alpha\eta_0,\B_0)$ is in $H^{1/2}$ if 
$\px^\alpha \eta_0$ is in $L^2$. The equivalence between \e{a41} and \e{a42} will follow if we show that 
$T_{\px^\alpha\eta_0}\B_0$ belongs to $H^{1/2}$, which follows from \e{esti:Tba} 
and the fact that $\B_0$ is in $\eC{\gamma-1}$. 
Finally, notice that \e{a43} follows from \e{n192} applied with $\eta,\psi$ replaced by $\eta_0,\psi_0$.
\end{proof}

\begin{prop}\label{ref:A4}Assume that 
$s$ and $\gamma$ are such that 
$$
\gamma\in ]7/2,+\infty[\setminus \mez\xN,\quad s> 2\gamma-\mez.
$$
Consider a solution $(\eta,\psi)$ of \e{a-121}, defined on the time interval $[T_0,T_1]$, which is continuous on $[T_0,T_1]$ 
with values in \e{a-124} and such that the Taylor coefficient 
is bounded from below by a positive constant. 
Assume that, at time $T_0$, $(\eta_0,\psi_0)=(\eta,\psi)\arrowvert_{t=T_0}$ satisfies
\be\label{a-126}
\begin{aligned}
&(x\px)\eta_0\in H^{s-1}(\xR),\quad (x\px)\psi_0\in \h{\mez,s-\tdm}(\xR),\\
&(x\px)\bigl( \psi_0-T_{\B(\eta_0)\psi_0}\eta_0\bigr)\in \h{\mez,s-1}(\xR).
\end{aligned}
\ee
Then 
\be\label{a-127}
\begin{aligned}
&Z\eta\in C^0([T_0,T_1];H^{s-1}(\xR)),
\quad Z\psi\in C^0([T_0,T_1];\h{\mez,s-\tdm}(\xR)),\\
&Z\bigl( \psi-T_{\B(\eta)\psi}\eta\bigr)\in C^0([T_0,T_1];\h{\mez,s-1}(\xR)).
\end{aligned}
\ee
\end{prop}
\begin{proof}
Since the equations~\e{a-121} are invariant by translation in time, 
we can assume without loss of generality that 
$T_0=0$. 

The proof is based on the analysis in Chapter~\ref{S:21} and the following observations: 
\begin{itemize}
\item If $(\eta,\psi)$ solves \e{a-121}, then the functions $\eta_\lambda$ and 
$\psi_\lambda$ defined by
\be\label{a40.1}
\eta_\lambda(t,x)=\lambda^{-2}\eta\left(\lambda t, \lambda^2 x\right),\quad \psi_\lambda(t,x)=\lambda^{-3}\psi\left(\lambda t,\lambda^2 x\right)
\qquad (\lambda>0)
\ee
are also solutions of \e{a-121}.
\item For any function $C^1$ function $u$, there holds
\be\label{a40.2}
Z u (t,x)=\frac{d}{d\lambda}u(\lambda t,\lambda^2 x)\Big\arrowvert_{\lambda=1}.
\ee
\item A bootstrap argument: It is sufficient to prove that there exists $T>\tzero$, depending only on 
$M_s$ defined by
\be\label{a55.7}
M_s\defn \sup_{t\in [\tzero,T_1]}
\Big[\lA \eta(t)\rA_{H^{s}}
+\blA \Dxmez \psi(t)\brA_{H^{s-\mez}}
+\blA \Dxmez\omega(t)\brA_{H^{s}}\Big],
\ee
such that 
\be\label{a-127-b}
\begin{aligned}
&Z\eta\in C^0([\tzero,T];H^{s-1}(\xR)),
\quad Z\psi\in C^0([\tzero,T];\h{\mez,s-\tdm}(\xR)),\\
&Z\bigl( \psi-T_{\B(\eta)\psi}\eta\bigr)\in C^0([\tzero,T];\h{\mez,s-1}(\xR)).
\end{aligned}
\ee
Let us explain why it is sufficient to prove \e{a-127-b}. Using the equations satisfied by 
$\eta$, $\psi$ and $\psi-T_{\B(\eta)\psi}\eta$ (see~\e{a-121} and the second equation of \e{a-196}) 
it is easily seen that 
\be\label{a-127-c}
\begin{aligned}
&\partial_t\eta\in C^0([\tzero,T_1];H^{s-1}(\xR)),
\quad \partial_t\psi\in C^0([\tzero,T_1];\h{\mez,s-\tdm}(\xR)),\\
&\partial_t\bigl( \psi-T_{\B(\eta)\psi}\eta\bigr)\in C^0([\tzero,T_1];\h{\mez,s-1}(\xR))
\end{aligned}
\ee
(and hence the same result holds with $\partial_t$ replaced by $t\partial_t$). Since 
$x\px=\mez (Z-t\partial_t)$, it follows from \e{a-127-b} and \e{a-127-c} (evaluated at time $T$) that 
\begin{align*}
&(x\px)\eta(T)\in H^{s-1}(\xR),\quad (x\px)\psi(T)\in \h{\mez,s-\tdm}(\xR),\\
&(x\px)\bigl( \psi(T)-T_{\B(\eta(T))\psi(T)}\eta(T)\bigr)\in \h{\mez,s-1}(\xR).
\end{align*}
Since the system \e{a-121} is invariant by translation in time, this means that we can apply the previous result 
with initial data at time $T$ instead of $\tzero$. This yields that \e{a-127-b}Ê
remains true when $[\tzero,T]$ is replaced by 
$[\tzero,\min(2T,T_1)]$. Iterating this reasoning, we obtain \e{a-127}.
\end{itemize}

We begin the proof by fixing some notations and explaining its strategy.

{\sc Notations}

Recall that, given two functions $\eta$ and $\psi$ we use the notations
\be\label{a40.3}
B=\frac{G(\eta)\psi+(\px \eta)\px \psi}{1+(\px\eta)^2},\quad 
V=\px\psi-B \px \eta,\quad 
\omega=\psi-T_{\B}\eta.
\ee
Also we define
\be\label{a40.4}
F(\eta)\psi=G(\eta)\psi-\Dx \omega+\px \big(T_V\eta\big),
\ee
and (recalling that $a$ is a positive function by assumption) 
\be\label{a40.5}
\ba
\ma&=\frac{1}{1+(\px\eta)^2}
\left(1+V\px \B - \B \px V-\mez G(\eta)V^2 -\mez G(\eta)\B^2-G(\eta)\eta\right),\\
\alpha&=\sqrt{a}-1.
\ea
\ee

Then, it follows from the proof of Proposition~\ref{T30} that, with the notations
\be\label{a40.6}
\vu=\begin{pmatrix} \eta\\ \Dxmez \psi\end{pmatrix} ,\quad 
\vU=\begin{pmatrix}\big(\id+T_{\alpha}\big)\eta \\ 
\Dxmez \omega\end{pmatrix},
\ee
one has
\begin{equation}\label{a-196}
\left\{
\begin{aligned}
&\partial_{t}\vU^1+T_{V}\partial_x \vU^1 - (\id+T_{\alpha})\Dxmez \vU^2 = \mathcal{F}^1,\\
&\partial_t U^2+\Dxmez T_{V\la\xi\ra^{-1/2}}\px U^2 + \Dxmez\left( (\id+T_\alpha)\vU^1\right) =\mathcal{F}^2,
\end{aligned}
\right.
\end{equation}
where
$$
\mathcal{F}^1\defn (\id +T_\alpha)\bigl(F(\eta)\psi-T_{\px V}\eta\bigr)+\Bigl\{ T_{\partial_t \alpha}+T_VT_{\px \alpha} 
+\bigl[ T_V,T_\alpha\bigr]\px \Bigr\}\eta,
$$
and
\begin{align*}
\mathcal{F}^2&=\Dxmez (T_\alpha T_\alpha -T_{\alpha^2})\eta\\
&\quad +\Dxmez (T_{V}  T_{\partial_x\eta}-T_{V \partial_x\eta})\B\\
&\quad +\Dxmez (T_{V \partial_x\B}-T_{V}  T_{\partial_x\B})\eta\\
&\quad +\mez \Dxmez \RBony(\B,\B)-\mez\Dxmez \RBony(V,V)\\
&\quad+\Dxmez T_V\RBony(\B,\partial_x\eta)
-\Dxmez \RBony(\B,V\px\eta).
\end{align*}
Notice that we write here the source terms as $\mathcal{F}^1,\mathcal{F}^2$ instead of $F^1,F^2$ as we wrote in the proof of Proposition~\ref{T30}. 
This is in order to avoid confusion with $F$ which is used later on as a compact notation for $F(\eta)\psi$. 

{\sc Strategy of the proof}

Consider $\lambda$ in $[1/2,3/2]$. We define $(\eta_\lambda,\psi_\lambda)$  by 
\e{a40.1} and denote by $B_\lambda$, $V_\lambda$, 
$\omega_\lambda$, $\vu_\lambda$, $\vU_\lambda$, $\ma_\lambda$, 
$\alpha_\lambda$, $\mathcal{F}_\lambda$ the functions obtained by 
replacing $(\eta,\psi)$ by $(\eta_\lambda,\psi_\lambda)$ in the previous expressions. 
These functions are defined for $t$ less than $T_1/\lambda\le 2T_1/3$.

The remark \e{a40.2} implies that, when $\lambda$ tends to $1$, 
$$
\frac{\eta_\lambda-\eta}{\lambda-1}, \quad 
\frac{\psi_\lambda-\psi}{\lambda-1},\quad 
\frac{\omega_\lambda-\omega}{\lambda-1}
$$
converges to $Z\eta$, $Z\psi$, $Z\omega$, respectively, in the sense of distributions. To prove the wanted result, we have to 
prove a uniform estimate for these quantities. Moreover, by using the bootstrap argument explained above, it is sufficient 
to prove an uniform estimate on some time interval $[\tzero,T]$ with $\tmtzero>0$ possibly small. Given $T$ in $[\tzero,2T_1/3]$, we define 
\be\label{a55.5}
\ba
M_\lambda(T)\defn \sup_{t\in [\tzero,T]}
\Big[&\lA \eta_\lambda(t)-\eta(t)\rA_{H^{s-1}}
+\blA \Dxmez (\psi_\lambda(t)-\psi(t))\brA_{H^{s-\tdm}}\\
&\quad+\blA \Dxmez(\omega_\lambda(t)-\omega(t))\brA_{H^{s-1}}\Big].
\ea
\ee

Our goal is to prove that 
there exist two constants $C>0$ and $T>\tzero$, 
depending only on $M_s$ as defined by \e{a55.7}, such that
\be\label{a51}
M_\lambda(T)\le C\la \lambda-1\ra.
\ee
Notice that assumption \e{a-126} implies that, at time $\tzero$, 
\be\label{a50}
\ba
M_\lambda(\tzero)=O(\la \lambda-1\ra).
\ea
\ee
Hereafter, we denote by $C$ various constants depending only on $M_s$, whose values may vary from places to places.

To prove \e{a51}, we shall prove three inequalities. The key step is to prove that there exists $C$ depending only on $M_s$ such that, 
for any $T$ in $[\tzero,2T_1/3]$, 
\be\label{a79}
\lA U_\lambda-U\rA_{L^\infty([\tzero,T];H^{s-1}(\xR))}\le 
Ce^{\tmtzero C}\big( |\lambda-1| +\tmtzero M_\lambda(T)\big).
\ee
We shall also prove that one can control a lower order norm. Namely, given 
$T$ in $[\tzero,2T_1/3]$, one introduces
\be\label{a55.55}
\ba
m_\lambda(T)\defn \sup_{t\in [\tzero,T]}
\Big[\lA \eta_\lambda(t)-\eta(t)\rA_{H^{s-2}}
+\blA \Dxmez (\psi_\lambda(t)-\psi(t))\brA_{H^{s-\frac{5}{2}}}\Big].
\ea
\ee
We shall prove that, for any $T$ in $[\tzero,2T_1/3]$,
\be\label{a79.225}
m_\lambda(T)\le Ce^{\tmtzero C}\big( |\lambda-1| +\tmtzero M_\lambda(T)\big)
\ee
and
\be\label{a79.3}
M_\lambda (T)\le C m_\lambda(T)+C \lA U_\lambda-U\rA_{L^\infty([\tzero,T];H^{s-1}(\xR))}.
\ee
Consequently, by combining these inequalities, we obtain that there 
exists $C>0$ 
depending only on $M_s$ such that, for any $T$ in $[\tzero,2T_1/3]$,
\be\label{a79.5}
M_\lambda(T)\le 
Ce^{\tmtzero C}\big( |\lambda-1| +\tmtzero M_\lambda(T)\big).
\ee
Then, there exists $T>\tzero$, depending on $M_s$, such that 
$M_\lambda(T)\le Ce^{C}|\lambda-1|+\mez M_\lambda(T)$ and hence 
$M_\lambda(T)\le 2Ce^{C}|\lambda-1|$.

To prove \e{a79}, we form an equation for $U_\lambda-U$ and estimate its 
$H^{s-1}$-norm. Write the equations \e{a-196} under the form
\be\label{a49}
\mathcal{L}(V,\alpha)\vU=\mathcal{F}.
\ee
Since $(\eta_\lambda,\psi_\lambda)$ solves 
\e{a-121}, we have
\be\label{a52}
\mathcal{L}(V_\lambda,\alpha_\lambda)\vU_\lambda=\mathcal{F}_\lambda.
\ee
It follows from \e{a49} and \e{a52} that
\be\label{a53}
\mathcal{L}(V_\lambda,\alpha_\lambda)(\vU_\lambda-\vU)
=\mathcal{F}_\lambda-\mathcal{F}
-\big(\mathcal{L}(V_\lambda,\alpha_\lambda)
-\mathcal{L}(V,\alpha)\big)\vU.
\ee

The proof is then in four steps. 
Firstly, we state an energy estimate for the equation \e{a53}. Secondly, we prove various estimates 
for $A(\eta_\lambda)\psi_\lambda-A(\eta)\psi$ where $A(\eta)$ denotes either $G(\eta)$, $F(\eta)$,\ldots 
This allows us to estimate the $L^\infty([\tzero,T];H^{s-1})$-norm of the right-hand side of \e{a53}. 
Thirdly, we estimate the $L^\infty([\tzero,T];H^{s-1})$-norm of $U_\lambda-U$. Then we conclude the proof.

\step{1}{Energy estimate}

Here we state and prove an energy estimates for the equation \e{a53}.
\begin{lemm}\label{ref:A5}
Let $\mu$ in $\xR$. Given $\tzero<T$, 
consider two real valued functions $\ti{V}$ and $\ti{\alpha}$ such that $\ti{V}$ belongs to $C^0([\tzero,T];\eC{1}(\xR))$ 
and $\ti{\alpha}$ belongs to $C^0([\tzero,T];\eC{\mez}(\xR))$. Assume that 
$\ti{U}\in C^0([\tzero,T];H^\mu(\xR))$ and $\ti{\mathcal{F}}\in L^1([\tzero,T];H^\mu(\xR))$ are real-valued and satisfy 
$\mathcal{L}(\ti{V},\ti{\alpha})\ti{U}=\ti{\mathcal{F}}$. Then
\be\label{a54}
\blA \ti{U}(t)\brA_{H^\mu}\le e^{tA(t)}\Big(\blA \ti{U}(\tzero)\brA_{H^\mu}+\blA \ti{\mathcal{F}}\brA_{L^1([\tzero,t];H^\mu(\xR))}\Big),
\ee
where
\be\label{a55}
A(t)\defn \sup_{t'\in [\tzero,t]}\Big[\blA \ti{V}(t')\brA_{\eC{1}}+\blA \ti{\alpha}(t')\brA_{\eC{\mez}}\Big].
\ee
\end{lemm}
\begin{proof}
Write
$$
\langle T_{\ti{V}}\partial_x \ti{\vU}^1 , \ti{\vU}^1\rangle_{H^\mu \times 
H^{\mu}}=\mez 
\big\langle \bigl(T_{\ti{V}}\partial_x+\bigl( T_{\ti{V}}\partial_x\bigr)^*\bigr) \ti{\vU}^1 , \ti{\vU}^1\big\rangle_{H^\mu \times 
H^{\mu}}.
$$
Since $\ti{V}$ is real-valued and $C^1$ in $x$, \e{esti:quant3} implies that 
$$
\blA T_{\ti{V}}\partial_x+\bigl( T_{\ti{V}}\partial_x\bigr)^*\brA_{\Fl{H^\mu}{H^\mu}}\les \blA \ti{V}\brA_{\eC{1}},
$$
and hence
\be\label{A56}
\bla \langle T_{\ti{V}}\partial_x \ti{\vU}^1 , \ti{\vU}^1\rangle_{H^\mu \times 
H^{\mu}}\bra \les \blA \ti{V}\brA_{\eC{1}}\blA \ti{U}\brA_{H^\mu}^2.
\ee
Similarly, writing 
$$
\Dxmez T_{\ti{V}\la\xi\ra^{-1/2}}\px =T_{\ti{V}} \px +\bigl[ \Dxmez , T_{\ti{V}\la\xi\ra^{-1/2}}\px\bigr] 
$$
and estimating the $\Fl{H^\mu}{H^\mu}$-norm of the commutator by means of \e{esti:quant2} applied with $\rho=1$, 
we find that
\be\label{A57}
\Big\vert \Big\langle \Dxmez T_{\ti{V}\la\xi\ra^{-1/2}}\px \ti{\vU}^2 , \ti{\vU}^2\Big\rangle_{H^\mu \times 
H^{\mu}}\Big\vert \les \blA \ti{V}\brA_{\eC{1}}\blA \ti{U}\brA_{H^\mu}^2.
\ee

Also, using \e{esti:quant3}Ê
with $\rho=1/2$ to estimate the $\Fl{H^{\mu-\mez}}{H^\mu}$-norm of 
$T_{\ti{\alpha}}-\big(T_{\ti{\alpha}}\big)^*$, we obtain that 
\begin{multline}\label{A58}
\Big\vert 
\Big\langle \Dxmez\left( (\id+T_{\ti{\alpha}})\ti{\vU}^1\right),\ti{\vU}^2\Big\rangle_{H^\mu\times H^\mu}
-\Big\langle (\id+T_{\ti{\alpha}})\Dxmez \vU^2 ,\ti{U}^1\Big\rangle_{H^\mu\times H^\mu}
\Big\vert\\
\les \blA \ti{\alpha}\brA_{\eC{\mez}}\blA \ti{U}\brA_{H^\mu}^2.
\end{multline}

By classical arguments, one can further assume that $\ti{\vU}$ is $C^1$ in time with values in $H^\mu$, so that 
the time derivative of $\blA \ti{\vU}\brA_{H^\mu}^2$ is given by 
$2\big\langle \partial_t \ti{U},\ti{U}\big\rangle_{H^\mu\times H^\mu}$. 
Then, by combining the previous estimates we find that
$$
\frac{d}{dt}\blA \ti{\vU}\brA_{H^\mu}^2\les \Bigl( \blA \ti{V}\brA_{\eC{1}}+\blA \ti{\alpha}\brA_{\eC{\mez}}\Bigr)
\blA \ti{\vU}\brA_{H^\mu}^2+\blA \ti{\mathcal{F}}\brA_{H^\mu}\blA \ti{U}\brA_{H^\mu},
$$
which yields the desired result.
\end{proof}

\step{2}{Estimates for the differences}

\begin{lemm}
Consider $s>4+\mez$, $\eta_1,\eta_2$ in $H^s(\xR)$ and $\psi\in \h{\mez,s-\mez}(\xR)$. 
Set
\be\label{a59}
M=\lA \eta_1\rA_{H^s}+\lA \eta_2\rA_{H^s}+\blA \Dxmez\psi\brA_{H^{s-\mez}}.
\ee
There exists a constant $C$ depending only on $M$ 
such that
\begin{alignat}{2}
&\lA G(\eta_2)\psi-G(\eta_1)\psi\rA_{H^{s-2}}&&\le C 
\lA \eta_2-\eta_1\rA_{H^{s-1}},\label{a60}\\
&\lA B(\eta_2)\psi-B(\eta_1)\psi\rA_{H^{s-2}}&&\le C 
\lA \eta_2-\eta_1\rA_{H^{s-1}},\label{a61}\\
&\lA V(\eta_2)\psi-V(\eta_1)\psi\rA_{H^{s-2}}&&\le C 
\lA \eta_2-\eta_1\rA_{H^{s-1}},\label{a62}\\
&\lA F(\eta_2)\psi-F(\eta_1)\psi\rA_{H^{s-1}}&&\le C 
\lA \eta_2-\eta_1\rA_{H^{s-1}}.\label{a63}
\end{alignat}
\end{lemm}
\begin{proof}
For $y$ in $[0,1]$, introduce 
\begin{align*}
g(y)&=G(\eta_1+y(\eta_2-\eta_1))\psi,\quad &&b(y)=B(\eta_1+y(\eta_2-\eta_1))\psi,\\
v(y)&=V(\eta_1+y(\eta_2-\eta_1))\psi,\quad &&f(y)=F(\eta_1+y(\eta_2-\eta_1))\psi.
\end{align*}

To prove \e{a60}, \e{a61}, and \e{a62}, we have to estimate 
the $H^{s-2}$-norm of $g(1)-g(0)$, $b(1)-b(0)$, and $v(1)-v(0)$. To do so, if $\varphi$ denotes either $g,b$ or $v$, 
we write
\be\label{a64}
\lA \varphi(1)-\varphi(0)\rA_{H^{s-2}}\le \int_0^1 \lA \varphi'(y)\rA_{H^{s-2}}\, dy. 
\ee
We shall prove that, for any fixed $y$ in $[0,1]$, 
$\lA \varphi'(y)\rA_{H^{s-2}}\le C\lA \eta_1-\eta_2\rA_{H^{s-1}}$ for some constant $C$ depending only on 
$M$ defined by \e{a59}. Similarly, to prove \e{a63}, we shall prove that 
$\lA f'(y)\rA_{H^{s-1}}\le C\lA \eta_1-\eta_2\rA_{H^{s-1}}$ for some constant $C$ depending only on $M$.

Let us prove \e{a60}. 
Fix $y$ in $[0,1]$ and set
$$
\eta(y)=\eta_1+y(\eta_2-\eta_1),\quad 
\dot{\eta}=\eta'(y)=\eta_2-\eta_1.
$$
We use the property, proved by Lannes~\cite{LannesJAMS}, that one has 
an explicit expression of the derivative of~$G(\eta)\psi$ 
with respect to~$\eta$. As in \e{n142}, one has
\be\label{a67}
g'(y)=-G(\eta(y))\big[ \dot{\eta} b(y)\big]-\px \big[ \dot{\eta} v(y)\big].
\ee
In this proof, we denote by $C$ various constants depending only on $M$ defined by \e{a59} (and independent of $y\in [0,1]$). With this notation, it follows 
from \e{2113}Ê
and the Sobolev embedding that 
$\lA b(y)\rA_{H^{s-1}}\le C$ and $\lA v(y)\rA_{H^{s-1}}\le C$. Also, it follows from \e{2113}Ê
and the Sobolev embedding that for any $s>3+\mez$, any $y$ in $[0,1]$ and any 
$f$ in $\h{\mez,s-\mez}(\xR)$,
$$
\lA G(\eta(y))f\rA_{H^{s-1}}\le C\blA \Dxmez f\brA_{H^{s-\mez}}.
$$
By using this estimate with $s$ replaced by $s-1$ and $f$ replaced by $\dot{\eta}B$, we find that
$$
\blA G(\eta(y))\big[ \dot{\eta} b(y)\big]\brA_{H^{s-2}}\le C\blA \dot{\eta}b(y)\brA_{H^{s-1}}.
$$
Since $H^{s-1}(\xR)$ is an algebra, this gives
$$
\blA G(\eta(y))\big[ \dot{\eta} b(y)\big]\brA_{H^{s-2}}
\le C\lA b(y)\rA_{H^{s-1}}\lA \dot{\eta}\rA_{H^{s-1}}\le C\blA \dot{\eta}\brA_{H^{s-1}}.
$$
On the other hand, using the fact that $H^{s-1}(\xR)$ is an algebra, one has
$$
\blA \px \big[ \dot{\eta} v(y)\big]\brA_{H^{s-2}}
\les \lA \dot{\eta}\rA_{H^{s-1}} \lA v(y)\rA_{H^{s-1}}\le C\lA \dot{\eta}\rA_{H^{s-1}}.
$$
By combining the two previous estimates we conclude that there exists a constant $C=C(M)$ such that, for any 
$y$ in $[0,1]$, one has
$\lA g'(y)\rA_{H^{s-2}}\le C\lA \dot{\eta}\rA_{H^{s-1}}$. Then \e{a60} follows from \e{a64}.

Since 
$$
b(y)=\frac{1}{1+(\px \eta(y))^2}\big(g(y)+(\px \psi)\px\eta (y)\big),\quad 
v(y)=\px\psi-b(y)\px \eta(y),
$$
and since $H^{s-2}(\xR)$ is an algebra for $s>2+\mez$, the previous estimate 
for the $H^{s-2}(\xR)$-norm of $g'(y)$ easily implies that
\be\label{a68}
\lA b'(y)\rA_{H^{s-2}}+\lA v'(y)\rA_{H^{s-1}}\le C\lA \dot{\eta}\rA_{H^{s-1}},
\ee
which proves \e{a61} and \e{a62} (using \e{a64}).

Let us prove \e{a63}. We want to prove that 
$\lA f'(y)\rA_{H^{s-1}}\le C\lA \eta_1-\eta_2\rA_{H^{s-1}}$. 
Since
$$
f(y)=g(y)-\Dx\big(\psi-T_{b(y)}\eta(y)\big)+\px\big(T_{v(y)}\eta(y)\big)
$$
it follows from \e{a67} that
\begin{align*}
f'(y)&=-G(\eta(y))\big[ \dot{\eta} b(y)\big]-\px \big[ \dot{\eta} v(y)\big]
+\Dx T_{b(y)} \dot{\eta}+\px T_{v(y)} \dot{\eta}\\
&\quad+\Dx T_{b'(y)}\eta(y)+\px T_{v'(y)}\eta(y).
\end{align*}
Replace $G(\eta(y))$ by $G(\eta(y))-\Dx+\Dx$ in the first term of the right-hand side to obtain $f'(y)=A_1+A_2+A_3$ with
\begin{align*}
A_1&=-\Dx \big( \dot{\eta} b(y)\big)-\px\big(\dot{\eta}v(y)\big)+\Dx T_{b(y)}\dot{\eta}
+\px T_{v(y)}\dot{\eta},\\
A_2&=-\big(G(\eta(y))-\Dx\big)\big(\dot{\eta}b(y)\big),\\
A_3&=\Dx T_{b'(y)}\eta(y)+\px T_{v'(y)}\eta(y).
\end{align*}

Let us estimate the $H^{s-1}$-norm of $A_2$. 
Since $s-3/2>3-1/2$, we can apply 
the estimate \e{n129} with $\mu=s$ (and the Sobolev embedding) to obtain that
$$
\lA A_2\rA_{H^{s-1}}\le C\big( \lA \eta(y)\rA_{H^s}\big)
\blA \Dxmez (\dot{\eta}b(y))\brA_{H^{s-\tdm}}.
$$
Now write
$$
\blA \Dxmez (\dot{\eta}b(y))\brA_{H^{s-\tdm}}\les \lA \dot{\eta}\rA_{H^{s-1}}
\lA b(y)\rA_{H^{s-1}}\le C\lA \dot{\eta}\rA_{H^{s-1}}.
$$
This prove that $\lA A_2\rA_{H^{s-1}}\le C\lA \dot{\eta}\rA_{H^{s-1}}$. 

Next we estimate the $H^{s-1}$-norm of $A_3$. It follows from \e{esti:quant0} that
$$
\lA A_3\rA_{H^{s-1}}\les \lA b'(y)\rA_{L^\infty}\lA \eta(y)\rA_{H^{s}}
+\lA v'(y)\rA_{L^\infty}\lA \eta(y)\rA_{H^{s}}.
$$
Since $s>5/2$, the Sobolev embedding implies that
$$
\lA A_3\rA_{H^{s-1}}\les \lA b'(y)\rA_{H^{s-2}}\lA \eta(y)\rA_{H^{s}}
+\lA v'(y)\rA_{H^{s-2}}\lA \eta(y)\rA_{H^{s}}.
$$
So the estimate \e{a68} for $\lA b'(y)\rA_{H^{s-2}}$ and 
$\lA v'(y)\rA_{H^{s-2}}$ imply that $\lA A_3\rA_{H^{s-1}}\le C\lA \dot{\eta}\rA_{H^{s-1}}$.

Finally, it remains to estimate the $H^{s-1}$-norm of $A_1$. 
Here we cannot estimate the terms separately: we need to exploit some cancellations and follow the proof of Lemma~$6.8$ in \cite{ABZ1}. 
Firstly, we paralinearize $\dot{\eta}b(y)$ and $\dot{\eta}v(y)$ (see~\e{defi:RBony}) 
to obtain that
\begin{align*}
A_1&=-\Dx (T_{\dot{\eta}}b(y))-\Dx \big(\RBony(\dot{\eta},b(y))\big)\\
&\quad -\px (T_{\dot{\eta}}v(y))-\px \big(\RBony(\dot{\eta},v(y)\big).
\end{align*}
Directly from \e{Bony3} and the Sobolev embedding we have
\begin{align*}
&\blA \Dx \big(\RBony(\dot{\eta},b(y))\big)\brA_{H^{s-1}}
\les \lA \dot{\eta}\rA_{H^{s-1}}\lA b(y)\rA_{H^{s-1}}\le C \lA \dot{\eta}\rA_{H^{s-1}},\\
&\blA \Dx \big(\RBony(\dot{\eta},v(y))\big)\brA_{H^{s-1}}
\les \lA \dot{\eta}\rA_{H^{s-1}}\lA v(y)\rA_{H^{s-1}}\le C \lA \dot{\eta}\rA_{H^{s-1}}.
\end{align*}

It remains to estimate the $H^{s-1}$-norm of $\widetilde{A}_1\defn 
-\Dx (T_{\dot{\eta}}b(y))-\px (T_{\dot{\eta}}v(y))$. This we now do using the identity 
$G(\eta(y))b(y)=-\px v(y)$ (see~\e{GBGV} or~\e{232e}). Write
$$
\widetilde{A}_1=-T_{\dot{\eta}}\Dx b(y)-T_{\dot{\eta}}\px v(y) 
+\big[ T_{\dot{\eta}},\Dx\big] b(y)-T_{\px \dot{\eta}}v(y),
$$
and replace $\px v(y)$ by $-G(\eta(y))b(y)$ in the second term, to obtain
$$
\widetilde{A}_1=T_{\dot{\eta}}\big( G(\eta(y))-\Dx\big)b(y)
+\big[ T_{\dot{\eta}},\Dx\big] b(y)-T_{\px \dot{\eta}}v(y).
$$
Using \e{esti:quant0}, \e{n129} and the Sobolev embedding, we have
\begin{align*}
\blA T_{\dot{\eta}}\big( G(\eta(y))-\Dx\big)b(y)\brA_{H^{s-1}}
&\les \lA \dot{\eta}\rA_{L^\infty}\blA \big( G(\eta(y))-\Dx\big)b(y)\brA_{H^{s-1}}\\
&\les C\lA \dot{\eta}\rA_{H^{s-1}}\lA b(y)\rA_{H^{s-1}}\le C \lA \dot{\eta}\rA_{H^{s-1}}.
\end{align*}
On the other hand, using \e{L312:1}Ê
and the Sobolev embedding, we have
$$
\blA \big[ T_{\dot{\eta}},\Dx\big] b(y)\brA_{H^{s-1}}\les 
\lA \dot{\eta}\rA_{\eC{1}}\lA b(y)\rA_{H^{s-1}}\le C \lA \dot{\eta}\rA_{H^{s-1}}.
$$
Also, using \e{esti:quant0} and the Sobolev embedding $H^{s-2}(\xR)\subset L^{\infty}(\xR)$, we have
$$
\blA T_{\px \dot{\eta}}v(y)\brA_{H^{s-1}}\les \lA \dot{\eta}\rA_{H^{s-1}}\lA v(y)\rA_{H^{s-1}}
\le C\lA \dot{\eta}\rA_{H^{s-1}}.
$$
This completes the proof of the lemma.
\end{proof}

Use the abbreviate notations
\begin{alignat*}{4}
G&=G(\eta)\psi,\quad &&B=B(\eta)\psi,\quad &&V=V(\eta)\psi,\quad 
&&F=F(\eta)\psi,\\
G_\lambda&=G(\eta_\lambda)\psi_\lambda,\quad &&B_\lambda=B(\eta_\lambda)\psi_\lambda,
\quad &&V_\lambda=V(\eta_\lambda)\psi_\lambda,\quad 
&&F_\lambda=F(\eta_\lambda)\psi_\lambda.
\end{alignat*}

\begin{lemm}\label{ref:A55}
There exists a constant $C$ depending only on $M_s$ 
such that for any $\lambda$ in $[1/2,3/2]$ and any $T$ in $[\tzero,2T_1/3]$,
\begin{alignat}{2}
&\lA G_\lambda-G\rA_{L^\infty([\tzero,T];H^{s-2})}&&\le C  M_\lambda(T),\label{a70}\\
&\lA B_\lambda-B\rA_{L^\infty([\tzero,T];H^{s-2})}&&\le C  M_\lambda(T),\label{a71}\\
&\lA V_\lambda-V\rA_{L^\infty([\tzero,T];H^{s-2})}&&\le C  M_\lambda(T),\label{a72}\\
&\lA F_\lambda-F\rA_{L^\infty([\tzero,T];H^{s-1})}&&\le C  M_\lambda(T),\label{a73}\\
&\lA a_\lambda-a\rA_{L^\infty([\tzero,T];H^{s-3})}&&\le C  M_\lambda(T),\label{a74}\\
&\lA \alpha_\lambda-\alpha\rA_{L^\infty([\tzero,T];H^{s-3})}&&\le C  M_\lambda(T)\label{a75},\\
&\lA \partial_t\alpha_\lambda-\partial_t\alpha\rA_{L^\infty([\tzero,T];H^{s-4})}&&\le C  M_\lambda(T).\label{a76}
\end{alignat}
\end{lemm}
\begin{proof}
We shall see that these inequalities hold with $M_\lambda(T)$ replaced by
\be\label{a76.7}
\sup_{t\in [\tzero,T]}
\Big[\lA \eta_\lambda(t)-\eta(t)\rA_{H^{s-1}}
+\blA \Dxmez (\psi_\lambda(t)-\psi(t))\brA_{H^{s-\tdm}}\Big].
\ee
Notice that, since $\psi_\lambda=\lambda^{-3}\psi(\lambda t,\lambda^2x)$, we have
\be\label{a77}
\sup_{\lambda\in [\mez,\tdm]}\sup_{t\in [\tzero,T]}\Big[ 
\lA \eta_\lambda\rA_{H^s}+
\blA \Dxmez \psi_\lambda(t)\brA_{H^{s-\mez}}\Big]
\les M_s 
\ee
where $M_s$ is defined by \e{a55.7}. 

To prove \e{a70}, we write
$$
G_\lambda-G=\big(G(\eta_\lambda)-G(\eta)\big)\psi_\lambda
+G(\eta)\big[ \psi_\lambda-\psi\big].
$$
The estimate \e{a60} and \e{a77}Ê
imply that
$$
\lA \big(G(\eta_\lambda)-G(\eta)\big)\psi_\lambda\rA_{H^{s-2}}
\le C \lA \eta_\lambda-\eta\rA_{H^{s-1}}.
$$
On the other hand, the estimate \e{2113}Ê
and the Sobolev embedding imply that 
$$
\lA G(\eta)\big[ \psi_\lambda-\psi\big]\rA_{H^{s-2}}
\le C \blA \Dxmez \big( \psi_\lambda-\psi\big)\brA_{H^{s-\tdm}}.
$$
By combining the two previous estimates we obtain \e{a70}. 

The proof of \e{a71}, \e{a72}, and \e{a73}Ê
are similar. 
Now the estimate \e{a74} follows from similar arguments, the previous estimates and 
the formula \e{formule:a}. The estimate \e{a75}Ê
follows from \e{a74}Ê
and the definition 
of $\alpha=\sqrt{a}-1$. To prove \e{a76}, one differentiates in time the formula \e{formule:a} using the rule \e{n187.2} and then one replaces in the expression thus obtained $\partial_t V$ and $\partial_t \eta$ by the expressions given by Lemma~\ref{L:A.4.1} (and one replaces $\partial_t B$ by 
$-V\px \B+\ma -1$ according to the definition $a=1+\partial_t B+V\px B$).
\end{proof}

\step{3}{Energy estimates for $U_\lambda-U$}

Hereafter, we denote by $C$ various constants depending only on 
$M_s$ (defined by \e{a55.7}), whose values may vary from places to places. 
With this notation, it follows from \e{2113}, \e{n192} and the Sobolev embedding that
\be\label{a55.8}
\lA V\rA_{H^{s-1}}\le C,\quad \lA \B\rA_{H^{s-1}}\le C, \quad 
\lA \alpha\rA_{\eC{1}}\le C.
\ee

Remembering that
\be\label{a80}
\mathcal{L}(V_\lambda,\alpha_\lambda)(\vU_\lambda-\vU)=\mathcal{F}_\lambda-\mathcal{F}
+\big(\mathcal{L}(V_\lambda,\alpha_\lambda)
-\mathcal{L}(V,\alpha)\big)\vU,
\ee
the wanted estimate \e{a79} will be obtained by 
applying Lemma~\ref{ref:A5} with 
\be\label{a55.9}
\ti{U}=U_\lambda-U,\quad \ti{V}=V_\lambda,\quad \ti{\alpha}=\alpha_\lambda, \quad 
\ti{\mathcal{F}}=\mathcal{F}_\lambda-\mathcal{F}
+\big(\mathcal{L}(V_\lambda,\alpha_\lambda)
-\mathcal{L}(V,\alpha)\big)\vU.
\ee

Since $V_\lambda=\lambda^{-1}V(\lambda t,\lambda^2x)$ and 
$\alpha_\lambda=\alpha(\lambda t,\lambda^2x)$, as can be checked 
by direct computations, it follows from \e{a55.8} and the Sobolev embedding that
\be\label{a56}
\ba
&\sup_{\lambda\in [\mez,\tdm]} \lA V_\lambda\rA_{L^\infty([\tzero,2T_1/3];\eC{1}(\xR))}
\les \lA V\rA_{L^\infty([\tzero,T_1];\eC{1}(\xR))}\le C,\\
&\sup_{\lambda\in [\mez,\tdm]} \lA \alpha_\lambda\rA_{L^\infty([\tzero,2T_1/3];\eC{\mez}(\xR))}
\les  \lA \alpha\rA_{L^\infty([\tzero,T_1];\eC{\mez}(\xR))}\le C.
\ea
\ee
Similarly, for any $T$ in $[\tzero,2T_1/3]$, the estimates \e{a72}Ê
and \e{a75} imply that
\be\label{a57}
\ba
&\sup_{\lambda\in [\mez,\tdm]} \lA V_\lambda-V\rA_{L^\infty([\tzero,T];L^\infty(\xR))}
\le CM_\lambda(T),\\
&\sup_{\lambda\in [\mez,\tdm]} \lA \alpha_\lambda-\alpha\rA_{L^\infty([\tzero,T];L^\infty(\xR))}
\le C M_\lambda(T).
\ea
\ee

We use \e{a56}Ê
to control the quantity $A$ defined by \e{a55}. Our next task consists in proving that 
the source  
term $\ti{\Fr}$ defined by \e{a55.9} satisfies
\be\label{a57.5}
\blA \ti{\Fr}\brA_{L^1([\tzero,T];H^{s-1})}\le \tmtzero C M_\lambda(T).
\ee
To do so, it is obviously sufficient to prove that 
$\blA \ti{\Fr}\brA_{L^\infty([\tzero,T];H^{s-1})}\le C M_\lambda(T)$. 
By \e{a57} and \e{esti:quant0} we have
$$
\blA \big(\mathcal{L}(V_\lambda,\alpha_\lambda)
-\mathcal{L}(V,\alpha)\big)\vU\brA_{L^\infty([0,T];H^{s-1})}\le 
C  M_\lambda(T). 
$$
On the other hand, by using the paradifferential rules recalled in Appendix~\ref{s2}, 
the estimates proved in Lemma~\ref{ref:A55} imply that 
$$
\blA \mathcal{F}_\lambda-\mathcal{F}\brA_{L^\infty([0,T];H^{s-1})}\le 
C  M_\lambda(T).
$$
This completes the proof of \e{a57.5} and hence gives the wanted estimate \e{a79}.

\step{4}{End of the proof}

It remains only to prove \e{a79.225} and \e{a79.3}. 

Let us prove that
\begin{align}
\lA \eta_\lambda-\eta\rA_{L^\infty([\tzero,T];H^{s-2}(\xR))}&\le 
Ce^{\tmtzero C}\big( |\lambda-1| +\tmtzero M_\lambda(T)\big),\label{a79.2}\\
\blA \Dxmez (\psi_\lambda-\psi)\brA_{L^\infty([\tzero,T];H^{s-\tdm}(\xR))}&\le 
Ce^{\tmtzero C}\big( |\lambda-1| +\tmtzero M_\lambda(T)\big)\label{a79.25}.
\end{align}
Using the previous notations, write
$\partial_t(\eta_\lambda-\eta)=G_\lambda-G$. By integrating in time this identity, it follows from \e{a70} 
that for any $T$ in $[\tzero,2T_1/3]$,
$$
\lA \eta_\lambda(t)-\eta(t)\rA_{H^{s-2}}\le \lA \eta_\lambda(\tzero)-\eta(\tzero)\rA_{H^{s-2}}
+\tmtzero C M_\lambda(T).
$$
So the estimate \e{a79.2} follows from $M_\lambda(\tzero)=O(\la \lambda-1\ra)$ (see~\e{a50}) 
and the fact that $\lA \eta_\lambda(\tzero)-\eta(\tzero)\rA_{H^{s-2}}$ is smaller than 
$M_\lambda(\tzero)$. 
The estimate \e{a79.25} is proved similarly. This proves \e{a79.225}.

It remains only to prove \e{a79.3}. By definitions \e{a55.5} and \e{a55.55}, and \e{a40.6} 
we have
$$
\ba
M_\lambda(T)\defn \sup_{t\in [\tzero,T]}
\Big[&\lA \eta_\lambda(t)-\eta(t)\rA_{H^{s-1}}
+\blA \Dxmez (\psi_\lambda(t)-\psi(t))\brA_{H^{s-\tdm}}\\
&\quad+\blA \vU_\lambda^2-\vU^2\brA_{H^{s-1}}\Big]
\ea
$$
and
$$
\ba
m_\lambda(T)\defn \sup_{t\in [\tzero,T]}
\Big[\lA \eta_\lambda(t)-\eta(t)\rA_{H^{s-2}}
+\blA \Dxmez (\psi_\lambda(t)-\psi(t))\brA_{H^{s-\frac{5}{2}}}\Big].
\ea
$$
So to prove \e{a79.3}, we need only prove that 
\begin{align}
&\lA \eta_\lambda-\eta\rA_{L^\infty([\tzero,T];H^{s-1}(\xR))}\le 
C m_\lambda(T)+C \lA U_\lambda-U\rA_{L^\infty([\tzero,T];H^{s-1}(\xR))},\label{a89.1}\\
&\blA \Dxmez (\psi_\lambda-\psi)\brA_{L^\infty([\tzero,T];H^{s-\tdm}(\xR))}\le 
C m_\lambda(T)+C \lA U_\lambda-U\rA_{L^\infty([\tzero,T];H^{s-1}(\xR))}.\label{a89.2}
\end{align}
We shall prove \e{a89.1} only. To do so, we shall write 
$\eta_\lambda-\eta$ in terms of $\vU_\lambda^1-\vU^1$ and in terms of a smoothing operator acting on $\eta_\lambda-\eta$. To do so, remembering that $\alpha=\sqrt{a}-1$, we first write that
\be\label{a90}
T_{\sqrt{a}}\eta=\big(\Id+T_\alpha\big)\eta+(T_1-\Id)\eta.
\ee
Then we let act a parametrix of $T_{\sqrt{a}}$, that is $T_{1/\sqrt{a}}$, to obtain
$$
\eta=T_{1/\sqrt{a}} T_{\sqrt{a}}\eta+\big(\Id-T_{1/\sqrt{a}}T_{\sqrt{a}}\big)\eta
$$
and hence, using \e{a90}, 
\be\label{a91}
\eta=T_{1/\sqrt{a}} \big(\Id+T_\alpha\big)\eta+T_{1/\sqrt{a}}(T_1-\Id)\eta
+\big(\Id-T_{1/\sqrt{a}}T_{\sqrt{a}}\big)\eta\\
\ee
Remembering that $\vU^1= \big(\Id+T_\alpha\big)\eta$, this yields $\eta=K \vU^1+R\eta$ where
$$
K=T_{1/\sqrt{a}},\quad R=
T_{1/\sqrt{a}}(T_1-\Id)
+\big(\Id-T_{1/\sqrt{a}}T_{\sqrt{a}}\big)
$$
Using obvious notations, one thus writes that
\begin{align*}
\eta_\lambda-\eta
&=K_\lambda\big[ U_\lambda^1-U^1\big]+R_\lambda\big[ \eta_\lambda-\eta\big]\\
&\quad +(K_\lambda-K)U^1+(R_\lambda-R)\eta,
\end{align*}
and hence
\begin{align*}
\lA \eta_\lambda-\eta\rA_{H^{s-1}}
&\le \lA K_\lambda\rA_{\Fl{H^{s-1}}{H^{s-1}}}\blA U_\lambda^1-U^1\brA_{H^{s-1}}
+\lA R_\lambda\rA_{\Fl{H^{s-2}}{H^{s-1}}} \lA \eta_\lambda-\eta\rA_{H^{s-2}}\\
&\quad +\lA K_\lambda-K\rA_{\Fl{H^{s-1}}{H^{s-1}}}\lA U^1\rA_{H^{s-1}}
+\lA R_\lambda-R\rA_{\Fl{H^{s-1}}{H^{s-1}}} \lA \eta\rA_{H^{s-1}}.
\end{align*}
Notice that
$$
\lA U^1\rA_{H^{s-1}}\le C,\quad \lA \eta\rA_{H^{s-1}}\le C, 
\quad  \lA \eta_\lambda-\eta\rA_{H^{s-2}}\le m_\lambda.
$$
Also, using \e{esti:quant0} and \e{esti:quant2-func} applied with $\rho=1$, we easily check 
that
$$
\lA K_\lambda\rA_{\Fl{H^{s-1}}{H^{s-1}}}\le C,\quad 
\lA R_\lambda\rA_{\Fl{H^{s-2}}{H^{s-1}}}\le C,
$$
where one used again that $\sqrt{a_\lambda}$ and 
$1/\sqrt{a_\lambda}$ are uniformly bounded in $L^\infty([\tzero,(2/3)T_1];C^{1}(\xR))$ 
with respect to $\lambda\in [1/2,3/2]$. 

Finally, to estimate $\lA K_\lambda-K\rA_{\Fl{H^{s-1}}{H^{s-1}}}$ and 
$\lA R_\lambda-R\rA_{\Fl{H^{s-1}}{H^{s-1}}}$ we 
apply \e{a74} with $s$ replaced by $s-1$, to obtain
$$
\lA a_\lambda-a\rA_{L^\infty([0,T];L^\infty)}\les 
\lA a_\lambda-a\rA_{L^\infty([0,T];H^{s-4})}\le Cm_\lambda(T).
$$
Indeed, as mentioned in the proof of Lemma~\ref{ref:A55}, the estimate \e{a74} remains true when $M_\lambda(T)$ is replaced by \e{a76.7}.
\end{proof}

We conclude this appendix by proving a technical result. 
Consider two functions $\eta$ and $\psi$ and use 
the notations recalled above (see~\e{a40.3} and \e{a40.5}) for $V$ and $\alpha$. We consider the operator $C$ defined by 
$$
C(\eta,\psi)\vU=\begin{pmatrix}
T_{V-\px\psi}\partial_x \vU^1- T_{\alpha+\mez\Dx\eta}\Dxmez \vU^2 \\
\Dxmez T_{(V-\px\psi)\la\xi\ra^{-1/2}}\px \vU^2
+ \Dxmez T_{\alpha+\mez\Dx\eta}\vU^1
\end{pmatrix}.
$$
The operator $C$ is of order $1$. The following result 
states that its real part is of order $0$ with tame estimates for its operator norm.
\begin{lemm}\label{ref:A56}
Consider $\varrho\in ]4,+\infty[$ and $\mu\in \xR$. 
For any $(\eta,\psi)\in \eC{\varrho}\times \eC{\mez,\varrho}$ such that 
$(\eta,\psi)$ belongs to the set $\Eg{\gamma}$ introduced after the statement of Proposition~\ref{ref:116}, and for 
any $\vU=(\vU^1,\vU^2)$ in $H^{\mu+1}(\xR)\times H^{\mu+1}(\xR)$, there holds
\begin{equation}\label{p273}
\big\lvert \RE \langle C(\eta,\psi)\vU,\vU \rangle_{H^\mu\times H^\mu} \bigr\rvert
\le K\Big( 
\lA \eta\rA_{\eC{\varrho}}+\blA \Dxmez \psi\brA_{\eC{\varrho}}\Big)^2
\blA \vU\brA_{H^\mu}^2.
\end{equation}
for some constant $K$ depending only on 
$\lA \eta\rA_{\eC{\varrho}}+\blA \Dxmez \psi\brA_{\eC{\varrho}}$.
\end{lemm}
\begin{proof}
Set $\ti{V}=V-\px\psi$ and $\ti{\alpha}=\alpha+\mez \Dx\eta$. 
It follows from \e{A56}, \e{A57}, and \e{A58} that
$$
\big\lvert \RE \langle C(\eta,\psi)\vU,\vU \rangle_{H^\mu\times H^\mu}
\le\Big( 
\blA \ti{V}\brA_{\eC{1}}+\lA \ti{\alpha}\rA_{\eC{\mez}}\Big)
\blA \vU\brA_{H^\mu}^2.
$$

So to prove~\e{p273} we need only prove that
$$
\blA \ti{V}\brA_{\eC{1}}+\lA \ti{\alpha}\rA_{\eC{\mez}} \le K (N_\varrho)
N_\varrho^2\quad\text{where }N_\varrho\defn 
\lA \eta\rA_{\eC{\varrho}}+\blA \Dxmez \psi\brA_{\eC{\varrho}}^2.
$$
Recall (cf \e{211-1}) that
\be\label{211-1a}
\lA B\rA_{\eC{\varrho-1}}+\lA V\rA_{\eC{\varrho-1}}
\le K(N_\varrho)N_\varrho.
\ee

Since $V-\px\psi=\B \px\eta$ the wanted estimate for $V-\px\psi$ follows 
from the previous inequality and the fact that $\eC{1}(\xR)$ is an algebra. 

Also, using \e{211-1a} and 
applying \e{211-1} with $\psi$ replaced by $B^2$ or $V^2$, there holds 
$$
\lA G(\eta)B^2\rA_{\eC{\varrho-1}}+\lA G(\eta)V^2\rA_{\eC{\varrho-1}}
\le K(N_\varrho)N_\varrho^2.
$$
It thus follows from the identity \eqref{n191} for $\ma$ that 
$$
\lA a-1+G(\eta)\eta\rA_{\eC{1}}\le C(\holder)\holder^2.
$$
Now recall from \e{n145}Ê
that, for any $\gamma>3$, there holds
$$
\lA G(\eta)\psi-\Dx\psi\rA_{\eC{\gamma-2}}\le C(\dbeta) \dbeta \dalpha,
$$
By using this estimate with $\eta=\psi$ we conclude that
$$
\lA a-1+\Dx\eta\rA_{\eC{1}}\le C(\holder)\holder^2.
$$
So the wanted estimate for $\alpha+\mez \Dx\eta$ follows from the 
definition of $\alpha=\sqrt{a}-1$.
\end{proof}


\addcontentsline{toc}{chapter}{Bibliography}

\backmatter

\bibliographystyle{plain}
\bibliography{bib_AD.bib}

\clearpage
\thispagestyle{empty}

\noindent\textbf{Thomas Alazard}\\
\noindent D\'epartement de Math\'ematiques et Applications\\
\noindent \'Ecole normale sup\'erieure et CNRS UMR 8553 \\ 
\noindent 45 rue d'Ulm\\
\noindent F-75230 Paris, France.

\vspace{2mm}

\noindent\textbf{Jean-Marc Delort}\\
\noindent Universit\'e Paris 13 \\
\noindent Sorbonne Paris Cit\'e, LAGA, CNRS (UMR 7539) \\
\noindent 99, avenue J.-B. Cl\'ement\\
\noindent F-93430 Villetaneuse

\end{document}